\documentclass[11pt, a4paper, twoside]{book}
\usepackage{vmargin}
\setmarginsrb  { 0.7in}  
               { 0.6in}  
               { 0.7in}  
               { 0.8in}  
               {  20pt}  
               {0.25in}  
               {   9pt}  
               { 0.3in}  
\usepackage{color}
\usepackage{amsmath}
\catcode`\@=11 \@addtoreset{equation}{section}

\catcode`\@=12

\usepackage{amssymb}
\usepackage{amsfonts}
\usepackage{xcolor}
\usepackage{graphics}
\usepackage{epsfig,psfrag,graphicx}
\usepackage{amssymb}
\usepackage{eepic,epic}
\usepackage{dsfont}

\usepackage{enumerate}

\usepackage{bbm}

\usepackage{epsfig} 
\usepackage{amsmath} 
\usepackage{amssymb} 
\usepackage{amsbsy} 
\usepackage{chngcntr}
\usepackage{eso-pic}
\usepackage{calc}
\usepackage{accents}

\newtheorem{definition}{Definition}[section]
\newtheorem{theoremnonum}{Theorem}
\newtheorem{theoremnonum_fr}{Théorème}
\newtheorem{propnonum}[theoremnonum]{Proposition}
\newtheorem{theorem}[definition]{Theorem}
\newtheorem{lemma}[definition]{Lemma}
\newtheorem{corollary}[definition]{Corollary}
\newtheorem{proposition}[definition]{Proposition}
\newtheorem{remark}[definition]{Remark}

\newcommand{\nc}{\newcommand}
\nc{\qed}{\mbox{}\nolinebreak\hfill \rule{2mm}{2mm}} 
\nc{\weak}{\rightharpoonup}
\nc{\weakstar}{\stackrel{\ast}{\rightharpoonup}} 
\nc{\proof}{{\bf Proof: }} 
\renewcommand{\div}{{{\mathrm{div}}_x}\,}
\newcommand{\vrho}{\varrho}
\nc{\modular}[1]{{\stackrel{ #1}{\longrightarrow\,}}}

\def\bbbone{{\mathchoice {\rm 1\mskip-4mu l}
{\rm 1\mskip-4mu l} {\rm 1\mskip-4.5mu l} {\rm 1\mskip-5mu l}}}

\def\vec#1{\boldsymbol{#1}}

\newcommand{\anl}{a_\veps^{n+l}}
\newcommand{\an}{a_\veps^{n+1}}
\newcommand{\vu}		{\vec{u}}
\newcommand{\vr}		{\vrho}
\newcommand{\vrm}       {\vrho^{(1)}_{\veps,M}}
\newcommand{\vre}		{\vr_\varepsilon}
\newcommand{\vrez}	{\vr_{0,\varepsilon}}

\newcommand{\vret}		{\wtilde{\vr}_\ep}
\newcommand{\ue}		{\vec{u}_\varepsilon}
\newcommand{\uez}		{\vec{u}_{0,\varepsilon}}
\newcommand{\ep}		{\varepsilon}
\newcommand{\temp}	{\vartheta}
\newcommand{\tem}		{\vartheta_\varepsilon}
\newcommand{\temz}	{\vartheta_{0,\varepsilon}}
\newcommand{\tems}	{\overline{\vartheta}}

\newcommand{\q}		{\vec{q}}
\newcommand{\n}		{\vec{n}}
\newcommand{\vU}		{\vec{U}}
\newcommand{\ess} 	{{\rm{ess}}}
\newcommand{\res}		{{\rm{res}}}
\newcommand{\ds}		{\,{\rm d}s}
\newcommand{\dx}		{\,{\rm d}x}
\newcommand{\dt}		{\, {\rm d}t}
\newcommand{\detau}		{\, {\rm d}\tau}
\newcommand{\dxdt}	{\, {\rm d}x{\rm d}t}

\newcommand{\U}		{\vec{U}}

\def\bbbone{{\mathchoice {\rm 1\mskip-4mu l}
{\rm 1\mskip-4mu l} {\rm 1\mskip-4.5mu l} {\rm 1\mskip-5mu l}}}

\renewcommand{\bbbone}{\mathds{1}}

\newcommand{\mbb}{\mathbb}

\newcommand{\mc}{\mathcal}
\newcommand{\mf}{\mathfrak}
\newcommand{\veps}{\varepsilon}
\newcommand{\vtheta}{\vartheta}
\newcommand{\what}{\widehat}
\newcommand{\wtilde}{\widetilde}
\newcommand{\vphi}{\varphi}
\newcommand{\oline}{\overline}

\newcommand{\ra}{\rightarrow}

\newcommand{\g}{\gamma}
\newcommand{\z}{\zeta}
\newcommand{\s}{\sigma}

\newcommand{\de}{\delta}

\newcommand{\lan}{\langle}
\newcommand{\ran}{\rangle}

\newcommand{\e}{\vec{e}}

\newcommand{\R}{\mathbb{R}}

\newcommand{\N}{\mathbb{N}}
\newcommand{\Z}{\mathbb{Z}}
\newcommand{\B}{\mathbb{B}}
\newcommand{\T}{\mathbb{T}^1}

\newcommand{\TT}{\mathbb{T}}

\newcommand{\h}{\mathbb{H}}

\renewcommand{\div}{{\rm div}\,}
\newcommand{\curl}{{\rm curl}\,}
\newcommand{\divh}{{\rm div}_h}
\newcommand{\curlh}{{\rm curl}_h}

\newcommand{\Id}{{\rm Id}\,}
\newcommand{\Supp}{{\rm Supp}\,}

\newcommand{\dbtilde}[1]{\accentset{\approx}{#1}}

\allowdisplaybreaks

\def\d{\partial}
\def\div{{\rm div}\,}

\begin{document}
  \begin{titlepage}
 \begingroup
    \fontsize{14pt}{12pt}\selectfont 
   \begin{center}

\vspace*{-2.5cm}
\begin{figure}[h]
    \centering
    \includegraphics[width=17cm]{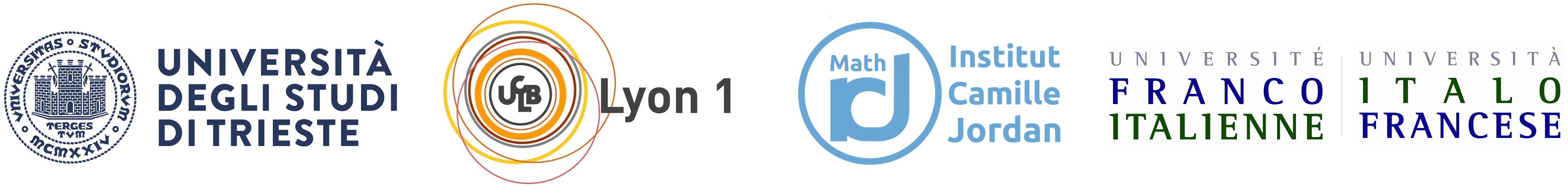}
\end{figure}

\vspace{0.5cm}

           {\textsc{Universit\`a degli Studi di Trieste}\\}
           \vspace{0.2cm}
           {\textit{Dipartimento di Matematica e Geoscienze}\\}
\vspace{0.4cm}  
           {\textsc{Universit\'e de Lyon, Universit\'e Claude Bernard Lyon 1}\\}
           \vspace{0.2cm}
           {\textit{Institut Camille Jordan, CNRS UMR 5208}\\}
\vspace{0.4cm}
{(also supported by \textsc{Universit\`a Italo-Francese/Universit\'e  Franco-Italienne)}\\}

\vspace{12pt}

           {XXXIV cycle in the Doctoral Program  \textbf{``Scienze della Terra, Fluidodinamica e Matematica. Interazioni e Metodiche'' } \\
(SSD: MAT/05)\\}
\vspace{0.4cm}
           { Doctoral School \textbf{``INFOMATHS''} (Math\'ematiques)}

\vspace{0.7cm}

           {Co-tutorial thesis submitted on November 30, 2021 for the degree of ``Philosophiæ Doctor'' by~:\\}
           \vspace{0.3cm}
           {\Large\bf {Gabriele SBAIZ}}
\vspace{0.4cm}

\rule{15cm}{1pt}

\vspace{13pt}
   \AddToShipoutPicture*{\put(130,284)%
{\includegraphics[width=17.7cm, height=3.4cm]{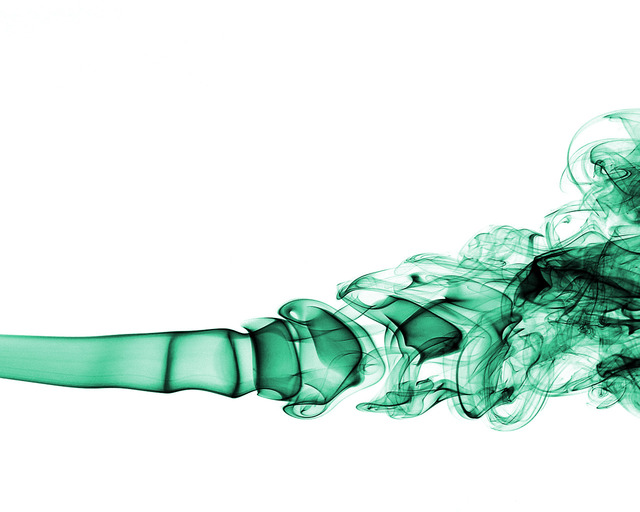}%
}} 
           {\textcolor{blue}{\LARGE \bf Some stability and instability issues in the dynamics of highly rotating fluids}}\\

\vspace{16pt}
\rule{15cm}{1pt}

\vspace{0.5cm}

\end{center}

 {
\begin{tabular}{ll}
\textbf{BERTI Massimiliano} &\textbf{External member}\\
\textit{SISSA} \vspace{0.1cm}\\
\textbf{DALIBARD Anne-Laure} &\textbf{External member}\\
\textit{Sorbonne Université} \vspace{0.1cm}\\        
\textbf{DEL SANTO Daniele} &\textbf{Supervisor}\\
\textit{Universit\`a degli Studi di Trieste} \vspace{0.1cm}\\        
\textbf{FANELLI Francesco}		&\textbf{Co-supervisor}\\
\textit{Universit\'e Claude Bernard Lyon 1} \vspace{0.1cm}\\
\textbf{IFTIMIE Drago\c{s}} &\textbf{Internal member}\\
\textit{Universit\'e Claude Bernard Lyon 1} \vspace{0.1cm}\\        
\textbf{NEČASOVÁ Š\'arka} &\textbf{Referee}\\
\textit{Academy of Sciences of the Czech Republic} \vspace{0.1cm}\\        
\textbf{PAICU Marius-Gheorghe} &\textbf{Referee}\\
\textit{Université de Bordeaux} \vspace{0.1cm}\\        
\textbf{ROSSET Edi} &\textbf{Internal member}\\
\textit{Universit\`a degli Studi di Trieste}        
		\end{tabular}
}

\vspace{0.2cm}

\begin{center}
Academic year 2020/2021
\end{center}
\endgroup


\end{titlepage}


\newpage
\null
\thispagestyle{empty}
\newpage


{\thispagestyle{empty} 

    \null\vfill 

    
    \begin{flushright}
      \textit{``Mathematics is the most beautiful and\\
       most powerful creation of the human spirit.''}\\
      
      \rule{5cm}{1pt}\\
      
       Stefan Banach
    \end{flushright}
    
    \bigskip
    
    \bigskip
    
    \bigskip
    
    \bigskip
    
    \bigskip
    
    \bigskip

    \begin{flushright}
      \textit{``En math\'ematique, c'est comme dans\\
       un roman policier ou un \'episode de Columbo: \\
le raisonnement par lequel le d\'etective \\
confond l'assassin est au moins aussi important\\
 que la solution du myst\`ere elle-même.''}\\
      
      \rule{5cm}{1pt}\\
      
       C\'edric Villani
    \end{flushright}

    \vfill\vfill\vfill\vfill\vfill\vfill\null
    
    \clearpage}
{\thispagestyle{empty}
\null\vfil
\vspace{60pt}
\begin{flushleft}{\Large \sl Dedicate to my extended family,\\
 which supports and stands me everyday.}\end{flushleft}
\vfil\null
}

\let\cleardoublepage\clearpage

\frontmatter
    \thispagestyle{plain}
    \begin{center}{\huge{\textit{Acknowledgements}} \par}\end{center}
    \normalsize
The first thing I would like to do in these pages is to thank my supervisors, my family, my girlfriend, my friends and my colleagues from Trieste and Lyon. 

\medskip

The first person I thank is Daniele Del Santo, my supervisor in Trieste. Daniele has taken care of me for a long time: (unlucky for him!) he has been my supervisor also for the bachelor's and master's thesis. All the time, he advises and supports me in order to get out my best. His long experience (in the life and in the work) has stimulated my research activity and his constant help allows me to overcome not only research issues but also life, administrative and bureaucratic troubles (not so trivial in view of the Italian and French rules!). In any time, he has words of comfort and he has always cheered me up, even in moment of greatest dejection: he finds always time to listen me!

We share a lot of things during these three years of the thesis: the working activity, the life in the laboratory, but even more the adventures/misadventures in missions and conferences.

To sum up, I can say that he is one of the people who has most inspired and influenced me!\\
Many, many thanks!

\medskip

Now, I have to thank my co-supervisor (in Lyon) of the thesis: Francesco Fanelli. He has always followed my research and he has shared lots of ideas to overcome mathematical issues which seemed to be insurmountable. There have been many times I wanted to finish a project early but he disagreed and asked me to do it better. I admire his work ethic (he is a stakhanovite!) and working with him, I learned how to improve quality of the research. He led me in every moment, sharing also his templates to write an article, a cover letter and a letter for a journal. Since writing was a hard task, I had to do it again and again (and again!) and he often advised me: ``you have to work more on it!''. When I was in Lyon, he was always available to discuss not only the progresses in the research. Even now that I am in Trieste, we keep in touch frequently by Zoom calls. 

For all the above reasons, I send him my sincere thanks. 

\medskip

I would also thank Professor Scipio Cuccagna for his useful remarks on Bochner integrals and Calderón–Zygmund theory; Professor Alessandro Fonda for his unmissable help in various occasions; Stefano Scrobogna for his patience and comprehension of my hard period in writing the manuscript; Professor Lucio Torelli for his daily pearls of wisdom and Guido Travaglia for his constant aid in informatics troubles. Moreover, I thank Professor Filippo Santambrogio for his numerous advices in practical situations; L\'eon Matar Tine who is my captain of the ICJ football team, and the members of the ``Comit\'e de Suivi Individuel'', Professors Thomas Lepoutre and Christophe Prange, for their insightful suggestions and remarks. 

\medskip

No less important, the acknowledgements to the referees Professor $\check{\rm S}$\'arka Ne$\check{\rm c}$asov\'a and Professor Marius-Gheorghe Paicu to have accepted to dedicate their time in reading my manuscript, and to Professors Massimiliano Berti, Anne-Laure Dalibard, Drago\c{s} Iftimie and Edi Rosset to make part of the defence committee. 

\medskip

Now, it is the turn of my friends and colleagues. I have to thank my friends forever: Daniele, Lara, Mirko and Susanna. During these three years, they have damped my complains (a lot of complains!) in front of a beer or a cup of tea/herbal tea. Moreover, I have to send many thanks to my colleagues in Trieste especially to Fabio, with whom we have always supported each other to survive in the ``jungle'' of the doctoral program, and to my colleagues in Lyon above all to David and Godfred.

I also have to thank my colleague Martin, who has read carefully the introduction in French and he has suggested changes which have allowed to improve the presentation. 

Together with David and his girlfriend Claudia, we share lots of dinner from Spanish to Italian cuisine passing through the French one. I still remember our cr$\hat{\rm e}$pes with an exaggerated quantity of Nutella. David is a very kind and thoughtful colleague and friend. We are very similar in thinking and viewing the work and the real life: for that reason we have easily discussed and talked on the most disparate arguments. I can say that he is like an ``adopted cousin''. He has only one defect: he is not interested in football at all!

Godfred besides being an officemate is also a friend. We share our thoughts, our preoccupations, our passions and our perspectives (during my year in Lyon) and we have strengthened each other by also going to work in weekends. Moreover, he is my ``English Professor'' since English is his native language.

\medskip

At last but not at least, I have to great thank my family and my girlfriend Miriam. For me it is a very hard task to find only few words to express my gratitude to them: there will never be enough space to thank them! 

My family helps me in any occasion and from all points of view: during the periods of greatest stress they support me both from the psychological viewpoint and in the small things, like making a cup of coffee or tea or just cooking my favourite dishes. I would tell only this anecdote to explain how important and constant has been their aid.
When I decided to move to Lyon for my period at the Institut Camille Jordan, we reached together the city by car: it is 800 km far from my small city in Italy. They stayed few days helping me to clean the apartment, to do the administrative documentation and to find the main services in the new city. Many many thanks!

To conclude this too short paragraph of acknowledgements, I have to thank Miriam. She really was (and she is) next to me in each occasion and she is able to highlight my merits and to obscure my defects. When I say that she was next to me in each occasion, I am not lying or exaggerating. Indeed, when I was in Paris or in Lyon she phoned me every day and we agreed on the choice of the film to watch (virtually) together during the evenings. What to say more? Only, thanks, thanks and thanks a lot!

{
    \vfil\vfil\null
    \clearpage
}
\thispagestyle{plain}
    \begin{center}{\huge{\textit{Abstract}} \par}\end{center}
    \normalsize
In the present thesis, we are interested in the description of the dynamics of flows on large scales, like the atmospheric and ocean currents on the Earth. In this context, the fluids are governed by rotational, weak compressibility and stratification effects, whose importance is ``measured'' by adimensional numbers: the Rossby, Mach and Froude numbers respectively. 
More those three physical parameters are small, more the relative effects are strong.

The first part of the thesis is dedicated to the analysis of a 3-D multi-scale problem called the \emph{full Navier-Stokes-Fourier} system where variations in density and temperature are taken into account and in addition the dynamics is influenced by the action of Coriolis, centrifugal and gravitational forces. We study, in the framework of weak solutions, the combined incompressible and fast rotation limits in the regime of small Mach, Froude and Rossby numbers ($Ma,\, Fr,\, Ro$ respectively) and for general ill-prepared initial data.
In the so-called multi-scale regime where some effect is predominant in the motion, precisely when the Mach number is of higher order than the Rossby number, we prove that the limit dynamics is described by an incompressible Oberbeck-Boussinesq type system. It is worth noticing that the velocity
field is purely horizontal at the limit (according to the so-renowned Taylor-Proudman theorem in geophysics), but surprisingly vertical effects on the temperature equation appear. These stratification effects are completely absent when $Fr$ exceeds $\sqrt{Ma}$, whereas they suddenly come into play as soon as one reaches the endpoint scaling $Fr=\sqrt{Ma}$.

Conversely, when the Mach and Rossby numbers have the same order of magnitude (the isotropic scaling), and in absence of the centrifugal force, we show convergence towards a quasi-geostrophic type equation for a stream-function
of the limit velocity field, coupled with a transport-diffusion equation for a quantity that mixes the target density and temperature profiles.

Following \textit{``le fil rouge''} of the asymptotic analysis, in the second part of the thesis, we examine the effects of high rotation (small Rossby number) for the 2-D incompressible density-dependent Euler system. With respect to the previous problem, now we deal with an incompressible system with a hyperbolic structure, where the viscosity effects are neglected. More precisely, the main goal is to perform the singular limit in the fast rotation regime, showing the convergence of the Euler equations to a quasi-homogeneous type system. The limit system is a coupled system of a transport equation for the density and a momentum equation for the velocity with a non-linear term of lower order, which combines the effects of fluctuations of the density and the velocity field. For the convergence process, a core point is to develop uniform (with respect to $Ro$) estimates in high regularity norms not to deteriorate the lifespan of solutions. Moreover, as a sub-product of the local well-posedness analysis (recall that the global existence of solutions is an open problem even in 2-D), we find an \textit{``asymptotically global''} well-posedness result: for small densities, the lifespan of solutions to the primitive and limiting systems tend to infinity. 

The proof of convergence of the two primitive problems (the Navier-Stokes-Fourier system and the Euler system, respectively) towards the reduced models is based on a compensated compactness argument. The key point is to use the structure of the underlying system of Poincar\'e waves in order to identify some compactness properties for suitable quantities.
Compared to previous results, our method enables to treat the whole range of parameters in the multi-scale problem, and also
to reach and go beyond the \emph{``critical''} choice $Fr\,=\,\sqrt{Ma}$.

\paragraph*{\small Keywords:} {\footnotesize Navier-Stokes-Fourier system; barotropic Navier-Stokes-Coriolis system; density-dependent incompressible Euler system; Coriolis force; gravity; stratification effects; singular perturbation problem;  multi-scale limit; low Mach, Froude and Rossby numbers; compensated compactness.}
{
    \vfil\vfil\null
    \clearpage
}
\thispagestyle{plain}
    \begin{center}{\huge{\textit{R\'esum\'e}} \par}\end{center}
    \normalsize

Dans cette th\`ese, nous nous int\'eressons \`a la description de la dynamique des fluides \`a grande \'echelle, comme les courants atmosph\'eriques et oc\'eaniques sur la plan\`ete Terre. Dans ce contexte, les fluides sont dirig\'es par des effets de rotation, de faible compressibilit\'e et de stratification, dont l'importance est ``mesurée'' par des nombres adimensionnels: respectivement les nombres de Rossby, Mach et Froude.
Plus ces trois paramètres physiques sont petits, plus les relatifs effets sont importants.


La premi\`ere partie de la th\`ese est ensuite consacr\'ee à l'analyse d'un problème multi-\'echelle 3-D appel\'e le syst\`eme de \emph{Navier-Stokes-Fourier complet} o\`u les variations de densit\'e et de temp\'erature sont prises en compte et en plus la dynamique est influenc\'ee par l'action de la force de Coriolis et des forces centrifuge et gravitationnelle. Nous \'etudions, dans le cadre des solutions faibles, la limite incompressible et à rotation rapide dans le r\'egime des petits nombres de Mach, Froude et Rossby ($Ma,\, Fr,\, Ro$ respectivement) et pour des donn\'ees initiales g\'en\'erales mal pr\'epar\'ees.
Dans le r\'egime appel\'e multi-\'echelles o\`u un effet est pr\'edominant dans le mouvement, pr\'ecis\'ement lorsque le nombre de Mach est d'ordre sup\'erieur au nombre de Rossby, nous montrons que la dynamique limite est d\'ecrite par un syst\`eme incompressible de type Oberbeck-Boussinesq. Il est \`a noter que le champ de vitesse est purement horizontal \`a la limite (selon le th\'eor\`eme si renomm\'e de Taylor-Proudman en g\'eophysique), mais \'etonnamment des effets verticaux apparaissent dans l'\'equation de la temp\'erature. Ces effets de stratification sont totalement absents lorsque $Fr$ d\'epasse $\sqrt{Ma}$, alors qu'ils entrent en jeu imm\'ediatement quand on considère l’\'echelle critique $Fr=\sqrt{Ma}$.

\`A l'inverse, lorsque les nombres de Mach et Rossby ont le même ordre de grandeur (l'échelle appelée isotrope), et en absence de la force centrifuge, on montre la convergence vers une équation de type quasi-géostrophique pour une fonction de flux liée au champ de vitesse limite, couplée à une équation de transport-diffusion pour une quantité qui mélange les profils limites de densité et de température.

En suivant \textit{le fil rouge} de l'analyse asymptotique, dans la deuxième partie de la thèse, nous examinons les effets de la rotation rapide (petit nombre de Rossby) pour le système d'Euler incompressible 2-D dépendant de la densité. Par rapport au problème précédent, maintenant nous sommes en présence d'un système incompressible et avec une structure hyperbolique où les effets de viscosité sont négligés. Plus précisément, l'objectif principal est d'effectuer la limite singulière dans le régime de rotation rapide, montrant la convergence des équations d'Euler vers un système de type quasi-homogène. Le système limite est un système couplé d'une équation de transport pour la densité et d'une équation de quantité de mouvement pour la vitesse avec un terme non linéaire d'ordre inférieur, qui combine les effets des fluctuations de la densité avec le champ de vitesse. Pour atteindre ce but, un point central est de développer des estimations uniformes (par rapport à $Ro$) dans des normes de haute régularité, pour ne pas détériorer la durée de vie des solutions. De plus, en tant que sous-produit de l'analyse du caractère bien posé local (rappelons que l'existence globale de solutions est un problème ouvert même en 2-D), nous trouvons un résultat de caractère bien posé \textit{``asymptotiquement globale''}: pour des petites densités, la durée de vie des solutions des systèmes primitif et limite tend vers l'infini.

La preuve de la convergence des deux problèmes primitifs (respectivement le système de Navier-Stokes-Fourier et le système d'Euler) vers les modèles réduits est basée sur un argument de compacité par compensation. Le point clé est d'utiliser la structure du système sous-jacent, appelé système d'ondes de Poincar\'e, afin d'identifier certaines propriétés de compacité pour des quantités appropriées.
Par rapport aux résultats précédents, notre méthode permet de traiter l'ensemble des paramètres du problème multi-échelles, et aussi pour atteindre et dépasser le choix \emph{``critique''} $Fr\,=\,\sqrt{Ma}$.

\paragraph*{\small Mot cl\'es:} {\footnotesize syst\`eme de Navier-Stokes-Fourier; syst\`eme barotrope de Navier-Stokes-Coriolis; syst\`eme d'Euler incompressible d\'ependant de la densit\'e; force de Coriolis; gravit\'e; effets de stratification; probl\`eme de perturbation singuli\`ere ; limite multi-\'echelles; faibles nombres de Mach, Froude et Rossby; compacit\'e par compensation.}
{
    \vfil\vfil\null
    \clearpage
}

\tableofcontents
\setcounter{secnumdepth}{3}

\newpage

\chapter{Notation and conventions}

Let $\mc B\subset\R^d$ with $d\geq 2$. Throughout the whole thesis, the symbol $\bbbone_{\mc B}$ denotes the characteristic function of $\mc B$.
The notation $C_c^\infty (\mc B)$ stands for the space of $\infty$-times continuously differentiable functions on $\R^d$ and having compact support in $\mc B$. The dual space $\mc D^{\prime}(\mc B)$ is the space of
distributions on $\mc B$. We use also the notation $C^0_w([0,T];\mc X)$, with $\mc X$ a Banach space, to refer to the space of continuous in time functions with values in $\mc X$ endowed with its weak topology. \\
Given $p\in[1,+\infty]$, by $L^p(\mc B)$ we mean the classical space of Lebesgue measurable functions $g$, where $|g|^p$ is integrable over the set $\mc B$ (with the usual modifications for the case $p=+\infty$).
We use also the notation $L_T^p(L^q)$ to indicate the space $L^p\big([0,T];L^q(\mc B)\big)$, with $T>0$.
Given $k \geq 0$, we denote by $W^{k,p}(\mc B)$ the Sobolev space of functions which belongs to $L^p(\mc B)$ together with all their derivatives up to order $k$. When $p=2$, we alternately use the
notation $W^{k,2}(\mc B)$ and  $H^k(\mc B)$.
We denote by $\dot{W}^{k,p}(\mc B)$ the corresponding homogeneous Sobolev spaces, i.e. $\dot{W}^{k,p}(\mc B) = \{ g \in L^1_{\rm loc}(\mc B)\, : \, D^\alpha g \in L^p(\mc B),\ |\alpha| = k \}$.
Recall that $\dot{W}^{k,p}$ is the completion of $C^\infty_c(\overline{\mc B})$ with respect to the $L^p$ norm of the $k$-th order derivatives. Moreover, the notation $B^s_{p,r}(\mc B)$ stands for the Besov spaces on $\mc B$ that are interpolation spaces between the Sobolev ones. \\
The symbol $\mc{M}^+(\mc B)$ denotes the cone of non-negative Borel measures on $\mc B$. For the sake of simplicity, we will omit from the notation the set $\mc B$, that we will explicitly point out if needed.

In the whole thesis, the symbols $c$ and $C$ will denote generic multiplicative constants, which may change from line to line, and which do not depend on the small parameter $\veps$.
Sometimes, we will explicitly point out the quantities which these constants depend on, by putting them inside brackets.\\
In addition, we agree to write $f\sim g$ whenever we have $c\, g\leq f \leq C\, g$, and $f\lesssim g$ if $f\leq Cg$.

Let $\big(f_\veps\big)_{0<\veps\leq1}$ be a family of functions in a normed space $Y$. If this family of functions is bounded in $Y$,  we use the notation $\big(f_\veps\big)_{\veps} \subset Y$.
 
\medbreak
As we will see in the sequel (we refer in particular to Chapters \ref{chap:multi-scale_NSF} and \ref{chap:BNS_gravity}), one of the main features of our asymptotic analysis is that the limit-flow
will be \emph{two-dimensional} and \emph{horizontal} along the plane orthogonal to the rotation axis.
Then, let us introduce some notation to describe better this phenomenon.

Let $\Omega$ be a domain in $\R^3$. We decompose $\vec x\in\Omega$
into $\vec x=(x^h,x^3)$, with $ x^h\in\R^2$ denoting its horizontal component. Analogously,
for a vector-field $\vec v=(v^1,v^2,v^3)\in\R^3$, we set $\vec v^h=(v^1,v^2)$ and we define the differential operators
$\nabla_h$ and $\div_{\!h}$ as the usual operators, but acting just with respect to $x^h$.
In addition, we define the operator $\nabla^\perp_h\,:=\,\bigl(-\d_2\,,\,\d_1\bigr)$.

Finally, we introduce the Helmholtz projection $\mbb{H}[\vec{v}]$ of a vector field $\vec{v}\in L^p(\Omega; \R^3)$ on the subspace of divergence-free vector fields. It is defined by 
the decomposition
	\begin{equation*}
	\vec{v} = \mbb{H}[\vec{v}] + \nabla_x \Psi\, ,
	\end{equation*}
where $\Psi \in \dot W^{1,p}(\Omega)$ is the unique solution of 
	$$\int_{\Omega} \nabla_x \Psi \cdot \nabla_{x} \varphi \dx = \int_{\Omega} \vec{v} \cdot \nabla_x \varphi  \dx
	\quad\mbox{for all } \varphi \in C^\infty_c (\overline{\Omega}),$$
which formally means: $\Delta \Psi = \div \vec{v}$ and 
$\vec{v}  \cdot \n |_{ \partial \Omega } =  0$.

The symbol $\h_h$ denotes instead the Helmholtz projection on $\R^2$.
Observe that, in the sense of Fourier multipliers, one has $\h_h\vec f\,=\,-\nabla_h^\perp(-\Delta_h)^{-1}\curlh\vec f$.

Moreover, since we will deal with a periodic problem in the $x^{3}$-variable, we also introduce the following decomposition: for a vector-field $X$, we write
\begin{equation} \label{eq:decoscil}
X(x)=\langle X\rangle (x^{h})+\dbtilde{X}(x)\quad\qquad
 \text{ with }\quad \langle X\rangle(x^{h})\,:=\,\frac{1}{\left|\T\right|}\int_{\T}X(x^{h},x^{3})\, \dx^{3}\,, \tag{OSC}
\end{equation}
where $\mbb{T}^1\,:=\,[-1,1]/\sim$ is the one-dimensional flat torus (here $\sim$ denotes the equivalence relation which identifies $-1$ and $1$)
and $\left|\T\right|$ denotes its Lebesgue measure.
Notice that $\dbtilde{X}$ has zero vertical average, and therefore we can write $\dbtilde{X}(x)=\d_{3}\dbtilde{Z}(x)$ with $\dbtilde{Z}$ having zero vertical average as well.

\chapter{Contributions of the thesis} 

\section*{The Navier-Stokes-Fourier problem: some physical insight}
In this thesis, we devote ourselves to the study of the behaviour of fluid flows characterized by large time and space scales. Typical examples of those flows are currents in the atmosphere and the ocean, but of course there are
many other cases where such fluids occur out of the Earth, like flows on stars or other celestial bodies.
At those scales, the effects of rotation of the ambient space (which in the case of oceans or atmosphere is the Earth) are not negligible, and the fluid motion is influenced by the action of a strong Coriolis force. There are  two other features that characterize the dynamics of these flows, usually called geophysical flows (see \cite{C-R}, \cite{Ped} and \cite{Val}, for instance): the compressibility or incompressibility of the fluid and the stratification effects (i.e. density variations, essentially due to the gravity).
The relevance of the previous attributes is ``measured'' by introducing, in the mathematical model, three positive adimensional parameters which, for the geophysical flows, are assumed to be small. Those parameters are:
\begin{itemize}
\item the \emph{Mach} number $Ma$, which sets the size of isentropic departures from incompressible flows: the more $Ma$ is small, the more compressibility effects are low; 
\item the \emph{Froude} number $Fr$, which measures the importance of the stratification effects in the dynamics: the more $Fr$ is small, the more gravitational effects are strong;
\item the \emph{Rossby} number $Ro$, that is related to the rotation of the ambient system: when $Ro$ is very small, the effects of the fast rotation are predominant in the dynamics.
\end{itemize}
We adopt a simplistic assumption (often assumed in physical and mathematical studies) which consists in restricting the attention to flows at mid-latitudes, i.e. flows which take place far enough from the poles and the equatorial zone. In this context, the variations of rotational effects due to the latitude are negligible.

Denote by $\vrho ,\, \vtheta\geq 0$ the density and the absolute temperature of the fluid, respectively, and by $\vec{u}\in \mathbb{R}^3$ its velocity field: the full 3-D Navier-Stokes-Fourier system in its
non-dimensional form, can be written (see e.g. \cite{F-N}) as
\begin{equation} \label{eq_i:NSF}
\begin{cases}
	\partial_t \vrho + \div (\vrho\vec{u})=0\  \\[3ex]
	\partial_t (\vrho\vec{u})+ \div(\vrho\vec{u}\otimes\vec{u}) + \dfrac{\e_3 \times \vrho\vec{u}}{Ro}\,  +    \dfrac{1}{Ma^2} \nabla_x p(\vrho,\vtheta) \\[1ex]
	\qquad \qquad \qquad \qquad \qquad \qquad \qquad \; \; \; =\div \mbb{S}(\vtheta,\nabla_x\vec{u}) + \dfrac{\vrho}{Ro^2} \nabla_x F  + \dfrac{\vrho}{Fr^2} \nabla_x G  \\[3ex]
	 \partial_t \bigl(\vrho s(\vrho, \vtheta)\bigr)  + \div \bigl(\vrho s (\vrho,\vtheta)\vec{u}\bigr) + \div\left(\dfrac{\q(\vtheta,\nabla_x \vtheta )}{\vtheta} \right)
	= \sigma\,, \tag{NSF}
\end{cases}
\end{equation}
which is setted in the infinite straight 3-D strip: 
\begin{equation} \label{eq:domain}
\Omega\,:=\,\R^2\times\,]0,1[\,. \tag{DOM}
\end{equation} 
In system \eqref{eq_i:NSF} above, the functions $s,\vec{q},\sigma$ are the specific entropy, the heat flux and the
entropy production rate respectively, and $\mbb{S}$ is the viscous stress tensor, which satisfies Newton's rheological law (see Subsections \ref{sss:primsys} and \ref{sss:structural} for the more precise formulation).

The Coriolis force is represented by 
\begin{equation} \label{def:Coriolis}
\mf C(\vr,\vu)\,:=\,\frac{1}{Ro}\,\vec e_3\times\vr\,\vu\,, \tag{COR}
\end{equation}
where $\vec e_3=(0,0,1)$ and the symbol $\times$ stands for the classical external product of vectors in $\R^3$. In particular, the previous definition implies that the rotation takes place around the vertical axis, and its strength does not depend on the latitude (see e.g. \cite{C-R} and \cite{Ped} for details). We point out that, despite all those simplifications, the obtained model is already able to capture several physically relevant phenomena occurring in the dynamics of geophysical flows: the so-called \emph{Taylor-Proudman theorem}, the formation
of \emph{Ekman layers} and the propagation of \emph{Poincar\'e waves}. We refer to \cite{C-D-G-G} for a more in-depth discussion. 
In the present thesis, we avoid boundary layer effects, i.e. the issue linked to the Ekman layers, by imposing \emph{complete-slip} boundary conditions.

As established by the \emph{Taylor-Proudman theorem} in geophysics, the fast rotation imposes a certain rigidity/stability, forcing the motion to take place on planes orthogonal to the rotation axis.
Therefore, the dynamics becomes purely two-dimensional and horizontal, and the fluid tends to move in vertical columns.

However, such an ideal configuration is hindered by another fundamental force acting at geophysical scales, the gravity, which works to restore vertical stratification of the density. The gravitational force is described in system \eqref{eq_i:NSF} by the term
\[
 \mc G(\vr)\,:=\,\frac{1}{Fr^2}\,\vr\,\nabla_xG\, ,
\]
where in our case $G(x)=\,G(x^3)\,=\,-\,x^3$. Moreover, the gravitational effects are weakened by the presence of the centrifugal force 
$$\mf F(\vr):=\frac{1}{Ro^2} \, \vr\, \nabla_x F\, , $$
 with $F(x)=|x^h|^2$. Such force is an inertial force that, at mid-latitude, slightly shifts the direction of the gravity. 

Thus, the competition between the stabilisation consequences, due to the rotational effects, and the vertical stratification (due to gravity), is translated in the model into the competition between the orders of magnitude of $Ro$ and $Fr$. 

Actually, it turns out that the gravity $\mc G$ acts in combination with the pressure force: 
$$
\mf P(\vr, \vtheta)\,:=\,\frac{1}{Ma^2}\,\nabla_x p(\vr, \vtheta)\,,
$$
where $p$ is a known smooth function of the density and the temperature of the fluid (see Subsection \ref{sss:structural}).

We notice the fact that the terms $\mf C,\, \mc G,\, \mf P$ and $\mf F$ enter into play in the model with a large prefactor, therefore our aim is to study the systems when $Ma, \, Fr$ and $Ro$ are small in different regimes. 
\section*{The multi-scale analysis}

At the mathematical level, in the last 30 years there has been a huge amount of works devoted to the rigorous justification, in various functional frameworks,
of the reduced models considered in geophysics. 

Reviewing the whole literature about this topic goes far beyond the scopes of this introductory part, therefore we make the choice of reporting only the works which
deal with the presence of the Coriolis force \eqref{def:Coriolis}.
We also decide to postpone, to the next part, the discussion about the incompressible models, because less pertinent for multi-scale analysis, due to the rigidity imposed by the divergence-free constraint
on the velocity field of the fluid.

The framework of compressible fluid models, instead, provides a much richer setting for the multi-scale analysis of geophysical flows.
In addition, we choose to focus our attention mostly on works dealing with viscous fluids and which perform the asymptotic study for general
ill-prepared initial data.

\subsection*{Previous results}
First results in the above direction were obtained by Feireisl, Gallagher and  Novotn\'y in \cite{F-G-N} and together with G\'erard-Varet in \cite{F-G-GV-N}, for the barotropic Navier-Stokes system (see also \cite{B-D-GV} for a preliminary study and \cite{G-SR_Mem} for the analysis of equatorial waves). There, the authors investigated the combined low Rossby number regime (fast rotation effects) with low Mach number regime (weak compressibility of the fluid) under the scaling
\begin{align} 
Ro\,&=\,\veps\tag{LOW RO}\\
Ma\,&=\,\veps^m \qquad\qquad \mbox{ with }\quad m\geq 0
\,, \tag{LOW MA}\label{eq:scale}
\end{align}
where $\veps\in\,]0,1]$ is a small parameter, which one lets go to $0$ in order to derive the reduced model. In the case when $m=1$ in \eqref{eq:scale}, the system presents an isotropic scaling, since $Ro$ and $Ma$ act at the same order of magnitude and the pressure and rotation terms keep in balance (the so-called \emph{quasi-geostrophic balance}) at the limit. The limit system
is identified as the so-called \emph{quasi-geostrophic equation} for a stream-function of the target velocity field.
In \cite{F-G-GV-N} when $m>1$ and with in addition the centrifugal force, instead, the pressure term predominates (over the Coriolis force) in the dynamics of the fluid. 
In this case, the limit system is described by a $2$-D incompressible Navier-Stokes system and the difficulties generated by the anisotropy of scaling are overcome by using dispersive estimates.

Afterwards, Feireisl and Novotn\'y continued the multi-scale analysis for the same system without the centrifugal force term yet, by considering
the effects of a low stratification, i.e. $Ma/Fr\rightarrow 0$ when $\veps\rightarrow 0^+$ (see \cite{F-N_AMPA}, \cite{F-N_CPDE}).
We refer to \cite{F_MA} for a similar study in the context of capillary models, where the choice $m=1$ was made, but the anisotropy was given by the scaling
fixed for the internal forces term (the so-called Korteweg stress tensor). In addition, we have to mention \cite{F_2019} for the case of large Mach numbers with respect to the Rossby parameter,
namely $0\leq m<1$ in \eqref{eq:scale}. Since, in that instance, the pressure gradient is not strong enough to compensate the Coriolis force,
in order to find some relevant limit dynamics one has to penalise the bulk viscosity coefficient.

The analysis for models presenting also heat transfer is much more recent, and has begun with the work \cite{K-M-N} by Kwon, Maltese and Novotn\'y. In that paper, the authors considered a multi-scale
problem for the full Navier-Stokes-Fourier system with Coriolis and gravitational forces ($F=0$ therein), taking the scaling
\begin{equation} \label{eq:scale-G}
{Fr}\,=\,\veps^n\,,\qquad\qquad\mbox{ with }\quad 1\,\leq\,n\,<\,\frac{m}{2}\,.\tag{LOW FR}
\end{equation}
In particular, in that paper, the choice \eqref{eq:scale-G} implied that $m>2$ and the case $n=m/2$ was left open. Similar restrictions on the parameters can be found in \cite{F-N_CPDE}
for the barotropic model. Such restrictions has to be ascribed to the techniques used for proving convergence, which are based on a combination of relative energy/relative entropy method with dispersive estimates derived from oscillatory integrals
(notice that an even larger restriction, $m>10$, appears in \cite{F-G-GV-N}). On the other hand, it is worth underlying that the relative energy methods allow to get a precise rate of convergence and to consider also inviscid and non-diffusive limits (in those cases, one does not dispose of a uniform bound
for $\nabla_x\vtheta$ and on $\nabla_x\vec{u}$).
The case when $m=1$ was handled in the subsequent work \cite{K-N} by Kwon and Novotn\'y, resorting to similar techniques
(however, the gravitational term is not penalised at all).
\subsection*{Novelties}
The first part of this thesis is devoted to the analysis of multi-scale problems,
by focusing on the full Navier-Stokes-Fourier system introduced in \eqref{eq_i:NSF}.
In a first instance, we improve the choice of scaling \eqref{eq:scale-G} taking the endpoint case $n=m/2$ with $m\geq1$ (this is the scaling adopted throughout the Chapter \ref{chap:multi-scale_NSF}). Of course, we are still in a regime of low stratification, since $Ma/Fr\ra0$, but having $Fr=\sqrt{Ma}$ 
allows us to capture some additional qualitative properties on the limit dynamics.
In addition, we add to the system the centrifugal force term $\nabla_x F$ (in the spirit of \cite{F-G-GV-N}), which is a source of technical troubles, due to its unboundedness.
Let us now comment all these issues in detail.

First of all, in absence of the centrifugal force, namely when $F=0$, we are able to perform incompressible, mild stratification and fast rotation limits for the \emph{whole range}
of values of $m\geq 1$, in the framework of \emph{finite energy weak solutions} to the Navier-Stokes-Fourier system \eqref{eq_i:NSF} and for general \emph{ill-prepared initial data}.
In the case $m>1$, the incompressibility and stratification effects are predominant with respect to the Coriolis force: then we prove convergence to the well-known \emph{Oberbeck-Boussinesq system} (see for instance Paragraph 1.6.2 of \cite{Z} for physical insights about that system),
giving a rigorous justification to this approximate model in the context of fast rotating fluids. Thus, we can state the following theorem (see Theorem \ref{th:m-geq-2} for the accurate statement). 
\begin{theoremnonum}\label{thm_1}
Consider system \eqref{eq_i:NSF}. Let $\Omega = \R^2 \times\,]0,1[\,$. Let $F=|x^{h}|^{2}$ and $G=-x^{3}$. Take $n=m/2$ and either ${m\geq 2}$, or ${m>1}$ and ${F=0}$. Then, 
	\begin{align*}
	\varrho_\ep \rightarrow 1  \\
	R_{\ep}:=\frac{\varrho_\ep - 1}{\ep^m}  \weakstar R  \\
	\vec{u}_\ep \weak \vec{U} \\
	\Theta_{\ep}:=\frac{\vartheta_\ep - \bar{\vartheta}}{\ep^m}  \weak \Theta \, ,
	\end{align*}
where in accordance to the Taylor-Proudman theorem, one has 
$$\vec{U} = (\vec U^h,0),\quad \quad \vec U^h=\vec U^h(t,x^h),\quad \quad \div_h\vec U^h=0.$$
Moreover, $\Big(\vec{U}^h,\, \, R ,\, \, \Theta \Big)$ solves, in the sense of distributions, the incompressible Oberbeck-Boussinesq type system
\begin{align*}
& \d_t \vec U^{h}+\div_{h}\left(\vec{U}^{h}\otimes\vec{U}^{h}\right)+\nabla_h\Gamma-\mu (\oline\vtheta )\Delta_{h}\vec{U}^{h}=\delta_2(m)\langle R\rangle\nabla_{h}F \\[2ex]& 
\d_t\Theta\,+\,\divh(\Theta\,\vec U^h)\,-\,\kappa(\oline\vtheta)\,\Delta\Theta\,=\,\oline\vtheta\,\vec{U}^h\cdot\nabla_h \overline{\mc G}
\\[2ex]
& \nabla_{x}\left( \d_\varrho p(1,\oline{\vtheta})\,R\,+\,\d_\vtheta p(1,\oline{\vtheta})\,\Theta \right)\,=\,\nabla_{x}G\,+\,\delta_2(m)\,\nabla_{x}F\, ,
\end{align*}
where
$\overline{\mc G}$ is the sum of external force $\,G\,+\,\delta_2(m)F$, $\Gamma \in \mc D^\prime$ and $\delta_2(m)= 1$ if $m=2$, $\delta_2(m)=0$ otherwise.
\end{theoremnonum}

We point out that the target velocity field is $2$-dimensional, according to the celebrated Taylor-Proudman theorem in geophysics: in the limit of high rotation, the fluid motion tends to have a planar behaviour, it takes place on planes orthogonal to the rotation axis (i.e. horizontal planes in our model) and is essentially constant along the vertical direction. We refer to
\cite{C-R}, \cite{Ped} and \cite{Val} for more details on the physical side. Notice however that, although the limit dynamics is purely horizontal, the limit density and temperature variations,
$R$ and $\Theta$ respectively, appear to be stratified: this is the main effect of taking $n=m/2$ for the Froude number in  \eqref{eq:scale-G}.
This is also the main qualitative property which is new here, with respect to the previous studies, and justifies the epithet of \emph{``critical''} scaling.

When $m=1$, instead, all the forces act at the same scale, and then they balance each other asymptotically for $\veps\ra0^+$. As a result, the limit motion is described by the so-called
\emph{quasi-geostrophic equation} for a suitable function $q$, which is linked to $R$ and $\Theta$ (respectively, the target density and temperature variations) and to the gravity,
and which plays the role of a stream-function for the limit velocity field. This quasi-geostrophic equation is coupled with a scalar transport-diffusion equation for a new quantity
$\Upsilon$, mixing  $R$ and $\Theta$. The precise statement of the following theorem can be found in Paragraph \ref{ss:results}.
\begin{theoremnonum}
Consider system \eqref{eq_i:NSF}. Let $\Omega = \R^2 \times\,]0,1[\,$. Let ${F=0}$ and $G=-x^3$. Take ${m=1}$ and $n=1/2$. 
Then, one has the same convergences found in Theorem \ref{thm_1} and $\vec U$ satisfies the Taylor-Proudman theorem. 

We define
$$
\Upsilon := \d_\vrho s(1,\oline{\vtheta}) R + \d_\vtheta s(1,\oline{\vtheta})\,\Theta$$ 
$$q= \d_\varrho p(1,\oline{\vtheta}) R  +\d_\vtheta p(1,\oline{\vtheta})\Theta -G\, ,
$$
then $q=q(t,x^h)$ and $ \vec{U}^{h}=\nabla_h^{\perp} q$.
Moreover, the couple $\Big(q,\, \, \Upsilon \Big)$ satisfies, in the sense of distributions,  
\begin{align*}
& \d_{t}\left(q-\Delta_{h}q\right) -\nabla_{h}^{\perp}q\cdot
\nabla_{h}\left( \Delta_{h}q\right) +\mu (\oline{\vtheta})
\Delta_{h}^{2}q=\langle X\rangle  \\[2.5ex]
& \d_{t} \Upsilon +\nabla_h^\perp q\cdot\nabla_h\Upsilon-\kappa(\oline\vtheta) \Delta \Upsilon\,=\,
\,\kappa(\oline\vtheta)\,\Delta_hq\, ,
\end{align*}
where $\langle X\rangle$ is a suitable ``external'' force.
\end{theoremnonum}
This is in the spirit of the result in \cite{K-N}, but once again, here we capture also gravitational effects in the limit, so that we cannot say
anymore that $R$ and $\Theta$ (and then $\Upsilon$) are horizontal; on the contrary, and somehow surprisingly, $q$ and the target velocity $\vec U$ are purely horizontal.

At this point, let us make a couple of remarks. First of all, we mention that, as announced above, we are able to add to the system the effects of the centrifugal force $\nabla_x F$.
Unfortunately, in this case the restriction $m\geq 2$ appears (which is still less severe than the ones imposed in \cite{F-G-GV-N}, \cite{F-N_CPDE} and \cite{K-M-N}).
However, we show that such a restriction is not of technical nature, but it is hidden in the structure of the wave system (see Proposition \ref{p:target-rho_bound} and Remark \ref{slow_rho}).
The result for $F\neq 0$ is analogous to the one presented above for the case $F=0$ and $m>1$: when $m>2$, the anisotropy of scaling is too large in order to see any effect due to $F$ in the limit, and no qualitative differences appear with respect to the instance when $F=0$; when $m=2$, instead, additional terms, related to $F$, appear in the Oberbeck-Boussinesq system (see Theorem \ref{thm_1}). In any case, the analysis will be considerably more complicated,
since $F$ is not bounded in $\Omega$ (defined in \eqref{eq:domain} above) and this will demand an additional localisation procedure (already employed in \cite{F-G-GV-N}).

We also point out that the classical existence theory of finite energy weak solutions for \eqref{eq_i:NSF} requires the physical domain to be a smooth \emph{bounded} subset of $\R^3$
 (see \cite{F-N} for a comprehensive study). The theory was later extended in \cite{J-J-N} to cover the case of unbounded domains, and this might appear suitable for us in view of \eqref{eq:domain}.
Nonetheless, the notion of weak solutions developed in \cite{J-J-N} is somehow milder than the classical one (the authors speak in fact of \emph{very weak solutions}), inasmuch as the
usual weak formulation of the entropy balance, i.e. the third equation in \eqref{eq_i:NSF}, has to be replaced by an inequality in the sense of distributions.
Now, such a formulation is not convenient for us, because, when deriving the system of acoustic-Poincar\'e waves, we need to combine the mass conservation
and the entropy balance equations together. In particular, this requires to have true equalities, satisfied in the (classical) weak sense.
In order to overcome this problem, we resort to the technique of \emph{invading domains} (see e.g. Chapter 8 of \cite{F-N}, \cite{F-Scho} and \cite{WK}):
namely, for each $\veps\in\,]0,1]$, we solve system \eqref{eq_i:NSF}, with the choice $n=m/2$ for the Froude number, in a smooth bounded domain $\Omega_\veps$, where $\big(\Omega_\veps\big)_\veps$ converges
(in a suitable sense) to $\Omega$ when $\veps\ra0^+$, with a rate higher than the wave propagation speed (which is proportional to $\veps^{-m}$).
Such an ``approximation procedure'' will need some extra work.

In order to prove our results, and get the improvement on the values of the different parameters, we propose a unified approach, which actually works both for the case $m>1$
(allowing us to treat the anisotropy of scaling quite easily) and for the case $m=1$ (allowing us to treat the more complicate singular perturbation operator).
This approach is based on \emph{compensated compactness} arguments, firstly employed by Lions and Masmoudi in \cite{L-M} for dealing with the incompressible limit of the barotropic
Navier-Stokes equations, and later adapted  by Gallagher and Saint-Raymond in \cite{G-SR_2006} to the case of fast rotating (incompressible homogeneous) fluids. More recent applications
of that method in the context of geophysical flows can be found in \cite{F-G-GV-N}, \cite{F_JMFM}, \cite{Fan-G} and \cite{F_2019}.

The quoted method does not give a quantitative convergence at all, but only qualitative. The technique is purely based on the algebraic structure of the system, which allows to find smallness (and vanishing to the limit) of suitable non-linear quantities, and fundamental compactness properties for other quantities.
These strong convergence properties are by no means evident, because the singular terms are responsible for strong oscillations in time (the so-called acoustic-Poincar\'e waves) of the solutions, which may finally prevent the convergence of the non-linearities.
Nonetheless, a fine study of the system for acoustic-Poincar\'e waves actually reveals compactness (for any $m\geq1$ if $F=0$, for $m\geq 2$ if $F\neq 0$) of a special quantity
$\g_\veps$, which combines (roughly speaking) the vertical averages of the momentum $\vec{V}_\veps=\vrho_\veps\vec{u}_\veps$ (of its vorticity, in fact) and of another function
$Z_\veps$, obtained as a linear combination of density and temperature variations (see Subsections \ref{sss:term1} and \ref{ss:convergence_1} for more details in this respect). Similar compactness properties have been highlighted in \cite{Fan-G} for incompressible density-dependent fluids
in $2$-D, and in \cite{F_2019} for treating a multi-scale problem at ``large'' Mach numbers.
In the end, the strong convergence of $\big(\g_\veps\big)_\veps$ turns out to be enough to take the limit in the convective term,
and to complete the proof of our results.

To conclude this part, let us mention that we expect the same technique to enable us to treat also the case $m=1$ and $F\neq0$ (this was the case in \cite{F-G-GV-N}, for barotropic flows).
Nonetheless, the presence of heat transfer deeply complicates the wave system, and new technical difficulties arise 
in the analysis of the convective term (the approach of \cite{F-G-GV-N}, in the case of constant temperature, does not work here).
For that reason, here we are not able to handle that case, which still remains open.

Another feature that remains uncovered in our analysis is the \emph{strong stratification} regime, namely when the ratio ${Ma}/{Fr}$ is of order $O(1)$. This regime is particularly delicate for fast rotating fluids.
This is in stark contrast with the results available about the derivation of the anelastic approximation, where rotation is neglected:
we refer e.g. to \cite{Masm}, \cite{BGL}, \cite{F-K-N-Z} and, more recently, \cite{F-Z} (see also \cite{F-N} and references therein for a
more detailed account of previous works). The reason for that has to be ascribed exactly to the competition between vertical stratification (due to gravity) and horizontal stability (which the Coriolis force tends to impose): in the strong stratification regime, vertical oscillations of the solution
(seem to) persist in the limit, and the available techniques do not allow at present to deal with this problem in its full generality.
Nonetheless, partial results have been obtained in the case of well-prepared initial data, by means of a relative entropy method: we refer to \cite{F-L-N}
for the first result, where the mean motion is derived, and to \cite{B-F-P} for an analysis of Ekman boundary layers in that framework.

\section*{Going beyond the critical scaling $Fr=\sqrt{Ma}$}
At this point, we are interested in going beyond the critical choice $Fr=\sqrt{Ma}$ considered in the previous paragraph and we would investigate other regimes where the stratification has an even more important effect. 

For clarity of exposition, we neglect the centrifugal effects and the heat transfer process in the fluid, focusing on the classical barotropic Navier-Stokes system: 
\begin{equation}\label{2D_Euler_system}
\begin{cases}
\partial_t \vrho + \div (\vrho\vec{u})=0\  \\[2ex]
	\partial_t (\vrho\vec{u})+ \div(\vrho\vec{u}\otimes\vec{u}) + \dfrac{\e_3 \times \vrho\vec{u}}{Ro}\,  +    \dfrac{1}{Ma^2} \nabla_x p(\vrho) 
=\div \mbb{S}(\nabla_x\vec{u})  + \dfrac{\vrho}{Fr^2} \nabla_x G\, . \tag{NSC}
\end{cases}
\end{equation}
The more general system presented in \eqref{eq_i:NSF} can be handled at the price of additional technicalities already discussed above (remember, in particular, the restriction on the Mach number due to the presence of the centrifugal force).  
The goal now is to perform the asymptotic limit for system \eqref{2D_Euler_system} in the regimes when we assume
\begin{equation*}
Ma=\veps^m, \quad \quad Ro=\veps \quad \quad \text{and}\quad \quad Fr=\veps^n
\end{equation*}
with
\begin{equation*} 
\mbox{ either }\qquad m\,>\,1\quad\mbox{ and }\quad m\,<\,2\,n\,\leq\,m+1\,,\qquad\qquad\mbox{ or }\qquad
m\,=\,1\quad\mbox{ and }\quad \frac{1}{2}\,<\,n\,<\,1\,.
\end{equation*}
The restriction $n<1$ when $m=1$ is imposed in order to avoid a strong stratification regime: as already mentioned before, it is not clear
at present how to deal with this case for general ill-prepared initial data, as all the available techniques seem to break down in that case.
On the other hand, the restriction $2\,n\leq m+1$ (for $m>1$) looks to be of technical nature. However, it comes out naturally
in at least two points of our analysis (see e.g. Subsections \ref{ss:ctl1_G} and \ref{sss:term2_G}), and it is not clear to us if, and how, it is possible to bypass it and consider the remaining range of values
$(m+1)/2\,<\,n\,<\,m$. 
Let us point out that, in our considerations, the relation $n<m$ holds always true,
so we will always work in a low stratification regime.

At the qualitative level, our main results will be quite similar to the ones presented in the previous part. In particular the limit dynamics will be the same, after distinguishing the two cases $m>1$ and $m= 1$ (see Theorems \ref{th:m>1} and \ref{th:m=1}). 
The main point, we put the accent on now, is how using in a fine way not only the structure of the system, but also the precise structure of each term in order to pass to the limit. To be more precise, the fact of considering
values of $n$ going above the threshold $2n=m$ is made possible thanks to special algebraic cancellations involving the gravity term in the system of wave equations.
Such cancellations owe to the peculiar form of the gravitational force, which depends on the vertical variable only, and they do not appear, in general, if one wants to consider the action of different forces on the system. As one may easily guess, the case $2n=m+1$ is more involved: indeed, this choice
of the scaling implies the presence of an additional bilinear term of order $O(1)$ in the computations; in turn, this term might not vanish in the limit, differently to what happens in the case $2n<m+1$. In order to see that this does not occur, and that this term indeed disappears in the limit process, one has to use more thoroughly the structure of the system to control the oscillations (see equation \eqref{rel_oscillations_bis} and computations below).

\section*{The Euler system: the incompressible case}
In the second part of this thesis, we change our focus dealing with an incompressible and inviscid system with a hyperbolic structure. More precisely, we are interested in describing the 2-D evolution of a fluid that takes places far enough from the physical boundaries. Therefore, in $\Omega:=\R^2$, the Euler type system reads 
\begin{equation}\label{full Euler_intro}
\begin{cases}
\d_t \vrho +\div (\vrho \vu)=0\\
\d_t (\vrho \vu)+\div (\vrho \vu \otimes \vu)+ \dfrac{1}{Ro}\vrho \vu^\perp+\nabla_x p=0\\
\div \vu =0\, ,\tag{E}
\end{cases}
\end{equation}
where $\vu^\perp:=(-u^2, u^1)$ is the rotation of angle $\pi/2$ of the velocity field $\vu=(u^1,u^2)$.
The pressure term $\nabla_x p$ represents the Lagrangian multiplier associated to the divergence-free constraint on the velocity field.

The main scope of that analysis will be to study the asymptotic behaviour of the system \eqref{full Euler_intro} when $Ro=\veps\rightarrow 0^+$. 

\subsection*{Known results}
We will limit ourselves to give a short exposition on known results dealing with density-dependent fluids. 
We refer instead to \cite{C-D-G-G} for an overview of the broad literature in the context of homogeneous rotating fluids (see also \cite{B-M-N_EMF} and \cite{B-M-N_AA} for the pioneering studies, concerning the homogeneous 3-D Euler and Navier-Stokes equations). 

In the compressible cases discussed above, the fact that  the pressure is a given function of the density implies a double advantage in the analysis: on the one hand, one can recover good uniform bounds for the oscillations (from the reference state) of the density; on the other hand, at the limit, one disposes of a stream-function relation between the densities and the velocities. 

On the contrary, although the incompressibility condition is physically well-justified for the geophysical fluids, only few studies tackle this case. We refer to \cite{Fan-G}, in which Fanelli and Gallagher have studied the fast rotation limit for viscous incompressible fluids with variable density. In the case when the initial density is a small perturbation of a constant state (the so-called \textit{slightly non-homogeneous} case), they proved convergence to a quasi-homogeneous type system. Instead, for general non-homogeneous fluids (the so-called \textit{fully non-homogeneous} case), they have shown that the limit dynamics is described in terms of the vorticity and the density oscillation function, since they lack enough regularity to prove convergence on the momentum equation itself (see more details below).

We have also to mention \cite{C-F_RWA}, where the authors rigorously prove the convergence of the ideal magnetohydrodynamics (MHD) equations towards a quasi-homogeneous type system (see also \cite{C-F_Nonlin} where the compensated compactness argument is adopted). Their method relies on a relative entropy inequality for the primitive system that allows to treat also the inviscid limit but requires well-prepared initial data. 

\subsection*{New results} 

In Chapter \ref{chap:Euler}, we tackle the asymptotic analysis (for $\veps\rightarrow 0^+$) in the case of density-dependent Euler system in the \textit{slightly non-homogeneous} context, i.e. when the initial density is a small perturbation of order $\veps$ of a constant profile (say $\overline{\vrho}=1$). These small perturbations around a constant reference state are physically justified by the so-called \textit{Boussinesq approximation} (see e.g. Chapter 3 of \cite{C-R} or Chapter 1 of \cite{Maj} in this respect). As a matter of fact, since the constant state $\overline{\vrho}=1$ is transported by a divergence-free vector field, the density can be written as $\vrho_\veps=1+\veps R_\veps$ at any time (provided this is true at $t=0$), where one can state good uniform bounds on $R_\veps$. We also point out that in the momentum equation of \eqref{full Euler_intro}, with the scaling $Ro=\veps$, the Coriolis term can be rewritten as
\begin{equation}\label{Coriolis}
 \frac{1}{\veps}\vrho_\veps \vec u_\veps^\perp=\frac{1}{\veps}\vec u_\veps^\perp+R_\veps\vec u_\veps^\perp \, .\tag{2D COR}
\end{equation}
We notice that, thanks to the incompressibility condition, the former term on the right-hand side of \eqref{Coriolis} is actually a gradient: it can be ``absorbed'' into the pressure term, which must scale as $1/\veps$. In fact, the only force that can compensate the effect of fast rotation in system \eqref{full Euler_intro} is, at geophysical scale, the pressure term: i.e. we can write $\nabla_x p_\veps= (1/\veps) \, \nabla_x \Pi_\veps$.

Let us point out that the \textit{fully non-homogeneous} case (where the initial density is a perturbation of an arbitrary state) is out of our study. This case is more involved and new technical troubles arise in the well-posedness analysis and in the asymptotic inspection. Indeed, as already highlighted in \cite{Fan-G} for the Navier-Stokes-Coriolis system, the limit dynamics is described by an underdetermined system which mixes the vorticity and the density fluctuations.
In order to depict the full limit dynamics (where the limit density variations and the limit velocities are decoupled), one had to assume stronger \textit{a priori} bounds than the ones which could be obtained by classical energy estimates. Nonetheless, the higher regularity involved is \textit{not} propagated uniformly in $\veps$ in general, due to the presence of the Coriolis term. 
In particular, the structure of the Coriolis term is more complicated  than the one in \eqref{Coriolis} above, since one has $\vrho_{\veps}=\oline \vrho+ \veps \sigma_\veps$ (with $\sigma_\veps$ the fluctuation), if at the initial time we assume $\vrho_{0, \veps}=\oline \vrho+ \veps R_{0,\veps}$ with $\oline \vrho$ the arbitrary reference state. At this point, if one plugs the previous decomposition of $\vrho_\veps$ in \eqref{Coriolis}, a term of the form $(1/\veps)\, \oline \vrho \ue^\perp$ appears: this term is a source of troubles in order to propagate the higher regularity estimates needed.

Equivalently, if one tries to divide the momentum equation in \eqref{full Euler_intro} by the density $\vrho_\veps$, then the previous issue is only translated on the analysis of the pressure term, which becomes $1/(\veps \vrho_\veps)\, \nabla_x \Pi_\veps$.

In light of all the foregoing discussion, let us now point out the main difficulties arising in our problem.

First of all, our model is an inviscid system with a hyperbolic type structure for which we can expect \textit{no} smoothing effects and \textit{no} gain of regularity. For that reason, it is natural to look at equations in \eqref{full Euler_intro} in a regular framework like the $H^s$ spaces with $s>2$. The Sobolev spaces $H^s(\R^2)$, for $s>2$, are in fact embedded in the space $W^{1,\infty}$ of globally Lipschitz functions: this is a minimal requirement to preserve the initial regularity (see e.g. Chapter 3 of \cite{B-C-D} and also \cite{D_JDE}, \cite{D-F_JMPA} for a broad discussion on this topic). 
Actually, all the Besov spaces $B^s_{p,r}(\R^d)$ which are embedded in $W^{1,\infty}(\R^d)$, a fact that occurs for $(s,p,r)\in \R\times [1,+\infty]^2$ such that 
\begin{equation}\label{Lip_assumption}
s>1+\frac{d}{p} \quad \quad \quad \text{or}\quad \quad \quad s=1+\frac{d}{p} \quad \text{and}\quad r=1\, ,\tag{LIP}
\end{equation}
are good candidates for the well-posedness analysis (see Appendix \ref{app:Tools} for more details). However, the choice of working in $H^s\equiv B^s_{2,2}$ is dictated by the presence of the Coriolis force: we will deeply exploit the antisymmetry of this singular term.  

Moreover, the fluid is assumed to be incompressible, so that the pressure term is just a Lagrangian multiplier and does \textit{not} give any information on the density, unlike in the compressible case. In addition, due to the non-homogeneity, the analysis of the gradient of the pressure term is much more involved since we have to deal with an elliptic equation with \textit{non-constant} coefficients, namely 
\begin{equation}\label{elliptic_eq_introduction}
-\div (A \, \nabla_x p)=\div F \quad \text{where}\quad \div F:=\div \left(\vu \cdot \nabla_x \vu+ \frac{1}{Ro} \vu^\perp \right)\quad \text{and}\quad A:=1/\vrho \, . \tag{ELL}
\end{equation}
The main difficulty is to get appropriate \textit{uniform} bounds (with respect to the rotation parameter) for the pressure term in the regular framework we will consider. We refer to \cite{D_JDE} and \cite{D-F_JMPA} for more details concerning the issues which arise in the analysis of the elliptic equation \eqref{elliptic_eq_introduction} in $B^s_{p,r}$ spaces.

Once we have analysed the pressure term, we show that system \eqref{full Euler_intro} is locally well-posed in the $H^s$ setting. It is worth to notice that, in the local well-posedness theorem, all the estimates will be \textit{uniform} with respect to the rotation parameter and, in addition, we will have that the time of existence is independent of $\veps$. 
\begin{theoremnonum}
Let $s>2$. For any $\veps>0$, 
there exists a time $T_\veps^\ast >0$ 
such that 
the system \eqref{full Euler_intro} has a unique solution $(\vrho_\veps, \vu_\veps, \nabla_x \Pi_\veps)$ where 
\begin{itemize}
\item $\vrho_\veps$ belongs to the space $C^0([0,T_\veps^\ast]\times \R^2)$ with $\nabla_x \vrho_\veps \in  C^0([0,T_\veps^\ast]; H^{s-1}(\R^2))$;
\item $\vu_\veps$ and $\nabla_x \Pi_\veps$ belong to the space $C^0([0,T_\veps^\ast]; H^s(\R^2))$.
\end{itemize}
Moreover, 
$$\inf_{\veps>0}T_\veps^\ast  >0\, .$$
\end{theoremnonum}
With the local well-posedness result at the hand, we perform the fast rotation limit for general \textit{ill-prepared} initial data. We show the convergence of system \eqref{full Euler_intro} towards what we call \textit{quasi-homogeneous} incompressible Euler system 
\begin{equation}\label{Q-H_E_intro}
\begin{cases}
\d_t R+\div (R\vu)=0 \\
\d_t \vec u+\div \left(\vec{u}\otimes\vec{u}\right)+R\vu^\perp+ \nabla_x \Pi =0 \\
\div \vec u\,=\,0\,, \tag{QHE}
\end{cases}
\end{equation}  
where $R$ represents the limit of fluctuations $R_\veps$. 
We also point out that in the momentum equation of \eqref{Q-H_E_intro} a non-linear term of lower order (i.e. $R\vec u^\perp$) appears: it is a sort of remainder in the convergence for the Coriolis term, recasted as in \eqref{Coriolis}.

Passing to the limit in the momentum equation of \eqref{full Euler_intro} is no more evident, although we are in the $H^s$ framework: the Coriolis term is responsible for strong oscillations in time of solutions (the already quoted \textit{Poincar\'e waves}) which may prevent the convergence of the convective term towards the one of \eqref{Q-H_E_intro}. To overcome this issue, we employ the same approach mentioned above: the compensated compactness technique.

Now, once the limit system is rigorously depicted, one could address its well-posedness issue: it is worth noticing that the global well-posedness of system \eqref{Q-H_E_intro} remains an open problem. However, roughly speaking, for $R_0$ small enough, the system \eqref{Q-H_E_intro} is ``close'' to the $2$-D homogeneous and incompressible Euler system, for which it is well-known the global well-posedness. Thus, it is natural to wonder if there exists an ``asymptotically global'' well-posedness result in the spirit of \cite{D-F_JMPA} and \cite{C-F_sub}: for small initial fluctuations $R_0$, the quasi-homogeneous system \eqref{Q-H_E_intro} behaves like the standard Euler equations and the lifespan of its solutions tends to infinity. In particular, as already shown in \cite{C-F_sub} for the quasi-homogeneous ideal MHD system (see also references therein), we prove that the lifespan of solutions to \eqref{Q-H_E_intro} goes as  
\begin{equation}\label{lifespan_Q-H}
T_\delta^\ast \sim \log \log \frac{1}{\delta}\,,\tag{LIFE} 
\end{equation}
where $\delta>0$ is the size of the initial fluctuations. 

This result for the time of existence of solutions to \eqref{Q-H_E_intro} pushes our attention to the study of the lifespan of solutions to the primitive system \eqref{full Euler_intro}.
For the $3$-D \textit{homogeneous} Euler system with the Coriolis force, Dutrifoy in \cite{Dut} has proved that the lifespan of solutions tends to infinity in the fast rotation regime (see also \cite{Gall}, \cite{Cha} and \cite{Scro}, where the authors inspected the lifespan of solutions in the context of viscous homogeneous fluids). For system \eqref{full Euler_intro} it is not clear to us how to find similar stabilization effects (due to the Coriolis term), in order to improve the lifespan of the solutions: for instance to show that $T_\veps^\ast\longrightarrow +\infty$ when $\veps\rightarrow 0^+$. Nevertheless, independently of the rotational effects, we are able to state an ``asymptotically global'' well-posedness result in the regime of \textit{small} oscillations, in the sense of \eqref{lifespan_Q-H}: namely, when the size of the initial fluctuation $R_{0,\veps}$ is small enough, of size $\delta >0$, the lifespan $T^\ast_\veps$ of the corresponding solution to system \eqref{full Euler_intro} can be bounded from below by $T^\ast_\veps\geq T^\ast(\delta)$, with $T^\ast (\delta)\longrightarrow +\infty$ when $\delta\rightarrow 0^+$ (see also \cite{D-F_JMPA} for a density-depend fluid in the absence of Coriolis force). More precisely, one has the following result.
\begin{propnonum}
The lifespan $T_\veps^\ast$ of the solution to the two-dimensional density-dependent incompressible Euler equations \eqref{full Euler_intro} with the Coriolis force is bounded from below by
\begin{equation}\label{improv_life_fullE_intro}
\frac{C}{\|\vec u_{0,\veps}\|_{H^s}}\log\left(\log\left(\frac{C\, \|\vec u_{0,\veps}\|_{H^s}}{\max \{\mc A_\veps(0),\, \veps \, \mc A_\veps(0)\, \|\vec u_{0,\veps}\|_{H^s}\}}+1\right)+1\right)\, ,\tag{BOUND}
\end{equation}
where $\mc A_\veps (0):= \|\nabla_x R_{0,\veps}\|_{H^{s-1}}+\veps\, \|\nabla_x R_{0,\veps}\|_{H^{s-1}}^{\lambda +1}$, for some suitable $\lambda\geq 1$.
\end{propnonum}
As an immediate corollary of the previous lower bound, if we consider the initial densities of the form $\vrho_{0,\veps}=1+\veps^{1+\alpha}R_{0,\veps}$ with $\alpha >0$, then we get $T^\ast_\veps\sim \log \log (1/\veps)$.

At this point, let us sketch the main steps to show \eqref{improv_life_fullE_intro} for the primitive system \eqref{full Euler_intro}.

The key point in the proof of \eqref{improv_life_fullE_intro} is to study the lifespan of solutions in critical Besov spaces. In those spaces, we can take advantage of the fact that, when $s=0$, the $B^0_{p,r}$ norm of solutions can be bounded \textit{linearly} with respect to the Lipschitz norm of the velocity, rather than exponentially (see the works \cite{Vis} by Vishik and \cite{H-K} by Hmidi and Keraani). Since the triplet $(s,p,r)$ has to satisfy \eqref{Lip_assumption}, the lowest regularity Besov space we can reach is $B^1_{\infty,1}$. Then if $\vec u$ belongs to $B^1_{\infty,1}$, the vorticity $\omega := -\d_2 u^1+\d_1 u^2$ has the desired regularity to apply the quoted improved estimates by Hmidi-Keraani and Vishik (see Theorem \ref{thm:improved_est_transport} in this respect). 

Analysing the vorticity formulation of the system, we discover that the \emph{curl} operator cancels the singular effects produced by the Coriolis force: that cancellation is not apparent, since the skew-symmetric property of the Coriolis term is out of use in the critical framework considered.

Finally, we need a continuation criterion (in the spirit of Baele-Kato-Majda criterion, see \cite{B-K-M}) which guarantees that we can ``measure'' the lifespan of solutions in the space of lowest regularity index, namely $s=r=1$ and $p=+\infty$. That criterion is valid under the assumptions that 
$$
\int_0^{T}  \big\| \nabla_x \vec u(t) \big\|_{L^\infty}  \dt < +\infty\qquad \text{with}\qquad T<+\infty\, .
$$
The previous criterion ensures that the lifespan of solutions found in the critical Besov spaces is the same as in the sought Sobolev functional framework, allowing us to conclude the proof (see considerations in Subsection \ref{ss:cont_criterion+consequences}). 

\medskip

\subsection*{Overview of the contents of this thesis}
Before moving on, we give a brief overview of the structure of the present thesis. 

The Chapter \ref{chap:geophysics} has the goal to ``dip'' the reader in the discipline of the geophysical fluid dynamics, giving a brief physical justification of the mathematical models we will consider in the next chapters. In Chapter \ref{chap:multi-scale_NSF} we address the study of the singular perturbation problem, given by the Navier-Stokes-Fourier equations, in the scaling which we call ``critical''. The next Chapter \ref{chap:BNS_gravity} is devoted to the improvement of the previous scaling: we will go beyond the ``critical'' threshold. Finally, in the last Chapter \ref{chap:Euler} we change a bit the model, dedicating ourselves to the asymptotic analysis for the density-dependent Euler equations. In addition, we will focus on the lifespan of its solutions, proving an ``asymptotically global'' well-posedness result. 

At the end of this thesis, there are two more sections dedicated to the future perspectives and an appendix containing some tools and well-known results employed throughout the manuscript.

\newpage
\chapter{Contributions de la th\`ese} 

\section*{Le problème de Navier-Stokes-Fourier: un aperçu physique}
Dans cette thèse, nous nous consacrons à l'étude du comportement des fluides caractérisés par de grandes échelles de temps et d'espace.
Des exemples typiques sont les courants dans l'atmosphère et l'océan, mais bien sûr il y a
de nombreux autres phénomènes liés aux fluides hors de la Terre, comme les écoulements sur une étoiles ou sur d'autres corps célestes.
\`A ces échelles, les effets de la rotation de l'environnement (qui dans le cas de l'atmosphère ou de l'océan est la Terre) ne sont pas négligeables, et le mouvement du fluide est influencé par l'action d'une forte force de Coriolis. Il y a deux autres éléments qui caractérisent la dynamique de ce type de fluides, appelés géophysiques (voir \cite{C-R}, \cite{Ped} et \cite{Val}, par exemple): la faible compressibilité du fluide et les effets de la stratification (les variations de densité, essentiellement à cause de la gravité). 
L'importance des attributs précédents est ``mesurée'' en introduisant, dans le modèle mathématique, trois paramètres adimensionnels positifs qui, pour les fluides géophysiques, sont supposés faibles. Ces paramètres sont:
\begin{itemize}
\item le nombre de \emph{Mach} $Ma$, qui fixe la taille des écarts isentropiques par rapport aux fluides incompressibles: plus $Ma$ est petit, plus les effets de compressibilité sont faibles;
\item le nombre de \emph{Froude} $Fr$, qui mesure l'importance des effets de la stratification dans la dynamique: plus $Fr$ est petit, plus les effets gravitationnels sont forts;
\item le nombre de \emph{Rossby} $Ro$, qui est lié à la rotation du système: lorsque $Ro$ est très petit, les effets de la rotation rapide sont prédominants dans la dynamique.
\end{itemize}

Dans notre contexte, nous adoptons une hypothèse simpliste (souvent supposée dans les études physiques et mathématiques) qui consiste à restreindre l'étude du fluide aux latitudes moyennes, c'est-à-dire aux écoulements qui se déroulent assez loin des pôles et de la zone équatoriale. Dans ce cas, les variations des effets de rotation dues à la latitude sont négligeables.
 
Notons par $\vrho ,\, \vtheta\geq 0$ respectivement  la densité et la température absolue du fluide, et par $\vec{u}\in \mathbb{R}^3$ son champ de vitesse: le système de Navier-Stokes-Fourier 3-D dans sa forme adimensionnelle peut être écrit (voir par exemple \cite{F-N}) comme
\begin{equation} \label{eq_i:NSF_fr}
\begin{cases}
 \partial_t \vrho + \div (\vrho\vec{u})=0\ \\[3ex] 
 \partial_t (\vrho\vec{u})+ \div(\vrho\vec{u}\otimes\vec{u}) + \dfrac{\e_3 \times \vrho\vec{u}}{Ro}\, + \dfrac{1}{Ma^2} \nabla_x p(\vrho,\vtheta) \\[1ex]     
 \qquad \qquad \qquad \qquad \qquad \qquad \qquad \; \; \; =\div \mbb{S}(\vtheta,\nabla_x\vec{u}) + \dfrac{\vrho}{Ro^2} \nabla_x F + \dfrac{\vrho}{Fr^2} \nabla_x G \\[3ex]   
 \partial_t \bigl(\vrho s(\vrho, \vtheta)\bigr) + \div \bigl(\vrho s (\vrho,\vtheta)\vec{u}\bigr) + \div\left(\dfrac {\q(\vtheta,\nabla_x \vtheta )}{\vtheta} \right) 
 = \sigma\,, \tag{NSF}
\end{cases}
\end{equation}
dans la bande 3-D infinie:
\begin{equation} \label{eq:domain_fr}
\Omega\,:=\,\R^2\times\,]0,1[\,. \tag{DOM}
\end{equation}
Dans le système \eqref{eq_i:NSF_fr} ci-dessus, les fonctions $s,\vec{q},\sigma$ sont respectivement l'entropie spécifique, le flux de chaleur et le taux de production d'entropie, et $\mbb{S}$ est le tenseur des contraintes visqueuses, qui satisfait la loi rhéologique de Newton (voir les sous-sections \ref{sss:primsys} and \ref{sss:structural} pour une formulation plus précise).
 
La force de Coriolis est représentée par
\begin{equation} \label{def:Coriolis_fr}
\mf C(\vr,\vu)\,:=\,\frac{1}{Ro}\,\vec e_3\times\vr\,\vu\,, \tag{COR}
\end{equation}
où $\vec e_3=(0,0,1)$ et le symbole $\times$ représente le produit extérieur usuel des vecteurs dans $\R^3$. En particulier, la définition précédente implique que la rotation a lieu autour de l'axe vertical, et sa force ne dépend pas de la latitude (voir par exemple \cite{C-R} et \cite{Ped} pour plus de détails). Nous soulignons que, malgré toutes ces simplifications, le modèle obtenu est déjà capable de capturer plusieurs phénomènes physiquement typiques dans la dynamique des écoulements géophysiques: le fameux \emph{théorème de Taylor-Proudman}, la formation des \emph{couches d'Ekman} et la propagation des \emph{ondes de Poincar\'e}. Nous renvoyons à \cite{C-D-G-G} pour une discussion plus approfondie.
Dans la présente thèse, nous évitons les effets de couche limite, i.e. le problème lié aux couches d'Ekman, en imposant des conditions aux limites appelées conditions de \emph{glissement complèt}.
 
Comme établi par le \emph{théorème de Taylor-Proudman} en géophysique, la rotation rapide impose une certaine rigidité/stabilité, forçant le mouvement à se dérouler sur des plans orthogonaux à l'axe de rotation.
Par conséquent, la dynamique devient purement bidimensionnelle et horizontale, et le fluide a tendance à se déplacer long de colonnes verticales.
 
Cependant, une telle configuration idéale est entravée par une autre force fondamentale agissant aux échelles géophysiques, la gravité, qui travaille à restaurer la stratification verticale de la densité. La force gravitationnelle est décrite dans le système \eqref{eq_i:NSF_fr} par le terme
\[
 \mc G(\vr)\,:=\,\frac{1}{Fr^2}\,\vr\,\nabla_xG\, ,
\]
où dans notre cas $G(x)=\,G(x^3)\,=\,-\,x^3$. De plus, les effets gravitationnels sont affaiblis par la présence de la force centrifuge 
$$\mf F(\vr):=\frac{1}{Ro^2}\,  \vr\, \nabla_x F\, ,$$
avec $F(x)=|x^h| ^2$. Une telle force est une force d'inertie qui, aux latitudes moyennes, décale légèrement la direction de la gravité.
 
Ainsi, la compétition entre les effets de stabilisation, dus à la rotation, et la stratification verticale (due à la gravité), se traduit dans le modèle par la compétition entre les ordres de grandeur de $Ro$ et $Fr$.
 
De plus, il apparaît que la gravité $\mc G$ agit en combinaison avec la force de pression:
$$
\mf P(\vr, \vtheta)\,:=\,\frac{1}{Ma^2}\,\nabla_x p(\vr, \vtheta)\,,
$$
où $p$ est une fonction lisse connue de la densité et de la température du fluide (voir sous-section \ref{sss:structural}).
 
Nous remarquons que les termes $\mf C,\, \mc G,\, \mf P$ et $\mf F$ entrent en jeu dans le modèle avec un grand pré-facteur, en conséquence notre but est d'étudier les systèmes quand $Ma, \, Fr$ et $Ro$ sont petits dans différents régimes.

\section*{L'analyse multi-échelle}
Au niveau mathématique, au cours des 30 dernières années, un nombre considérable de travaux a été consacré à la justification rigoureuse, dans divers cadres fonctionnels, des modèles réduits considérés en géophysique.
 
L'examen de l'ensemble de la littérature sur ce sujet dépasse largement le cadre de cette partie introductive, c'est pourquoi nous faisons le choix de ne rapporter que les travaux qui abordent la présence de la force de Coriolis \eqref{def:Coriolis_fr}.
Nous décidons également de reporter, à la partie suivante, la discussion sur les modèles incompressibles, car moins pertinents pour l'analyse multi-échelle, en raison de la rigidité imposée par la contrainte de divergence nulle sur le champ de vitesse du fluide.
 
Les modèles de fluides compressibles, au contraire, fournit un cadre beaucoup plus riche pour l'analyse multi-échelle des écoulements géophysiques.
De plus, nous avons choisi de nous concentrer principalement sur les travaux traitant des fluides visqueux et qui effectuent l'étude asymptotique pour des données initiales mal préparées.
 
 \subsection*{Résultats précédents}
Les premiers résultats, dans le sens mentionné ci-dessus, ont été obtenus par Feireisl, Gallagher et Novotn\'y dans \cite{F-G-N} et avec G\'erard-Varet dans \cite{F-G-GV-N}, pour le système barotrope de Navier-Stokes (voir aussi \cite{B-D-GV} pour une étude préliminaire et \cite{G-SR_Mem} pour l'analyse des ondes équatoriales). Dans ces travaux, les auteurs ont étudié le régime combiné du faible nombre de Rossby (effets de rotation rapide) avec un régime de faible nombre de Mach (faible compressibilité du fluide) sous l'échelle 
\begin{align}
Ro\,&=\,\veps\tag{LOW RO}\\
Ma\,&=\,\veps^m \qquad\qquad \mbox{ avec }\quad m\geq 0
\,, \tag{LOW MA} \label{eq:scale_fr}
\end{align}
où $\veps\in\,]0,1]$ est un petit paramètre, qu’on aimerait faire tendre vers $0$ afin de dériver le modèle réduit. Dans le cas où $m=1$ dans \eqref{eq:scale_fr}, le système présente une échelle isotrope, puisque $Ro$ et $Ma$ agissent au même ordre de grandeur et les termes de pression et de rotation restent en équilibre (l'\emph{équilibre quasi-géostrophique}) à la limite. Le système limite est identifié par l’\emph{équation quasi-géostrophique} pour une fonction de flux du champ de vitesse.
Au contraire, dans \cite{F-G-GV-N} lorsque $m>1$ et avec en plus la force centrifuge, le terme de pression prédomine (sur la force de Coriolis) dans la dynamique du fluide. 
Dans ce cas, le système limite est décrit par un système de Navier-Stokes incompressible en 2-D et les difficultés générées par l'anisotropie d'échelle ont été surmontées en utilisant des estimations de dispersion.
 
Par la suite, Feireisl et Novotn\'y ont poursuivi l'analyse multi-échelle pour le même système, encore une fois sans le terme de force centrifuge, en considérant les effets d'une faible stratification, c’est-à-dire $Ma/Fr\rightarrow 0$ lorsque $\veps\rightarrow 0^+$ (voir \cite{F-N_AMPA}, \cite{F-N_CPDE}).
Nous renvoyons à \cite{F_MA} pour une étude similaire dans le cadre des modèles capillaires, où le choix $m=1$ a été fait, mais l'anisotropie a été donnée par l'échelle fixée pour le terme de forces internes (appelé tenseur des contraintes de Korteweg). De plus, il faut mentionner \cite{F_2019} pour le cas des grands nombres de Mach par rapport au paramètre de Rossby, à savoir $0\leq m<1$ dans \eqref{eq:scale_fr}. Dans ce cas, le gradient de pression n'est pas assez fort pour compenser la force de Coriolis, et afin de trouver une dynamique limite pertinente, il faut pénaliser le coefficient de viscosité.
 
L'analyse des modèles présentant aussi des transferts de chaleur est beaucoup plus récente, et a été commencé avec le papier \cite{K-M-N} de Kwon, Maltese et Novotn\'y. Dans cet article, les auteurs ont considéré une approche multi-échelle pour le système de Navier-Stokes-Fourier complet avec Coriolis et la force gravitationnelle (et $F=0$), en prenant l'échelle
\begin{equation} \label{eq:scale-G_fr}
{Fr}\,=\,\veps^n\,,\qquad\qquad\mbox{ avec }\quad 1\,\leq\,n\,<\,\frac{m}{2}\, .\tag{LOW FR}
\end{equation}
En particulier, dans cet article, le choix \eqref{eq:scale-G_fr} impliquait que $m>2$ et le cas $n=m/2$ était laissé ouvert. Des restrictions similaires sur les paramètres peuvent être trouvées dans \cite{F-N_CPDE} pour le modèle barotrope. Ces restrictions doivent être attribuées aux techniques utilisées pour prouver la convergence, qui sont basées sur une combinaison de méthode d'énergie relative/entropie relative avec des estimations dispersives
(on note qu'une restriction encore plus grande, $m>10$ , apparaît dans \cite{F-G-GV-N}). D'autre part, on souligne que les méthodes d'énergie relative permettent d'obtenir un taux de convergence précis et de considérer également des limites non visqueuses et non diffusives (dans ces cas, on ne dispose pas d'une borne uniforme pour $\nabla_x\vtheta$ et sur $\nabla_x\vec{u}$).
Le cas où $m=1$ a été traité postérieurement dans l’article \cite{K-N} de Kwon et Novotn\'y, en recourant à des techniques similaires (cependant, le terme gravitationnel n'est pas pénalisé).
\subsection*{Nouveautés} 
La première partie de cette thèse est consacrée à l'étude des problèmes multi-échelle, en se concentrant sur le système de Navier-Stokes-Fourier complet introduit dans \eqref{eq_i:NSF_fr}.
Dans un premier temps, nous améliorons le choix de l'échelle \eqref{eq:scale-G_fr} en prenant le cas limite $n=m/2$ avec $m\geq1$ (c'est l'échelle adoptée dans le chapitre \ref{chap:multi-scale_NSF}). Bien sûr, nous sommes toujours dans un régime de faible stratification, puisque $Ma/Fr\ra0$, mais le choix $Fr=\sqrt{Ma}$  nous permet de capturer quelques propriétés qualitatives supplémentaires sur la dynamique limite.
De plus, nous ajoutons au système le terme de force centrifuge $\nabla_x F$ (dans l'esprit de \cite{F-G-GV-N}), qui est source des problèmes techniques dus à son caractère non borné.
Commentons maintenant toutes ces questions en détail.
 
Tout d'abord, en absence de la force centrifuge, c'est-à-dire $F=0$, nous sommes capables d'effectuer la limite incompressible, avec une faible stratification et une rotation rapide pour \emph{toute la gamme} de valeurs $m\geq 1$, dans le cadre des \emph{solutions faibles d'énergie finie} de le système de Navier-Stokes-Fourier \eqref{eq_i:NSF_fr} et pour des \emph{données initiales mal préparées}.
Dans le cas $m>1$, les effets d'incompressibilité et de stratification sont prédominants par rapport à la force de Coriolis: on prouve alors la convergence vers le bien connu \emph{système d'Oberbeck-Boussinesq} (voir par exemple le paragraphe 1.6.2 de \cite{Z} pour des explications physiques sur ce système), donnant une justification rigoureuse à ce modèle approché dans le contexte des fluides en rotation rapide. Ainsi, nous pouvons énoncer le théorème suivant (voir le Théorème \ref {th:m-geq-2} pour l'énoncé précis).
\begin{theoremnonum_fr} \label{thm_1_fr}
On considère le système \eqref{eq_i:NSF_fr}. Soit $\Omega = \R^2 \times\,]0,1[\,$. Soit $F=|x^{h}|^{2}$ et $G=-x^{3}$. On prenne $n=m/2$ et ou bien ${m\geq 2}$, ou ${m>1}$ et ${ F=0}$. Alors, on a les convergences suivantes:
 \begin{align*}
 \varrho_\ep \rightarrow 1 \\ 
 R_{\ep}:=\frac{\varrho_\ep - 1}{\ep^m} \weakstar R \\   
 \vec{u}_\ep \weak \vec{U} \\
 \Theta_{\ep} :=\frac{\vartheta_\ep - \bar{\vartheta}}{\ep^m} \weak \Theta \, , 
 \end{align*}
où en accord avec le théorème de Taylor-Proudman, on a
$$\vec{U} = (\vec U^h,0),\quad \quad \vec U^h=\vec U^h(t,x^h),\quad \quad \div_h\vec U^h=0.$$
De plus, $\Big(\vec{U}^h,\, \, R ,\, \, \Theta \Big)$ résout, au sens des distributions, le système incompressible de type Oberbeck-Boussinesq
\begin{align*}
& \d_t \vec U^{h}+\div_{h}\left(\vec{U}^{h}\otimes\vec{U}^{h}\right)+\nabla_h\Gamma-\mu (\oline\vtheta )\Delta_{h}\vec{U}^{h}=\delta_2(m)\langle R\rangle\nabla_{h}F \\[2ex]&
\d_t\Theta\,+\,\divh(\Theta\,\vec U^h)\,-\,\kappa(\oline\vtheta)\,\Delta\Theta\,=\,\oline\vtheta \,\vec{U}^h\cdot\nabla_h \overline{\mc G}
\\[2ex]
& \nabla_{x}\left( \d_\varrho p(1,\oline{\vtheta})\,R\,+\,\d_\vtheta p(1,\oline{\vtheta})\,\Theta \right)\,=\,\nabla_{x}G\,+\,\delta_2(m)\,\nabla_{x}F\, ,
\end{align*}
où
$\overline{\mc G}$ est la somme de forces externes $\,G\,+\,\delta_2(m)F$ , $\Gamma \in \mc D^\prime$ et $\delta_2(m )= 1$ si $m=2$ , $\delta_2(m)=0$ sinon.
\end{theoremnonum_fr}
 
Nous soulignons qu'à la limite le champ de vitesse est de dimension $2$, selon le célèbre théorème de Taylor-Proudman en géophysique: à la limite en rotation rapide, le mouvement du fluide a un comportement planaire, il se déroule sur des plans orthogonaux à l'axe de rotation (c'est-à-dire des plans horizontaux dans notre modèle) et il est essentiellement constant dans la direction verticale. On se réfère à
\cite{C-R}, \cite{Ped} et \cite{Val} pour plus de détails sur la formulation physique. Notez cependant que, bien que la dynamique limite soit purement horizontale, la densité limite et les variations de température,
$R$ et $\Theta$ respectivement, sont stratifiées: c'est l'effet principal du choix $n=m/2$ pour le nombre de Froude dans \eqref{eq:scale-G_fr}. 
C'est aussi la principale propriété qualitative qui est nouvelle dans notre travail par rapport aux études précédentes et qui justifie l'épithète d'échelle \emph{``critique''}.
 
Lorsque $m=1$, au contraire, toutes les forces agissent à la même échelle, puis s'équilibrent asymptotiquement pour $\veps\ra0^+$. 

En conséquence, le mouvement limite est décrit par une \emph{équation quasi-géostrophique} pour une fonction $q$, qui est liée à $R$ et $\Theta$ (respectivement, la densité et les variations de température à la limite) et à la gravité, et qui joue le rôle de fonction de flux pour le champ de vitesse limite. 
Cette équation quasi-géostrophique est couplée à une équation de transport-diffusion scalaire pour une nouvelle grandeur
$\Upsilon$, qui mélange $R$ et $\Theta$. 
L'énoncé précis du théorème suivant se trouve dans le paragraphe \ref{ss:results}. 
\begin{theoremnonum_fr}
On considère le système \eqref{eq_i:NSF_fr}. Soit $\Omega = \R^2 \times\,]0,1[\,$. Soit ${F=0}$ et $G=-x^3$. On prenne ${m=1}$ et $n=1/2$.
Alors, on a les mêmes convergences trouvées dans le Théorème \ref{thm_1_fr} et $\vec U$ satisfait le théorème de Taylor-Proudman.
Par ailleurs définissons
$$
\Upsilon := \d_\vrho s(1,\oline{\vtheta}) R + \d_\vtheta s(1,\oline{\vtheta})\,\Theta$$
et
$$q= \d_\varrho p(1,\oline{\vtheta}) R +\d_\vtheta p(1,\oline{\vtheta})\Theta -G \, .
$$
Alors $q=q(t,x^h)$ et $ \vec{U}^{h}=\nabla_h^{\perp} q$.
De plus, le couple $\Big(q,\, \, \Upsilon \Big)$ satisfait, au sens des distributions, 
\begin{align*}
& \d_{t}\left(q-\Delta_{h}q\right) -\nabla_{h}^{\perp}q\cdot
\nabla_{h}\left( \Delta_{h}q\right) +\mu (\oline{\vtheta})
\Delta_{h}^{2}q=\langle X\rangle \\[2.5ex] 
& \d_{t} \Upsilon +\nabla_h^\perp q\cdot\nabla_h\Upsilon-\kappa(\oline\vtheta) \Delta \Upsilon\,=\,
\,\kappa(\oline\vtheta)\,\Delta_hq\, ,
\end{align*}
où $\langle X\rangle$ est une force ``externe'' appropriée.
\end{theoremnonum_fr}
Ce théorème est dans l'esprit du résultat de \cite{K-N}, mais ici encore sont captés à la limite les effets gravitationnels, de sorte qu'il n'est plus possible d'affirmer que $R$ et $\Theta$ (donc $\Upsilon$) sont horizontaux. En revanche, et de façon surprenante, $q$ et la vitesse limite $\vec U$ sont purement horizontales.
 
\`A ce stade, faisons quelques remarques. Tout d'abord, mentionnons que, comme déjà annoncé, nous sommes capables d'ajouter au système les effets de la force centrifuge $\nabla_x F$.
Malheureusement, dans ce cas apparaît la restriction $m\geq 2$ (qui est quand même moins sévère que celles imposées dans \cite{F-G-GV-N}, \cite{F-N_CPDE} et \cite{K-M-N}).
Cependant, nous montrons qu'une telle restriction n'est pas de nature technique, mais qu'elle est cachée dans la structure du système d'ondes (voir proposition \ref{p:target-rho_bound} et remarque \ref{slow_rho}).
Le résultat pour $F\neq 0$ est analogue à celui présenté ci-dessus pour le cas $F=0$ et $m>1$: quand $m>2$, l'anisotropie de l'échelle est trop grande pour voir les effets dus à $F$ à la limite, et aucune différence qualitative n'apparaît par rapport à l'instance où $F=0$; lorsque $m=2$, en revanche, des termes supplémentaires liés à $F$ apparaissent dans le système de Oberbeck-Boussinesq (voir théorème \ref{thm_1_fr}). Dans tous les cas, l'analyse sera considérablement plus compliquée, puisque $F$ n'est pas bornée dans le domaine $\Omega$ (défini dans \eqref {eq:domain_fr} ci-dessus) et cela demandera une procédure de localisation supplémentaire (déjà employée dans \cite{F-G-GV-N}).
 
Soulignons en outre que la théorie classique de l'existence des solutions faibles d'énergie finie pour \eqref{eq_i:NSF_fr} exige que le domaine physique soit un sous-ensemble \emph {borné} et lisse de $\R^3$ (voir \cite{F-N} pour une étude complète). La théorie a ensuite été étendue dans \cite{J-J-N} pour couvrir le cas des domaines non bornés, et cela pourrait nous sembler approprié à notre cas.

Néanmoins, la notion de solutions faibles développée dans \cite{J-J-N} est en quelque sorte plus faible que la notion usuel (les auteurs parlent en fait de \emph{solutions très faibles}), dans la mesure où la formulation faible habituelle du bilan d'entropie, c'est-à-dire la troisième équation de \eqref {eq_i:NSF_fr}, doit être remplacée par une inégalité au sens des distributions.
Une telle formulation ne nous convient pas, car, lors de la dérivation du système d'ondes acoustiques de Poincar\'e, nous devons combiner la conservation de la masse et l'équation de l'entropie. En particulier, cela nécessite d'avoir de vraies égalités, satisfaites au sens faible (usuel).
Afin de pallier à ce problème, on recourt à la technique des \emph{domaines envahissants} (voir par exemple le chapitre 8 de \cite{F-N}, \cite{F-Scho} et \cite{WK}): pour chaque $\veps\in\,]0,1]$, on résout le système \eqref {eq_i:NSF_fr}, avec le choix $n=m/2$ pour le nombre de Froude, dans un domaine lisse $ \Omega_\veps$, où $\big(\Omega_\veps\big)_\veps$ converge (dans un sens approprié) vers $\Omega$, lorsque $\veps\ra0^+$, plus vite que la vitesse de propagation des ondes (qui est proportionnelle à $\veps^{-m}$).
Une telle ``procédure d'approximation'' nécessitera un travail supplémentaire.

Afin de prouver nos résultats, et d'obtenir l'amélioration sur les valeurs des différents paramètres, nous proposons une approche unifiée, qui fonctionne en fait à la fois pour le cas $m>1$ (permettant de traiter assez facilement l'anisotropie de l'échelle) et pour le cas $m=1$ (permettant de traiter l'opérateur de perturbation singulier plus compliqué).
Cette approche est basée sur les arguments de \emph{compacité par compensation}, d'abord employés par Lions et Masmoudi dans \cite{L-M} pour traiter la limite incompressible du système barotrope de Navier-Stokes, et plus tard adaptées par Gallagher et Saint-Raymond dans \cite {G-SR_2006} au cas des fluides en rotation rapide (incompressibles et homogènes). Des applications plus récentes de cette méthode dans le cadre des écoulements géophysiques se trouvent dans \cite{F-G-GV-N}, \cite{F_JMFM}, \cite{Fan-G} et \cite{F_2019}.
 
La méthode citée ne donne pas du tout une convergence quantitative, mais seulement qualitative. La technique est purement basée sur la structure algébrique du système, qui permet de trouver la petitesse (et disparaissant à la limite) de quantités non linéaires appropriées, et des propriétés de compacité pour d'autres quantités.
Ces propriétés de convergence forte ne sont en aucun cas évidentes, car les termes singuliers sont responsables de fortes oscillations en temps des solutions (les ondes dites acoustiques de Poincar\'e), qui peuvent empêcher la convergence des non-linéarités.
Néanmoins, une étude fine du système des ondes acoustiques de Poincar\'e révèle en fait la compacité (pour tout $m\geq1$ si $F=0$, pour $m\geq 2$ si $F\neq 0$) d’une quantité spéciale
$\g_\veps$, qui combine (grossièrement) les moyennes verticales de la quantité de mouvement $\vec{V}_\veps=\vrho_\veps\vec{u}_\veps$ (de son tourbillon, en fait) et d'une autre fonction
$Z_\veps$, obtenu comme une combinaison linéaire des variations de densité et de température (voir les sous-sections \ref{sss:term1} et \ref{ss:convergence_1} pour plus de détails sur ce sujet). Des propriétés de compacité similaires ont été mises en évidence dans \cite{Fan-G} pour les fluides incompressibles dépendant de la densité en $2$-D, et dans \cite {F_2019} pour traiter un problème multi-échelles aux ``grands'' nombres de Mach.
\`A la fin, la convergence forte de $\big(\g_\veps\big)_\veps$ s'avère suffisante pour prendre la limite dans le terme convectif, et pour compléter la preuve de nos résultats.
 
Pour conclure cette partie, on remarque que nous nous attendons à ce que la même technique puisse aussi marcher dans le cas $m=1$ et $F\neq0$ (ce fut le cas dans \cite{F-G-GV-N}, pour des écoulements barotropes).
Néanmoins, la présence de transfert de chaleur complique profondément le système des ondes, et de nouvelles difficultés techniques surviennent dans l'analyse du terme convectif (l'approche de \cite{F-G-GV-N}, dans le cas de température constante, ne fonctionne pas ici).
Pour cette raison, nous ne sommes pas capables de traiter ici cette condition, qui reste toujours une question ouverte.
 
Une autre caractéristique, qui n'est pas  traitée dans notre analyse, est le régime de \emph{forte stratification}, c’est-à-dire que le rapport ${Ma}/{Fr}$ est d'ordre $O(1)$. Ce régime est particulièrement délicat pour les fluides en rotation rapide.
Cela contraste fortement avec les résultats disponibles sur la dérivation de l'approximation anélastique, où la rotation est négligée: nous renvoyons pour exemple à \cite{Masm}, \cite{BGL}, \cite{F-K-N-Z} et, plus récemment, \cite{F-Z} (voir aussi \cite{F-N} et ses références pour un compte rendu plus détaillé des travaux antérieurs). La raison est précisément à attribuer à la compétition entre la stratification verticale (due à la gravité) et la stabilité horizontale (que la force de Coriolis tend à imposer): dans le régime de forte stratification, les oscillations verticales de la solution (semblent) persister à la limite, et les techniques disponibles ne permettent pas actuellement de traiter ce problème dans toute sa généralité.
Néanmoins, des résultats partiels ont été obtenus dans le cas de données initiales bien préparées, au moyen d'une méthode d'entropie relative: nous nous référons à \cite{F-L-N}
pour le premier résultat, où le mouvement moyen est dérivé, et à \cite{B-F-P} pour une analyse des couches limites d'Ekman dans ce cadre.
 
 \newpage
 
\section*{Aller au-delà de l'échelle critique $Fr=\sqrt{Ma}$}
\`A ce stade, nous sommes intéressés par l'étude des valeurs $Fr$ supérieures à la valeur critique $Fr=\sqrt{Ma}$ considéré dans le paragraphe précédente et nous étudierons d'autres régimes où la stratification a un effet encore plus important.
 
Pour la clarté d'exposition, nous négligeons les effets de la force centrifuge et le processus de transfert de chaleur dans le fluide, en nous concentrant sur le système barotrope classique de Navier-Stokes:
\begin{equation} 
\begin{cases}
\partial_t \vrho + \div (\vrho\vec{u})=0\ \\[2ex] 
 \partial_t (\vrho\vec{u})+ \div(\vrho\vec{u}\otimes\vec{u}) + \dfrac{\e_3 \times \vrho\vec{u}}{Ro}\, + \dfrac{1}{Ma^2} \nabla_x p(\vrho)     
=\div \mbb{S}(\nabla_x\vec{u}) + \dfrac{\vrho}{Fr^2} \nabla_x G\, . \tag{NSC} 
\end{cases}
\end{equation}
Le système plus général présenté dans \eqref{eq_i:NSF_fr} peut être manipulé au prix de technicités supplémentaires déjà évoquées plus en haut (rappelons en particulier la restriction sur le nombre de Mach due à la présence de la force centrifuge). 
Le but est maintenant d'effectuer la limite asymptotique pour le système \eqref{2D_Euler_system} dans les régimes quand on suppose
\begin{equation*}
Ma=\veps^m, \quad \quad Ro=\veps \quad \quad \text{et}\quad \quad Fr=\veps^n
\end{equation*}
avec
\begin{equation*}
\mbox{ bien ou }\qquad m\,>\,1\quad\mbox{ et }\quad m\,<\,2\,n\,\leq\,m+1\,,\qquad\qquad \mbox{ ou }\qquad
m\,=\,1\quad \mbox{ et }\quad \frac{1}{2}\,<\,n\,<\,1\,.
\end{equation*}
La restriction $n<1$ lorsque $m=1$ est imposée afin d'éviter un régime de stratification forte: comme déjà mentionné précédemment, il n'est pas clair, à l'heure actuelle, comment traiter ce cas pour des données initiales générales mal préparés, car toutes les techniques disponibles semblent échouer dans ce cas-là.
En revanche, la restriction $2\,n\leq m+1$ (pour $m>1$) semble être de nature technique. Cependant, elle sort naturellement dans au moins deux points de notre analyse (voir par exemple les sous-sections \ref{ss:ctl1_G} et \ref{sss:term2_G}), et il n’est pas clair si, et comment, il est possible de la contourner et de considérer la gamme de valeurs restante
$(m+1)/2\,<\,n\,<\,m$.
Précisons que, dans nos considérations, la relation $n<m$ est toujours vraie, donc nous travaillerons dans un régime de faible stratification.

Au niveau qualitatif, nos principaux résultats seront assez similaires à ceux présentés dans la partie précédente. En particulier la dynamique limite sera la même, après avoir distingué les deux cas $m>1$ et $m= 1$ (voir théorèmes \ref{th:m>1} et \ref{th:m=1}).
L'essentiel, sur lequel nous mettons l'accent maintenant, est de savoir comment utiliser de manière fine non seulement la structure du système, mais aussi la structure précise de chaque terme pour passer à la limite. Pour être plus précis, le fait de considérer les valeurs de $n$ dépassant le seuil $2n=m$ est rendu possible grâce à des annulations algébriques spéciales faisant intervenir le terme de gravité dans le système des ondes.
Telles annulations sont dues à la forme particulière de la force gravitationnelle, qui ne dépend que de la variable verticale, et elles n'apparaissent pas, en général, si on veut considérer l'action de différentes forces sur le système. Comme on peut facilement le deviner, le cas $2n=m+1$ est plus complexe: en effet, ce choix d'échelle implique la présence dans les calculs d'un terme bilinéaire supplémentaire d'ordre $O(1)$; à son tour, ce terme pourrait ne pas disparaître à la limite, contrairement à ce qui se passe dans le cas $2n<m+1$. Afin de voir que cela ne se produit pas, et que ce terme disparaît bien dans le processus limite, il faut utiliser plus minutieusement la structure du système pour contrôler les oscillations (voir l'équation \eqref{rel_oscillations_bis} et les calculs là-dessous).
 
\section*{Le système d'Euler: le cas incompressible}
Dans la deuxième partie de cette thèse, nous changeons d'orientation en traitant un système incompressible, non visqueux et avec une structure hyperbolique. Plus précisément, nous nous intéressons à décrire l'évolution 2-D d'un fluide qui se déroule suffisamment loin des frontières physiques. Par conséquent, le système de type Euler, dans $\Omega:=\R^2$, est
\begin{equation} \label{full Euler_intro_fr}
\begin{cases}
\d_t \vrho +\div (\vrho \vu)=0\\
\d_t (\vrho \vu)+\div (\vrho \vu \otimes \vu)+ \dfrac{1}{Ro}\vrho \vu^\perp+\nabla_x p=0\\
\div \vu =0\, ,\tag{E}
\end{cases}
\end{equation}
où $\vu^\perp:=(-u^2, u^1)$ est la rotation d'angle $\pi/2$ du champ de vitesse $\vu=(u^1,u^2)$ .
Le terme de pression $\nabla_x p$ représente le multiplicateur de Lagrange associé à la contrainte de divergence nulle sur le champ de vitesse.

La portée principale de cette analyse sera d'étudier le comportement asymptotique du système \eqref{full Euler_intro_fr} lorsque $Ro=\veps\rightarrow 0^+$. 

\subsection*{Résultats connus}
Nous nous limiterons à donner un bref exposé sur les résultats connus concernant les fluides dépendant de la densité.
Nous nous référons plutôt à \cite{C-D-G-G} pour un aperçu de la littérature dans le contexte des fluides homogènes en rotation rapide (voir aussi \cite{B-M-N_EMF} et \cite{B-M-N_AA} pour les études pionnières, concernant les équations homogènes d'Euler et de Navier-Stokes en 3-D).
 
Dans les cas compressibles évoqués au-dessus, le fait que la pression soit une fonction donnée de la densité, implique un double avantage dans l'analyse: d'une part, on peut récupérer de bonnes bornes uniformes pour les oscillations (à partir de l'état de référence) de la densité; de plus, à la limite, on dispose d'une relation flux-fonction entre les densités et les vitesses. 
 
En revanche, bien que la condition d'incompressibilité soit physiquement bien justifiée pour les fluides géophysiques, peu d'études abordent ce cas. Nous nous référons à \cite{Fan-G}, dans lequel Fanelli et Gallagher ont étudié la limite en rotation rapide pour des fluides incompressibles, visqueux et à densité variable. Dans le cas où la densité initiale est une petite perturbation d'un état constant (le cas dit \textit{légèrement non-homogène}), les auteurs ont prouvé la convergence vers un système de type quasi-homogène. Par contre, pour les fluides non-homogènes généraux (le cas dit \textit{totalement non-homogène}), Fanelli et Gallagher ont montré que la dynamique limite est décrite en termes du tourbillon et des fonctions d'oscillation de densité, car il n'y avait pas de régularité suffisante pour prouver la convergence sur l'équation de la quantité de mouvement elle-même (voir plus de détails ci-dessous).
 
Il faut aussi mentionner \cite{C-F_RWA}, où les auteurs prouvent rigoureusement la convergence des équations idéales de la magnétohydrodynamique (MHD) vers un système de type quasi-homogène (voir aussi \cite{C-F_Nonlin} où l'argument de compacité par compensation est adopté). Leur méthode repose sur des inégalités d'entropie relative pour le système primitif qui permettent de traiter également la limite non visqueuse mais nécessite des données initiales bien préparées.
 
\subsection*{Nouveaux résultats}
 
Dans le chapitre \ref{chap:Euler}, nous abordons l'analyse asymptotique (pour $\veps\rightarrow 0^+$) dans le cas d'un système d'Euler dépendant de la densité dans le contexte \textit {légèrement non-homogène}, c'est-à-dire lorsque la densité initiale est une petite perturbation d'ordre $\veps$ d'un profil constant (disons $\overline{\vrho}=1$). Ces petites perturbations autour d'un état de référence constant sont physiquement justifiées par l’\textit {approximation de Boussinesq} (voir par exemple le chapitre 3 de \cite{C-R} ou le chapitre 1 de \cite{Maj} à cet égard). En effet, puisque l'état constant $\overline{\vrho}=1$ est transporté par un champ vectoriel à divergence nulle, la densité peut s'écrire comme $\vrho_\veps=1+\veps R_\veps$ pour tous les temps (à condition que cela soit vrai à $t=0$ ), où l'on peut énoncer de bonnes bornes uniformes sur $R_\veps$. Nous soulignons aussi que dans l'équation de quantité de mouvement de \eqref{full Euler_intro_fr}, avec l'échelle $Ro=\veps$, le terme de Coriolis peut être réécrit comme
\begin{equation}\label{Coriolis_fr}
 \frac{1}{\veps}\vrho_\veps \vec u_\veps^\perp=\frac{1}{\veps}\vec u_\veps^\perp+R_\veps\vec u_\veps^\perp \, .\tag{2D COR}
\end{equation}
On remarque que, grâce à la condition d'incompressibilité, le premier terme du membre de droite de \eqref{Coriolis_fr} est en fait un gradient: il peut être ``absorbé'' dans le terme de pression, qui doit s'échelonner comme $1/\veps$. En fait, la seule force qui peut compenser l'effet de la rotation rapide dans le système \eqref{full Euler_intro_fr} est, à l'échelle géophysique, le terme de pression: c'est-à-dire qu'on peut écrire $\nabla_x p_\veps= (1/\veps)\, \nabla_x \Pi_\veps$.

Précisons que le cas \textit{totalement non-homogène} (où la densité initiale est une perturbation d'un état arbitraire) est hors de notre étude. Ce cas est plus complexe et de nouveaux problèmes techniques surviennent dans l'analyse du caractère bien posé et dans l'inspection asymptotique. En effet, comme déjà souligné dans \cite{Fan-G} pour le système de Navier-Stokes-Coriolis, la dynamique limite est décrite par un système sous-déterminé qui mélange le tourbillon et les fluctuations de densité.
Afin de décrire la dynamique limite complète (où les variations de densité limite et les vitesses limites sont découplées), il fallait supposer des bornes \textit {a priori} plus fortes que celles qui pourraient être obtenues par des estimations d'énergie classiques. Néanmoins, la haute régularité impliquée n'est \textit{pas} propagée uniformément en $\veps$ en général, à cause de la présence du terme de Coriolis.
En particulier, la structure du terme de Coriolis est plus compliquée que celle montrée en \eqref{Coriolis_fr}, puisqu'on a $\vrho_{\veps}=\oline \vrho+ \veps \sigma_\veps$ (avec $\sigma_ \veps$ la fluctuation), si au départ on suppose $\vrho_{0, \veps}=\oline \vrho+ \veps R_{0,\veps}$ avec $\oline \vrho$ l'état de référence arbitraire. \`A ce stade, si on insère la décomposition précédente de $\vrho_\veps$ dans \eqref{Coriolis_fr}, un terme de la forme $(1/\veps)\, \oline \vrho \ue^\perp$ apparaît: ce terme est source de difficultés pour propager les estimations de haute régularité nécessaires.

De manière équivalente, si on essaie de diviser l'équation de quantité de mouvement dans \eqref{full Euler_intro_fr} par la densité $\vrho_\veps$, alors le problème précédent ne se traduit que sur l'analyse du terme de pression, qui devient $1/(\veps \vrho_\veps)\, \nabla_x \Pi_\veps$.

\`A la lumière de toute la discussion qui précède, signalons maintenant les principales difficultés qui se posent dans notre problème.
 
Tout d'abord, notre modèle est un système non visqueux et de type hyperbolique pour lequel nous ne pouvons \textit{pas} nous attendre à des effets de lissage et de gain de régularité. Pour cette raison, il est naturel de regarder les équations \eqref{full Euler_intro_fr} dans un cadre régulier comme les espaces $H^s$ avec $s>2$. Les espaces de Sobolev $H^s(\R^2)$, pour $s>2$, sont en fait plongés dans l'espace $W^{1,\infty}$ des fonctions globalement Lipschitziennes: c'est une exigence minimale pour préserver la régularité initiale (voir par exemple le chapitre 3 de \cite{B-C-D} et aussi \cite{D_JDE}, \cite{D-F_JMPA} pour une large discussion sur ce sujet).
En fait, tous les espaces de Besov $B^s_{p,r}(\R^d)$ qui sont inclus dans $W^{1,\infty}(\R^d)$, un fait qui se produit pour $( s,p,r)\in \R\times [1,+\infty]^2$ tels que
\begin{equation} \label{Lip_assumption_fr}
s>1+\frac{d}{p} \quad \quad \quad \text{ou}\quad \quad \quad s=1+\frac{d}{p} \quad \text{et}\quad r=1\, ,\tag{LIP}
\end{equation}
sont de bons candidats pour l'analyse du caractère bien posé (reportez-vous à l'annexe \ref{app:Tools} pour plus de détails). Cependant, le choix de travailler dans $H^s\equiv B^s_{2,2}$ est dicté par la présence de la force de Coriolis: nous exploiterons en profondeur l'antisymétrie de ce terme singulier. 
 
Par ailleurs, le fluide est supposé incompressible, de sorte que le terme de pression n'est qu'un multiplicateur de Lagrange et ne donne \textit{pas} d'informations sur la densité, contrairement au cas compressible. De plus, à cause de la non-homogénéité, l'analyse du gradient de pression est beaucoup plus complexe puisqu'on doit étudier une équation elliptique à coefficients \textit{non constants}, c’est-à-dire
\begin{equation}\label{elliptic_eq_introduction_fr}
-\div (A \, \nabla_x p)=\div F \quad \text{où}\quad \div F:=\div \left(\vu \cdot \nabla_x \vu+ \frac{1}{Ro} \vu^\perp \right)\quad \text{et}\quad A:=1/\vrho \, . \tag{ELL}
\end{equation}
La principale difficulté est d'obtenir des bornes \textit{uniformes} appropriées (par rapport au paramètre de rotation) pour le terme de pression dans le cadre régulier que nous considérerons. 
Nous renvoyons à \cite{D_JDE} et \cite{D-F_JMPA} pour plus de détails concernant les problèmes qui se posent dans l'analyse de l'équation elliptique \eqref{elliptic_eq_introduction_fr} dans les espaces $B^s_{p,r}$.

Après avoir analysé le terme de pression, nous montrons que le système \eqref{full Euler_intro_fr} est localement bien posé dans le cadre $H^s$. On note que, dans le théorème pour le caractère bien posé local (ci-dessous), toutes les estimations sont \textit{uniformes} par rapport au paramètre de rotation et, en plus, nous avons que le temps d'existence est indépendant de $\veps$.
\begin{theoremnonum_fr}\label{thm3_intro_french}
Soit $s>2$. Pour tout $\veps>0$, 
il existe un temps $T_\veps^\ast >0$
tel que 
le système \eqref{full Euler_intro_fr} a une solution unique $(\vrho_\veps, \vu_\veps, \nabla_x \Pi_\veps)$ où
\begin{itemize}
\item $\vrho_\veps$ appartient à l'espace $C^0([0,T_\veps^\ast]\times \R^2)$ avec $\nabla_x \vrho_\veps \in C^0([ 0,T_\veps^\ast]; H^{s-1}(\R^2))$;  
\item $\vu_\veps$ et $\nabla_x \Pi_\veps$ appartiennent à l'espace $C^0([0,T_\veps^\ast]; H^s(\R^2))$.
\end{itemize}
Par ailleurs, 
$$\inf_{\veps>0}T_\veps^\ast >0\, .$$
\end{theoremnonum_fr}
Une fois que nous avons énoncé le résultat du caractère bien posé local dans le théorème \ref{thm3_intro_french}, on passe à la limite en rotation rapide pour des données initiales générales \emph{mal préparées}. Nous prouvons la convergence du système \eqref{full Euler_intro_fr} vers ce que nous appelons système d’Euler incompressible \textit{quasi-homogène}
\begin{equation} \label{Q-H_E_intro_fr}
\begin{cases}
\d_t R+\div (R\vu)=0 \\
\d_t \vec u+\div \left(\vec{u}\otimes\vec{u}\right)+R\vu^\perp+ \nabla_x \Pi =0 \\
\div \vec u\,=\,0\,, \tag{QHE}
\end{cases}
\end{equation} 
où $R$ représente la limite des fluctuations $R_\veps$.
Signalons aussi que dans l'équation de quantité de mouvement de \eqref{Q-H_E_intro_fr} apparaît un terme non linéaire d'ordre inférieur (i.e. $R\vec u^\perp$): c'est une sorte de reste dans la convergence pour le terme de Coriolis, défini dans \eqref{Coriolis_fr}.

Le passage à la limite dans l'équation de la quantité de mouvement \eqref{full Euler_intro_fr} n'est plus évident, bien que nous soyons dans le cadre $H^s$: le terme de Coriolis est responsable de fortes oscillations en temps des solutions (les \textit{ondes de Poincar\'e}) qui peuvent empêcher la convergence du terme convectif vers celui de \eqref{Q-H_E_intro_fr}. Pour surmonter ce problème, nous utilisons la même approche mentionnée au-dessus: la technique de compacité par compensation.
 
Par ailleurs, il est intéressant de noter que le caractère bien posé global du système \eqref{Q-H_E_intro_fr} reste un problème ouvert. Cependant, \textit{grosso modo}, pour $R_0$ assez petit, le système \eqref{Q-H_E_intro_fr} est ``proche'' du système 2-D d'Euler homogène et incompressible, pour lequel il est bien connu le caractère bien posé global. Ainsi, il est naturel de se demander s'il existe un résultat de caractère bien posé ``asymptotiquement global'' dans l'esprit de \cite{D-F_JMPA} et \cite{C-F_sub}: pour de petites fluctuations initiales $R_0$, le système quasi-homogène \eqref{Q-H_E_intro_fr} se comporte comme les équations d'Euler standards et la durée de vie de ses solutions tend vers l'infini. En particulier, comme déjà montré dans \cite{C-F_sub} pour le système MHD idéal et quasi-homogène (voir aussi la bibliographie de \cite{C-F_sub}), on prouve que la durée de vie des solutions de \eqref{Q-H_E_intro_fr} satisfait 
\begin{equation} \label{lifespan_Q-H_fr}
T_\delta^\ast \sim \log \log \frac{1}{\delta}\,,\tag{LIFE}
\end{equation}
où $\delta>0$ est la taille des fluctuations initiales.
 
Ce résultat pour le temps d'existence des solutions de \eqref{Q-H_E_intro_fr} pousse notre attention vers l'étude de la durée de vie des solutions du système primitif \eqref{full Euler_intro_fr}.
Pour le système d'Euler en 3-D \textit {homogène} avec la force de Coriolis, Dutrifoy dans \cite{Dut} a prouvé que la durée de vie des solutions tend vers l'infini dans le régime de la rotation rapide (voir aussi \cite{Gall}, \cite{Cha} et \cite{Scro}, où les auteurs ont inspecté la durée de vie des solutions dans le contexte de fluides homogènes et visqueux). Pour le système \eqref{full Euler_intro_fr} il n'est pas clair comment trouver des effets de stabilisation similaires (dus au terme de Coriolis), afin d'améliorer la durée de vie des solutions: c’est-à-dire montrer que $T_\veps^\ast \longrightarrow +\infty$ quand $\veps\rightarrow 0^+$. Néanmoins, indépendamment des effets rotationnels, nous sommes en mesure d'énoncer un résultat de caractère bien posé ``asymptotiquement global'' dans le régime d'oscillations \textit{petites}, au sens de \eqref{lifespan_Q-H_fr}: à savoir, quand la taille de la fluctuation initiale $R_{0,\veps}$ est assez petite, de taille $\delta >0$, la durée de vie $T^\ast_\veps$ de la solution correspondante du système \eqref{full Euler_intro_fr} est bornée inférieurement par $T^\ast_\veps\geq T^\ast(\delta)$, avec $T^\ast (\delta)\longrightarrow +\infty$ quand $\delta\rightarrow 0^+ $ (voir aussi \cite{D-F_JMPA} pour un fluide dépendant de la densité mais en absence de la force de Coriolis). Plus précisément, on a le résultat suivant.
\begin{propnonum}
La durée de vie $T_\veps^\ast$ de la solution des équations d'Euler incompressibles dépendant de la densité en deux dimensions \eqref{full Euler_intro_fr} avec la force de Coriolis est bornée inférieurement par
\begin{equation} \label{improv_life_fullE_intro_fr}
\frac{C}{\|\vec u_{0,\veps}\|_{H^s}}\log\left(\log\left(\frac{C\, \|\vec u_{0, \veps}\|_{H^s}}{\max \{\mc A_\veps(0),\, \veps \, \mc A_\veps(0)\, \|\vec u_{0, \veps}\|_{H^s}\}}+1\right)+1\right)\, ,\tag{BOUND}
\end{equation}
où $\mc A_\veps (0):= \|\nabla_x R_{0,\veps}\|_{H^{s-1}}+\veps\, \|\nabla_x R_{0,\veps }\|_{H^{s-1}}^{\lambda +1}$, pour un $\lambda\geq 1$ approprié.
\end{propnonum}
Comme corollaire immédiat de la minoration précédente, si on considère des densités initiales de la forme $\vrho_{0,\veps}=1+\veps^{1+\alpha}R_{0,\veps}$ avec $ \alpha >0$, alors on obtient $T^\ast_\veps\sim \log \log (1/\veps)$.

De suite, on va illustrer schématiquement les principales étapes pour montrer \eqref{improv_life_fullE_intro_fr} pour le système primitif \eqref{full Euler_intro_fr}.
 
Le point clé de la preuve de \eqref{improv_life_fullE_intro_fr} est d'étudier la durée de vie des solutions dans des espaces critiques de Besov. Dans ces espaces, on peut profiter du fait que, lorsque $s=0$, la norme $B^0_{p,r}$ des solutions peut être bornée \textit{linéairement} par rapport à la norme de Lipschitz de la vitesse, plutôt qu'exponentiellement (voir les travaux \cite{Vis} de Vishik et \cite{H-K} de Hmidi et Keraani). Puisque le triplet $(s,p,r)$ doit satisfaire \eqref{Lip_assumption_fr}, l'espace de Besov de régularité la plus basse que nous puissions atteindre est $B^1_{\infty,1}$. Par conséquent, si $\vec u$ appartient à $B^1_{\infty,1}$, le tourbillon $\omega := -\d_2 u^1+\d_1 u^2$ a la régularité désirée pour appliquer les estimations améliorées de Hmidi-Keraani et Vishik (voir le théorème \ref{thm:improved_est_transport} à cet égard).

En analysant la formulation en tourbillon du système, on découvre que l'opérateur \emph{curl} annule les effets singuliers produits par la force de Coriolis: cette annulation n'est pas apparente, parce que la propriété antisymétrique du terme de Coriolis est hors d'usage dans le cadre critique considéré.

Enfin, nous avons besoin d'un critère de continuation (dans l'esprit du critère de Baele-Kato-Majda, voir \cite{B-K-M}) qui garantit qu’on puisse ``mesurer'' la durée de vie des solutions dans l'espace d'indice de régularité plus faible, $ s=r=1$ et $p=+\infty$. Ce critère est valable sous l'hypothèse que
$$
\int_0^{T} \big\| \nabla_x \vec u(t) \big\|_{L^\infty} \dt < +\infty\qquad \text{avec}\qquad T<+\infty\, .   
$$
Le critère précédent assure que la durée de vie des solutions trouvées dans les espaces critiques de Besov est la même que dans le cadre fonctionnel de Sobolev recherché: ce qui nous permet de conclure la preuve (voir les considérations dans la sous-section \ref{ss:cont_criterion+consequences}).
 
\medskip

\subsection*{Aperçu du contenu de cette thèse}
 
Avant de poursuivre, nous donnons un bref aperçu de la structure de la présente thèse.
 
Le chapitre \ref{chap:geophysics} a pour objectif de ``plonger'' le lecteur dans la discipline de la dynamique géophysique des fluides, en donnant une justification physique succincte des modèles mathématiques que nous considérerons dans les prochains chapitres. Dans le chapitre \ref{chap:multi-scale_NSF} nous abordons l'étude du problème de perturbation singulière, donné par les équations de Navier-Stokes-Fourier, dans l'échelle que nous appelons ``critique''. Le chapitre \ref{chap:BNS_gravity} est consacré à l'amélioration de l'échelle précédente: nous dépasserons le seuil ``critique''. Enfin, dans le dernier chapitre \ref{chap:Euler} nous modifions un peu le modèle et nous allons travailler sur l'analyse asymptotique des équations d'Euler dépendant de la densité. De plus, nous nous concentrerons sur l’étude de la durée de vie de ses solutions, prouvant un résultat de caractère bien posé ``asymptotiquement global''.
 
\`A la fin de cette thèse, se trouvent deux autres sections consacrées aux perspectives d'avenir et une annexe contenant quelques outils et des résultats célèbres utilisés dans le manuscrit.


\newpage
\null
\thispagestyle{empty}
\newpage


\mainmatter
\chapter{Some geophysical considerations}\label{chap:geophysics}

\begin{quotation}

The scope of this chapter is to introduce the mathematical features of geophysical flows. The main reference is the book \cite{C-R} (see also \cite{K-C-D}, \cite{Ped}, \cite{Val}, \cite{Z}). We will briefly discuss the physical motivations of the mathematical models we will consider in the next chapters. The equations presented in the following paragraphs will be derived mainly from physical considerations. For this reason, the functions which will appear in the sequel have to be considered smooth. 

\medbreak

Let us give an overview of the chapter. First of all, after a brief introduction (see Section \ref{sect:intro}), we present the two main characters which influence the dynamics of the geophysical fluids: the rotation (see Section \ref{sec:rotation}) and the stratification (we refer to Section \ref{sec:strat}). In Section \ref{sec:budgets} we show how to derive, from the physical point of view, the budget equations and Section \ref{sec:Bous_app} is instead devoted to the Boussinesq approximation. Next, in Section \ref{sec:scales_motion}, we perform a scale analysis and we define some important dimensionless numbers. Section \ref{sec:TP_thm} is dedicated to the celebrated Taylor-Proudman theorem. We conclude this chapter talking about stratified and quasi-incompressible fluids (see Section \ref{sec:strat_q-i}), and rewriting the Navier-Stokes system in its dimensionless form (Section \ref{sec:NS_dimenless}). 
\end{quotation}

\section{The geophysical fluid flows}\label{sect:intro}
The geophysical fluid dynamics (GFD) studies the naturally occurring flows on large scales that mostly take place on Earth but also on other planets or stars. The discipline encompasses the motion of both fluid phases: liquids (e.g. waters in the ocean) and gases (e.g. air in the Earth's atmosphere or in other planets). In addition, it is on large scales that the common features of atmospheric and oceanic dynamics come to light. In most of the problems concerning GFD, either the ambient rotation (of Earth, planets, stars) or density differences (warm and cold air masses, fresh and saline water) or both assume a relevant importance. Typical problems arising in geophysical fluid dynamics are, for example, the variability of the atmosphere (weather and climate dynamics), of the ocean (waves, vortices and currents) and vortices on other planet (Jupiter's Great Red Spot, see Figure \ref{fig:Jupiter}), and convection in stars. 
The effects of rotation and those of stratification distinguish the GFD from the traditional fluid mechanics. The fact that the ambient is rotating (e.g. the Earth's rotation around its axis) introduces in the equations the presence of two acceleration terms that, in view of the Newton's second law of motion, can be interpreted as forces: the Coriolis force and the centrifugal force. On the one hand, although the centrifugal effects are more palpable (on a planetary scale), they play a negligible role in the dynamics. On the other hand, the less intuitive Coriolis force turns out to be a crucial character in describing the behaviour of geophysical motions. The major effect of the Coriolis force is to impose certain vertical rigidity to the fluid: if the Coriolis effect is strong enough, we could observe that the homogeneous flow displaces itself in vertical columns: the particle along the same vertical move together and retain their alignment over long periods of time (e.g. currents in the Western North Atlantic). 
This property is known as \emph{Taylor-Proudman theorem}. That result was firstly derived by S. S. Hough in 1897, but was named after the works (in 1916-1917) of G. I. Taylor and J. Proudman. Five years later, G. I. Taylor verified experimentally such a property.

\begin{figure}[htbp]
\centering
\includegraphics[scale=0.15]{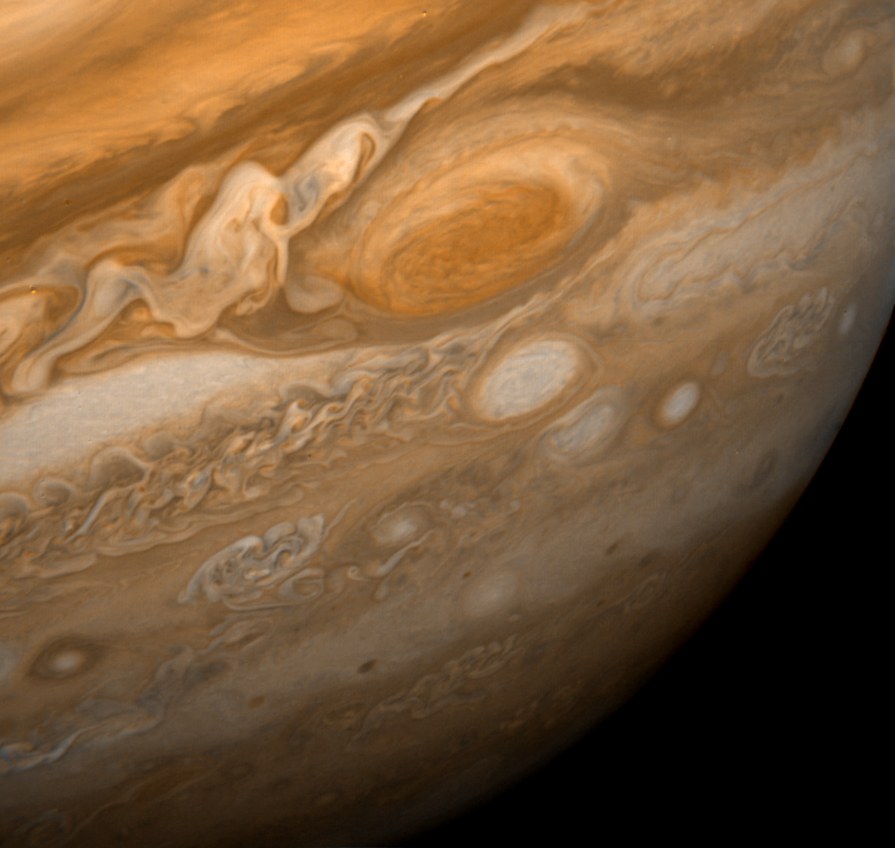}
\caption{Jupiter's Great Red Spot (1979)}\label{fig:Jupiter}
\end{figure}


In the large scale atmospheric and oceanic flows, the previous state of perfect vertical rigidity is not realized due to the fact that the rotation is not sufficiently fast and due to the appearance of stratification, i.e. density variations. The cause of those vertical effects is attributable to the presence of the gravitational force, which tends to lower the regions of the fluid with heavier density and to raise the lightest. Under equilibrium conditions, the fluid is stably stratified in stacked horizontal layers of decreasing density. However, the fluid motions disturb this equilibrium that gravity tends to restore. 

We conclude this part pointing out that the advances in GFD touch considerably our real life. The progress in the ability to predict with some confidence the paths of hurricanes has led the creation of warning systems that have saved and will save numerous lives in sea and coastal areas (e.g. we can think to the Hurricane Frances in 2004, see Figure \ref{fig:Hurricane}).
\begin{figure}[htbp]
\centering
\includegraphics[scale=0.15]{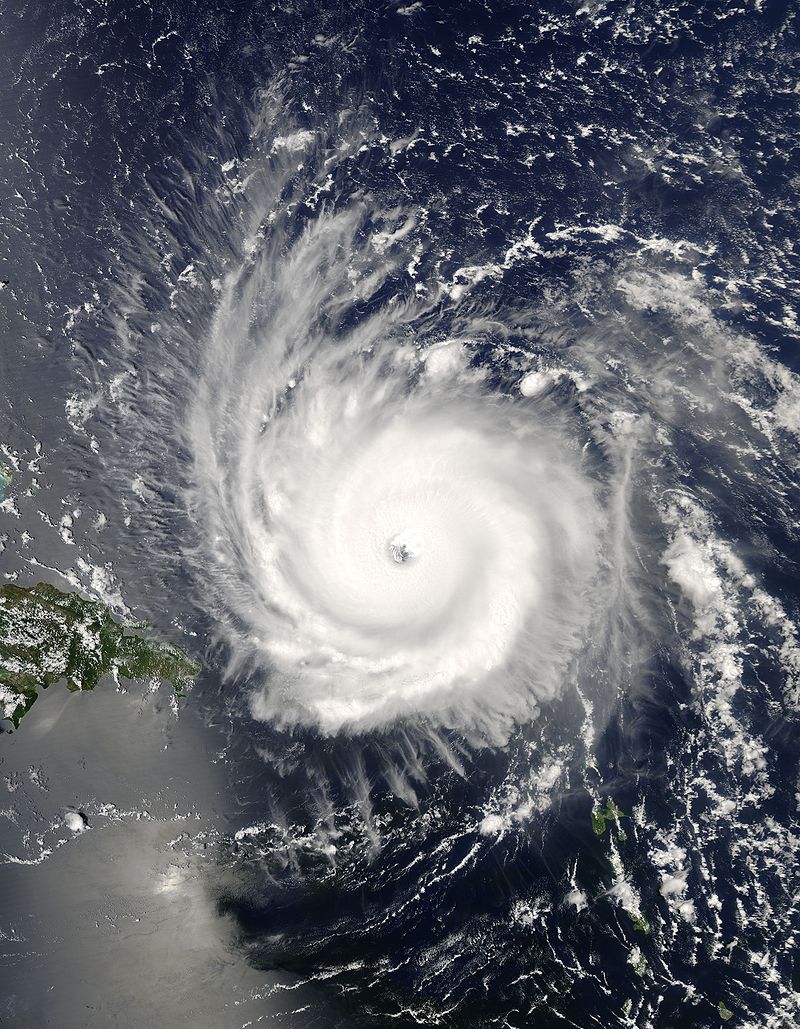}
\caption{Hurricane Frances (2004)}\label{fig:Hurricane}
\end{figure}

Another fundamental aspect is that the combined dynamics of atmosphere and oceans contribute to the global climate. The behaviour of the atmosphere modulates, for example, the agricultural success and the ocean currents affect navigation, fisheries and disposal of pollution. Thus, understanding and reliably predicting of geophysical events and trends are scientific, economic, humanitarian and political priorities. 

\newpage

\section{First character: the rotation} \label{sec:rotation}
We are now interested in which scales the ambient rotation is no more negligible in the fluid dynamics. For that reason, we introduce the following criterion considering the velocity $U$ and length $L$ scales of motion. If a particle at speed $U$ covers the distance $L$ in a time larger than or comparable to a rotation period (of the Earth, for example), we can imagine that the trajectory is influenced by the ambient rotation. Therefore, we write 
$$ \oline\veps:=\frac{\text{time of one revolution}}{\text{time taken by a particle to cover }L \text{ at } U}=\frac{2\pi/\underline{\Omega}}{L/U}=\dfrac{2\pi U}{\underline{\Omega}L} \, ,$$
where $\underline{\Omega}:=\frac{2\pi}{\text{time of one revolution}}$ is the ambient rotation rate. 

If $\oline \veps\lesssim 1$, then we can conclude that the rotation is important. In geophysical flows the previous inequality holds, since e.g. an ocean current usually flows at 10 cm/s over a distance of 10 km or a wind blows at 10 m/s in an anticyclonic formation 1000 km wide. 
\subsection{The Coriolis force} 
In this paragraph, we give a short mathematical inspection about the rotating framework of reference. To simplify the exposition, we focus on the two-dimensional case. Let $X^1$ and $X^2$ be the axes of the inertial framework of reference and $x^1$, $x^2$ be those of the rotating framework with angular velocity $\underline{\Omega}$. We denote by $\vec I$, $\vec J$ and $\vec i$, $\vec j$ the corresponding unit vectors (see Figure \ref{fig:ref} below). 
\begin{figure}[htbp]
\centering
\includegraphics[scale=0.2]{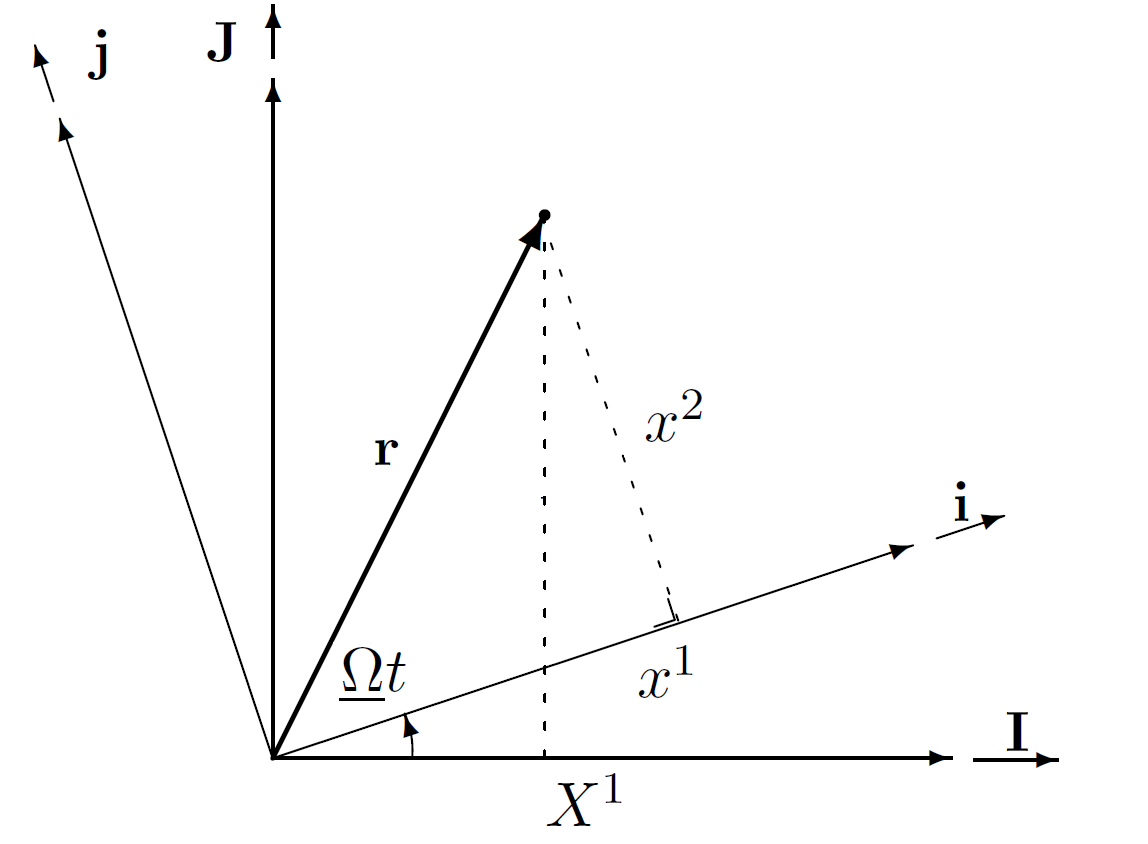}
\caption{Inertial framework versus rotating framework}\label{fig:ref}
\end{figure}

Then, it follows that
\begin{equation*}
\begin{split}
\vec I&=\vec i \cos (\underline{\Omega}t)-\vec j\sin (\underline{\Omega}t)\\
\vec J&=\vec i \sin (\underline{\Omega} t)+\vec j \cos(\underline{\Omega}t)
\end{split}
\end{equation*}
and the position vector is defined as 
\begin{equation}\label{definition_r}
\begin{split}
\vec r&=X^1\, \vec I+X^2\, \vec J\\
&=x^1\, \vec i+x^2\, \vec j\, .
\end{split}
\end{equation}
Thus, it is easy to find that 
\begin{equation*}
\begin{split}
x^1&=X^1\cos (\underline{\Omega}t)+X^2\sin (\underline{\Omega}t)\\
x^2&=-X^1\sin (\underline{\Omega}t)+X^2\cos (\underline{\Omega}t)\, .
\end{split}
\end{equation*}
At this point, taking the first derivative in time yields:
\begin{equation*}
\begin{split}
\frac{dx^1}{dt}&=\frac{dX^1}{dt}\cos (\underline{\Omega}t)+\frac{dX^2}{dt}\sin (\underline{\Omega}t)+\underline{\Omega}x^2\\
\frac{dx^2}{dt}&=-\frac{dX^1}{dt}\sin (\underline{\Omega}t)+\frac{dX^2}{dt}\cos (\underline{\Omega}t)-\underline{\Omega}x^1\, .
\end{split}
\end{equation*}
The previous expressions are the components of the relative velocity:
\begin{equation*}
\vec u =\frac{dx^1}{dt}\, \vec i+\frac{dx^2}{dt}\, \vec j=u^1\, \vec i+u^2 \, \vec j\, .
\end{equation*}
Similarly the absolute velocity is defined as 
\begin{equation*}
\vec U=\frac{dX^1}{dt}\, \vec I+\frac{dX^2}{dt}\, \vec J\, .
\end{equation*}
Rewriting the absolute velocity in terms of the rotating framework, we get 
\begin{equation*}
\begin{split}
\vec U&=\left(\frac{dX^1}{dt}\cos (\underline{\Omega}t)+\frac{dX^2}{dt}\sin (\underline{\Omega}t)\right)\, \vec i+ \left(-\frac{dX^1}{dt}\sin (\underline{\Omega}t)+\frac{dX^2}{dt}\cos (\underline{\Omega}t)\right)\, \vec j\\
&=U^1\, \vec i+ U^2\, \vec j\, .
\end{split}
\end{equation*}
Then, comparing the absolute and relative velocities, one has 
\begin{equation*}
\begin{split}
U^1&=u^1-\underline{\Omega}x^2\\
U^2&=u^2+\underline{\Omega}x^1\, .
\end{split}
\end{equation*}
This means that the absolute velocity is the relative velocity with in addition the entrainment speed caused by the ambient rotation. In a similar manner, we can deduce that 
\begin{equation*}
\frac{d^2x^1}{dt^2}=\left(\frac{d^2X^1}{dt^2}\cos (\underline{\Omega}t)+\frac{d^2X^2}{dt^2}\sin (\underline{\Omega}t)\right)+ 2\underline{\Omega}U^2-|\underline{\Omega}|^2x^1
\end{equation*}
and 
\begin{equation*}
\frac{d^2x^2}{dt^2}=\left(-\frac{d^2X^1}{dt^2}\sin (\underline{\Omega}t)+\frac{d^2X^2}{dt^2}\cos (\underline{\Omega}t)\right)- 2\underline{\Omega}U^1-|\underline{\Omega}|^2x^2\, .
\end{equation*}
In terms of acceleration, we have 
\begin{equation*}
\begin{split}
\vec a &=\frac{d^2x^1}{dt^2}\, \vec i+\frac{d^2x^2}{dt^2}\, \vec j=a^1\, \vec i+a^2\, \vec j\\
\vec A&=\frac{d^2X^1}{dt^2}\, \vec I+\frac{d^2X^2}{dt^2}\, \vec J=A^1\, \vec i+A^2\, \vec j
\end{split}
\end{equation*}
and so 
\begin{equation*}
\begin{split}
A^1&=a^1-2\underline{\Omega}u^2-|\underline{\Omega}|^2x^1\\
A^2&=a^2+2\underline{\Omega}u^1-|\underline{\Omega}|^2x^2\, .
\end{split}
\end{equation*}
Now, we notice that the absolute acceleration differs from the relative acceleration for two contributions: the term proportional to $\underline{\Omega}$ and the relative velocity, which is called \emph{Coriolis acceleration}; the term proportional to $|\underline{\Omega}|^2$ and the relative coordinates, i.e. the \emph{centrifugal acceleration}. The centrifugal force acts as an outward pull, whereas the Coriolis force depends on the relative speed. 

In three-dimensions, one can repeat the above computations deriving 
\begin{equation*}
\begin{split}
\vec U&=\vec u+ \underline{\vec \Omega}\times \vec r\\
\vec A&=\vec a+2 \underline{\vec \Omega}\times \vec u+ \underline{\vec \Omega}\times \left( \underline{\vec \Omega}\times \vec r\right)\, ,
\end{split}
\end{equation*}
where the symbol $\times$ stands for the external products by vectors in $\R^3$ and $\underline{\vec \Omega}=\underline{\Omega}\, \vec k$ (with $\vec k$ the unit vector in the third dimension). This means that if we would to take a derivative in time in the inertial framework, we have to apply 
$$ \frac{d}{dt}+\underline{\vec \Omega}\times $$
in the rotating framework of reference. 
\subsection{The centrifugal force}
Unlike the Coriolis force which is proportional to the velocity (as we have seen above), the centrifugal force depends on the rotation rate and the distance of the particle to the rotation axis. The centrifugal force is responsible for the slightly flattened shape of the planets.

\begin{figure}[htbp]
\centering
\includegraphics[scale=0.2]{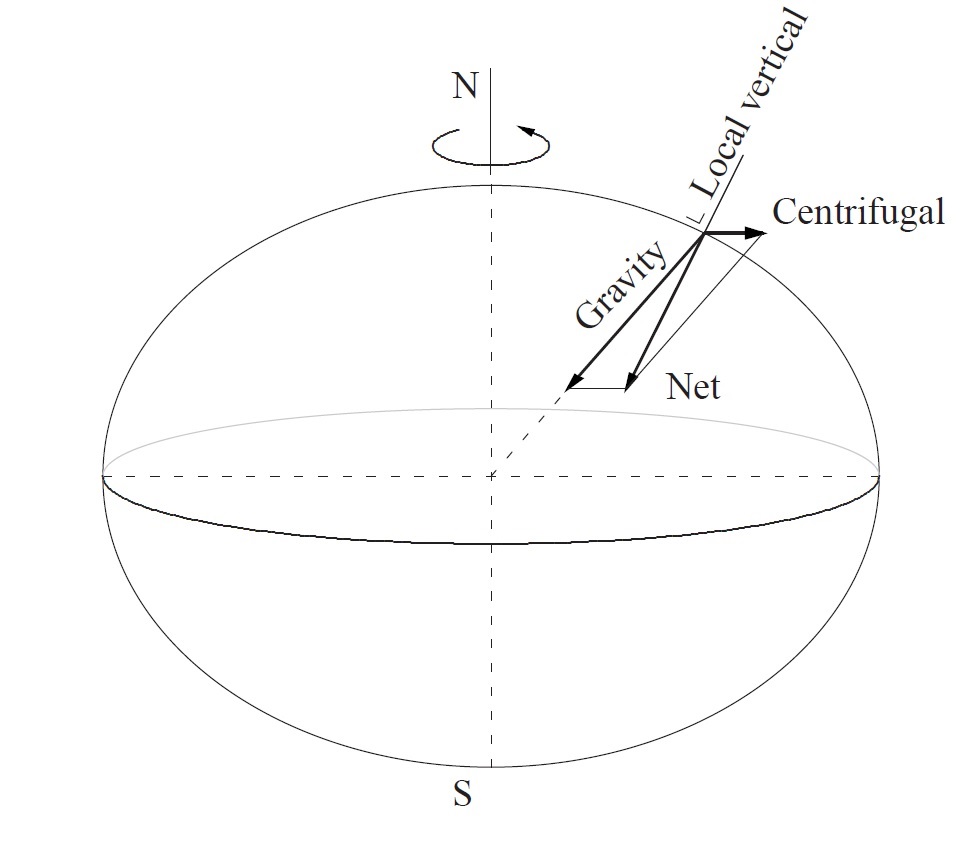}
\caption{The effects of the centrifugal force}\label{fig:centrifugal-force}
\end{figure}

For example, due to the centrifugal effects, the terrestrial equatorial radius is 6378 km, slightly greater than its polar radius of 6357 km. Moreover, the centrifugal effects cause an outward pull to particles, that in any case they don't fly out to space thanks to the gravity. However, the centrifugal force affects the gravity: it shifts the direction of the gravity away from the Earth's center, thus weakening the gravitational effects. 

\section{Second character: the stratification effects}\label{sec:strat}
As mentioned above, geophysical fluids consist of fluid masses of different densities, that the gravitational action tends to arrange in horizontal layers. On the contrary, the motion disturbs this equilibrium raising the dense zones and lowering the lighter ones. In this way, the potential energy increase at the cost of decreasing the kinetic energy and thus the flow slows down. Therefore, the importance of stratification can be evaluated in terms of potential and kinetic energies. We denote by $\Delta \vrho$ the scale of density variations and $H$ is its height scale. We perturb the stratification, raising a fluid particle of density $\vrho_0+\Delta \vrho$ over the height $H$ and lowering a fluid element of density $\vrho_0$ over the same $H$. The corresponding change in potential energy, per unit of volume, is 
\begin{equation*}
(\vrho_0 +\Delta \vrho)gH-\vrho_0gH=gH\,\Delta \vrho\, .
\end{equation*} 
Now, we define
\begin{equation*}
\oline \sigma:=\dfrac{\frac{1}{2}\vrho_0|U|^2}{gH\, \Delta \vrho}\, ,
\end{equation*}
which is the comparative energy ratio between the kinetic energy (per unit of volume) $\frac{1}{2}\vrho_0|U|^2$ and the potential energy. Therefore, if $\oline \sigma\lesssim 1$ the stratification effects cannot be ignored in the dynamics of the fluid. 

In geophysical flows, an interesting situation is when rotation and stratification effects are both important, i.e. $\oline \veps\sim 1$ and $\oline \sigma \sim 1$. This implies that 
\begin{equation*}
L\sim \frac{U}{\underline{\Omega}}\quad \quad \text{and}\quad \quad U\sim \sqrt{\frac{\Delta \vrho}{\vrho_0}gH}\, .
\end{equation*}
In this way, we have a fundamental length scale 
$$ L\sim \frac{1}{\underline{\Omega}}\sqrt{\frac{\Delta \vrho}{\vrho_0}gH}\, . $$
On the Earth, the typical length and velocity scales for the atmosphere and oceans are, respectively, 
\begin{equation*}
\begin{split}
L_{\rm atm}\sim 500\; {\rm km}\quad \quad &\text{and}\quad \quad U_{\rm atm}\sim 30\; {\rm m/s}\\
L_{\rm oc}\sim 60\; {\rm km}\quad \quad &\text{and}\quad \quad U_{\rm oc}\sim 4\; {\rm m/s}\, .
\end{split}
\end{equation*}
We point out that, in general, the oceanic motions are slower and slightly more confined than the atmospheric flows. 
\section{Mass, momentum, energy and entropy budgets}\label{sec:budgets}
The object of this section is to establish, via physical considerations, the equations governing the movement of geophysical flows.

\subsection{The continuity equation}
We consider an infinitesimal cube with volume $\Delta V=\Delta x^1 \Delta x^2 \Delta x^3$ that is fixed in space. 
\begin{figure}[htbp]
\centering
\includegraphics[scale=0.3]{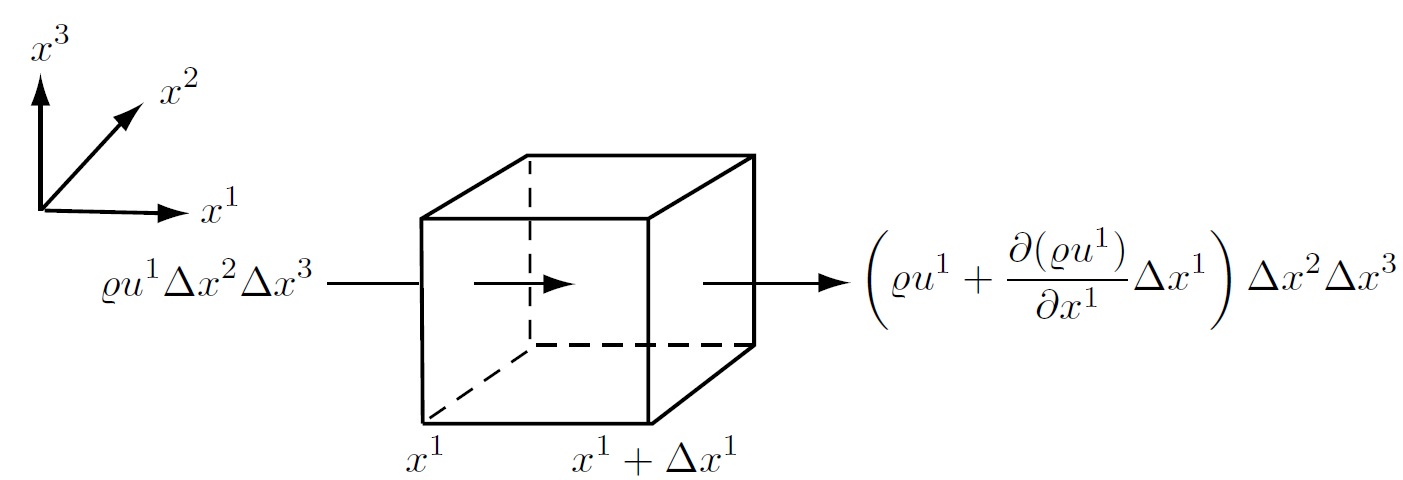}
\caption{Mass conservation in an infinitesimal cube}\label{fig:cube}
\end{figure}

The fluid crosses the cube in the $x^1$-direction, passing through the faces in the $x^2$-$x^3$ plane of area $\Delta A=\Delta x^2 \Delta x^3$. The accumulation of the fluid in-out, in the $x^1$-direction, is:
\begin{equation*}
\Delta x^2 \Delta x^3 \left[(\vrho u^1)(x^1,x^2,x^3)-(\vrho u^1)(x^1+\Delta x^1, x^2,x^3)\right]=-\frac{\d (\vrho u^1)}{\d x^1}(x^1,x^2,x^3)\, \Delta x^1 \Delta x^2 \Delta x^3\, ,
\end{equation*}
where $\vrho$ is the density of the fluid (in kg/$\rm m^3$) and $u^1$ (in m/s) is the first component of the flow velocity $\vec u=(u^1,u^2,u^3)$.
Similarly for the $x^2$ and $x^3$-directions, we have
$$ -\left[\frac{\d(\vrho u^2)}{\d x^2}+\frac{\d (\vrho u^3)}{\d x^3}\right]\Delta x^1 \Delta x^2 \Delta x^3\, . $$
This net accumulation of the fluid must be accompanied by an increase of fluid mass within the volume, represented by 
$$ \frac{\d \vrho}{\d t}\Delta x^1 \Delta x^2 \Delta x^3\, . $$
Since the mass is conserved, one has 
\begin{equation*}
\Delta x^1 \Delta x^2 \Delta x^3 \left[\frac{\d \vrho}{\d t}+\frac{\d (\vrho u^1)}{\d x^1}+\frac{\d (\vrho u^2)}{\d x^2}+\frac{\d (\vrho u^3)}{\d x^3}\right]=0
\end{equation*}
and, therefore, 
\begin{equation}\label{chap1:mass_conservation}
\d_t \vrho +\div (\vrho \vec u)=0\, .
\end{equation}
The previous equation \eqref{chap1:mass_conservation} is the so-called \emph{mass continuity equation}. Sturm in \cite{Sturm} reports that Leonardo da Vinci had already derived a simplified form of the statement of mass conservation in the 15th century. However, the three-dimensional form had to be accredited to Leonhard Euler (1707-1783).
\subsection{The momentum budget}
At this point, we are interested in performing budgets on momentum and energy. We sketch the approach on the momentum in 3-D. 

We consider $\vrho u^1$ the momentum which can be changed by forces and by in-out flow of momentum. The budget momentum fluxes can be calculated as in the case of masses except that $\vrho$ is now replaced by $\vrho u^1$. Instead, the forces applied to the infinitesimal cube in the $x^1$-direction are: 
\begin{equation*}
\begin{split}
&-p(x^1,x^2,x^3)\Delta x^2 \Delta x^3+p(x^1+\Delta x^1, x^2,x^3) \Delta x^2 \Delta x^3\\
&+ \mbb S^{x^1 x^2}(x^1,x^2,x^3) \Delta x^1 \Delta x^3- \mbb S^{x^1 x^2}(x^1, x^2+\Delta x^2,x^3)\Delta x^1 \Delta x^3\\
&+ \mbb S^{x^1 x^1}(x^1,x^2,x^3) \Delta x^2 \Delta x^3- \mbb S^{x^1 x^1}(x^1+\Delta x^1, x^2,x^3)\Delta x^2 \Delta x^3\\
&+ \mbb S^{x^1 x^3}(x^1,x^2,x^3) \Delta x^1 \Delta x^2- \mbb S^{x^1 x^3}(x^1, x^2,x^3+\Delta x^3)\Delta x^1 \Delta x^2\, ,
\end{split}
\end{equation*}
where the viscous stresses $\mbb S$ depend on the nature of the matter. We have also assumed that $\mbb S^{x^j x^k}=\mbb S^{x^k x^j}$ (with $j \neq k$): if these stresses had not the same intensity, the cube would be subjected to an unbalanced torque.

\medskip

\begin{figure}[htbp]
\centering
\includegraphics[scale=0.3]{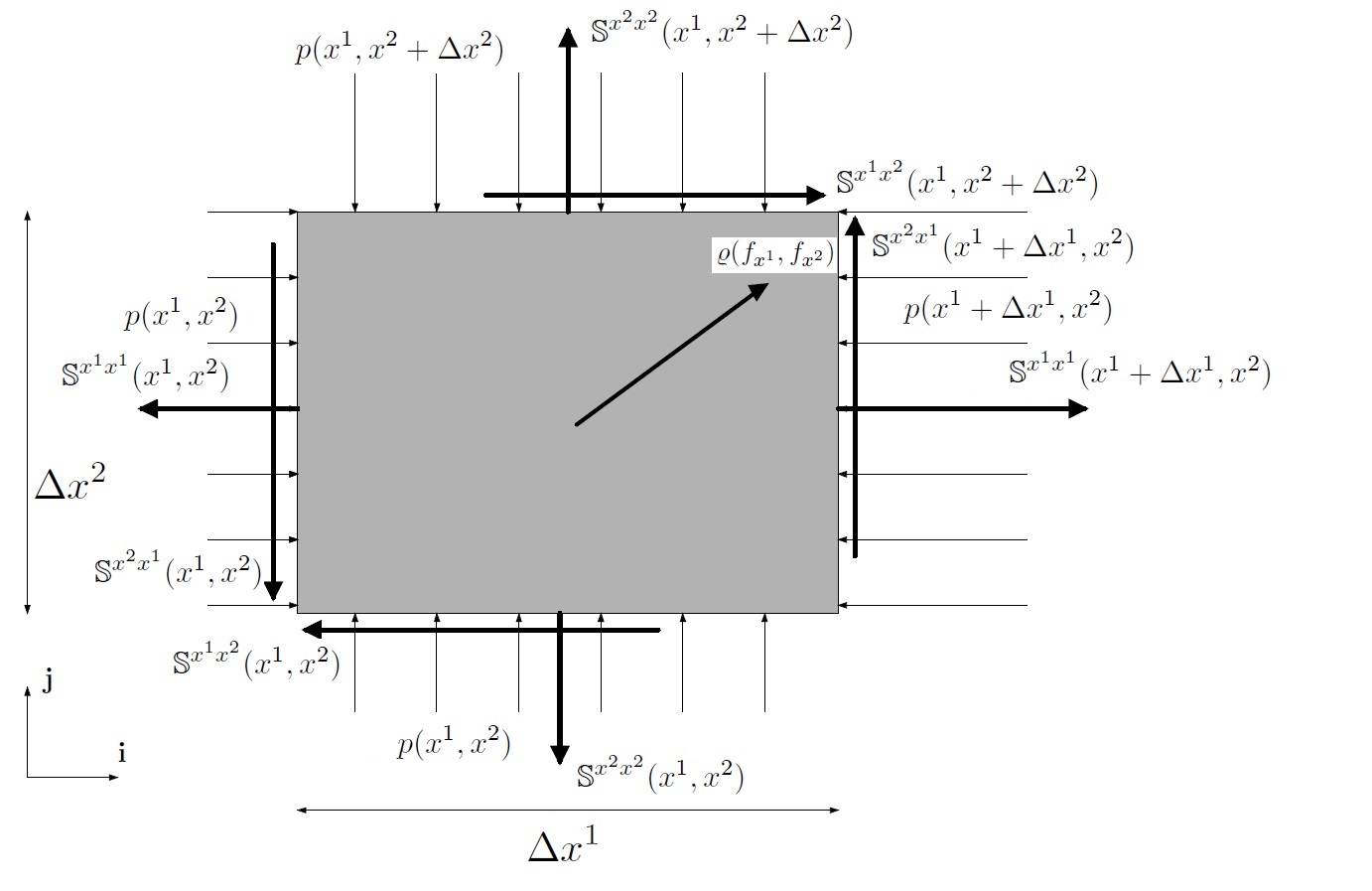}
\caption{Two-dimensional situation with forces acting on the fluid parcel}\label{fig:moementum_budget}
\end{figure}
Therefore, with these forces and the in-out flow of momentum, we derive (for the $x^1$-direction):
\begin{equation*}
\frac{\d}{\d t}(\vrho u^1)+\frac{\d}{\d x^1}(\vrho u^1 u^1)+\frac{\d}{\d x^2}(\vrho u^1 u^2)+\frac{\d}{\d x^3}(\vrho u^1 u^3)=-\frac{\d p}{\d x^1}+\frac{\d \mbb S^{x^1 x^1}}{\d x^1}+\frac{\d \mbb S^{x^1 x^2}}{\d x^2}+\frac{\d \mbb S^{x^1 x^3}}{\d x^3}\, .
\end{equation*}
Similarly to the $x^2$ and $x^3$-directions (taking into account the gravitational force $\vrho g x^3 \vec k$), we obtain the momentum equation:
\begin{equation}\label{chap1:mom_equation}
\d_t (\vrho \vec u)+\div (\vrho \vec u \otimes \vec u)=-\nabla p + \div \mbb S -\vrho g x^3 \vec k\, ,
\end{equation}
where $\vec k=\vec e_3=(0,0,1)$ is the unit vector directed along the vertical axis. 

Since we are interested in fluids which are heavily subjected to the rotational effects of the ambient, we have to make use of the relation
$$ \left(\frac{\d}{\d t}\vec u\right)_{\rm inertial}=\frac{\d}{\d t}\vec u+2\underline{\vec \Omega}\times \vec u+\underline{\vec \Omega}\times \left(\underline{\vec \Omega}\times \vec r\right)\, . $$
In the previous relation one can recognize:
\begin{itemize}
\item the Coriolis force $\underline{\vec \Omega}\times \vec u$;
\item the centrifugal force $\underline{\vec \Omega}\times \left(\underline{\vec \Omega}\times \vec r\right)=\frac{1}{2}\nabla \left(\left|\underline{\vec \Omega}\times \vec r\right|^2\right)$, where $\vec r$ is the position defined in \eqref{definition_r}.
\end{itemize}
Then, at the end recalling \eqref{chap1:mom_equation}, we get the momentum equation in the rotational framework:
\begin{equation*}\label{chap1:mom_equation_rot}
\d_t (\vrho \vec u)+\div (\vrho \vec u \otimes \vec u)+2\underline{\vec \Omega}\times \vrho \vec u+\nabla p = \div \mbb S -\vrho g x^3 \vec k+\frac{1}{2}\vrho \nabla \left(\left|\underline{\vec \Omega}\times \vec r\right|^2\right)
\end{equation*}
and in general 
\begin{equation}\label{chap1:mom_equation_gen}
\d_t (\vrho \vec u)+\div (\vrho \vec u \otimes \vec u)+2\underline{\vec \Omega}\times \vrho \vec u+\nabla p = \div \mbb S +\vrho \vec f\, ,
\end{equation}
where $\vrho \vec f$ is called \emph{body force}.
\subsection{The energy budget}
The energy density $\mc E$ can be written as
$$ \mc E := \frac{1}{2}\vrho |\vec u|^2+\vrho e\, , $$
where the function $e$ denotes the specific internal energy. 

Taking $\underline{\vec \Omega}= \Omega \vec k$ and multiplying the momentum equation \eqref{chap1:mom_equation_gen} on $\vec u$, we deduce the kinetic energy balance: 
\begin{equation*}
\d_t \left(\frac{1}{2}\vrho |\vec u|^2\right)+\div \left(\frac{1}{2}\vrho |\vec u|^2 \vec u\right)=\div (\mbb T \vec u)- \mbb T:\nabla \vec u +\vrho \vec f\cdot \vec u\, ,
\end{equation*}
where we have defined $\mbb A:\mbb B:=\sum_{j,k=1}^3 A^{jk}B^{kj}$ and the stress tensor 
$$ \mbb T:=\mbb S-p\,  \Id $$
via the Stokes' law. We recall that $\Id$ represents the identity matrix.

On the other hand, by virtue of the \emph{First law of thermodynamics}, the energy changes of system are due only to external sources, i.e.
\begin{equation}\label{chap1:energy_relation}
\d_t(\vrho e)+\div (\vrho e \vec u)+ \div \vec q=\mbb S:\nabla \vec u- p\, \div \vec u +\vrho Q\, ,
\end{equation} 
where $\vrho Q$ represents the volumetric rate of the internal energy production, and $\vec q$ is the internal energy flux. 

Therefore, the energy balance reads:
\begin{equation*}
\d_t \mc E + \div (\mc E \vec u)+ \div (\vec q -\mbb S \vec u+p\vec u)=\vrho \vec f +\vrho Q\, .
\end{equation*}
\subsection{The entropy budget}\label{subsec:entropy_bud}
In accordance with the \emph{Second law of thermodynamics}, the quantities $p$, $e$ and $s$ are linked trough the Gibbs' relation
\begin{equation*}
\vtheta D_{\vrho,\vtheta}s=D_{\vrho, \vtheta}e +p\, D_{\vrho, \vtheta}\left(\frac{1}{\vrho}\right)\, ,
\end{equation*} 
where $D_{\vrho, \vtheta}$ stands for the differential with respect to the density $\vrho$ and the temperature $\vtheta$, and $s$ is the specific entropy. 
Accordingly, the internal energy balance equation \eqref{chap1:energy_relation} can be rewritten in the form of entropy balance
\begin{equation*}
\d_t(\vrho s)+\div (\vrho s \vec u)+ \div \left(\frac{\vec q}{\vtheta}\right)=\sigma +\frac{\vrho}{\vtheta}Q\, ,
\end{equation*}
with the entropy production rate
$$ \sigma :=\frac{1}{\vtheta}\left(\mbb S:\nabla \vec u-\frac{\vec q\cdot \nabla \vtheta}{\vtheta}\right)\, . $$
Moreover, the \emph{Second law of thermodynamics} postulates that $\sigma$ must be non-negative for any admissible thermodynamic process.
\section{Boussinesq approximation}\label{sec:Bous_app}
In most geophysical flows, the density of the fluid has ``small'' oscillations around a mean value. Indeed, variations in density within one ocean basin rarely exceed 0.3\% (instead, in the atmosphere, density variations due to the winds are usually no more than 5\%). So, it appears physically justifiable to assume that the fluid density $\vrho$ does not depart much from a reference state $\oline \vrho$, i.e.
\begin{equation*}
\vrho=\oline \vrho+\vrho^\prime (t, x^1,x^2,x^3)\quad \quad \text{with}\quad \quad |\vrho^\prime|\ll \oline\vrho\, .
\end{equation*}
Therefore, the continuity equation \eqref{chap1:mass_conservation} becomes 
\begin{equation*}
\oline \vrho \left(\frac{\d u^1}{\d x^1}+\frac{\d u^2}{\d x^2}+\frac{\d u^3}{\d x^3}\right)+\vrho^\prime \left(\frac{\d u^1}{\d x^1}+\frac{\d u^2}{\d x^2}+\frac{\d u^3}{\d x^3}\right)+ \left(\frac{\d \vrho^\prime}{\d t}+u^1\frac{\d \vrho^\prime}{\d x^1}+u^2\frac{\d \vrho^\prime}{\d x^2}+u^3\frac{\d \vrho^\prime}{\d x^3}\right)=0\, .
\end{equation*} 
Since $|\vrho^\prime|\ll \oline \vrho$, only the first group of terms has to be retained so that
$$ \frac{\d u^1}{\d x^1}+\frac{\d u^2}{\d x^2}+\frac{\d u^3}{\d x^3}=0\, . $$
Physically, the statement means that we are dealing with an incompressible fluid. 
\section{Scales of motion and dimensionless numbers}\label{sec:scales_motion}
We perform in this section an analysis of scales characterizing the geophysical flows. First of all, we compare the time, length and velocity scales with respect to the ambient rotation rate $\underline{\Omega}$. Typically, one has 
$$ T\gtrsim \frac{1}{\underline{\Omega}}\quad \quad \text{and}\quad \quad \frac{U}{L}\lesssim \underline{\Omega}\, . $$ 
It is generally not required to discriminate between the two horizontal directions and velocities, respectively: we assign indeed the same length scale $L$ and velocity scale $U$ for the horizontal components (respectively). The same, however, cannot be said of the vertical direction. Geophysical flows are in fact confined in domain, which are wider than they are thick: $H/L$ is small. If we assume the \emph{Boussinesq approximation}, the terms in the continuity equation (in its reduced form) have orders of magnitude
$$ \frac{U}{L}\; , \quad \quad \frac{U}{L}\; , \quad \quad \frac{W}{H} \; .$$ 
By geophysical considerations, the vertical velocity scale must by constrained by
$$ W\lesssim \frac{H}{L}U $$
and by virtue of $H\ll L$, one has $W\ll U$. In other words, large-scale geophysical flows are shallow ($H \ll L$) and nearly two-dimensional ($W\ll U$). 

At this point, we consider the momentum equation \eqref{chap1:mom_equation_gen} under the Boussinesq approximation, in which (only for clarity of exposition) we take $\div \mbb S=\nu \Delta \vec u$, $\vec f=g\vec k$ and $\underline{\vec \Omega}=\underline{\Omega}\cos \varphi \, \vec j+\underline{\Omega}\sin \varphi \, \vec k$ where $\varphi$ is the latitude. Then, the equation reads
\begin{equation}\label{chap1:system_momentum_3D}
\begin{cases}
\frac{\d u^1}{\d t}+u^1\frac{\d u^1}{\d x^1}+u^2\frac{\d u^1}{\d x^2}+u^3\frac{\d u^1}{\d x^3}+f_\ast u^3-fu^2=-\frac{1}{\oline \vrho}\frac{\d p}{\d x^1}+\nu \left(\frac{\d^2 u^1}{\d x^1 \d x^1}+\frac{\d^2 u^1}{\d x^2 \d x^2}+\frac{\d^2 u^1}{\d x^3 \d x^3}\right)\\
\frac{\d u^2}{\d t}+u^1\frac{\d u^2}{\d x^1}+u^2\frac{\d u^2}{\d x^2}+u^3\frac{\d u^2}{\d x^3}+f u^1=-\frac{1}{\oline \vrho}\frac{\d p}{\d x^2}+\nu \left(\frac{\d^2 u^2}{\d x^1 \d x^1}+\frac{\d^2 u^2}{\d x^2 \d x^2}+\frac{\d^2 u^2}{\d x^3 \d x^3}\right)\\
\frac{\d u^3}{\d t}+u^1\frac{\d u^3}{\d x^1}+u^2\frac{\d u^3}{\d x^2}+u^3\frac{\d u^3}{\d x^3}-f_\ast u^1=-\frac{1}{\oline \vrho}\frac{\d p}{\d x^3}-\frac{g\vrho}{\oline \vrho}+\nu \left(\frac{\d^2 u^3}{\d x^1 \d x^1}+\frac{\d^2 u^3}{\d x^2 \d x^2}+\frac{\d^2 u^3}{\d x^3 \d x^3}\right)\, ,
\end{cases}
\end{equation}
with $f:=2\underline{\Omega}\sin \varphi$ and $f_\ast := 2 \underline{\Omega}\cos \varphi$.

Let us consider the $x^1$-momentum: the terms scale sequentially as 
\begin{equation}\label{chap1:scales_mom}
\frac{U}{T}\; , \; \quad \frac{U^2}{L}\; ,\; \quad \frac{U^2}{L}\; ,\;\quad \frac{WU}{H}\; , \; \quad \underline{\Omega}W\; ,\; \quad \underline{\Omega}U\; ,\; \quad \frac{P}{\oline \vrho L}\; , \; \quad \nu \frac{U}{L^2}\; , \; \quad \nu \frac{U}{L^2}\; , \; \quad \nu \frac{U}{H^2}\, . 
\end{equation}
Due to the fact that $W\ll U$, the term $\underline{\Omega}W$ is always smaller than $\underline{\Omega}U$ and can be safely neglected. 

Moreover, since the rotation has a fundamental importance in the geophysical fluid dynamics, the pressure term will scale as the Coriolis force, i.e. 
$$ \frac{P}{\oline \vrho L}=\underline{\Omega}U\, . $$
For physical considerations also the last three terms in \eqref{chap1:scales_mom} are small:
$$ \frac{U}{L^2}\, ,\,  \frac{U}{H^2}\lesssim \underline{\Omega}U\, .  $$
Similar arguments apply to the $x^2$-momentum equation. Now, let us analyse the vertical momentum in \eqref{chap1:system_momentum_3D}, which scales as
\begin{equation*}
\frac{W}{T}\; , \quad \quad \frac{UW}{L}\; , \quad \quad \frac{UW}{L}\; ,\quad \quad \frac{W^2}{H}\; ,\quad \quad \underline{\Omega}U\; ,\quad \quad \frac{P}{\oline \vrho H}\; , \quad \quad\frac{g\Delta \vrho}{\oline \vrho}\; , \quad \quad \nu \frac{W}{L^2}\; , \quad \quad \nu \frac{W}{L^2}\; , \quad \quad \nu \frac{W}{H^2}\, . 
\end{equation*} 
Regarding the first term $W/T$ one has 
$$ \frac{W}{T}\lesssim \underline{\Omega}W\ll \underline{\Omega}U\, . $$
Analogously, for the next three terms. Now,
$$ \frac{\underline{\Omega}U}{\frac{P}{\oline \vrho H}}\sim \frac{H}{L}\ll 1\, . $$
The previous relation establishes the smallness of the fifth term. As already seen before:
$$ \frac{W}{L^2}\, ,\,\frac{W}{H^2}\ll \underline{\Omega}U\, . $$
At the end only two terms remain, leading to the so-called \emph{hydrostatic balance}
$$ 0=-\frac{1}{\oline \vrho}\frac{\d p}{\d x^3}-\frac{g \vrho}{\oline \vrho}\, , $$
and we observe that, in absence of stratification, $p$ is nearly $x^3$-independent. 

At this point, our main goal is to introduce some important dimensionless numbers. In the previous scale analysis the term $\underline{\Omega}U$ was central. A division of \eqref{chap1:scales_mom} by $\underline{\Omega}U$, allows us to compare the importance of the other terms with respect to the Coriolis force, yielding (for the $x^1$-momentum): 
\begin{equation*}
\frac{1}{\underline{\Omega} T}\; , \quad \quad \frac{U}{\underline{\Omega}L}\; , \quad \quad \frac{U}{\underline{\Omega}L}\; ,\quad \quad \frac{WL}{UH}\frac{U}{\underline{\Omega}L}\; ,\quad \quad 1\; ,\quad \quad \frac{P}{\oline \vrho \underline{\Omega}LU}\; , \quad \quad  \frac{\nu}{\underline{\Omega}L^2}\; , \quad \quad  \frac{\nu}{\underline{\Omega}L^2}\; , \quad \quad  \frac{\nu}{\underline{\Omega}H^2}\, . 
\end{equation*}
The first ratio
\begin{equation*}
Ro_T:=\frac{1}{\underline{\Omega}T}
\end{equation*}
is called \emph{temporal Rossby number}.

\newpage

The next one, is the so-called \emph{Rossby number}
\begin{equation*}
Ro:=\frac{U}{\underline{\Omega}L}\, ,
\end{equation*}
which compare advection to Coriolis force: it is at most on the order of unity. 

In addition, the last ratio measures the importance of vertical friction and it is called \emph{Ekman number}
$$ Ek:=\frac{\nu}{\underline{\Omega}H^2}\, . $$
If now we focus our attention on the $x^3$-momentum equation, we have also
$$ \frac{P}{\oline \vrho H}\; ,\quad \quad \frac{g\Delta \vrho}{\oline \vrho}\, . $$
Taking the ratio between these two quantities, we obtain
$$ \frac{gH\Delta \vrho}{P}=\frac{gH\Delta \vrho}{\oline \vrho \underline{\Omega}LU}=Ro \, \frac{gH\Delta \vrho}{\oline \vrho U^2}\, . $$
This leads to another adimensional number: the \emph{Richardson number}
$$ Ri:=\frac{gH\Delta \vrho}{\oline \vrho U^2}\, . $$
\section{The Taylor-Proudman theorem}\label{sec:TP_thm}
Let us now focus on rapidly rotating fluids, ignoring the frictional and density-variation effects, i.e. 
$$ Ro_T\ll 1\; , \quad \quad Ro\ll 1\; ,\quad \quad Ek\ll 1\, . $$
Therefore, we get
\begin{equation}\label{T-P_system}
\begin{cases}
-fu^2=-\frac{1}{\oline \vrho}\frac{\d p}{\d x^1}\\
fu^1=-\frac{1}{\oline \vrho}\frac{\d p}{\d x^2}\\
0=-\frac{1}{\oline \vrho}\frac{\d p}{\d x^3}\\
\div \vec u=0\, .
\end{cases}
\end{equation}
If now we take the vertical derivative in the first and second equations of \eqref{T-P_system}, then we infer that 
$$ -f\frac{\d u^2}{\d x^3}=-\frac{1}{\oline \vrho}\frac{\d}{\d x^3}\left(\frac{\d p}{\d x^1}\right)=0\\ $$
and
$$ f\frac{\d u^1}{\d x^3}=-\frac{1}{\oline \vrho}\frac{\d}{\d x^3}\left(\frac{\d p}{\d x^2}\right)=0\, .\\ $$
The previous relations mean $\d_{3}u^1=\d_{3}u^2=0$.

This is the so-celebrated \emph{Taylor-Proudman theorem}. Physically, it states that the horizontal velocity field has no vertical shear and the particles on the same vertical move together (the \emph{Taylor columns}). 
\begin{figure}[htbp]
\centering
\includegraphics[scale=0.3]{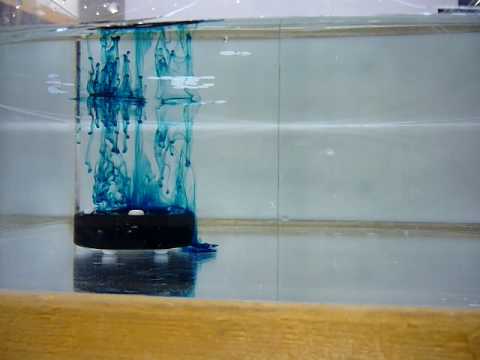}
\caption{Side view of a Taylor columns experiment}\label{fig:Taylor_column_exp}
\end{figure}
\section{Stratified and quasi-incompressible fluids}\label{sec:strat_q-i}
\subsection{The Brunt-V\"ais\"al\"a frequency}
Until now, we have devoted our attention to the effects of rotation, and stratification was avoided. Therefore, we introduce in a  first instance a basic measure of stratification called \emph{Brunt-V\"ais\"al\"a frequency} and later the accompanying dimensionless ratio, the \emph{Froude number}. 

We consider a fluid in static equilibrium. We take a fluid parcel (of volume $V$) at height $x^3$ above a reference level with density $\vrho (x^3)$ and we displace it to the higher level $x^3+h$ where the ambient density is $\vrho (x^3+h)$, see Figure \ref{fig:B-V_frequency} below. If the fluid is incompressible, by Archimede buoyancy principle, the particle is subject to the force
$$ g\left( \vrho (x^3)-\vrho(x^3+h)\right) V\, . $$ 
Thus, the Newton's law yields
$$ \vrho(x^3)V\frac{d^2 h}{d t^2}=g\left( \vrho (x^3)-\vrho(x^3+h)\right) V\, . $$
Using a Taylor expansion for the term on the right-hand side and under the Boussinesq approximation, one gets
$$ \frac{d^2h}{dt^2}-\frac{g}{\oline \vrho}\frac{d\vrho}{dx^3}h=0\, . $$
If the density is decreasing with the height ($d\vrho/dx^3<0$), we can define the \emph{Brunt-V\"ais\"al\"a frequency} as
$$ N^2:=-\frac{g}{\oline \vrho}\frac{d\vrho}{dx^3}\, . $$ 
\begin{figure}[htbp]
\centering
\includegraphics[scale=0.2]{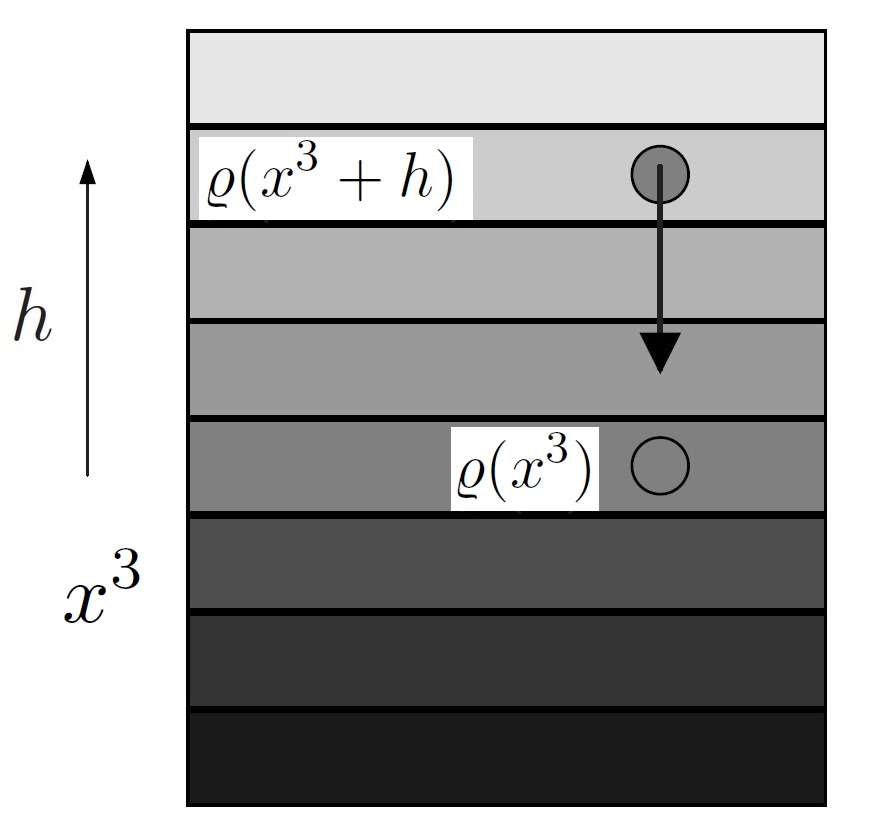}
\caption{Fluid parcel in a stratified environment}\label{fig:B-V_frequency}
\end{figure}

Physically this means that, when we displace upward the parcel, its weight is heavier than the ambient: then, it is subjected to a downward force. The particle, going down, acquires a vertical velocity and when it reaches the original level, goes further downward (due to the inertia). At this point, the particle is surrounded by an heavier ambient so that it is recalled upward and oscillations persist around the equilibrium level. 
\subsection{The measurement of stratification: the Froude number}
In this paragraph, we illustrate how to derive the physical Froude number. 
\begin{figure}[htbp]
\centering
\includegraphics[scale=0.3]{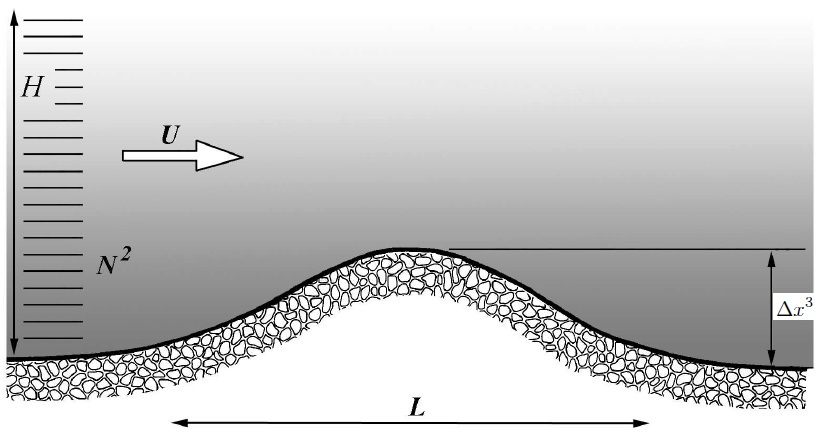}
\caption{Deep oceanic currents over an irregular bottom}\label{fig:Froude_num}
\end{figure}

We consider a stratified fluid of thickness $H$ and frequency $N$. We suppose its speed equal to $U$ over an obstacle of length $L$ and height $\Delta x^3$ (see Figure \ref{fig:Froude_num} above). We can think about deep oceanic currents over an irregular seabed. 

The obstacle forces the fluid to displace itself also vertically and hence requires some additional gravitational energy. Stratification will act to minimize such vertical displacement, forcing the flow to get around the block. To climb the impediment the fluid needs a vertical velocity 
$$ W=\frac{\Delta x^3}{T}=\frac{U\Delta x^3}{L}\, . $$ 
At this point, the vertical displacement produces a density variation 
$$\Delta \vrho =\left|\frac{d\vrho}{dx^3}\right|\Delta x^3=\frac{\oline \vrho N^2}{g}\Delta x^3 \, .$$
As a consequence, one has also a pressure disturbance that, due to the hydrostatic balance, is
$$ \Delta P=gH\Delta \vrho=\oline \vrho N^2H\Delta x^3\, , $$
which in turn causes a change in the fluid velocity
\begin{equation}\label{chap1:rel_fluid_vel}
 \frac{U^2}{L}=\frac{\Delta P}{\oline \vrho L} \, .
\end{equation}
Therefore, the last relation \eqref{chap1:rel_fluid_vel} tells us that $U^2=N^2H\Delta x^3$. 

If now we take the ratio $\frac{W/H}{U/L}$, we obtain that
$$ \frac{W/H}{U/L}=\frac{\Delta x^3}{H}=\frac{U^2}{N^2H^2}\, . $$
We note if $U$ is less than $NH$, $W/H$ is less than $U/L$. This implies that the variation in the vertical direction cannot fully meet the horizontal displacement: the fluid is then deflected horizontally. 

In addition, the stronger the stratification, the smaller is $U$ compared to $NH$ and thus $W/H$ with respect to $U/L$. For that reason, to measure the stratification, we define the \emph{Froude number}
$$ Fr:=\frac{U}{NH}\, . $$
\subsection{The Mach number}\label{subsect:Mach_number}
To define the Mach number, we consider a flow in which the density-changes induced by the pressure are isentropic, i.e. 
$$ \frac{\d p}{\d x^i}=\oline c^2\frac{\d \vrho}{\d x^i} \quad \quad \quad \text{for }\; i=1,2,3 \; ,$$
where $\oline c$ is the sound speed. Therefore, the continuity equation reads:
$$ \div \vec u=-\frac{1}{\vrho}\frac{D\vrho}{Dt}=-\frac{1}{\vrho \oline c^2}\frac{Dp}{Dt} \, ,$$
with $D/Dt=\d/\d t+\vec u\cdot \nabla$ the material derivative. 

Using the following dimensionless variables (for $i,j=1,2,3$)
\begin{equation*}
x^i_\ast=\frac{x^i}{L}\; ,\quad \quad t_\ast=\frac{Ut}{L}\; ,\quad \quad u^j_\ast=\frac{u^j}{U}\; , \quad \quad p_\ast=\frac{p}{\oline \vrho U^2}\; , \quad \quad \vrho_\ast=\frac{\vrho}{\oline \vrho}\, ,
\end{equation*}
where $\oline \vrho$ is a reference density, we obtain
$$ \div_\ast \vec u_\ast=-\frac{U^2}{\oline c^2}\frac{1}{\vrho_\ast}\frac{Dp_\ast}{Dt_\ast}\, , $$
with $\div_\ast= L\, \div$.

Then, we can define the \emph{Mach number} as
$$ Ma:= \frac{U}{\oline c}\, , $$
which sets the size of isentropic departures from incompressible flow: the flows are considered incompressible when $Ma<0.3$.
\section{The dimensionless Navier-Stokes equations}\label{sec:NS_dimenless} 
We start now from the Navier-Stokes momentum equation for incompressible flows ($\div \vec u=0$) of the following form
\begin{equation}\label{Sect:rel_NS}
\vrho \left(\frac{D}{Dt}\vec u\right)_{\rm inertial}=-\nabla p+\vrho \vec g+\nu \Delta \vec u\, .
\end{equation}
Remembering the connection between the inertial and the rotational frameworks, one has 
$$ \left(\frac{D}{Dt}\vec u\right)_{\rm inertial}=\frac{D}{Dt}\vec u +2\underline{\vec \Omega}\times \vec u+\underline{\vec \Omega}\times \left(\underline{\vec \Omega}\times \vec r\right)=\frac{D}{Dt}\vec u +2\underline{\vec \Omega}\times \vec u+ \nabla \left(\frac{1}{2}\left|\underline{\vec \Omega}\times \vec r\right|^2\right)\, ,$$
where $\underline{\vec \Omega}=\underline{\Omega}\vec k$ ($\underline{\Omega}$ is the scalar magnitude associated) and we will call $F:=\frac{1}{2}\left|\underline{\vec \Omega}\times \vec r\right|^2$. 

Moreover, we recall that the effects of compressibility can be recasted from the continuity equation (see the previous Subsection \ref{subsect:Mach_number}). 

So, relation \eqref{Sect:rel_NS} can be rendered dimensionless by defining (for $i,j=1,2,3$)
\begin{equation*}
x^i_\ast=\frac{x^i}{L}\; ,\quad \quad t_\ast=\underline{\Omega}t\; ,\quad \quad u^j_\ast=\frac{u^j}{U}\; , \quad \quad p_\ast=\frac{p}{\vrho U^2}\; , \quad \quad g^j_\ast=\frac{g^j}{g}\, ,
\end{equation*} 
where $g$ is the acceleration of gravity. 

Therefore, we get
\begin{equation}\label{system_NSFC_nodim}
St\frac{\d \vec u_\ast}{\d t_\ast}+\vec u_\ast \cdot \nabla_\ast \vec u_\ast +\frac{2}{Ro}\vec e_3\times \vec u_\ast =-\nabla_\ast p_\ast +\frac{1}{Fr^2}g_\ast+\frac{1}{Re}\Delta_\ast \vec u_\ast+\frac{1}{Ro^2}\nabla_\ast F_\ast
\end{equation}
where $\nabla_\ast =L\nabla$. 

In the previous equation the symbols $St$ and $Re$ stay for the \emph{Strouhal} and \emph{Reynolds} numbers (see \cite{K-C-D} for more details).

\let\cleardoublepage\clearpage

\part{The Navier-Stokes-Coriolis equations}

\let\cleardoublepage\clearpage

\chapter{A multi-scale limit}\label{chap:multi-scale_NSF}
\begin{quotation}
In this chapter we address the singular perturbation problem given by the full Navier-Stokes-Fourier system, which (for the reader's convenience) we remember to be
\begin{equation}\label{chap2:syst_NSFC}
\begin{cases}
	\partial_t \vrho + \div (\vrho\vec{u})=0\  \\[3ex]
	\partial_t (\vrho\vec{u})+ \div(\vrho\vec{u}\otimes\vec{u}) + \dfrac{\e_3 \times \vrho\vec{u}}{Ro}\,  +    \dfrac{1}{Ma^2} \nabla_x p(\vrho,\vtheta) \\[1ex]
	\qquad \qquad \qquad \qquad \qquad \qquad \qquad \; \; \; =\div \mbb{S}(\vtheta,\nabla_x\vec{u}) + \dfrac{\vrho}{Ro^2} \nabla_x F  + \dfrac{\vrho}{Fr^2} \nabla_x G  \\[3ex]
	 \partial_t \bigl(\vrho s(\vrho, \vtheta)\bigr)  + \div \bigl(\vrho s (\vrho,\vtheta)\vec{u}\bigr) + \div\left(\dfrac{\q(\vtheta,\nabla_x \vtheta )}{\vtheta} \right)
	= \sigma\,, 
\end{cases}
\end{equation}
with $St=1$ and $Re=1$ (see also system \eqref{system_NSFC_nodim} in Section \ref{sec:NS_dimenless}). 


The contents of this chapter are included in the article \cite{DS-F-S-WK}.

\medskip

Let us now give an outline of the chapter.

In Section \ref{s:result} we collect our assumptions and we state our main results. In Section \ref{s:sing-pert} we study the singular perturbation problem,
stating uniform bounds on our family of weak solutions and establishing constraints that the limit points have to satisfy. Section \ref{s:proof} is devoted to the proof of the convergence
result for $m\geq2$ and $F\neq0$, employing the compensated compactness technique. In the last Section \ref{s:proof-1}, with the same approach, we prove the convergence result for
$m=1$ and $F=0$; actually, in absence of the centrifugal force, the same argument shows convergence for any $m>1$.
\end{quotation}

\section{The Navier-Stokes-Fourier system} \label{s:result}

In this section, we formulate our working hypotheses (see Subsection \ref{ss:FormProb}) and we state our main results
(in Subsection \ref{ss:results}).

\subsection{Setting of the problem} \label{ss:FormProb}

In this subsection, we present the rescaled Navier-Stokes-Fourier system with Coriolis, centrifugal and gravitational forces, which we are going to consider in our study, and we
formulate the main working hypotheses. 
The material of this part is mostly classical: unless otherwise specified, we refer to \cite{F-N} for details.
Paragraph \ref{sss:equilibrium} contains some original contributions, concerning the analysis of the equilibrium states under our hypotheses on the specific form of the centrifugal and gravitational forces.

 \subsubsection{Primitive system}\label{sss:primsys}
To begin with, let us introduce the ``primitive system'', i.e. the rescaled compressible Navier-Stokes-Fourier system \eqref{chap2:syst_NSFC},
supplemented with the scaling 
\begin{equation} \label{eq_i:scales}
Ro=\veps \, , \quad Ma=\veps^m \quad \text{and}\quad Fr=\veps^{m/2}\,, \qquad\qquad \mbox{ for some }\quad m\geq 1\, ,
\end{equation}
where $\veps\in\,]0,1]$ is a small parameter.
Thus, the system consists of the continuity equation (conservation of mass), the momentum equation, the entropy balance and the total energy balance: respectively,
\begin{align}
&	\partial_t \vre + \div (\vre\ue)=0 \label{ceq}\tag{NSF$_{\ep}^1$} \\
&	\partial_t (\vre\ue)+ \div(\vre\ue\otimes\ue) + \frac{1}{\ep}\,\e_3 \times \vre\ue +    \frac{1}{\ep^{2m}} \nabla_x p(\vre,\tem) 
	= \label{meq}\tag{NSF$_{\ep}^2$} \\
&	\qquad\qquad\qquad\qquad\qquad\qquad\qquad\qquad
=\div \mbb{S}(\tem,\nabla_x\ue) + \frac{\vre}{\ep^2} \nabla_x F  + \frac{\vre}{\ep^m} \nabla_x G \nonumber \\
&	 \partial_t \bigl(\vre s(\vre, \tem)\bigr)  + \div \bigl(\vre s (\vre,\tem)\ue\bigr) + \div\left(\frac{\q(\tem,\nabla_x \tem )}{\tem} \right)
	= \sigma_\ep \label{eiq}\tag{NSF$_{\ep}^3$} \\
&	\frac{d}{dt} \int_{\Omega_\veps} 
	\left( \frac{\ep^{2m}}{2} \vre|\ue|^2 +  \vre e(\vre,\tem) - {\ep^m} \vre G - {\ep^{2(m-1)}} \vre F \right) dx = 0\,. \label{eeq}\tag{NSF$_{\ep}^4$}
	\end{align}
The unknowns are the fluid mass density $\vre=\vre(t,x)\geq0$ of the fluid, its velocity field $\ue=\ue(t,x)\in\R^3$ and
its absolute temperature $\tem=\tem(t,x)\geq0$, $t\in \;  ]0,T[$ , $x\in \Omega_\veps$ which fills, for $\veps \in \; ]0,1]$ fixed, the bounded domain 
\begin{equation}\label{dom}
	\Omega_\ep  :=    {B}_{L_\veps} (0) \times\,]0,1[\;,\qquad\qquad\mbox{ where }\qquad L_\veps\,:=\,\frac{1}{\ep^{m+ \delta}}\,L_0
	\end{equation}
for $\delta >0$ and for some $L_0>0$ fixed. Here above, we have denoted by ${B}_{l}(x_0)$  the Euclidean ball of center $x_0$ and radius $l$ in  $\R^2$.
Notice that, roughly speaking, we have the property
$$\Omega_\ep\, \longrightarrow \, \Omega := \R^2 \times\,]0,1[\, \quad \mbox{ as } \ep \to 0^+\,.$$ 
\begin{remark} \label{smooth domains}
We explicitly point out that, throughout all the chapter, we tacitly assume \emph{rounded corners} in \eqref{dom}.
In this way, we can apply the standard weak solutions existence theory developed in \cite{F-N}, which requires $C^{2,\nu}$ regularity, with $\nu \in (0,1)$, on the space domain.  
\end{remark}
The pressure $p$, the specific internal energy 
$e$ and the specific entropy $s$ are given scalar valued functions of $\vr$ and $\temp$ which are related through 
Gibbs' equation	
	\begin{equation}\label{gibbs}
	\temp D_{\vrho,\vtheta} s = D_{\vrho,\vtheta} e + p D_{\vrho, \vtheta} \left( \frac{1}{\vr}\right),	
	\end{equation}
where the symbol $D_{\vrho,\vtheta}$ stands for the differential with respect to the variables $\vr$ and $\vartheta$ (see also Subsection \ref{subsec:entropy_bud}). The viscous stress tensor in \eqref{meq} is given by Newton's rheological law
	\begin{equation}\label{S}
	\mbb{S}(\tem,\nabla_x \ue) = \mu(\tem)\left( \nabla_x\ue \,+\, ^t\nabla_x \ue  \,-\, \frac{2}{3}\div \ue \,  \Id \right)
	\,+\, \eta(\tem) \div\ue \,  \Id\,,
	\end{equation}
for two suitable coefficients $\mu$ and $\eta$ (we refer to Paragraph \ref{sss:structural} below for the precise hypotheses), and here the apex $t$ stands for the transpose operator. 

Moreover, the entropy production rate $\sigma_\ep$ in \eqref{eiq} satisfies 
	\begin{equation}\label{ss}
	 \sigma_\ep \geq \frac{1}{\tem} \left({\ep^{2m}} \mbb{S}(\tem,\nabla_x\ue) : \nabla_x \ue
	 - \frac{\q(\tem,\nabla_x \tem )\cdot \nabla_x \tem}{\tem}  \right). 
	\end{equation}
The heat flux $\q$ in \eqref{eiq} is determined by Fourier's law
	\begin{equation}\label{q}
	\q(\tem,\nabla_x \tem)= - \kappa(\tem) \nabla_x \tem ,
	\end{equation}
where $\kappa>0$ is the heat-conduction coefficient. The term $\e_3\times\vrho_\veps\vec u_\veps$ takes into account the (strong) Coriolis force acting on the fluid.
Next, we turn our attention to centrifugal and gravitational forces, $F$ and $G$ respectively.	
We assume that they are of the form
	\begin{equation}\label{assFG}
	F(x) = \left|x^h\right|^2\qquad\mbox{ and }\qquad G(x)= -x^3\,.
	\end{equation}
The precise expression of $F$ and $G$ will be useful in Paragraph \ref{sss:equilibrium} below (and even more in Chapter \ref{chap:BNS_gravity}), but the previous assumptions are certainly not optimal from the point of view of the weak solutions theory.
\begin{remark}	
For the existence theory of weak solutions to our system, it would be enough to assume $F\in W^{1,\infty}_{\rm loc}\bigl(\R^2 \times\,]0,1[\,\bigr)$ to satisfy
$$ 
	F(x) \geq 0,\quad  \ F(x_1,x_2,- x_3) = F(x_1,x_2,x_3),\quad 
	\
	 | \nabla_x F(x) | \leq c\, (1 + |x^h| )\; \; \mbox{ for all }\; \;  x\in \R^2\times\,]0,1[
$$
and $G\in W^{1,\infty}\bigl(\R^2 \times\,]0,1[\,\bigr)$.
\end{remark}
	
The system is supplemented  with \emph{complete slip boundary conditions}, namely
	\begin{equation}\label{bc1-2}
	\big(\ue \cdot \n_\veps\big) _{|\partial \Omega_\veps} = 0,
	\quad \mbox{ and } \quad
	\bigl([ \mbb{S} (\tem, \nabla_x \ue) \n_\veps ] \times \n_\veps\bigr)_{|\d\Omega_\veps} = 0\,,
	\end{equation}
where $\vec{n}_\veps$ denotes the outer normal to the boundary $\d\Omega_\veps$.
We also suppose that the boundary of physical space is \emph{thermally isolated}, i.e. one has
	\begin{equation}\label{bc3}
	\big(\q\cdot\n_\veps\big)_{|\partial {\Omega_\veps}}=0\,.
	\end{equation}	

\begin{remark} \label{r:speed-waves}
Notice that, as $\delta>0$ in \eqref{dom} and the speed of sound is proportional to $\ep^{-m}$, hypothesis \eqref{dom} guarantees that $\d B_{L_\veps}(0)\times\,]0,1[\,$
of the boundary $\d\Omega_\veps$ of $\Omega_\ep$ becomes irrelevant when one considers the behaviour of acoustic waves on some compact set of the physical space. We refer to
Subsections \ref{ss:acoustic} and \ref{ss:unifbounds_1} for details about this point. 
\end{remark}

\subsubsection{Structural restrictions} \label{sss:structural}
Now we need to  impose structural restrictions on the thermodynamical functions $p$, $e$, $s$ as well as on  
the diffusion coefficients $\mu$, $\eta$, $\kappa$. We start by setting, for some real number $a>0$,
	\begin{equation}\label{pp1}
	p(\vrho,\vtheta)= p_M(\vrho,\vtheta) + \frac{a}{3} \vtheta^{4}\,,\qquad\qquad\mbox{ where }\qquad p_M(\vrho,\vtheta)\,=\,\vtheta^{5/2} P\left(\frac{\vrho}{\vtheta^{3/2}}\right)\,.
	\end{equation}
The first component $p_M$ in relation \eqref{pp1} corresponds  to the standard molecular pressure of a general \textit{monoatomic} gas (see Section 1.4 of \cite{F-N}), while the second one represents the thermal radiation. Here above,
	\begin{equation}\label{pp2}
	P\in C^1 [0,\infty)\cap C^2(0,\infty),\qquad P(0)=0,\qquad P'(Z)>0\quad \mbox{ for all }\,Z\geq 0\,,
	\end{equation}
which in particular implies the positive compressibility condition
	\begin{equation}\label{p_comp}
	\partial_\vr  p (\vr,\temp)>0.
	\end{equation}
Additionally, we assume that $\partial_\temp e(\vr,\temp)$ is positive and bounded (see below): this translates into the condition
	\begin{equation}\label{pp3}
	0<\frac{ \frac{5}{3} P(Z) - P'(Z)Z }{ Z }< c \qquad\qquad\mbox{ for all }\; Z > 0\,.
	\end{equation}
In view of  \eqref{pp3}, we have that $Z \mapsto P(Z) / Z^{5/3}$ is a decreasing function and in addition we assume
	\begin{equation}\label{pp4}
	\lim\limits_{Z\to +\infty} \frac{P(Z)}{Z^{5/3}} = P_\infty >0\,.
	\end{equation}

Accordingly to Gibbs' relation \eqref{gibbs}, the specific internal energy and the specific entropy can be written in the following forms:
$$
	e(\vr,\temp) = e_M(\vr,\temp) + a\frac{\temp^{4}}{\vr}\,, \quad\quad
	s(\vr,\temp)= S\left(\frac{\vr}{\temp^{3/2}}\right) + \frac{4}{3} a \frac{\temp^{3}}{\vr}\,,
$$
where we have set
	\begin{equation}\label{ss1s}
e_M(\vr,\temp)=\frac{3}{2} \frac{\temp^{5/2}}{\vr} P\left( \frac{\vr}{\temp^{3/2}} \right)\qquad\mbox{ and }\qquad
S'(Z) = -\frac{3}{2} \frac{ \frac{5}{3} P(Z) - Z P'(Z)}{Z^2}\quad \mbox{ for all }\, Z>0\,.
	\end{equation}

The diffusion coefficients $\mu$ (shear viscosity), $\eta$ (bulk viscosity) and $\kappa$ (heat conductivity)  are assumed to be continuously differentiable 
functions of the temperature  $\temp \in [0,\infty[\,$, satisfying the following  growth conditions for all $\temp\geq 0$:
	\begin{equation}\label{mu}
	0<\underline\mu(1+\temp) \leq \mu(\temp) \leq \overline\mu (1 + \temp),
\quad
	0 \leq \eta(\temp) \leq \overline\eta(1 + \temp), 
\quad
	0 <  \underline\kappa (1 + \temp^3) \leq \kappa(\temp) \leq \overline\kappa(1+\temp^3),
	\end{equation}
 where $\underline\mu$, $\overline\mu$, $\overline\eta$, $\underline\kappa$ and $\overline\kappa$ are positive constants. Let us remark that the above assumptions may be not optimal from the point of view of the existence theory.

\begin{remark}\label{rmk:pressure_choice}
We point out that the choice in taking the pressure  as above (which formulation describes the pressure for a \textit{monoatomic gas}) is dictated by the fact that we follow the ``solid'' existence theory developed by Feireisl and Novotn\'y in the book \cite{F-N}. 
Any other formulation for the pressure, for which one possesses an existence theory, is allowed in our analysis (see e.g. \cite{B-D_JMPA} for the polytropic gas case, and references therein): as we will see in the next chapter, if $\vtheta$ is constant, one can take in \eqref{pp4} any $\gamma >d/2$ as exponent (where $d$ is the dimension). We refer to \cite{Lions_2}, \cite{F-N-P} and references therein, in this respect (see also \cite{PL_Lions} for a first result in that context). 
\end{remark}

\subsubsection{Analysis of the equilibrium states} \label{sss:equilibrium}

For each scaled system \eqref{ceq} to \eqref{eeq}, the so-called \emph{equilibrium states}  consist of static density $\vret$ and constant temperature distribution  $\tems>0$
satisfying
	\begin{equation}\label{prF}
\nabla_x p(\vret,\tems) = \ep^{2(m-1)} \vret \nabla_x F + \ep^m  \vret  \nabla_x G \quad \mbox{ in } \Omega.
	\end{equation}	
For later use, it is convenient to state \eqref{prF} on whole set $\Omega$. 	
Notice that, \textsl{a priori}, it is not known that the target temperature has to be constant: this follows from the fact that $\nabla_x \temp_\ep$  needs to vanish as $\ep \to 0^+$
(see Section 4.2 of \cite{F-N} for more comments about this).

%
%

Equation \eqref{prF} identifies $\wtilde{\vrho}_\veps$ up to an additive constant: normalizing it to $0$, and taking the target density to be $1$, we get
\begin{equation} \label{eq:target-rho}
  \Pi(\vret)\,=\,\wtilde{F}_\veps\,:=\, \ep^{2(m-1)} F + \ep^m G\,,\qquad\qquad \mbox{ where }\qquad 
\Pi(\vrho) = \int_1^{\vrho} \frac{\partial_\varrho p(z,\oline{\vtheta})}{z} {\rm d}z\,.
\end{equation}

From this relation, we immediately get the following properties:
\begin{itemize}
 \item[(i)] when $m>1$, or $m=1$ and $F=0$, for any $x\in\Omega$ one has $\wtilde{\vrho}_\veps(x)\longrightarrow 1$ in the limit $\veps\ra0^+$;
 \item[(ii)] for $m=1$ and $F\neq0$, $\bigl(\wtilde{\vrho}_\veps\bigr)_\veps$ converges pointwise to $\wtilde{\vrho}$, where
$$
\wtilde\vrho\quad\mbox{ is a solution of the problem}\qquad \Pi\bigl(\wtilde{\vrho}(x)\bigr)\,=\,F(x)\,,\ \mbox{ with }\ x\in\Omega\,.
$$
In particular,
$\wtilde{\vrho}$ is non-constant, of class $C^2(\Omega)$ (keep in mind assumptions \eqref{pp2} and \eqref{p_comp} above) and it depends just on the horizontal variables due to \eqref{assFG}.
\end{itemize}

We are now going to study more in detail the equilibrium densities $\wtilde\vrho_\veps$.
In order to keep the discussion as general as possible, we are going to consider both cases (i) and (ii) listed above, even though our results will concern only case (i).

The first problem we have to face is that the right-hand side of \eqref{eq:target-rho} may be negative: this means that $\wtilde{\varrho}_\veps$ can go below the value $1$ in some regions of $\Omega$.
Nonetheless, the next statement excludes the presence of vacuum.
\begin{lemma} \label{l:target-rho_pos}
Let the  centrifugal force $F$ and the gravitational force $G$ be given by \eqref{assFG}.
Let $\bigl(\wtilde{\vrho}_\veps\bigr)_{0<\veps\leq1}$ be a family of static solutions to equation \eqref{eq:target-rho}
on $\Omega$.

Then, there exist an $\veps_0>0$ and a $0<\rho_*<1$ such that $\wtilde{\vrho}_\veps\geq\rho_*$ for all $\veps\in\,]0,\veps_0]$
and all $x\in\Omega$.
\end{lemma}

\begin{proof}
Let us consider the case $m>1$ (hence $F\neq0$) first. Suppose, by contradiction, that there exists a sequence $\bigl(\veps_n,x_n\bigr)_n$ such that
$0\,\leq\,\wtilde{\vrho}_{\veps_n}(x_n)\,\leq\,1/n$, and we will observe that the sequence $\bigl(x_n\bigr)_n$ cannot be bounded. Indeed, relation
\eqref{eq:target-rho}, computed on $\wtilde{\vrho}_{\veps_n}(x_n)$, would immediately imply that $\wtilde{\vrho}_{\veps_n}(x_n)$ should
rather converge to $1$.
In any case, since $1/n<1$ for $n\geq2$ and $x^3\in\; ]0,1[\, $, we deduce that
$$
-\,(\veps_n)^m\,\leq\,\wtilde{F}_{\veps_n}(x_n)\,=\,(\veps_n)^{2(m-1)}\,|\,(x_n)^h\,|^2\,-\,(\veps_n)^m\,(x_n)^3\,<\,0\,,
$$
which in particular implies that $\wtilde{F}_{\veps_n}(x_n)$ has to go to $0$ for $\veps\ra0^+$. As a consequence, since
$\Pi(1)=0$, by the mean value theorem (see e.g. Chapter 5 of \cite{Rud}) and \eqref{eq:target-rho} we get
$$
\wtilde{F}_{\veps_n}(x_n)\,=\,\Pi\bigl(\wtilde{\vrho}_{\veps_n}(x_n)\bigr)\,=\,\Pi'(z_n)\,\bigl(\wtilde{\vrho}_{\veps_n}(x_n)-1\bigr)\,=\,
\frac{\d_\varrho p\bigl(z_n,\oline{\vtheta}\bigr)}{z_n}\,\bigl(\wtilde{\vrho}_{\veps_n}(x_n)-1\bigr)\,\longrightarrow\,0\,,
$$
for some $z_n\in\,]\wtilde{\vrho}_{\veps_n}(x_n),1[\,\subset\,]0,1[\,$, for all $n\in\N$. In turn, this relation,
combined with \eqref{p_comp}, implies that $\wtilde{\vrho}_{\veps_n}(x_n)\longrightarrow 1$,
which is in contradiction with the fact that it has to be $\leq1/n$ for any $n\in\N$.

The case $m=1$ and $F=0$ can be treated in a similar way. Let us now assume that $m=1$ and $F\neq0$: relation \eqref{eq:target-rho} in this case becomes
\begin{equation} \label{eq:target_m=1}
\Pi\bigl(\wtilde{\vrho}_\veps(x)\bigr)\,=\,|\,x^h\,|^2\,-\,\veps\,x^3\,.
\end{equation}
We observe that the right-hand side of this identity is negative on the set
$\left\{0\,\leq\,|\,x^h\,|^2\,\leq\,\veps\,x^3\right\}$. By definition \eqref{eq:target-rho}, this is equivalent to having $\wtilde{\vrho}_\veps(x)\leq1$.

In particular, the smallest value of $\wtilde{\vrho}_\veps(x)$
is attained for $x=x^0=(0,0,1)$, for which $\Pi\bigl(\wtilde{\vrho}_\veps(x^0)\bigr)=-\veps$.
On the other hand, fixed a $x^0_\veps$ such that $|\,(x^0_\veps)^h\,|^2=\veps$ and $(x^0_\veps)^3=1$, we have
$\Pi\bigl(\wtilde{\vrho}_\veps(x_\veps^0)\bigr)=0$, and then $\wtilde{\vrho}_\veps(x^0_\veps)=1$.
Therefore, by mean value theorem again we get
\begin{align*}
-\,\veps\;=\;\Pi\bigl(\wtilde{\vrho}_\veps(x^0)\bigr)\,-\,\Pi\bigl(\wtilde{\vrho}_\veps(x^0_\veps)\bigr)\,&=\,
\frac{\d_\varrho p\bigl(\wtilde{\vrho}_\veps(y_\veps),\oline{\vtheta}\bigr)}{\wtilde{\vrho}_\veps(y_\veps)}\,
\bigl(\wtilde{\vrho}_\veps(x^0)\,-\,\wtilde{\vrho}_\veps(x^0_\veps)\bigr)\,=\,\frac{\d_\varrho p\bigl(\wtilde{\vrho}_\veps(y_\veps),\oline{\vtheta}\bigr)}{\wtilde{\vrho}_\veps(y_\veps)}\,
\bigl(\wtilde{\vrho}_\veps(x^0)\,-\,1\bigr),
\end{align*}
for some suitable point $y_\veps=\bigl((x_\veps)^h,1\bigr)$ lying on the line connecting $x^0=(0,0,1)$ with $x^0_\veps$.

From this equality and the structural hypothesis \eqref{p_comp}, since $\wtilde{\vrho}_\veps(x^0)\,-\,1<0$ (due to the fact that $\Pi\bigl(\wtilde{\vrho}_\veps(x^0)\bigr)<0$),
we deduce that $\wtilde{\vrho}_\veps(y_\veps)>0$ for all $\veps>0$. On the other hand, \eqref{eq:target_m=1} says that,
for $x^3$ fixed, the function $\Pi\circ\wtilde{\vrho}_\veps$ is radially increasing on $\R^2$: then, in particular
$\wtilde{\vrho}_\veps(y_\veps)\leq\wtilde{\vrho}_\veps(x^0_\veps)=1$.

Finally, thanks to these relations and the regularity properties \eqref{pp1} and \eqref{pp2}, we see that
$$
\wtilde{\vrho}_\veps(x^0)\,=\,1\,-\,\veps\,
\frac{\wtilde{\vrho}_\veps(y_\veps)}{\d_\varrho p\bigl(\wtilde{\vrho}_\veps(y_\veps),\oline{\vtheta}\bigr)}
$$
remains strictly positive, at least for $\veps$ small enough.
\qed
\end{proof}

\medbreak
For simplicity, and without loss of any generality, we assume from now on that $\veps_0=1$ in Lemma \ref{l:target-rho_pos}.

Next, denoted as above $B_l(0)$ the ball in the horizontal variables $x^h\in\R^2$ of center $0$ and radius $l>0$, we define the cylinder \emph{with smoothed corners}
$$
	\mbb B_{L} := \left\{ x\in \Omega \ : \ |x^h| < L  \right\}=B_L(0)\times\, ]0,1[\, .
$$
We can now state the next boundedness property for the family $\bigl(\wtilde{\vrho}_\veps\bigr)_\veps$.
\begin{lemma} \label{l:target-rho_bound}
Let $m\geq1$. Let $F$ and $G$ satisfy \eqref{assFG}. Then, for any $l>0$, there exists a constant $C(l)>1$ such that for all $\veps\in\,]0,1]$ one has
\begin{equation} \label{est:target-rho_in}
\wtilde{\vrho}_\veps\,\leq\,C(l)\qquad\qquad\mbox{ on }\qquad \oline{\mbb{B}}_{l}\,.
\end{equation}

If $F=0$, then there exists $C>1$ such that, for all $\veps\in\,]0,1]$ and all $x\in\Omega$, one has $\left|\wtilde\vrho_\veps(x)\right|\leq C$.
\end{lemma}

\begin{proof}
Let us focus on the case $m>1$ and $F\neq 0$ for a while.
In order to see \eqref{est:target-rho_in}, we proceed in two steps. First of all, we fix $\veps$ and we show that $\wtilde{\varrho}_\veps$
is bounded in the previous set. Assume it is not: there exists a sequence $\bigl(x_n\bigr)_{n}\subset \oline{\mbb{B}}_{l}$
such that $\wtilde{\varrho}_\veps(x_n)\geq n$. But then, thanks to hypothesis \eqref{pp3}, we can write
$$
\Pi\bigl(\wtilde{\varrho}_\veps(x_n)\bigr)\,\geq\,\int^n_1\frac{\d_\varrho p(z,\oline{\vtheta})}{z}{\rm\,d}z\,\geq\,C(\oline{\vtheta})\,
\int^{n/\oline{\vtheta}^{3/2}}_{1/\oline{\vtheta}^{3/2}}\frac{P(Z)}{Z^2}{\rm\,d}Z\,,
$$
and, by use of \eqref{pp4}, it is easy to see that the last integral diverges to $+\infty$ for $n\ra+\infty$. On the other hand,
on the set $\oline{\mbb{B}}_{l}$, the function $\wtilde{F}_\veps$ is uniformly bounded by the constant
$l^2+1$, and, recalling formula \eqref{eq:target-rho}, these two facts are in contradiction one with other.

So, we have proved that, $\wtilde{\varrho}_\veps\,\leq\,C(\veps,l)$ on the set $\oline{\mbb{B}}_{l}$.
But, thanks to point (i) below \eqref{eq:target-rho}, the pointwise convergence of $\wtilde{\vrho}_\veps$ to $1$ becomes uniform in the
previous set, so that the constant $C(\veps,l)$ can be dominated by a new constant $C(l)$, just depending on the fixed $l$. 

Let us now take $m=1$ and $F\neq0$. 
We start by observing that, again, the following property holds true:
for any $\veps$ and any $l>0$ fixed, one has $\wtilde{\vrho}_\veps\,\leq\,C(\veps,l)$ in $\oline{\mbb{B}}_{l}$.
Furthermore, by point (ii) below \eqref{eq:target-rho} we have that $\wtilde\vrho\in C^2(\Omega)$, and then $\wtilde\vrho$ is locally bounded:
for any $l>0$ fixed, we have $\wtilde{\vrho}\,\leq\,C(l)$ on the set $\oline{\mbb{B}}_{l}$. On the other hand, the pointwise convergence of
$\bigl(\wtilde{\vrho}_\veps\bigr)_\veps$ towards $\wtilde\vrho$ becomes uniform on the compact set $\oline{\mbb{B}}_{l}$:
gluing these facts together, we infer that, in the previous bound for $\wtilde{\vrho}_\veps$, we can replace $C(\veps,l)$ by a constant $C(l)$ which is uniform in $\veps$.

Consider now the case $F=0$, and any value $m\geq1$. In this case, relation \eqref{eq:target-rho} becomes
$$
\Pi\big(\wtilde\vrho_\veps\big)\,=\,\veps^m\,G,\qquad\text{which implies}\qquad \left|\Pi\big(\wtilde\vrho_\veps\big)\right|\,\leq\,C\quad\mbox{ in }\;\Omega\,.
$$
At this point, as a consequence of the structural assumptions \eqref{pp1}, \eqref{pp3} and \eqref{pp4}, we observe that
$\Pi(z)\longrightarrow+\infty$ for $z\ra+\infty$. Then, $\wtilde\vrho_\veps$ must be uniformly bounded in $\Omega$.

This completes the proof of the lemma. \qed
\end{proof}

\medbreak
We conclude this paragraph by showing some additional bounds, which will be relevant in the sequel.
\begin{proposition} \label{p:target-rho_bound}
Let $F\neq0$. For any $l>0$, on the cylinder $\oline{\mbb{B}}_{l}$ one has, for any $\veps\in\,]0,1]$:
\begin{enumerate}[(1)]
\item $
\left|\wtilde{\vrho}_\veps(x)\,-\,1\right|\,\leq\,C(l)\,\veps^m\, \text{ if }m\geq2$; 
\item $
\left|\wtilde{\vrho}_\veps(x)\,-\,1\right|\,\leq\,C(l)\,\veps^{2(m-1)}\, \text{ if }1<m<2
$;
\item $
\left|\wtilde{\vrho}_\veps(x)\,-\,\wtilde{\varrho}(x)\right|\,\leq\,C(l)\,\veps\, \text{ if }m=1$. 
\end{enumerate}

When $F=0$ and $m\geq1$, instead, one has $\left|\wtilde{\vrho}_\veps(x)\,-\,1\right|\,\leq\,C\,\veps^m$, for a constant $C>0$ which is uniform in $x\in\Omega$ and in $\veps\in\,]0,1]$.
\end{proposition}

\begin{proof}
Assume $F\neq0$ for a while. Let $m\geq2$. Thanks to the Lemma \ref{l:target-rho_bound}, the estimate on $\left|\wtilde{\vrho}_\veps(x)\,-\,1\right|$ easily
follows applying the mean value theorem (see again e.g. Chapter 5 of \cite{Rud}) to equation \eqref{eq:target-rho}, and noticing that
$$
\sup_{z\in[\rho_*,C(l)]}\left|\frac{z}{\d_\varrho p(z,\oline{\vtheta})}\right|\,<\,+\infty\,,
$$ 
on $\oline{\mbb{B}}_l$ for any fixed $l>0$.
According to the hypothesis $m\geq2$, we have $2(m-1)\geq m$. The claimed bound then follows.
The proof of the inequality for $1<m<2$ is analogous, using this time that $2(m-1)\leq m$.

In order to prove the inequality for $m=1$, we consider the equations satisfied by $\wtilde{\vrho}_\veps$ and $\wtilde{\vrho}$: we have
$$
\Pi\bigl(\wtilde{\vrho}_\veps(x)\bigr)\,=\,|\,x^h\,|^2\,-\,\veps\,x^3\qquad\qquad\mbox{ and }\qquad\qquad
\Pi\bigl(\wtilde{\vrho}(x)\bigr)\,=\,|\,x^h\,|^2\,.
$$
Now, we take the  difference and we apply again the mean value theorem, finding
$$
\Pi'\big(z_\veps(x)\big)\,\big(\wtilde{\vrho}_\veps(x)\,-\,\wtilde{\varrho}(x)\big)\,=\,-\veps\,x^3\,,
$$
for some $z_\veps(x)\in\,]\wtilde{\vrho}_\veps(x),\wtilde{\varrho}(x)[\,$ (or with exchanged extreme points, depending on $x$). By Lemma \ref{l:target-rho_bound},
we have uniform (in $\veps$) bounds on the set $\oline{\mbb{B}}_{l}$, depending on $l$, for $\wtilde{\vrho}_\veps(x)$ and
$\wtilde{\varrho}(x)$: then, from the previous identity, on this cylinder we find
$$
\left|\wtilde{\vrho}_\veps(x)\,-\,\wtilde{\varrho}(x)\right|\,\leq\,C(l)\,\veps\,.
$$

The bounds in the case $F=0$ can be shown in an analogous way. 
The proposition is now  completely proved.
\qed
\end{proof}

\medbreak

%

From now on, we will focus on the following cases:
\begin{equation} \label{eq:choice-m}
\mbox{ either }\quad m\geq2\,,\qquad\quad \mbox{ or }\quad\qquad m\geq1\quad\mbox{ and }\quad F=0\,.
\end{equation}
Notice that in all those cases, the target density profile $\wtilde\vrho$ is constant, namely $\wtilde\vrho\equiv1$.


\subsubsection{Initial data and finite energy weak solutions} \label{sss:data-weak}

We address the singular perturbation problem described in Paragraph \ref{sss:primsys} for general \emph{ill prepared initial data}, in the framework of \emph{finite energy weak solutions},
whose theory was developed in \cite{F-N}.
Since we work with weak solutions based on dissipation estimates and control of entropy production rate, we need to assume that the initial data are close to the equilibrium states $(\vret,\tems)$ that we have just identified.
Namely, we consider initial densities and temperatures of the following form:
	\begin{equation}\label{in_vr}
	\vrez = \vret + \ep^m \vrez^{(1)} 
	\qquad\qquad\mbox{ and }\qquad\qquad
	\temz = \tems + \ep^m \Theta_{0,\veps}\,.
	\end{equation}
For later use, let us introduce also the following decomposition of the initial densities: 
\begin{equation} \label{eq:in-dens_dec}
\vrho_{0,\veps}\,=\,1\,+\,\veps^m\,R_{0,\veps}\qquad\qquad\mbox{ with }\qquad
R_{0,\veps}\,=\,\vrho_{0,\veps}^{(1)}\,+\,\wtilde r_\veps\,,\qquad \wtilde r_\veps\,:=\,\frac{\wtilde\vrho_\veps-1}{\veps^m}\,.
\end{equation}
Notice that $\wtilde r_\veps$ is in fact a datum of the system, since it only depends on $p$, $F$ and $G$.

We suppose $\vrez^{(1)}$ and $\Theta_{0,\veps}$ to be bounded measurable functions satisfying the controls
	\begin{align}
\sup_{\veps\in\,]0,1]}\left\|  \vrez^{(1)} \right\|_{(L^2\cap L^\infty)(\Omega_\veps)}\,\leq \,c\,,\qquad 
\sup_{\veps\in\,]0,1]}\left(\left\|\Theta_{0,\veps}\right\|_{L^\infty(\Omega_\veps)}\,+\,\left\| \sqrt{\wtilde\vrho_\veps}\,\Theta_{0,\veps}\right\|_{L^2(\Omega_\veps)}\right)\,\leq\, c\,,\label{hyp:ill_data}
	\end{align}
together with the mean-free conditions
$$
\int_{\Omega_\veps}  \vrez^{(1)} \dx = 0\qquad\qquad\mbox{ and }\qquad\qquad \int_{\Omega_\veps}\Theta_{0,\veps} \dx = 0\,.
$$
As for the initial velocity fields, we will assume instead the following uniform bounds:
\begin{equation} \label{hyp:ill-vel}
 	\sup_{\veps\in\,]0,1]}\left(\left\| \sqrt{\wtilde\vrho_\veps} \vec{u}_{0,\ep} \right\|_{L^2(\Omega_\veps)}\,+\, \left\| \vec{u}_{0,\ep}  \right\|_{L^\infty(\Omega_\veps)}\right)\,  \leq\, c\,.
\end{equation}

\begin{remark} \label{r:ill_data}
In view of Lemma \ref{l:target-rho_pos}, the conditions in \eqref{hyp:ill_data} and \eqref{hyp:ill-vel} imply in particular that
$$
\sup_{\veps\in\,]0,1]}\left(\left\| \Theta_{0,\veps}\right\|_{L^2(\Omega_\veps)}\,+\,\left\| \vec{u}_{0,\ep}  \right\|_{L^2(\Omega_\veps)}\right)\,\leq\,c\,.
$$
\end{remark}


Thanks to the previous uniform estimates, up to extraction, we can assume that
\begin{equation} \label{conv:in_data}
\vrho^{(1)}_0\,:=\,\lim_{\veps\ra0}\vrho^{(1)}_{0,\veps}\;,\qquad
R_0\,:=\,\lim_{\veps\ra0}R_{0,\veps}\;,\qquad
\Theta_0\,:=\,\lim_{\veps\ra0}\Theta_{0,\veps}\;,\qquad
\vec{u}_0\,:=\,\lim_{\veps\ra0}\vec{u}_{0,\veps}\,,
\end{equation}
where we agree that the previous limits are taken in the weak-$*$ topology of $L_{\rm loc}^\infty(\Omega)\,\cap\,L_{\rm loc}^2(\Omega)$.

\medbreak


Let us specify better what we mean for \emph{finite energy weak solution} (see \cite{F-N} for details). First of all, the equations have to be satisfied in a distributional sense:
	\begin{equation}\label{weak-con}
	-\int_0^T\int_{\Omega_\veps} \left( \vre \partial_t \varphi  + \vre\ue \cdot \nabla_x \varphi \right) \dxdt = 
	\int_{\Omega_\veps} \vrez \varphi(0,\cdot) \dx,
	\end{equation}
for any $\varphi\in C^\infty_c([0,T[\,\times \overline\Omega_\veps)$;
	\begin{align}
	&\int_0^T\!\!\!\int_{\Omega_\veps}  
	\left( - \vre \ue \cdot \partial_t \vec\psi - \vre [\ue\otimes\ue]  : \nabla_x \vec\psi 
	+ \frac{1}{\ep} \, \e_3 \times (\vre \ue ) \cdot \vec\psi  - \frac{1}{\ep^{2m}} p(\vre,\tem) \div \vec\psi  \right) \dxdt \label{weak-mom} \\
	& =\int_0^T\!\!\!\int_{\Omega_\veps} 
	\left(- \mbb{S}(\vtheta_\veps,\nabla_x\vec u_\veps)  : \nabla_x \vec\psi +  \left(\frac{1}{\ep^2} \vre \nabla_x F +  \frac{1}{\ep^m} \vre \nabla_x G \right)\cdot \vec\psi \right) \dxdt 
	+ \int_{\Omega_\veps}\vrez \uez  \cdot \vec\psi (0,\cdot) \dx, \nonumber
	\end{align}
for any test function $\vec\psi\in C^\infty_c([0,T[\,\times \overline\Omega_\veps; \R^3)$ such that $\big(\vec\psi \cdot \n_\veps\big)_{|\partial {\Omega_\veps}} = 0$;
	\begin{align}
	\int_0^T\!\!\!\int_{\Omega_\veps} 
	& \Bigl( - \vre s(\vre,\tem) \partial_t \varphi  -  \vre s(\vre,\tem) \ue \cdot \nabla_x \varphi \Bigr) \dxdt \label{weak-ent} \\
	& - \int_0^T\int_{\Omega_\veps}  \frac{\q(\vtheta_\veps,\nabla_x\vtheta_\veps)}{\tem} \cdot \nabla_x \varphi \dxdt - \langle \sigma_\ep; \varphi  \rangle_{ [{\cal{M}}; C^0]([0,T]\times \overline\Omega_\veps)}
	= \int_{\Omega_\veps} \vrez s(\vrez,\temz) \varphi (0,\cdot) \dx, \nonumber
	\end{align}
for any $\varphi\in C^\infty_c([0,T[\,\times \overline\Omega_\veps)$, with $\sigma_\ep \in {\cal{M}}^+ ([0,T]\times \overline\Omega_\veps)$. 
In addition, we require that the energy identity 
\begin{align}
	\int_{\Omega_\veps}
	& \left( \frac{1}{2} \vre |\ue|^2  +  \frac{1}{\ep^{2m}} \vre e(\vre,\tem)  -  \frac{1}{\ep^2} \vre F - \frac{1}{\ep^m} \vre G \right) (t) \dx  \label{weak-eng} \\
	&=  \int_{\Omega_\veps} \left( \frac{1}{2} \vrez |\uez|^2  +  \frac{1}{\ep^{2m}} \vrez e(\vrez,\temz) -  \frac{1}{\ep^2} \varrho_{0,\ep} F  -\frac{1}{\ep^m} \varrho_{0,\ep} G  \right) \dx \nonumber
	\end{align}
holds true for almost every $t\in\,]0,T[\,$. Notice that this is the integrated version of \eqref{eeq}.

Under the previous assumptions (collected in Paragraphs \ref{sss:primsys} and \ref{sss:structural} and here above), at any \emph{fixed} value of the parameter $\veps\in\,]0,1]$,
the existence of a global in time finite energy weak solution $(\vrho_\veps,\vec u_\veps,\vtheta_\veps)$ to system \eqref{ceq} to \eqref{eeq}, related to the initial datum
$(\vrho_{0,\veps},\vec u_{0,\veps},\vtheta_{0,\veps})$, has been proved in e.g. \cite{F-N} (see Theorems 3.1 and 3.2 therein).
Moreover, 
the following regularity of solutions $( \vre, \ue, \tem )$  can be obtained,  
which justifies all the integrals appearing in \eqref{weak-con} to \eqref{weak-eng}: for any $T>0$ fixed, one has
	\begin{equation*}
	\vre \in C^0_{w}\big([0,T];L^{5/3}(\Omega_\ep)\big),\quad 
	\vre \in L^q\big((0,T)\times \Omega_\ep\big) \ \mbox{ for some } q>\frac{5}{3}\,,
\qquad
	\ue \in L^2\big([0,T]; W^{1,2}(\Omega_\ep;\R^3)\big)\,.
	\end{equation*}
In addition, the mapping $t \mapsto (\vre\ue)(t,\cdot)$ is weakly continuous, and one has $(\vre)_{|t=0} = \vrez$ together with $(\vre\ue)_{|t=0}= \vrez\uez$. Finally,
the absolute temperature $\tem$ is a measurable function, $\tem>0$ a.e. in $\R_+ \times \Omega_\ep$, and given any $T>0$, one has
	\begin{equation*}
	\tem \in L^2\big([0,T]; W^{1,2}(\Omega_\ep)\big)\cap L^{\infty}\big([0,T]; L^4 (\Omega_\ep)\big), \quad 
	\log \tem \in L^2\big([0,T]; W^{1,2}(\Omega_\ep)\big)\,.
	\end{equation*}

Notice that, in view of \eqref{ceq}, the total mass is conserved in time, in the following sense: for almost every $t\in[0,+\infty[\,$,
one has
\begin{equation} \label{eq:mass_conserv}
\int_{\Omega_\veps}\bigl(\vre(t)\,-\,\vret\bigr)\,\dx\,=\,0\,.
\end{equation}

Let us now remark that, since the entropy production rate is a non-negative measure, and in particular it may possess jumps, the total entropy 
$\vre s(\vre,\tem)$ may not be weakly continuous in time. To avoid 
this problem, we introduce a time lifting $\Sigma_\ep$ of the measure $\sigma_\ep$ (see Paragraph 5.4.7 in \cite{F-N} for details) by the following formula:
	\begin{equation}\label{lift0}
	\langle \Sigma_\ep , \varphi \rangle = \langle \sigma_\ep , I[\varphi] \rangle,
\quad \mbox{ where }\quad
	I [\varphi] (t,x) = \int _0^t \varphi (\tau,x) {\rm\, d}\tau\quad\mbox{for any } \varphi \in L^1(0,T; C^0(\overline\Omega_\veps)).
	\end{equation}
	The time lifting $\Sigma_\ep$ can be identified with an abstract function 
$\Sigma_\ep \in L^{\infty}_{w}(0,T; {\cal{M}}^+(\overline{\Omega}))$, where $L_{w}^\infty$ stands for ``weakly measurable'',
and $\Sigma_\veps$ is defined by the relation
	\begin{equation*}
	\langle \Sigma_\ep(\tau),\varphi \rangle = \lim\limits_{\delta \to 0^+} \langle \sigma_\ep , \psi_\delta\varphi \rangle,
\quad \mbox{ with } \quad
	\psi_\delta(t) =
	\left\{\begin{array}{cc}
	0 					& {\mbox{for }} t\in [0,\tau), \\
	\frac{1}{\delta}(t - \tau) 	& {\mbox{for }} t\in (\tau, \tau + \delta), \\
	1					& {\mbox{for }} t \geq \tau +\delta.
	\end{array}\right.
	\end{equation*}
In particular, the measure $\Sigma_\ep$ is well-defined for any $\tau\in[0,T]$, and the mapping 
$\tau \to \Sigma_\ep(\tau)$ is non-increasing in the sense of measures.

Then, the weak formulation of the entropy balance can be equivalently rewritten as
	\begin{equation*} 
	\begin{split}
	& \int_{\Omega_\veps} 
	 \left[  \vre s(\vre,\tem)(\tau)\varphi(\tau)
	-   \vrez s(\vrez,\temz)\varphi(0)  \right] \dx 
	 + \langle \Sigma_\ep(\tau),\varphi(\tau) \rangle  -  \langle \Sigma_\ep(0),\varphi(0) \rangle\\
	& = \int_0^\tau  \langle \Sigma_\ep,\partial_t \varphi \rangle \dt
	+ \int_0^\tau \int_{\Omega_\veps} 
	\left( 
	 \vre s(\vre,\tem)   \partial_t \varphi  +  \vre s(\vre,\tem)\ue \cdot \nabla_x \varphi 
	 +  \frac{\q(\vtheta_\veps,\nabla_x\vtheta_\veps)}{\tem} \cdot \nabla_x \varphi 
	\right) \dxdt,
	\end{split}
	\end{equation*}
for any $\varphi\in C^\infty_c([0,T]\times \overline\Omega_\veps)$, and  the mapping 
	$
	t \to \vre s(\vre,\tem)(t,\cdot) + \Sigma_\ep(t) \ 
	$
 is continuous with values in ${\cal{M}}^+(\overline{\Omega}_\veps)$, provided that 
${\cal{M}}^+$ is endowed with the weak-$*$ topology.
\begin{remark} 
We explicitly point out that the previous properties are \emph{not} uniform in the small parameter $\veps$. In order to deduce uniform properties on our family of weak solutions
$\bigl(\vrho_\veps,\vec u_\veps,\vtheta_\veps\bigr)_\veps$,
we ``measure'' the energy of the solutions with respect to the energy
at the equilibrium states $\left(\wtilde\vrho_\veps,0,\oline\vtheta\right)$.
\end{remark}

\medbreak

To conclude this part, we introduce the ballistic free energy function
$$
	H_{\tems}(\vr,\temp)\,:=\,\vr \bigl( e(\vr,\temp) - \tems s(\vr,\temp)  \bigr)\,,
$$
and we define the \emph{relative entropy functional} (for details, see in particular Chapters 1, 2 and 4 of \cite{F-N})
$$
\mc E\left(\rho,\theta\;|\;\wtilde\vrho_\veps,\oline\vtheta\right)\,:=\,H_{\tems}(\rho,\theta) - (\rho - \vret)\,\partial_\vrho H_{\tems}(\vret,\tems)
- H_{\tems}(\vret,\tems)\,.
$$
First of all, we  notice that, by \eqref{eq:target-rho} and Gibbs' relation \eqref{gibbs}, equation \eqref{prF} can be rewritten as
	\begin{equation*}
	\partial_\vrho H_{\oline{\vtheta}} (\vret,\tems) \, =\, \ep^{2(m-1)} F + \ep^m G
	\end{equation*}
in $\Omega_\veps$ (up to some constant, that we have normalized to $0$).

Then, combining the total energy balance \eqref{weak-eng}, the entropy equation \eqref{weak-ent} and the mass conservation \eqref{eq:mass_conserv}, we obtain 
the following total dissipation balance, for any $\veps>0$ fixed:
\begin{align}
&\hspace{-0.7cm} \int_{\Omega_\veps}\frac{1}{2}\vre|\ue|^2(t) \dx\,+\,\frac{1}{\ep^{2m}}\int_{\Omega_\veps}\mc E\left(\vrho_\veps,\vtheta_\veps\;|\;\wtilde\vrho_\veps,\oline\vtheta\right) \dx
+  \frac{\tems}{\ep^{2m}}\sigma_\ep \left[[0,t]\times \overline\Omega_\veps\right] \label{est:dissip} \\
&\qquad\qquad\qquad\qquad\qquad\qquad\qquad\qquad
\,\leq\,
\int_{\Omega_\veps}\frac{1}{2}\vrez|\uez|^2 \dx\,+\,
\frac{1}{\ep^{2m}}\int_{\Omega_\veps}\mc E\left(\vrho_{0,\veps},\vtheta_{0,\veps}\;|\;\wtilde\vrho_\veps,\oline\vtheta\right) \dx\,.
\nonumber
\end{align}

Inequality \eqref{est:dissip} will be the only tool to derive uniform estimates for the family of weak solutions we consider. As a matter of fact,
we will establish in Lemma \ref{l:initial-bound} below that, under the previous assumptions on the initial data, the quantity on the right-hand side of \eqref{est:dissip}
is uniformly bounded for any $\veps\in\,]0,1]$.


\subsection{Main results}\label{ss:results}

\medbreak
We can now state our main results. The first statement concerns the case when low Mach number effects are predominant with respect to the fast rotation, i.e. $m>1$. 
For technical reasons which will appear clear in the course of the proof, when $F\neq0$ we need to take $m\geq2$.

We also underline that the limit dynamics of $\vec{U}$ is purely horizontal (see \eqref{eq_lim_m:momentum} below) on the plane $\R^2\times \{0\}$ accordingly to the celebrated Taylor-Proudman theorem.
Nonetheless the equations that involve $R$ and $\Theta$ (see \eqref{eq_lim_m:temp} and \eqref{eq_lim_m:Boussinesq} below) depend also on the vertical variable. 
\begin{theorem}\label{th:m-geq-2}
For any $\veps\in\,]0,1]$, let $\Omega_\ep$ be the domain defined by \eqref{dom} and $\Omega = \R^2 \times\,]0,1[\,$. 
Let  $p$, $e$, $s$  satisfy  Gibbs' relation \eqref{gibbs} and structural hypotheses from \eqref{pp1} to \eqref{ss1s}, and suppose that the diffusion coefficients $\mu$, $\eta$, $\kappa$
enjoy growth conditions  \eqref{mu}. Let $G\in W^{1,\infty}(\Omega)$ be given as in \eqref{assFG}.
Take either $m\geq 2$ and $F\in W_{loc}^{1,\infty}(\Omega)$ as in \eqref{assFG}, or $m>1$ and $F=0$. \\
For any fixed value of $\veps \in \; ]0,1]$, let initial data $\left(\vrho_{0,\veps},\vec u_{0,\veps},\vtheta_{0,\veps}\right)$ verify the hypotheses fixed in Paragraph \ref{sss:data-weak}, and let
$\left( \vre, \ue, \tem\right)$ be a corresponding weak solution to system \eqref{ceq} to \eqref{eeq}, supplemented with structural hypotheses from \eqref{S} to \eqref{q} and with boundary conditions \eqref{bc1-2} and \eqref{bc3}.
Assume that the total dissipation balance \eqref{est:dissip} is satisfied.
Let $\left(R_0,\vec u_0,\Theta_0\right)$ be defined as in \eqref{conv:in_data}.

Then, for any $T>0$, one has the following convergence properties:
	\begin{align*}
	\varrho_\ep \rightarrow 1 \qquad\qquad \mbox{ in } \qquad &L^{\infty}\big([0,T]; L_{\rm loc}^{5/3}(\Omega )\big) \\
	R_\veps:=\frac{\varrho_\ep - 1}{\ep^m}  \weakstar R \qquad\qquad \mbox{ weakly-$*$ in }\qquad &L^\infty\bigl([0,T]; L^{5/3}_{\rm loc}(\Omega)\bigr) \\
	\Theta_\veps:=\frac{\vartheta_\ep - \bar{\vartheta}}{\ep^m}  \weak \Theta \qquad \mbox{ and }\qquad \vec{u}_\ep \weak \vec{U}
	\qquad\qquad \mbox{ weakly in }\qquad &L^2\big([0,T];W_{\rm loc}^{1,2}(\Omega)\big)\,, 
	\end{align*}	
where $\vec{U} = (\vec U^h,0)$, with $\vec U^h=\vec U^h(t,x^h)$ such that $\divh\vec U^h=0$. In addition, the triplet $\Big(\vec{U}^h,\, \, R ,\, \, \Theta \Big)$ is a weak solution
to the incompressible Oberbeck-Boussinesq system  in $\R_+ \times \Omega$:
\begin{align}
& \d_t \vec U^{h}+\divh\left(\vec{U}^{h}\otimes\vec{U}^{h}\right)+\nabla_h\Gamma-\mu (\oline\vtheta )\Delta_{h}\vec{U}^{h}=\delta_2(m)\lan R\ran\nabla_{h}F \label{eq_lim_m:momentum} \\
& c_p(1,\oline\vtheta)\,\Bigl(\d_t\Theta\,+\,\divh(\Theta\,\vec U^h)\Bigr)\,-\,\kappa(\oline\vtheta)\,\Delta\Theta\,=\,
\oline\vtheta\,\alpha(1,\oline\vtheta)\,\vec{U}^h\cdot\nabla_h\big({\,G\,+\,\delta_2(m)F}\big)
\label{eq_lim_m:temp} \\
& \nabla_{x}\Big( \d_\varrho p(1,\oline{\vtheta})\,R\,+\,\d_\vtheta p(1,\oline{\vtheta})\,\Theta \Big)\,=\,\nabla_{x}G\,+\,\delta_2(m)\,\nabla_{x}F\, , \label{eq_lim_m:Boussinesq}
\end{align}
supplemented with the initial conditions
$$
\vec{U}_{|t=0}=\h_h\left(\lan\vec{u}^h_{0}\ran\right)\quad \text{ and }\quad
\Theta_{|t=0}\,=\,\frac{\oline\vtheta}{c_p(1,\oline\vtheta)}\,\Big(\d_\vrho s(1,\oline\vtheta)\,R_0\,+\,\d_\vtheta s(1,\oline\vtheta)\,\Theta_0\,+\,
\alpha(1,\oline\vtheta)\,{\big(\,G\,+\,\delta_2(m)F\big)}\Big)
$$
and the boundary condition $\nabla_x \Theta \cdot\vec{n}_{|\d\Omega}\,=\,0$,
where $\vec{n}$ is the outer normal to $\d\Omega\,=\,\{x_3=0\}\cup\{x_3=1\}$. \\
In \eqref{eq_lim_m:momentum}, $\Gamma$ is a distribution in $\mc D'(\R_+\times\R^2)$ and we have set $\delta_2(m)=1$ if $m=2$, $\delta_2(m)=0$ otherwise.
In \eqref{eq_lim_m:temp}, we have defined 
$$
c_p(\vrho,\vtheta)\,:=\,\d_\vtheta e(\vrho,\vtheta)\,+\,\alpha(\vrho,\vtheta)\,\frac{\vtheta}{\vrho}\,\d_\vtheta p(\vrho,\vtheta)\,,\qquad 
\alpha(\vrho,\vtheta)\,:=\,\frac{1}{\vrho}\,\frac{\d_\vtheta p(\vrho,\vtheta)}{\d_\vrho p(\vrho,\vtheta)}\,.
$$
\end{theorem}


\begin{remark}\label{r:lim delta theta} 
%
We notice that, after defining
$$
\Upsilon := \d_\vrho s(1,\oline{\vtheta})R + \d_\vtheta s(1,\oline{\vtheta})\,\Theta\qquad\mbox{ and }\qquad
\Upsilon_0\,:=\,\d_\vrho s(1,\oline\vtheta)\,R_0\,+\,\d_\vtheta s(1,\oline\vtheta)\,\Theta_0\,,
$$
from equation \eqref{eiq} one would get, in the limit $\veps\ra0^+$, the equation
\begin{equation} \label{eq:Upsilon}
\d_{t} \Upsilon +\divh\left( \Upsilon \vec{U}^{h}\right) -\frac{\kappa(\oline\vtheta)}{\oline\vtheta} \Delta \Theta =0\,,
\qquad\qquad \Upsilon_{|t=0}\,=\,\Upsilon_0\,,
\end{equation}
which is closer to the formulation of the target system given in \cite{K-M-N} and \cite{K-N}.
From \eqref{eq:Upsilon} one easily recovers \eqref{eq_lim_m:temp} by using \eqref{eq_lim_m:Boussinesq}. Formulation \eqref{eq_lim_m:temp} is in the spirit of Chapters 4 and 5 of \cite{F-N}.
\end{remark}

The case $m=1$ realizes the \emph{quasi-geostrophic balance} in the limit. Namely the Mach and Rossby numbers have the same order of magnitude, and they keep in balance in the whole asymptotic process.
The next statement is devoted to this case. Due to technical reasons, in this instance we have to assume $F=0$.
Indeed, when $F\neq 0$, the coexistence of the centrifugal effects and the heat transfer deeply complicates the wave system and new technical troubles arise.
\begin{theorem} \label{th:m=1_F=0}
For any $\veps\in\,]0,1]$, let $\Omega_\ep$ be the domain defined by \eqref{dom} and $\Omega = \R^2 \times\,]0,1[\,$.
Let  $p$, $e$, $s$  satisfy \eqref{gibbs} and the structural hypotheses from \eqref{pp1} to \eqref{ss1s}, and suppose that the diffusion coefficients $\mu$, $\eta$, $\kappa$
enjoy \eqref{mu}. Let $F=0$ and $G\in W^{1,\infty}(\Omega)$ be as in \eqref{assFG}. Take $m=1$. \\
For any fixed value of $\veps$, let initial data $\left(\vrho_{0,\veps},\vec u_{0,\veps},\vtheta_{0,\veps}\right)$ verify the hypotheses fixed in Paragraph \ref{sss:data-weak}, and let
$\left( \vre, \ue, \tem\right)$ be a corresponding weak solution to system \eqref{ceq} to \eqref{eeq}, supplemented with structural hypotheses from \eqref{S} to \eqref{q} and with boundary conditions \eqref{bc1-2} and \eqref{bc3}.
Assume that the total dissipation balance \eqref{est:dissip} is satisfied.
Let $\left(R_0,\vec u_0,\Theta_0\right)$ be defined as in \eqref{conv:in_data}.

Then, for any $T>0$, the convergence properties stated in the previous theorem still hold true: namely, one has
	\begin{align*}
		\varrho_\ep \rightarrow 1 \qquad\qquad \mbox{ in } \qquad &L^{\infty}\big([0,T]; L_{\rm loc}^{5/3}(\Omega )\big) \\
	R_\veps:=\frac{\varrho_\ep - 1}{\ep}  \weakstar R \qquad\qquad \mbox{ weakly-$*$ in }\qquad &L^\infty\bigl([0,T]; L^{5/3}_{\rm loc}(\Omega)\bigr) \\
	\Theta_\veps:=\frac{\vartheta_\ep - \bar{\vartheta}}{\ep}  \weak \Theta \qquad \mbox{ and }\qquad \vec{u}_\ep \weak \vec{U}
	\qquad\qquad \mbox{ weakly in }\qquad &L^2\big([0,T];W_{\rm loc}^{1,2}(\Omega)\big)\,, 
	\end{align*}	
where $\vec{U} = (\vec U^h,0)$, with $\vec U^h=\vec U^h(t,x^h)$ such that $\divh\vec U^h=0$. Moreover, let us introduce the real number $\mc A>0$ by the formula
\begin{equation} \label{def:A}
\mc{A}\,=\,\d_\vrho p(1,\oline{\vtheta})\,+\,\frac{\left|\d_\vtheta p(1,\oline{\vtheta})\right|^2}{\d_\vtheta s(1,\oline{\vtheta})}\,,
\end{equation}
and define
$$
\Upsilon\, := \,\d_\vrho s(1,\oline{\vtheta}) R + \d_\vtheta s(1,\oline{\vtheta})\,\Theta\qquad\mbox{ and }\qquad
q\,:=\, \d_\varrho p(1,\oline{\vtheta})R +\d_\vtheta p(1,\oline{\vtheta})\Theta -G-1/2\,.
$$
Then we have
$$
q\,=\,q(t,x^h)\,=\,\d_\varrho p(1,\oline{\vtheta})\lan R\ran\,+\,\d_\vtheta p(1,\oline{\vtheta})\lan\Theta\ran\qquad\mbox{ and }\qquad
\vec{U}^{h}=\nabla_h^{\perp} q\,.
$$
Moreover, the couple $\Big(q,\Upsilon \Big)$ satisfies (in the weak sense) the  quasi-geostrophic type system
\begin{align}
& \d_{t}\left(\frac{1}{\mc{A}}q-\Delta_{h}q\right) -\nabla_{h}^{\perp}q\cdot
\nabla_{h}\left( \Delta_{h}q\right) +\mu (\oline{\vtheta})
\Delta_{h}^{2}q\,=\,\frac{1}{\mc A}\,\lan X\rangle \label{eq_lim:QG}  \\
& c_p(1,\oline\vtheta)\Big(\d_{t} \Upsilon +\nabla_h^\perp q\cdot\nabla_h\Upsilon \Big)-\kappa(\oline\vtheta) \Delta \Upsilon\,=\,
\kappa(\oline\vtheta)\,\alpha(1,\oline\vtheta)\,\Delta_hq\,, \label{eq_lim:transport}
\end{align}
supplemented with the initial conditions
$$
\left(\frac{1}{\mc{A}}q-\Delta_{h}q\right)_{|t=0}=\left(\lan R_0\ran+\frac{1}{2\mc A}\right)-\curlh\lan\vec u^h_{0}\ran\,,\quad
\Upsilon_{|t=0}=\d_\vrho s(1,\oline\vtheta)R_0+\d_\vtheta s(1,\oline\vtheta)\Theta_0
$$
and the boundary condition
\begin{equation} \label{eq:bc_limit_2}
\nabla_x\big(\Upsilon\,+\,\alpha(1,\oline\vtheta)\,G\big) \cdot\vec{n}_{|\d\Omega}\,=\,0\,,
\end{equation}
where $\vec{n}$ is the outer normal to the boundary $\d\Omega\,=\,\{x_3=0\}\cup\{x_3=1\}$. In \eqref{eq_lim:QG}, we have defined 
\begin{equation}\label{def:X}
 X:=\mc B\frac{\kappa(\oline\vtheta)}{c_p(1,\oline \vtheta)}\left(\Delta \Upsilon\,+\,\alpha(1,\oline\vtheta)\,\Delta_hq -\frac{1}{\kappa(\oline\vtheta)}\nabla^\perp_hq \cdot \nabla_h  \Upsilon\right) \quad \quad  \text{with}\quad \quad \mc B := \frac{\d_\vtheta p(1,\oline \vtheta)}{\d_\vtheta s(1,\oline \vtheta)}\, . 
\end{equation}
\end{theorem}

\begin{remark} \label{r:limit_1}
Observe that $q$ and $\Upsilon$  
can be equivalently chosen for describing the target problem. Indeed, straightforward computations show that
\begin{align*}
R\,&=\,-\,\frac{1}{\beta}\,\big(\d_\vtheta p(1,\oline\vtheta)\,\Upsilon\,-\,\d_\vtheta s(1,\oline\vtheta)\,q\,-\,\d_\vtheta s(1,\oline\vtheta)\,G\big) \\
\Theta\,&=\,\frac{1}{\beta}\,\big(\d_\vrho p(1,\oline\vtheta)\,\Upsilon\,-\,\d_\vrho s(1,\oline\vtheta)\,q\,-\,\d_\vrho s(1,\oline\vtheta)\,G\big)\,,
\end{align*}
where we have set $\beta\,=\,\d_\vrho p(1,\oline\vtheta)\,\d_\vtheta s(1,\oline\vtheta)\,-\,\d_\vtheta p(1,\oline\vtheta)\,\d_\vrho s(1,\oline\vtheta)$. In particular, equation \eqref{eq_lim:transport}
can be deduced from \eqref{eq:Upsilon}, which is valid also when $m=1$, using the expression of $\Theta$ and the fact that
$$
\beta\,=\,c_p(1,\oline\vtheta)\,\frac{\d_\vrho p(1,\oline\vtheta)}{\oline\vtheta}\,.
$$
Here we have chosen to formulate the target entropy balance equation in terms of $\Upsilon$ (as in \cite{K-N}) rather than $\Theta$ (as in Theorem \ref{th:m-geq-2} above),
because the equation for $\Upsilon$ looks simpler (indeed, the equation for $\Theta$ would make a term in $\d_tq$ appear). The price to pay is the non-homogeneous boundary
condition \eqref{eq:bc_limit_2}, which may look a bit unpleasant.
\end{remark}

As pointed out for Theorem \ref{th:m-geq-2}, we notice that, despite the function $q$ is defined in terms of $G$, the dynamics described by \eqref{eq_lim:QG} is purely horizontal.
On the contrary, dependence on $x^3$ and vertical derivatives appear in \eqref{eq_lim:transport}.

\begin{remark} \label{r:energy}
We have not investigated here the well-posedness of the target problems, formulated in Theorems \ref{th:m-geq-2} and \ref{th:m=1_F=0}. Very likely, when $F=0$, by standard energy methods
(see e.g. \cite{C-D-G-G}, \cite{F-G-N}, \cite{DeA-F}) it is possible to prove that those systems are well-posed in the energy space, globally in time. \\
Yet, it is not clear for us that the solutions identified in the previous theorems are (the unique) finite energy weak solutions to the target problems.
\end{remark}

\section{Analysis of the singular perturbation} \label{s:sing-pert}

The purpose of this section is twofold. First of all, in Subsection \ref{ss:unif-est} we establish uniform bounds and further properties for our family
of weak solutions. Then, we study the singular operator underlying to the primitive equations \eqref{ceq} to \eqref{eeq}, and determine constraints that the limit points of our family of weak solutions
have to satisfy (see Subsection \ref{ss:ctl1}).

\subsection{Uniform bounds}\label{ss:unif-est}

This section is devoted to establish uniform bounds on the sequence $\bigl(\vrho_\veps,\vec u_\veps,\vtheta_\veps\bigr)_\veps$.
Since the Coriolis term does not contribute to the total energy balance of the system, most of the bounds can be proven as in the case without rotation; we refer to \cite{F-N} for details.
First of all, let us introduce some preliminary material.

\subsubsection{Preliminaries} \label{sss:unif-prelim}
Let us recall here some basic notations and results, which we need in proving our convergence results. We refer to Sections 4, 5 and 6 of \cite{F-N} for more details.

Let us introduce the so-called ``essential'' and ``residual'' sets. Recall that the positive constant $\rho_*$ has been defined in Lemma \ref{l:target-rho_pos}.
Following the approach of \cite{F-N}, we define
$$
	 {\cal{O}}_{\ess}\,: = \, \left[2\rho_*/3\, ,\, 2 \right]\, \times\, \left[\tems/2\,,\, 2 \tems\right]\,,\qquad  
	 {\cal{O}}_{\res}\,: =\, \,]0,+\infty[\,^2\setminus {\cal{O}}_{\ess}\,.
$$
Then, we fix a smooth function
$\mf{b} \in C^\infty_c \bigl( \,]0,+\infty[\,\times\,]0,+\infty[\, \bigr)$ such that $0\leq \mf b\leq 1, \ \mf b\equiv1$  on the set $ {\cal{O}}_{\ess}$, and
we introduce the decomposition on essential and residual part of a measurable function $h$ as follows:
	\begin{equation*}\label{ess-def}
	h = [h]_{\ess} + [h]_{\res},\qquad\mbox{ with }\quad  [ h]_{\ess} := \mf b(\vre,\tem) h\,,\quad \  [h]_{\res} = \bigl(1-\mf b(\vre,\tem)\bigr)h\,.
	\end{equation*}
We also introduce the sets $\mc{M}^\veps_{\ess}$ and $\mc{M}^\veps_{\res}$, defined as 
	$$\mc{M}^\veps_{\ess}  := \left\{ (t,x) \in\, ]0,T[\, \times\, \Omega_\veps \ : \ \bigl(\varrho_\ep(t,x),\vartheta_\ep(t,x)\bigr) \in  {\cal{O}}_{\ess} \right\}\qquad\mbox{ and }\qquad
	\mc{M}^\veps_{\res}  := \big(\,]0,T[\,\times\,\Omega_\veps\,\big) \setminus \mc{M}^\veps_{\ess}\,,$$
and their version at fixed time $t\geq0$, i.e. 
	$$\mc{M}^\veps_{\ess} [t] := \{ x \in \Omega_\veps \  : (t,x) \in \mc{M}^\veps_{\ess} \}\qquad \mbox{ and }\qquad
	\mc{M}^\veps_{\res}[t]  := \Omega_\veps \setminus \mc{M}^\veps_{\ess}[t]\,.$$

The next result, which will be useful in the next subsection, is the analogous of Lemma 5.1 in \cite{F-N} in our context. Here we need to pay attention to the fact that,
when $F\neq0$, the estimates for the equilibrium states (recall Proposition \ref{p:target-rho_bound}) are not uniform on the whole $\Omega_\veps$.
\begin{lemma}\label{l:H}
Fix $m\geq1$ and let $\vret$ and $\tems$ be the static states identified in Paragraph \ref{sss:equilibrium}. Under the previous assumptions, and with the notations introduced above,
we have the following properties.

Let $F\neq0$. For all $l>0$, there exist $\veps(l)$ and positive constants $c_j\,=\,c_j(\rho_*,\oline\vtheta,l)$, with $j=1,2,3$, such that, for all $0<\veps\leq\veps(l)$,
the next properties hold true, for all $x\in\oline{\mbb B}_{l}$:
\begin{enumerate}[(a)]
 \item for all $(\rho,\theta)\,\in\,\mc O_{\ess}$, one has
$$
c_1\,\left(\left|\rho-\wtilde\vrho_\veps(x)\right|^2\,+\,\left|\theta-\oline\vtheta\right|^2\right)\,\leq\,\mc E\left(\rho,\theta\;|\;\wtilde\vrho_\veps(x),\oline\vtheta\right)\,\leq\,
c_2\,\left(\left|\rho-\wtilde\vrho_\veps(x)\right|^2\,+\,\left|\theta-\oline\vtheta\right|^2\right)\,;
$$
\item for all $(\rho,\theta)\,\in\,\mc O_{\res}$, one has
$$
\mc E\left(\rho,\theta\;|\;\wtilde\vrho_\veps(x),\oline\vtheta\right)\,\geq\,c_3\,.
$$
\end{enumerate}

When $F=0$, the previous constants $\big(c_j\big)_{j=1,2,3}$ can be chosen to be independent of $l>0$.
\end{lemma}
\begin{proof}
Let us start by considering the case $F\neq0$. Fix $m\geq 1$. In view of Lemma \ref{l:target-rho_pos} and Proposition \ref{p:target-rho_bound}, for all $l>0$ fixed, there exists $\veps(l)$ such that,
for all $\veps\leq\veps(l)$, we have $\wtilde\vrho_\veps(x)\,\in\,[\rho_*,3/2]\,\subset\,\mc O_{\ess}$ for all $x\in \oline{\mbb B}_{l}$.
With this inclusion at hand, the first inequality is an immediate consequence of the decomposition
\begin{align*}
\mc E\left(\rho,\theta\;|\;\wtilde\vrho_\veps,\oline\vtheta\right)\,&=\,\Bigl(H_{\tems}(\rho,\theta)-H_{\tems}(\rho,\oline\vtheta)\Bigr)\,+\,
\Bigl(H_{\tems}(\rho,\oline\vtheta) - H_{\tems}(\wtilde\vrho_\veps,\oline\vtheta) - (\rho - \vret)\,\partial_\vrho H_{\tems}(\vret,\tems)\Bigr) \\
&=\,\d_\vtheta H_{\oline\vtheta}(\rho,\eta)\,\bigl(\vtheta-\oline\vtheta\bigr)\,+\, \frac{1}{2}\d^2_{\vrho\vrho}H_{\oline\vtheta}(z_\veps,\oline\vtheta)\,\bigl(\rho-\wtilde\vrho_\veps\bigr)^2\,,
\end{align*}
for some suitable $\eta$ belonging to the interval connecting $\theta$ and $\oline\vtheta$, and $z_\veps$ belonging to the interval connecting $\rho$ and $\wtilde\vrho_\veps$. Indeed,
it is enough to use formulas (2.49) and (2.50) of \cite{F-N}, together with the fact that we are in the essential set.

Next, thanks again to the property $\wtilde\vrho_\veps(x)\,\in\,[\rho_*,3/2]\,\subset\,\mc O_{\ess}$, we can conclude, exactly as in relation (6.69) of \cite{F-N}, that
$$
\inf_{(\rho,\theta)\in\mc O_\res}\mc E\left(\rho,\theta\;|\;\wtilde\vrho_\veps,\oline\vtheta\right)\,\geq\,
\inf_{(\rho,\theta)\in\d\mc O_\ess}\mc E\left(\rho,\theta\;|\;\wtilde\vrho_\veps,\oline\vtheta\right)\,\geq\,c\,>\,0\,.
$$


The case $F=0$ follows by similar arguments, using that the various constants in Lemma \ref{l:target-rho_bound} and Proposition \ref{p:target-rho_bound} are uniform in $\Omega$.
This completes the proof of the lemma.
\qed
\end{proof}

\subsubsection{Uniform estimates for the family of weak solutions} \label{sss:uniform}

With the total dissipation balance \eqref{est:dissip} and Lemma \ref{l:H} at hand, we can derive uniform bounds for our family of weak solutions.
Since this derivation is somehow classical, we limit ourselves to recall the main inequalities and sketch the proofs; we refer the reader to Chapters 5, 6 and 8 of \cite{F-N} for details.

To begin with, we remark that, owing to the assumptions fixed in Paragraph \ref{sss:data-weak} on the initial data and to the structural hypotheses of Paragraphs \ref{sss:primsys} and \ref{sss:structural},
the right-hand side of \eqref{est:dissip} is \emph{uniformly bounded} for all $\veps\in\,]0,1]$.
\begin{lemma} \label{l:initial-bound}
Under the assumptions fixed in Paragraphs \ref{sss:primsys}, \ref{sss:structural} and \ref{sss:data-weak}, there exists an absolute constant $C>0$ such that, for all $\veps\in\,]0,1]$,
one has
$$
\int_{\Omega_\veps} \frac{1}{2}\vrez|\uez|^2\,\dx + \frac{1}{\ep^{2m}}\int_{\Omega_\veps}\mc E\left(\vrho_{0,\veps},\vtheta_{0,\veps}\;|\;\wtilde\vrho_\veps,\oline\vtheta\right)\,\dx\,\leq\,C\,.
$$
\end{lemma}

\begin{proof}
The boundedness of the first term in the left-hand side is an obvious consequence of \eqref{hyp:ill-vel} and \eqref{hyp:ill_data} for the density. So, let us show how to control the term containing
$\mc E\left(\vrho_{0,\veps},\vtheta_{0,\veps}\;|\;\wtilde\vrho_\veps,\oline\vtheta\right)$. Owing to Taylor formula, one has
\begin{align*}
\mc E\left(\vrho_{0,\veps},\vtheta_{0,\veps}\;|\;\wtilde\vrho_\veps,\oline\vtheta\right)\,&=\,
\d_\vtheta H_{\oline\vtheta}(\vrho_{0,\veps},\eta_{0,\veps})\,\bigl(\vtheta_{0,\veps}-\oline\vtheta\bigr)\,+\,\frac{1}{2}\,
\d^2_{\vrho\vrho}H_{\oline\vtheta}(z_{0,\veps},\oline\vtheta)\,\bigl(\vrho_{0,\veps}-\wtilde\vrho_\veps\bigr)^2\,,
\end{align*}
where we can write $\eta_{0,\veps}(x)\,=\,\oline\vtheta\,+\,\veps^m\,\lambda_\veps(x)\,\Theta_{0,\veps}$ and $z_{0,\veps}\,=\,\wtilde\vrho_\veps\,+\,\veps^m\,\z_\veps(x)\,\vrho^{(1)}_{0,\veps}$,
with both the families $\bigl(\lambda_\veps\bigr)_\veps$ and $\bigl(\z_\veps\bigr)_\veps$ belonging to $L^\infty(\Omega_\veps)$, uniformly in $\veps$ (in fact,
$\lambda_\veps(x)$ and $\z_{\veps}(x)$ belong to the interval $\,]0,1[\,$ for all $x\in\Omega_\veps$).
Notice that $\big(\eta_{0,\veps}\big)_\veps\,\subset\,L^\infty(\Omega_\veps)$ and that $\eta_{0,\veps}\geq c_1>0$ and $z_{0,\veps}\geq c_2>0$ (at least for $\veps$ small enough).
By the structural hypotheses fixed in Paragraph \ref{sss:structural} (and in particular Gibbs' law), we get (see also formula (2.50) in \cite{F-N})
\begin{align} \label{eq:d_th-H_th}
\d_\vtheta H_{\oline\vtheta}(\vrho_{0,\veps},\eta_{0,\veps})\,&=\,4\,a\,\eta_{0,\veps}^2\,\bigl(\eta_{0,\veps}-\oline\vtheta\bigr)\,+\,
\frac{\vrho_{0,\veps}}{\eta_{0,\veps}}\,\bigl(\eta_{0,\veps}-\oline\vtheta\bigr)\,\d_\vtheta e_M(\rho_{0,\veps},\eta_{0,\veps})\,.
\end{align}
In view of condition \eqref{pp3}, we gather that $\left|\d_\vtheta e_M\right|\,\leq\,c$; therefore, from hypotheses \eqref{hyp:ill_data} and Remark \ref{r:ill_data}
it is easy to deduce that
$$
\frac{1}{\veps^{2m}}\int_{\Omega_\veps}\d_\vtheta H_{\oline\vtheta}(\vrho_{0,\veps},\eta_{0,\veps})\,\bigl(\vtheta_{0,\veps}-\oline\vtheta\bigr)\dx\,\leq\,C\,.
$$

Moreover, by \eqref{pp1} we get (keep in mind formula (2.49) of \cite{F-N})
$$
\d^2_{\vrho\vrho}H_{\oline\vtheta}(z_{0,\veps},\oline\vtheta)\,=\,\frac{1}{z_{0,\veps}}\,\d_\vrho p_M(z_{0,\veps},\oline\vtheta)\,=\,\frac{1}{\sqrt{\oline\vtheta}}\,\frac{1}{Z_{0,\veps}}\,P'(Z_{0,\veps})\,,
$$
where we have set $Z_{0,\veps}\,=\,z_{0,\veps}\,\oline\vtheta^{-3/2}$. Now, thanks to \eqref{pp3} again and to the fact that $z_{0,\veps}$ is strictly positive, we can estimate,
for some positive constants which depend also on $\oline\vtheta$,
$$
\frac{1}{Z_{0,\veps}}\,P'(Z_{0,\veps})\,\leq\,C\,\frac{P(Z_{0,\veps})}{Z_{0,\veps}^2}\,\leq\,C\left(\frac{P(Z_{0,\veps})}{Z_{0,\veps}^2}\,\bbbone_{\{0\leq Z_{0,\veps}\leq1\}}\,+\,
\frac{P(Z_{0,\veps})}{Z_{0,\veps}^{5/3}}\,\bbbone_{\{Z_{0,\veps}\geq1\}}\right)\,\leq\,C\,,
$$
where we have used also \eqref{pp4}.

Hence, we can check that
$$
\frac{1}{2\veps^{2m}}\int_{\Omega_\veps}\d^2_{\vrho\vrho}H_{\oline\vtheta}(z_{0,\veps},\oline\vtheta)\,\bigl(\vrho_{0,\veps}-\wtilde\vrho_\veps\bigr)^2\dx\,\leq\,C\,.
$$
This inequality completes the proof of the lemma.
\qed
\end{proof}

\medbreak

Owing to the previous lemma, from \eqref{est:dissip} we gather, for any $T>0$, the estimates
\begin{align}
	\sup_{t\in[0,T]} \| \sqrt{\vre}\ue\|_{L^2(\Omega_\veps;\R^3)}\, &\leq\,c \label{est:momentum} \\
	\| \sigma_\ep\|_{{\mathcal{M}}^+ ([0,T]\times\oline\Omega_\veps )}\, &\leq \,\ep^{2m}\, c\,. \label{est:sigma}
\end{align}
Fix now any $l>0$. Employing Lemma~\ref{l:H} (and keeping track of the dependence of constants only on $l$), we deduce
	\begin{align}
	\sup_{t\in[0,T]} \left\| \left[ \dfrac{\vre - \vret}{\ep^m}\right]_\ess (t) \right\|_{L^2(\mbb B_{l})}\,+\,
	\sup_{t\in[0,T]} \left\| \left[ \dfrac{\tem - \tems}{\ep^m}\right]_\ess (t) \right\|_{L^2(\mbb B_{l})}\,&\leq\, c(l)\,. \label{est:rho_ess} 
	\end{align}
In addition, we infer also that the measure of the ``residual set'' is small: more precisely, we have
\begin{equation}\label{est:M_res-measure}
	\sup_{t\in[0,T]} \int_{\mbb B_{l}}	\bbbone_{\mc{M}^\veps_\res[t]} \dx\,\leq \,\ep^{2m}\, c(l)\,.
	\end{equation}

{\begin{remark}\label{rmk:cut-off}
When $F=0$, thanks to Lemma \ref{l:target-rho_bound} and Proposition \ref{p:target-rho_bound},
one can see that estimates \eqref{est:rho_ess} and \eqref{est:M_res-measure} hold on the whole $\Omega_\veps$, without any need of taking the localisation on the cylinders $\B_l$.
From this observation, it is easy to see that, when $F=0$, we can replace $\mathbb{B}_l$ with the whole $\Omega_\veps$ in all the following estimates.
\end{remark}}

Now, we fix $l\geq 0$. We estimate
$$
\int_{\B_{l}}\left|\left[\vrho_\veps\,\log\vrho_\veps\right]_\res\right|\dx\,=\,
\int_{\B_{l}}\left|\vrho_\veps\,\log\vrho_\veps\right|\,\bbbone_{\{0\leq\vrho_\veps\leq2\rho_*/3\}}\dx\,+\,
\int_{\B_{l}}\left|\vrho_\veps\,\log\vrho_\veps\right|\,\bbbone_{\{\vrho_\veps\geq2\}}\dx\,.
$$
Thanks to \eqref{est:M_res-measure}, the former term in the right-hand side is easily controlled by $\veps^{2m}$, up to a suitable multiplicative constant also depending on $l$.
As for the latter term, we have to argue in a different way. Owing to inequalities \eqref{pp2}, \eqref{p_comp}, \eqref{pp3} and \eqref{pp4}, we get that
$\d^2_\vrho H_{\oline\vtheta}(\vrho,\oline\vtheta)\geq C/\vrho$; therefore, by direct integration we find
\begin{align*}
C\,\vrho_\veps\,\log\vrho_\veps\,-\,C\left(\vrho_\veps-1\right)\,&\leq\,H_{\oline\vtheta}(\vrho_\veps,\oline\vtheta)\,-\,H_{\oline\vtheta}(1,\oline\vtheta)\,-\,
\d_\vrho H_{\oline\vtheta}(1,\oline\vtheta)(\vrho_\ep-1) \\
&\leq\,\mc E\left(\vrho_\veps,\vtheta_\veps\;|\;\wtilde\vrho_\veps,\oline\vtheta\right)\,+\,\mc E\left(\wtilde\vrho_\veps,\oline\vtheta\;|\;1,\oline\vtheta\right)\,+\,
\Big( \d_\vrho H(\wtilde\vrho_\veps,\oline\theta) - \d_\vrho H(1,\oline\theta)\Big) \big( \vrho_\veps - \wtilde\vrho_\veps \big)\, ,
\end{align*}
since an expansion analogous to \eqref{eq:d_th-H_th} allows to gather that $H_{\oline\vtheta}(\vrho_\veps,\oline\vtheta)\,-\,H_{\oline\vtheta}(\vrho_\veps,\vtheta_\veps)\,\leq\,0$.
On the one hand, using \eqref{est:dissip}, Proposition \ref{p:target-rho_bound} and \eqref{est:M_res-measure} one deduces
\begin{equation*} \label{est:rho-log_prelim}
\left|\int_{\B_{l}\cap\mc O_\res}\left(\mc E\left(\vrho_\veps,\vtheta_\veps\;|\;\wtilde\vrho_\veps,\oline\vtheta\right)\,+\,
\mc E\left(\wtilde\vrho_\veps,\oline\vtheta\;|\;1,\oline\vtheta\right)\,+\,\Big( \d_\vrho H(\wtilde\vrho_\veps,\oline\theta) - \d_\vrho H(1,\oline\theta)\Big) \big( \vrho_\veps - \wtilde\vrho_\veps \big)\,\right)\dx \right|\,\leq\,C\,\veps^{2m}\,.
\end{equation*}
On the other hand, $\vrho_\veps\log\vrho_\veps-\left(\vrho_\veps-1\right)\,\geq\,\vrho_\veps\left(\log\vrho_\veps-1\right)\,\geq\,(1/2)\,\vrho_\veps\,\log\vrho_\veps$ whenever
$\vrho_\veps\geq e^2$. Hence,  since we have
$$
\int_{\B_{l}}\left|\vrho_\veps\,\log\vrho_\veps\right|\,\bbbone_{\{2\leq\vrho_\veps\leq e^2\}}\dx\,\leq\,C\,\veps^{2m}
$$
owing to \eqref{est:M_res-measure} again,
we finally infer that, for any fix $l>0$,
\begin{equation} \label{est:rho*log-rho}
\sup_{t\in[0,T]}\int_{\B_{l}}\left|\left[\vrho_\veps\,\log\vrho_\veps\right]_\res(t)\right|\dx\,\leq\,c(l)\,\veps^{2m}\,.
\end{equation}


Owing to inequality \eqref{est:rho*log-rho}, we deduce (exactly as in \cite{F-N}, see estimates (6.72) and (6.73) therein) that
	\begin{align}
	\sup_{t\in [0,T]} \int_{\B_{l}}
	\bigl(
	\left| [ \vre e(\vre,\tem)]_\res\right| +  \left| [ \vre s(\vre,\tem)]_\res\right|\bigr)\,\dx\,&\leq\,\ep^{2m}\, c (l)\,, \label{est:e-s_res}
	\end{align}
which in particular implies (again, we refer to Section 6.4.1 of \cite{F-N} for details) the following bounds:
\begin{align}
	\sup_{t\in [0,T]} \int_{\B_{l}} [ \vre]^{5/3}_\res (t)\,\dx \,+\,\sup_{t\in [0,T]} \int_{\B_{l}} [ \tem]^{4}_\res (t)\, \dx
	\,&\leq\,\ep^{2m}\, c (l)\,. \label{est:rho_res} 
\end{align}

Let us move further. In view of \eqref{S}, \eqref{ss}, \eqref{q} and \eqref{mu}, relation \eqref{est:sigma} implies
	\begin{align}
	\int_0^T \left\| \nabla_x \ue +\, ^t\nabla_x \ue  - \frac{2}{3} \div \ue \, \Id \right\|^2_{L^2(\Omega_\veps;\R^{3\times3})}\, \dt\,
	&\leq\, c \label{est:Du}  \\
	\int_0^T \left\| \nabla_x \left(\frac{\tem - \tems}{\ep^m}\right) \right\|^2_{L^2(\Omega_\veps;\R^3)}\, \dt\, +\,
	\int_0^T \left\| \nabla_x \left(\frac{\log(\tem) - \log(\tems)}{\ep^m}\right) \right\|^2_{L^2(\Omega_\veps;\R^3)} \,\dt\,
	&\leq\, c\,. \label{est:D-theta}
	\end{align}

{Thanks to the previous inequalities and \eqref{est:M_res-measure}, we can argue as in Subsection 8.2 of \cite{F-N}:
by generalizations of respectively Poincar\'e and Korn-Poincar\'e inequalities (see Propositions \ref{app:poincare_prop} and \ref{app:korn-poincare_prop}), for all $l>0$ we gather also}
\begin{align}
\int_0^T \left\| \frac{\tem - \tems}{\ep^m} \right\|^2_{W^{1,2}({\B_l};\R^3)}\, \dt \,+\,
\int_0^T \left\| \frac{\log(\tem) - \log(\tems)}{\ep^m} \right\|^2_{W^{1,2}({\B_l};\R^3)}\, \dt\,&\leq\,c(l) \label{est:theta-Sob} \\
\int_0^T \left\| \ue \right\|^2_{W^{1,2}(\B_l; \R^3)} \dt\,&\leq\,c(l)\,. \label{est:u-H^1}
\end{align}

Finally, we discover that
\begin{align}
\int^T_0\left\|\left[\frac{\vrho_\veps\,s(\vrho_\veps,\vtheta_\veps)}{\veps^m}\right]_{\res}\right\|^{2}_{L^{30/23}(\B_{l})}\dt\,+\,
\int^T_0\left\|\left[\frac{\vrho_\veps\,s(\vrho_\veps,\vtheta_\veps)}{\veps^m}\right]_{\res}\,
\vec{u}_\veps\right\|^{2}_{L^{30/29}(\B_{l})}\dt\,&\leq\,c(l)  \label{est:rho-s_res} \\
\int^T_0\left\|\frac{1}{\veps^m}\,\left[\frac{\kappa(\vtheta_\veps)}{\vtheta_\veps}\right]_{\res}\,
\nabla_{x}\vtheta_\veps(t)\right\|^{2}_{L^{1}(\B_l)}\dt\,&\leq\,c(l)\,.  \label{est:Dtheta_res}
\end{align}
The argument for proving \eqref{est:rho-s_res} and \eqref{est:Dtheta_res} is similar to one employed in the proof of Proposition 5.1 of \cite{F-N}, but here it is important
to get bounds for the $L^2$ norm in time (see also Remark \ref{r:bounds} below). Indeed, we have that
\begin{equation} \label{5.58_book}
\left[\vrho_\veps\,s(\vrho_\veps,\vtheta_\veps)\right]_{\res}\leq C\, \left[ \vrho_\veps +\vrho_\veps \, |\log \vrho_\veps |\, +\vrho_\veps \, |\log \vtheta_\veps -\log \oline \vtheta|+\vtheta_{\veps}^{3}\, \right]_{\res} 
\end{equation}
and thanks to the previous uniform bounds \eqref{est:rho_res} and \eqref{est:theta-Sob}, one has that
$\big(\left[\vrho_\veps \right]_{\res}\big)_\veps\subset L_{T}^{\infty}( L_{\rm loc}^{5/3})$,
$\big(\left[\vrho_\veps \, |\log \vrho_\veps |\,\right]_{\res}\big)_\veps\subset L_{T}^{\infty}( L_{\rm loc}^{q})$ for all $1\leq q< 5/3$ (see relation (5.60) in \cite{F-N}),
$\big(\left[\vrho_\veps \, |\log \vtheta_\veps -\log \oline \vtheta|\, \right]_{\res}\big)_\veps\subset L_{T}^{2}( L_{\rm loc}^{30/23})$ and finally
$\big(\left[\vtheta_{\veps}^{3}\, \right]_{\res}\big)_\veps\subset L_{T}^{\infty}( L_{\rm loc}^{4/3})$.
Let us recall that the inclusion symbol means that the sequences are uniformly bounded in the respective spaces.
Then, it follows that the first term in \eqref{est:rho-s_res}  is in $L_T^{2}(L_{\rm loc}^{30/23})$. 
Next, taking \eqref{5.58_book} we obtain
\begin{equation*}
\left[\vrho_\veps\,s(\vrho_\veps,\vtheta_\veps)\ue\right]_{\res}\leq C\, \left[ \vrho_\veps\ue +\vrho_\veps \, |\log \vrho_\veps |\, \ue\, +\vrho_\veps \, |\log \vtheta_\veps -\log \oline \vtheta|\,\ue +\vtheta_{\veps}^{3}\ue \, \right]_{\res} 
\end{equation*}
and using the uniform bounds \eqref{est:rho_res} and \eqref{est:u-H^1}, we have that $\big(\left[\vrho_\veps\ue\right]_{\res}\big)_\veps\subset L_T^{2}(L_{\rm loc}^{30/23})$. 
Now, we look at the second term. We know that $\big(\left[\vrho_\veps \, |\log \vrho_\veps |\, \right]_{\res}\big)_\veps\subset L_{T}^{\infty}( L_{\rm loc}^{q})$ for all $1\leq q< 5/3$ and
$\ue \in L_T^{2}(L_{\rm loc}^{6})$ (thanks to Sobolev embeddings, see Theorem \ref{app:sob_embedd_thm}). Then, we take $q$ such that $1/p:=1/q+1/6<1$ and so
$$ \big(\left[\vrho_\veps \, |\log \vrho_\veps |\, \ue\,\right]_{\res}\big)_\veps\subset L_T^{2}(L_{\rm loc}^{p})\, . $$
Keeping \eqref{est:rho_res}, \eqref{est:theta-Sob} and \eqref{est:momentum} in mind and using that
$$\left[\vrho_\veps \, |\log \vtheta_\veps -\log \oline \vtheta|\, \vec{u}_\veps\, \right]_{\res}=
\left[\sqrt{\vrho_\veps}\, |\log \vtheta_\veps -\log \oline \vtheta|\, \sqrt{\vrho_\veps}\, \vec{u}_{\veps}\, \right]_{\res}\,,$$
we obtain that the third term is uniformly bounded in $L_{T}^{2}( L_{\rm loc}^{30/29})$.
Using again the uniform bounds, we see that the last term is in $L_T^{\infty}(L_{\rm loc}^{12/11})$.
Thus, we obtain \eqref{est:rho-s_res}.  \\
To get \eqref{est:Dtheta_res}, we use instead the following estimate (see Proposition 5.1 of \cite{F-N}):
\begin{equation*}
\left[\frac{k(\vtheta_\veps )}{\vtheta_\veps}\right]_{\res}\left|\frac{\nabla_{x}\vtheta_\veps}{\veps^{m}}\right| \leq C
\left(\left|\frac{\nabla_{x}(\log \vtheta_\veps )}{\veps^m}\right|+\left[\vtheta_{\veps}^{2}\right]_{\res}\left|\frac{\nabla_{x}\vtheta_{\veps}}{\veps^m}\right|\right)\,.
\end{equation*} 
Owing to the previous uniform bounds, the former term is uniformly bounded in $L_{T}^{2}( L_{\rm loc}^{2})$ and the latter one is uniformly bounded in $L_{T}^{2}( L_{\rm loc}^{1})$.
So, we obtain the estimate \eqref{est:Dtheta_res}.

\begin{remark} \label{r:bounds}
Differently from \cite{F-N}, here we have made the integrability indices in \eqref{est:rho-s_res} and \eqref{est:Dtheta_res} explicit.
In particular, having the $L^2$ norm in time will reveal to be fundamental for the compensated compactness argument, see Lemma \ref{l:source_bounds} below.
\end{remark}

\subsection{Constraints on the limit dynamics}\label{ss:ctl1}

In this subsection, we establish some properties that the limit points of the family $\bigl(\vrho_\veps,\vec u_\veps,\vtheta_\veps\bigr)_\veps$ have to satisfy.
These are static relations, which do not characterise the limit dynamics yet.

\subsubsection{Preliminary considerations} \label{sss:constr_prelim}
To begin with, we propose an extension of Proposition 5.2 of \cite{F-N}, which will be heavily used in the sequel. Two are the novelties here: firstly,
for the sake of generality we will consider a non-constant density profile $\wtilde\vrho$ in the limit (although this property is not used in our analysis); in addition, due to the
centrifugal force, when $F\neq0$ our result needs a localization procedure on compact sets. 

\begin{proposition} \label{p:prop_5.2}
Let $m\geq1$ be fixed. Let $\wtilde\vrho_\veps$ and $\oline\vtheta$ be the static solutions identified and studied in Paragraph \ref{sss:equilibrium}, and take $\wtilde\vrho$
to be the pointwise limit of the family $\left(\wtilde\vrho_\veps\right)_\veps$ (in particular, $\wtilde\vrho\equiv1$ if $m>1$ or $m=1$ and $F=0$).
Let $(\varrho_\ep)_\ep$ and $(\vartheta_\ep)_\ep$ be sequences of non-negative measurable functions, and define
$$
R_\veps\,:=\,\frac{\varrho_\ep - \wtilde\vrho}{\ep^m}\qquad\mbox{ and }\qquad
\Theta_\veps\,:=\,\frac{\vartheta_\ep - \oline\vartheta}{\ep^m}\,.
$$
Suppose that, in the limit $\veps\ra0^+$, one has the convergence properties
\begin{equation}\label{hyp_p5.2:1}
\left[R_\veps\right]_{\rm ess}\, \weakstar\, R \quad\mbox{ and }\quad
\left[\Theta_\veps\right]_{\rm ess}\, \weakstar\, \Theta\qquad\quad \mbox{ in the weak-$*$ topology of} \ L^\infty\bigl([0,T];L^2(K)\bigr)\,,
\end{equation}
for any compact $K\subset \Omega$, and that, for any $L>0$, one has
\begin{equation}\label{hyp_p5.2:2}
\sup_{t\in[0,T]} \int_{\mbb B_{L}} \bbbone_{\mathcal{M}^\ep_{\rm res} [t] }\,dx\, \leq \,c(L)\,\ep^{2m}\,.
\end{equation}

Then, for any given function $G \in C^1(\overline{\mathcal{O}}_{\rm ess})$,  one has the convergence
$$
	\frac{[G(\varrho_\ep,\vartheta_\ep)]_{\ess} - G(\wtilde\vrho,\oline\vartheta)}{\ep^m}\,
	\weakstar\,\partial_\vrho G(\wtilde\vrho,\oline\vartheta)\,R\, +\,\partial_\vtheta G(\wtilde\vrho,\oline\vartheta)\,\Theta
	\qquad \mbox{ in the weak-$*$ topology of} \ L^\infty\bigl([0,T];L^2(K)\bigr)\,,
$$
for any compact $K\subset \Omega$.
\end{proposition}

\begin{proof}
The case $\wtilde\vrho\equiv1$ follows by a straightforward adaptation of the proof of Proposition 5.2 of \cite{F-N}. So, let us immediately focus on the case $m=1$ and $F\neq0$,
so that the target profile $\wtilde\vrho$ is non-constant.

We start by observing that, by virtue of \eqref{hyp_p5.2:2} and Lemma \ref{l:target-rho_bound}, the estimates
$$
\frac{1}{\veps}\,\left\|\left[G(\wtilde\vrho,\oline\vtheta)\right]_{\rm res}\right\|_{L^1(\mbb B_L)}\,\leq\,C(L)\,\veps\qquad\mbox{ and }\qquad
\frac{1}{\veps}\,\left\|\left[G(\wtilde\vrho,\oline\vtheta)\right]_{\rm res}\right\|_{L^2(\mbb B_L)}\,\leq\,C(L)
$$
hold true, for any $L>0$ fixed. Combining those bounds with hypothesis \eqref{hyp_p5.2:1}, after taking $L>0$ so large that $K\subset\mbb B_L$, we see that it is enough to prove the convergence
\begin{equation} \label{conv:to-prove}
\int_{K}\left[\frac{G(\varrho_\ep,\vartheta_\ep)- G(\wtilde\vrho,\oline\vartheta)}{\ep}\,-\,
\partial_\vrho G(\wtilde\vrho,\oline\vartheta)\,R_\veps\,-\,\partial_\vtheta G(\wtilde\vrho,\oline\vartheta)\,\Theta_\veps\right]_{\ess}\,\psi\,\dx\,\longrightarrow\,0
\end{equation}
for any compact $K$ fixed and any $\psi\in L^1\bigl([0,T];L^2(K)\bigr)$.

Next, we remark that, whenever $G\in C^2(\oline{\mc O}_{\rm ess})$, we have
\begin{align}
&\left|\left[\frac{G(\varrho_\ep,\vartheta_\ep)- G(\wtilde\vrho,\oline\vartheta)}{\ep}\,-\,
\partial_\vrho G(\wtilde\vrho,\oline\vartheta)\,R_\veps-\partial_\vtheta G(\wtilde\vrho,\oline\vartheta)\,\Theta_\veps\right]_{\ess}\right|\,\leq \label{est:prelim_5.2} \\
&\qquad\qquad\qquad\qquad\qquad\qquad\qquad\quad
\leq\,C\,\veps\,\left\|{\rm Hess}(G)\right\|_{L^\infty(\oline{\mc O}_{\rm ess})}\left(\left[R_\veps\right]_{\rm ess}^2+\left[\Theta_\veps\right]_{\rm ess}^2\right), \nonumber
\end{align}
where we have denoted by ${\rm Hess}(G)$ the Hessian matrix of the function $G$ with respect to its variables $(\vrho,\vtheta)$.
In particular, \eqref{est:prelim_5.2} implies the estimate
\begin{align}
\left\|\left[\frac{G(\varrho_\ep,\vartheta_\ep)- G(\wtilde\vrho,\oline\vartheta)}{\ep}\,-\,
\partial_\vrho G(\wtilde\vrho,\oline\vartheta)\,R_\veps\,-\,\partial_\vtheta G(\wtilde\vrho,\oline\vartheta)\,\Theta_\veps\right]_{\ess}\right\|_{L^\infty_T(L^1(K))}\,
\leq\,C\,\veps\,. \label{est:prop_5.2}
\end{align}
Property \eqref{conv:to-prove} then follows from \eqref{est:prop_5.2}, after noticing that both terms $\left[G(\varrho_\ep,\vartheta_\ep)- G(\wtilde\vrho,\oline\vartheta)\right]_\ess/\veps$
and $\left[\partial_\vrho G(\wtilde\vrho,\oline\vartheta)\,R_\veps+\partial_\vtheta G(\wtilde\vrho,\oline\vartheta)\,\Theta_\veps\right]_{\ess}$
are uniformly bounded in $L_T^\infty\bigl(L^2(K)\bigr)$.

Finally, when $G$ is just $C^1(\oline{\mc O}_{\rm ess})$, we approximate it by a family of smooth functions $\bigl(G_n\bigr)_{n\in\N}$, uniformly in $C^1(\oline{\mc O}_{\rm ess})$.
Obviously, for each $n$, convergence \eqref{conv:to-prove} holds true for $G_n$. Moreover, we have
$$
\left|\left[\frac{G(\varrho_\ep,\vartheta_\ep)- G(\wtilde\vrho,\oline\vartheta)}{\ep}\right]_{\rm ess}-
\left[\frac{G_n(\varrho_\ep,\vartheta_\ep)- G_n(\wtilde\vrho,\oline\vartheta)}{\ep}\right]_{\rm ess}\right|\,\leq\,C\,
\left\|G\,-\,G_n\right\|_{C^1(\oline{\mc O}_{\rm ess})}\left(\left[R_\veps\right]_{\rm ess}+\left[\Theta_\veps\right]_{\rm ess}\right)\,,
$$
and a similar bound holds for the terms presenting partial derivatives of $G$. In particular, these controls entail that the remainders, created replacing $G$ by $G_n$ in
\eqref{conv:to-prove}, are uniformly small in $\veps$, whenever $n$ is sufficiently large.
This completes the proof of the proposition.
\qed
\end{proof}

\medbreak
From now on, we will focus on the two cases \eqref{eq:choice-m}: either $m\geq2$ and possibly $F\neq0$, or $m\geq1$ and $F=0$. We explain this in the next remark.
\begin{remark} \label{slow_rho}
If $1<m<2$ and $F\neq 0$, the structure of the wave system (see Paragraph 4.1.1) is much more complicated, since the centrifugal force term becomes singular; in turn, this prevents us from proving
that the quantity $\g_\veps$ (see details below) is compact, a fact which is a key point in the convergence step. On the other hand, the idea of combining the centrifugal force term
with $\g_\veps$, in order to gain compactness of a new quantity, does not seem to work either, because, owing to temperature variations (and differently from [15] where the temperature was constant), there is no direct relation
between the centrifugal force and the pressure term. 
\end{remark}



Recall that, in both cases presented in \eqref{eq:choice-m}, the limit density profile is always constant, say $\wtilde\vrho\equiv1$. Let us fix an arbitrary positive time
$T>0$, which we keep fixed until the end of this paragraph.
Thanks to \eqref{est:rho_ess}, \eqref{est:rho_res} and Proposition \ref{p:target-rho_bound}, we get
	\begin{equation}\label{rr1}
\| \vre - 1 \|_{L^\infty_T(L^2 + L^{5/3}(K))}\,\leq\, \ep^m\,c(K) \qquad \mbox{ for all }\;
K \subset \Omega\quad\mbox{ compact.}
	\end{equation}
In particular, keeping in mind the notations introduced in \eqref{in_vr} and \eqref{eq:in-dens_dec}, we can define
\begin{equation} \label{def_deltarho}
R_\veps\,:= \frac{\varrho_\ep -1}{\ep^m} = \,\vrho_\veps^{(1)}\,+\,\wtilde{r}_\veps\;,\qquad\quad\mbox{ where }\quad
\vrho_\veps^{(1)}(t,x)\,:=\,\frac{\vre-\wtilde{\vrho}_\veps}{\ep^m}\quad\mbox{ and }\quad
\wtilde{r}_\veps(x)\,:=\,\frac{\wtilde{\vrho}_\veps-1}{\ep^m}\,.
\end{equation}
Thanks to \eqref{est:rho_ess}, \eqref{est:rho_res} and Proposition \ref{p:target-rho_bound}, the previous quantities verify the following bounds:
\begin{equation}\label{uni_varrho1}
\sup_{\veps\in\,]0,1]}\left\|\vrho_\veps^{(1)}\right\|_{L^\infty_T(L^2+L^{5/3}({\B_{l}}))}\,\leq\, c \qquad\qquad\mbox{ and }\qquad\qquad
\sup_{\veps\in\,]0,1]}\left\| \wtilde{r}_\veps \right\|_{L^{\infty}(\B_{l})}\,\leq\, c \,.
\end{equation}
As usual, here above the radius $l>0$ is fixed (and the constants $c$ depend on it). In addition, in the case $F=0$, there is no need of localising in $\B_l$, and one gets instead
\begin{equation*}
\sup_{\veps\in\,]0,1]}\left\|\vrho_\veps^{(1)}\right\|_{L^\infty_T(L^2+L^{5/3}(\Omega_\veps))}\,\leq\, c \qquad\qquad\mbox{ and }\qquad\qquad
\sup_{\veps\in\,]0,1]}\left\| \wtilde{r}_\veps \right\|_{L^{\infty}(\Omega_\veps)}\,\leq\,\sup_{\veps\in\,]0,1]}\left\| \wtilde{r}_\veps \right\|_{L^{\infty}(\Omega)}\,\leq\, c \,.
\end{equation*}
In view of the previous properties, there exist $\vrho^{(1)}\in L^\infty_T(L^{5/3}_{\rm loc})$ and
$\wtilde{r}\in L^\infty_{\rm loc}$ such that (up to the extraction of a suitable subsequence)
\begin{equation} \label{conv:rr}
\vrho_\veps^{(1)}\,\weakstar\,\vrho^{(1)}\qquad\quad \mbox{ and }\qquad\quad \wtilde{r}_\veps\,\weakstar\,\wtilde{r}\,,
\end{equation}
where we understand that limits are taken in the weak-$*$ topology of the respective spaces.
Therefore,
	\begin{equation}\label{conv:r}
	R_\veps\, \weakstar\,R\,:=\,\vrho^{(1)}\,+\,\wtilde{r}\qquad\qquad\qquad \mbox{ weakly-$*$ in }\quad L^\infty\bigl([0,T]; L^{5/3}_{\rm loc}(\Omega)\bigr)\,.
	\end{equation}
Observe that $\wtilde r$ can be interpreted as a datum of our problem.
Moreover, owing to Proposition~\ref{p:target-rho_bound} and \eqref{est:rho_ess}, we also get
$$
	\left[R_\veps \right]_{\rm  ess}\weakstar R\qquad\qquad \mbox{ weakly-$*$ in }\quad L^\infty\bigl([0,T]; L^2_{\rm loc}(\Omega)\bigr)\,.
$$

In a pretty similar way, we also find that
\begin{align}
\Theta_\veps\,:=\,\frac{\vtheta_\veps\,-\,\oline{\vtheta}}{\veps^m}\,&\rightharpoonup\,\Theta
\qquad\qquad\mbox{ in }\qquad L^2\bigl([0,T];W^{1,2}_{\rm loc}(\Omega)\bigr) \label{conv:theta} \\
\vec{u}_\veps\,&\weak\,\vec{U}\qquad\qquad\mbox{ in }\qquad L^2\bigl([0,T];W_{\rm loc}^{1,2}(\Omega)\bigr)\,. \label{conv:u}
\end{align}

Let us infer now some properties that these weak limits have to satisfy, starting with the case of anisotropic scaling, namely, in view of \eqref{eq:choice-m}, either $m\geq2$, or $m>1$ and $F=0$.

\subsubsection{The case of anisotropic scaling} \label{ss:constr_2}

When $m\geq 2$, or $m>1$ and $F=0$, the system presents multiple scales, which act and interact at the same time; however, the low Mach number limit has a predominant effect.
As established in the next proposition, this fact imposes some rigid constraints on the target profiles.


\begin{proposition} \label{p:limitpoint}
Let $m\geq2$, or $m>1$ and $F=0$ in \eqref{ceq} to \eqref{eeq}.
Let $\left( \vre, \ue, \tem\right)_{\veps}$ be a family of weak solutions, related to initial data $\left(\vrho_{0,\veps},\vec u_{0,\veps},\vtheta_{0,\veps}\right)_\veps$
verifying the hypotheses of Paragraph \ref{sss:data-weak}. Let $(R, \vec{U},\Theta )$ be a limit point of the sequence
$\left( R_\veps, \ue,\Theta_\veps\right)_{\veps}$, as identified in Subsection \ref{sss:constr_prelim}. Then,
\begin{align}
&\vec{U}\,=\,\,\Big(\vec{U}^h\,,\,0\Big)\,,\qquad\qquad \mbox{ with }\qquad \vec{U}^h\,=\,\vec{U}^h(t,x^h)\quad \mbox{ and }\quad \div_{\!h}\,\vec{U}^h\,=\,0 \label{eq:anis-lim_1} \\[1ex]
&\nabla_x\Big(\d_\varrho p(1,\oline{\vtheta})\,R\,+\,\d_\vtheta p(1,\oline{\vtheta})\,\Theta\Big)\,=\,\nabla_x G\,+\,\delta_2(m)\nabla_x F
\qquad\qquad\mbox{ a.e. in }\;\,\R_+\times \Omega \label{eq:anis-lim_2} \\[1ex]
&\d_{t} \Upsilon +\div_{h}\left( \Upsilon \vec{U}^{h}\right) -\frac{\kappa(\oline\vtheta)}{\oline\vtheta} \Delta \Theta =0\,,\qquad\qquad
\mbox{ with }\qquad \Upsilon\,:=\,\d_\vrho s(1,\oline{\vtheta})R + \d_\vtheta s(1,\oline{\vtheta})\,\Theta\,,
\label{eq:anis-lim_3}
\end{align}
where 
the last equation is supplemented with the initial condition
$\Upsilon_{|t=0}=\d_\vrho s(1,\oline\vtheta)\,R_0\,+\,\d_\vtheta s(1,\oline\vtheta)\,\Theta_0$.
\end{proposition}

\begin{proof}
Let us focus here on the case $m\geq 2$ and $F\neq 0$. A similar analysis yields the result also in the case $m>1$, provided we take $F=0$.

First of all, let us consider the weak formulation of the mass equation \eqref{ceq}: for any test function $\varphi\in C_c^\infty\bigl(\R_+\times\Omega\bigr)$, denoting $[0,T]\times K\,=\,\Supp \, \varphi$, with $\vphi(T,\cdot)\equiv0$, we have
$$
-\int^T_0\int_K\bigl(\vrho_\veps-1\bigr)\,\d_t\varphi \dxdt\,-\,\int^T_0\int_K\vrho_\veps\,\vec{u}_\veps\,\cdot\,\nabla_{x}\varphi \dxdt\,=\,
\int_K\bigl(\vrho_{0,\veps}-1\bigr)\,\varphi(0,\,\cdot\,)\dx\,.
$$
We can easily pass to the limit in this equation, thanks to the strong convergence $\vrho_\veps\longrightarrow1$ provided by \eqref{rr1} and the weak convergence of
$\vec{u}_\veps$ in $L_T^2\bigl(L^6_{\rm loc}\bigr)$ (by \eqref{conv:u} and Sobolev embeddings): we find
$$
-\,\int^T_0\int_K\vec{U}\,\cdot\,\nabla_{x}\varphi \dxdt\,=\,0\, ,
$$
for any test function $\varphi \, \in C_c^\infty\bigl([0,T[\,\times\Omega\bigr)$, which in particular implies
\begin{equation} \label{eq:div-free}
\div \U = 0 \qquad\qquad\mbox{ a.e. in }\; \,\R_+\times \Omega\,.
\end{equation}

Let us now consider the momentum equation \eqref{meq}, in its weak formulation \eqref{weak-mom}.
First of all, we test the momentum equation on $\veps^m\,\vec\phi$, for a smooth compactly supported $\vec\phi$.
By use of the uniform bounds we got in Subsection \ref{ss:unif-est}, it is easy to see that the only terms which do not converge to $0$ are the ones involving the pressure and
the gravitational force; in the endpoint case $m=2$, we also have the contribution of the centrifugal force. Hence, let us focus on them, and more precisely on the quantity
\begin{align}
\Xi\,:&=\,\frac{\nabla_x p(\vrho_\veps,\vtheta_\veps)}{\veps^m}\,-\,\veps^{m-2}\,\vrho_\veps\nabla_x F\,-\,\vrho_\veps\nabla_x G \label{eq:mom_rest_1}\\ 
&=\frac{1}{\veps^m}\nabla_x\left(p(\vrho_\veps,\vtheta_\veps)\,-\,p(\wtilde{\vrho}_\veps,\oline{\vtheta})\right)\,-\,
\veps^{m-2}\,\left(\vrho_\veps-\wtilde{\vrho}_\veps\right)\nabla_x F\,-\,\left(\vrho_\veps-\wtilde{\vrho}_\veps\right)\nabla_x G\,, \nonumber
\end{align}
where we have used relation \eqref{prF}.
By uniform bounds and \eqref{conv:r}, the second and third terms in the right-hand side of \eqref{eq:mom_rest_1} converge to $0$,
when tested against any smooth compactly supported $\vec\phi$; notice that this is true actually for any $m>1$.
On the other hand, for the first item we can use the decomposition
$$
\frac{1}{\veps^m}\,\nabla_x\left(p(\vrho_\veps,\vtheta_\veps)\,-\,p(\wtilde{\vrho}_\veps,\oline{\vtheta})\right)\,=\,
\frac{1}{\veps^m}\,\nabla_x\left(p(\vrho_\veps,\vtheta_\veps)\,-\,p(1,\oline{\vtheta})\right)\,-\,
\frac{1}{\veps^m}\,\nabla_x\left(p(\wtilde{\vrho}_\veps,\oline{\vtheta})\,-\,p(1,\oline{\vtheta})\right)\,.
$$

Due to the smallness of the residual set \eqref{est:M_res-measure} and to estimate \eqref{est:rho_res}, decomposing $p$ into essential and residual part
and then applying Proposition \ref{p:prop_5.2}, we get the convergence
$$
\frac{1}{\veps^m}\,\nabla_x\left(p(\vrho_\veps,\vtheta_\veps)\,-\,p(1,\oline{\vtheta})\right)\;\stackrel{*}{\rightharpoonup}\;
\nabla_x\left(\d_\varrho p(1,\oline{\vtheta})\,R\,+\,\d_\vtheta p(1,\oline{\vtheta})\,\Theta\right)
$$
in $L_T^\infty(H^{-1}_{\rm loc})$, for any $T>0$.
On the other hand, a Taylor expansion of $p(\,\cdot\,,\oline{\vtheta})$ up to the second order around $1$ gives, together with Proposition \ref{p:target-rho_bound}, the bound
$$
\left\|\frac{1}{\veps^m}\,\left(p(\wtilde{\vrho}_\veps,\oline{\vtheta})\,-\,p(1,\oline{\vtheta})\right)\,-\,
\d_\varrho p(1,\oline{\vtheta})\,\wtilde{r}_\veps\right\|_{L^\infty(K)}\,\leq\,C(K)\,\veps^m\, ,
$$
for any compact set $K\subset\Omega$. From the previous estimate we deduce that
$\left(p(\wtilde{\vrho}_\veps,\oline{\vtheta})\,-\,p(1,\oline{\vtheta})\right)/\veps^m\,\longrightarrow\,
\d_\varrho p(1,\oline{\vtheta})\,\wtilde{r}$ in e.g. $\mc{D}'\bigl(\R_+\times\Omega\bigr)$.

Putting all these facts together and keeping in mind relation \eqref{conv:r}, thanks to \eqref{eq:mom_rest_1} we finally find the celebrated \emph{Boussinesq relation}
	\begin{equation} \label{eq:rho-theta}
	\nabla_x\left(\d_\varrho p(1,\oline{\vtheta})\,\vrho^{(1)}\,+\,\d_\vtheta p(1,\oline{\vtheta})\,\Theta\right)\,=\,0
	\qquad\qquad\mbox{ a.e. in }\; \R_+\times \Omega\,.
	\end{equation}

\begin{remark} \label{r:F-G}
Notice that, dividing \eqref{prF} by $\veps^m$ and passing to the limit in it, one gets the identity
$$
\d_\varrho p(1,\oline{\vtheta})\,\nabla_x\wtilde{r} \,=\,\nabla_x G\,+\,\delta_2(m)\nabla_x F\,,
$$
where we have set $\delta_2(m)=1$ if $m=2$, $\delta_2(m)=0$ otherwise.
Hence, relation \eqref{eq:rho-theta} is equivalent to equality \eqref{eq:anis-lim_2}, 
which might be more familiar to the reader (see formula (5.10) in Chapter 5 of \cite{F-N}).
\end{remark}



Up to now, the contribution of the fast rotation in the limit has not been seen: this is due to the fact that the incompressible
limit takes place faster than the high rotation limit, because $m>1$. Roughly speaking, the rotation term enters into the singular
perturbation operator as a ``lower order'' part; nonetheless, being singular, it does impose some conditions on the limit dynamics.

To make this rigorous, we test \eqref{meq} on $\veps\,\vec\phi$, where this time we take $\vec\phi\,=\,\curl\vec\psi$, for some smooth compactly supported $\vec\psi\,\in C^\infty_c\bigl([0,T[\,\times\Omega\bigr)$.
Once again, by uniform bounds we infer that the $\d_t$ term, the convective term and the viscosity term all converge to $0$ when $\veps\ra0^+$.
As for the pressure and the external forces, we repeat the same manipulations as before: making use of relation \eqref{prF} again, we are reconducted to work on
$$
\int^T_0\int_K\left(\frac{1}{\veps^{2m-1}}\nabla_x\left(p(\vrho_\veps,\vtheta_\veps)\,-\,p(\wtilde{\vrho}_\veps,\oline{\vtheta})\right)\,-\,
\frac{\vrho_\veps-\wtilde{\vrho}_\veps}{\veps}\,\nabla_x F\,-\,\frac{\vrho_\veps-\wtilde{\vrho}_\veps}{\veps^{m-1}}\nabla_x G\right)\cdot\vec\phi\,\dx\,dt\,,
$$
where the compact set $K\subset\Omega$ is such that $\Supp\vec\phi\subset[0,T[\,\times K$, and $\veps>0$ is small enough.
According  to \eqref{rr1}, the two forcing terms converge to $0$, in the limit for $\veps\ra0^+$; on the other hand, the first term (which has no chance to be bounded uniformly in $\veps$)
simply vanishes, due to the fact that $\vec\phi\,=\,{\rm curl}\,\vec\psi$.

Finally, using a priori bounds and properties \eqref{conv:r} and
\eqref{conv:u}, it is easy to see that the rotation term converges to $\int^T_0\int_K\e_3\times\vec{U}\cdot\vec\phi$.
In the end, passing to the limit for $\veps\ra0^+$ we find
$$
\mbb{H}\left(\e_3\times\vec{U}\right)\,=\,0\qquad\qquad\text{and so}\qquad\qquad \e_3\times\vec{U}\,=\,\nabla_x\Phi\,
$$
for some potential function $\Phi$. From this relation, which in components reads 
\begin{equation}\label{limit_U_components}
\begin{pmatrix}
-U^2 \\ 
U^1\\
0
\end{pmatrix} =\begin{pmatrix}
\d_1 \Phi  \\ 
\d_2 \Phi \\ 
\d_3 \Phi
\end{pmatrix} \, ,
\end{equation}
we deduce that $\Phi=\Phi(t,x^h)$, i.e. $\Phi$ does not depend
on $x^3$, and that the same property is inherited by $\vec{U}^h\,=\,\bigl(U^1,U^2\bigr)$, i.e. $\vec{U}^h\,=\,\vec{U}^h(t,x^h)$. Furthermore, from \eqref{limit_U_components}, it is also easy to see that
the $2$-D flow given by $\vec{U}^h$ is incompressible, namely $\div_{\!h}\,\vec{U}^h\,=\,0$.
Combining this fact with \eqref{eq:div-free}, we infer that $\d_3 U^3\,=\,0$; on the other hand, thanks to the boundary condition
\eqref{bc1-2} we must have $\bigl(\vec{U}\cdot\vec{n}\bigr)_{|\d\Omega}\,=\,0$. Keeping in mind that
$\d\Omega\,=\,\bigl(\R^2\times\{0\}\bigr)\cup\bigl(\R^2\times\{1\}\bigr)$, we finally get $U^3\,\equiv\,0$,
whence \eqref{eq:anis-lim_1} finally follows.

Next, we observe that we can by now pass to the limit in the weak formulation \eqref{weak-ent} of \eqref{eiq}.
The argument being analogous to the one used in \cite{F-N} (see Paragraph 5.3.2), we only sketch it.
First of all, testing \eqref{eiq} on $\varphi/\veps^m$, for some $\varphi\in C^\infty_c\bigl([0,T[\,\times\Omega\bigr)$, and using \eqref{ceq}, for $\veps>0$ small enough we get
\begin{align}
&-\int^T_0\!\!\int_K\vrho_\veps\left(\frac{s(\vrho_\veps,\vtheta_\veps)-s(1,\oline{\vtheta})}{\veps^m}\right)\d_t\varphi -
\int^T_0\!\!\int_K\vrho_\veps\left(\frac{s(\vrho_\veps,\vtheta_\veps)-s(1,\oline{\vtheta})}{\veps^m}\right)\vec{u}_\veps\cdot\nabla_x\varphi \label{weak:entropy}  \\
&+\int^T_0\!\!\int_K\frac{\kappa(\vtheta_\veps)}{\vtheta_\veps}\,\frac{1}{\veps^m}\,\nabla_x\vtheta_\veps\cdot\nabla_x\varphi-
\frac{1}{\veps^m}\,\langle\sigma_\veps,\varphi\rangle_{[\mc{M}^+,C^0]([0,T]\times K)} =
\int_K\vrho_{0,\veps}\left(\frac{s(\vrho_{0,\veps},\vtheta_{0,\veps})-s(1,\oline{\vtheta})}{\veps^m}\right)\varphi(0)\,.  \nonumber
\end{align}

To begin with, let us decompose
\begin{align}
&\vrho_\veps\left(\frac{s(\vrho_\veps,\vtheta_\veps)-s(1,\oline{\vtheta})}{\veps^m}\right) =  \label{eq:dec_rho-s} \\
&=[\vrho_\veps]_{\ess}\left(\frac{[s(\vrho_\veps,\vtheta_\veps)]_{\ess}-s(1,\oline{\vtheta})}{\veps^m}\right) + 
\left[\frac{\vrho_\veps}{\veps^m}\right]_{\res}\left([s(\vrho_\veps,\vtheta_\veps)]_{\ess}-s(1,\oline{\vtheta})\right) +
\left[\frac{\vrho_\veps\,s(\vrho_\veps,\vtheta_\veps)}{\ep^m}\right]_{\res}\,.  \nonumber
\end{align}
Thanks to \eqref{est:rho_res}, we discover that the second term in the right-hand side strongly converges to $0$ in
$L_T^\infty(L^{5/3}_{\rm loc})$. 
Also the third term converges to $0$ in the space $ L_T^2(L^{30/23}_{\rm loc})$, as a consequence of \eqref{est:M_res-measure} and \eqref{est:rho-s_res}.
Notice that these terms converge to $0$ even when multiplied by $\vec{u}_\veps$: to see this, it is enough to put \eqref{est:M_res-measure},
\eqref{est:rho-s_res}, \eqref{est:u-H^1} and the previous properties together.

As for the first term in the right-hand side of \eqref{eq:dec_rho-s}, Propositions \ref{p:prop_5.2} and \ref{p:target-rho_bound} and estimate \eqref{rr1} imply that it
weakly converges to $\d_\vrho s(1,\oline{\vtheta})\,R\,+\,\d_\vtheta s(1,\oline{\vtheta})\,\Theta$, where $R$ and $\Theta$ are defined respectively in \eqref{conv:r}
and \eqref{conv:theta}. On the other hand, an application of the Div-Curl Lemma (see Theorem \ref{app:div-curl_lem}) gives 
$$
[\vrho_\veps]_{\ess}\left(\frac{[s(\vrho_\veps,\vtheta_\veps)]_{\ess}-s(1,\oline{\vtheta})}{\veps^m}\right)\,\vec{u}_\veps\,\rightharpoonup\,
\Bigl(\d_\vrho s(1,\oline{\vtheta})\,R\,+\,\d_\vtheta s(1,\oline{\vtheta})\,\Theta\Bigr)\,\vec{U}
$$
in the space $L_T^2(L^{3/2}_{\rm loc})$.  
In addition, from \eqref{est:sigma} we deduce that
$$
\frac{1}{\veps^m}\,\langle\sigma_\veps,\varphi\rangle_{[\mc{M}^+,C^0]([0,T]\times K)}\,\longrightarrow\,0
$$
when $\veps\ra0^+$.
Finally, a separation into essential and residual part of the coefficient $\kappa(\vtheta_\veps)/\vtheta_\veps$, together with \eqref{mu}, \eqref{est:rho_ess},
\eqref{est:rho_res},   \eqref{est:theta-Sob}  and \eqref{est:Dtheta_res} gives
$$
\frac{\kappa(\vtheta_\veps)}{\vtheta_\veps}\,\frac{1}{\veps^m}\,\nabla_x\vtheta_\veps\,\rightharpoonup\,
\frac{\kappa(\oline\vtheta)}{\oline\vtheta}\,\nabla_x\Theta\qquad\qquad\mbox{ in }\qquad L^2\bigl([0,T];L^{1}_{\rm loc}(\Omega)\bigr)\,.
$$

In the end, we have proved that equation \eqref{weak:entropy} converges, for $\veps\ra0^+$, to equation
\begin{align}
&-\int^T_0\int_\Omega\Bigl(\d_\vrho s(1,\oline{\vtheta})R + \d_\vtheta s(1,\oline{\vtheta})\,\Theta\Bigr)\left(\d_t\varphi + 
\vec{U}\cdot\nabla_x\varphi\right)\dxdt \,+ \label{eq:ent_bal_lim_1} \\
&\qquad\qquad+ \int^T_0\int_\Omega\frac{\kappa(\oline\vtheta)}{\oline\vtheta} \nabla_x\Theta\cdot\nabla_x\varphi\dxdt =
\int_\Omega\Bigl(\d_\vrho s(1,\oline{\vtheta})\,R_0\,+\,\d_\vtheta s(1,\oline{\vtheta})\,\Theta_0\Bigr)\,\varphi(0)\dx\, , \nonumber
\end{align}
for all $\varphi \in C_c^\infty([0,T[\,\times\Omega)$, with $T>0$ any arbitrary time.
Relation \eqref{eq:ent_bal_lim_1} means that the quantity $\Upsilon$, defined in \eqref{eq:anis-lim_3}, is a weak solution of that equation,
related to the initial datum $\Upsilon_0:=\d_\vrho s(1,\oline\vtheta)\,R_0\,+\,\d_\vtheta s(1,\oline\vtheta)\,\Theta_0$.
Equation \eqref{eq:anis-lim_3} is in fact an equation for $\Theta$ only, keep in mind Remark \ref{r:lim delta theta}.
\qed
\end{proof}

\subsubsection{The case of isotropic scaling} \label{ss:constr_1}

We focus now on the case of isotropic scaling, namely $m=1$. Recall that, in this instance, we also set $F=0$. In this case, the fast rotation and weak compressibility
effects are of the same order; in turn, this allows to reach the so-called \emph{quasi-geostrophic balance} in the limit (see equation \eqref{eq:streamq} below).
\begin{proposition}  \label{p:limit_iso}
Take $m=1$ and $F=0$ in system \eqref{ceq} to \eqref{eeq}.
Let $\left( \vre, \ue, \tem\right)_{\veps}$ be a family of weak solutions to \eqref{ceq} to \eqref{eeq}, associated with initial data
$\left(\vrho_{0,\veps},\vec u_{0,\veps},\vtheta_{0,\veps}\right)$ verifying the hypotheses fixed in Paragraph \ref{sss:data-weak}.
Let $(R, \vec{U},\Theta )$ be a limit point of the sequence $\left(R_{\veps} , \ue,\Theta_\veps\right)_{\veps}$, as identified in Subsection
\ref{sss:constr_prelim}.
Then,
\begin{align}
&\vec{U}\,=\,\,\Big(\vec{U}^h\,,\,0\Big)\,,\qquad\qquad \mbox{ with }\qquad \vec{U}^h\,=\,\vec{U}^h(t,x^h)\quad \mbox{ and }\quad \div_{\!h}\,\vec{U}^h\,=\,0 \nonumber \\[1ex] 
&\vec{U}^h\,=\,\nabla^\perp_hq
\;\mbox{ a.e. in }\;\,]0,T[\, \times \Omega\, ,
\quad\mbox{ with }  \label{eq:streamq} \\[1ex]
& q\,=\,q(t,x^h)\,:=\,\d_\varrho p(1,\oline{\vtheta})R+\d_\vtheta p(1,\oline{\vtheta})\Theta-G-1/2 \label{eq:for q} \\[1ex]
&\d_{t} \Upsilon +\divh\left( \Upsilon \vec{U}^{h}\right) -\frac{\kappa(\oline\vtheta)}{\oline\vtheta} \Delta \Theta =0\,,\qquad\quad\mbox{ with }\qquad 
\Upsilon_{|t=0}\,=\,\Upsilon_0\,, \nonumber 
\end{align}
where $ \Upsilon$ and $\Upsilon_0$ are the same quantities defined in Proposition \ref{p:limitpoint}.
\end{proposition}

\begin{proof}
Arguing as in the proof of Proposition \ref{p:limitpoint}, it is easy to pass to the limit in the continuity equation and in the entropy balance. In particular, we obtain again
equations \eqref{eq:div-free} and \eqref{eq:ent_bal_lim_1}.

The only changes concern the analysis of the momentum equation, written in its weak formulation \eqref{weak-mom}.
We start by testing it on $\veps\,\vec\phi$, for a smooth compactly supported $\vec\phi$. Similarly to what done above, the uniform bounds of Subsection \ref{ss:unif-est}
allow us to say that the only quantity which does not vanish in the limit is the sum of the terms involving the Coriolis force, the pressure and the gravitational force:
$$
\vec{e}_{3}\times \vrho_{\veps}\ue\,+\frac{\nabla_x \left( p(\vrho_\veps,\vtheta_\veps)-p(\widetilde{\vrho}_\veps,\vtheta_\veps)\right)}{\veps}\,-\,
\left(\vrho_\veps-\widetilde{\vrho}_\veps \right)\nabla_x G\,=\,\mc O(\veps)\,.
$$
From this relation, following the same computations performed in the proof of Proposition \ref{p:limitpoint}, in the limit $\veps\ra0^+$ we obtain that
$$ 
\vec{e}_{3}\times \vec{U}+\nabla_x\left(\d_\varrho p(1,\oline{\vtheta})\,\vrho^{(1)}\,+\,\d_\vtheta p(1,\oline{\vtheta})\,\Theta\right)\,=\,0 \qquad\qquad\mbox{ a.e. in }\; \R_+\times \Omega\,.
$$ 
After defining $q$ as in \eqref{eq:for q}, i.e.
$$
q\,:=\,\d_\varrho p(1,\oline{\vtheta})R+\d_\vtheta p(1,\oline{\vtheta})\Theta-G-1/2
$$
and keeping Remark \ref{r:F-G} in mind, this equality can be equivalently written as 
$$ 
\vec{e}_{3}\times \vec{U}+\nabla_xq\,=\,0 \qquad\qquad\mbox{ a.e. in }\; \R_+ \times \Omega\,.
$$ 
As done in the proof to Proposition \ref{p:limitpoint}, from this relation we immediately deduce that $q=q(t,x^h)$ and $\vec{U}^h=\vec{U}^h(t,x^h)$.
In addition, we get $\vec{U}^h=\nabla^\perp_hq$, whence we gather that $q$ can be viewed as a stream function for $\vec U^h$. 
Using \eqref{eq:div-free}, we infer that $\d_{3}U^{3}=0$, which in turn implies that $U^{3}\equiv0$, thanks to \eqref{bc1-2}.
The proposition is thus proved.
\qed
\end{proof}

\begin{remark} \label{r:q}
Notice that $q$ is defined up to an additive constant. We fix it to be $-1/2$, in order to compensate the vertical mean of $G$ and have a cleaner expression for $\lan q\ran$ (see 
Theorem \ref{th:m=1_F=0}). As a matter of fact, it is $\lan q\ran$ the natural quantity to look at, see also Subsection \ref{ss:limit_1} in this respect.
\end{remark}

\section{Convergence in presence of the centrifugal force}\label{s:proof}

In this section we complete the proof of Theorem \ref{th:m-geq-2}, in the case when $m\geq2$ and $F\neq0$. In the case $m>1$ and $F=0$, some arguments of the proof slightly change,
due to the absence of the (unbounded) centrifugal force: we refer to Section \ref{s:proof-1} below for more details.

\medbreak
The uniform bounds of Subsection \ref{ss:unif-est} allow us to pass to the limit in the mass and entropy equations, but they are
not enough for proving convergence in the weak formulation of the momentum equation: the main problem relies on identifying the weak limit of the convective term
$\vrho_\veps\,\vec u_\veps\otimes\vec u_\veps$.
For this, we need to control the strong oscillations in time of the solutions: this is the aim of Subsection \ref{ss:acoustic}. 
In Subsection \ref{ss:convergence}, by using a compensated compactness argument together with Aubin-Lions Theorem (see Theorem \ref{app:Aubin_thm}), we establish strong convergence of suitable quantities related to the velocity fields.
This property, which deeply relies on the structure of the wave system, allows us to pass to the limit in our equations (see Subsection \ref{ss:limit}).

\subsection{Analysis of the acoustic waves} \label{ss:acoustic}

The goal of the present subsection is to describe oscillations of solutions. First of all, we recast our equations into a wave system; there we also implement a localisation
procedure, due to the presence of the centrifugal force. Then, we establish uniform bounds for the quantities appearing in the wave system. Finally, we apply a regularisation in space for all the quantities, which is preparatory in view of the computations of Subsection \ref{ss:convergence}.

\subsubsection{Formulation of the acoustic equation} \label{sss:wave-eq}

Let us define 
$$
\vec{V}_\veps\,:=\,\vrho_\veps\vec{u}_\veps\,.
$$
We start by writing the continuity equation in the form
\begin{equation} \label{eq:wave_mass}
\veps^m\,\d_t\vrho^{(1)}_\veps\,+\,\div\vec{V}_\veps\,=\,0\,.
\end{equation}
Of course, this relation, as well as the other ones which will follow, has to be read in the weak form.

Using continuity equation and resorting to the time lifting \eqref{lift0} of the measure $\sigma_\veps$, straightforward computations
lead us to the following form of the entropy balance:
$$
\veps^m\d_t\!\left(\vrho_\veps\,\frac{s(\vrho_\veps,\vtheta_\veps)-s(\wtilde{\vrho}_\veps,\oline{\vtheta})}{\veps^m}-
\frac{1}{\veps^m}\Sigma_\veps\right)\,=\,\veps^m\,\div\!\!\left(\frac{\kappa(\vtheta_\veps)}{\vtheta_\veps}
\frac{\nabla_x\vtheta_\veps}{\veps^m}\right)+s(\wtilde{\vrho}_\veps,\oline{\vtheta})\div\!\!\left(\vrho_\veps\,\vec{u}_\veps\right)-
\div\!\!\left(\vrho_\veps s(\vrho_\veps,\vtheta_\veps)\vec{u}_\veps\right),
$$
where, with a little abuse of notation, we use the identification
$\int_{\Omega_\veps}\Sigma_\veps\,\varphi\,dx\,=\,\langle\Sigma_\veps,\varphi\rangle_{[\mc{M}^+,C^0]}$. Next, since $\wtilde{\vrho}_\veps$
is smooth (recall relation \eqref{eq:target-rho} above), the previous equation can be finally written as
\begin{align}
&\veps^m\,\d_t\left(\vrho_\veps\,\frac{s(\vrho_\veps,\vtheta_\veps)-s(\wtilde{\vrho}_\veps,\oline{\vtheta})}{\veps^m}\,-\,
\frac{1}{\veps^m}\Sigma_\veps\right)\,=  \label{eq:wave_entropy} \\
&\qquad\quad=\,
\veps^m\,\biggl(\div\!\left(\frac{\kappa(\vtheta_\veps)}{\vtheta_\veps}\,\frac{\nabla_x\vtheta_\veps}{\veps^m}\right)\,-\,
\vrho_\veps\,\vec{u}_\veps\,\cdot\,\frac{1}{\veps^m}\,\nabla_x s(\wtilde{\vrho}_\veps,\oline{\vtheta})\,-\,
\div\!\left(\vrho_\veps\,\frac{s(\vrho_\veps,\vtheta_\veps)-s(\wtilde{\vrho}_\veps,\oline{\vtheta})}{\veps^m}\,\vec{u}_\veps\right)\biggr)\,.
\nonumber
\end{align}

Now, we turn our attention to the momentum equation. By \eqref{prF} we find
\begin{align}
&\veps^m\,\d_t\vec{V}_\veps\,+\,\nabla_x\left(\frac{p(\vrho_\veps,\vtheta_\veps)-p(\wtilde{\vrho}_\veps,\oline{\vtheta})}{\veps^m}\right)\,+\,\veps^{m-1}\,\e_3\times \vec V_\veps\,=\,
\veps^{2(m-1)}\frac{\vrho_\veps-\wtilde{\vrho}_\veps}{\veps^m}\nabla_x F\,+ \label{eq:wave_momentum} \\
&\qquad\qquad\qquad\qquad\qquad
+\,\veps^m\left(\div\mbb{S}\!\left(\vtheta_\veps,\nabla_x\vec{u}_\veps\right)\,-\,\div\!\left(\vrho_\veps\vec{u}_\veps\otimes\vec{u}_\veps\right)\,+\,
\frac{\vrho_\veps-\wtilde{\vrho}_\veps}{\veps^m}\nabla_x G\right)\,. \nonumber
\end{align}

At this point, let us introduce two real numbers $\mc{A}$ and $\mc{B}$, such that the following relations are satisfied:
\begin{equation} \label{relnum}
\mc{A}\,+\,\mc{B}\,\d_\vrho s(1,\oline{\vtheta})\,=\,\d_\vrho p(1,\oline{\vtheta})\qquad\mbox{ and }\qquad
\mc{B}\,\d_\vtheta s(1,\oline{\vtheta})\,=\,\d_\vtheta p(1,\oline{\vtheta})\,.
\end{equation}
Due to Gibbs' law \eqref{gibbs} and the structural hypotheses of Paragraph \ref{sss:structural} (see also Chapter 8 of \cite{F-N} and \cite{F-Scho}), we notice that
$\mc A$ is given by formula \eqref{def:A}, and $\mc A>0$.

Taking a linear combination of \eqref{eq:wave_mass} and \eqref{eq:wave_entropy}, with coefficients respectively $\mc{A}$ and $\mc{B}$,
and keeping in mind equation \eqref{eq:wave_momentum}, we finally get the wave system
\begin{equation} \label{eq:wave_syst}
\left\{\begin{array}{l}
       \veps^m\,\d_tZ_\veps\,+\,\mc{A}\,\div\vec{V}_\veps\,=\,\veps^m\,\left(\div\vec{X}^1_\veps\,+\,X^2_\veps\right) \\[1ex]
       \veps^m\,\d_t\vec{V}_\veps\,+\,\nabla_x Z_\veps\,+\,\veps^{m-1}\,\e_3\times \vec V_\veps\,=\,\veps^m\,\left(\div\mbb{Y}^1_\veps\,+\,\vec{Y}^2_\veps\,+\,\nabla_x Y^3_\veps\right)\,,
       \qquad\big(\vec{V}_\veps\cdot\vec n\big)_{|\d\Omega_\veps}\,=\,0\,,
       \end{array}
\right.
\end{equation}
where we have defined the quantities
\begin{eqnarray*}
Z_\veps & := & \mc{A}\,\vrho^{(1)}_\veps\,+\,\mc{B}\,\left(\vrho_\veps\,
\frac{s(\vrho_\veps,\vtheta_\veps)-s(\wtilde{\vrho}_\veps,\oline{\vtheta})}{\veps^m}\,-\,\frac{1}{\veps^m}\Sigma_\veps\right) \\
\vec{X}^1_\veps & := & \mc{B}\left(\frac{\kappa(\vtheta_\veps)}{\vtheta_\veps}\,\frac{\nabla_x\vtheta_\veps}{\veps^m}\,-\,
\vrho_\veps\,\frac{s(\vrho_\veps,\vtheta_\veps)-s(\wtilde{\vrho}_\veps,\oline{\vtheta})}{\veps^m}\,\vec{u}_\veps\right) \\
X^2_\veps & := & -\,\mc{B}\,\vrho_\veps\,\vec{u}_\veps\,\cdot\,\frac{1}{\veps^m}\,\nabla_x s(\wtilde{\vrho}_\veps,\oline{\vtheta}) \\
\mbb{Y}^1_\veps & := & \mbb{S}\!\left(\vtheta_\veps,\nabla\vec{u}_\veps\right)\,-\,\vrho_\veps\vec{u}_\veps\otimes\vec{u}_\veps \\
\vec{Y}^2_\veps & := & \frac{\vrho_\veps-\wtilde{\vrho}_\veps}{\veps^m}\nabla_x G\,+\,
\veps^{m-2}\,\frac{\vrho_\veps-\wtilde{\vrho}_\veps}{\veps^m}\nabla_x F \\
Y^3_\veps & := &\frac{1}{\veps^{m}}\left( \mc{A}\,\frac{\vrho_\veps-\wtilde{\vrho}_\veps}{\veps^m}\,+\mc{B}\,\vrho_\veps\,
\frac{s(\vrho_\veps,\vtheta_\veps)-s(\wtilde{\vrho}_\veps,\oline{\vtheta})}{\veps^m}\,-\,\mc{B}\,\frac{1}{\veps^m}\Sigma_\veps\,-\,
\frac{p(\vrho_\veps,\vtheta_\veps)-p(\wtilde{\vrho}_\veps,\oline{\vtheta})}{\veps^m}\right)\,.
\end{eqnarray*}

We remark that system \eqref{eq:wave_syst} has to be read in the weak sense: for any $\varphi\in C_c^\infty\bigl([0,T[\,\times\oline\Omega_\veps\bigr)$, one has
$$
-\,\veps^m\,\int^T_0\int_{\Omega_\veps} Z_\veps\,\d_t\varphi\,-\,\mc{A}\,\int^T_0\int_{\Omega_\veps} \vec{V}_\veps\cdot\nabla_x\varphi\,=\,
\veps^{m}\int_{\Omega_\veps} Z_{0,\veps}\,\varphi(0)\,+\,\veps^m\,\int^T_0\int_{\Omega_\veps}\left(-\,\vec{X}^1_\veps\cdot\nabla_x\varphi\,+\,X^2_\veps\,\varphi\right)\,,
$$
and also, for any $\vec{\psi}\in C_c^\infty\bigl([0,T[\,\times\oline\Omega_\veps;\R^3\bigr)$ such that $\big(\vec\psi \cdot \n_\veps\big)_{|\partial {\Omega_\veps}} = 0$, one has
\begin{align*}
&-\,\veps^m\,\int^T_0\int_{\Omega_\veps}\vec{V}_\veps\cdot\d_t\vec{\psi}\,-\,\int^T_0\int_{\Omega_\veps} Z_\veps\,\div\vec{\psi}\,+\,\veps^{m-1}\int^T_0\int_{\Omega_\veps} \e_3\times\vec V_\veps\cdot\vec\psi \\
&\qquad\qquad\qquad\qquad
=\,\veps^{m}\int_{\Omega_\veps}\vec{V}_{0,\veps}\cdot\vec{\psi}(0)\,+\,\veps^m\,\int^T_0\int_{\Omega_\veps}\left(-\,\mbb{Y}^1_\veps:\nabla_x\vec{\psi}\,+\,\vec{Y}^2_\veps\cdot\vec{\psi}\,-\,
Y^3_\veps\,\div\vec{\psi}\right)\,,
\end{align*}
where we have set
\begin{equation} \label{def:wave-data}
Z_{0,\veps}\,=\,\mc{A}\,\vrho^{(1)}_{0,\veps}\,+\,\mc{B}\,\left(\vrho_{0,\veps}\,
\frac{s(\vrho_{0,\veps},\vtheta_{0,\veps})-s(\wtilde{\vrho}_\veps,\oline{\vtheta})}{\veps^m}\right) \qquad\mbox{ and }\qquad
\vec{V}_{0,\veps}\,=\,\vrho_{0,\veps}\,\vec{u}_{0,\veps}\,.
\end{equation}

At this point, analogously to \cite{F-G-GV-N}, for any fixed $l>0$, let us introduce a smooth cut-off
\begin{align} 
&\chi_l\in C^\infty_c(\R^2)\quad \mbox{ radially decreasing}\,,\qquad \mbox{ with }\quad 0\leq\chi_l\leq1\,, \label{eq:cut-off} \\
&\mbox{such that }\quad \chi_l\equiv1 \ \mbox{ on }\ \B_l\,, \quad \chi_l\equiv0 \ \mbox{ out of }\ \B_{2l}\,,\quad \left|\nabla_{h}\chi_l(x^h)\right|\,\leq\,C(l)\; \; \text{for any}\;x^h\in\R^2\,. \nonumber
\end{align}
Then we define
\begin{equation} \label{def:L-W}
\Lambda_{\veps,l}\,:=\,\chi_l\,Z_\veps\,=\,\chi_l\,\mc{A}\,\vrho^{(1)}_\veps\,+\,\chi_l\,\mc{B}\,\left(\vrho_\veps\,
\frac{s(\vrho_\veps,\vtheta_\veps)-s(\wtilde{\vrho}_\veps,\oline{\vtheta})}{\veps^m}\,-\,\frac{1}{\veps^m}\Sigma_\veps\right)
\quad\mbox{ and }\quad \vec W_{\veps,l}\,:=\,\chi_l\,\vec V_\veps\,.
\end{equation}
For notational convenience, in what follows we keep using the notation
$\Lambda_{\veps}$ and $\vec W_\veps$ instead of $\Lambda_{\veps,l}$ and $\vec W_{\veps,l}\, $, tacitly meaning the dependence on $l$.  
So system \eqref{eq:wave_syst} becomes
\begin{equation} \label{eq:wave_syst2}
\left\{\begin{array}{l}
       \veps^m\,\d_t\Lambda_\veps\,+\,\mc{A}\,\div\vec{W}_\veps\,=\,
       \veps^m f_\veps \\[1ex]
       \veps^m\,\d_t\vec{W}_\veps\,+\,\nabla_x \Lambda_\veps\,+\,\veps^{m-1}\,\e_3\times \vec W_\veps\,=\, 
       \veps^m \vec{G}_\veps\,,
       \qquad\big(\vec{W}_\veps\cdot\vec n\big)_{|\d\Omega_\veps}\,=\,0\,,
       \end{array}
\right.
\end{equation}
where we have defined $f_\veps\,:=\,\div\vec{F}^1_\veps\,+\,F^2_\veps\;$ and $\;\vec G_\veps\,:=\,\div\mbb{G}^1_\veps\,+\,\vec{G}^2_\veps\,+\,\nabla_x G^3_\veps$, with
\begin{align*}
 & \vec{F}^1_\veps\,=\,\chi_l\,\vec{X}^1_\veps\qquad\qquad\mbox{ and }\qquad\qquad
F^2_\veps\,=\,\chi_l X^2_\veps\,-\,\vec{X}^1_\veps\cdot\nabla_{x}\chi_l\,+\,\mc{A}\,\vec V_\veps\cdot \nabla_{x}\chi_l\, \\
& \mbb{G}^1_\veps\,=\,\chi_l\,\mbb{Y}^1_\veps\;,\qquad
\vec G^2_\veps\,=\,\chi_l\,\vec Y^2_\veps\,+\,\left(\frac{Z_\veps}{\veps^{m}}-Y^3_\veps \right)\,\nabla_{x}\chi_l\,-\,^t\mbb{Y}^1_\veps\cdot \nabla_{x}\chi_l\qquad\mbox{ and }\qquad
G^3_\veps\,=\,\chi_l\,Y^3_\veps\,.
\end{align*}

\subsubsection{Uniform bounds} \label{sss:w-bounds}

Here we use estimates of Subsection \ref{ss:unif-est} in order to show uniform bounds for the solutions and the data in the wave equation \eqref{eq:wave_syst2}.
We start by dealing with the ``unknowns'' $\Lambda_\veps$ and $\vec W_\veps$.

\begin{lemma} \label{l:S-W_bounds}
Let $\bigl(\Lambda_\veps\bigr)_\veps$  and $\bigl(\vec W_\veps\bigr)_\veps$ be defined as above. Then, for any $T>0$ and all $\ep \in \, ]0,1]$, one has
$$
	\| \Lambda_\veps\|_{L^\infty_T(L^2+L^{5/3}+L^1+\mc{M}^+)} \leq c(l)\, ,\quad\quad
	 \| \vec W_\veps\|_{L^2_T(L^2+L^{30/23})} \leq c(l) \, .
$$
\end{lemma}

\begin{proof}
We start by writing $\vec W_\veps\,=\,\vec W^1_\veps\,+\,\vec W_\veps^2$, where
$$
\vec W^1_\veps\,:=\,\chi_l\,[\vrho_\veps]_{\ess}\,\vec u_\veps\qquad\qquad\mbox{ and }\qquad\qquad
\vec W^2_\veps\,:=\,\chi_l\,[\vrho_\veps]_{\res}\,\vec u_\veps\,.
$$
Since the density and temperature are uniformly bounded on the essential set, by \eqref{est:u-H^1} we infer that $\vec W_\veps^1$ is uniformly bounded in
$L_T^2(L^2)$. On the other hand, by \eqref{est:rho_res} and \eqref{est:u-H^1} again, we easily deduce that $\vec W_\veps^2$ is uniformly bounded in
$L_T^2(L^p)$, where $3/5+1/6\,=\,1/p$. The claim about $\vec W_\veps$ is hence proved.

Let us now consider $\Lambda_\veps$, defined in \eqref{def:L-W}, i.e.
$$ 
\Lambda_\veps=\Lambda_{\veps,l}\,:=\,\chi_l\,Z_\veps\,=\,\chi_l\,\mc{A}\,\vrho^{(1)}_\veps\,+\,\chi_l\,\mc{B}\,\left(\vrho_\veps\,
\frac{s(\vrho_\veps,\vtheta_\veps)-s(\wtilde{\vrho}_\veps,\oline{\vtheta})}{\veps^m}\,-\,\frac{1}{\veps^m}\Sigma_\veps\right)\, .
$$
First of all, owing to the bounds $\left\|\Sigma_\veps\right\|_{L^\infty_T(\mc{M}^+)}\,\leq\,C\,\|\sigma_\veps\|_{\mc{M}^+_{t,x}}$ and \eqref{est:sigma}, we have that
$$
\left\|\frac{1}{\veps^{2m}}\,\chi_l\,\Sigma_\veps\right\|_{L^\infty_T(\mc{M}^+)} \leq c(l)\,,
$$
uniformly in $\veps>0$. Next, we can write the following decomposition:
$$
\vrho_\veps\,\chi_l\,\frac{s(\vrho_\veps,\vtheta_\veps)-s(\wtilde{\vrho}_\veps,\oline{\vtheta})}{\veps^m}\,=\,
\frac{1}{\veps^m}\,\chi_l\,\left(\vrho_\veps\,s(\vrho_\veps,\vtheta_\veps)\,-\,\wtilde{\vrho}_\veps\,s(\wtilde{\vrho}_\veps,\oline{\vtheta})\right)\,-\,
\chi_l\,\vrho_\veps^{(1)}\,s(\wtilde{\vrho}_\veps,\oline{\vtheta})\,,
$$
where the latter term in the right-hand side is bounded in $L^\infty_T(L^2+L^{5/3})$ in view of \eqref{uni_varrho1} and Proposition~\ref{p:target-rho_bound}. Concerning the former term,
we can write it as
\begin{equation}\label{eq:ub_1}
\frac{1}{\veps^m}\chi_l \left(\varrho_\ep s(\vrho_\veps,\vtheta_\veps)-\varrho_\ep s( \wtilde{\vrho}_\veps,\oline\vtheta)\right)=
\frac{1}{\veps^m}\chi_l \bigl[\varrho_\ep s(\vrho_\veps,\vtheta_\veps)-\varrho_\ep s(\wtilde{\vrho}_\veps,\oline\vtheta)\bigr]_{\ess}+
\frac{1}{\veps^m}\chi_l \bigl[\varrho_\ep s(\vrho_\veps,\vtheta_\veps)\bigr]_{\res}\,,
\end{equation}
since the support of $\chi_l\varrho_\ep s(\wtilde{\vrho}_\veps,\oline\vtheta)$ is contained in the essential set by Proposition \ref{p:target-rho_bound}, for small enough $\veps$
(depending on the fixed $l>0$).
By \eqref{est:e-s_res}, the last term on the right-hand side of \eqref{eq:ub_1} is uniformly bounded in $L^\infty_T(L^1)$; as  for the first term, a Taylor expansion at the first order, together with
inequality \eqref{est:rho_ess} and the structural restrictions on $s$, immediately yields its uniform boundedness in $L^\infty_T(L^2)$.

The lemma is hence completely proved.
\qed
\end{proof}

\medbreak
In the next lemma, we establish bounds for the source terms in the system of acoustic waves \eqref{eq:wave_syst2}.
\begin{lemma} \label{l:source_bounds}
For any $T>0$ fixed, let us define the following spaces:
\begin{itemize}
 \item $
\mc X_1\,:=\,L^2\Bigl([0,T];\big(L^2+L^{1}+L^{3/2}+L^{30/23}+L^{30/29}\big)(\Omega)\Bigr)$;
\item  $\mc X_2\,:=\,L^2\Bigl([0,T];\big(L^2+L^1+L^{4/3}\big)(\Omega)\Bigr)$;
\item  $\mc X_3\,:=\,\mc X_2\,+\,L^\infty\Bigl([0,T];\big(L^2+L^{5/3}+L^1\big)(\Omega)\Bigr)$;
\item $\mc X_4\,:=\,L^\infty\Bigl([0,T];\big(L^2+L^{5/3}+L^1+\mc{M}^+\big)(\Omega)\Bigr)$.
\end{itemize}
Then, for any $l>0$ fixed, one has the following bounds, uniformly in $\veps\in\,]0,1]$:
$$
\left\|\vec F^1_\veps\right\|_{\mc X_1}\,+\,\left\|F^2_\veps\right\|_{\mc X_1}\,+\,\left\|\mbb{G}^1_\veps\right\|_{\mc X_2}\,+\,\left\|\vec G^2_\veps\right\|_{\mc X_3}\,+\,
\left\|G^3_\veps\right\|_{\mc X_4}\,\leq\,C(l)\,.
$$
In particular, the sequences $\bigl( f_\veps\bigr)_\veps$ and
$\bigl(\vec{G}_\veps\bigr)_\veps$, defined in system \eqref{eq:wave_syst2}, are uniformly bounded in the space $L^{2}\big([0,T];W^{-1,1}(\Omega)\big)$, 
thus 
in $L^{2}\big([0,T];H^{-s}(\Omega)\big)$, for all $s>5/2$.
\end{lemma}

\begin{proof}
We start by dealing with $\vec F^1_\veps$. By relations \eqref{est:D-theta} and \eqref{est:Dtheta_res}, it is easy to see that
$$
\left\|\frac{1}{\veps^m}\,\chi_l\,\frac{\kappa(\vtheta_\veps)}{\vtheta_\veps}\,\nabla_{x}\vtheta_\veps\right\|_{L^2_T(L^2+L^1)}\,\leq\,c(l)\,.
$$
On the other hand, the analysis of the term
$$
\vrho_\veps\,\chi_l\,\frac{s(\vrho_\veps,\vtheta_\veps)-s(\wtilde{\vrho}_\veps,\oline{\vtheta})}{\veps^m}\,\vec u_\veps
$$
is based on an analogous decomposition as used in the proof of Lemma \ref{l:S-W_bounds} and on uniform bounds of Paragraph \ref{sss:uniform}: these facts allow us to bound
it in $L^2_T(L^{3/2}+L^{30/23}+L^{30/29})$.  

The bounds for $F^2_\veps$ easily follow from the previous ones and Lemma \ref{l:S-W_bounds} (indeed, the analysis for $\vec{W}_\veps$
applies also to the terms of the form $\vrho_\veps\vec{u}_\veps$ which appear in the definition of $F^2_\veps$), provided we show that
$$
\frac{1}{\veps^m}\,\left|\chi_l\,\nabla_{x}\wtilde{\vrho}_\veps\right|\,\leq\,C(l)\,.
$$
The previous bound immediately follows from the equation
$$
\nabla_{x}\wtilde{\vrho}_\veps\,=\,\frac{\wtilde{\vrho}_\veps}{\d_\vrho p(\wtilde{\vrho}_\veps,\oline\vtheta)}\,\left(\veps^{2(m-1)}\,\nabla_{x} F\,+\,\veps^m\,\nabla_{x} G\right)\,,
$$
which derives from \eqref{prF}.
Hence, by Proposition \ref{p:target-rho_bound} and the definitions given in \eqref{assFG}, we get
$$
\frac{1}{\veps^m}\,\left|\chi_l\,\nabla_{x}\wtilde{\vrho}_\veps\right|\,\leq\,C(l)\,\left(\veps^{2(m-1)-m}+1\right)\,\leq\,C(l)\,.
$$

The bound on  $\mbb{G}^1_\veps$ is an immediate consequence of \eqref{est:Du} and \eqref{est:momentum}.

Let us focus now on the term $\vec G^2_\veps$. The control of the term $\,^t\mbb{Y}^1_\veps\cdot\nabla_{x}\chi_l$ is the same as above. The control of $\chi_l\vec Y^2_\veps$, instead,
gives rise to a bound in $L^\infty_T(L^2+L^{5/3})$: this is easily seen once we write
$$
\chi_l\,\vec Y^2_\veps\,=\,\chi_l\,\vrho_\veps^{(1)}\,\nabla_{x} G\,+\,\veps^{m-2}\,\chi_l\,\vrho_\veps^{(1)}\,\nabla_{x}F
$$
and we use \eqref{uni_varrho1} and \eqref{assFG}. Finally, we have the equality
\begin{equation*}
\begin{split}
\nabla_x \chi_l\,\left( \frac{Z_\veps}{\veps^{m}}-Y^3_\veps \right)&=\nabla_x \chi_{l}\,\left(\frac{p(\vrho_\veps,\vtheta_\veps)-p(\wtilde{\vrho}_\veps,\oline\vtheta)}{\veps^m}\right)\\
&=\nabla_x \chi_{l}\left[\frac{p(\vrho_\veps,\vtheta_\veps)-p(\wtilde{\vrho}_\veps,\oline\vtheta)}{\veps^m}\right]_\ess+\nabla_x \chi_{l}\left[\frac{p(\vrho_\veps,\vtheta_\veps)}{\veps^m}\right]_\res.
\end{split}
\end{equation*}
The second term in the last line is uniformly bounded in $L^\infty_T(L^1)$, in view of \eqref{est:rho_res}.
For the first term, instead,
we can proceed as in \eqref{eq:ub_1}.

At this point, we switch our attention to the term $G^3_\veps$, whose analysis is more involved. 
%
By definition, we have
\begin{equation*}
\begin{split}
\chi_l\,Y^3_\veps  &:= \frac{1}{\veps^{m}}\,\chi_l\,\left( \mc{A}\,\frac{\vrho_\veps-\wtilde{\vrho}_\veps}{\veps^m}\,+\mc{B}\,\vrho_\veps\,
\frac{s(\vrho_\veps,\vtheta_\veps)-s(\wtilde{\vrho}_\veps,\oline{\vtheta})}{\veps^m}\,-\,\mc{B}\,\frac{1}{\veps^m}\Sigma_\veps\,-\,
\frac{p(\vrho_\veps,\vtheta_\veps)-p(\wtilde{\vrho}_\veps,\oline{\vtheta})}{\veps^m}\right) \\
&=\frac{1}{\veps^{m}}\,\chi_l\,\left( \mc{A}\,\frac{\vrho_\veps-\wtilde{\vrho}_\veps}{\veps^m}\,+\mc{B}\, \frac{s(\vrho_\veps,\vtheta_\veps)-s(\wtilde{\vrho}_\veps,\oline{\vtheta})}{\veps^m}\,-
\frac{p(\vrho_\veps,\vtheta_\veps)-p(\wtilde{\vrho}_\veps,\oline{\vtheta})}{\veps^m}\right)\\
&\qquad-\,\mc{B}\,\frac{1}{\veps^{2m}}\,\chi_l\,\Sigma_\veps\,+\,\mc{B}\,\chi_l\,\left(\frac{\vrho_\veps -1}{\veps^{m}}\right)\,
\frac{s(\vrho_\veps,\vtheta_\veps)-s(\wtilde{\vrho}_\veps,\oline{\vtheta})}{\veps^m}\,,
\end{split}
\end{equation*}
with $\mc A$ and $\mc B$ (see definition \eqref{relnum} above) such that
$$  
\mc{A}\,+\,\mc{B}\,\d_\vrho s(1,\oline{\vtheta})\,=\,\d_\vrho p(1,\oline{\vtheta})\qquad\mbox{ and }\qquad
\mc{B}\,\d_\vtheta s(1,\oline{\vtheta})\,=\,\d_\vtheta p(1,\oline{\vtheta})\,.
$$
Next, we use a Taylor expansion to write
\begin{equation*}
\begin{split}
s(\vrho_\veps,\vtheta_\veps)-s(\widetilde{\vrho}_\veps,\oline{\vtheta})&=s(\vrho_\veps,\vtheta_\veps)-s(1,\oline{\vtheta})+s(1,\oline \vtheta)-s(\widetilde{\vrho}_\veps,\oline{\vtheta})\\
&=\d_\vrho \, s(1,\oline \vtheta )\, (\vrho_\veps -1)+\d_\vtheta \, s(1,\oline \vtheta )\, (\vtheta_\veps -\oline \vtheta)+\frac{1}{2}\,{\rm Hess}(s)[\xi_1 ,\eta_1]\begin{pmatrix}
\vrho_\veps -1 \\ 
\vtheta_\veps -\oline\vtheta
\end{pmatrix}\cdot\begin{pmatrix}
\vrho_\veps -1 \\ 
\vtheta_\veps -\oline\vtheta
\end{pmatrix} \\
&+\d_\vrho \, s(1,\oline \vtheta )\, (1-\widetilde{\vrho}_\veps )+\frac{1}{2}\, \d_{\vrho \vrho}^2\,s(\xi_{2},\oline\vtheta)\, (\widetilde{\vrho}_\veps -1)^{2}\\
&=\d_\vrho \, s(1,\oline \vtheta )\, (\vrho_\veps - \widetilde{\vrho}_\veps)+\d_\vtheta \, s(1,\oline \vtheta )\, (\vtheta_\veps -\oline \vtheta)\\
&+\frac{1}{2}\left({\rm Hess}(s)[\xi_1 ,\eta_1]\begin{pmatrix}
\vrho_\veps -1 \\ 
\vtheta_\veps -\oline\vtheta
\end{pmatrix}\cdot\begin{pmatrix}
\vrho_\veps -1 \\ 
\vtheta_\veps -\oline\vtheta
\end{pmatrix}+\d_{\vrho \vrho}^2\,s(\xi_{2},\oline \vtheta)\, (\widetilde{\vrho}_\veps -1)^{2}\right)\,,
\end{split}
\end{equation*}
where $\xi_1,\xi_2,\eta_1$ are suitable points between $1$ and $\vrho_\veps$, $1$ and $\widetilde{\vrho}_\veps$, $\oline \vtheta$ and $ \vtheta_\veps $ respectively, and
we have denoted by ${\rm Hess}(s)[\xi,\eta]$ the Hessian matrix of the function $s$ with respect to its variables $\big(\vrho,\vtheta\big)$, computed at the point $(\xi,\eta)$.
Analogously, for the pressure term we have
\begin{equation*}
\begin{split}
p(\vrho_\veps,\vtheta_\veps)-p(\widetilde{\vrho}_\veps,\oline{\vtheta}) 
&=\d_\vrho \, p(1,\oline \vtheta )\, (\vrho_\veps - \widetilde{\vrho}_\veps)+\d_\vtheta \, p(1,\oline \vtheta )\, (\vtheta_\veps -\oline \vtheta)\\
&+\frac{1}{2}\left({\rm Hess}(p)[\xi_3 ,\eta_2]\begin{pmatrix}
\vrho_\veps -1 \\ 
\vtheta_\veps -\oline\vtheta
\end{pmatrix}\cdot\begin{pmatrix}
\vrho_\veps -1 \\ 
\vtheta_\veps -\oline\vtheta
\end{pmatrix}+\d_{\vrho \vrho}^2\,p(\xi_{4},\oline \vtheta)\, (\widetilde{\vrho}_\veps -1)^{2}\right)\,,
\end{split}
\end{equation*}
where $\xi_3,\xi_4,\eta_2$ are still between $1$ and $\vrho_\veps$, $1$ and $\widetilde{\vrho}_\veps$, $\oline \vtheta$ and $ \vtheta_\veps $ respectively.
Using now \eqref{relnum}, we find that the first order terms cancel out, and we are left with
\begin{equation*}
\begin{split}
\chi_l\,Y^3_\veps  &=\frac{\mc{B}}{2\veps^{2m}}\,\chi_l\,\left( \,{\rm Hess}(s)[\xi_1 ,\eta_1]\begin{pmatrix}
\vrho_\veps -1 \\ 
\vtheta_\veps -\oline\vtheta
\end{pmatrix}\cdot\begin{pmatrix}
\vrho_\veps -1 \\ 
\vtheta_\veps -\oline\vtheta
\end{pmatrix}+\d_{\vrho \vrho}^2\,s(\xi_{2},\oline \vtheta)\, (\widetilde{\vrho}_\veps -1)^{2}\right)\\
&-\,\frac{1}{2\veps^{2m}}\,\chi_l\,\left({\rm Hess}(p)[\xi_3 ,\eta_2]\begin{pmatrix}
\vrho_\veps -1 \\ 
\vtheta_\veps -\oline\vtheta
\end{pmatrix}\cdot\begin{pmatrix}
\vrho_\veps -1 \\ 
\vtheta_\veps -\oline\vtheta
\end{pmatrix}+\d_{\vrho \vrho}^2\,p(\xi_{4},\oline \vtheta)\, (\widetilde{\vrho}_\veps -1)^{2}\right)\\
&-\,\frac{\mc B}{\veps^{2m}}\,\chi_l\,\Sigma_\veps\,+\,\mc{B}\,\chi_l\,\left(\frac{\vrho_\veps -1}{\veps^{m}}\right)\,
\frac{s(\vrho_\veps,\vtheta_\veps)-s(\wtilde{\vrho}_\veps,\oline{\vtheta})}{\veps^m}\, .
\end{split}
\end{equation*} 
Thanks to the uniform bounds established in Paragraph \ref{sss:uniform} and the decomposition into essential and residual parts, the claimed control in the space $\mc X_4$ follows.
\qed
\end{proof}

\subsubsection{Regularization and description of the oscillations}\label{sss:w-reg}

Following \cite{F-N_AMPA} and \cite{F-N_CPDE} (see also \cite{Ebin}), it is convenient to reformulate our problem \eqref{ceq} to \eqref{eeq}, supplemented with complete slip boundary
conditions \eqref{bc1-2} and \eqref{bc3}, in a completely equivalent way, in the domain
$$
\wtilde{\Omega}_\veps\,:=\,{B}_{L_\veps} (0) \times \mbb{T}^1\,,\qquad\qquad\mbox{ with }\qquad\mbb{T}^1\,:=\,[-1,1]/\sim\,,
$$
where $\sim$ denotes the equivalence relation which identifies $-1$ and $1$.
For this, it is enough to extend $\varrho_\veps$, $\vtheta_\veps$, and $\vec u_\veps^h$ as even functions with respect to $x^{3}$, $u_\veps^3$ and $G$ as odd functions.

Correspondingly, we consider also the wave system \eqref{eq:wave_syst2} to be satisfied in the new domain $\wtilde\Omega_\veps$. It goes without saying that the uniform bounds
established above hold true also when replacing $\Omega$ with $\wtilde\Omega$, where we have set
$$\wtilde\Omega\,:=\,\R^2 \times \mbb{T}^1\,.$$
Notice that the wave speed in \eqref{eq:wave_syst2} is proportional to $\veps^{-m}$, while, in view of assumption \eqref{dom}, the domains $\wtilde\Omega_\veps$ are expanding at speed proportional
to $\veps^{-m-\de}$, for some $\delta>0$.
Therefore, no interactions of the acoustic-Poincar\'e waves with the boundary of $\wtilde\Omega_\veps$ take place (see also Remark \ref{r:speed-waves} in this respect),
for any finite time $T>0$ and sufficiently small $\ep>0$.
Thanks to this fact and the spatial localisation given by the cut-off function $\chi_l$, we can assume that \eqref{eq:wave_syst2} is satisfied
(still in a weak sense) on the whole $\wtilde\Omega$.

Now, for any $M\in\N$ let us consider the low-frequency cut-off operator ${S}_{M}$ of a Littlewood-Paley decomposition, as introduced in \eqref{eq:S_j}. We define 
\begin{equation*}
\Lambda_{\varepsilon ,M}={S}_{M}\Lambda_{\veps}\qquad\qquad \text{ and }\qquad\qquad \vec{W}_{\veps ,M}={S}_{M}\vec{W}_{\veps}\, .
\end{equation*} 
The following result holds true.
Recall that we are omitting from the notation the dependence of all quantities on $l>0$, due to multiplication by the cut-off function $\chi_l$ fixed above.
\begin{proposition} \label{p:prop approx}
For any $T>0$, we have the following convergence properties, in the limit $M\rightarrow +\infty $:
\begin{equation}\label{eq:approx var}
\begin{split}
&\sup_{0<\veps\leq1}\, \left\|\Lambda_{\varepsilon }-\Lambda_{\varepsilon ,M}\right\|_{L^{\infty}([0,T];H^{s})}\longrightarrow 0\quad  \forall s<-3/2-\delta \\
&\sup_{0<\veps\leq1}\, \left\|\vec{W}_{\varepsilon }-\vec{W}_{\varepsilon ,M}\right\|_{L^{\infty}([0,T];H^{s})}\longrightarrow 0\quad \forall s<-4/5-\delta\,,
\end{split}
\end{equation}
for any $\delta >0$.
Moreover, for any $M>0$, the couple $(\Lambda_{\veps ,M},\vec W_{\veps ,M})$ satisfies the approximate wave equations
\begin{equation}\label{eq:approx wave}
\left\{\begin{array}{l}
       \veps^m\,\d_t\Lambda_{\veps ,M}\,+\,\mc{A}\,\div\vec{W}_{\veps ,M}\,=\,\veps^m\,f_{\veps ,M} \\[1ex]
       \veps^m\,\d_t\vec{W}_{\veps ,M}\,+\veps^{m-1}\,e_{3}\times \vec{W}_{\veps ,M}+\,\nabla_x \Lambda_{\veps,M}\,=\,\veps^m\,\vec G_{\veps ,M}\, ,
       \end{array}
\right.
\end{equation}
where $(f_{\veps ,M})_{\veps}$ and $(\vec G_{\veps ,M})_{\veps}$ are families of smooth (in the space variables) functions satisfying, for any $s\geq0$, the uniform bounds
\begin{equation}\label{eq:approx force}
\sup_{0<\veps\leq1}\, \left\|f_{\veps ,M}\right\|_{L^{2}([0,T];H^{s})}\,+\,\sup_{0<\veps\leq1}\,\left\|\vec G_{\veps ,M}\right\|_{L^{2}([0,T];H^{s})}\,\leq\, C(l,s,M)\,,
\end{equation}
where the constant $C(l,s,M)$ depends on the fixed values of $l>0$, $s\geq 0$ and $M>0$, but not on $\veps>0$.
\end{proposition}

\begin{proof}
Thanks to Lemma \ref{lemma_sobolev_H^s}, properties \eqref{eq:approx var} are straightforward consequences of the uniform bounds establish in Subsection \ref{sss:w-bounds}. 

Next, applying the operator ${S}_{M}$ to \eqref{eq:wave_syst2} immediately gives us system \eqref{eq:approx wave}, where we have set 
\begin{equation*}
f_{\veps ,M}:={S}_{M}\left(\div\vec{F}^1_\veps\,+\,F^2_\veps\right)\qquad \text{ and }\qquad \vec G_{\veps ,M}:={S}_{M}\left(\div\mbb{G}^1_\veps\,+\,\vec{G}^2_\veps\,+\,\nabla_x G^3_\veps\right)\,.
\end{equation*}
Thanks to Lemma \ref{l:source_bounds} and \eqref{eq:LP-Sob} (and also Lemma \ref{lemma_sobolev_H^s}), it is easy to verify inequality \eqref{eq:approx force}.
\qed
\end{proof}

\medbreak
We also have an important decomposition for the approximated velocity fields and their $\curl$.
\begin{proposition} \label{p:prop dec}
For any $M>0$ and any $\veps\in\,]0,1]$, the following decompositions hold true:
\begin{equation*}
\vec{W}_{\veps ,M}\,=\,
\veps^{m}\vec{t}_{\veps ,M}^{1}+\vec{t}_{\veps ,M}^{2}\qquad\mbox{ and }\qquad
\curl \vec{W}_{\veps ,M}=\veps^{m}\vec{T}_{\veps ,M}^{1}+\vec{T}_{\veps ,M}^{2}\,,
\end{equation*}
where, for any $T>0$ and $s\geq 0$, one has 
\begin{align*}
&\left\|\vec{t}_{\veps ,M}^{1}\right\|_{L^{2}([0,T];H^{s})}+\left\|\vec{T}_{\veps ,M}^{1}\right\|_{L^{2}([0,T];H^{s})}\leq C(l,s,M) \\
&\left\|\vec{t}_{\veps ,M}^{2}\right\|_{L^{2}([0,T];H^{1})}+\left\|\vec{T}_{\veps ,M}^{2}\right\|_{L^{2}\left([0,T];L^2\right)}\leq C(l)\,,
\end{align*}
for suitable positive constants $C(l,s,M)$ and $C(l)$, which are uniform with respect to $\veps\in\,]0,1]$.
\end{proposition}

\begin{proof}
We start by defining
\begin{equation} \label{eq:t-T}
\vec{t}_{\veps,M}^{1}\,:=\,{S}_{M}\left(\chi_{l}\left(\frac{\vrho_\veps -1}{\veps^{m}}\right)\vec{u}_{\veps}\right) \qquad\mbox{ and }\qquad
\vec{t}_{\veps,M}^{2}\,:=\,{S}_{M}\left(\chi_{l}\vec{u}_{\veps}\right)\,.
\end{equation}
Then, it is apparent that $\vec{W}_{\veps ,M}\,=\,\veps^m\vec t_{\veps,M}^{1}\,+\,\vec t_{\veps,M}^{2}$.
The decomposition of $\curl \vec W_{\veps,M}$ is also easy to get, if we set $\vec T_{\veps,M}^j\,:=\,\curl \vec t_{\veps,M}^j$, for $j=1,2$.
We have to prove uniform bounds for all those terms.
But this is an easy verification, thanks to the $L^\infty_T(L^{5/3}_{\rm loc})$ bound on $R_\veps$ and the $L^2_T(H^{1}_{\rm loc})$ bound on $\vec{u}_{\veps}$, for any fixed time $T>0$ 
(recall the estimates obtained in Subsection \ref{ss:unif-est} above).

On the one hand, for the estimate 
\begin{equation*}
\left\|\vec{t}_{\veps ,M}^{1}\right\|_{L_{T}^{2}(H^{s})}+\left\|\vec{T}_{\veps ,M}^{1}\right\|_{L_{T}^{2}(H^{s})}\leq C(l,s,M)\, ,
\end{equation*}
it is sufficient to employ relation \eqref{eq:LP-Sob} and Lemma \ref{lemma_sobolev_H^s}.

On the other hand, we have 
\begin{equation*}
\begin{split}
\left\| S_{M}\left( \chi_{l} \ue \right) \right\|_{H^{1}}^{2}&\leq C \sum_{j=- 1}^{M-1}2^{2j}\left\| \Delta_{j} \left(\chi_l \ue\right)\right\|_{L^{2}}^{2}\\
&\leq C\|\chi_l \ue\|_{L^2}^2+C\sum_{j=0}^{M-1}\left\| \Delta_{j}\nabla_{x} \left(\chi_l \ue \right)\right\|_{L^{2}}^{2}\leq C(l)\, ,
\end{split}
\end{equation*}
and the estimate for $\vec{T}_{\veps ,M}^{2}$ follows from analogous computations. This completes the proof.
\qed
\end{proof}


\subsection{Convergence of the non-linear convective term} \label{ss:convergence}
In this subsection we show convergence of the convective term, by using a compensated compactness argument.
Namely, we manipulate this term, by performing algebraic computations on the wave system formulated above. 
As a consequence, we derive two key pieces of information: on the one hand, we see that some non-linear terms are small remainders (in the sense specified by relations \eqref{eq:test-f}
and \eqref{eq:remainder} below); on the other hand, we derive a compactness property for a new quantity, called $\g_{\veps,M}$.

The first step is to reduce the study to the case of smooth vector fields $\vec{W}_{\veps ,M}$.

\begin{lemma} \label{lem:convterm}
Let $T>0$. For any $\vec{\psi}\in C_c^\infty\bigl([0,T[\,\times\wtilde\Omega;\R^3\bigr)$, we have 
\begin{equation*}
\lim_{M\rightarrow +\infty} \limsup_{\veps \rightarrow 0^+}\left|\int_{0}^{T}\int_{\wtilde\Omega} \vrho_\veps\,\vec{u}_\veps\otimes \vec{u}_\veps: \nabla_{x}\vec{\psi}\, \dxdt-
\int_{0}^{T}\int_{\wtilde\Omega} \vec{W}_{\veps ,M}\otimes \vec{W}_{\veps,M}: \nabla_{x}\vec{\psi}\, \dxdt\right|=0\, .
\end{equation*}
\end{lemma}

\begin{proof}
Let $\vec \psi\in C_c^\infty\bigl(\R_+\times\wtilde\Omega;\R^3\bigr)$, with $\Supp\vec\psi\subset[0,T]\times K$, for some compact set $K\subset\wtilde\Omega$.
Then, we take $l>0$ in \eqref{eq:cut-off} so large that $K\subset \wtilde{\mbb{B}}_{l}\,:=\,B_l(0)\times\T$.
Therefore, using \eqref{def_deltarho}, we get
$$
\int_{0}^{T}\int_{\wtilde\Omega} \vrho_\veps\,\vec{u}_\veps\otimes \vec{u}_\veps: \nabla_{x}\vec{\psi}\,=\,
\int_{0}^{T}\int_{K}(\chi_l\,\vec{u}_\veps)\otimes\vec{u}_\veps:\nabla_{x}\vec{\psi}\,+\veps^{m}\int_{0}^{T}\int_{K}R_\veps\,\vec{u}_\veps\otimes \vec{u}_\veps:\nabla_{x}\vec{\psi}\,.
$$
As a consequence of the uniform bounds $\big(\vec{u}_{\veps}\big)_\veps\subset L^{2}_{T}(L^{6}_{\rm loc})$ and $\big(R_{\veps}\big)_\veps\subset L^{\infty}_{T}(L_{\rm loc}^{5/3})$
(recall \eqref{uni_varrho1} above), the second integral in the right-hand side is of order $\veps^{m}$. As for the first one, 
using \eqref{eq:t-T}, we can write
$$
\int_{0}^{T}\int_{K}(\chi_l\,\vec{u}_\veps)\otimes \vec{u}_\veps:\nabla_{x}\vec{\psi}\,=\,\int_{0}^{T}\int_{K}\vec{t}^2_{\veps,M}\otimes\vec{u}_\veps:\nabla_{x}\vec{\psi}
+\int_{0}^{T}\int_{K} \,(\Id-{S}_{M})(\chi_l\,\vec{u}_\veps)\otimes\vec{u}_\veps: \nabla_{x}\vec{\psi}\,.
$$
Observe that, in view of characterisation \eqref{eq:LP-Sob}, one has the property (see also Lemma \ref{lemma_sobolev_H^s})
\begin{equation*}
\left\|(\Id-{S}_{M})(\chi_l\,\vec{u}_\veps)\right\|_{L_{T}^{2}(L^{2})}\,\leq\,C\,2^{-M}\,\left\|\nabla_{x}(\chi_l\,\vec{u}_\veps)\right\|_{L_{T}^{2}(L^{2})}\,\leq\,C(l)\,2^{-M}\,.
\end{equation*}
Therefore, it is enough to consider the first term in the right-hand side of the last relation: we have
$$
\int_{0}^{T}\int_{K}\vec{t}^2_{\veps,M}\otimes \vec{u}_\veps:\nabla_{x}\vec{\psi}\,=\,\int_{0}^{T}\int_K\vec{t}^2_{\veps,M}\otimes\vec{t}^2_{\veps,M}:\nabla_{x}\vec{\psi}\,+\,
\int_{0}^{T}\int_{K} \,\vec{t}^2_{\veps,M}\otimes(\Id-{S}_{M})(\chi_l\,\vec{u}_\veps): \nabla_{x}\vec{\psi}\,,
$$
where, for the same reason as before, we gather that
\begin{equation*}
\lim_{M\rightarrow +\infty}\limsup_{\veps \rightarrow 0^+}\left|\int_{0}^{T}\int_{K}\vec{t}^2_{\veps,M}\otimes (\Id-{S}_{M})(\chi_l\,\vec{u}_\veps): \nabla_{x}\vec{\psi}\right|=0\, .
\end{equation*}

It remains us to consider the integral 
$$
\int_{0}^{T}\int_K\vec{t}^2_{\veps,M}\otimes\vec{t}^2_{\veps,M}:\nabla_{x}\vec{\psi}\,=\,\int_{0}^{T}\int_{K} \vec{W}_{\veps ,M}\otimes \vec t^2_{\veps,M}: \nabla_{x}\vec{\psi}
-\veps^{m}\int_{0}^{T}\int_{K}\vec t^1_{\veps,M}\otimes \vec t^2_{\veps,M}: \nabla_{x}\vec{\psi}\,,
$$
where we notice that, owing to Proposition \ref{p:prop dec}, the latter term in the right-hand side is of order $\veps^{m}$, so it vanishes at the limit.
As a last step, we write
$$
\int_{0}^{T}\int_{K} \vec{W}_{\veps ,M}\otimes \vec t^2_{\veps,M}: \nabla_{x}\vec{\psi}\,=\,
\int_{0}^{T}\int_{K} \vec{W}_{\veps ,M}\otimes \vec W_{\veps,M}: \nabla_{x}\vec{\psi}\,-\,\veps^m\int_{0}^{T}\int_{K} \vec{W}_{\veps ,M}\otimes \vec t^1_{\veps,M}: \nabla_{x}\vec{\psi}\,.
$$
Using Lemma \ref{l:S-W_bounds} together with Bernstein's inequalities of Lemma \ref{l:bern}, we see that the latter integral in the right-hand side is of order $\veps^{m}$.
This concludes the proof of the lemma.
\qed
\end{proof}

\medbreak
From now on, in order to avoid the appearance of (irrelevant) multiplicative constants everywhere, we suppose that the torus $\T$ has been normalised so that its Lebesgue measure is equal to $1$.

In view of the previous lemma and of Proposition \ref{p:limitpoint}, for any test-function
\begin{equation} \label{eq:test-f}
\vec\psi\in C_c^\infty\big([0,T[\,\times\wtilde\Omega;\R^3\big)\qquad\qquad \mbox{ such that }\qquad \div\vec\psi=0\quad\mbox{ and }\quad \d_3\vec\psi=0\,,
\end{equation}
we have to pass to the limit in the term 
\begin{align*}
-\int_{0}^{T}\int_{\wtilde\Omega} \vec{W}_{\veps ,M}\otimes \vec{W}_{\veps ,M}: \nabla_{x}\vec{\psi}\,&=\,\int_{0}^{T}\int_{\wtilde\Omega} \div\left(\vec{W}_{\veps ,M}\otimes
\vec{W}_{\veps ,M}\right) \cdot \vec{\psi}\,.
\end{align*}
Notice that the integration by parts above is well-justified, since all the quantities inside the integrals are smooth. At this point, we observe that,
resorting to the notation in \eqref{eq:decoscil} presented in the introductory part, we can write
$$
\int_{0}^{T}\int_{\wtilde\Omega} \div\left(\vec{W}_{\veps ,M}\otimes \vec{W}_{\veps ,M}\right) \cdot \vec{\psi}\,=\,
\int_{0}^{T}\int_{\R^2} \left(\mc{T}_{\veps ,M}^{1}+\mc{T}_{\veps, M}^{2}\right)\cdot\vec{\psi}^h\,,
$$
where we have defined the terms
\begin{equation} \label{def:T1-2}
\mc T^1_{\veps,M}\,:=\, \divh\left(\langle \vec{W}_{\veps ,M}^{h}\rangle\otimes \langle \vec{W}_{\veps ,M}^{h}\rangle\right)\qquad \mbox{ and }\qquad
\mc T^2_{\veps,M}\,:=\, \divh\left(\langle \dbtilde{\vec{W}}_{\veps ,M}^{h}\otimes \dbtilde{\vec{W}}_{\veps ,M}^{h}\rangle \right)\,.
\end{equation}
So, it is enough to focus on each of them separately.
For notational convenience, from now on we will generically denote by $\mc{R}_{\veps ,M}$ any remainder term, that is any term satisfying the property
\begin{equation} \label{eq:remainder}
\lim_{M\rightarrow +\infty}\limsup_{\veps \rightarrow 0^+}\left|\int_{0}^{T}\int_{\wtilde\Omega}\mc{R}_{\veps ,M}\cdot \vec{\psi}\, \dxdt\right|=0\, ,
\end{equation}
for all test functions $\vec{\psi}\in C_c^\infty\bigl([0,T[\,\times\wtilde\Omega;\R^3\bigr)$ as in \eqref{eq:test-f}. 

\subsubsection{The analysis of the $\mc{T}_{\veps ,M}^{1}$ term}\label{sss:term1}
We start by dealing with $\mc T^1_{\veps,M}$. Standard computations give
\begin{align}
\mc{T}_{\veps ,M}^{1}\,&=\,\divh\left(\langle \vec{W}_{\veps ,M}^{h}\rangle\otimes \langle \vec{W}_{\veps ,M}^{h}\rangle\right)=
\divh\langle \vec{W}_{\veps ,M}^{h}\rangle\, \langle \vec{W}_{\veps ,M}^{h}\rangle+\langle \vec{W}_{\veps ,M}^{h}\rangle \cdot \nabla_{h}\langle \vec{W}_{\veps ,M}^{h}\rangle \label{eq:T1} \\
&=\divh\langle \vec{W}_{\veps ,M}^{h}\rangle\, \langle \vec{W}_{\veps ,M}^{h}\rangle+\frac{1}{2}\, \nabla_{h}\left(\left|\langle \vec{W}_{\veps ,M}^{h}\rangle\right|^{2}\right)+
\curlh\langle \vec{W}_{\veps ,M}^{h}\rangle\,\langle \vec{W}_{\veps ,M}^{h}\rangle^{\perp}\,. \nonumber
\end{align}
Notice that we can forget about the second term, because it is a perfect gradient and we are testing against divergence-free test functions.
For the first term, we take advantage of system \eqref{eq:approx wave}: averaging the first equation with respect to $x^{3}$ and multiplying it by $\langle \vec{W}^h_{\veps ,M}\rangle$, we arrive at
$$
\divh\langle \vec{W}_{\veps ,M}^{h}\rangle\,\langle \vec{W}_{\veps ,M}^{h}\rangle\,=\,-\frac{\veps^{m}}{\mc{A}}\d_t\langle \Lambda_{\veps ,M}\rangle \langle \vec{W}_{\veps ,M}^{h}\rangle+
\frac{\veps^{m}}{\mc{A}} \langle f_{\veps ,M}^{h}\rangle \langle \vec{W}_{\veps ,M}^{h}\rangle\,=\,
\frac{\veps^{m}}{\mc{A}}\langle \Lambda_{\veps ,M}\rangle \d_t \langle \vec{W}_{\veps ,M}^{h}\rangle +\mc{R}_{\veps ,M}\,.
$$
We remark that the term presenting the total derivative in time is in fact a remainder.
We use now the horizontal part of \eqref{eq:approx wave}, where we take the vertical average and then multiply by $\langle \Lambda_{\veps ,M}\rangle$: we gather
\begin{align*}
\frac{\veps^{m}}{\mc{A}}\langle \Lambda_{\veps ,M}\rangle \d_t \langle \vec{W}_{\veps ,M}^{h}\rangle &=
-\frac{1}{\mc{A}} \langle \Lambda_{\veps ,M}\rangle \nabla_{h}\langle \Lambda_{\veps ,M}\rangle+\frac{\veps^{m}}{\mc{A}}\langle \Lambda_{\veps ,M}\rangle \langle \vec G_{\veps ,M}^{h}\rangle-
\frac{\veps^{m-1}}{\mc{A}}\langle \Lambda_{\veps ,M}\rangle\langle \vec{W}_{\veps ,M}^{h}\rangle^{\perp}\\
&=-\frac{\veps^{m-1}}{\mc{A}}\langle \Lambda_{\veps ,M}\rangle\langle \vec{W}_{\veps ,M}^{h}\rangle^{\perp}-\frac{1}{2\mc{A}}  \nabla_{h}\left( \left| \langle \Lambda_{\veps ,M}\rangle \right|^{2}\right)+\mc{R}_{\veps ,M}\\
&=-\frac{\veps^{m-1}}{\mc{A}}\langle \Lambda_{\veps ,M}\rangle\langle \vec{W}_{\veps ,M}^{h}\rangle^{\perp}+\mc{R}_{\veps ,M}\, ,
\end{align*}
where we repeatedly exploited the properties proved in Proposition \ref{p:prop approx} and we included in the remainder term also the perfect gradient.
Inserting this relation into \eqref{eq:T1}, we find
\begin{equation*}
\mc{T}_{\veps ,M}^{1}= \g_{\veps,M}\,\langle\vec{W}_{\veps,M}^{h}\rangle^{\perp}+\mc{R}_{\veps,M}\,,
\qquad\qquad\mbox{ with }\qquad \gamma_{\veps, M}:=\curlh\langle \vec{W}_{\veps ,M}^{h}\rangle\,-\,\frac{\veps^{m-1}}{\mc{A}}\langle \Lambda_{\veps ,M}\rangle\,.
\end{equation*}


We observe that, for passing to the limit in $\mc{T}_{\veps ,M}^{1}$, there is no other way than finding some strong convergence property for $\vec W_{\veps,M}$. 
Such a property is in fact hidden in the structure of the wave system \eqref{eq:approx wave}: in order to exploit it, some work on the term $\g_{\veps,M}$ is needed.
We start by rewriting the vertical average of the first equation in \eqref{eq:approx wave} as
\begin{equation*}
\frac{\veps^{2m-1}}{\mc{A}}\,\d_t \langle \Lambda_{\veps ,M} \rangle\,+\,\veps^{m-1}\div_{h} \langle \vec{W}_{\veps ,M}^{h}\rangle\,=\,\frac{\veps^{2m-1}}{\mc{A}}\, \langle f_{\veps ,M}^{h}\rangle\,.
\end{equation*}
On the other hand, taking the vertical average of the horizontal components of \eqref{eq:approx wave} and then applying $\curlh$, we obtain the relation
\begin{equation*}
\veps^m\,\d_t\curlh\langle \vec{W}_{\veps ,M}^{h}\rangle\,+\veps^{m-1}\,\divh\langle \vec{W}_{\veps ,M}^{h}\rangle\, =\,\veps^m \curlh\langle\vec G_{\veps ,M}^{h}\rangle\, .
\end{equation*}
Summing up the last two equations, we discover that
\begin{equation} \label{eq:gamma}
\d_{t}\gamma_{\veps,M}\,=\,\curlh\langle \vec G_{\veps ,M}^{h}\rangle\,-\,\frac{\veps^{m-1}}{\mc{A}}\,\langle f_{\veps ,M}^{h}\rangle \, .
\end{equation}
Thanks to estimate \eqref{eq:approx force} in Proposition \ref{p:prop approx}, we discover that (for any $M>0$ fixed) the family 
$\left(\d_{t}\,\gamma_{\veps,M}\right)_{\veps}$
is uniformly bounded (with respect to $\veps$) in e.g. $L_{T}^{2}(L^{2})$. 
On the other hand, thanks to Lemma \ref{l:S-W_bounds} and Sobolev embeddings, we have that (for any $M>0$ fixed)
the sequence $(\gamma_{\veps,M})_{\veps}$ is uniformly bounded (with respect to $\veps$) in the space $L_{T}^{2}(H^{1})$.
Since the embedding $H_{\rm loc}^{1}\hookrightarrow L_{\rm loc}^{2}$ is compact, the Aubin-Lions Theorem (see again the Appendix \ref{appendixA}) implies that, for any $M>0$ fixed, the family $(\gamma_{\veps,M})_{\veps}$ is compact
in $L_{T}^{2}(L_{\rm loc}^{2})$. Then, it converges strongly (up to extracting a subsequence) to a tempered distribution $\gamma_M$ in the same space. 
Of course, by definition of $\g_{\veps,M}$ (and whenever $m>1$), this tells us that also $\big(\curlh\lan\vec W_{\veps,M}^h\ran\big)_\veps$ is compact in $L^2_T(L^2_{\rm loc})$.

Now, we have that $\gamma_{\veps ,M}$ converges strongly to $\gamma_M$ in $L_{T}^{2}(L_{\rm loc}^{2})$ and $\langle \vec{W}_{\veps ,M}^{h}\rangle$ converges weakly to
$\langle \vec{W}_{M}^{h}\rangle$ in $L_{T}^{2}(L_{\rm loc}^{2})$ (owing to Proposition \ref{p:prop dec}, for instance). Then, we deduce that
\begin{equation*}
\gamma_{\veps,M}\langle \vec{W}_{\veps ,M}^{h}\rangle^{\perp}\longrightarrow \gamma_M \langle \vec{W}_{M}^{h}\rangle^{\perp}\qquad \text{ in }\qquad \mc{D}^{\prime}\big(\R_+\times\R^2\big)\,.
\end{equation*}
Observe that, by definition of $\g_{\veps,M}$, we must have $\gamma_M=\curlh\langle \vec{W}_{M}^{h}\rangle$. On the other hand, by Proposition \ref{p:prop dec} and \eqref{eq:t-T},
we know that $\langle \vec{W}_{M}^{h}\rangle= \lan{S}_{M}(\chi_l\vec{U}^{h})\ran$.

In the end, we have proved that, for any $T>0$ and any test-function $\vec \psi$ as in \eqref{eq:test-f},
one has the convergence (at any $M\in\N$ fixed, when $\veps\ra0^+$)
\begin{equation} \label{eq:limit_T1}
\int_{0}^{T}\int_{\R^2}\mc{T}_{\veps ,M}^{1}\cdot\vec{\psi}^h\,\dx^h\dt\,\longrightarrow\,
\int^T_0\int_{\R^2}\curlh\lan{S}_{M}(\chi_l\vec{U}^{h})\ran\; \lan{S}_{M}\big(\chi_l(\vec{U}^{h})^{\perp}\big)\ran\cdot\vec\psi^h\,\dx^h\dt\,.
\end{equation}

\subsubsection{Dealing with the term $\mc{T}_{\veps ,M}^{2}$}\label{sss:term2}
Let us now consider the term $\mc{T}_{\veps ,M}^{2}$, defined in \eqref{def:T1-2}. By the same computation as above, we infer that
\begin{align}
\mc{T}_{\veps ,M}^{2}\,
&=\,\langle \divh (\dbtilde{\vec{W}}_{\veps ,M}^{h})\;\;\dbtilde{\vec{W}}_{\veps ,M}^{h}\rangle+\frac{1}{2}\, \langle \nabla_{h}| \dbtilde{\vec{W}}_{\veps ,M}^{h}|^{2} \rangle+
\langle \curlh\dbtilde{\vec{W}}_{\veps ,M}^{h}\,\left( \dbtilde{\vec{W}}_{\veps ,M}^{h}\right)^{\perp}\rangle\, . \label{eq:T2} 
\end{align}

Let us now introduce now the quantities
$$
\dbtilde{\Phi}_{\veps ,M}^{h}\,:=\,( \dbtilde{\vec{W}}_{\veps ,M}^{h})^{\perp}-\d_{3}^{-1}\nabla_{h}^{\perp}\dbtilde{\vec{W}}_{\veps ,M}^{3}\qquad\mbox{ and }\qquad
\dbtilde{\omega}_{\veps ,M}^{3}\,:=\,\curlh \dbtilde{\vec{W}}_{\veps ,M}^{h}\,.
$$
Then we can write
\begin{equation*}
\left( \curl \dbtilde{\vec{W}}_{\veps ,M}\right)^{h}\,=\,\d_3 \dbtilde{\Phi}_{\veps ,M}^{h}\qquad \text{ and }\qquad
\left( \curl \dbtilde{\vec{W}}_{\veps ,M}\right)^{3}\,=\,\dbtilde{\omega}_{\veps ,M}^{3}\,.
\end{equation*}
In addition, from the momentum equation in \eqref{eq:approx wave}, where we take the mean-free part and then the $\curl$, we deduce the equations
\begin{equation} \label{eq:eq momentum term2}
\begin{cases}
\veps^{m}\d_t\dbtilde{\Phi}_{\veps ,M}^{h}-\veps^{m-1}\dbtilde{\vec{W}}_{\veps ,M}^{h}=\veps^m\left(\d_{3}^{-1}\curl\dbtilde{\vec G}_{\veps ,M} \right)^{h}\\[1ex]
\veps^{m}\d_t\dbtilde{\omega}_{\veps ,M}^{3}+\veps^{m-1}\divh\dbtilde{\vec{W}}_{\veps ,M}^{h}=\veps^m\,\curlh\dbtilde{\vec G}_{\veps ,M}^{h}\, .
\end{cases}
\end{equation}
Making use of the relations above and of Propositions \ref{p:prop approx} and \ref{p:prop dec}, we get
\begin{equation*}
\begin{split}
\curlh\dbtilde{\vec{W}}_{\veps ,M}^{h}\;\left(\dbtilde{\vec{W}}_{\veps ,M}^{h}\right)^{\perp}&=\dbtilde{\omega}_{\veps ,M}^{3}\left(\dbtilde{\vec{W}}_{\veps ,M}^{h}\right)^{\perp}\,=\,
\veps \d_t\!\left( \dbtilde{\Phi}_{\veps ,M}^{h}\right)^{\perp}\dbtilde{\omega}_{\veps ,M}^{3}-
\veps\dbtilde{\omega}_{\veps ,M}^{3}\left(\left(\d_{3}^{-1}\curl\dbtilde{\vec G}_{\veps ,M}\right)^{h}\right)^\perp  \\
&=-\veps \left( \dbtilde{\Phi}_{\veps ,M}^{h}\right)^{\perp}\d_t\dbtilde{\omega}_{\veps ,M}^{3}+\mc{R}_{\veps ,M}=
\left( \dbtilde{\Phi}_{\veps ,M}^{h}\right)^{\perp}\,\divh\dbtilde{\vec{W}}_{\veps ,M}^{h}+\mc{R}_{\veps ,M}\, .
\end{split}
\end{equation*}
Hence, including also the gradient term into the remainders, from \eqref{eq:T2} we arrive at 
\begin{align*}
\mc{T}_{\veps ,M}^{2}\,&=\,\langle \divh\dbtilde{\vec{W}}_{\veps ,M}^{h}\,\left(\dbtilde{\vec{W}}_{\veps ,M}^{h}+\left(\dbtilde{\Phi}_{\veps ,M}^{h}\right)^{\perp}\right) \rangle+\mc{R}_{\veps ,M} \\
&=\,\langle \div \dbtilde{\vec{W}}_{\veps ,M}\left(\dbtilde{\vec{W}}_{\veps ,M}^{h}+\left(\dbtilde{\Phi}_{\veps ,M}^{h}\right)^{\perp}\right) \rangle -
\langle \d_3 \dbtilde{\vec{W}}_{\veps ,M}^{3}\left(\dbtilde{\vec{W}}_{\veps ,M}^{h}+\left(\dbtilde{\Phi}_{\veps ,M}^{h}\right)^{\perp}\right) \rangle+\mc{R}_{\veps ,M}\, .
\end{align*}
The second term on the right-hand side of the last line is actually another remainder. Indeed, using the definition of the function $\dbtilde{\Phi}_{\veps ,M}^{h}$ and the fact
that the test function $\vec\psi$ does not depend on $x^3$, one has
\begin{equation*}
\begin{split}
\d_3 \dbtilde{\vec{W}}_{\veps ,M}^{3}\left(\dbtilde{\vec{W}}_{\veps ,M}^{h}+\left(\dbtilde{\Phi}_{\veps ,M}^{h}\right)^{\perp}\right)&=\d_3 \left(\dbtilde{\vec{W}}_{\veps ,M}^{3}\left(\dbtilde{\vec{W}}_{\veps ,M}^{h}+\left(\dbtilde{\Phi}_{\veps ,M}^{h}\right)^{\perp}\right)\right) - \dbtilde{\vec{W}}_{\veps ,M}^{3}\, \d_3\left(\dbtilde{\vec{W}}_{\veps ,M}^{h}+\left(\dbtilde{\Phi}_{\veps ,M}^{h}\right)^{\perp}\right)\\
&=\mc{R}_{\veps ,M}-\frac{1}{2}\nabla_{h}\left|\dbtilde{\vec{W}}_{\veps ,M}^{3}\right|^{2}=\mc{R}_{\veps ,M}\, .
\end{split}
\end{equation*}
As for the first term, instead, we use the first equation in \eqref{eq:approx wave} to obtain
\begin{equation*}
\begin{split}
\div \dbtilde{\vec{W}}_{\veps ,M}\left(\dbtilde{\vec{W}}_{\veps ,M}^{h}+\left(\dbtilde{\Phi}_{\veps ,M}^{h}\right)^{\perp}\right)&=-\frac{\veps^{m}}{\mc{A}} \d_t \dbtilde{\Lambda}_{\veps ,M}\left(\dbtilde{\vec{W}}_{\veps ,M}^{h}+\left(\dbtilde{\Phi}_{\veps ,M}^{h}\right)^{\perp}\right)+\mc{R}_{\veps ,M}\\
&=\frac{\veps^{m}}{\mc{A}} \dbtilde{\Lambda}_{\veps ,M}\, \d_t\left(\dbtilde{\vec{W}}_{\veps ,M}^{h}+\left(\dbtilde{\Phi}_{\veps ,M}^{h}\right)^{\perp}\right)+\mc{R}_{\veps ,M}\, .
\end{split}
\end{equation*}
Now, equations \eqref{eq:approx wave} and \eqref{eq:eq momentum term2} immediately yield that
\begin{equation*}
\frac{\veps^{m}}{\mc{A}} \dbtilde{\Lambda}_{\veps ,M}\, \d_t\left(\dbtilde{\vec{W}}_{\veps ,M}^{h}+\left(\dbtilde{\Phi}_{\veps ,M}^{h}\right)^{\perp}\right)=
\mc{R}_{\veps ,M}-\frac{1}{\mc{A}}\dbtilde{\Lambda}_{\veps ,M}\, \nabla_{h}\left(\dbtilde{\Lambda}_{\veps ,M}\right)=
\mc{R}_{\veps ,M}-\frac{1}{2\mc{A}}\nabla_{h}\left|\dbtilde{\Lambda}_{\veps ,M}\right|^{2}=\mc{R}_{\veps ,M}\,.
\end{equation*}

This relation finally implies that $\mc{T}_{\veps ,M}^{2}\,=\,\mc R_{\veps,M}$ is a remainder, in the sense of relation \eqref{eq:remainder}:
for any $T>0$ and any test-function $\vec \psi$ as in \eqref{eq:test-f},
one has the convergence
(at any $M\in\N$ fixed, when $\veps\ra0^+$)
\begin{equation} \label{eq:limit_T2}
\int_{0}^{T}\int_{\R^2}\mc{T}_{\veps ,M}^{2}\cdot\vec{\psi}^h\,\dx^h\dt\,\longrightarrow\,0\,.
\end{equation}

\begin{remark}\label{r:T1-T2}
Due to the presence of the term $\vec Y^2_\veps$ in \eqref{eq:wave_syst}, the choice $m\geq2$ is fundamental.
However, as soon as $F=0$, our analysis applies also in the case when $1<m<2$.
\end{remark}
 
\subsection{The limit dynamics} \label{ss:limit} 
With the convergences established in \eqref{conv:r} to \eqref{conv:u} and in Subsection \ref{ss:convergence}, we can pass to the limit in equation \eqref{weak-mom}.
Since all the integrals will be made on $\R^2$ (in view of the choice of the test functions in \eqref{eq:test-f} above), we can safely come back to the notation on $\Omega$ instead of $\wtilde\Omega$.

To begin with, we take a test-function $\vec\psi$ as in \eqref{eq:test-f}, specifically
$$
\vec{\psi}=\big(\nabla_{h}^{\perp}\phi,0\big)\,,\qquad\qquad\mbox{ with }\qquad \phi\in C_c^\infty\big([0,T[\,\times\R^2\big)\,,\quad \phi=\phi(t,x^h)\,.
$$
For such a $\vec\psi$, all the gradient terms vanish identically, as well as all the contributions
due to the vertical component of the equation. Hence, after using also \eqref{prF}, equation \eqref{weak-mom} becomes
\begin{align}
\int_0^T\!\!\!\int_{\Omega}  
& \left( -\vre \ue^h \cdot \partial_t \vec\psi^h -\vre \ue^h\otimes\ue^h  : \nabla_h \vec\psi^h
+ \frac{1}{\ep}\vre\big(\ue^{h}\big)^\perp\cdot\vec\psi^h\right)\, \dxdt \label{eq:weak_to_conv}\\
& =\int_0^T\!\!\!\int_{\Omega} 
\left(-\mbb{S}(\vtheta_\veps,\nabla_x\vec\ue): \nabla_x \vec\psi+\frac{1}{\veps^{2}}(\vrho_\veps -\widetilde{\vrho}_\veps)\, \nabla_x F\cdot \vec\psi\right)\,\dxdt+
\int_{\Omega}\vrez \uez  \cdot \vec\psi(0,\cdot)\dx\,. \nonumber
\end{align}
Making use of the uniform bounds of Subsection \ref{ss:unif-est}, we can pass to the limit in the $\d_t$ term, in the viscosity term and in the centrifugal force.
Moreover, our assumptions imply that $\vrho_{0,\veps}\vec{u}_{0,\veps}\rightharpoonup \vec{u}_0$ in $L_{\rm loc}^2$. 

Let us consider now the Coriolis term. We can write
\begin{align*}
\int_0^T\!\!\!\int_{\Omega}\frac{1}{\ep}\vre\big(\ue^{h}\big)^\perp\cdot\vec\psi^h\,&=\,\int_0^T\!\!\!\int_{\mbb{R}^2}\frac{1}{\ep}\langle\vre \ue^{h}\rangle \cdot \nabla_{h}\phi\,=\,
-\veps^{m-1}\int_0^T\!\!\!\int_{\mbb{R}^2}\langle R_\veps\rangle\, \d_t\phi\,-\,\veps^{m-1}\int_{\mbb{R}^2}\langle  R_{0,\veps}\rangle\, \phi(0,\cdot )\,, 
\end{align*}
which of course converges to $0$ when $\veps\ra0^+$. Notice that the second equality derives from the mass equation \eqref{weak-con}, tested against $\phi$: namely,
\begin{equation*}
-\veps^m\int_0^T\!\!\!\int_{\mbb{R}^2}\langle\frac{\vrho_\veps-1}{\veps^m}\rangle\, \d_t\phi\,-\,\int_0^T\!\!\!\int_{\mbb{R}^2}\langle \vrho_\veps \ue^h\rangle \cdot \nabla_{h}\phi\,=\,
\veps^m\int_{\mbb{R}^2}\langle\frac{\vrho_{0,\veps}-1}{\veps^m}\rangle\, \phi(0,\cdot )\,.
\end{equation*}

It remains to deal with the convective term $\vrho_\veps \ue^h \otimes \ue^h$. For this, we take  advantage of
Lemma \ref{lem:convterm} and relations \eqref{eq:limit_T1} and \eqref{eq:limit_T2}.
Next, we remark that, since $\vec U^h\in L^2_T(H_{\rm loc}^1)$ by \eqref{conv:u}, from \eqref{est:sobolev} we gather the strong convergence
$S_M(\chi_l\vec U^h)\longrightarrow\chi_l\vec{U}^{h}$ in $L_{T}^{2}(H^{s})$ for any $s<1$ and any $l>0$ fixed, in the limit for $M\rightarrow +\infty$.
Therefore,  in the term on the right-hand side of \eqref{eq:limit_T1}, we can perform equalities \eqref{eq:T1} backwards, and then pass to the limit also for $M\ra+\infty$.
Using that $\chi_l\equiv1$ on $\Supp\vec\psi$ by construction, we finally get the convergence (for $\veps\ra0^+$)
\begin{equation*}
\int_0^T\int_{\Omega} \vre \ue^h\otimes\ue^h  : \nabla_h \vec\psi^h\, \longrightarrow\, \int_0^T\int_{\R^2}\vec{U}^h\otimes\vec{U}^h  : \nabla_h \vec\psi^h\,.
\end{equation*}

In the end, letting $\varepsilon \rightarrow 0^+$ in \eqref{eq:weak_to_conv}, we may infer that 
\begin{align*}
&\int_0^T\!\!\!\int_{\R^2} \left(\vec{U}^{h}\cdot \d_{t}\vec\psi^h+\vec{U}^{h}\otimes \vec{U}^{h}:\nabla_{h}\vec\psi^h\right)\, \dx^h\dt\\
&\qquad\qquad= \int_0^T\!\!\!\int_{\R^2} \left(\mu(\oline\vartheta )\nabla_{h}\vec{U}^{h}:\nabla_{h}\vec\psi^h-\delta_2(m)\lan\varrho^{(1)}\ran\nabla_{h}F\cdot \vec\psi^h\right)\, \dx^h\dt-
\int_{\R^2}\lan\vec{u}_{0}^{h}\ran\cdot \vec\psi^h(0,\cdot)\dx^h\,,
\end{align*}
where $\delta_2(m)=1$ if $m=2$, $\delta_2(m)=0$ otherwise. At this point, Remark \ref{r:F-G} applied to the case $m=2$ yields the equality
$\d_\vrho p(1,\oline\vtheta)\,\nabla_h\lan\wtilde r\ran\,=\,\nabla_hF$. Therefore, keeping in mind that $R=\vrho^{(1)}+\wtilde r$, we get
$$
\lan\varrho^{(1)}\ran\nabla_{h}F\,=\,\lan R\ran\nabla_{h}F\,-\,\lan\wtilde r\ran\nabla_{h}F\,=\,\lan R\ran\nabla_{h}F\,-\,\frac{\d_\vrho p(1,\oline\vtheta)}{2}\,\nabla_h\left|\lan\wtilde r\ran\right|^2\,.
$$
Of course, the perfect gradient disappears from the weak formulation. Using this observation in the target momentum equation written above, we finally deduce \eqref{eq_lim_m:momentum}.
This completes the proof of Theorem \ref{th:m-geq-2}, in the case when $m\geq2$ and $F\neq0$.

When $m>1$ and $F=0$, most of the arguments above still apply. We refer to the next section for more details.

\section{Proof of the convergence in the case when $F=0$} \label{s:proof-1}
In the present section we prove the convergence result in the case $F=0$. For the sake of brevity, we focus on the case $m=1$, completing in this way the proof to Theorem \ref{th:m=1_F=0}.
The case $m>1$ follows by a similar analysis, using at the end the compensated compactness argument depicted in Subsection \ref{ss:convergence} (recall also Remark \ref{r:T1-T2} above).



\subsection{Analysis of the acoustic-Poincar\'e waves}\label{ss:unifbounds_1} 

We start by remarking that system \eqref{ceq} to \eqref{eeq} can be recasted in the form \eqref{eq:wave_syst}, with $m=1$: with the same notation introduced in Paragraph \ref{sss:wave-eq},
and after setting $X_\veps\,:=\,\div\vec{X}^1_\veps\,+\,X^2_\veps$ and $\vec Y_\veps\,:=\,\div\mbb{Y}^1_\veps\,+\,\vec{Y}^2_\veps\,+\,\nabla_x Y^3_\veps$,
we have
\begin{equation} \label{eq:wave_m=1}
\left\{\begin{array}{l}
       \veps\,\d_tZ_\veps\,+\,\mc{A}\,\div\vec{V}_\veps\,=\,\veps\,X_\veps \\[1ex]
       \veps\,\d_t\vec{V}_\veps\,+\,\nabla_x Z_\veps\,+\,\,\e_3\times \vec V_\veps\,=\,\veps\,\vec Y_\veps\,,\qquad\qquad\big(\vec{V}_\veps\cdot\vec n\big)_{|\d\Omega_\veps}\,=\,0\,,
       \end{array}
\right.
\end{equation}
where $\bigl(Z_\veps\bigr)_\veps$  and $\bigl(\vec V_\veps\bigr)_\veps$ are defined as in Paragraph \ref{sss:wave-eq}.
This system is supplemented with the initial datum $\big(Z_{0,\veps},\vec V_{0,\veps}\big)$, where these two functions are defined as in relation \eqref{def:wave-data} above.


\subsubsection{Uniform bounds and regularisation}

In the next lemma, we establish uniform bounds for $Z_\veps$ and $\vec V_\veps$. Its proof is an easy adaptation of the one given in Lemma \ref{l:S-W_bounds}, hence omitted. One has
to use the fact that, since $F=0$, all the bounds obtained in the previous sections hold now on the whole $\Omega_\veps$, with constants which are uniform in $\veps\in\,]0,1]$;
therefore, one can abstain from using the cut-off functions $\chi_l$.
\begin{lemma} \label{l:S-W_bounds_1}
For any $T>0$ and all $\ep \in \, ]0,1]$, we have
$$
\sup_{\veps\in\,]0,1]}\| Z_\veps\|_{L^\infty_T((L^2+L^{5/3}+L^1+\mc{M}^+)(\Omega_\veps))} \leq c\, ,\quad\quad
\sup_{\veps\in\,]0,1]}\| \vec V_\veps\|_{L^2_T((L^2+L^{30/23})(\Omega_\veps))} \leq c \, .
$$

\end{lemma}

\medbreak
Now, we state the analogous of Lemma \ref{l:source_bounds} for $m=1$ and $F=0$.
\begin{lemma} \label{l:source_bounds_1}
For $\veps\in\,]0,1]$, let us introduce the following spaces:
\begin{enumerate}[(i)]
 \item $\mc X^\veps_1\,:=\,L^2_{\rm loc}\Bigl(\R_+;\big(L^2+L^{1}+L^{3/2}+L^{30/23}+L^{30/29}\big)(\Omega_\veps)\Bigr)$;
\item $\mc X^\veps_2\,:=\,L^2_{\rm loc}\Bigl(\R_+;\big(L^2+L^1+L^{4/3}\big)(\Omega_\veps)\Bigr)$;
\item $\mc X^\veps_3\,:=\,L^\infty_{\rm loc}\Bigl(\R_+;\big(L^2+L^{5/3}\big)(\Omega_\veps)\Bigr)$;
\item $\mc X^\veps_4\,:=\,L^\infty_{\rm loc}\Bigl(\R_+;\big(L^2+L^{5/3}+L^1+\mc{M}^+\big)(\Omega_\veps)\Bigr)$.
\end{enumerate}

Then, one has the following uniform bound, for a constant $C>0$ independent of $\veps\in\,]0,1]$:
$$
\left\|\vec X^1_\veps\right\|_{\mc X_1^\veps}\,+\,\left\|X^2_\veps\right\|_{\mc X_1^\veps}\,+\,\left\|\mbb Y^1_\veps\right\|_{\mc X_2^\veps}\,+\,
\left\|\vec{Y}^2_\veps\right\|_{\mc X_3^\veps}\,+\,\left\|Y^3_\veps\right\|_{\mc X_4^\veps}\,\leq\,C\,.
$$
In particular, one has that the sequences $(X_\veps)_\veps$ and $(\vec Y_\veps)_\veps$, defined in system \eqref{eq:wave_m=1}, verify\footnote{For any $s\in\R$, we denote by $\lfloor s \rfloor$ the entire part of $s$,
i.e. the greatest integer smaller than or equal to $s$.}
$$
\left\|X_\veps\right\|_{L^2_T(H^{-\lfloor s \rfloor-1}(\Omega_\veps))}\,+\,\left\|\vec Y_\veps\right\|_{L^2_T(H^{-\lfloor s \rfloor-1}(\Omega_\veps))}\,\leq\,C
,$$
for all $s>5/2$ and for a constant $C>0$ independent of $\veps\in\,]0,1]$.
\end{lemma}

\begin{proof}
The proof follows the main lines of the proof of Lemma \ref{l:source_bounds}. Here, we limit ourselves to point out that we have a slightly better control on
$\vec Y^2_\veps\,=\,\vrho_\veps^{(1)}\,\nabla_{x} G$, whose boundedness in $\mc X^\veps_3$ follows from \eqref{assFG} and the estimate analogous to \eqref{uni_varrho1} for the case $F=0$.
\qed
\end{proof}

 \medbreak
The next step consists in regularising all the terms appearing in \eqref{eq:wave_m=1}. Here we have to pay attention: since the domains $\Omega_\veps$ are bounded,
we cannot use the Littlewood-Paley operators $S_M$ directly.
Rather than multiplying by a cut-off function $\chi_l$ as done in the previous section (a procedure which would create more complicated forcing terms in the wave system), we use here the arguments
of Chapter 8 of \cite{F-N} (see also \cite{F-Scho}, \cite{WK}), based on finite propagation speed properties for \eqref{eq:wave_m=1}.

First of all, similarly to Paragraph \ref{sss:w-reg} above, we extend our domains $\Omega_\veps$ and $\Omega$ by periodicity in the third variable and denote
$$
\wtilde\Omega_\veps\,:=\,{B}_{L_\veps} (0) \times \mbb{T}^1\qquad\qquad\mbox{ and }\qquad\qquad
\wtilde\Omega\,:=\,\R^2 \times \mbb{T}^1\,.
$$
Thanks to the complete slip boundary conditions \eqref{bc1-2} and \eqref{bc3}, system \eqref{ceq} to \eqref{eeq} can be equivalently reformulated in $\wtilde\Omega_\veps$. 
Analogously, the wave system \eqref{eq:wave_m=1} can be recasted in $\wtilde\Omega_\veps$ in a completely equivalent way. From now on, we will focus on
the equations satisfied on the domain $\wtilde\Omega_\veps$.

Next, we fix a smooth radially decreasing function $\omega\in{C}^\infty_c(\mbb{R}^3)$, such that $0\leq\omega\leq1$, $\omega(x)=0$ for $|x|\geq1$ and
$\int_{\R^3}\omega(x) \dx=1$. Next, we define the mollifying kernel $\big(\omega_M\big)_{M\in\N}$ by the formula
$$
\omega_M(x)\,:=\,2^{3M}\,\,\omega\!\left(2^Mx\right)\qquad\qquad \text{for any}\;\,M\in\N\,\; \text{and any}\;\,x\in\R^3\,.
$$
Then, for any tempered distribution $\mf S=\mf S(t,x)$ on $\R_+\times\wtilde\Omega$ and any $M\in\N$, we define
$$
\mf S_M\,:=\,\omega_M\,*\,\mf S\,,
$$
where the convolution is taken only with respect to the space variables.
Applying the mollifier $\omega_M$ to \eqref{eq:wave_m=1}, we deduce that $Z_{\veps,M}\,:=\,\omega_M*Z_\veps$ and $\vec V_{\veps,M}\,:=\,\omega_M*\vec V_\veps$ satisfy the regularised
wave system
\begin{equation} \label{eq:reg-wave}
\left\{\begin{array}{l}
       \veps\,\d_tZ_{\veps,M}\,+\,\mc{A}\,\div\vec{V}_{\veps,M}\,=\,\veps\,X_{\veps,M} \\[1ex]
       \veps\,\d_t\vec{V}_{\veps,M}\,+\,\nabla_x Z_{\veps,M}\,+\,\,\e_3\times \vec V_{\veps,M}\,=\,\veps\,\vec Y_{\veps,M}
       \end{array}
\right.
\end{equation}
in the domain $\R_+\times\wtilde\Omega_{\veps,M}$, where we have defined
\begin{equation} \label{def:O_e-M}
\wtilde\Omega_{\veps,M}\,:=\,\left\{x\in\wtilde\Omega_\veps\;:\quad{\rm dist}(x,\d\wtilde\Omega_\veps)\geq2^{-M} \right\}\,.
\end{equation}
Since the mollification commutes with standard derivatives, we notice that
$X_{\veps,M}\,=\,\div\vec{X}^1_{\veps,M}\,+\,X^2_{\veps,M}$ and $\vec Y_{\veps,M}\,=\,\div\mbb{Y}^1_{\veps,M}\,+\,\vec{Y}^2_{\veps,M}\,+\,\nabla_x Y^3_{\veps,M}$. 
Moreover, system \eqref{eq:reg-wave} is supplemented with the initial data
$$
Z_{0,\veps,M}\,:=\,\omega_M*Z_{0,\veps}\qquad\qquad\mbox{ and }\qquad\qquad \vec V_{0,\veps,M}\,:=\,\omega_M*\vec V_{0,\veps}\,.
$$

In accordance with Lemmas \ref{l:S-W_bounds_1} and \ref{l:source_bounds_1}, by standard properties of mollifying kernels (see Theorem \ref{app:thm_mollifiers}), we get the following properties: for all
$k\in\N$, one has
\begin{align*}
\left\|Z_{\veps,M}\right\|_{L^\infty_T(H^k(\wtilde\Omega_{\veps,M}))}\,+\,\left\|\vec V_{\veps,M}\right\|_{L^2_T(H^k(\wtilde\Omega_{\veps,M}))}\,\leq\,C(k,M) \\
\left\|X_{\veps,M}\right\|_{L^2_T(H^k(\wtilde\Omega_{\veps,M}))}\,+\,\left\|\vec Y_{\veps,M}\right\|_{L^2_T(H^k(\wtilde\Omega_{\veps,M}))}\,\leq\,C(k,M)\,,
\end{align*}
for some positive constants $C(k,M)$, only depending on the fixed $k$ and $M$. Of course, the constants blow up when $M\ra+\infty$, but they are uniform for $\veps\in\,]0,1]$.

We have the following statement, analogous to Proposition \ref{p:prop dec} above. Its proof is also similar, hence omitted. In addition, we notice that the strong convergence follows
from standard properties of the mollifying kernel.
\begin{proposition} \label{p:dec_1}
For any $M>0$ and any $\veps\in\,]0,1]$, we have
\begin{equation*}
\vec{V}_{\veps ,M}\,=\,
\veps\,\vec{v}_{\veps ,M}^{(1)}\,+\,\vec{v}_{\veps ,M}^{(2)}\,,
\end{equation*}
together with the following bounds: for any $T>0$, any compact set $K\subset\wtilde\Omega$ and any $s\in\N$, one has 
(for $\veps>0$ small enough, depending only on $K$) 
\begin{align*}
\left\|\vec{v}_{\veps ,M}^{(1)}\right\|_{L^{2}([0,T];H^{s}(K))}\,\leq\,C(K,s,M) \qquad\qquad\mbox{ and }\qquad\qquad
\left\|\vec{v}_{\veps ,M}^{(2)}\right\|_{L^{2}([0,T];H^{1}(K))}\,\leq\,C(K)\,,
\end{align*}
for suitable positive constants $C(K,s,M)$ and $C(K)$ depending only on the quantities in the brackets, but uniform with respect to $\veps\in\,]0,1]$.
\end{proposition}

In particular, we deduce the following fact: for any $T>0$ and any compact $K\subset\wtilde\Omega$, there exist $\veps_K>0$ and $M_K\in\N$ such that, for all $\veps\in\,]0,\veps_K]$ and all $M\geq M_K$,
there are positive constants $C(K)$ and $C(K,M)$ for which
\begin{equation} \label{est:V_e-M_conv}
\left\|\vec V_\veps\,-\,\vec V_{\veps,M}\right\|_{L^2_T(L^2(K))}\,\leq\,C(K,M)\,\veps\,+\,C(K)\,2^{-M}\,.
\end{equation}

\subsubsection{Finite propagation speed and consequences}
In this paragraph we show that, for the scopes of our study, we can safely assume that system \eqref{eq:reg-wave} is set in the whole $\wtilde\Omega$ and it is supplemented
with compactly supported initial data and external forces.

Take smooth initial data $\mc Z_0$ and $\vec{\mc V_0}$ and forces $\mf X$ and $\vec{\mc Y}$. Consider, in $\R_+\times\wtilde\Omega$, the wave system
\begin{equation} \label{eq:wave_Omega}
\left\{\begin{array}{l}
       \veps\,\d_t\mc Z\,+\,\mc{A}\,\div\vec{\mc V}\,=\,\veps\,\mf X \\[1ex]
       \veps\,\d_t\vec{\mc V}\,+\,\nabla_x\mc Z\,+\,\,\e_3\times \vec{\mc V}\,=\,\veps\,\vec{\mc Y}\,, 
       \end{array}
\right.
\end{equation}
supplemented with initial data  $\mc Z_{|t=0}\,=\,\mc Z_0$ and $\vec{\mc V}_{|t=0}\,=\,\vec{\mc V_0}$.

System \eqref{eq:wave_Omega} is a symmetrizable (in the sense of Friedrichs) first-order
hyperbolic system with a skew-symmetric $0$-th order term. Therefore, 
classical arguments based on energy methods (see e.g. Chapter 3 of \cite{M-2008} and Chapter 7 of \cite{Ali}) allow to establish the properties of finite propagation speed and domain of dependence for solutions to \eqref{eq:wave_Omega}.

Namely, set $\lambda\,:=\,\sqrt{\mc A}/\veps$ to be the propagation speed of acoustic-Poincar\'e waves.
Let $\mf B$ be a cylinder 
included in $\wtilde\Omega$.
Then one has the following two properties.
\begin{enumerate}[(i)]
 \item \emph{Domain of dependence}: assume that 
$$
\Supp\mc Z_0\,,\;\Supp\vec{\mc V_0}\,\subset\,\mf B\,,\qquad\qquad\qquad \Supp\mf X(t)\,,\;\Supp\vec{\mc Y}(t)\,\subset\,\mf B\quad\mbox{ for a.a. }t\in[0,T]\,;
$$
then the corresponding solution $\big(\mc Z,\vec{\mc V}\big)$ to \eqref{eq:wave_Omega} is \emph{identically zero} outside the cone
$$
\Big\{(t,x)\in\,]0,T[\,\times\,\wtilde\Omega\;:\quad {\rm dist}\big(x,\mf B\big)\,<\,\lambda\,t\Big\}\,.
$$
\item \emph{Finite propagation speed}: define the set
$$
\mf B_{\lambda T}\,:=\,\Big\{x\in\wtilde\Omega\;:\quad {\rm dist}\big(x,\mf B\big)\,<\,\lambda\,T\Big\}
$$
and assume that
$$
\Supp\mc Z_0\,,\;\Supp\vec{\mc V_0}\,\subset\,\mf B_{\lambda T}\,,\qquad\qquad\qquad \Supp\mf X(t)\,,\;\Supp\vec{\mc Y}(t)\,\subset\,\mf B_{\lambda T}\quad\mbox{ for a.a. }t\in[0,T]\,;
$$
then the solution $\big(\mc Z,\vec{\mc V}\big)$ is \emph{uniquely determined} by the data inside the cone
$$
\mc C_{\lambda T}\,:=\,\Big\{(t,x)\in\,]0,T[\,\times\mf B_{\lambda T}\;:\quad {\rm dist}\big(x,\d\mf B_{\lambda T}\big)\,>\,\lambda\,t\Big\}\,,
$$
and in particular in the space-time cylinder $\,]0,T[\,\times\,\mf B$.
\end{enumerate}

\medbreak
Next, fix any test-function $\vec\psi\in C^\infty_c\big(\R_+\times\wtilde\Omega;\R^3\big)$, and let $T>0$ and the compact set $K\subset\wtilde\Omega$ be such that $\Supp\vec\psi\subset[0,T[\,\times K$.
Take a cylindrical neighborhood $\mf B$ of $K$ in $\wtilde\Omega$. 
It goes without saying that there exist an $\veps_0=\veps_0(\mf B)\in\,]0,1]$ and a $M_0=M_0(\mf B)\in\N$ such that
\begin{equation}\label{cyl-neigh}
\oline{\mf B}\,\subset\subset\,\wtilde\Omega_{\veps,M}\qquad\qquad \mbox{ for all }\qquad 0<\veps\leq\veps_0\quad\mbox{ and }\quad M\geq M_0\,,
\end{equation}
where the set $\wtilde\Omega_{\veps,M}$ has been defined in \eqref{def:O_e-M} above. Take now a cut-off function $\mf h\in C^\infty_c(\wtilde\Omega)$ such that
$\mf h\equiv1$ on $\mf B$ (and hence on $K$), and solve problem \eqref{eq:wave_Omega} with compactly supported data
$$
\mc Z_0\,=\,\mf h\,Z_{0,\veps,M}\,,\qquad \vec{\mc V_0}\,=\,\mf h\,\vec V_{0,\veps,M}\,,\qquad
\mf X\,=\,\mf h\,X_{\veps,M}\,,\qquad \vec{\mc Y}\,=\,\mf h\,\vec{Y}_{\veps,M}\,.
$$
We point out that all the data are now localised around the compact set $K$.
Owing to assumption \eqref{dom}, the domains $\wtilde{\Omega}_{\veps,M}$ are expanding at speed proportional to $\veps^{-(1+\de)}$, whereas, in view of finite propagation speed,
the support of the solution is expanding at speed proportional to $\veps^{-1}$ (keep in mind also Remark \ref{r:speed-waves}). Thus, thanks to the inclusion \eqref{cyl-neigh},
the previous discussion implies that, up to take a smaller $\veps_0$, for any $\veps\leq\veps_0$ the corresponding solution
$\big(\mc Z,\vec{\mc V}\big)$ of \eqref{eq:wave_Omega}
has support inside a cylinder $\wtilde\B_L\,:=\,B_L(0)\times\T\subset\wtilde\Omega_\veps$, for some $L=L(T,K,\lambda)>0$, and it must coincide with the solution $\big(Z_{\veps,M},\vec V_{\veps,M}\big)$
of \eqref{eq:reg-wave} on the set $\,]0,T[\,\times\,\mf B$, for all $0<\veps\leq \veps_0$ and all $M\geq M_0$.
In particular, for all $0<\veps\leq \veps_0$ and all $M\geq M_0$ we have
$$
\mc Z\,\equiv\,Z_{\veps,M}\quad\mbox{ and }\quad \vec{\mc V}\,\equiv\,\vec V_{\veps,M}\qquad\qquad\mbox{ on }\qquad \Supp\vec\psi\,.
$$

The previous argument shows that, without loss of generality, we can assume that the regularised wave system \eqref{eq:reg-wave} is verified on the whole $\wtilde\Omega$,
with compactly supported initial data and forces, and with solutions supported on some cylinder $\wtilde\B_L$.
In particular, we can safely work with system \eqref{eq:reg-wave} and its smooth solutions $\big(Z_{\veps,M},\vec V_{\veps,M}\big)$ in the computations below.

\subsection{Convergence of the convective term} \label{ss:convergence_1}
Here we tackle the convergence of the convective term, employing again a compensated compactness argument. The first step is to reduce the study to the case of smooth vector fields $\vec{V}_{\veps ,M}$.
Arguing as in Lemma \ref{lem:convterm}, and using Proposition \ref{p:dec_1} and property \eqref{est:V_e-M_conv}, one can easily prove the following approximation result.
Again, the proof is omitted.

\begin{lemma} \label{lem:convterm_1}
Let $T>0$. For any $\vec{\psi}\in C_c^\infty\bigl([0,T[\,\times\wtilde\Omega;\R^3\bigr)$, we have 
\begin{equation*}
\lim_{M\rightarrow +\infty} \limsup_{\veps \rightarrow 0^+}\left|\int_{0}^{T}\int_{\wtilde\Omega} \vrho_\veps\,\vec{u}_\veps\otimes \vec{u}_\veps: \nabla_{x}\vec{\psi}\, \dxdt-
\int_{0}^{T}\int_{\wtilde\Omega} \vec{V}_{\veps ,M}\otimes \vec{V}_{\veps,M}: \nabla_{x}\vec{\psi}\,\ dxdt\right|=0\, .
\end{equation*}
\end{lemma}


\medbreak
Assume now $\vec\psi\in C_c^\infty\big([0,T[\,\times\wtilde\Omega;\R^3\big)$ such that $\div\vec\psi=0$ and $\d_3\vec\psi=0$.
Thanks to the previous lemma, it is enough to pass to the limit in the smooth term 
\begin{align*}
-\int_{0}^{T}\int_{\wtilde\Omega} \vec{V}_{\veps ,M}\otimes \vec{V}_{\veps ,M}: \nabla_{x}\vec{\psi}\,&=\,
\int_{0}^{T}\int_{\wtilde\Omega}\div\left(\vec{V}_{\veps ,M}\otimes \vec{V}_{\veps ,M}\right) \cdot \vec{\psi}\,=\,
\int_{0}^{T}\int_{\R^2} \left(\mc{T}_{\veps ,M}^{1}+\mc{T}_{\veps, M}^{2}\right)\cdot\vec{\psi}^h\,,
\end{align*}
where, for simplicity, we agree that the torus $\T$ has been normalised so that its Lebesgue measure is equal to $1$ and, analogously to \eqref{def:T1-2}, we have introduced the quantities
$$ 
\mc T^1_{\veps,M}\,:=\, \divh\left(\langle \vec{V}_{\veps ,M}^{h}\rangle\otimes \langle \vec{V}_{\veps ,M}^{h}\rangle\right)\qquad \mbox{ and }\qquad
\mc T^2_{\veps,M}\,:=\, \divh\left(\langle \dbtilde{\vec{V}}_{\veps ,M}^{h}\otimes \dbtilde{\vec{V}}_{\veps ,M}^{h}\rangle \right)\,.
$$ 

We notice that the analysis of $\mc{T}_{\veps ,M}^{2}$ is similar to the one performed in Paragraph \ref{sss:term2}, up to taking $m=1$ and replacing $\vec{W}_{\veps ,M}$ and
$\Lambda_{\veps ,M}$ by $\vec{V}_{\veps ,M}$ and $Z_{\veps ,M}$ respectively.
Indeed, it deeply relies on system \eqref{eq:eq momentum term2}, which remains unchanged when $m=1$.
Also in this case, we find \eqref{eq:limit_T2}.

Therefore, we can focus on the term $\mc{T}_{\veps ,M}^{1}$ only. Its study presents some differences with respect to Paragraph \ref{sss:term1}, so let us give the full details.
To begin with, like in \eqref{eq:T1}, we have
\begin{equation*}
\mc{T}_{\veps ,M}^{1}\,=\,
\divh\langle \vec{V}_{\veps ,M}^{h}\rangle\;\; \langle \vec{V}_{\veps ,M}^{h}\rangle+\frac{1}{2}\, \nabla_{h}\left(\left|\langle \vec{V}_{\veps ,M}^{h}\rangle\right|^{2}\right)+
\curlh\langle \vec{V}_{\veps ,M}^{h}\rangle\;\;\langle \vec{V}_{\veps ,M}^{h}\rangle^{\perp}\,.
\end{equation*}

Of course, we can forget about the second term, because it is a perfect gradient.
For the first term, we use system \eqref{eq:reg-wave}: averaging the first equation with respect to $x^{3}$ and multiplying it by $\langle \vec{V}^h_{\veps ,M}\rangle$, we get
\begin{equation*}
\divh\langle \vec{V}_{\veps ,M}^{h}\rangle\;\;\langle \vec{V}_{\veps ,M}^{h}\rangle\,=\,-\frac{\veps}{\mc{A}}\d_t\langle Z_{\veps ,M}\rangle \langle \vec{V}_{\veps ,M}^{h}\rangle+
\frac{\veps}{\mc{A}} \langle X_{\veps ,M}\rangle \langle \vec{V}_{\veps ,M}^{h}\rangle\,=\,
\frac{\veps}{\mc{A}}\langle Z_{\veps ,M}\rangle \d_t \langle \vec{V}_{\veps ,M}^{h}\rangle +\mc{R}_{\veps ,M}\,.
\end{equation*}
We now use the horizontal part of \eqref{eq:reg-wave}, multiplied by $\langle Z_{\veps ,M}\rangle$, and we gather
\begin{equation*}
\begin{split}
\frac{\veps}{\mc{A}}\langle Z_{\veps ,M}\rangle \d_t \langle \vec{V}_{\veps ,M}^{h}\rangle &=-\frac{1}{\mc{A}} \langle Z_{\veps ,M}\rangle \nabla_{h}\langle Z_{\veps ,M}\rangle-
\frac{1}{\mc{A}}\langle Z_{\veps ,M}\rangle\langle \vec{V}_{\veps ,M}^{h}\rangle^{\perp}+\frac{\veps}{\mc{A}}\langle Z_{\veps ,M}\rangle \langle \vec Y_{\veps ,M}^{h}\rangle\\
&=-\frac{1}{\mc{A}}\langle Z_{\veps ,M}\rangle\langle \vec{V}_{\veps ,M}^{h}\rangle^{\perp}+\mc{R}_{\veps ,M}\, .
\end{split}
\end{equation*}
This latter relation yields that
\begin{equation*}
\mc{T}_{\veps ,M}^{1}\,=\,\left(\curlh\langle \vec{V}_{\veps ,M}^{h}\rangle-\frac{1}{\mc{A}}\langle Z_{\veps ,M}\rangle \right)\langle \vec{V}_{\veps ,M}^{h}\rangle^{\perp}+\mc{R}_{\veps ,M} .
\end{equation*}

Now we use the horizontal part of \eqref{eq:reg-wave}: 
averaging it with respect to the vertical variable and
applying the operator $\curlh$, we find
\begin{equation*}
\veps\,\d_t\curlh\langle \vec{V}_{\veps ,M}^{h}\rangle\,+\,\divh\langle \vec{V}_{\veps ,M}^{h}\rangle \,=\,\veps\, \curlh\langle \vec Y_{\veps ,M}^{h}\rangle\, .
\end{equation*}
Taking the difference of this equation with the first one in \eqref{eq:reg-wave}, we discover that
\begin{equation*}
\d_t\g_{\veps,M}
\,=\,\curlh\langle \vec Y_{\veps ,M}^{h}\rangle\,-\,\frac{1}{\mc{A}}\,\langle X_{\veps ,M}\rangle\,,\qquad\qquad \mbox{ with }\qquad
\gamma_{\veps, M}:=\curlh\langle \vec{V}_{\veps ,M}^{h}\rangle\,-\,\frac{1}{\mc{A}}\langle Z_{\veps ,M}\rangle\,.
\end{equation*}
An argument analogous to the one used after \eqref{eq:gamma} above, based on Aubin-Lions Theorem, shows also in this case that
$(\gamma_{\veps,M})_{\veps}$ is compact in $L_{T}^{2}(L_{\rm loc}^{2})$. Then, this sequence converges strongly (up to extraction of a suitable subsequence)
to a tempered distribution $\gamma_M$ in the same space. 

Since $\gamma_{\veps ,M}\longrightarrow \gamma_M$ strongly in $L_{T}^{2}(L_{\rm loc}^{2})$ and
$\langle \vec{V}_{\veps ,M}^{h}\rangle\rightharpoonup\langle \vec{V}_{M}^{h}\rangle$ weakly in $L_{T}^{2}(L_{\rm loc}^{2})$, we deduce that
\begin{equation*}
\gamma_{\veps,M}\,\langle \vec{V}_{\veps ,M}^{h}\rangle^{\perp}\,\longrightarrow\, \gamma_M\, \langle \vec{V}_{M}^{h}\rangle^{\perp}\qquad \text{ in }\qquad \mc{D}^{\prime}\big(\R_+\times\R^2\big),
\end{equation*}
where $\langle \vec{V}_{M}^{h}\rangle=\lan{\omega}_{M}*\vec{U}^{h}\ran$ and $ \gamma_M=\curlh\lan{\omega}_{M}*\vec{U}^{h}\ran-(1/\mc{A})\langle Z_{M}\rangle$.
Notice that, in view of \eqref{conv:rr}, \eqref{conv:theta}, \eqref{est:sigma}, Proposition \ref{p:prop_5.2} and the definitions given in \eqref{relnum}, we have
$$
Z_M\,=\,\d_\vrho p(1,\oline\vtheta)\,\omega_M*\vrho^{(1)}\,+\,\d_\vtheta p(1,\oline\vtheta)\,\omega_M*\Theta\,=\,\omega_M*q\,,
$$
where $q$ is the quantity defined in \eqref{eq:for q}. 
Owing to the regularity of the target velocity $\vec U^h$, we can pass to the limit also for $M\ra+\infty$, thus finding that
\begin{equation} \label{eq:limit_T1-1}
\int^T_0\!\!\!\int_{\wtilde\Omega}\vrho_\veps\,\vec{u}_\veps\otimes \vec{u}_\veps: \nabla_{x}\vec{\psi}\, \dxdt\,\longrightarrow\,
\int^T_0\!\!\!\int_{\R^2}\big(\vec U^h\otimes\vec U^h:\nabla_h\vec\psi^h\,+\,\frac{1}{\mc A}\,q\,(\vec U^h)^\perp\cdot\vec\psi^h\big)\dx^h\dt\, ,
\end{equation}
for all test functions $\vec\psi$ such that $\div\vec\psi=0$ and $\d_3\vec\psi=0$. Recall the convention $|\T|=1$.
Notice that, since $\vec U^h=\nabla_h^\perp q$, the last term in the integral on the right-hand side is actually zero.

\subsection{End of the proof} \label{ss:limit_1}
Thanks to the previous analysis, we are now ready to pass to the limit in equation \eqref{weak-mom}.
As done above, we take a test-function $\vec\psi$ such that
$$
\vec{\psi}=\big(\nabla_{h}^{\perp}\phi,0\big)\,,\qquad\qquad\mbox{ with }\qquad \phi\in C_c^\infty\big([0,T[\,\times\R^2\big)\,,\quad \phi=\phi(t,x^h)\,.
$$
Notice that $\div\vec\psi=0$ and $\d_3\vec\psi=0$. Then, all the gradient terms and all the contributions coming from the vertical
component of the momentum equation vanish identically, when tested against such a $\vec\psi$. In particular, we have 
$$ \int_0^T\int_{\Omega}\frac{1}{\veps}\vrho_\veps \nabla_x G\cdot \vec \psi\dxdt\equiv 0\, . $$

So, equation \eqref{weak-mom} reduces
to\footnote{Remark that, in view of our choice of the test-functions, we can safely come back to the notation on $\Omega$ instead of $\wtilde\Omega$.}
\begin{align*}
\int_0^T\!\!\!\int_{\Omega}  \left( -\vre \ue \cdot \partial_t \vec\psi -\vre \ue\otimes\ue  : \nabla \vec\psi
+ \frac{1}{\ep}\vre\big(\ue^{h}\big)^\perp\cdot\vec\psi^h+\mbb{S}(\vtheta_\veps,\nabla_x\vec\ue): \nabla_x \vec\psi\right)
 =\int_{\Omega}\vrez \uez  \cdot \vec\psi(0,\cdot)\,.
\end{align*}

As done in Subsection \ref{ss:limit}, we can limit ourselves to consider the rotation and convective terms only.
As for the former term, we start by using the mass equation in \eqref{ceq} tested against $\phi$: we get (recalling also \eqref{def_deltarho})
\begin{equation*} 
\begin{split}
-\int_0^T\!\!\!\int_{\R^2} \left( \lan R_{\varepsilon}\ran\, \d_{t}\phi +\frac{1}{\veps}\, \lan\vrho_{\veps}\ue^{h}\ran\cdot \nabla_{h}\phi\right)=
\int_{\R^2}\lan R_{0,\varepsilon }\ran\, \phi (0,\cdot ) \, ,
\end{split}
\end{equation*}
whence we deduce that
\begin{align*}
\int_0^T\!\!\!\int_{\Omega}\frac{1}{\ep}\vre\big(\ue^{h}\big)^\perp\cdot\vec\psi^h\,&=\,\int_0^T\!\!\!\int_{\mbb{R}^2}\frac{1}{\ep}\langle\vre \ue^{h}\rangle \cdot \nabla_{h}\phi\, \\
&=\,-\,\int_0^T\!\!\!\int_{\mbb{R}^2}\langle R_\veps\rangle\, \d_t\phi\,-\,\int_{\mbb{R}^2}\langle R_{0,\veps}\rangle\, \phi(0,\cdot )\,. 
\end{align*}

Letting now $\varepsilon \rightarrow 0^+$, thanks to the previous relation and \eqref{eq:limit_T1-1}, we finally gather
\begin{align*}
&-\int_0^T\!\!\!\int_{\R^2} \left(\vec{U}^{h}\cdot \d_{t}\nabla_{h}^{\perp} \phi+ \vec{U}^{h}\otimes \vec{U}^{h}:\nabla_{h}(\nabla_{h}^{\perp}\phi )+\lan R\ran\, \d_t \phi \right)\, \dx^h\dt\\
&=-\int_0^T\!\!\!\int_{\R^2} \mu (\oline\vartheta )\nabla_{h}\vec{U}^{h}:\nabla_{h}(\nabla_{h}^{\perp}\phi ) \, \dx^h\, \dt+\int_{\R^2}\left(\lan\vec{u}_{0}^{h}\ran\cdot \nabla _{h}^{\perp}\phi (0,\cdot )+
\lan R_{0}\ran\, \phi (0,\cdot )\right) \, \dx^h\, .
\end{align*}
Now, keeping in mind the relation for $R$ in Remark \ref{r:limit_1}, we have 
\begin{equation*}
\begin{split}
R\,&=\,-\,\frac{1}{\beta}\,\big(\d_\vtheta p(1,\oline\vtheta)\,\Upsilon\,-\,\d_\vtheta s(1,\oline\vtheta)\,q\,-\,\d_\vtheta s(1,\oline\vtheta)\,G\big)\\
&=-\frac{1}{\mc A}\left(\mc B \Upsilon-q-G\right)\, ,
\end{split}
\end{equation*}
where we have also employed the definitions of $\mc A$ and $\mc B$ in \eqref{relnum}. 

Thus, we can write 
\begin{equation*}
\begin{split}
-\int_0^T\!\!\!\int_{\R^2} \lan R\ran\, \d_t \phi \, \dx^h\dt&=\frac{1}{\mc A}\int_0^T\!\!\!\int_{\R^2}\left(\mc B \, \lan \Upsilon \ran -q+\frac{1}{2}\right) \, \d_t \phi \, \dx^h\dt\\
&=\frac{1}{\mc A}\int_0^T\!\!\!\int_{\R^2}\Big(\mc B \, \lan \Upsilon \ran -q\Big) \, \d_t \phi \, \dx^h\dt- \int_{\R^2}\frac{1}{2\mc A}\, \phi (0,\cdot )\dx^h\, .
\end{split}
\end{equation*}

At the end, noticing that $\Upsilon$ solves (in the sense of distributions) the transport-diffusion equation \eqref{eq_lim:transport}, we get
\begin{align*}
&-\int_0^T\!\!\!\int_{\R^2} \left(\vec{U}^{h}\cdot \d_{t}\nabla_{h}^{\perp} \phi+ \vec{U}^{h}\otimes \vec{U}^{h}:\nabla_{h}(\nabla_{h}^{\perp}\phi )+\frac{1}{\mc A}\, q\, \d_t \phi \right)\, \dx^h\dt\\
&=-\int_0^T\!\!\!\int_{\R^2} \mu (\oline\vartheta )\nabla_{h}\vec{U}^{h}:\nabla_{h}(\nabla_{h}^{\perp}\phi ) \, \dx^h\, \dt+\frac{1}{\mc A}\int_0^T\int_{\R^2} \lan X \ran \, \phi \dx^h \dt\\
&+\int_{\R^2}\left(\lan\vec{u}_{0}^{h}\ran\cdot \nabla _{h}^{\perp}\phi (0,\cdot )+
\left(\lan R_{0}\ran+\frac{1}{2\mc A}\right)\phi (0,\cdot )\right) \, \dx^h\, ,
\end{align*}
where we have defined $X$ as in \eqref{def:X}. 
 
Theorem \ref{th:m=1_F=0} is finally proved.



\chapter{On the influence of gravity}\label{chap:BNS_gravity}

\begin{quotation}
In this chapter, we continue the investigation we began in Chapter \ref{chap:multi-scale_NSF}, regarding the multi-scale analysis of mathematical models for geophysical flows.
Our focus here is on the effect of gravity in regimes of \emph{low stratification}, but which go beyond the choice of the scaling that, in light
of previous results, we call ``critical''. For clarity of exposition, we consider the barotropic Navier-Stokes system with the Coriolis force, i.e. 
\begin{equation}\label{chap3:BNSC}
\begin{cases}
\partial_t \vrho + \div (\vrho\vec{u})=0\  \\[2ex]
	\partial_t (\vrho\vec{u})+ \div(\vrho\vec{u}\otimes\vec{u}) + \dfrac{\e_3 \times \vrho\vec{u}}{Ro}\,  +    \dfrac{1}{Ma^2} \nabla_x p(\vrho) 
=\div \mbb{S}(\nabla_x\vec{u})  + \dfrac{\vrho}{Fr^2} \nabla_x G\, ,
\end{cases}
\end{equation}
where we will take $Ro=\veps$, $Ma=\veps^m$ and $Fr=\veps^n$ with $m$ and $n$ in suitable ranges (see \eqref{eq:scale-our} below in this respect). 

The results presented in this chapter are contained in \cite{DS-F-S-WK_sub}.

\medbreak

Before moving on, let us give a recapitulation of the contents of chapter.

In Section \ref{s:result_G} we collect our assumptions and we state the main theorems.
In Section \ref{s:energy} we show the main consequences of the finite energy condition on the family of weak solutions we are going to consider.
Namely, we derive uniform bounds in suitable norms, which allow us to extract weak-limit points, and we explore the constraints those limit points have to satisfy. In Sections \ref{s:proof_G} and \ref{s:proof-1_G}, we complete the proof of our main results, showing convergence in the weak formulation of the equations
in the cases $m>1$ and $m=1$, respectively. 
\end{quotation}


\section{Setting of the problem and main results} \label{s:result_G}

In this section, we introduce the primitive system and formulate our working framework (see Subsection \ref{ss:FormProb_G}), then we state our main results
(in Subsection \ref{ss:results_G}).

 \subsection{The barotropic Navier-Stokes-Coriolis system} \label{ss:FormProb_G}

As already said in the introductory part, in this chapter we assumed that the motion of the fluid is described by system \eqref{chap2:syst_NSFC} with constant density and without the centrifugal force.

Thus, given a small parameter $\veps\in\,]0,1]$, the barotropic Navier-Stokes system with Coriolis and gravitational forces (see system \eqref{chap3:BNSC} in this respect) reads as follows:
\begin{align}
&	\partial_t \vre + \div (\vre\ue)=0 \label{ceq_G}\tag{NSC$_{\ep}^1$} \\
&	\partial_t (\vre\ue)+ \div(\vre\ue\otimes\ue) + \frac{1}{\ep}\,\e_3 \times \vre\ue +    \frac{1}{\ep^{2m}} \nabla_x p(\vre) 
	=\div \mbb{S}(\nabla_x\ue)  + \frac{\vre}{\ep^{2n}} \nabla_x G\, ,
	\label{meq_G}\tag{NSC$_{\ep}^2$} 
	\end{align}
where we recall that $m$ and $n$ are taken 
\begin{equation}\label{eq:scale-our}
\mbox{ either }\qquad m\,>\,1\quad\mbox{ and }\quad m\,<\,2\,n\,\leq\,m+1\,,\qquad\qquad\mbox{ or }\qquad
m\,=\,1\quad\mbox{ and }\quad \frac{1}{2}\,<\,n\,<\,1\,.
\end{equation}
As in the previous Chapter \ref{chap:multi-scale_NSF}, here the unknowns in equations \eqref{ceq_G}-\eqref{meq_G} are the density $\vre=\vre(t,x)\geq0$ of the fluid and its velocity field $\ue=\ue(t,x)\in\R^3$, where $t\in\R_+$ but now $x\in \Omega:=\R^2 \times\; ]0,1[\,$.
The viscous stress tensor in \eqref{meq_G} is given by Newton's rheological law, that we recall here, 
	\begin{equation}\label{S_G}
	\mbb{S}(\nabla_x \ue) = \mu\left( \nabla_x\ue + \,^t\nabla_x \ue  - \frac{2}{3}\div \ue\,  \Id \right)
	+ \eta\, \div\ue \, \Id\,,
	\end{equation}
where $\mu>0$ and $\eta\geq 0$ now don't depend on the temperature variations. 
As for the gravitational force, 
we recall its formulation (see \eqref{assFG} in this regard):
	\begin{equation}\label{assG}
	 G(x)= -x^3\,.
	\end{equation}
The precise expression of $G$ will be useful in some computations below, 
although some generalisations are certainly possible.

	
As done in the previous Chapter \ref{chap:multi-scale_NSF}, the system is supplemented  with \emph{complete slip boundary conditions}, namely 
	\begin{align}
	\big(\ue \cdot \n\big) _{|\partial \Omega} = 0
	\quad &\mbox{ and } \quad
	\bigl([ \mbb{S} (\nabla_x \ue) \n ] \times \n\bigr)_{|\d\Omega} = 0\,,  \label{bc1-2_G}
	\end{align}
where $\vec{n}$ denotes the outer normal to the boundary $\d\Omega\,=\,\{x_3=0\}\cup\{x_3=1\}$.
Notice that this is a true simplification, because it avoids complications due to the presence of Ekman boundary layers, when passing to the limit
$\veps\ra0^+$.

\begin{remark} \label{r:period-bc}
As it is well-known (see e.g. Subsection \ref{sss:w-reg} and \cite{Ebin}), the equations \eqref{ceq_G}-\eqref{meq_G}, supplemented by the complete slip boundary boundary conditions \eqref{bc1-2_G},
can be recasted as a periodic problem with respect to the vertical variable, in the new domain
$$
\Omega\,=\,\R^2\,\times\,\mbb{T}^1\,,\qquad\qquad\mbox{ with }\qquad\mbb{T}^1\,:=\,[-1,1]/\sim\,,
$$
where $\sim$ denotes the equivalence relation which identifies $-1$ and $1$. Indeed, the equations are invariant if we extend
$\vrho$ and $\vec u^h$ as even functions with respect to $x^3$, and $u^3$ as an odd function.

In what follows, we will always assume that such modifications have been performed on the initial data, and
that the respective solutions keep the same symmetry properties.
\end{remark}

Now we need to  impose structural restrictions on the pressure function $p$. We assume that
	\begin{equation}\label{pp1_G}
	p\in C^1 [0,\infty)\cap C^2(0,\infty),\qquad p(0)=0,\qquad p'(\vrho )>0\quad \mbox{ for all }\,\vrho\geq 0\, .
	\end{equation}
Additionally to \eqref{pp1_G}, we require that (remember also Remark \ref{rmk:pressure_choice}) 
	\begin{equation}\label{pp2_G}
\text{there exists}\quad \;\g\,>\,\frac{3}{2}\quad\mbox{ such that }\qquad
\lim\limits_{\vrho \to +\infty} \frac{p^\prime(\vrho)}{\vrho^{\gamma -1}} = p_\infty >0\, .
	\end{equation}
Without loss of generality, we can suppose that $p$ has been renormalised so that $p^\prime (1)=1$.  

\begin{remark}
For a more detailed discussion about the choice $\gamma>\oline \gamma:=3/2$, which is fundamental for the existence theory, we address the reader to \cite{F-N-P} by Feireisl, Novotný and Petzeltová, and references therein. In particular, we remark that in two space dimensions $\oline \gamma$ can be improved up to $1$.  
\end{remark}


\subsubsection{Equilibrium states} \label{sss:equilibrium_G}

Next, we focus our attention on the so-called \emph{equilibrium states}. For each value of $\veps\in\,]0,1]$ fixed, the equilibria of system \eqref{ceq_G}-\eqref{meq_G} consist of static densities $\vret$ satisfying
	\begin{equation}\label{prF_G}
\nabla_x p(\vret) = \ep^{2(m-n)} \vret \nabla_x G  \qquad \mbox{ in }\; \Omega\,.
	\end{equation}	

%
%

Equation \eqref{prF_G} identifies $\wtilde{\vrho}_\veps$ up to an additive constant: taking the target density to be $1$, we get
\begin{equation} \label{eq:target-rho_G}
  H^\prime(\vret)=\, \ep^{2(m-n)} G + H^\prime (1)\,,\qquad\qquad \mbox{ where }\qquad 
H(\vrho) = \vrho \int_1^{\vrho} \frac{ p(z)}{z^2} {\rm d}z\,.
\end{equation}
Notice that relation \eqref{eq:target-rho_G} implies that 
\begin{equation*}
H^{\prime \prime}(\vrho)=\frac{p^\prime (\vrho)}{\vrho} \quad \text{ and }\quad H^{\prime \prime}(1)=1\, .
\end{equation*}

Therefore, we infer that, whenever $m\geq1$ and $m>n$ as in the present chapter, for any $x\in\Omega$ one has $\wtilde{\vrho}_\veps(x)\longrightarrow 1$ in the limit $\veps\ra0^+$.
More precisely, the next statement collects all the necessary properties of the static states. It corresponds to Lemma \ref{l:target-rho_pos} and Proposition \ref{p:target-rho_bound} of
Chapter \ref{chap:multi-scale_NSF}.
\begin{proposition} \label{p:target-rho_bound_G}
Let the  gravitational force $G$ be given by \eqref{assG}.
Let $\bigl(\wtilde{\vrho}_\veps\bigr)_{0<\veps\leq1}$ be a family of static solutions to equation \eqref{prF_G}
in $\Omega$.

Then, there exist an $\veps_0>0$ and a $0<\rho_*<1$ such that $\wtilde{\vrho}_\veps\geq\rho_*$ for all $\veps\in\,]0,\veps_0]$
and all $x\in\Omega$.
In addition, for any $\veps\in\,]0,\veps_0]$, one has:
\begin{equation*}
\left|\wtilde{\vrho}_\veps(x)\,-\,1\right|\,\leq\,C\,\veps^{2(m-n)}\, ,
\end{equation*}
for a constant $C>0$ which is uniform in $x\in\Omega$ and in $\veps\in\,]0,1]$.
\end{proposition}

Without loss of any generality, we can assume that $\veps_0=1$ in Proposition \ref{p:target-rho_bound_G}.

\medbreak
In light of this analysis, it is natural to try to solve system \eqref{ceq_G}-\eqref{meq_G} in $\Omega$, supplemented with the \emph{far field conditions}
\begin{equation} \label{ff}
\varrho_{\varepsilon}\longrightarrow \vret \qquad \mbox{ and } \qquad \ue \longrightarrow 0 \qquad\qquad \text{ as }\quad |x|\rightarrow +\infty \, .
\end{equation}


\subsubsection{Initial data and finite energy weak solutions} \label{sss:data-weak_G}

In view of the boundary conditions \eqref{ff} ``at infinity'', we assume that the initial data are close (in a suitable sense) to the equilibrium states $\vret$ that we have just identified.
Namely, we consider initial densities of the following form:
	\begin{equation}\label{in_vr_G}
	\vrez = \vret + \ep^m \vrez^{(1)}. 
	\end{equation}
For later use, let us introduce also the following decomposition of the initial densities: 
\begin{equation} \label{eq:in-dens_dec_G}
\vrho_{0,\veps}\,=\,1\,+\,\veps^{2(m-n)}\,R_{0,\veps}\qquad\qquad\mbox{ with }\qquad
R_{0,\veps}\,=\,\wtilde r_\veps\,+\,\veps^{2n-m}\, \vrho_{0,\veps}^{(1)}\,,\qquad \wtilde r_\veps\,:=\,\frac{\wtilde\vrho_\veps-1}{\veps^{2(m-n)}}\,,
\end{equation}
where again $\wtilde r_\veps$ is a datum of the system.

We suppose the density perturbations $\vrez^{(1)}$ to be measurable functions and satisfy the control
	\begin{align}
\sup_{\veps\in\,]0,1]}\left\|  \vrez^{(1)} \right\|_{(L^2\cap L^\infty)(\Omega)}\,\leq \,c\,,\label{hyp:ill_data_G}
	\end{align}
together with the ``mean-free condition''
$$
\int_{\Omega}  \vrez^{(1)} \dx = 0\,.
$$
As for the initial velocity fields, we assume the following uniform bound
\begin{equation} \label{hyp:ill-vel_G}
 	\sup_{\veps\in\,]0,1]}\left\| \sqrt{\wtilde\vrho_\veps} \vec{u}_{0,\ep} \right\|_{L^2(\Omega)}\,  \leq\, c\, \quad \text{which also implies}\quad \sup_{\veps\in\,]0,1]}\,\left\| \vec{u}_{0,\ep}  \right\|_{L^2(\Omega)}\,\leq\,c\,.
\end{equation}



Thanks to the previous uniform estimates, up to extraction, we can identify the limit points
\begin{align} 
\vrho^{(1)}_0\,:=\,\lim_{\veps\ra0}\vrho^{(1)}_{0,\veps}\qquad\text{ weakly-$\ast$ in }&\qquad L^\infty \cap L^2\label{conv:in_data_vrho}\\
\vec{u}_0\,:=\,\lim_{\veps\ra0}\vec{u}_{0,\veps}\qquad \text{ weakly in }&\qquad L^2\label{conv:in_data_vel}\,.
\end{align}

\medbreak


At this point, let us specify better what we mean by \emph{finite energy weak solution} (see \cite{F-N} for details). 
\begin{definition} \label{d:weak}
Let $\Omega=\R^2 \times \, ]0,1[$. Fix $T>0$ and $\veps>0$. Let $(\vrho_{0,\veps}, \vec u_{0,\veps})$ be an initial datum satisfying \eqref{in_vr_G} to \eqref{hyp:ill-vel_G}. We say that the couple $(\vrho_\veps, \vec u_\veps)$ is a \emph{finite energy weak solution} of the system \eqref{ceq_G}-\eqref{meq_G} in $\,]0,T[\,\times \Omega$,
supplemented with the boundary conditions \eqref{bc1-2_G} and far field conditions \eqref{ff}, related to the initial datum $(\vrho_{0,\veps}, \vec u_{0,\veps})$, if the following conditions hold:
\begin{enumerate}[(i)]
\item the functions $\vrho_\veps$ and $\ue$ belong to the class
\begin{equation*}
\vrho_\veps\geq 0\,,\;\;\; \vrho_\veps - \widetilde{\vrho}_\veps\,\in L^\infty\big(\,]0,T[\,; L^2+L^\gamma (\Omega)\big)\,,\;\;\;
\ue \in L^2\big(\,]0,T[\,;H^1(\Omega)\big),\;\;\; \big(\ue \cdot \n\big) _{|\partial \Omega} = 0;
\end{equation*}
\item the equations have to be satisfied in a distributional sense:
	\begin{equation}\label{weak-con_G}
	-\int_0^T\int_{\Omega} \left( \vre \partial_t \varphi  + \vre\ue \cdot \nabla_x \varphi \right) \dxdt = 
	\int_{\Omega} \vrez \varphi(0,\cdot) \dx
	\end{equation}
for any $\varphi\in C^\infty_c([0,T[\,\times \overline\Omega)$ and
	\begin{align}
	&\int_0^T\!\!\!\int_{\Omega}  
	\left( - \vre \ue \cdot \partial_t \vec\psi - \vre [\ue\otimes\ue]  : \nabla_x \vec\psi 
	+ \frac{1}{\ep} \, \e_3 \times (\vre \ue ) \cdot \vec\psi  - \frac{1}{\ep^{2m}} p(\vre) \div \vec\psi  \right) \dxdt \label{weak-mom_G} \\
	& =\int_0^T\!\!\!\int_{\Omega} 
	\left(- \mbb{S}(\nabla_x\vec u_\veps)  : \nabla_x \vec\psi +  \frac{1}{\ep^{2n}} \vre \nabla_x G\cdot \vec\psi \right) \dxdt 
	+ \int_{\Omega}\vrez \uez  \cdot \vec\psi (0,\cdot) \dx \nonumber
	\end{align}
for any test function $\vec\psi\in C^\infty_c([0,T[\,\times \overline\Omega; \R^3)$ such that $\big(\vec\psi \cdot \n \big)_{|\partial {\Omega}} = 0$;
\item the energy inequality holds for almost every $t\in (0,T)$:
\begin{align}
&\hspace{-0.7cm} \int_{\Omega}\frac{1}{2}\vre|\ue|^2(t) \dx\,+\,\frac{1}{\ep^{2m}}\int_{\Omega}\mc E\left(\vrho_\veps,\wtilde\vrho_\veps\right)(t) \dx
+  \int_0^t\int_{\Omega} \mbb S(\nabla_x \ue):\nabla_x \ue \, \dx {\rm d}\tau  \label{est:dissip_G} \\
&\qquad\qquad\qquad\qquad\qquad\qquad\qquad\qquad
\,\leq\,
\int_{\Omega}\frac{1}{2}\vrez|\uez|^2 \dx\,+\,
\frac{1}{\ep^{2m}}\int_{\Omega}\mc E\left(\vrho_{0,\veps},\wtilde\vrho_\veps\right) \dx\, ,
\nonumber
\end{align}
where $
\mc E\left(\rho,\wtilde\vrho_\veps\right)\,:=\,H(\rho) - (\rho - \vret)\, H^\prime(\vret)
- H(\vret)
$ is the \emph{relative internal energy} of the fluid.
\end{enumerate}
The solution is \emph{global} if the previous conditions hold for all $T>0$.
\end{definition}
	
Under the assumptions fixed above, for any \emph{fixed} value of the parameter $\veps\in\,]0,1]$,
the existence of a global in time finite energy weak solution $(\vrho_\veps,\vec u_\veps)$ to system \eqref{ceq_G}-\eqref{meq_G}, related to the initial datum
$(\vrho_{0,\veps},\vec u_{0,\veps})$, in the sense of the previous definition, can be proved as in the classical case, see e.g. \cite{Lions_2}, \cite{Feireisl}. 
Notice that the mapping $t \mapsto (\vre\ue)(t,\cdot)$ is weakly continuous, and one has $(\vre)_{|t=0} = \vrez$ together with $(\vre\ue)_{|t=0}= \vrez\uez$. 

We remark also that, in view of \eqref{ceq_G}, the total mass is conserved in time: for almost every $t\in[0,+\infty[\,$,
one has
\begin{equation*} \label{eq:mass_conserv_G}
\int_{\Omega}\bigl(\vre(t)\,-\,\vret\bigr)\,\dx\,=\,0\,.
\end{equation*}

To conclude, as already highlighted in Chapter \ref{chap:multi-scale_NSF}, in our framework of finite energy weak solutions,
inequality \eqref{est:dissip_G} will be the only tool to derive uniform estimates for the family of weak solutions we are going to consider.

\subsection{Main theorems}\label{ss:results_G}

We can now state our main results. We point out that, due to the scaling \eqref{eq:scale-our}, the relation $m>n$ is always true,
so we will always be in a low stratification regime.

The first statement concerns the case when the effects linked to the pressure term are predominant in the dynamics (with respect to the fast rotation), i.e. $m>1$.

\begin{theorem}\label{th:m>1}
Let $\Omega= \R^2 \times\,]0,1[\,$ and $G\in W^{1,\infty}(\Omega)$ be as in \eqref{assG}. Take $m>1$ and $m+1\geq 2n >m$.
For any fixed value of $\veps \in \; ]0,1]$, let initial data $\left(\vrho_{0,\veps},\vec u_{0,\veps}\right)$ verify the hypotheses fixed in Paragraph \ref{sss:data-weak_G}, and let
$\left( \vre, \ue\right)$ be a corresponding weak solution to system \eqref{ceq_G}-\eqref{meq_G}, supplemented with the structural hypotheses  \eqref{S_G} on $\mbb{S}(\nabla_x \ue)$ and with boundary conditions \eqref{bc1-2_G} and  far field conditions \eqref{ff}.
Let $\vec u_0$ be defined as in \eqref{conv:in_data_vel}.

Then, for any $T>0$, one has that
	\begin{align*}
	\varrho_\ep \rightarrow 1 \qquad\qquad &\mbox{ strongly in } \qquad L^{\infty}\big([0,T]; L_{\rm loc}^{\min\{2,\g\}}(\Omega )\big) \\
	\vec{u}_\ep \weak \vec{U}
	\qquad\qquad &\mbox{ weakly in }\qquad L^2\big([0,T];H^{1}(\Omega)\big)\,, 
	\end{align*}	
where $\vec{U} = (\vec U^h,0)$, with $\vec U^h=\vec U^h(t,x^h)$ such that $\divh\vec U^h=0$. In addition, the vector field $\vec{U}^h $ is a weak solution
to the following homogeneous incompressible Navier-Stokes system  in $\R_+ \times \R^2$,
\begin{align}
& \d_t \vec U^{h}+\divh\left(\vec{U}^{h}\otimes\vec{U}^{h}\right)+\nabla_h\Gamma-\mu \Delta_{h}\vec{U}^{h}=0\, , \label{eq_lim_m:momentum_G} 
\end{align}
for a suitable pressure function $\Gamma\in\mc D'(\R_+\times\R^2)$ and related to the initial condition
$$
\vec{U}_{|t=0}=\h_h\left(\lan\vec{u}^h_{0}\ran\right)\, .
$$
\end{theorem}



When $m=1$, the Mach and Rossby numbers have the same order of magnitude, and they keep in the so-called \emph{quasi-geostrophic balance} at the limit. Namely, the next statement is devoted to this isotropic case.
\begin{theorem} \label{th:m=1}
Let $\Omega = \R^2 \times\,]0,1[\,$ and $G\in W^{1,\infty}(\Omega)$ be as in \eqref{assG}. Take $m=1$ and $1/2<n<1$. 
For any fixed value of $\veps \in \; ]0,1]$, let initial data $\left(\vrho_{0,\veps},\vec u_{0,\veps}\right)$ verify the hypotheses fixed in Paragraph \ref{sss:data-weak_G}, and let
$\left( \vre, \ue\right)$ be a corresponding weak solution to system \eqref{ceq_G}-\eqref{meq_G}, supplemented with the structural hypotheses  \eqref{S_G} on $\mbb{S}(\nabla_x \ue)$ and with boundary conditions \eqref{bc1-2_G} and far field conditions \eqref{ff}.
Let $\left(\vrho^{(1)}_0,\vec u_0\right)$ be defined as in \eqref{conv:in_data_vrho} and \eqref{conv:in_data_vel}.

Then, for any $T>0$, one has the following convergence properties:
	\begin{align*}
	\varrho_\ep \rightarrow 1 \qquad\qquad &\mbox{ strongly in } \qquad L^{\infty}\big([0,T]; L_{\rm loc}^{\min\{2,\g\}}(\Omega )\big) \\
	\vrho^{(1)}_\veps:=\frac{\varrho_\ep - \widetilde{\vrho_\veps}}{\ep}  \weakstar \vrho^{(1)} \qquad\qquad &\mbox{ weakly-$*$ in }\qquad L^{\infty}\big([0,T]; L^{2}+L^{\gamma}(\Omega )\big) \\
	\vec{u}_\ep \weak \vec{U}
	\qquad\qquad &\mbox{ weakly in }\qquad L^2\big([0,T];H^{1}(\Omega)\big)\,,
	\end{align*}	
where, as above, $\vec{U} = (\vec U^h,0)$, with $\vec U^h=\vec U^h(t,x^h)$ such that $\divh\vec U^h=0$. 
Moreover, one has the balance $\vec U^h=\nabla_h^\perp \vrho^{(1)}$, and $\vrho^{(1)}$ satisfies (in the weak sense) the  quasi-geostrophic equation
\begin{align}
& \d_{t}\left(\vrho^{(1)}-\Delta_{h}\vrho^{(1)}\right) -\nabla_{h}^{\perp}\vrho^{(1)}\cdot
\nabla_{h}\left( \Delta_{h}\vrho^{(1)}\right) +\mu 
\Delta_{h}^{2}\vrho^{(1)}\,=\,0\,, \label{eq_lim:QG_G}  
\end{align}
supplemented with the initial condition
$$
\left(\vrho^{(1)}-\Delta_{h}\vrho^{(1)}\right)_{|t=0}=\langle \vrho_0^{(1)}\rangle-\curlh\lan\vec u^h_{0}\ran\, .
$$
\end{theorem}

\section{Consequences of the energy inequality} \label{s:energy}

In Definition \ref{d:weak}, we have postulated that the family of weak solutions $\big(\vrho_\veps,\vu_\veps\big)_\veps$ considered in Theorems
\ref{th:m>1} and \ref{th:m=1} satisfies the energy inequality \eqref{est:dissip_G}. 

In this section, we take advantage of the energy inequality to infer uniform bounds for $\big(\vrho_\veps,\vu_\veps\big)_\veps$: this will be done in Subsection \ref{ss:unif-est_G}.
Thanks to those bounds, we can extract (in Subsection \ref{ss:ctl1_G}) weak-limit points of the sequence of solutions and deduce some properties these limit points
have to satisfy.

\subsection{Uniform bounds and weak limits}\label{ss:unif-est_G}

This subsection is devoted to establish uniform bounds on the sequence $\bigl(\vrho_\veps,\vec u_\veps\bigr)_\veps$.
This can be done as in the classical case (see e.g. \cite{F-N} for details), since again
the Coriolis term does not contribute to the total energy balance of the system.
However, for the reader's convenience, let us present some details.

To begin with, let us recall the partition of the space domain $\Omega$ into the so-called ``essential'' and ``residual'' sets.
For this, for $t>0$ and for all $\veps\in\,]0,1]$, we define the sets
$$
\Omega_\ess^\veps(t)\,:=\,\left\{x\in\Omega\;\big|\quad \vrho_\veps(t,x)\in\left[1/2\,\rho_*\,,\,2\right]\right\}\,,\qquad\Omega^\veps_\res(t)\,:=\,\Omega\setminus\Omega^\veps_\ess(t)\,,
$$
where the positive constant $\rho_*>0$ has been defined in Proposition \ref{p:target-rho_bound_G}.

Next, we observe that
\[
\Big[\mc E\big(\rho(t,x),\wtilde\vrho_\veps(x)\big)\Big]_\ess\,\sim\,\left[\rho-\wtilde\vrho_\veps(x)\right]_\ess^2
\qquad\quad \mbox{ and }\qquad\quad
\Big[\mc E\big(\rho(t,x),\wtilde\vrho_\veps(x)\big)\Big]_\res\,\geq\,C\left(1\,+\,\big[\rho(t,x)\big]_\res^\g\right)\,,
\]
where $\vret$ is the static density state identified in Paragraph \ref{sss:equilibrium_G}. 
Here above, the multiplicative constants are all strictly positive and may depend on $\rho_*$, and we agree to write $A\sim B$ whenever there exists a ``universal'' constant $c>0$ such that $(1/c)B\leq A\leq c\, B$.

Thanks to the previous observations, we easily see that, under the assumptions fixed in Section \ref{s:result_G} on the initial data,
the right-hand side of \eqref{est:dissip_G} is \emph{uniformly bounded} for all $\veps\in\,]0,1]$: specifically, we have
$$
\int_{\Omega} \frac{1}{2}\vrez|\uez|^2\,\dx + \frac{1}{\ep^{2m}}\int_{\Omega}\mc E\left(\vrho_{0,\veps},
\, \wtilde\vrho_\veps\right)\,\dx\,\leq\,C\,.
$$

Owing to the previous inequalities and the finite energy condition \eqref{est:dissip_G} on the family of weak solutions,
it is quite standard to derive, for any time $T>0$ fixed and any $\veps\in\,]0,1]$, the following estimates, that we recall here for the reader's convenience:
\begin{align}
	\sup_{t\in[0,T]} \| \sqrt{\vre}\ue\|_{L^2(\Omega;\, \R^3)}\, &\leq\,c \label{est:momentum_G} \\	
	\sup_{t\in[0,T]} \left\| \left[ \dfrac{\vre - \vret}{\ep^m}\right]_\ess (t) \right\|_{L^2(\Omega)}\,&\leq\, c \label{est:rho_ess_G} \\
	\sup_{t\in[0,T]} \int_{\Omega}	\bbbone_{\mc{M}^\veps_\res[t]} \,dx\,&\leq \, c\,\ep^{2m} \label{est:M_res-measure_G}\\
	\sup_{t\in [0,T]} \int_{\Omega} [ \vre]^{\gamma}_\res (t)\,\dx \,
	\,&\leq\,c\,\ep^{2m} \label{est:rho_res_G} \\
	\int_0^T \left\| \nabla_x \ue +\, ^t\nabla_x \ue  - \frac{2}{3} \div \ue\, \Id \right\|^2_{L^2(\Omega ;\, \R^{3\times3})}\, \dt\,
	&\leq\, c\, . \label{est:Du_G} 
	\end{align}
We refer to \cite{F-N} (see also \cite{F-G-N}, \cite{F-G-GV-N}, \cite{F_2019}) for the details of the computations.

Owing to \eqref{est:Du_G} and a generalisation of the Korn-Poincar\'e inequality (see Proposition \ref{app:korn-poincare_prop} in the Appendix), we gather that
$\big(\nabla\vu_\veps\big)_\veps\,\subset\,L^2_T(L^2)$. On the other hand, by arguing as in \cite{F-G-N}, we can use
\eqref{est:momentum_G}, \eqref{est:M_res-measure_G} and \eqref{est:rho_res_G} to deduce that also
$\big(\vu_\veps\big)_\veps\,\subset\,L^2_T(L^2)$. Putting those bounds together, we finally infer that
\begin{equation}\label{bound_for_vel}
\int_0^T \left\|\ue  \right\|^2_{H^{1}(\Omega ;\, \R^{3})}\, \dt\,\leq \, c\, .
\end{equation} 
In particular, there exist $\vU\,\in\,L^2_{\rm loc}\big(\R_+;H^1(\Omega;\R^3)\big)$ such that, up to a suitable extraction (not relabelled here),
we have
\begin{equation} \label{conv:u_G}
\vu_\veps\,\rightharpoonup\,\vU\qquad\qquad \mbox{ in }\quad L^2_{\rm loc}\big(\R_+;H^1(\Omega;\R^3)\big)\,.
\end{equation}

Let us move further and consider the density functions. The previous estimates on the density tell us that we must find a finer decomposition for the densities. As a matter
of fact, for any time $T>0$ fixed, we have
	\begin{equation}\label{rr1_G}
\| \vre - 1 \|_{L^\infty_T(L^2 + L^{\gamma} + L^\infty)}\,\leq\,c\,  \ep^{2(m-n)}\,.
	\end{equation}
In order to see \eqref{rr1_G}, to begin with, we write 
\begin{equation}\label{rel:density_1}
|\vrho_\veps-1|\,\leq\,|\vrho_\veps-\widetilde{\vrho}_\veps|+|\widetilde{\vrho}_\veps-1|\,.
\end{equation}	
From \eqref{est:rho_ess_G}, we infer that $\big[\vr_\veps\,-\,\wtilde{\vr}_\veps\big]_\ess$ is of order $O(\veps^m)$ in $L^\infty_T(L^2)$.
For the residual part of the same term, we can use \eqref{est:rho_res_G} to discover that it is of order $O(\veps^{2m/\gamma})$.
Observe that, if $1<\g<2$, the higher order is $O(\veps^m)$, whereas, in the case $\g\geq2$, by use of \eqref{est:rho_res_G} and \eqref{est:M_res-measure_G} again, it is easy to get
\begin{equation} \label{est:res_g>2}
\left\|\left[\vr_\veps\,-\,\wtilde{\vr}_\veps\right]_\res\right\|_{L^\infty_T(L^2)}^2\,\leq\,C\,\veps^{2m}\,.
\end{equation}
Finally, we apply Proposition \ref{p:target-rho_bound_G} to control the last term in the right-hand side of \eqref{rel:density_1}.
In the end, estimate \eqref{rr1_G} is proved.

This having been established, and keeping in mind the notation introduced in \eqref{in_vr_G} and \eqref{eq:in-dens_dec_G}, we can introduce the
density oscillation functions
\[ 
R_\veps\,:=\, \frac{\varrho_\ep -1}{\ep^{2(m-n)}}\, =\,\wtilde{r}_\veps\,+\,\veps^{2n-m}\,\vrho_\veps^{(1)}\,,
\] 
where we have defined
\begin{equation} \label{def_deltarho_G}
\vrho_\veps^{(1)}(t,x)\,:=\,\frac{\vre-\wtilde{\vrho}_\veps}{\ep^m}\qquad\mbox{ and }\qquad
\wtilde{r}_\veps(x)\,:=\,\frac{\wtilde{\vrho}_\veps-1}{\ep^{2(m-n)}}\,.
\end{equation}

Thanks again to \eqref{est:rho_ess_G}, \eqref{est:rho_res_G} and Proposition \ref{p:target-rho_bound_G}, we see that the previous quantities verify the following uniform bounds, for any time $T>0$ fixed:
\begin{equation}\label{uni_varrho1_G}
\sup_{\veps\in\,]0,1]}\left\|\vrho_\veps^{(1)}\right\|_{L^\infty_T(L^2+L^{\gamma}({\Omega}))}\,\leq\, c \qquad\qquad\mbox{ and }\qquad\qquad
\sup_{\veps\in\,]0,1]}\left\| \wtilde{r}_\veps \right\|_{L^{\infty}(\Omega)}\,\leq\, c \,.
\end{equation}
In view of the previous properties, there exist $\vrho^{(1)}\in L^\infty_T(L^2+L^{\gamma})$ and
$\wtilde{r}\in L^\infty$ such that (up to the extraction of a new suitable subsequence)
\begin{equation} \label{conv:rr_G}
\vrho_\veps^{(1)}\,\weakstar\,\vrho^{(1)}\qquad\qquad \mbox{ and }\qquad\qquad \wtilde{r}_\veps\,\weakstar\,\wtilde{r}
\end{equation}
in the weak-$*$ topology of the respective spaces.
In particular, we get
	\begin{equation*} 
	R_\veps\, \weakstar\,\widetilde{r} \qquad\qquad\qquad \mbox{ weakly-$*$ in }\qquad L^\infty\bigl([0,T]; L^{\min\{\g,2\}}_{\rm loc}(\Omega)\bigr)\, . 
	\end{equation*}

\begin{remark} \label{r:g>2}
Observe that, owing to \eqref{est:res_g>2}, when $\g\geq2$ we get
\[
\sup_{\veps\in\,]0,1]}\left\|\vrho_\veps^{(1)}\right\|_{L^\infty_T(L^2)}\,\leq\, c\, .
\]
Therefore, in that case we actually have that $\vrho^{(1)}\,\in\,L^\infty_T(L^2)$ and $\vrho_\veps^{(1)}\,\stackrel{*}{\rightharpoonup}\,\vrho^{(1)}$ in that space.

Analogously, when $\g\geq2$ we also get
\[
\| \vre - 1 \|_{L^\infty_T(L^2 + L^\infty)}\,\leq\,c\,  \ep^{2(m-n)}.
\]

\end{remark}

\subsection{Constraints on the limit}\label{ss:ctl1_G}

In this subsection, we establish some properties that the limit points of the family $\bigl(\vrho_\veps,\vec u_\veps \bigr)_\veps$, which have been
identified here above, have to satisfy.

We first need a preliminary result about the decomposition of the pressure function, which will be useful in the following computations.
\begin{lemma}\label{lem:manipulation_pressure}
Let $(m,n)\in \R^2$ verify the condition $m+1\,\geq\,2n\,>\,m\geq 1$. Let $p$ be the pressure term satisfying the structural hypotheses \eqref{pp1_G} and \eqref{pp2_G}. Then, for any $\veps\in\,]0,1]$, one has
 \begin{equation}\label{p_decomp_lemma}
\begin{split}
\frac{1}{\veps^{2m}}\,\nabla_x\Big(p(\vrho_\veps)\,-\,p(\wtilde{\vrho}_\veps)\Big)\,=\,
\frac{1}{\veps^m}\nabla_x \Big(p^\prime (1)\vrho_\veps^{(1)}\Big)\,+\, 
\frac{1}{\veps^{2n-m}}\,\nabla_x \Pi_\veps\, ,
\end{split}
 \end{equation}
 where the functions $\vr_\veps^{(1)}$ have been introduced in \eqref{def_deltarho_G} and for all $T>0$ the family $\big(\Pi_\veps\big)_\veps$ verifies the uniform bound
\begin{equation}\label{unif-bound-Pi}
\left\|\Pi_\veps\right\|_{L^\infty_T(L^1+L^2+L^\g)}\,\leq\,C\,.
\end{equation}
When $\g\geq2$, one can dispense of the space $L^\g$ in the previous control of $\big(\Pi_\veps\big)_\veps$.
\end{lemma}

\begin{proof}
First of all, we write
\begin{equation}\label{rel_p}
\begin{split}
\frac{1}{\veps^{2m}}\,\nabla_x\Big(p(\vrho_\veps)\,-\,p(\wtilde{\vrho}_\veps)\Big)\,&=\,\frac{1}{\veps^{2m}}\,\nabla_x\Big(p(\vrho_\veps)\,-\,p(\wtilde{\vrho}_\veps)-p^\prime(\widetilde{\vrho}_\veps)(\vrho_\veps -\widetilde{\vrho}_\veps ) \Big)\\
&\qquad\qquad+\,\frac{1}{\veps^{m}}\,\nabla_x\Big(\big(p^\prime(\widetilde{\vrho}_\veps)\,-\,p'(1)\big)\,\vr_\veps^{(1)}\Big)\,+\,
\frac{1}{\veps^m}\nabla_x \Big(p^\prime (1)\vrho_\veps^{(1)}\Big)\,.
\end{split}
\end{equation}

We start by analysing the first term on the right-hand side of \eqref{rel_p}. For the essential part, we can employ a Taylor expansion
to write
$$
\left[p(\vrho_\veps)-p(\widetilde{\vrho}_\veps)-p^\prime(\widetilde{\vrho}_\veps)(\vrho_\veps-\widetilde{\vrho}_\veps)\right]_\ess=\left[p^{\prime \prime}(z_\veps )(\vrho_\veps - \widetilde{\vrho}_\veps)^2\right]_\ess\, ,
$$
where $z_\veps$ is a suitable point between $\vrho_\veps$ and $\widetilde{\vrho}_\veps$. Thanks to the uniform bound \eqref{est:rho_ess_G}, we have that this term is of order $O(\veps^{2m})$ in $L^\infty_T(L^1)$, for any $T>0$ fixed. For the residual part, we can use \eqref{est:M_res-measure_G} and \eqref{est:rho_res_G}, together with the boundedness of the profiles $\wtilde\vr_\veps$ (keep in mind Proposition \ref{p:target-rho_bound_G}), to deduce that
\[
\left\|\left[p(\vrho_\veps)-p(\widetilde{\vrho}_\veps)-p^\prime(\widetilde{\vrho}_\veps)(\vrho_\veps-\widetilde{\vrho}_\veps)\right]_\res\right\|_{L^\infty_T(L^1)}\,\leq C\,\veps^{2m}\,.
\]
We refer to e.g. Lemma 4.1 of \cite{F_2019} for details.

In a similar way, a Taylor expansion for the second term on the right-hand side of \eqref{rel_p} gives
$$
\big( p^\prime(\widetilde{\vrho}_\veps)-p^\prime(1)\big)\vrho^{(1)}_\veps\,=\,p''(\eta_\veps )( \widetilde{\vrho}_\veps -1)\vrho^{(1)}_\veps\, ,
$$
where $\eta_\veps$ is a suitable point between $\widetilde{\vrho}_\veps$ and 1. Owing to Proposition \ref{p:target-rho_bound_G} again and to bound \eqref{uni_varrho1_G},
we infer that this term is of order $O(\veps^{2(m-n)})$ in $L^\infty_T(L^2+L^\g)$, for any time $T>0$ fixed.

Then, defining 
\begin{equation}\label{nablaPiveps}
\Pi_\veps:= \frac{1}{\veps^{2(m-n)}}\left[\frac{p(\vrho_\veps )-p(\widetilde{\vrho}_\veps)}{\veps^m}-p^\prime(1)\vrho_\veps^{(1)}\right] 
\end{equation}
we have the control \eqref{unif-bound-Pi}.

The final statement concerning the case $\g\geq2$ easily follows from Remark \ref{r:g>2}.
This completes the proof of the lemma. 
\qed
\end{proof}

\medbreak
Notice that the last term appearing in \eqref{p_decomp_lemma} is singular in $\veps$. This is in stark contrast with the situation considered in previous works, see e.g. \cite{F-G-N}, \cite{F-G-GV-N}, \cite{F-N_CPDE} and \cite{F_2019}. However, its gradient structure will play a fundamental role in the computations below.

This having been pointed out, we can now analyse the constraints on the weak-limit points $\big(\vr^{(1)},\vec U\big)$, identified in relations \eqref{conv:u_G} and \eqref{conv:rr_G} above. 

\subsubsection{The case of large values of the Mach number: $m>1$} \label{ss:constr_2_G}

We start by considering the case of anisotropic scaling, namely $m>1$ and $m+1\geq 2n>m$. Notice that, in particular, one has $m>n$.

\begin{proposition} \label{p:limitpoint_G}
Let $m>1$ and $m+1\geq 2n>m$ in \eqref{ceq_G}-\eqref{meq_G}.
Let $\left( \vre, \ue \right)_{\veps}$ be a family of weak solutions, related to initial data $\left(\vrho_{0,\veps},\vec u_{0,\veps}\right)_\veps$
verifying the hypotheses of Paragraph \ref{sss:data-weak_G}. Let $(\vrho^{(1)}, \vec{U} )$ be a limit point of the sequence
$\left(\vrho_\veps^{(1)}, \ue\right)_{\veps}$, as identified in Subsection \ref{ss:unif-est_G}. Then,
\begin{align}
&\vec{U}\,=\,\,\Big(\vec{U}^h\,,\,0\Big)\,,\qquad\qquad \mbox{ with }\qquad \vec{U}^h\,=\,\vec{U}^h(t,x^h)\quad \mbox{ and }\quad \div_{\!h}\,\vec{U}^h\,=\,0\,,  \label{eq:anis-lim_1_G} \\[1ex]
&\nabla_x \vrho^{(1)}\,=\, 0
\qquad\qquad\mbox{ in }\;\,\mc D^\prime(\R_+\times \Omega)\,. \label{eq:anis-lim_2_G} 
\end{align}
\end{proposition}

\begin{proof} First of all, let us consider the weak formulation of the mass equation \eqref{ceq_G}: for any test function $\varphi\in C_c^\infty\bigl(\R_+\times\Omega\bigr)$, denoting $[0,T]\times K\,:=\,\Supp \, \varphi$, with $\vphi(T,\cdot)\equiv0$, we have
$$
-\int^T_0\int_K\bigl(\vrho_\veps-1\bigr)\,\d_t\varphi \dxdt\,-\,\int^T_0\int_K\vrho_\veps\,\vec{u}_\veps\,\cdot\,\nabla_{x}\varphi \dxdt\,=\,
\int_K\bigl(\vrho_{0,\veps}-1\bigr)\,\varphi(0,\,\cdot\,)\dx\,.
$$
We can easily pass to the limit in this equation, thanks to the strong convergence $\vrho_\veps\longrightarrow1$, provided by \eqref{rr1_G}, and the weak convergence of
$\vec{u}_\veps$ in $L_T^2\bigl(L^6_{\rm loc}\bigr)$, provided by \eqref{conv:u_G} and Sobolev embeddings. Notice that one always has $1/\g\,+\,1/6\,\leq\,1$. In this way,
we find
$$
-\,\int^T_0\int_K\vec{U}\,\cdot\,\nabla_{x}\varphi \dxdt\,=\,0\, ,
$$
for any test function $\varphi \, \in C_c^\infty\bigl(\R_+\times\Omega\bigr)$ taken as above. The previous relation in particular implies
\begin{equation} \label{eq:div-free_G}
\div \U = 0 \qquad\qquad\mbox{ a.e. in }\; \,\R_+\times \Omega\,.
\end{equation}

Next, we test the momentum equation \eqref{meq_G} on $\veps^m\,\vec\phi$, for a smooth compactly supported $\vec\phi$.
Using the uniform bounds established in Subsection \ref{ss:unif-est_G}, it is easy to see that the term presenting the derivative in time, the viscosity term and the convective
term converge to $0$, in the limit $\veps\ra0^+$. Since $m>1$, also the Coriolis term vanishes when $\veps\ra0^+$. It remains us to consider
the pressure and gravity terms in the weak formulation \eqref{weak-mom_G} of the momentum equation:
using relation \eqref{prF_G}, we see that we can couple them to write
\begin{align}
\frac{1}{\veps^{2m}}\,\nabla_x p(\vrho_\veps)-\, \frac{1}{\veps^{2n}}\,\vrho_\veps\nabla_x G\,=\,\frac{1}{\veps^{2m}}\nabla_x\Big(p(\vrho_\veps)\,-\,p(\wtilde{\vrho}_\veps)\Big)-\,\veps^{m-2n}\vrho_\veps^{(1)}\nabla_x G\,. \label{eq:mom_rest_1_G}
\end{align}
By \eqref{uni_varrho1_G} and the fact that $m>n$, we readily see that the last term in the right-hand side of \eqref{eq:mom_rest_1_G} converges to $0$,
when tested against any smooth compactly supported $\veps^m\,\vec\phi$.
At this point, we use Lemma \ref{lem:manipulation_pressure} to treat the first term on the right-hand side of \eqref{eq:mom_rest_1_G}.
So, taking $\vec\phi \in C^\infty_c([0,T[\, \times \Omega)$ (for some $T>0$), we test the momentum equation against $\veps^m\,\vec\phi$ and using \eqref{conv:rr}, in the limit $\veps\ra0^+$ we find that 
$$ \int_0^T \int_\Omega p'(1) \vr^{(1)} \div  \vec\phi \dxdt = 0\, .$$
Recalling that $p^\prime (1)=1$, the previous relation implies \eqref{eq:anis-lim_2} for $\vr^{(1)}$.

In particular, that relation implies that $\vrho^{(1)}(t,x)\,=\,c(t)$ for almost all $(t,x)\in\R_+\times\Omega$, for a suitable function $c=c(t)$ depending only on time.

Now, in order to see effects due to the fast rotation in the limit, we need to ``filter out'' the contribution coming from the low Mach number.
To this end, we test \eqref{meq_G} on $\veps\,\vec\phi$, where this time we take $\vec\phi\,=\,\curl\vec\psi$, for some smooth compactly supported $\vec\psi\,\in C^\infty_c\bigl([0,T[\,\times\Omega\bigr)$, with $T>0$.
Once again, by uniform bounds we infer that the $\d_t$ term, the convective term and the viscosity term all converge to $0$ when $\veps\ra0^+$.
As for the pressure and the gravitational force, we argue as in \eqref{eq:mom_rest_1_G}: since the structure of $\vec\phi$ kills any gradient term, we are left with the convergence
of the integral
$$
\int^T_0\int_\Omega\veps^{m-2n+1}\vrho_\veps^{(1)}\nabla_x G\cdot\vec\phi\,\dxdt\,\longrightarrow\,
\delta_0(m-2n+1)\int^T_0\int_\Omega\vrho^{(1)}\nabla_x G\cdot\vec\phi\,\dxdt\,,
$$
where $\de_0(\zeta)\,=\,1$ if $\zeta=0$, $\de_0(\zeta)\,=\,0$ otherwise.
Finally, arguing as done for the mass equation, we see that the Coriolis term converges to the integral $\int^T_0\int_\Omega\e_3\times\vec{U}\cdot\vec\phi$.

Consider the case $m+1>2n$ for a while.
Passing to the limit for $\veps\ra0^+$, we find that $\mbb{H}\left(\e_3\times\vec{U}\right)\,=\,0$, which implies that
$\e_3\times\vec{U}\,=\,\nabla_x\Phi$, for some potential function $\Phi$. From this relation, one easily deduces that $\Phi=\Phi(t,x^h)$, i.e. $\Phi$ does not depend
on $x^3$, and that the same property is inherited by $\vec{U}^h\,=\,\bigl(U^1,U^2\bigr)$, i.e. one has $\vec{U}^h\,=\,\vec{U}^h(t,x^h)$. Furthermore, since $\vec U^h\,=\,-\,\nabla_h^\perp\Phi$, we get that $\div_{\!h}\,\vec{U}^h\,=\,0$.
At this point, we
combine this fact with \eqref{eq:div-free_G} to infer that $\d_3 U^3\,=\,0$; but, thanks to the boundary condition
\eqref{bc1-2_G}, we must have $\bigl(\vec{U}\cdot\vec{n}\bigr)_{|\d\Omega}\,=\,0$, which implies that $U^3$ has to vanish at the boundary of $\Omega$.
Thus, we finally deduce that $U^3\,\equiv\,0$, whence \eqref{eq:anis-lim_1_G} follows (see also the proof of Proposition \ref{p:limitpoint} in this regard).

Now, let us focus on the case when $m+1=2n$. The previous computations show that, when $\veps\ra0^+$, we get
\begin{equation}\label{eq:streamfunction_1} 
\vec{e}_{3}\times \vec{U}+\vrho^{(1)}\nabla_x G\,=\,\nabla_x\Phi \qquad\qquad\mbox{ in }\; \mc D^\prime(\R_+\times \Omega)\,,
\end{equation}
for a new suitable function $\Phi$. However, owing to \eqref{eq:anis-lim_2_G}, we can say that $\vrho^{(1)}\nabla_x G\,=\,\nabla_x\big(\vr^{(1)}\,G\big)$; hence, the previous relations
can be recasted as $\e_3\times\vec U\,=\,\nabla_x\wtilde\Phi$, for a new scalar function $\wtilde\Phi$. Therefore, the same analysis as above applies,
allowing us to gather \eqref{eq:anis-lim_1_G} also in the case $m+1=2n$.
\qed
\end{proof}

\subsubsection{The case $m=1$} \label{ss:constr_1_G}

Now we focus on the case $m=1$. In this case, the fast rotation and weak compressibility
effects are of the same order: this allows to reach the so-called \emph{quasi-geostrophic balance} in the limit.


\begin{proposition}  \label{p:limit_iso_G}
Take $m=1$ and $1/2<n<1$ in system \eqref{ceq_G}-\eqref{meq_G}.
Let $\left( \vre, \ue\right)_{\veps}$ be a family of weak solutions to \eqref{ceq_G}-\eqref{meq_G}, associated with initial data
$\left(\vrho_{0,\veps},\vec u_{0,\veps}\right)$ verifying the hypotheses fixed in Paragraph \ref{sss:data-weak_G}.
Let $(\vrho^{(1)}, \vec{U} )$ be a limit point of the sequence $\left(\vrho^{(1)}_{\veps} , \ue\right)_{\veps}$, as identified in Subsection \ref{ss:unif-est_G}.
Then,
\begin{align}
\vr^{(1)}\,=\,\vr^{(1)}(t,x^h)\quad\mbox{ and }\quad\vec{U}\,=\,\,\Big(\vec{U}^h\,,\,0\Big)\,,\quad \mbox{ with }\quad 
\vec{U}^h\,=\,\nabla^\perp_h \vrho^{(1)}
\;\mbox{ a.e. in }\;\R_+ \times \R^2\,.  \label{eq:for q_G}
\end{align}
In particular, one has $\vec U^h\,=\,\vec U^h(t,x^h)$ and $\div_{\!h}\vec U^h\,=\,0$.
\end{proposition}
\begin{proof}
Arguing as in the proof of Proposition \ref{p:limitpoint_G}, it is easy to pass to the limit in the continuity equation. In particular, we obtain again relation \eqref{eq:div-free_G} for $\vec U$.

Only the analysis of the momentum equation changes a bit with respect to the previous case $m>1$. Now, since the most singular terms are of order $\veps^{-1}$ (keep in mind Lemma \ref{lem:manipulation_pressure}), we test
the weak formulation \eqref{weak-mom_G} of the momentum equation against $\veps\,\vec\phi$, where $\vec \phi$ is a smooth compactly supported function. Similarly to what done above, the uniform bounds of Subsection \ref{ss:unif-est_G} allow us to infer that the only quantity which does not vanish in the limit is the sum of the terms involving the Coriolis force, the pressure and the gravitational force: more precisely, using also Lemma \ref{lem:manipulation_pressure}, we have
$$
\vec{e}_{3}\times \vrho_{\veps}\ue\,+\frac{\nabla_x \Big( p(\vrho_\veps)-p(\widetilde{\vrho}_\veps)\Big)}{\veps}\,-\,
\veps^{2(1-n)}\vrho_\veps^{(1)}\nabla_x G\,=\,\mc O(\veps)\, ,
$$
in the sense of $\mc D'(\R_+\times\Omega)$.
Following the same computations performed in the proof of Proposition \ref{p:limitpoint_G}, in the limit $\veps\ra0^+$ it is easy to get that
$$ 
\vec{e}_{3}\times \vec{U}+\nabla_x\left(p^\prime (1) \vrho^{(1)}\right)\,=\,0\qquad\qquad\mbox{ in }\; \mc D'\big(\R_+\times \Omega\big)\,.
$$ 
After recalling that $p^\prime (1)=1$, this equality can be equivalently written as 
$$ 
\vec{e}_{3}\times \vec{U}+\nabla_x \vrho^{(1)}\,=\,0 \qquad\qquad\mbox{ a.e. in }\; \R_+ \times \Omega\,.
$$ 
Notice that $\vec U$ is in fact in $L^2_{\rm loc}(\R_+;L^2)$, therefore so is $\nabla_x \vr^{(1)}$; hence the previous relation is in fact satisfied almost everywhere
in $\R_+\times\Omega$.

At this point, we can repeat the same argument used in the proof of Proposition \ref{p:limitpoint_G} to deduce \eqref{eq:for q_G}.
The proposition is thus proved.
\qed
\end{proof}

\section{Convergence in the case $m>1$}\label{s:proof_G}

In this section, we complete the proof of Theorem \ref{th:m>1}. Namely, we show convergence in the weak formulation of the primitive system, in the case when $m>1$ and $m+1\geq 2n>m$. 

In Proposition \ref{p:limitpoint_G}, we have already seen how passing to the limit in the mass equation.
However, problems arise when tackling the convergence in the momentum equation. Indeed, the analysis carried out so far
is not enough to identify the weak limit of the convective term $\vrho_\veps\,\vec u_\veps\otimes\vec u_\veps$, which is highly non-linear.
For proving that this term converges to the expected limit $\vec U\otimes\vec U$, the key point is to control the fast oscillations in time of the solutions,
generated by the singular terms in the momentum equation. For this, we will use a compensated compactness argument and we exploit the algebraic structure of the wave system
underlying the primitive equations \eqref{ceq_G}-\eqref{meq_G}.

In Subsection \ref{ss:acoustic_G}, we start by giving a quite accurate description of those fast oscillations. 
Then, using that description, we are able, in Subsection \ref{ss:convergence_G}, to establish two fundamental properties: on the one hand, strong convergence of a suitable quantity related to the velocity fields; on the other hand, the other terms, which do not involve that quantity, tend to vanish when $\veps\ra0^+$.
In turn, this allows us to complete, in Subsection \ref{ss:limit_G}, the proof of the convergence.

\subsection{Analysis of the strong oscillations} \label{ss:acoustic_G}

The goal of the present subsection is to describe the fast oscillations in time of the solutions. First of all, we recast our equations into a wave system. Then, we establish uniform bounds for the quantities appearing in the wave system. Finally, we apply a regularisation
in space procedure for all the quantities, which is preparatory in view of the computations of Subsection \ref{ss:convergence_G}.

\subsubsection{Formulation of the acoustic wave system} \label{sss:wave-eq_G}

We introduce the quantity
$$
\vec{V}_\veps\,:=\,\vrho_\veps\vec{u}_\veps\,.
$$
Then, straightforward computations show that we can recast the continuity equation in the form
\begin{equation} \label{eq:wave_mass_G}
\veps^m\,\d_t\vrho^{(1)}_\veps\,+\,\div\vec{V}_\veps\,=\,0\,,
\end{equation}
where $\vrho^{(1)}_\veps$ is defined in \eqref{def_deltarho_G}.
Next, thanks to Lemma \ref{lem:manipulation_pressure} and the static relation \eqref{prF_G}, we can derive the following form of the momentum equation:
\begin{align}
\veps^m\,\d_t\vec{V}_\veps\,+\,\veps^{m-1}\,\e_3\times \vec V_\veps\,+p^\prime(1)\,\nabla_x \vrho_{\veps}^{(1)}\,&=\,
\veps^{2(m-n)}\left(\vrho_\veps^{(1)}\nabla_x G\,-\,\nabla_x\Pi_\veps 
\right) \label{eq:wave_momentum_G} \\
&\qquad\qquad
+\,\veps^m\,\Big(\div\mbb{S}\!\left(\nabla_x\vec{u}_\veps\right)\,-\,\div\!\left(\vrho_\veps\vec{u}_\veps\otimes\vec{u}_\veps\right)
\Big)\,. \nonumber
\end{align}

Then, if we define
\begin{equation}\label{def_f-g}
\vec f_\veps :=\div\big(\mbb{S}\!\left(\nabla_x\vec{u}_\veps\right)\,-\,\vrho_\veps\vec{u}_\veps\otimes\vec{u}_\veps\big)\qquad \mbox{ and }\qquad
\vec g_\veps :=\vrho_\veps^{(1)}\nabla_x G\,-\,\nabla_x\Pi_\veps\,,
\end{equation}
recalling that we have normalised the pressure function so that $p^\prime (1)=1$, we can rewrite the primitive system \eqref{ceq_G}-\eqref{meq_G} in the following form:
\begin{equation} \label{eq:wave_syst_G}
\left\{\begin{array}{l}
       \veps^m\,\d_t \vrho^{(1)}_\veps\,+\,\div\vec{V}_\veps\,=\,0 \\[1ex]
       \veps^m\,\d_t\vec{V}_\veps\,+\,\nabla_x \vrho_\veps^{(1)}\,+\,\veps^{m-1}\,\e_3\times \vec V_\veps\,=\,\veps^m\,\vec f_\veps +\veps^{2(m-n)}\vec g_\veps\,.
       \end{array}
\right.
\end{equation}

We remark that system \eqref{eq:wave_syst_G} has to be read in the weak sense: for any $\varphi\in C_c^\infty\bigl([0,T[\,\times \oline\Omega\bigr)$, one has
$$
-\,\veps^m\,\int^T_0\int_{\Omega} \vrho^{(1)}_\veps\,\d_t\varphi\,-\,\int^T_0\int_{\Omega} \vec{V}_\veps\cdot\nabla_x\varphi\,=\,
\veps^{m}\int_{\Omega} \vrho^{(1)}_{0,\veps}\,\varphi(0)\,\,,
$$
and also, for any $\vec{\psi}\in C_c^\infty\bigl([0,T[\,\times \oline\Omega;\R^3\bigr)$ such that $(\vec \psi \cdot \vec n)_{|\partial \Omega}=0$, one has
\begin{align*}
&\hspace{-0.5cm}
-\,\veps^m\,\int^T_0\int_{\Omega}\vec{V}_\veps\cdot\d_t\vec{\psi}\,-\,\int^T_0\int_{\Omega} \vrho^{(1)}_\veps\,\div\vec{\psi}\,+\,\veps^{m-1}\int^T_0\int_{\Omega} \e_3\times\vec V_\veps\cdot\vec\psi \\
&\qquad\qquad\qquad\qquad\qquad
=\,\veps^{m}\int_{\Omega}\vrho_{0,\veps}\,\vec{u}_{0,\veps}\cdot\vec{\psi}(0)\,+\,\veps^m\,\int^T_0\int_{\Omega} \vec f_\veps \cdot\vec{\psi}+\,\veps^{2(m-n)}\,\int^T_0\int_{\Omega} \vec g_\veps \cdot\vec{\psi}\,.
\end{align*}

Here we use estimates of Subsection \ref{ss:unif-est_G} in order to show uniform bounds for the solutions and the data in the wave equation \eqref{eq:wave_syst_G}.
We start by dealing with the ``unknown'' $\vec V_\veps$. Splitting the term into essential and residual parts, one can obtain for all $T>0$, 
\begin{equation}\label{eq:V_bounds}
\|\vec V_\veps\|_{L^\infty_T(L^2+L^{2\gamma/(\gamma +1)})}\leq c\, .
\end{equation}
In the next lemma, we establish bounds for the source terms in the system of acoustic waves \eqref{eq:wave_syst_G}.

\begin{lemma} \label{l:source_bounds_G}
Write $\vec f_\veps\,=\,\div \wtilde{\vec f}_\veps$ and $\vec g_\veps\,=\,\vec g^1_\veps\,-\,\nabla_x\Pi_\veps$, where we have defined the quantities $\wtilde{\vec f}_\veps:=\mbb{S}\!\left(\nabla_x\vec{u}_\veps\right)-\vrho_\veps \ue \otimes \ue$, $\vec g^1_\veps\,:=\,\vr_\veps^{(1)}\,\nabla_xG$ and the functions $\Pi_\veps$ have been introduced in \eqref{nablaPiveps} of Lemma \ref{lem:manipulation_pressure}.

For any $T>0$ fixed, one has the uniform embedding properties
\[
\big(\wtilde{\vec f}_\veps\big)_\veps\,\subset\,L^2_T(L^2+L^1)\qquad\mbox{ and }\qquad \big(\vec g^1_\veps\big)_\veps\,\subset\,L^2_T(L^2+L^\g)\,.
\]
In the case $\g\geq2$, we may get rid of the space $L^\g$ in the control of $\big(\vec g^1_\veps\big)_\veps$. 

In particular, the sequences $\bigl(\vec f_\veps\bigr)_\veps$ and
$\bigl(\vec g_\veps\bigr)_\veps$, defined in system \eqref{eq:wave_syst_G}, are uniformly bounded in the space $L^{2}\big([0,T];H^{-s}(\Omega)\big)$, for all $s>5/2$.
\end{lemma}

\begin{proof}
From \eqref{est:momentum_G}, \eqref{est:Du_G} and \eqref{bound_for_vel}, we immediately infer the uniform bound for the family $\big(\wtilde{\vec f}_\veps\big)_\veps$ in $L^2_T(L^1+L^2)$,
from which we deduce also the uniform boundedness of $\big(\vec f_\veps\big)_\veps$ in $L^2_T(H^{-s})$, for any $s>5/2$ (see Theorem \ref{app:thm_dual_Sob} in this respect).

Next, for bounding $\big(\vec g^1_\veps\big)_\veps$ we simply use \eqref{uni_varrho1_G}, together with Remark \ref{r:g>2} when $\g\geq2$.
Keeping in mind the bounds established in Lemma \ref{lem:manipulation_pressure}, the uniform estimate for $\big(\vec g_\veps\big)_\veps$ follows.
\qed
\end{proof}

\subsubsection{Regularization and description of the oscillations}\label{sss:w-reg_G}



As already mentioned in Remark \ref{r:period-bc}, in order to apply the Littlewood-Paley theory, it is convenient to reformulate problem \eqref{ceq_G}-\eqref{meq_G} in the new domain (which we keep calling $\Omega$, with a little abuse of notation)
$$ \Omega:= \R^2 \times \mbb{T}^1,\quad \quad \text{with}\quad \quad \mbb{T}^1:=[-1,1]/\sim\, . $$

In addition, to avoid the appearing of the (irrelevant) multiplicative constants on the computations, we suppose that the torus $\mbb{T}^1$ has been renormalised so that its Lebesgue measure is equal to 1. 

Now, for any $M\in\N$ we consider the low-frequency cut-off operator ${S}_{M}$ of a Littlewood-Paley decomposition, as introduced in \eqref{eq:S_j} of Section \ref{app:LP}. Then, we define 
\begin{equation}\label{def_reg_vrho-V}
\vrho^{(1)}_{\varepsilon ,M}={S}_{M}\vrho^{(1)}_{\veps}\qquad\qquad \text{ and }\qquad\qquad \vec{V}_{\veps ,M}={S}_{M}\vec{V}_{\veps}\, .
\end{equation} 

The previous regularised quantities satisfy the following properties.

\begin{proposition} \label{p:prop approx_G}
For any $T>0$, we have the following convergence properties, in the limit $M\rightarrow +\infty $:
\begin{equation}\label{eq:approx var_G}
\begin{split}
&\sup_{0<\veps\leq1}\, \left\|\vrho^{(1)}_{\varepsilon }-\vrm\right\|_{L^{\infty}([0,T];H^{-s})}\longrightarrow 0\qquad
\forall\,s>\max\left\{0,3\left(\frac{1}{\g}\,-\,\frac12\right)\right\}\\
&\sup_{0<\veps\leq1}\, \left\|\vec{V}_{\varepsilon }-\vec{V}_{\varepsilon ,M}\right\|_{L^{\infty}([0,T];H^{-s})}\longrightarrow 0\qquad
\forall\,s>\frac{3}{2\,\g}\,.
\end{split}
\end{equation}
Moreover, for any $M>0$, the couple $(\vrm,\vec V_{\veps ,M})$ satisfies the approximate wave equations
\begin{equation}\label{eq:approx wave_G}
\left\{\begin{array}{l}
       \veps^m\,\d_t \vrm \,+\,\,\div\vec{V}_{\veps ,M}\,=\,0 \\[1ex]
       \veps^m\,\d_t\vec{V}_{\veps ,M}\,+\veps^{m-1}\,e_{3}\times \vec{V}_{\veps ,M}+\,\nabla_x \vrm\,=\,\veps^m\,\vec f_{\veps ,M}\,+\veps^{2(m-n)} \vec g_{\veps,M} \, ,
       \end{array}
\right.
\end{equation}
where $(\vec f_{\veps ,M})_{\veps}$ and $(\vec g_{\veps ,M})_{\veps}$ are families of smooth (in the space variables) functions satisfying, for any $s\geq0$, the uniform bounds
\begin{equation}\label{eq:approx force_G}
\sup_{0<\veps\leq1}\, \left\|\vec f_{\veps ,M}\right\|_{L^{2}([0,T];H^{s})}\,+\,\sup_{0<\veps\leq1}\,\left\|\vec g_{\veps ,M}\right\|_{L^{\infty}([0,T];H^{s})}\,\leq\, C(s,M)\,,
\end{equation}
where the constant $C(s,M)$ depends on the fixed values of $s\geq 0$ and $M>0$, but not on $\veps>0$.
\end{proposition}

\begin{proof}
Thanks to characterization \eqref{eq:LP-Sob} of $H^{s}$, properties \eqref{eq:approx var_G} are straightforward consequences of the uniform bounds establish in Subsection \ref{ss:unif-est_G}. For instance, let us consider the functions $\vr^{(1)}_\veps$: when $\g\geq2$, owing to Remark \ref{r:g>2} one has
$\big(\vr^{(1)}_\veps\big)_\veps\,\subset\,L^\infty_T(L^2)$, and then we use estimate \eqref{est:sobolev} from Section \ref{app:LP}.
When $1<\g<2$, instead, we first apply the dual Sobolev embedding (see Theorem \ref{app:thm_dual_Sob}) to infer that $\big(\vr^{(1)}_\veps\big)_\veps\,\subset\,L^\infty_T(H^{-\s})$,
with $\s\,=\,\s(\g)\,=\,3\big(1/\g-1/2\big)$, and then we use \eqref{est:sobolev} again. The bounds for the momentum
$\big(\vec V_\veps\big)_\veps$ can be deduced by a similar argument, after observing that $2\g/(\g+1)<2$ always.

Next, applying the operator ${S}_{M}$ to \eqref{eq:wave_syst_G} immediately gives us system \eqref{eq:approx wave_G}, where we have set 
\begin{equation*}
\vec f_{\veps ,M}:={S}_{M}\vec f_\veps \qquad \text{ and }\qquad \vec g_{\veps ,M}:={S}_{M}\vec g_\veps\,.
\end{equation*}
Thanks to Lemma \ref{l:source_bounds_G} and \eqref{eq:LP-Sob}, it is easy to verify inequality \eqref{eq:approx force_G}. 
\qed
\end{proof}

\medbreak
At this point, we will need also the following important decomposition for the momentum vector fields $\vec V_{\veps,M}$ and their $\curl$.
\begin{proposition} \label{p:prop dec_G}
For any $M>0$ and any $\veps\in\,]0,1]$, the following decompositions hold true:
\begin{equation*}
\vec{V}_{\veps ,M}\,=\,
\veps^{2(m-n)}\vec{t}_{\veps ,M}^{1}+\vec{t}_{\veps ,M}^{2}\qquad\mbox{ and }\qquad
\curl \vec{V}_{\veps ,M}=\veps^{2(m-n)}\vec{T}_{\veps ,M}^{1}+\vec{T}_{\veps ,M}^{2}\,,
\end{equation*}
where, for any $T>0$ and $s\geq 0$, one has 
\begin{align*}
&\left\|\vec{t}_{\veps ,M}^{1}\right\|_{L^{2}([0,T];H^{s})}+\left\|\vec{T}_{\veps ,M}^{1}\right\|_{L^{2}([0,T];H^{s})}\leq C(s,M) \\
&\left\|\vec{t}_{\veps ,M}^{2}\right\|_{L^{2}([0,T];H^{1})}+\left\|\vec{T}_{\veps ,M}^{2}\right\|_{L^{2}\left([0,T];L^2\right)}\leq C\,,
\end{align*}
for suitable positive constants $C(s,M)$ and $C$, which are uniform with respect to $\veps\in\,]0,1]$.
\end{proposition}

\begin{proof}
We decompose $\vec{V}_{\veps ,M}\,=\,\veps^{2(m-n)}\vec t_{\veps,M}^{1}\,+\,\vec t_{\veps,M}^{2}$, where we define
\begin{equation} \label{eq:t-T_G}
\vec{t}_{\veps,M}^{1}\,:=\,{S}_{M}\left(\frac{\vrho_\veps -1}{\veps^{2(m-n)}}\, \vec{u}_{\veps}\right) \qquad\mbox{ and }\qquad
\vec{t}_{\veps,M}^{2}\,:=\,{S}_{M} \vec{u}_{\veps}\,.
\end{equation}
The decomposition of $\curl \vec V_{\veps,M}$ follows after setting $\vec T_{\veps,M}^j\,:=\,\curl \vec t_{\veps,M}^j$, for $j=1,2$.

We have to prove uniform bounds for all those terms, by using the estimates established in Subsection \ref{ss:unif-est_G} above.
First of all, we have that $\big(\vu_\veps\big)_\veps\,\subset\,L^2_T(H^1)$, for any $T>0$ fixed. Then, we immediately gather the sought bounds for the vector fields $\vec t_{\veps,M}^2$ and $\vec T_{\veps,M}^2$.

For the families of $\vec t_{\veps,M}^1$ and $\vec T_{\veps,M}^1$, instead, we have to use the bounds provided by \eqref{rr1_G} and (when $\g\geq2$)
Remark \ref{r:g>2}. In turn, we see that for any $T>0$,
\[
\left(\frac{\vrho_\veps -1}{\veps^{2(m-n)}}\, \vec{u}_{\veps}\right)\,\subset\,L^2_T(L^1+L^2+L^{6\g/(\g+6)})\,\hookrightarrow\,
L^2_T(H^{-\s})\,,
\]
for some $\s>0$ large enough. Therefore, the claimed bounds follow thanks to the regularising effect of the operators $S_M$. The proof of the proposition
is thus completed.
\qed
\end{proof}


\subsection{Analysis of the convection} \label{ss:convergence_G}
In this subsection we show the convergence of the convective term. 
The first step is to reduce its analysis to the case of smooth vector fields $\vec{V}_{\veps ,M}$.

\begin{lemma} \label{lem:convterm_G}
Let $T>0$. For any $\vec{\psi}\in C_c^\infty\bigl([0,T[\,\times\Omega;\R^3\bigr)$, we have 
\begin{equation*}
\lim_{M\rightarrow +\infty} \limsup_{\veps \rightarrow 0^+}\left|\int_{0}^{T}\int_{\Omega} \vrho_\veps\,\vec{u}_\veps\otimes \vec{u}_\veps: \nabla_{x}\vec{\psi}\, \dxdt-
\int_{0}^{T}\int_{\Omega} \vec{V}_{\veps ,M}\otimes \vec{V}_{\veps,M}: \nabla_{x}\vec{\psi}\, \dxdt\right|=0\, .
\end{equation*}
\end{lemma}

\begin{proof}
The proof is very similar to the one of Lemma \ref{lem:convterm} from Chapter \ref{chap:multi-scale_NSF}, for this reason we just outline it.

One starts by using the decomposition $\vr_\veps\,=\,1\,+\,\veps^{2(m-n)}\,R_\veps$ to reduce (owing to the uniform bounds of Subsection \ref{ss:unif-est_G}) the convective term to the ``homogeneous counterpart'': for any test function $\vec\psi\in C^\infty_c\big(\R_+\times\Omega;\R^3\big)$, one has
\[
\lim_{\veps \rightarrow 0^+}\left|\int_{0}^{T}\int_{\Omega} \vrho_\veps\,\vec{u}_\veps\otimes \vec{u}_\veps: \nabla_{x}\vec{\psi}\, \dxdt-
\int_{0}^{T}\int_{\Omega}\vec{u}_\veps\otimes\vec{u}_\veps:\nabla_{x}\vec{\psi}\,\dxdt\right|\,=\,0\,.
\]
Notice that, here, one has to use that $\g\geq 3/2$.

After that, we write
$\vu_\veps\,=\,S_M \vu_\veps\,+\,(\Id-S_M)\vu_\veps\,=\,\vec t^2_{\veps,M}\,+\,(\Id-S_M)\vu_\veps$.
Using Proposition \ref{p:prop dec_G} and the fact that $\left\|(\Id-{S}_{M})\,\vec{u}_\veps\right\|_{L_{T}^{2}(L^{2})}\leq C\,2^{-M}\|\nabla_x\vec u_\veps\|_{L^2_T(L^2)}\,\leq C\,2^{-M}$, which holds in view of estimate \eqref{est:sobolev} from Section \ref{app:LP} and the uniform bound \eqref{bound_for_vel},
one can conclude.
\qed
\end{proof}

\medbreak

From now on, for notational convenience, we  generically denote by $\mc{R}_{\veps ,M}$ any remainder term, that again is any term satisfying the property
\begin{equation} \label{eq:remainder_G}
\lim_{M\rightarrow +\infty}\limsup_{\veps \rightarrow 0^{+}}\left|\int_{0}^{T}\int_{\Omega}\mc{R}_{\veps ,M}\cdot \vec{\psi}\, \dxdt\right|=0\, ,
\end{equation}
for all test functions $\vec{\psi}\in C_c^\infty\bigl([0,T[\,\times\Omega;\R^3\bigr)$ lying in the kernel of the singular perturbation operator,
namely (in view of Proposition \ref{p:limitpoint_G}) such that
\begin{equation} \label{eq:test-f_G}
\vec\psi\in C_c^\infty\big([0,T[\,\times\Omega;\R^3\big)\qquad\qquad \mbox{ with }\qquad \div\vec\psi=0\quad\mbox{ and }\quad \d_3\vec\psi=0\,.
\end{equation}
Notice that, in order to pass to the limit in the weak formulation of the momentum equation and derive the limit system, it is enough to use test functions $\vec\psi$ as above.

Thus, for $\vec\psi$ as in \eqref{eq:test-f_G}, 
we have to pass to the limit in the term 
\begin{align*}
-\int_{0}^{T}\int_{\Omega} \vec{V}_{\veps ,M}\otimes \vec{V}_{\veps ,M}: \nabla_{x}\vec{\psi}\,&=\,\int_{0}^{T}\int_{\Omega} \div\left(\vec{V}_{\veps ,M}\otimes
\vec{V}_{\veps ,M}\right) \cdot \vec{\psi}\,.
\end{align*}
Notice that the integration by parts above is well-justified, since all the quantities inside the integrals are now smooth with respect to the space variable. Owing to the structure of the test function, and
resorting to the notation \eqref{eq:decoscil} setted in the introductory part, we remark that we can write
$$
\int_{0}^{T}\int_{\Omega} \div\left(\vec{V}_{\veps ,M}\otimes \vec{V}_{\veps ,M}\right) \cdot \vec{\psi}\,=\,
\int_{0}^{T}\int_{\R^2} \left(\mc{T}_{\veps ,M}^{1}+\mc{T}_{\veps, M}^{2}\right)\cdot\vec{\psi}^h\,,
$$
where we have defined the terms
\begin{equation} \label{def:T1-2_G}
\mc T^1_{\veps,M}\,:=\, \divh\left(\langle \vec{V}_{\veps ,M}^{h}\rangle\otimes \langle \vec{V}_{\veps ,M}^{h}\rangle\right)\qquad \mbox{ and }\qquad
\mc T^2_{\veps,M}\,:=\, \divh\left(\langle \dbtilde{\vec{V}}_{\veps ,M}^{h}\otimes \dbtilde{\vec{V}}_{\veps ,M}^{h}\rangle \right)\,.
\end{equation}

In the next two paragraphs, we will deal with each one of those terms separately. We borrow most of the arguments from Chapter \ref{chap:multi-scale_NSF}
(see also \cite{F-G-GV-N}, \cite{F_2019} for a similar approach). However, the special structure of the gravity force will play a key role here,
in order (loosely speaking) to compensate the stronger singularity due to our scaling $2n>m$. Finally, we point out that, in what follows, all the equalities (which will involve the derivative in time) will hold in the sense of distributions.

\subsubsection{Convergence of the vertical averages}\label{sss:term1_G}
We start by dealing with $\mc T^1_{\veps,M}$. It is standard to write
\begin{align}
\mc{T}_{\veps ,M}^{1}\,&=\,\divh\left(\langle \vec{V}_{\veps ,M}^{h}\rangle\otimes \langle \vec{V}_{\veps ,M}^{h}\rangle\right)=
\divh\langle \vec{V}_{\veps ,M}^{h}\rangle\, \langle \vec{V}_{\veps ,M}^{h}\rangle+\langle \vec{V}_{\veps ,M}^{h}\rangle \cdot \nabla_{h}\langle \vec{V}_{\veps ,M}^{h}\rangle \label{eq:T1_G} \\
&=\divh\langle \vec{V}_{\veps ,M}^{h}\rangle\, \langle \vec{V}_{\veps ,M}^{h}\rangle+\frac{1}{2}\, \nabla_{h}\left(\left|\langle \vec{V}_{\veps ,M}^{h}\rangle\right|^{2}\right)+
\curlh\langle \vec{V}_{\veps ,M}^{h}\rangle\,\langle \vec{V}_{\veps ,M}^{h}\rangle^{\perp}\,. \nonumber
\end{align}
Notice that the second term is a perfect gradient, so it vanishes when tested against divergence-free test functions. Hence, we can treat it as
a remainder, in the sense of \eqref{eq:remainder_G}.

For the first term in the second line of \eqref{eq:T1_G}, instead, we take advantage of system \eqref{eq:approx wave_G}: averaging the first equation with respect to $x^{3}$ and multiplying it by $\langle \vec{V}^h_{\veps ,M}\rangle$, we arrive at
$$
\divh\langle \vec{V}_{\veps ,M}^{h}\rangle\,\langle \vec{V}_{\veps ,M}^{h}\rangle\,=\,-\veps^m\d_t\langle \vrm\rangle \langle \vec{V}_{\veps ,M}^{h}\rangle\,=\,
\veps^m\langle \vrm\rangle \d_t \langle \vec{V}_{\veps ,M}^{h}\rangle +\mc{R}_{\veps ,M}\,.
$$
We remark that the term presenting the total derivative in time is in fact a remainder, thanks to the factor $\veps^m$ in front of it.
Now, we use the horizontal part of \eqref{eq:approx wave_G}, where we first take the vertical average and then multiply by $\langle \vrm\rangle$:
since $m>1$, we gather
\begin{align*}
&\veps^m\langle \vrm \rangle \d_t \langle \vec{V}_{\veps ,M}^{h}\rangle \\
&\qquad\quad=
- \langle \vrm\rangle \nabla_{h}\langle \vrm \rangle-
\veps^{m-1}\langle \vrm \rangle\langle \vec{V}_{\veps ,M}^{h}\rangle^{\perp} 
+\veps^{m}\langle \vrm \rangle \langle \vec f_{\veps ,M}^{h}\rangle+\veps^{2(m-n)}\langle \vrm \rangle \langle \vec g_{\veps ,M}^{h}\rangle\\
&\qquad\quad=-\veps^{m-1}\langle \vrm \rangle\langle \vec{V}_{\veps ,M}^{h}\rangle^{\perp}+\mc{R}_{\veps ,M}\, ,
\end{align*}
where we have repeatedly exploited the properties proved in Proposition \ref{p:prop approx_G} and we have included in the remainder term also the perfect gradient.
Inserting this relation into \eqref{eq:T1_G} yields
\begin{equation*}
\mc{T}_{\veps ,M}^{1}=  \left(\curlh\langle \vec{V}_{\veps ,M}^{h}\rangle\,-\,\veps^{m-1}\langle \vrm \rangle\right)   
\langle\vec{V}_{\veps,M}^{h}\rangle^{\perp}+\mc{R}_{\veps,M}\,.
\end{equation*}
Observe that the first addendum appearing in the right-hand side of the previous relation is bilinear. Thus, in order to pass to the limit in it, one needs some strong convergence properties. As a matter of fact, in the next computations we will work on the regularised wave system \eqref{eq:approx wave_G} to show that the quantity
\[
\Gamma_{\veps, M}:=\curlh\langle \vec{V}_{\veps ,M}^{h}\rangle\,-\,\veps^{m-1}\langle \vrm \rangle
\]
is \emph{compact} in some suitable space. In particular, as $m>1$, also $\curlh\langle \vec{V}_{\veps ,M}^{h}\rangle$ is compact.

In order to see this, we write the vertical average of the first equation in \eqref{eq:approx wave_G} as
\begin{equation*}
\veps^{2m-1}\,\d_t \langle \vrm \rangle\,+\,\veps^{m-1}\div_{h} \langle \vec{V}_{\veps ,M}^{h}\rangle\,=0\,.
\end{equation*}
Next, we take the vertical average of the horizontal components of \eqref{eq:approx wave_G} and then we apply $\curlh$: one obtains
\begin{equation*}
\veps^m\,\d_t\curlh\langle \vec{V}_{\veps ,M}^{h}\rangle\,+\veps^{m-1}\,\divh\langle \vec{V}_{\veps ,M}^{h}\rangle\, =\,\veps^m \curlh\langle\vec f_{\veps ,M}^{h}\rangle+\veps^{2(m-n)} \curlh\langle\vec g_{\veps ,M}^{h}\rangle\, .
\end{equation*}
At this point, we recall the definition \eqref{def_f-g} of $\vec g_\veps$, and we see that $\curlh\langle\vec g_{\veps ,M}^{h}\rangle\,\equiv\,0$.
This property is absolutely fundamental, since it allows to erase the last term in the previous relation, which otherwise would have represented an obstacle to get compactness for $\Gamma_{\veps,M}$. Indeed, thanks to this observation, we can sum up the last two equations to get
\begin{equation} \label{eq:gamma_G}
\d_{t}\Gamma_{\veps,M}\,=\,\curlh\langle \vec f_{\veps ,M}^{h}\rangle\, .
\end{equation}
Using estimate \eqref{eq:approx force_G} in Proposition \ref{p:prop approx_G}, we discover that, for any $M>0$ fixed, the family 
$\left(\d_{t}\,\Gamma_{\veps,M}\right)_{\veps}$ is uniformly bounded (with respect to $\veps$) in e.g. the space $L_{T}^{2}(L^{2})$. 
On the other hand, we have that, again for any $M>0$ fixed,
the sequence $(\Gamma_{\veps,M})_{\veps}$ is uniformly bounded (with respect to $\veps$) e.g. in the space $L_{T}^{2}(H^{1})$.
Since the embedding $H_{\rm loc}^{1}\hookrightarrow L_{\rm loc}^{2}$ is compact, the Aubin-Lions Theorem implies that, for any $M>0$ fixed, the family $(\Gamma_{\veps,M})_{\veps}$ is compact
in $L_{T}^{2}(L_{\rm loc}^{2})$. Then, up to extraction of a suitable subsequence (not relabelled here), that family converges strongly to a tempered distribution $\Gamma_M$ in the same space. 

Now, we have that $\Gamma_{\veps ,M}$ converges strongly to $\Gamma_M$ in $L_{T}^{2}(L_{\rm loc}^{2})$ and $\langle \vec{V}_{\veps ,M}^{h}\rangle$ converges weakly to
$\langle \vec{V}_{M}^{h}\rangle$ in $L_{T}^{2}(L_{\rm loc}^{2})$ (owing to Proposition \ref{p:prop dec_G}, for instance). Then, we deduce 
\begin{equation*}
\Gamma_{\veps,M}\langle \vec{V}_{\veps ,M}^{h}\rangle^{\perp}\longrightarrow \Gamma_M \langle \vec{V}_{M}^{h}\rangle^{\perp}\qquad \text{ in }\qquad \mc{D}^{\prime}\big(\R_+\times\R^2\big)\,.
\end{equation*}
Observe that, by definition of $\Gamma_{\veps,M}$, we must have $\Gamma_M=\curlh\langle \vec{V}_{M}^{h}\rangle$. On the other hand, owing to Proposition \ref{p:prop dec_G} and \eqref{eq:t-T_G}, we know that $\langle \vec{V}_{M}^{h}\rangle= \lan{S}_{M}\vec{U}^{h}\ran$.
Therefore,  in the end we have proved that, for $m>1$ and $m+1\geq 2n >m$,
one has the convergence (at any $M\in\N$ fixed, when $\veps\ra0^+$)
\begin{equation} \label{eq:limit_T1_G}
\int_{0}^{T}\int_{\R^2}\mc{T}_{\veps ,M}^{1}\cdot\vec{\psi}^h\dx^h\dt\,\longrightarrow\,
\int^T_0\int_{\R^2}\curlh\lan{S}_{M}\vec{U}^{h}\ran\; \lan{S}_{M}(\vec{U}^{h})^{\perp}\ran\cdot\vec\psi^h\dx^h\dt\, ,
\end{equation}
for any $T>0$ and for any test-function $\vec \psi$ as in \eqref{eq:test-f_G}.

\subsubsection{Vanishing of the oscillations}\label{sss:term2_G}

We now focus on the term $\mc{T}_{\veps ,M}^{2}$, defined in \eqref{def:T1-2_G}. Recall that $m>1$. In what follows, we consider separately the two cases  $m+1>2n$ and $m+1=2n$. As a matter of fact, in the case when $m+1=2n$, a bilinear term involving $\vec g_{\veps,M}$ has no power of $\veps$ in front of it, so it is not clear that it converges to $0$ and, in fact, it might persist in the limit, giving rise to an additional term in the target system. To overcome this issue and show that this actually does not happen, we deeply exploit the structure of the wave system to recover a quantitative smallness for that term (namely, in terms of positive powers of $\veps$).

\subsubsection*{The case $m+1>2n$}\label{sss:term2_osc}

Starting from the definition of $\mc T_{\veps,M}^2$, the same computations as above yield
\begin{align}
\mc{T}_{\veps ,M}^{2}\,
&=\,\langle \divh (\dbtilde{\vec{V}}_{\veps ,M}^{h})\;\;\dbtilde{\vec{V}}_{\veps ,M}^{h}\rangle+\frac{1}{2}\, \langle \nabla_{h}| \dbtilde{\vec{V}}_{\veps ,M}^{h}|^{2} \rangle+
\langle \curlh\dbtilde{\vec{V}}_{\veps ,M}^{h}\,\left( \dbtilde{\vec{V}}_{\veps ,M}^{h}\right)^{\perp}\rangle\, . \label{eq:T2_G} 
\end{align}

Let us now introduce the quantities
$$
\dbtilde{\Phi}_{\veps ,M}^{h}\,:=\,( \dbtilde{\vec{V}}_{\veps ,M}^{h})^{\perp}-\d_{3}^{-1}\nabla_{h}^{\perp}\dbtilde{\vec{V}}_{\veps ,M}^{3}\qquad\mbox{ and }\qquad
\dbtilde{\omega}_{\veps ,M}^{3}\,:=\,\curlh \dbtilde{\vec{V}}_{\veps ,M}^{h}\,.
$$
Then we can write
\begin{equation*}
\left( \curl \dbtilde{\vec{V}}_{\veps ,M}\right)^{h}\,=\,\d_3 \dbtilde{\Phi}_{\veps ,M}^{h}\qquad \text{ and }\qquad
\left( \curl \dbtilde{\vec{V}}_{\veps ,M}\right)^{3}\,=\,\dbtilde{\omega}_{\veps ,M}^{3}\,.
\end{equation*}
In addition, from the momentum equation in \eqref{eq:approx wave_G}, where we take the mean-free part and then the $\curl$, we deduce the equations
\begin{equation} \label{eq:eq momentum term2_G}
\begin{cases}
\veps^{m}\d_t\dbtilde{\Phi}_{\veps ,M}^{h}-\veps^{m-1}\dbtilde{\vec{V}}_{\veps ,M}^{h}=\veps^m\left(\d_{3}^{-1}\curl\dbtilde{\vec f}_{\veps ,M} \right)^{h}+\veps^{2(m-n)}\left(\d_{3}^{-1}\curl\dbtilde{\vec g}_{\veps ,M} \right)^{h}\\[1ex]
\veps^{m}\d_t\dbtilde{\omega}_{\veps ,M}^{3}+\veps^{m-1}\divh\dbtilde{\vec{V}}_{\veps ,M}^{h}=\veps^m\,\curlh\dbtilde{\vec f}_{\veps ,M}^{h}\, .
\end{cases}
\end{equation}
Making use of the relations above, recalling the definitions in \eqref{def_f-g}, and thanks to Propositions \ref{p:prop approx_G} and \ref{p:prop dec_G}, we can write
\begin{equation}\label{rel_oscillations}
\begin{split}
\curlh\dbtilde{\vec{V}}_{\veps ,M}^{h}\;\left(\dbtilde{\vec{V}}_{\veps ,M}^{h}\right)^{\perp}&=\dbtilde{\omega}_{\veps ,M}^{3}\left(\dbtilde{\vec{V}}_{\veps ,M}^{h}\right)^{\perp}\,\\
&=\veps \d_t\!\left( \dbtilde{\Phi}_{\veps ,M}^{h}\right)^{\perp}\dbtilde{\omega}_{\veps ,M}^{3}-
\veps\dbtilde{\omega}_{\veps ,M}^{3}\left(\left(\d_{3}^{-1}\curl\dbtilde{\vec f}_{\veps ,M}\right)^{h}\right)^\perp\\
&\qquad\qquad\qquad\qquad\qquad\qquad
-\veps^{m+1-2n}\, \dbtilde{\omega}_{\veps ,M}^{3}\left(\left(\d_{3}^{-1}\curl\dbtilde{\vec g}_{\veps ,M}\right)^{h}\right)^\perp  \\
&=-\veps \left( \dbtilde{\Phi}_{\veps ,M}^{h}\right)^{\perp}\d_t\dbtilde{\omega}_{\veps ,M}^{3}+\mc{R}_{\veps ,M}=
\left( \dbtilde{\Phi}_{\veps ,M}^{h}\right)^{\perp}\,\divh\dbtilde{\vec{V}}_{\veps ,M}^{h}+\mc{R}_{\veps ,M}\, .
\end{split}
\end{equation}
We point out that, thanks to the scaling $m+1>2n$, we could include in the remainder also the last term appearing in the second equality, which was of order $O(\veps^{m+1-2n})$. 

Hence, putting the gradient term into $\mc R_{\veps,M}$, from \eqref{eq:T2_G} we arrive at 
\begin{align*}
\mc{T}_{\veps ,M}^{2}\,&=\,\langle \divh\dbtilde{\vec{V}}_{\veps ,M}^{h}\,\left(\dbtilde{\vec{V}}_{\veps ,M}^{h}+\left(\dbtilde{\Phi}_{\veps ,M}^{h}\right)^{\perp}\right) \rangle+\mc{R}_{\veps ,M} \\
&=\,\langle \div \dbtilde{\vec{V}}_{\veps ,M}\left(\dbtilde{\vec{V}}_{\veps ,M}^{h}+\left(\dbtilde{\Phi}_{\veps ,M}^{h}\right)^{\perp}\right) \rangle -
\langle \d_3 \dbtilde{\vec{V}}_{\veps ,M}^{3}\left(\dbtilde{\vec{V}}_{\veps ,M}^{h}+\left(\dbtilde{\Phi}_{\veps ,M}^{h}\right)^{\perp}\right) \rangle+\mc{R}_{\veps ,M}\, .
\end{align*}

At this point, the computations mainly follow the same lines of \cite{F-G-GV-N} (see also \cite{F_2019}).
First of all, we notice that, in the last line, the second term on the right-hand side is another remainder. Indeed, using the definition of the function $\dbtilde{\Phi}_{\veps ,M}^{h}$ and the fact
that the test-function $\vec\psi$ does not depend on $x^3$, one has
\begin{equation*}
\begin{split}
\d_3 \dbtilde{\vec{V}}_{\veps ,M}^{3}\left(\dbtilde{\vec{V}}_{\veps ,M}^{h}+\left(\dbtilde{\Phi}_{\veps ,M}^{h}\right)^{\perp}\right)&=\d_3 \left(\dbtilde{\vec{V}}_{\veps ,M}^{3}\left(\widetilde{\vec{V}}_{\veps ,M}^{h}+\left(\dbtilde{\Phi}_{\veps ,M}^{h}\right)^{\perp}\right)\right) - \dbtilde{\vec{V}}_{\veps ,M}^{3}\, \d_3\left(\dbtilde{\vec{V}}_{\veps ,M}^{h}+\left(\dbtilde{\Phi}_{\veps ,M}^{h}\right)^{\perp}\right)\\
&=\mc{R}_{\veps ,M}-\frac{1}{2}\nabla_{h}\left|\dbtilde{\vec{V}}_{\veps ,M}^{3}\right|^{2}=\mc{R}_{\veps ,M}\, .
\end{split}
\end{equation*}
Next, in order to deal with the first term, we use the first equation in \eqref{eq:approx wave_G} to obtain
\begin{equation*}
\begin{split}
\div \dbtilde{\vec{V}}_{\veps ,M}\left(\dbtilde{\vec{V}}_{\veps ,M}^{h}+\left(\dbtilde{\Phi}_{\veps ,M}^{h}\right)^{\perp}\right)&=-\veps^{m} \d_t \dbtilde{\vrho}^{(1)}_{\veps ,M}\left(\dbtilde{\vec{V}}_{\veps ,M}^{h}+\left(\dbtilde{\Phi}_{\veps ,M}^{h}\right)^{\perp}\right)+\mc{R}_{\veps ,M}\\
&=\veps^{m} \dbtilde{\vrho}_{\veps,M}^{(1)}\, \d_t\left(\dbtilde{\vec{V}}_{\veps ,M}^{h}+\left(\dbtilde{\Phi}_{\veps ,M}^{h}\right)^{\perp}\right)+\mc{R}_{\veps ,M}\, .
\end{split}
\end{equation*}
Now, equations \eqref{eq:approx wave_G} and \eqref{eq:eq momentum term2_G} immediately yield that
\begin{equation*}
\veps^{m}\dbtilde{\vrho}^{(1)}_{\veps ,M}\, \d_t\left(\dbtilde{\vec{V}}_{\veps ,M}^{h}+\left(\dbtilde{\Phi}_{\veps ,M}^{h}\right)^{\perp}\right)=
\mc{R}_{\veps ,M}-\dbtilde{\vrho}_{\veps ,M}^{(1)}\, \nabla_{h}\dbtilde{\vrho}^{(1)}_{\veps ,M}=
\mc{R}_{\veps ,M}-\frac{1}{2}\nabla_{h}\left|\dbtilde{\vrho}_{\veps ,M}^{(1)}\right|^{2}=\mc{R}_{\veps ,M}\,.
\end{equation*}

This relation finally implies that $\mc{T}_{\veps ,M}^{2}\,=\,\mc R_{\veps,M}$ is a remainder, in the sense of relation \eqref{eq:remainder_G}:
for any $T>0$ and any test-function $\vec \psi$ as in \eqref{eq:test-f_G},
one has the convergence
(at any $M\in\N$ fixed, when $\veps\ra0^{+}$)
\begin{equation} \label{eq:limit_T2_G}
\int_{0}^{T}\int_{\R^2}\mc{T}_{\veps ,M}^{2}\cdot\vec{\psi}^h\dx^h\dt\,\longrightarrow\,0\,.
\end{equation}

\subsubsection*{The case $m+1=2n$}\label{sss:term2_osc_bis}
In the case $m+1=2n$, most of the previous computations may be reproduced exactly in the same way.
The only (fundamental) change concerns relation \eqref{rel_oscillations}: since now $m+1-2n=0$, that equation reads
\begin{equation}\label{rel_oscillations_bis}
\begin{split}
\curlh\dbtilde{\vec{V}}_{\veps ,M}^{h}\;\left(\dbtilde{\vec{V}}_{\veps ,M}^{h}\right)^{\perp}\,=\,\left( \dbtilde{\Phi}_{\veps ,M}^{h}\right)^{\perp}\,\divh\dbtilde{\vec{V}}_{\veps ,M}^{h}\,- \dbtilde{\omega}_{\veps ,M}^{3}\left(\left(\d_{3}^{-1}\curl\dbtilde{\vec g}_{\veps ,M}\right)^{h}\right)^\perp+\mc{R}_{\veps ,M}\,,
\end{split}
\end{equation}
and, repeating the same computations performed for $\mc T^2_{\veps, M}$ in the previous paragraph, we have
\begin{equation*}\label{T^2-bis}
\mc T^2_{\veps, M}= \mc R_{\veps, M}-\langle\dbtilde{\omega}_{\veps ,M}^{3}\left(\left(\d_{3}^{-1}\curl\dbtilde{\vec g}_{\veps ,M}\right)^{h}\right)^\perp \rangle\, .
\end{equation*}
Hence, the main difference with respect to the previous case is that we have to take care of the term $\widetilde{\omega}_{\veps ,M}^{3}\left(\left(\d_{3}^{-1}\curl\dbtilde{\vec g}_{\veps ,M}\right)^{h}\right)^\perp $, which is non-linear and of order $O(1)$, so it may potentially
give rise to oscillations which persist in the limit.

In order to show that this does not happen, we make use of definition \eqref{def_f-g} of $\vec g_{\veps}$ to compute
\begin{align*}
\left(\curl\dbtilde{\vec g}_{\veps ,M}\right)^{h,\perp}\,&=\,\left(\curl \left(\dbtilde{\vrho}_{\veps, M}^{(1)}\nabla_x G-\nabla_x\dbtilde{\Pi}_{\veps,M}\right)\right)^{h,\perp} \\
&=\,
\begin{pmatrix}
-\d_2 \dbtilde{\vrho}^{(1)}_{\veps, M} \\ 
\d_1 \dbtilde{\vrho}^{(1)}_{\veps ,M} \\ 
0
\end{pmatrix}^{h,\perp}\,=-\,\nabla_h \dbtilde{\vrho}_{\veps,M}^{(1)}\, .
\end{align*}
From this relation, in turn we get
\begin{equation}\label{T^2}
\mc T^2_{\veps, M}\,=\,\mc R_{\veps, M}\,+\,\langle \dbtilde{\omega}_{\veps ,M}^{3}\,  \d_3^{-1}\nabla_h \dbtilde{\vrho}_{\veps,M}^{(1)} \rangle \, .
\end{equation} 
Now, we have to employ the potential part of the momentum equation in \eqref{eq:approx wave_G}, which has not been used so far. Taking the oscillating
component of the solutions, we obtain 
\begin{equation*}
\nabla_h \dbtilde{\vrho}_{\veps,M}^{(1)}\,=-\, \veps^m\,\d_t\dbtilde{\vec{V}}^h_{\veps ,M}\,-\veps^{m-1} (\dbtilde{\vec{V}}^h_{\veps ,M})^\perp+\veps^m\,\dbtilde{\vec f}^h_{\veps ,M}\,+\veps^{2(m-n)} \dbtilde{\vec g}^h_{\veps,M}= -\, \veps^m\,\d_t\dbtilde{\vec{V}}^h_{\veps ,M}\,+ \mc R_{\veps,M}\,.
\end{equation*}
Inserting this relation into \eqref{T^2} and using \eqref{eq:eq momentum term2_G}, we finally gather
\begin{equation*}
\mc T^2_{\veps, M}=-\veps^m \langle \dbtilde{\omega}_{\veps ,M}^{3}\, \d_t\d_3^{-1}\dbtilde{\vec{V}}^h_{\veps ,M}  \rangle +\mc R_{\veps,M}=
\veps^m \langle \d_t \dbtilde{\omega}_{\veps ,M}^{3}\, \d_3^{-1}\dbtilde{\vec{V}}^h_{\veps ,M}  \rangle +\mc R_{\veps,M}=\mc R_{\veps,M}\, ,
\end{equation*}
because we have taken $m>1$.

This relation finally implies that, also in the case when $m+1=2n$, $\mc{T}_{\veps ,M}^{2}$ is a remainder: for any $T>0$ and any test-function $\vec \psi$ as in \eqref{eq:test-f_G}, one has the convergence \eqref{eq:limit_T2_G}.

\subsection{The limit system} \label{ss:limit_G} 
Thanks to the computations of the previous subsections, we can now pass to the limit in equation \eqref{weak-mom_G}. Recall that $m>1$ and $m+1\geq 2n >m$ here.

To begin with, we take a test-function $\vec\psi$ as in \eqref{eq:test-f_G}, specifically
\begin{equation} \label{eq:test-2}
\vec{\psi}=\big(\nabla_{h}^{\perp}\phi,0\big)\,,\qquad\qquad\mbox{ with }\qquad \phi\in C_c^\infty\big([0,T[\,\times\R^2\big)\,,\quad \phi=\phi(t,x^h)\,.
\end{equation}
We point out that since all the integrals will be made on $\R^2$ (in view of the choice of the test functions in \eqref{eq:test-2} above), we can work on the domain $\Omega=\R^2 \times \, ]0,1[\, $.

In addition, for such $\vec\psi$ as in \eqref{eq:test-2}, all the gradient terms vanish identically, as well as all the contributions
due to the vertical component of the equation. In particular, we do not see any contribution of the pressure and gravity terms: equation \eqref{weak-mom_G} becomes
\begin{align}
\int_0^T\!\!\!\int_{\Omega}  
& \left( -\vre \ue^h \cdot \partial_t \vec\psi^h -\vre \ue^h\otimes\ue^h  : \nabla_h \vec\psi^h
+ \frac{1}{\ep}\vre\big(\ue^{h}\big)^\perp\cdot\vec\psi^h\right)\, \dxdt \label{eq:weak_to_conv_G}\\
&\qquad\qquad\qquad\qquad =-\int_0^T\!\!\!\int_{\Omega} 
\mbb{S}(\nabla_x\vec\ue): \nabla_x \vec\psi\dxdt+
\int_{\Omega}\vrez \uez  \cdot \vec\psi(0,\cdot)\dx\,. \nonumber
\end{align}

Making use of the uniform bounds of Subsection \ref{ss:unif-est_G}, we can pass to the limit in the $\d_t$ term and in the viscosity term.
Moreover, our assumptions imply that $\vrho_{0,\veps}\vec{u}_{0,\veps}\rightharpoonup \vec{u}_0$ in e.g. $L_{\rm loc}^2$. 
Next, the Coriolis term can be arranged in a standard way: using the structure of $\vec\psi$ and the mass equation \eqref{weak-con_G}, we can write
\begin{align*}
\int_0^T\!\!\!\int_{\Omega}\frac{1}{\ep}\vre\big(\ue^{h}\big)^\perp\cdot\vec\psi^h\,&=\,\int_0^T\!\!\!\int_{\mbb{R}^2}\frac{1}{\ep}\langle\vre \ue^{h}\rangle \cdot \nabla_{h}\phi\,=\,
-\veps^{m-1}\int_0^T\!\!\!\int_{\mbb{R}^2}\langle \vrho^{(1)}_\veps\rangle\, \d_t\phi\,-\,\veps^{m-1}\int_{\mbb{R}^2}\langle  \vrho^{(1)}_{0,\veps}\rangle\, \phi(0,\cdot )\,, 
\end{align*}
which of course converges to $0$ when $\veps\ra0^+$.

It remains us to tackle the convective term $\vrho_\veps \ue^h \otimes \ue^h$.
For it, we take  advantage of Lemma \ref{lem:convterm_G} and relations \eqref{eq:limit_T1_G} and \eqref{eq:limit_T2_G}, but
we still have to take care of the convergence for $M\ra+\infty$ in \eqref{eq:limit_T1_G}.
We start by performing equalities \eqref{eq:T1_G} backwards in the term on the right-hand side of \eqref{eq:limit_T1_G}: thus, we have to pass to the limit for $M\ra+\infty$
in
\[
\int^T_0\int_{\R^2}\vec U_M^h\otimes\vec U_M^h : \nabla_h \vec\psi^h\,\dx^h\,\dt\,.
\]
Now, we remark that, since $\vec U^h\in L^2_T(H^1)$ by \eqref{conv:u_G}, from \eqref{est:sobolev} we gather the strong convergence
$S_M \vec U^h\longrightarrow \vec{U}^{h}$ in $L_{T}^{2}(H^{s})$ for any $s<1$, in the limit for $M\rightarrow +\infty$.
Then, passing to the limit for $M\ra+\infty$ in the previous relation is an easy task: we finally get, for $\veps\ra0^+$, 
\begin{equation*}
\int_0^T\int_{\Omega} \vre \ue^h\otimes\ue^h  : \nabla_h \vec\psi^h\, \longrightarrow\, \int_0^T\int_{\R^2}\vec{U}^h\otimes\vec{U}^h  : \nabla_h \vec\psi^h\,.
\end{equation*}

In the end, we have shown that, letting $\varepsilon \rightarrow 0^+$ in \eqref{eq:weak_to_conv_G}, one obtains
\begin{align*}
&\int_0^T\!\!\!\int_{\R^2} \left(\vec{U}^{h}\cdot \d_{t}\vec\psi^h+\vec{U}^{h}\otimes \vec{U}^{h}:\nabla_{h}\vec\psi^h\right)\, \dx^h \dt=
\int_0^T\!\!\!\int_{\R^2} \mu \nabla_{h}\vec{U}^{h}:\nabla_{h}\vec\psi^h \, \dx^h \dt-
\int_{\R^2}\lan\vec{u}_{0}^{h}\ran\cdot \vec\psi^h(0,\cdot)\dx^h,
\end{align*}
for any test-function $\vec\psi$ as in \eqref{eq:test-f_G}.
This implies \eqref{eq_lim_m:momentum_G}, concluding the proof of Theorem \ref{th:m>1}.

\section{Proof of the convergence for $m=1$} \label{s:proof-1_G}

In the present section, we complete the proof of the convergence in the case $m=1$ and $1/2<n<1$.
We will use again the compensated compactness argument depicted in Subsection \ref{ss:convergence_G},
and in fact most of the computations apply also in this case.

\subsection{The acoustic-Poincar\'e waves system}\label{ss:unifbounds_1_G} 

When $m=1$, the wave system \eqref{eq:wave_syst_G} takes the form
\begin{equation} \label{eq:wave_m=1_G}
\left\{\begin{array}{l}
       \veps\,\d_t \vrho_\veps^{(1)}\,+\,\div\vec{V}_\veps\,=\,0 \\[1ex]
       \veps\,\d_t\vec{V}_\veps\,+\,\nabla_x \vrho^{(1)}_\veps\,+\,\,\e_3\times \vec V_\veps\,=\,\veps\,\vec f_\veps+\veps^{2(1-n)}\vec g_\veps\,,
       \end{array}
\right.
\end{equation}
where $\bigl(\vrho^{(1)}_\veps\bigr)_\veps$  and $\bigl(\vec V_\veps\bigr)_\veps$ are defined as in Paragraph \ref{sss:wave-eq_G}.
This system is supplemented with the initial datum $\big(\vrho^{(1)}_{0,\veps},\vr_{0,\veps}\vec u_{0,\veps}\big)$.

Next, we regularise all the quantities, by applying the Littlewood-Paley cut-off operator $S_M$ to \eqref{eq:wave_m=1_G}: we deduce that $\vrho^{(1)}_{\veps,M}$ and $\vec V_{\veps,M}$, defined as in \eqref{def_reg_vrho-V}, satisfy the regularised wave system
\begin{equation} \label{eq:reg-wave_G}
\left\{\begin{array}{l}
       \veps\,\d_t \vrho_{\veps,M}^{(1)}\,+\,\div\vec{V}_{\veps,M}\,=\,0 \\[1ex]
       \veps\,\d_t\vec{V}_{\veps,M}\,+\,\nabla_x \vrho^{(1)}_{\veps,M}\,+\,\,\e_3\times \vec V_{\veps,M}\,=\,\veps\,\vec f_{\veps,M}+\veps^{2(1-n)}\vec g_{\veps,M}\, ,
       \end{array}
\right.
\end{equation}
in the domain $\R_+\times\Omega$, where we recall that $\vec f_{\veps,M}:=S_M \vec f_\veps$ and $\vec g_{\veps,M}:=S_M \vec g_\veps$.
It goes without saying that a result similar to Proposition \ref{p:prop approx_G} holds true also in this case.

As it is apparent from the wave system \eqref{eq:wave_m=1_G} and its regularised version, when $m=1$ the pressure term and the Coriolis term are in balance, since they are of the same order. This represents the main change with respect to the case $m>1$, and it comes into play in the compensated compactness argument. Therefore, despite most of the computations may be repeated identical as in the previous section, let us present
the main points of the argument.

\subsection{Handling the convective term when $m=1$} \label{ss:convergence_1_G}

Let us take care of the convergence of the convective term when $m=1$. 

First of all, it is easy to see that Lemma \ref{lem:convterm_G} still holds true. Therefore, 
given a test-function $\vec\psi\in C_c^\infty\big([0,T[\,\times\Omega;\R^3\big)$ such that $\div\vec\psi=0$ and $\d_3\vec\psi=0$,
we have to pass to the limit in the term
\begin{align*}
-\int_{0}^{T}\int_{\Omega} \vec{V}_{\veps ,M}\otimes \vec{V}_{\veps ,M}: \nabla_{x}\vec{\psi}\,&=\,
\int_{0}^{T}\int_{\Omega}\div\left(\vec{V}_{\veps ,M}\otimes \vec{V}_{\veps ,M}\right) \cdot \vec{\psi}\,=\,
\int_{0}^{T}\int_{\R^2} \left(\mc{T}_{\veps ,M}^{1}+\mc{T}_{\veps, M}^{2}\right)\cdot\vec{\psi}^h\,,
\end{align*}
where we agree again that the torus $\T$ has been normalised so that its Lebesgue measure is equal to $1$ and we have adopted the same notation as in \eqref{def:T1-2_G}.

At this point, we notice that the analysis of $\mc{T}_{\veps ,M}^{2}$ can be performed as in Paragraph \ref{sss:term2_G}, because we have
$m+1>2n$, i.e. $n<1$. \emph{Mutatis mutandis}, we find relation \eqref{eq:limit_T2_G} also in the case $m=1$.

Let us now deal with the term $\mc{T}_{\veps ,M}^{1}$. Arguing as in Paragraph \ref{sss:term1_G}, we may write it as
\begin{equation*}
\mc{T}_{\veps ,M}^{1}\,=\,\left(\curlh\langle \vec{V}_{\veps ,M}^{h}\rangle-\langle \vrho^{(1)}_{\veps ,M}\rangle \right)\langle \vec{V}_{\veps ,M}^{h}\rangle^{\perp}+\mc{R}_{\veps ,M} .
\end{equation*}
Now we use the horizontal part of \eqref{eq:reg-wave_G}: 
averaging it with respect to the vertical variable and applying the operator $\curlh$, we find
\begin{equation*}
\veps\,\d_t\curlh\langle \vec{V}_{\veps ,M}^{h}\rangle\,+\,\divh\langle \vec{V}_{\veps ,M}^{h}\rangle \,=\,
\veps\, \curlh\langle \vec f_{\veps ,M}^{h}\rangle\, .
\end{equation*}
Taking the difference of this equation with the first one in \eqref{eq:reg-wave_G}, we discover that
\begin{equation*}
\d_t\wtilde\Gamma_{\veps,M}
\,=\,\curlh\langle \vec f_{\veps ,M}^{h}\rangle\,,\qquad\qquad \mbox{ where }\qquad
\wtilde\Gamma_{\veps, M}:=\curlh\langle \vec{V}_{\veps ,M}^{h}\rangle\,-\,\langle \vrho^{(1)}_{\veps ,M}\rangle\,.
\end{equation*}
An argument analogous to the one used after \eqref{eq:gamma_G} above, based on Aubin-Lions Theorem, shows that
$\big(\wtilde\Gamma_{\veps,M}\big)_{\veps}$ is compact in e.g. $L_{T}^{2}(L_{\rm loc}^{2})$. Then, this sequence converges strongly (up to extraction of a suitable subsequence, not relabelled here) to a tempered distribution $\wtilde\Gamma_M$ in the same space. 

Using the previous property, we may deduce that
\begin{equation*}
\wtilde\Gamma_{\veps,M}\,\langle \vec{V}_{\veps ,M}^{h}\rangle^{\perp}\,\longrightarrow\, \wtilde\Gamma_M\, \langle \vec{V}_{M}^{h}\rangle^{\perp}\qquad \text{ in }\qquad \mc{D}^{\prime}\big(\R_+\times\R^2\big),
\end{equation*}
where we have $\langle \vec{V}_{M}^{h}\rangle=\lan S_M\vec{U}^{h}\ran$ and $\wtilde\Gamma_M=\curlh\lan S_M \vec{U}^{h}\ran-\langle S_M\vrho^{(1)}\rangle$.

Owing to the regularity of the target velocity $\vec U^h$, we can pass to the limit also for $M\ra+\infty$, as detailed in Subsection \ref{ss:limit_G} above. Thus, we find
\begin{equation} \label{eq:limit_T1-1_G}
\int^T_0\!\!\!\int_{\Omega}\vrho_\veps\,\vec{u}_\veps\otimes \vec{u}_\veps: \nabla_{x}\vec{\psi}\, \dxdt\,\longrightarrow\,
\int^T_0\!\!\!\int_{\R^2}\big(\vec U^h\otimes\vec U^h:\nabla_h\vec\psi^h\,-\, \vrho^{(1)}\,(\vec U^h)^\perp\cdot\vec\psi^h\big)\dx^h\dt,
\end{equation}
for all test functions $\vec\psi$ such that $\div\vec\psi=0$ and $\d_3\vec\psi=0$. Recall the convention $|\T|=1$.
Notice that, since $\vec U^h=\nabla_h^\perp \vrho^{(1)}$ when $m=1$ (keep in mind Proposition \ref{p:limit_iso_G}), the last term in the integral on the right-hand side is actually zero.

\subsection{End of the study} \label{ss:limit_1_G}
Thanks to the previous analysis, we are now ready to pass to the limit in equation \eqref{weak-mom_G} also in the case when $m=1$.
For this, we take a test-function $\vec\psi$ as in \eqref{eq:test-2};
notice in particular that $\div\vec\psi=0$ and $\d_3\vec\psi=0$. Then, once again all the gradient terms and all the contributions coming from the vertical
component of the momentum equation vanish identically, when tested against such a $\vec\psi$. Recall that all the integrals will be performed on $\R^2$. 
So, equation \eqref{weak-mom_G} reduces
to
\begin{align*}
\int_0^T\!\!\!\int_{\Omega}  \left( -\vre \ue \cdot \partial_t \vec\psi -\vre \ue\otimes\ue  : \nabla \vec\psi
+ \frac{1}{\ep}\vre\big(\ue^{h}\big)^\perp\cdot\vec\psi^h+\mbb{S}(\nabla_x\vec\ue): \nabla_x \vec\psi\right)
 =\int_{\Omega}\vrez \uez  \cdot \vec\psi(0,\cdot)\,.
\end{align*}

For the rotation term, we can test the first equation in \eqref{eq:wave_m=1_G} against $\phi$ to get
\begin{equation*} 
\begin{split}
-\int_0^T\!\!\!\int_{\R^2} \left( \lan \vrho^{(1)}_{\varepsilon}\ran\, \d_{t}\phi +\frac{1}{\veps}\, \lan\vrho_{\veps}\ue^{h}\ran\cdot \nabla_{h}\phi\right)=
\int_{\R^2}\lan \vrho^{(1)}_{0,\varepsilon }\ran\, \phi (0,\cdot ) \, ,
\end{split}
\end{equation*}
whence we deduce that
\begin{align*}
\int_0^T\!\!\!\int_{\Omega}\frac{1}{\ep}\vre\big(\ue^{h}\big)^\perp\cdot\vec\psi^h\,&=\,\int_0^T\!\!\!\int_{\mbb{R}^2}\frac{1}{\ep}\langle\vre \ue^{h}\rangle \cdot \nabla_{h}\phi\,=\,-\,\int_0^T\!\!\!\int_{\mbb{R}^2}\langle \vrho^{(1)}_\veps\rangle\, \d_t\phi\,-\,\int_{\mbb{R}^2}\langle \vrho^{(1)}_{0,\veps}\rangle\, \phi(0,\cdot )\,. 
\end{align*}

In addition, the convergence of the convective term has been performed in \eqref{eq:limit_T1-1_G}, and for the other terms, we can argue as in Subsection \ref{ss:limit_G}.
Hence, letting $\varepsilon \rightarrow 0^+$ in the equation above, we get
\begin{align*}
&-\int_0^T\!\!\!\int_{\R^2} \left(\vec{U}^{h}\cdot \d_{t}\nabla_{h}^{\perp} \phi+ \vec{U}^{h}\otimes \vec{U}^{h}:\nabla_{h}(\nabla_{h}^{\perp}\phi )+\vrho^{(1)}\, \d_t \phi \right)\dx^h \dt\\
&\qquad\qquad=-\int_0^T\!\!\!\int_{\R^2} \mu \nabla_{h}\vec{U}^{h}:\nabla_{h}(\nabla_{h}^{\perp}\phi ) \dx^h \dt+\int_{\R^2}\left(\lan\vec{u}_{0}^{h}\ran\cdot \nabla _{h}^{\perp}\phi (0,\cdot )+
\lan \vrho^{(1)}_{0}\ran\phi (0,\cdot )\right) \dx^h\, ,
\end{align*}
which is the weak formulation of equation \eqref{eq_lim:QG_G}. In the end, also Theorem \ref{th:m=1} is proved.


\part{The Euler equations}
\chapter{The fast rotation limit for Euler}\label{chap:Euler}

\begin{quotation}
We conclude this thesis describing the behaviour of a fluid which, in contrast with Chapters \ref{chap:multi-scale_NSF} and \ref{chap:BNS_gravity}, evolves far from the physical boundaries. An example of such fluid flows is represented by currents in the middle of the oceans, i.e. far enough from the surface and the bottom. In this context, we can assume the following physical approximations:
\begin{itemize}
\item the fluid is incompressible;
\item the density is a small perturbation of a constant profile (the \textit{Boussinesq} approximation);
\item the domain is the $\R^2$ plane: the motion of a $3$-D highly rotating fluid is, in a first approximation, planar (the \textit{Taylor-Proudman} theorem). 
\end{itemize}
Then, the system reads: 
\begin{equation}\label{full Euler}
\begin{cases}
\d_t \vrho +\div (\vrho \vu)=0\\
\d_t (\vrho \vu)+\div (\vrho \vu \otimes \vu)+ \dfrac{1}{Ro}\vrho \vu^\perp+\nabla p=0\\
\div \vu =0
\end{cases}
\end{equation}
in the domain $\Omega =\R^2$. 
Since in this chapter there will be no more competition between the horizontal and vertical scales, for notational convenience, we will drop everywhere (in the differential operators) the subscript $x$. 

With respect to the previous systems (see \eqref{chap2:syst_NSFC} and \eqref{chap3:BNSC} in this respect), the Euler equations \eqref{full Euler} is an incompressible 
system without the viscosity effects and with a hyperbolic structure. For that reason, we will need different analysis techniques (see Appendix \ref{app:Tools}) and in addition the functional framework will change (now we will be in regular spaces) in order to preserve the initial regularity. 

The topics presented here are part of the pre-print \cite{Sbaiz}.

\medbreak
Let us now give a summary of the chapter.

In Section \ref{s:result_E}, we collect our assumptions and we state our main results. 
In Section \ref{s:well-posedness_original_problem}, we investigate the well-posedness issues in the Sobolev spaces $H^s$ for any $s>2$.
In Section \ref{s:sing-pert_E}, we study the singular perturbation problem, establishing constraints that the limit points have to satisfy and proving the convergence to the quasi-homogeneous Euler system thanks to a \textit{compensated compactness} technique. 
In Section \ref{s:well-posedness_Q-H} we review, for the quasi-homogeneous limiting system, the results presented in \cite{C-F_RWA} and \cite{C-F_sub}, and 
we explicitly derive the lifespan of solutions to the limit equations (see relation \eqref{T-ast:improved} in this regard).

In the last section, we deal with the lifespan analysis for system \eqref{full Euler} and we point out some consequences of the continuation criterion we have established (see in particular Subsection \ref{ss:cont_criterion+consequences}).  
\end{quotation}

\section{The density-dependent Euler problem} \label{s:result_E}

In this section, we formulate our working setting (Subsection \ref{ss:FormProb_E}) and we state the main results
(Subsection \ref{ss:results_E}).

\subsection{Formulation} \label{ss:FormProb_E}

In this subsection, we present the rescaled density-dependent Euler equations with the Coriolis force, which we are going to consider in our study, and we formulate the main working hypotheses.

To begin with, let us introduce the ``primitive system'', that is the rescaled incompressible Euler system \eqref{full Euler}, supplemented with the scaling $Ro=\veps$, where $\veps \in \,]0,1]$ is a small parameter. Thus, the system consists of continuity equation (conservation of mass), the momentum equation and the divergence-free condition: respectively 
\begin{equation}\label{Euler_eps}
\begin{cases}
\d_t \vrho_{\veps} +\div (\vrho_{\veps} \vu_{\veps})=0\\
\d_t (\vrho_{\veps} \vu_{\veps})+\div (\vrho_{\veps} \vu_{\veps} \otimes \vu_{\veps})+ \dfrac{1}{\veps}\vrho_{\veps} \vu_{\veps}^{\perp}+\dfrac{1}{\veps}\nabla \Pi_\veps=0\\
\div \vu_{\veps} =0\, .
\end{cases}
\end{equation}
We point out that, here above in \eqref{Euler_eps}, the domain $\Omega$ is the plane $\R^2$ and the unknowns are $\vrho_\veps\in \R_+$ and $\vec u_\veps\in \R^2$. 

In \eqref{Euler_eps}, the pressure term has to scale like $1/\veps$, since it is the only force that allows to compensate the effect of fast rotation, at the geophysical scale.

From now on, in order to make the Lipschitz condition \eqref{Lip_cond} holds, we fix 
$$ s>2\, . $$

We assume that the initial density is a small perturbation of a constant profile.
Namely, we consider initial densities of the following form:
	\begin{equation}\label{in_vr_E}
	\vrez = 1 + \ep \, R_{0,\veps} \, ,
	\end{equation}
where we suppose $R_{0,\veps}$ to be a bounded measurable function satisfying the controls
	\begin{align}
&\sup_{\veps\in\,]0,1]}\left\|  R_{0,\veps} \right\|_{L^\infty(\R^2)}\,\leq \, C \label{hyp:ill_data_R_0}\\
&\sup_{\veps\in\,]0,1]}\left\| \nabla R_{0,\veps} \right\|_{H^{s-1}(\R^2)}\,\leq \, C \label{hyp:ill_data_nablaR_0}
	\end{align}
and the initial mass density is bounded and bounded away from zero, i.e. for all $\veps \in\;]0,1]$: 
\begin{equation}\label{assumption_densities}
0<\underline{\vrho}\leq   \vrho_{0,\veps}(x) \,\leq \, \overline{\vrho}\, , \qquad x\in \R^2
\end{equation}
where $\underline{\vrho},\overline{\vrho}>0$ are positive constants.

As for the initial velocity fields, due to framework needed for the well-posedness issues, we require the following uniform bound
\begin{equation}\label{hyp:data_u_0}
\sup_{\veps\in\,]0,1]}\left\|  \vu_{0,\veps} \right\|_{H^s(\R^2)}\,\leq \, C\, .
\end{equation}
Thanks to the previous uniform estimates, we can assume (up to passing to subsequences) that there exist $R_0 \in W^{1,\infty}(\R^2)$, with $\nabla R_0\in H^{s-1}(\R^2)$, and $\vec u_0\in H^s(\R^2)$ such that 
\begin{equation}\label{init_limit_points}
\begin{split}
R_0:=\lim_{\veps \rightarrow 0}R_{0,\veps} \quad &\text{in}\quad L^\infty (\R^2) \\
\nabla R_0:=\lim_{\veps \rightarrow 0}\nabla R_{0,\veps}\quad &\text{in}\quad H^{s-1} (\R^2)\\
\vu_0:=\lim_{\veps \rightarrow 0}\vu_{0,\veps}\quad &\text{in}\quad H^s (\R^2)\, ,
\end{split}
\end{equation} 
where we agree that the previous limits are taken in the corresponding weak-$\ast$ topology.

\subsection{Main statements}\label{ss:results_E}

\medbreak
We can now state our main results. We recall the notation $\big(f_\veps\big)_{\veps} \subset X$ to denote that the family $\big(f_\veps\big)_{\veps}$ is uniformly (in $\veps$) bounded in $X$.

The following theorem establishes the local well-posedness of system \eqref{Euler_eps} in the Sobolev spaces $B^s_{2,2}\equiv H^s$ (see Section \ref{s:well-posedness_original_problem}) and gives a lower bound for the lifespan of solutions (see Section \ref{s:lifespan_full}). 

\begin{theorem}\label{W-P_fullE}
For any $\veps \in\, ]0,1]$, let initial densities $\vrho_{0,\veps}$ be as in \eqref{in_vr_E} and satisfy the controls \eqref{hyp:ill_data_R_0} to \eqref{assumption_densities}. Let $\vu_{0,\veps}$ be divergence-free vector fields such that $\vu_{0,\veps} \in H^s(\R^2)$ for $s>2$. \\
Then, for any $\veps>0$, 
there exists a time $T_\veps^\ast >0$ such that 
the system \eqref{Euler_eps} has a unique solution $(\vrho_\veps, \vu_\veps, \nabla \Pi_\veps)$ where 
\begin{itemize}
\item $\vrho_\veps$ belongs to the space $C^0([0,T_\veps^\ast]\times \R^2)$ with $\nabla \vrho_\veps \in  C^0([0,T_\veps^\ast]; H^{s-1}(\R^2))$;
\item $\vu_\veps$ belongs to the space $C^0([0,T_\veps^\ast]; H^s(\R^2))$;
\item $\nabla \Pi_\veps$ belongs to the space $C^0([0,T_\veps^\ast]; H^s(\R^2))$.
\end{itemize}
Moreover, the lifespan $T_\veps^\ast$ of the solution to the two-dimensional density-dependent incompressible Euler equations with the Coriolis force is bounded from below by
\begin{equation}\label{improv_life_fullE}
\frac{C}{\|\vec u_{0,\veps}\|_{H^s}}\log\left(\log\left(\frac{C\, \|\vec u_{0,\veps}\|_{H^s}}{\max \{\mc A_\veps(0),\, \veps \, \mc A_\veps(0)\, \|\vec u_{0,\veps}\|_{H^s}\}}+1\right)+1\right)\, ,
\end{equation}
where $\mc A_\veps (0):= \|\nabla R_{0,\veps}\|_{H^{s-1}}+\veps\, \|\nabla R_{0,\veps}\|_{H^{s-1}}^{\lambda +1}$, for some suitable $\lambda\geq 1$.

In particular, one has 
$$\inf_{\veps>0}T_\veps^\ast  >0\, .$$
\end{theorem} 
Looking at \eqref{improv_life_fullE}, we stress the fact that the fast rotational effects are not enough to state a global well-posedness result for system \eqref{Euler_eps}, in the sense that $T^\ast_\veps$ does not tend to $+\infty$ when $\veps\rightarrow 0^+$.

Now, once we have stated the local in time well-posedness for system \eqref{Euler_eps} in the Sobolev spaces $H^s$, in Section \ref{s:sing-pert_E} we address the singular perturbation problem describing, in a rigorous way, the limit dynamics depicted by the quasi-homogeneous incompressible Euler system \eqref{system_Q-H_thm} below. 
\begin{theorem}\label{thm:limit_dynamics}
Let $s>2$. For any fixed value of $\veps \in \; ]0,1]$, let initial data $\left(\vrho_{0,\veps},\vec u_{0,\veps}\right)$ verify the hypotheses fixed in Paragraph \ref{ss:FormProb_E}, and let
$\left( \vre, \ue\right)$ be a corresponding solution to system \eqref{Euler_eps}.
Let $\left(R_0,\vec u_0\right)$ be defined as in \eqref{init_limit_points}.

Then, one has the following convergence properties:
	\begin{align*}
	\varrho_\ep \rightarrow 1 \qquad\qquad \mbox{ in } \qquad &L^\infty\big([0,T^\ast]; L^\infty(\R^2 )\big)\,, \\
	R_\veps:=\frac{\varrho_\ep - 1}{\ep}  \weakstar R \qquad\qquad \mbox{ in }\qquad &L^\infty\bigl([0,T^\ast]; L^\infty(\R^2)\bigr)\,, \\
	\nabla R_\veps  \weakstar \nabla R \qquad\qquad \mbox{ in }\qquad &L^\infty\bigl([0,T^\ast]; H^{s-1}(\R^2)\bigr)\,, \\
 \vec{u}_\ep \weakstar \vec{u}
	\qquad\qquad \mbox{ in }\qquad &L^\infty\big([0,T^\ast];H^s(\R^2)\big)\, .
	\end{align*}	
In addition, $\Big(R\, ,\, \vec{u}  \Big)$ is a solution
to the quasi-homogeneous incompressible Euler system  in $[0,T^\ast] \times \R^2$:
\begin{equation}\label{system_Q-H_thm}
\begin{cases}
\d_t R+\div (R\vu)=0 \\
\d_t \vec u+\div \left(\vec{u}\otimes\vec{u}\right)+R\vu^\perp+ \nabla \Pi =0  \\
\div \vec u\,=\,0\, ,
\end{cases}
\end{equation}
where $\nabla \Pi$ is a suitable pressure term belonging to $L^\infty\big([0,T^\ast];H^s(\R^2)\big)$. 
\end{theorem}

\begin{remark} 
Due to the fact that the system \eqref{system_Q-H_thm} is well-posed in the previous functional setting (see Theorem \ref{thm:well-posedness_Q-H-Euler} below), we get the convergence of the whole sequence of weak solutions to the solutions of the target equations on the large time interval where the weak solutions to the primitive equations exist.
\end{remark}

Then, at the limit, we have found that the dynamics is prescribed by the quasi-homogeneous incompressible Euler system \eqref{system_Q-H_thm}, for which we state the local well-posedness in $H^s$ (see Section \ref{s:well-posedness_Q-H}). It is worth to remark that the global well-posedness issue for this system is still an open problem. 

\begin{theorem}\label{thm:well-posedness_Q-H-Euler}
Take $s>2$. Let $\big(R_0,u_0 \big)$ be initial data such that $R_0\in L^{\infty}(\R^2)$ and
$\vu_0 \,\in  H^s(\R^2)$, with $\nabla R_0\in H^{s-1}(\R^2)$ and $\div \vu_0\,=\,0$.

Then, there exists a time $T^\ast > 0$ such that, on $[0,T^\ast]\times\R^2$, problem \eqref{system_Q-H_thm} has a unique solution $(R,\vu, \nabla \Pi)$ with the following properties:
\begin{itemize}
 \item $R\in C^0\big([0,T^\ast]\times \R^2\big)$ and $\nabla R\in C^0\big([0,T^\ast];H^{s-1}(\R^2)\big)$;
 \item $\vu$ belongs to $C^0\big([0,T^\ast]; H^{s}(\R^2)\big)$;
 \item the pressure term $\nabla \Pi$ is in $C^0\big([0,T^\ast];H^s(\R^2)\big)$. 
 \end{itemize}

In addition, the lifespan $T^\ast>0$ of the solution $(R, \vu, \nabla \Pi)$ to the $2$-D quasi-homogeneous Euler system \eqref{system_Q-H_thm} enjoys the following lower bound:
\begin{equation}\label{improved_low_bound}
T^\ast\geq \frac{C}{\|\vu_0\|_{H^s}}\log\left(\log \left(C\frac{\|\vu_0\|_{H^s}}{\|R_0\|_{L^\infty}+\|\nabla R_0\|_{H^{s-1}}}+1\right)+1\right)\, ,
\end{equation}
where $C>0$ is a ``universal'' constant, independent of the initial datum.

\end{theorem}
The proof of the previous ``asymptotically global'' well-posedness result is presented in Subsection \ref{ss:improved_lifespan}. 

\section{Well-posedness for the original problem}\label{s:well-posedness_original_problem}
This section is devoted to the well-posedness issue in the $H^s$ spaces stated in Theorem \ref{W-P_fullE}. We recall that, due to the Littlewood-Paley theory, we have the equivalence between $H^s$ and $B^s_{2,2}$ spaces.

We also underline that in this section we keep $\veps \in \; ]0,1] $ fixed. However, we will take care of it, explicitly pointing out the dependence to the Rossby number in all the computations in order to get controls that are uniform with respect to the $\veps$-parameter. The choice in keeping explicit the dependence on the rotational parameter is motivated by the fact that we will perform the fast rotation limit (see Section \ref{s:sing-pert_E} below). 

First of all, since $\vrho_\veps$ is a small perturbation of a constant profile, we set 
\begin{equation}\label{def_a_veps}
 \alpha_\veps :=\frac{1}{\vrho_\veps}-1=\veps a_\veps\quad \text{with}\quad a_\veps:=-R_\veps/\vrho_\veps \, .
\end{equation}
The choice of looking at $\alpha_\veps$ is dictated by the fact that we will extensively exploit the elliptic equation 
\begin{equation}\label{general_elliptic_est}
-\div \left(\alpha_\veps \nabla \Pi_\veps \right)=\, \veps\, \div \left(\ue \cdot \nabla \ue +\frac{1}{\veps}\vec u_\veps^\perp \right)\, .
\end{equation} 
Now, using the divergence-free condition, we can rewrite the system \eqref{Euler_eps} in the following way (see also Lemma 3 in \cite{D-F_JMPA}): 
\begin{equation}\label{Euler-a_eps_1}
\begin{cases}
\d_t a_{\veps} +\vu_\veps \cdot \nabla a_\veps=0\\
\d_t \vu_{\veps}+ \vu_{\veps} \cdot \nabla \vu_{\veps}+ \dfrac{1}{\veps} \vu_{\veps}^{\perp}+(1+\veps a_\veps)\dfrac{1}{\veps}\nabla \Pi_\veps=0\\
\div \vu_{\veps} =0\, ,
\end{cases}
\end{equation}
with the initial condition $(a_\veps, \ue)_{|t=0}=(a_{0,\veps},\vu_{0,\veps})$. 

We start by presenting the proof of existence of solutions at the claimed regularity. 
For that scope, we follow a standard procedure: first, we construct a sequence of smooth approximate solutions. Next, we deduce uniform bounds (with respect to the approximation parameter and also to $\veps$) for those regular solutions. Finally, by use of those uniform bounds and an energy method, together with an interpolation argument, we are able to take the limit in the approximation parameter and gather the existence of a solution to the original problem.  

We end this Section \ref{s:well-posedness_original_problem}, proving uniqueness of solutions in the claimed functional setting, by using a relative entropy method.

\subsection{Construction of smooth approximate solutions}\label{sec:construction_smooth_sol}
For any $n\in \N$, let us define 
\begin{equation*}
(a_{0,\veps}^n, \vec u_{0,\veps}^n ):= (S_n a_{0,\veps}, S_n \vec u_{0,\veps})\, ,
\end{equation*}
where $S_{n}$ is the low frequency cut-off operator introduced in \eqref{eq:S_j} in Section \ref{app:LP}. We stress also the fact that $a_{0,\veps}\in C^0_{\rm loc}$, since $a_{0,\veps}$ and $\nabla a_{0,\veps}$ are in $L^\infty$.

Then, for any $n\in \N$, we have the density functions $a_{0,\veps}^n\in L^\infty$. Moreover, one has that $\nabla a_{0,\veps}^n$ and $\vec u_{0,\veps}^n$ belong to $H^\infty:=\bigcap_{\sigma \in \R} H^\sigma$ which is embedded (for a suitable topology on $H^\infty$) in the space $C_b^\infty$ of $C^\infty$ functions which are globally bounded together with all their derivatives. 

In addition, by standard properties of mollifiers (see Section \ref{section:molli} of the Appendix \ref{appendixA}), one has the following strong convergences 
\begin{equation}\label{conv_in_data}
\begin{split}
a^n_{0,\veps}\rightarrow a_{0,\veps}\quad  &\text{in}\quad C_{\rm loc}^{0} \\
\nabla a^n_{0,\veps}\rightarrow \nabla a_{0,\veps}\quad  &\text{in}\quad  H^{s-1}\\
\vu_{0,\veps}^n\rightarrow \vu_{0,\veps} \quad &\text{in}\quad H^s \, .
\end{split}
\end{equation}

This having been established, we are going to define a sequence of approximate solutions to system \eqref{Euler-a_eps_1} by induction. First of all, we set $(a_\veps^0,\vu_\veps^0, \nabla \Pi_\veps^0)=(a^0_{0,\veps},\vu^0_{0,\veps},0)$.  Then, for all $\sigma \in \R$, we have that $\nabla a_{\veps}^0 ,\vec u_{\veps}^0 \in H^\sigma$ and $a^0_\veps \in L^\infty$ with $\div \vec u_{\veps}^0=0$. Next, assume that the couple $(a_\veps^n, \ue^n)$ is given such that, for all $\sigma \in \R$,
\begin{equation*}
 a^n_\veps \in C^0(\R_+;L^\infty) \quad \nabla a_{\veps}^n ,\vec u_{\veps}^n \in C^0(\R_+; H^\sigma)\quad  \quad \text{and}\quad \div \vec u_{\veps}^n=0\, .
\end{equation*}
First of all, we define $a_\veps^{n+1}$ as the unique solution to the linear transport equation
\begin{equation}\label{mass_eq}
\d_t a_{\veps}^{n+1} +\vu_\veps^n \cdot \nabla a_\veps^{n+1}=0 \quad \text{with}\quad {(\an)}_{|t=0}=a_{0,\veps}^{n+1}\, .
\end{equation}
Since, by inductive hypothesis and embeddings, $\ue^n$ is divergence-free, smooth and uniformly bounded with all its derivatives, we can deduce that $\an \in L^\infty(\R_+; L^\infty )$. Moreover, from

\begin{equation*}
\d_t\, \d_i \an +\ue^{n}\cdot \nabla\, \d_i \an =-\d_i \ue^n \cdot \nabla \an \quad \text{with}\quad {(\d_i\an)}_{|t=0}=\d_i a_{0,\veps}^{n+1} \quad \text{for}\; i=1,2
\end{equation*} 
and thanks to the Theorem \ref{thm_transport}, we can propagate all the $H^{\sigma}$ norms of the initial datum. We deduce that $a_\veps^{n+1}\in C^0(\R_+;L^\infty)$ and $\nabla a_\veps^{n+1}\in C^0(\R_+;H^{\sigma})$ for any $\sigma \in \R$. 
Next, we consider the approximate linear iteration
\begin{equation}\label{approx_iteration_1}
\begin{cases}
\d_t \vu_{\veps}^{n+1}+ \vu_{\veps}^n \cdot \nabla \vu_{\veps}^{n+1}+ \dfrac{1}{\veps} \vu_{\veps}^{\perp, n+1}+(1+\veps a_\veps^{n+1})\dfrac{1}{\veps}\nabla \Pi_\veps^{n+1}=0\\
\div \vu_{\veps}^{n+1} =0\\
(\vu_\veps^{n+1})_{|t=0}=\vu_{0,\veps}^{n+1}\, .
\end{cases}
\end{equation}
At this point, one can solve the previous linear problem finding a unique solution $\vu_\veps^{n+1}\in C^0(\R_+; H^\sigma )$ for any $\sigma \in \R$ and the pressure term $\nabla \Pi_\veps^{n+1}$ can be uniquely determined (we refer to \cite{D_AT} for details in this respect).

\subsection{Uniform estimates for the approximate solutions}\label{ss:unif_est}

We now have to show (by induction) uniform bounds for the sequence $(a_\veps^n, \ue^n, \nabla \Pi_\veps^n)_{n\in \N}$ we have constructed above.

We start by finding uniform estimates for $a_\veps^{n+1}$. Thanks to equation \eqref{mass_eq} and the divergence-free condition on $\ue^n$, we can propagate the $L^\infty$ norm for any $t\geq 0$:
\begin{equation}\label{eq:transport_density}
\|a_\veps^{n+1}(t)\|_{L^\infty} \leq\|a_{0,\veps}^{n+1}\|_{L^\infty}\leq C \|a_{0,\veps}\|_{L^\infty}\, .
\end{equation} 
At this point we want to estimate $\nabla a_\veps^{n+1}$ in $H^{s-1}$. We have for $i=1,2$:
\begin{equation*}
\d_t\, \d_i \an +\ue^{n}\cdot \nabla\, \d_i \an =-\d_i \ue^n \cdot \nabla \an \, .
\end{equation*}
Taking the non-homogeneous dyadic blocks $ \Delta_j$, we obtain 
\begin{equation*}
\d_t \Delta_j \, \d_i a_\veps^{n+1}+\ue^n\cdot \nabla  \Delta_j\, \d_i \an=[\ue^n\cdot \nabla,\Delta_j]\, \d_i \an -\Delta_j\left(\d_i \ue^n \cdot \nabla \an \right)\, .
\end{equation*}
Multiplying the previous relation by $ \Delta_j\, \d_i \an$, we have
\begin{equation*}
\| \Delta_j\, \d_i \an(t)\|_{L^2}\leq  \| \Delta_j\, \d_i a_{0,\veps}^{n+1}\|_{L^2}+C \int_0^t\left(\left\|[\ue^n\cdot \nabla, \Delta_j] \, \d_i \an\right\|_{L^2}+\|\Delta_j\left(\d_i \ue^n \cdot \nabla \an \right) \|_{L^2} \right)\, \detau  \, .
\end{equation*}
We apply now the second commutator estimate stated in Lemma \ref{l:commutator_est} to the former term in the integral on the right-hand side, getting
\begin{equation*}
2^{j(s-1)}\left\|[\ue^n\cdot \nabla, \Delta_j]\, \d_i \an\right\|_{L^2}\leq C\, c_j(t)\left(\|\nabla \ue^n\|_{L^\infty}\|\d_i \an\|_{ H^{s-1}}+\|\nabla \ue^n\|_{H^{s-1}}\|\d_i \an \|_{L^\infty}\right)\, ,
\end{equation*}
where $(c_j(t))_{j\geq -1}$ is a sequence in the unit ball of $\ell^2$.

Instead, the latter term can be bounded in the following way:
\begin{equation*}\label{eq:nabla-u_nabla-a}
2^{j(s-1)}\|\Delta_j \left(\d_i \ue^n \cdot \nabla \an \right)\|_{L^2}\leq C\, c_j(t)\,  \|\nabla \ue^n\|_{H^{s-1}}\|\nabla \an \|_{H^{s-1}}\, .
\end{equation*}

Then, due to the embedding $H^\sigma(\R^2)\hookrightarrow L^\infty (\R^2)$ for $\sigma>1$, one has
\begin{equation*}
2^{j(s-1)}\| \Delta_j\, \nabla \an(t)\|_{L^2}\leq 2^{j(s-1)}\| \Delta_j \nabla a_{0,\veps}^{n+1}\|_{L^2}+\int_0^t C\, c_j(\tau )\left(\| \ue^n\|_{ H^s}\|\nabla \an\|_{ H^{s-1}}\right)\, \detau\, . 
\end{equation*}
At this point, after summing on indices $j\geq -1$, thanks to the Minkowski inequality (for which we refer to Proposition \ref{app:mink}) combined with a Gr\"onwall type argument (see Section \ref{app:sect_gron}), we finally obtain 
\begin{equation}\label{est:a^(n+1)}
\sup_{0\leq t\leq T}\|\nabla \an (t)\|_{H^{s-1}}\leq \|\nabla a_{0,\veps}^{n+1}\|_{H^{s-1}}\, \exp \left(\int_0^T C \, \|\ue^n(t)\|_{ H^s}\, \dt\right)\, .
\end{equation}

Now, we have to estimate the velocity field $\ue^{n+1}$ and for that purpose we start with the $L^2$ estimate. We take the momentum equation in the original form:
\begin{equation}\label{mom_eq_original_prob}
\vrho_\veps^{n+1}\left(\d_t\ue^{n+1}+\ue^n\cdot \nabla \ue^{n+1}\right)+\frac{1}{\veps}\vrho_\veps^{n+1}\ue^{\perp, n+1}+\frac{1}{\veps}\nabla \Pi_\veps^{n+1}=0 \, ,
\end{equation}
where we construct $\vrho_\veps^{n+1}:=1/(1+\veps \an)$ starting from $\an$. Notice that $\vrho_\veps^{n+1}$ satisfies the transport equation
\begin{equation*}
\d_t \vrho_\veps^{n+1}+\ue^{n}\cdot \nabla \vrho_\veps^{n+1}=0\, .
\end{equation*}  

At this point, we test equation \eqref{mom_eq_original_prob} against $\ue^{n+1}$. We integrate by parts on $\R^2$, deriving the following estimate: 
$$ \int_{\R^2} \vrho_\veps^{n+1}\d_t|\ue^{n+1}|^2+\int_{\R^2}\vrho_\veps^{n+1}\ue^n\cdot \nabla |\ue^{n+1}|^2=0 \, ,$$
which implies (making use of the transport equation for $\vrho_\veps^{n+1}$)
\begin{equation*}
 \left\|\sqrt{\vrho_\veps^{n+1}(t)}\,  \ue^{n+1}(t)\right\|_{L^2}\leq \left\|\sqrt{\vrho_{0,\veps}^{n+1}} \, \vec u_{0,\veps}^{n+1}\right\|_{L^2}\, .
\end{equation*}
From the previous bound, due to the assumption \eqref{assumption_densities}, we  can deduce the preservation of the $L^2$ norm for the velocity field $\ue^{n+1}$:
\begin{equation*}
 \left\| \ue^{n+1}(t)\right\|_{L^2}\leq C\left\| \vec u_{0,\veps}^{n+1}\right\|_{L^2}\leq C\left\| \vec u_{0,\veps}\right\|_{L^2}\, .
\end{equation*}

Taking now the operator $ \Delta_j$ in the momentum equation in \eqref{approx_iteration_1}, we obtain
\begin{equation*}
\d_t \Delta_j \ue^{n+1}+\ue^n\cdot \nabla \Delta_j \ue^{n+1}=[\ue^n\cdot \nabla, \Delta_j]\ue^{n+1}-\frac{1}{\veps}\Delta_j\ue^{\perp ,n+1}-\Delta_j\left[\left(1+\veps a_\veps^{n+1}\right)\frac{1}{\veps}\nabla \Pi_\veps^{n+1}\right]
\end{equation*}
and multiplying again by $\Delta_j\ue^{n+1}$, we have cancellations so that 
\begin{equation}\label{eq:vel_est_dyadic}
\left\|\Delta_j\ue^{n+1}(t)\right\|_{L^2}\leq \left\|\Delta_j\vec u_{0,\veps}^{n+1}\right\|_{L^2}+C\int_0^t \left(\left\|[\ue^n\cdot \nabla, \Delta_j]\ue^{n+1}\right\|_{L^2}+\left\|\Delta_j\left( a_\veps^{n+1}\nabla \Pi_\veps^{n+1}\right)\right\|_{L^2}\right) \, \detau\, .
\end{equation}
As done before, we employ here the commutator estimates of Lemma \ref{l:commutator_est} in order to have 
\begin{equation*}
\begin{split}
2^{js} \left\|[\ue^n\cdot \nabla,  \Delta_j]\ue^{n+1}\right\|_{L^2}&\leq C\,c_j\, \left(\|\nabla \ue^n\|_{L^\infty}\|\ue^{n+1}\|_{ H^s}+\|\nabla \ue^{n+1}\|_{L^\infty}\|\ue^n\|_{ H^s}\right)\\
&\leq C\,c_j\, \left(\| \ue^n\|_{ H^s}\|\ue^{n+1}\|_{ H^s}\right)\, .
\end{split}
\end{equation*}
For the latter term on the right-hand side of \eqref{eq:vel_est_dyadic}, we take advantage of the Bony decomposition (see Paragraph \ref{app_paradiff}) and we apply Proposition \ref{prop:app_fine_tame_est}. We may infer that 
\begin{equation*}
\begin{split}
 \left\|a_\veps^{n+1}\nabla \Pi_\veps^{n+1}\right\|_{ H^s}\leq  C\left(\|\an \|_{L^\infty}+\|\nabla a_\veps^{n+1}\|_{ H^{s-1}}\right)\|\nabla \Pi_\veps^{n+1}\|_{ H^s}\, .
\end{split}
\end{equation*}
To finish with, we have to find a uniform bound for the pressure term. For that, we apply the $\div$ operator in \eqref{approx_iteration_1}. Thus, we aim at solving the elliptic problem 
\begin{equation}\label{eq:elliptic_problem}
-\div \left(\left(1+\veps a_\veps^{n+1}\right)\nabla \Pi_\veps^{n+1}\right)=\, \veps\, \div (\ue^n \cdot \nabla \ue^{n+1} )- \curl \ue^{n+1}\, .
\end{equation}
Thanks to the assumption \eqref{assumption_densities} and Lemma \ref{lem:Lax-Milgram_type}, we can obtain 
\begin{equation}\label{est:Pi_L^2}
\begin{split}
\|\nabla \Pi_\veps^{n+1}\|_{L^2}&\leq C\left(\veps \, \|\ue^n\cdot \nabla \ue^{n+1}\|_{L^2}+\|\ue^{\perp,n+1}\|_{L^2}\right)\\
&\leq  C\left(\veps \, \|\ue^n\|_{L^2} \| \ue^{n+1}\|_{H^s}+\|\ue^{n+1}\|_{L^2}\right)\, .
\end{split}
\end{equation} 
Now, we apply the spectral cut-off operator $ \Delta_j$ to \eqref{eq:elliptic_problem}. We get 
\begin{equation*}
-\, \div \left( A_\veps \Delta_j\nabla \Pi_\veps^{n+1}\right)=\div \left(\left[ \Delta_j,A_\veps\right]\nabla \Pi_\veps^{n+1}\right)+\, \div  \Delta_j F_\veps\, ,
\end{equation*}
for all $j\geq 0$ and where we have defined $A_\veps:=\left(1+\veps a_\veps^{n+1}\right)$ and $F_\veps:=\veps \ue^n \cdot \nabla \ue^{n+1}+\ue^{\perp,n+1}$.
Hence multiplying both sides by $ \Delta_j  \Pi_\veps^{n+1}$ and integrating over $\R^2$, we have 
\begin{equation*}
-\int_{\R^2}\Delta_j\Pi_\veps^{n+1} \div \left( A_\veps \Delta_j\nabla \Pi_\veps^{n+1}\right) \, \dx= \int_{\R^2}\Delta_j\Pi_\veps^{n+1} \div \left( \left[\Delta_j, A_\veps\right] \nabla \Pi_\veps^{n+1}\right) \, \dx+\int_{\R^2}\Delta_j\Pi_\veps^{n+1} \div \Delta_j F_\veps \, \dx.
\end{equation*}
Since for $j\geq 0$ we have $\|\Delta_j \nabla \Pi_\veps^{n+1}\|_{L^2}\sim 2^j\|\Delta_j  \Pi_\veps^{n+1}\|_{L^2}$ (according to Lemma \ref{l:bern}) and using H\"older's inequality (see Proposition \ref{app:hold_ineq}) for the right-hand side, we obtain for all $j\geq 0$: 
\begin{equation*}
2^{j}\| \Delta_j \nabla \Pi_\veps^{n+1}\|^2_{L^2}\leq C\| \Delta_j \nabla \Pi_\veps^{n+1}\|_{L^2}\left( \| \div \left[ \Delta_j,A_\veps \right]\nabla \Pi_\veps^{n+1}\|_{L^2}+\|\div \Delta_j F_\veps\|_{L^2} \right)\, .
\end{equation*}
To deal with the former term on the right-hand side, we take advantage of the following commutator estimate (see Lemma \ref{l:commutator_pressure}):
\begin{equation*}
\| \div \left[ \Delta_j,A_\veps \right]\nabla \Pi_\veps^{n+1}\|_{L^2}\leq C\, c_j \, 2^{-j(s-1)}\|\nabla A_\veps\|_{H^{s-1}}\|\nabla \Pi_\veps^{n+1}\|_{H^{s-1}}\, ,
\end{equation*}
for a suitable sequence $(c_j)_{j}$ belonging to the unit sphere of $\ell^2$.

After multiplying by $2^{j(s-1)}$, we get
 \begin{equation*}
2^{js}\| \Delta_j \nabla \Pi_\veps^{n+1}\|_{L^2}\leq C\left(c_j\,  \|\nabla A_\veps\|_{H^{s-1}}\|\nabla \Pi_\veps^{n+1}\|_{H^{s-1}}+2^{j(s-1)}\|\div \Delta_j F_\veps\|_{L^2} \right)\, .
\end{equation*}
Taking the $\ell^2$ norm of both sides and summing up the low frequency blocks related to $\Delta_{-1}\nabla \Pi_\veps^{n+1}$, we may have 
\begin{equation*}
\|\nabla \Pi_\veps^{n+1}\|_{H^s}\leq C\left(  \|\nabla A_\veps\|_{H^{s-1}}\|\nabla \Pi_\veps^{n+1}\|_{H^{s-1}}+\|\div  F_\veps\|_{H^{s-1}}+\|\Delta_{-1}\nabla \Pi_\veps^{n+1}\|_{L^2} \right)\, .
\end{equation*}
We observe that $\|\Delta_{-1}\nabla \Pi_\veps^{n+1}\|_{L^2}\leq C\|\nabla \Pi_\veps^{n+1}\|_{L^2}$ and 
\begin{equation*}
\|\nabla \Pi_\veps^{n+1}\|_{H^{s-1}}\leq C\|\nabla \Pi_\veps^{n+1}\|_{L^2}^{1/s}\|\nabla \Pi_\veps^{n+1}\|_{H^s}^{1-1/s}\, .
\end{equation*}
Therefore,
\begin{equation*}
\|\nabla \Pi_\veps^{n+1}\|_{H^s}\leq C\left(  \|\nabla A_\veps\|_{H^{s-1}}\|\nabla \Pi_\veps^{n+1}\|_{L^2}^{1/s}\|\nabla \Pi_\veps^{n+1}\|_{H^{s}}^{1-1/s}+\|\div  F_\veps\|_{H^{s-1}}+\|\nabla \Pi_\veps^{n+1}\|_{L^2} \right)\, .
\end{equation*}
Then applying Young's inequality (see Proposition \ref{app:young_ineq}) we finally infer that 
\begin{equation}\label{est_Pi_H^s_1}
\|\nabla \Pi_\veps^{n+1}\|_{H^s}\leq C\left(  \left(1+\|\nabla A_\veps\|_{H^{s-1}}\right)^s \|\nabla \Pi_\veps^{n+1}\|_{L^2}+\|\div  F_\veps\|_{H^{s-1}} \right)\, .
\end{equation}
It remains to analyse the term $\div F_\veps$ where $F_\veps:=\veps \ue^n \cdot \nabla \ue^{n+1}+\ue^{\perp,n+1}$. Due to the divergence-free conditions, we can write 
\begin{equation*}
\div (\ue^n \cdot \nabla \ue^{n+1} )=\nabla \ue^{n}:\nabla \ue^{n+1}
\end{equation*}
and as $H^{s-1}$ is an algebra, the term $\div (\ue^n \cdot \nabla \ue^{n+1} )$ is in $H^{s-1}$, with 
\begin{equation}\label{est:div_u}
\|\div (\ue^n \cdot \nabla \ue^{n+1} )\|_{H^{s-1}}\leq C \|\ue^n\|_{H^s}\|\ue^{n+1}\|_{H^s}\, .
\end{equation}
Putting \eqref{est:Pi_L^2} and \eqref{est:div_u} into \eqref{est_Pi_H^s_1}, we find that 
\begin{equation}\label{est_Pi_final}
\begin{split}
\|\nabla \Pi_\veps^{n+1}\|_{ H^{s}}&\leq C  \left(1+\veps \|\nabla a_\veps^{n+1}\|_{ H^{s-1}}\right)^s \left(\veps \|\ue^n\|_{L^2}\|\ue^{n+1}\|_{ H^s}+\|\ue^{\perp,n+1}\|_{L^2}\right)\\
&+C\, \left(\veps \|\ue^{n}\|_{ H^{s}}\|\ue^{n+1}\|_{ H^{s}} +\|\ue^{\perp,n+1}\|_{ H^{s}}\right)\\
&\leq C  \left(1+\veps \|\nabla a_\veps^{n+1}\|_{ H^{s-1}}\right)^s \left(\veps \|\ue^n\|_{H^s}+1\right)\|\ue^{n+1}\|_{ H^s}\, ,
\end{split}
\end{equation} 
which implies the $L^\infty_T(H^s)$ estimate for the pressure term:
\begin{equation}\label{est:Pi^(n+1)}
\|\nabla \Pi_\veps^{n+1}\|_{L^{\infty}_T H^{s}}\leq C \left(1+\veps \|\nabla a_\veps^{n+1}\|_{L^\infty_T  H^{s-1}}\right)^s\left(\veps \|\ue^n\|_{L^\infty_T H^s}+1\right)\|\ue^{n+1}\|_{L^\infty_TH^s}\, .
\end{equation}
Combining all the previous estimates together with a Gr\"onwall type inequality (again we refer to Section \ref{app:sect_gron}), we finally obtain an estimate for the velocity field:
\begin{equation}\label{est:u^(n+1)}
\sup_{0\leq t\leq T}\|\ue^{n+1}(t)\|_{H^s}\leq \|\vec u_{0,\veps}^{n+1}\|_{H^s}\exp \left(\int_0^T A_n(t)\, \dt\right)\, ,
\end{equation}
where 
\begin{equation*}
\begin{split}
A_n(t)=C\left(\|\an (t)\|_{L^\infty}+\|\nabla a_\veps^{n+1}(t)\|_{ H^{s-1}}\right)\left(1+\veps \|\nabla a_\veps^{n+1}(t)\|_{ H^{s-1}}\right)^s\left(\veps \|\ue^n(t)\|_{H^{s}}+1\right)+C\|\ue^n(t)\|_{ H^s}\, .
\end{split}
\end{equation*}

We point out that the above constants $C$ do not depend on $n$ nor on $\veps$.

The scope in what follows is to obtain uniform estimates by induction. Thanks to the assumptions stated in Paragraph \ref{ss:FormProb_E}, we can suppose that the initial data satisfy 
\begin{equation*}
\|a_{0,\veps}\|_{L^\infty}\leq \frac{C_0}{2}\, ,\quad \quad \|\nabla a_{0,\veps}\|_{H^{s-1}}\leq \frac{C_1}{2} \quad \quad \text{and}\quad \quad \|\vec{u}_{0,\veps}\|_{H^s}\leq \frac{C_2}{2} \, ,
\end{equation*}
for some $C_0, C_1, C_2>0$. Due to the relation \eqref{eq:transport_density} we immediately infer that, for all $n\geq 0$, 
\begin{equation*}
\|\an \|_{L^\infty_t L^\infty}\leq C\|a_{0,\veps}\|_{L^\infty}\leq C\, C_0 \quad \text{for all }t\in \R_+.
\end{equation*}
At this point, the aim is to show (by induction) that the following uniform bounds hold  for all $n\geq 0$:
\begin{equation}\label{eq:unifbounds}
\begin{split}
&\|\nabla a_\veps^{n+1}\|_{L^\infty_{T^\ast}H^{s-1}}\leq C_1\\
&\|\ue^{n+1}\|_{L^\infty_{T^\ast}H^s}\leq C_2\\
&\|\nabla \Pi_\veps^{n+1}\|_{L^\infty_{T^\ast}H^s}\leq C_3\, ,
\end{split}
\end{equation}
provided that $T^\ast$ is sufficiently small.

The previous estimates in \eqref{eq:unifbounds} obviously hold for $n=0$. At this point, we will prove them for $n+1$, supposing that the controls in \eqref{eq:unifbounds} are true for $n$.
From \eqref{est:a^(n+1)}, \eqref{est:u^(n+1)} and \eqref{est:Pi^(n+1)} we obtain 
\begin{align*}
&\|\nabla a_\veps^{n+1}\|_{L^\infty_{T}H^{s-1}}\leq \frac{C_1}{2}\exp \Big(CTC_2\Big)\\
&\|\ue^{n+1}\|_{L^\infty_{T}H^s}\leq \frac{C_2}{2}\exp \Big(CT(C_0+C_1)\left(1+\veps C_1\right)^s(\veps C_2+1)C_2 \Big)\\
&\|\nabla \Pi_\veps^{n+1}\|_{L^\infty_{T}H^s}\leq C(\veps C_2+1)\left(1+\veps C_1\right)^s\|\ue^{n+1}\|_{L^\infty_{T}H^s}\, .
\end{align*}
So we can choose $T^{\ast}$ such that $\exp \Big(\max\{C_0+C_1,\, 1\}\, CT\left(1+ C_1\right)^s(1+ C_2)\, C_2 \Big)\leq 2$. Notice that $T^\ast$ does not depend on $\veps$. 
 
Thus, by induction, \eqref{eq:unifbounds} holds for the step $n+1$, and therefore it is true for any $n\in \N$. 

\subsection{Convergence}\label{ss:conv_H^s}
To prove the convergence, we will estimate the difference between two iterations at different steps. First of all, let us define 
\begin{equation*}
\widetilde{a}_\veps^{n}:=a_\veps^n -a^n_{0,\veps}\, ,
\end{equation*}
that satisfies the transport equation
\begin{equation*}
\begin{cases}
\d_t \widetilde{a}^n_\veps+\ue^{n-1}\cdot \nabla \widetilde{a}_\veps^n=-\ue^{n-1}\cdot \nabla a_{0,\veps}^n\\
\widetilde{a}{_{\veps}^n}_{|t=0}=0\, .
\end{cases}
\end{equation*}
Hence, since the right-hand side is definitely uniformly bounded (with respect to $n$) in $L^1_{\rm loc}(\R_+;L^2)$, from classical results for transport equations we get that $(\widetilde{a}^n_\veps)_{n\in \N}$ is uniformly bounded in $C^0([0,T];L^2)$. Now, we want to prove that the sequence $(\widetilde{a}^n_\veps, \ue^n, \nabla \Pi^n_\veps)_{n\in \N}$ is a Cauchy sequence in $C^0([0,T];L^2)$. 
So, let us define, for $(n,l)\in \N^2$, the following quantities,
\begin{align}
&\delta {a}_\veps^{n,l}:=\anl-a_\veps^n \nonumber\\
&\delta\widetilde{a}_\veps^{n,l} :=\widetilde{a}_\veps^{n+l}-\widetilde{a}_\veps^{n}=\delta a_\veps^{n,l}-\delta a_{0,\veps}^{n,l}\, , \quad \text{ where }\quad \delta {a}_{0,\veps}^{n,l}:=a_{0,\veps}^{n+l}-a_{0,\veps}^n \nonumber\\
&\delta \ue^{n,l}:=\ue^{n+l}-\ue^n \label{general_uniform_bounds}\\
&\delta \Pi_\veps^{n,l}:=\Pi_\veps^{n+l}-\Pi_\veps^n \nonumber \, .
\end{align}
Of course, we have that $\div \delta \vec u_\veps^{n,l}=0$ for any $(n,l)\in \N^2$.

The previous quantities defined in \eqref{general_uniform_bounds} solve the system
\begin{equation}\label{syst_convergence_analysis}
\begin{cases}
\d_t \delta \widetilde{a}_{\veps}^{n,l} +\vu_\veps^{n+l-1} \cdot \nabla \delta \widetilde{a}_\veps^{n,l}=-\delta\ue^{n-1,l}\cdot \nabla a_\veps^n-\ue^{n+l-1}\cdot \nabla \delta a^{n,l}_{0,\veps}\\[1ex]
\d_t\delta \vu_{\veps}^{n,l}+ \vu_{\veps}^{n+l-1} \cdot \nabla \delta\vu_{\veps}^{n,l}=-\delta \ue^{n-1,l} \cdot \nabla \ue^n \\[0.5ex]
\qquad \qquad \qquad \qquad \qquad \qquad \qquad \qquad - \dfrac{1}{\veps} (\delta \vu_{\veps}^{n,l})^\perp-(1+\veps a_\veps^{n+l})\dfrac{1}{\veps}\nabla \delta \Pi_\veps^{n,l}-\delta a_\veps^{n,l} \nabla \Pi_\veps^n \\[2ex]
\div  \delta \vu_{\veps}^{n,l} =0 \\[1ex]
(\delta \widetilde{a}_\veps^{n,l},\delta \vu_\veps^{n,l})_{|t=0}=(0,\delta \vu_{0,\veps}^{n,l})\, .
\end{cases}
\end{equation}

We perform an energy estimate for the first equation in \eqref{syst_convergence_analysis}, getting
\begin{equation*}
\|\delta \widetilde{a}_\veps^{n,l}(t)\|_{L^2}\leq C\int_0^t\left(\|\nabla a_\veps^n\|_{L^\infty}\|\delta \vu_\veps^{n-1,l}\|_{L^2}+\|\ue^{n+l-1}\|_{L^2}\|\nabla \delta a_{0,\veps}^{n,l}\|_{L^\infty}\right)\, \detau\, .
\end{equation*} 
Moreover, from the momentum equation multiplied by $\delta \ue^{n,l}$, integrating by parts over $\R^2$, we obtain 
\begin{equation*}
\begin{split}
\int_{\R^2}\frac{1}{2}\d_t\, |\delta \ue^{n,l}|^2=-\int_{\R^2}(\delta \ue^{n-1,l}\cdot \nabla \ue^{n})\cdot\delta \ue^{n,l}+\int_{\R^2}(a_\veps^{n+l}\, \nabla \delta \Pi_\veps^{n,l})\cdot \delta \ue^{n}
+\int_{\R^2}(\delta a_\veps^{n,l}\, \nabla  \Pi_\veps^{n})\cdot \delta \ue^{n,l}\, ,
\end{split}
\end{equation*} 
which implies 
\begin{equation*}
\begin{split}
\|\delta \ue^{n,l}(t)\|_{L^2}&\leq C \|\delta \vec u_{0,\veps}^{n,l}\|_{L^2}+C\int_0^t\|\nabla \ue^{n}(\tau)\|_{L^\infty}\|\delta \ue^{n-1,l}(\tau)\|_{L^2}+\|a_\veps^{n+l}(\tau)\|_{L^\infty}\|\nabla \delta \Pi_\veps^{n,l}(\tau)\|_{L^2}\, \detau\\
&+C\int_0^t\left(\|\delta \widetilde{a}_\veps^{n,l}(\tau )\|_{L^2}+\|\delta a_{0,\veps}^{n,l}\|_{L^\infty}\right)\|\nabla \Pi_\veps^{n}(\tau)\|_{L^2 \cap L^\infty} \, \detau\, ,
\end{split}
\end{equation*} 
where we have also employed the fact that $\delta a_\veps^{n,l} =\delta \widetilde{a}_\veps^{n,l}+\delta a_{0,\veps}^{n,l}$. 

Finally, for the pressure term we take the $\div$ operator in the momentum equation of system \eqref{syst_convergence_analysis}, obtaining 
\begin{equation*}
-\div \left((1+\veps a_\veps^{n+l})\frac{1}{\veps}\nabla \delta \Pi_\veps^{n,l}\right)=\div \left(-\delta \vu_{\veps}^{n,l} \cdot \nabla \vu_{\veps}^{n+l-1} +\delta \ue^{n-1,l} \cdot \nabla \ue^n + \frac{1}{\veps} (\delta \vu_{\veps}^{n,l})^\perp+\delta a_\veps^{n,l} \nabla \Pi_\veps^n \right),
\end{equation*} 
so that we have
\begin{equation}\label{Pi_cauchy}
\begin{split}
\|\nabla \delta \Pi_\veps^{n,l}\|_{L^2}&\leq C\veps \left(\|\delta
\ue^{n-1,l}\|_{L^2}\|\nabla \ue^n\|_{L^\infty}+\|\delta
a_\veps^{n,l}\, \nabla \Pi_\veps^n\|_{L^2}\right)\\
&+C\veps\|\delta
\ue^{n,l}\|_{L^2}\|\nabla \ue^{n+l-1}\|_{L^\infty}+C\|(\delta \ue^{n,l})^\perp\|_{L^2}\\
&\leq C\veps \left(\|\delta
\ue^{n-1,l}\|_{L^2}\|\nabla \ue^n\|_{L^\infty}+\|\delta
\widetilde{a}_\veps^{n,l}\|_{L^2}\|\nabla \Pi_\veps^n\|_{L^\infty}+\|\delta
a_{0,\veps}^{n,l}\|_{L^\infty}\|\nabla \Pi_\veps^n\|_{L^2}\right)\\
&+C\veps\|\delta
\ue^{n,l}\|_{L^2}\|\nabla \ue^{n+l-1}\|_{L^\infty} +C\|\delta \ue^{n,l}\|_{L^2}\, .
\end{split}
\end{equation}
At this point, applying Gr\"onwall lemma (see Lemma \ref{app:lem_gron}) and using the bounds established in Paragraph \ref{ss:unif_est}, we thus argue that for $t\in [0,T^\ast]$:
\begin{equation*}
\begin{split}
\|\delta \widetilde{a}_\veps^{n,l}(t)\|_{L^2}+\|\delta \ue^{n,l}(t)\|_{L^2}&\leq C_{T^\ast}\left(\|\nabla \delta a_{0,\veps}^{n,l}\|_{L^\infty}+\| \delta a_{0,\veps}^{n,l}\|_{L^\infty}+\|\delta \vec u_{0,\veps}^{n,l}\|_{L^2}\right)\\
&+C_{T^\ast}\int_0^t\left(\|\delta \widetilde{a}_\veps^{n-1,l}(\tau)\|_{L^2}+\|\delta \ue^{n-1,l}(\tau)\|_{L^2}\right) \detau\, ,
\end{split}
\end{equation*}
where the constant $C_{T^\ast}$ depends on $T^\ast$ and on the initial data. 

After setting
\begin{equation*}
F_0^n:=\sup_{l\geq 0}\left(\|\nabla \delta a_{0,\veps}^{n,l}\|_{L^\infty}+\| \delta a_{0,\veps}^{n,l}\|_{L^\infty}+\|\delta \vec u_{0,\veps}^{n,l}\|_{L^2}\right)
\end{equation*}
and
\begin{equation*}
G^n(t):=\sup_{l\geq 0}\sup_{[0,t]}\left(\|\delta \widetilde{a}_\veps^{n,l}(\tau)\|_{L^2}+\|\delta \ue^{n,l}(\tau)\|_{L^2}\right), 
\end{equation*}
by induction we may conclude that, for all $t\in [0,T^\ast]$, 
\begin{equation*}
\begin{split}
G^n(t)\leq C_{T^\ast}\sum_{k=0}^{n-1}\frac{(C_{T^\ast}T^\ast)^k}{k!}F_0^{n-k}+\frac{\left(C_{T^\ast}T^\ast\right)^n}{n!}G^0(t)\, .
\end{split}
\end{equation*}
Now, bearing \eqref{conv_in_data} in mind, we have 
\begin{equation*}
\lim_{n\rightarrow +\infty}F_0^n=0\, .
\end{equation*}
Hence, we may infer that
\begin{equation}\label{est_fin_conv}
\lim_{n\rightarrow +\infty}\sup_{l\geq 0}\sup_{t\in [0,T^\ast]}\left(\|\delta \widetilde{a}_\veps^{n,l}(t)\|_{L^2}+\|\delta \ue^{n,l}(t)\|_{L^2}\right)=0\, .
\end{equation}

Property \eqref{est_fin_conv} implies that both $(\widetilde{a}_\veps^n)_{n\in \N}$ and $(\ue^n)_{n\in \N}$ are Cauchy sequences in $C^0([0,T^\ast];L^2)$: therefore, such sequences converge to some functions $\widetilde{a}_\veps$ and $\ue$ in the same space. Taking advantage of previous computations in \eqref{Pi_cauchy}, we have also that $(\nabla \Pi_\veps^n)_{n\in \N}$ converge to a function $\nabla \Pi_\veps$ in $C^0([0,T^\ast];L^2)$.

Now, we define $a_\veps:= a_{0,\veps} +\widetilde{a}_\veps$. Hence, $ a_\veps -a_{0,\veps}$ is in $C^0([0,T^\ast];L^2)$.
Moreover, as $(\nabla a_\veps^n)_{n\in \N}$ is uniformly bounded in $L^\infty ([0,T^\ast];H^{s-1})$ and Sobolev spaces have the \textit{Fatou property} (we refer to Proposition \ref{proposition_Fatou} in this respect), we deduce that $\nabla a_\veps$ belongs to the same space. Moreover, since $(a^n_\veps )_{n\in \N}$ is uniformly bounded in $L^\infty ([0,T^\ast]\times \R^2)$, we also have that $a_\veps\in L^\infty ([0,T^\ast]\times \R^2)$. Analogously, as $(\ue^n)_{n\in \N}$ and $(\nabla \Pi_\veps^n)_{n\in \N}$ are uniformly bounded in $L^\infty ([0,T^\ast];H^s)$, we deduce that $\ue$ and $\nabla \Pi_\veps$ belong to $L^\infty ([0,T^\ast];H^s)$. 
 

Due to an interpolation argument, we see that the above sequences converge strongly in every intermediate $C^0([0,T^\ast];H^\sigma)$ for all $\sigma <s$. This is enough to pass to the limit in the equations satisfied by $(a_\veps^n,\ue^n,\nabla \Pi_\veps^n)_{n\in \N}$. Hence, $(a_\veps, \ue,\nabla \Pi_\veps)$ satisfies the original problem \eqref{Euler-a_eps_1}.

This having been established, we look at the time continuity of $a_\veps$. We exploit the transport equation:
$$\d_t a_{\veps}=-\vu_\veps \cdot \nabla a_\veps\, ,$$
noticing that the term on the right-hand side belongs to $L^\infty_{T^\ast}(L^\infty)$. Thus, we can deduce that $\d_t a_\veps \in L^\infty_{T^\ast}(L^\infty)$. Moreover, by embeddings, we already know that $\nabla a_\veps \in L^\infty_{T^\ast}(L^\infty)$. The previous two relations imply that $a_\veps \in W^{1,\infty}_{T^\ast}(L^\infty)\cap L^\infty_{T^\ast}(W^{1,\infty})$. That give us the desired regularity property $a_\veps \in C^0([0,T^\ast]\times \R^2)$. In addition, looking at the momentum equation in \eqref{Euler-a_eps_1} and employing Theorem \ref{thm_transport}, one obtains the claimed time regularity property for $\ue$. At this point, the time regularity for the pressure term $\nabla \Pi_\veps$ is recovered from the elliptic problem \eqref{general_elliptic_est}.

\subsection{Uniqueness}\label{ss:uniqueness_H^s}
We conclude this section showing the uniqueness of solutions in our framework.

We start by stating a uniqueness result, that is a consequence of a standard stability result based on energy methods. Since the proof is similar to the convergence argument of the previous paragraph, we will omit it (see e.g. \cite{D_JDE} for details). We recall that, in what follows, the parameter $\veps>0$ is fixed.
\begin{theorem} \label{th:uniq}
Let $\left(\vrho_\veps^{(1)}, \ue^{(1)}, \nabla \Pi_\veps^{(1)}\right)$ and $\left(\vrho_\veps^{(2)}, \ue^{(2)}, \nabla \Pi_\veps^{(2)}\right)$ be two solutions
to the Euler system \eqref{Euler_eps} associated with the initial data $\left(\vrho_{0,\veps}^{(1)}, \vec u_{0,\veps}^{(1)}\right)$ and $\left(\vrho_{0,\veps}^{(2)}, \vec u_{0,\veps}^{(2)}\right)$. Assume that, for some $T>0$, one has the following properties:
\begin{enumerate}[(i)]
\item the densities $\vrho_\veps^{(1)}$ and $\vrho_\veps^{(2)}$ are bounded and bounded away from zero;
\item the quantities $\delta \vrho_\veps:=\vrho_\veps^{(2)}-\vrho_\veps^{(1)}$ and $\delta \ue:=\ue^{(2)}-\ue^{(1)}$ belong to the space $C^1\big([0,T];L^2(\R^2)\big)$;
\item $\nabla \ue^{(1)}$, $\nabla \vrho_\veps^{(1)}$ and $\nabla \Pi_\veps^{(1)}$ belong to $ L^1\big([0,T];L^\infty(\R^2)\big)$.
\end{enumerate}

Then, for all $t\in[0,T]$, we have the stability inequality: 
\begin{equation}\label{stability_ineq}
\|\delta \vrho_\veps (t)\|_{L^2}+ \left\|\sqrt{\vrho_\veps^{(2)}(t)}\, \delta \ue (t)\right\|_{L^2}\leq \left(\|\delta \vrho_{0,\veps}\|_{L^2}+ \left\|\sqrt{\vrho_{0,\veps}^{(2)}}\, \delta \vec u_{0,\veps}\right\|_{L^2}\right)\, e^{CA(t)}\, ,
\end{equation}
for a universal constant $C>0$, where we have defined
\begin{equation*}\label{eq:def_A(t)}
A(t):=\int_0^t\left(\left\|\frac{\nabla \vrho_\veps^{(1)}}{\sqrt{\vrho_\veps^{(2)}}}\right\|_{L^\infty}+\left\|\frac{\nabla \Pi_\veps^{(1)}}{\vrho_\veps^{(1)}\sqrt{\vrho_\veps^{(2)}}}\right\|_{L^\infty}+\|\nabla \ue^{(1)}\|_{L^\infty}\right)\, \detau\, .
\end{equation*} 
\end{theorem}

It is worth to notice that, adapting the \textit{relative entropy} arguments presented
in Subsection 4.3 of \cite{C-F_RWA}, we can replace (in the statement above) the $C^1_T(L^2)$ requirement for $\delta \vrho_\veps$ and $\delta \ue$ with the $C^0_T
(L^2)$ regularity. However, one needs to pay an additional $L^2$ assumption on the densities. In this way, we will have a weak-strong uniqueness type result and we will prove it in the next theorem.

Concerning weak-strong results for density-dependent fluids, we refer to \cite{Ger}, where Germain exhibited a weak-strong uniqueness property within  a class of (weak) solutions to the compressible Navier-Stokes system satisfying a relative entropy inequality with respect to a (hypothetical) strong solution of the same problem (see also the work \cite{F-N-S} by Feireisl, Novotn\'y and Sun). 
Moreover, in \cite{F-J-N}, the authors established the weak-strong uniqueness property in the class of finite energy weak solutions, extending thus the classical results of Prodi \cite{Pr} and Serrin \cite{Ser} to the class of compressible flows.

Before presenting the proof of the weak-strong uniqueness result, we state the definition of a \textit{finite energy weak solution} to system (2.1), such that $\vrho_{0,\veps}-1\in L^2(\R^2)$. We also recall that our densities have the form $\vrho_\veps=1+\veps R_\veps$.
\begin{definition}\label{weak_sol_L^2}
Let $T>0$ and $\veps \in \; ]0,1]$ be fixed. Let $(\vrho_{0,\veps},\vec u_{0,\veps})$ be an initial datum fulfilling the assumptions in Paragraph \ref{ss:FormProb_E}.
We say that $(\vrho_\veps, \ue)$ is a \textit{finite energy weak solution} to system \eqref{Euler_eps} in $[0,T]\times \R^2$, related to the previous initial datum, if:
\begin{itemize}
\item $\vrho_\veps \in L^\infty([0,T]\times \R^2)$ and $\vrho_\veps -1\in C^0([0,T];L^2(\R^2))$;
\item $\ue \in L^\infty([0,T];L^2(\R^2))\cap C_w^0([0,T];L^2(\R^2))$;
\item the mass equation is satisfied in the weak sense: 
\begin{equation*}
\int_0^T\int_{\R^2}\Big(\vrho_\veps \, \d_t \varphi+ \vrho_\veps \ue \cdot \nabla \varphi\Big) \, \dxdt +\int_{\R^2}\vrho_{0,\veps}\varphi(0, \cdot )\dx =\int_{\R^2}\vrho_\veps (T)\varphi(T,\cdot)\dx \, ,
\end{equation*}
for all $\varphi\in C_c^\infty ([0,T]\times \R^2;\R)$;
\item the divergence-free condition $\div \ue =0$ is satisfied in $\mathcal{D}^\prime (]0,T[\, \times \R^2)$;
\item the momentum equation is satisfied in the weak sense:
\begin{equation*}
\begin{split}
\int_0^T\int_{\R^2}\Big(\vrho_\veps \ue \cdot \d_t \vec \psi +[\vrho_\veps\ue \otimes \ue] :\nabla \vec \psi - \frac{1}{\veps}\vrho_\veps \ue^\perp \cdot \vec \psi \Big) \dxdt+ \int_{\R^2}\vrho_{0,\veps}& \vec u_{0,\veps}\cdot \vec \psi (0)\dx \\
&=\int_{\R^2}\vrho_\veps(T)\vec u_\veps(T)\vec \psi(T)\dx,
\end{split}
\end{equation*}
for any $\vec \psi \in C_c^\infty([0,T]\times \R^2;\R^2)$ such that $\div \vec \psi=0$;
\item for almost every $t\in [0,T]$, the two following energy balances hold true:
\begin{equation*}
\int_{\R^2} \vrho_\veps (t) |\ue (t)|^2\dx\leq \int_{\R^2} \vrho_{0,\veps} |\vec u_{0,\veps}|^2\dx \quad \text{and}\quad \int_{\R^2}(\vrho_{\veps}(t)-1)^2\dx\leq \int_{\R^2}(\vrho_{0,\veps}-1)^2\dx\, .
\end{equation*}
\end{itemize}
\end{definition}

\begin{theorem} \label{th:uniq_bis}
Let $\veps\in \, ]0,1]$ be fixed. Let $\left(\vrho_\veps^{(1)}, \ue^{(1)}\right)$ and $\left(\vrho_\veps^{(2)}, \ue^{(2)}\right)$ be two finite energy weak solutions
to the Euler system \eqref{Euler_eps} as in Definition \ref{weak_sol_L^2} with initial data $\left(\vrho_{0,\veps}^{(1)}, \vec u_{0,\veps}^{(1)}\right)$ and $\left(\vrho_{0,\veps}^{(2)}, \vec u_{0,\veps}^{(2)}\right)$. Assume that, for some $T>0$, one has the following properties:
\begin{enumerate}[(i)]
\item $\nabla \ue^{(1)}$ and $\nabla \vrho_\veps^{(1)}$ belong to $ L^1\big([0,T];L^\infty(\R^2)\big)$;
\item $\nabla \Pi_\veps^{(1)}$ is in $ L^1\big([0,T];L^\infty(\R^2)\cap L^2(\R^2)\big)$.
\end{enumerate}

Then, for all $t\in[0,T]$, we have the stability inequality \eqref{stability_ineq}. 
\end{theorem}
\begin{proof} We start by defining, for $i=1,2$:
\begin{equation*}
R_\veps^{(i)}:=\frac{\vrho_\veps^{(i)}-1}{\veps}\quad \text{and}\quad R_{0,\veps}^{(i)}:=\frac{\vrho_{0,\veps}^{(i)}-1}{\veps}\, ,
\end{equation*}
and we notice that, owing to the continuity equation in \eqref{Euler_eps} and the divergence-free conditions $\div \vec u_\veps^{(i)}=0$, one has
\begin{equation}\label{ws:transport_R}
\d_tR_\veps^{(i)}+\div (R_\veps^{(i)}\vec u_\veps^{(i)})=0 \quad \text{with}\quad R_\veps^{(i)}(0)=R_{0,\veps}^{(i)}.
\end{equation}
 
For simplicity of notation, we fix $\veps=1$ throughout this proof and let us assume for a while the couple $(R^{(1)},\vec u^{(1)})$ be a pair of smooth functions such that $R^{(1)},\vec u^{(1)}\in C^\infty_c(\R_+\times \R^2)$ and $\div \vec u^{(1)}=0$, with the support of $R^{(1)}$ and $\vec u^{(1)}$ included in $[0,T]\times \R^2$.
First of all, we use $\vec u^{(1)}$ as a test function in the weak formulation of the momentum equation, finding that 
\begin{align}
\int_{\R^2} \vrho^{(2)}(T)\vec u^{(2)}(T)\cdot \vec u^{(1)}(T) \dx&= \int_{\R^2}\vrho_0^{(2)}\vec u_0^{(2)}\cdot \vec u_0^{(1)} \dx +\int_0^T\int_{\R^2}\vrho^{(2)}\vec u^{(2)}\cdot \d_t\vec u^{(1)} \dxdt \label{mom-eq-reg-1}\\
&+\int_0^T\int_{\R^2}(\vrho^{(2)}\vec u^{(2)}\otimes \vec u^{(2)}):\nabla \vec u^{(1)} \dxdt+\int_0^T\int_{\R^2}\vrho^{(2)}\vec u^{(2)}\cdot (\vec u^{(1)})^\perp \dxdt \nonumber\, ,
\end{align}
where we have also noted that $ (\vec u^{(2)})^\perp\cdot \vec u^{(1)}=- \vec u^{(2)}\cdot (\vec u^{(1)})^\perp$.

Next, testing the mass equation against $|\vec u^{(1)}|^2/2$, we obtain
\begin{equation}\label{mom-eq-reg-2}
\begin{split}
\frac{1}{2}\int_{\R^2}\vrho^{(2)}(T)|\vec u^{(1)}(T)|^2\dx&=\frac{1}{2}\int_{\R^2}\vrho_0^{(2)}|\vec u_0^{(1)}|^2\dx +\int_0^T\int_{\R^2}\vrho^{(2)}\vec u^{(1)}\cdot \d_t \vec u^{(1)} \dxdt\\
&+\frac{1}{2}\int_0^T\int_{\R^2}\vrho^{(2)}\vec u^{(2)}\cdot \nabla |\vec u^{(1)}|^2 \dxdt\\
&=\frac{1}{2}\int_{\R^2}\vrho_0^{(2)}|\vec u_0^{(1)}|^2\dx +\int_0^T\int_{\R^2}\vrho^{(2)}\vec u^{(1)}\cdot \d_t \vec u^{(1)} \dxdt\\
&+\frac{1}{2}\int_0^T\int_{\R^2}(\vrho^{(2)}\vec u^{(2)}\otimes\vec u^{(1)}): \nabla \vec u^{(1)} \dxdt.
\end{split}
\end{equation}
Recall also that the energy inequality reads
\begin{equation*}
\frac{1}{2}\int_{\R^2}\vrho^{(2)}(T)|\vec u^{(2)}(T)|^2\dx \leq\dfrac{1}{2}\int_{\R^2}\vrho_0^{(2)}|\vec u_0^{(2)}|^2\dx \, .
\end{equation*}
Now, we take care of the density oscillations $R^{(i)}$. We test the transport equation \eqref{ws:transport_R} for $R^{(2)}$ against $R^{(1)}$, getting 
\begin{align}
\int_{\R^2}R^{(2)}(T)R^{(1)}(T)\dx&=\int_{\R^2}R_0^{(2)}R_0^{(1)}\dx \label{transp-eq-reg-1}\\
&+\int_0^T\int_{\R^2}R^{(2)}\d_tR^{(1)}\dxdt+\int_0^T\int_{\R^2}R^{(2)}\vec u^{(2)}\cdot \nabla R^{(1)}\dxdt \nonumber .
\end{align}
Recalling Definition \ref{weak_sol_L^2}, we have the following energy balance:
\begin{equation*}
\int_{\R^2}|R^{(2)}(T)|^2\dx \leq\int_{\R^2}|R^{(2)}_0|^2\dx\, .
\end{equation*}
At this point, testing $\d_t1+\div(1\, \vec u^{(2)})=0$ against $|R^{(1)}|^2/2$, we may infer that 
\begin{equation}\label{transp-eq-reg-2}
\frac{1}{2}\int_{\R^2} |R^{(1)}(T)|^2\dx=\frac{1}{2}\int_{\R^2}|R_0^{(1)}|^2 \dx +\int_0^T\int_{\R^2}R^{(1)}\d_tR^{(1)}\dxdt+\int_0^T\int_{\R^2}R^{(1)}\vec u^{(2)}\cdot \nabla R^{(1)}\dxdt .
\end{equation}
Now, for notational convenience, let us define 
\begin{equation*}
\delta R:=R^{(2)}-R^{(1)}\quad \text{and}\quad \delta \vec u:=\vec u^{(2)}-\vec u^{(1)}\, .
\end{equation*} 
Putting all the previous relations together, we obtain 
\begin{align}
\frac{1}{2}\int_{\R^2}\Big(\vrho^{(2)}(T)|\delta \vec u(T)|^2+|\delta R(T)|^2\Big) \dx&\leq \frac{1}{2}\int_{\R^2}\Big(\vrho_0^{(2)}|\delta \vec u_0|^2+|\delta R_0|^2\Big) \dx-\int_0^T\int_{\R^2}\vrho^{(2)}\vec u^{(2)}\cdot (\vec u^{(1)})^\perp\dxdt \nonumber\\
&-\int_0^T\int_{\R^2}\Big(\vrho^{(2)}\delta \vec u\cdot \d_t\vec u^{(1)}+ \d_tR^{(1)}\, \delta R\Big)\dxdt \label{rel_en_ineq}\\
&-\int_0^T\int_{\R^2}\Big((\vrho^{(2)}\vec u^{(2)}\otimes \delta \vec u) :\nabla \vec u^{(1)}+\delta R \, \vec u^{(2)}\cdot \nabla R^{(1)}\Big)\dxdt \nonumber\, .
\end{align}
Next, we remark that we can write 
$$ (\vrho^{(2)}\vec u^{(2)}\otimes \delta \vec u) :\nabla \vec u^{(1)}=(\vrho^{(2)}\vec u^{(2)}\cdot \nabla \vec u^{(1)})\cdot \delta \vec u $$
and that we have $\vec u^{(2)}\cdot (\vec u^{(1)})^\perp=\delta \vec u\cdot (\vec u^{(1)})^\perp$ by orthogonality. 

Therefore, relation \eqref{rel_en_ineq} can be recasted as
\begin{equation*}\label{rel_en_ineq_1}
\begin{split}
\frac{1}{2}\int_{\R^2}\Big(\vrho^{(2)}(T)|\delta \vec u(T)|^2+|\delta R(T)|^2\Big) \dx&\leq \frac{1}{2}\int_{\R^2}\Big(\vrho_0^{(2)}|\delta \vec u_0|^2+|\delta R_0|^2\Big) \dx\\
&-\int_0^T\int_{\R^2}\vrho^{(2)}\Big(\d_t \vec u^{(1)}+\vec u^{(2)}\cdot \nabla \vec u^{(1)}+(\vec u^{(1)})^\perp\Big)\cdot \delta \vec u\, \dxdt \\
&-\int_0^T\int_{\R^2}\Big(\d_tR^{(1)}+\vec u^{(2)}\cdot \nabla R^{(1)}\Big)\, \delta R\, \dxdt\, .
\end{split}
\end{equation*}
At this point, we add and subtract the quantities $\pm (\vrho^{(2)}\vec u^{(1)}\cdot \nabla \vec u^{(1)})\cdot \, \delta \vec u\pm \vrho^{(2)}\frac{1}{\vrho^{(1)}}\nabla \Pi^{(1)}\cdot \, \delta \vec u$ and $\pm (\vec u^{(1)}\cdot \nabla R^{(1)}) \, \delta R$, yielding 
\begin{equation}\label{rel_en_ineq_2}
\begin{split}
\frac{1}{2}\int_{\R^2}\Big(\vrho^{(2)}(T)|\delta \vec u(T)|^2+|\delta R(T)|^2\Big) \dx&\leq \frac{1}{2}\int_{\R^2}\Big(\vrho_0^{(2)}|\delta \vec u_0|^2+|\delta R_0|^2\Big) \dx\\
&-\int_0^T\int_{\R^2}\Big(\vrho^{(2)}\delta \vec u\cdot \nabla \vec u^{(1)} +\delta R\frac{1}{\vrho^{(1)}}\nabla \Pi^{(1)}\Big)\cdot \delta \vec u \, \dxdt  \\
&-\int_0^T\int_{\R^2}(\delta \vec u\cdot \nabla R^{(1)})\, \delta R \dxdt\, ,
\end{split}
\end{equation}
where we have used the fact that $(R^{(1)},\, \vec u^{(1)})$ is solution to the Euler system and $\int_{\R^2}\nabla \Pi^{(1)}\cdot \delta \vec u \, \dx=0$. 

Therefore, setting $\mc E(T):=\|\sqrt{\vrho^{(2)}(T)}\delta \vec u(T)\|^2_{L^2}+\|\delta R(T)\|^2_{L^2}$, from relation \eqref{rel_en_ineq_2} we can deduce that
\begin{equation*}\label{rel_en_inq_3}
\mc E(T)\leq \mc E(0)+C\int_0^T\left(\|\nabla \vec u^{(1)}\|_{L^\infty}+\left\|\frac{1}{\sqrt{\vrho^{(2)}}\vrho^{(1)}}\nabla \Pi^{(1)}\right\|_{L^\infty}+\left\|\frac{1}{\sqrt{\vrho^{(2)}}}\nabla R^{(1)}\right\|_{L^\infty}\right)\mc E(t) \dt\, .
\end{equation*}
An application of Gr\"onwall lemma (we refer to Section \ref{app:sect_gron}) yields the desired stability inequality \eqref{stability_ineq}. 

In order to get the result, having the regularity stated in the theorem, we argue by density. 

Thanks to the regularity (stated in Definition \ref{weak_sol_L^2}) of weak solutions to the Euler equations \eqref{Euler_eps} and assumption $(i)$ of the theorem, all the terms appearing in relations \eqref{mom-eq-reg-1} and \eqref{mom-eq-reg-2} are well-defined, if in addition we have $\d_t \vec u^{(1)}\in L^1_T(L^2)$. However, this condition on the time derivative of the velocity field $\vec u^{(1)}$ comes for free from the momentum equation 
\begin{equation}\label{vec_1}
 \d_t \vec u^{(1)}=-\left(\vec u^{(1)}\cdot \nabla \vec u^{(1)}+(\vec u^{(1)})^\perp + \frac{1}{\vrho^{(1)}}\nabla \Pi^{(1)}\right)\, . 
 \end{equation}
Since $\vec u^{(1)}\in L^\infty_T(L^2)$ with $\nabla \vec u^{(1)}\in L^1_T(L^\infty)$ and $\vrho^{(1)}\in L^\infty_T(L^\infty)$, condition $(ii)$ implies that the right-hand side of \eqref{vec_1} is in $L^1_T(L^2)$. Recalling the regularity in Definition \ref{weak_sol_L^2} of $\vec u^{(1)}$, one gets $\vec u^{(1)}\in W^{1,1}_T(L^2)$ and hence $\vec u^{(1)}\in C^0_T(L^2)$. 

Analogously in order to justify computations in \eqref{transp-eq-reg-1} and \eqref{transp-eq-reg-2}, besides the previous regularity conditions, one needs the additional assumption $\d_t R^{(1)}\in L^1_T(L^2)$. Once again, one can take advantage of the continuity equation \eqref{ws:transport_R} to obtain the required regularity for $\d_t R^{(1)}$. Finally, condition $(ii)$ is necessary to make sense of relation \eqref{rel_en_ineq_2}.

This concludes the proof of the theorem.

\qed
\end{proof}

\section{Asymptotic analysis} \label{s:sing-pert_E}

The main goal of this section is to show the convergence when $\veps\rightarrow 0^+$: we achieve it employing a \textit{compensated compactness} technique. We point out that, in the sequel, the time $T>0$ is fixed by the existence theory developed in Section \ref{s:well-posedness_original_problem}.

We will show that \eqref{full Euler} converges towards a limit system, represented by the quasi-homogeneous incompressible Euler equations:
\begin{equation}\label{QH-Euler_syst}
\begin{cases}
\d_tR+\vu \cdot \nabla R=0\\
\d_t \vu +\vu \cdot \nabla \vu +R\vu^\perp +\nabla \Pi=0\\
\div \vu=0\, .
\end{cases}
\end{equation}
The previous system consists of a transport equation for the quantity $R$ (that can be interpreted as the deviation with respect to the constant density profile) and an Euler type equation for the limit velocity field $\vec u$.


In Section \ref{s:well-posedness_original_problem}, we have proved that the sequence $(\vrho_\veps, \ue, \nabla \Pi_\veps)_\veps$ is uniformly bounded (with respect to $\veps$) in suitable spaces. 
Next, thanks to the uniform bounds, we extract weak limit points, for which one has to find some constraints: the singular terms have to vanish at the limit (see Subsection \ref{sss:unif-prelim_E}).

 Finally, after performing the \textit{compensated compactness} arguments, we describe the limit dynamics (see Paragraph \ref{ss:wave_system} below).

The choice of using this technique derives from the fact that the oscillations in time of the solutions are out of control (see Subsection \ref{ss:wave_system}). To overcome this issue, rather than employing the standard $H^s$ estimates, we take advantage of the weak formulation of the problem. We test the equations against divergence-free test functions: this will lead to useful cancellations. In particular, we avoid to study the pressure term.
At the end, we close the argument by noticing that the weak limit solutions are actually regular solutions. 



\subsection{Preliminaries and constraint at the limit} \label{sss:unif-prelim_E}
We start this subsection by recalling the uniform bounds developed in Section \ref{s:well-posedness_original_problem}. The fluctuations $R_{\veps}$ satisfy the controls
	\begin{align*}
&\sup_{\veps\in\,]0,1]}\left\|  R_{\veps} \right\|_{L^\infty_T(L^\infty)}\,\leq \, C\\
&\sup_{\veps\in\,]0,1]}\left\| \nabla R_{\veps} \right\|_{L^\infty_T(H^{s-1})}\,\leq \, C \, ,
	\end{align*}
where $R_\veps:=(\vrho_\veps-1)/\veps$ as above.

As for the velocity fields, we have obtained the following uniform bound:
\begin{equation*}
\sup_{\veps\in\,]0,1]}\left\|  \vu_{\veps} \right\|_{L^\infty_T(H^s)}\,\leq \, C\, .
\end{equation*}
Thanks to the previous uniform estimates, we can assume (up to passing to subsequences) that there exist $R\in L^\infty_T( W^{1,\infty})$, with $\nabla R\in L^\infty_T(H^{s-1})$, and $\vec u \in L^\infty_T(H^s)$ such that 
\begin{equation}\label{limit_points}
\begin{split}
R:=\lim_{\veps \rightarrow 0^+}R_{\veps} \quad &\text{in}\quad L^\infty_T(L^\infty) \\
\nabla R:=\lim_{\veps \rightarrow 0^+}\nabla R_{\veps}\quad &\text{in}\quad L^\infty_T( H^{s-1})\\
\vu:=\lim_{\veps \rightarrow 0^+}\vu_{\veps}\quad &\text{in}\quad L^\infty_T(H^s) \, ,
\end{split}
\end{equation} 
where we agree that the previous limits are taken in the corresponding weak-$\ast$ topology.

\begin{remark}\label{rmk:conv_rho_1}
It is evident that $\vrho_\veps -1=O(\veps)$ in $L^\infty_T (L^{\infty})$ and therefore that $\vrho_\veps \ue$ weakly-$\ast$ converge to $\vec  u$ e.g. in the space $L^\infty_T( L^2)$.
\end{remark}

Next, we notice that the solutions stated in Theorem \ref{W-P_fullE} are \textit{strong} solutions. In particular, they satisfy in a weak sense the mass equation and the momentum equation, respectively:  
\begin{equation}\label{weak-con_E}
	-\int_0^T\int_{\R^2} \left( \vre \partial_t \varphi  + \vre\ue \cdot \nabla_x \varphi \right) \dxdt = 
	\int_{\R^2} \vrez \varphi(0,\cdot) \dx\, ,
	\end{equation}
for any $\varphi\in C^\infty_c([0,T[\,\times \R^2;\R)$;
	\begin{align}
	&\int_0^T\!\!\!\int_{\R^2}  
	\left( - \vre \ue \cdot \partial_t \vec\psi - \vre [\ue\otimes\ue]  : \nabla_x \vec\psi 
	+ \frac{1}{\ep} \,  \vre \ue^\perp \cdot \vec\psi  \right) \dxdt 
	= \int_{\R^2}\vrez \uez  \cdot \vec\psi (0,\cdot) \dx\, , \label{weak-mom_E} 
	\end{align}
for any test function $\vec\psi\in C^\infty_c([0,T[\,\times \R^2; \R^2)$ such that $\div \vec\psi=0$;
 
Moreover, the divergence-free condition on $\ue$ is satisfied in $\mc D^\prime (]0,T[\times \R^2)$.

Before going on, in the following lemma, we 
characterize the limit for the quantity $R_\veps \ue$. We recall that $R_\veps$ satisfies 
\begin{equation}\label{eq_transport_R_veps}
\d_t R_\veps =-\div (R_\veps \ue)\, , \quad \quad (R_\veps)_{|t=0}=R_{0,\veps}.
\end{equation}  
\begin{lemma}\label{l:reg_Ru_veps}
Let $(R_\veps)_\veps$ be uniformly bounded in $L^\infty_T(L^\infty(\R^2))$ with $(\nabla R_\veps)_\veps\subset L^\infty_T(H^{s-1}(\R^2))$, and let the velocity fields $(\ue)_\veps$ be uniformly bounded in $L^\infty(H^s(\R^2))$. Moreover, for any $\veps \in\; ]0,1]$, assume that the couple $(R_\veps, \vec u_\veps)$ solves the transport equation \eqref{eq_transport_R_veps}. Let $(R, \vec u)$ be the limit point identified in \eqref{limit_points}. Then, up to an extraction:
\begin{enumerate}
\item[(i)] $R_\veps \rightarrow R$ in $C^0_T(C^0_{\rm loc}(\R^2))$;
\item[(ii)] the product $R_\veps \ue$ converges to $R\vec u$ in the distributional sense.
\end{enumerate}
\end{lemma}

\begin{proof}
We look at the transport equation \eqref{eq_transport_R_veps} for $R_\veps$. 
We employ Proposition \ref{prop:app_fine_tame_est} on the term in the right-hand side, obtaining 
\begin{equation*}
\|R_\veps \ue\|_{H^s}\leq C\left(\|R_\veps\|_{L^\infty}\|\ue  \|_{H^s}+\|\nabla R_\veps\|_{H^{s-1}}\|\ue\|_{L^\infty}\right)\, .
\end{equation*}
By embeddings, this implies that the sequence $(\d_t R_\veps)_\veps$ is uniformly bounded e.g. in $L^\infty_T (L^\infty)$ and so $(R_\veps)_\veps$ is bounded in $W^{1,\infty}_T(L^\infty)$ uniformly in $\veps$. On the other hand, we know that $(\nabla R_{\veps})_\veps$ is bounded in $L^\infty_T (L^\infty)$. Then, by Ascoli-Arzel\`a theorem (see Theorem \ref{app:th_A-A}), we gather that the family $(R_\veps)_\veps$ is compact in e.g. $C^0_T (C_{\rm loc}^{0})$ and hence we deduce the strong convergence property, up to passing to a suitable subsequence (not relabelled here), 
\begin{equation*}
R_\veps \rightarrow R \quad \text{in} \quad C^0([0,T]\, ;C_{\rm loc}^{0})\, .
\end{equation*}
Finally, since $(\ue)_\veps$ is weakly-$\ast$ convergent e.g. in $L^\infty_T( L^2)$ to $\vec u$, we get $R_\veps \ue \weakstar R \vec  u$ in the space $L^\infty_T( L_{\rm loc}^2)$.
\qed
\end{proof}


Now, as anticipated in the introduction of this section, we have to highlight the constraint that the limit points have to satisfy. We have to point out that this condition does not fully characterize the limit dynamics (see Subsection \ref{ss:wave_system} below). 

The only singular term (of order $\veps^{-1}$) appearing in the equations is the Coriolis force. Then, we test the momentum equation in \eqref{weak-mom_E} 
against $\veps \vec \psi$ with $\vec \psi \in C_c^\infty([0,T[\, \times \R^2;\R^2)$ such that $\div \vec \psi =0$.
Keeping in mind the assumptions on the initial data and due to the fact that $(\vrho_\veps \ue )_\veps$ is uniformly bounded in e.g. $L^\infty_T(L^2) $ and so is $(\vrho_\veps \ue \otimes \ue)_\veps$ in $L^\infty_T(L^1)$, it follows that all the terms in equation \eqref{weak-mom_E}, apart from the Coriolis operator, go to $0$ in the limit for $\veps\rightarrow 0^+$.

Therefore, 
we infer that, for any $\vec \psi \in C_c^\infty([0,T[\, \times \R^2;\R^2)$ such that $\div \vec \psi =0$,
\begin{equation*}
\lim_{\veps \rightarrow 0^+}\int_0^T\int_{\R^2}\vrho_\veps \ue^\perp \cdot \vec \psi \dx \dt=\int_0^T \int_{\R^2}\vec u^\perp \cdot \vec \psi \dx \dt=0\, .
\end{equation*}
This property tells us that $\vec u^\perp =\nabla \pi$, for some suitable function $\pi$. 

However, this relation does \textit{not} add more information on the limit dynamics, since we already know that the divergence-free condition $\div \ue=0$ is satisfied for all $\veps>0$. 


\subsection{Wave system and convergence}\label{ss:wave_system}
The goal of the present subsection is to describe oscillations of solutions in order to show convergence to the limit system. The Coriolis term is responsible for strong oscillations in time of solutions, which may prevent the convergence. To overcome this issue we implement a strategy based on \textit{compensated compactness} arguments. Namely, we perform algebraic manipulations on the wave system (see \eqref{wave system} below), in order to derive 
compactness properties for the quantity $\wtilde\gamma_\veps:=\curl (\vrho_\veps \ue)$. This will be enough to pass to the limit in the momentum equation (and, in particular, in the convective term). 

Let us define 
\begin{equation*}
\vec{V}_\veps:=\vrho_\veps \ue 
\, ,
\end{equation*}
that is uniformly bounded in $L^\infty_T(H^s)$, due to Proposition \ref{prop:app_fine_tame_est}. 

Now, using the fact that $\vrho_\veps=1+\veps R_\veps$, we recast the continuity equation in the following way:
\begin{equation}\label{eq:wave_mass_E}
\veps\d_t R_\veps+\div \vec V_\veps=0\, .
\end{equation}
In light of the uniform bounds and convergence properties stated in Lemma \ref{l:reg_Ru_veps}, we can easily pass to the limit in the previous formulation (or rather in \eqref{eq_transport_R_veps}) finding  
\begin{equation}\label{eq:limit_transp_R}
\d_t R+\div(R \vec u)=0\, .
\end{equation}
We decompose
$$ \vrho_\veps \ue^\perp=\ue^\perp +\veps\, R_\veps \ue^\perp $$
and from the momentum equation one can deduce
\begin{equation}\label{eq:wave_mom}
\veps \d_t \vec V_\veps +\nabla \Pi_\veps +\ue^\perp= \veps f_\veps\, ,
\end{equation}
where we have defined 
\begin{equation}\label{def_f}
f_\veps:=-\div (\vrho_\veps \ue \otimes \ue)-R_\veps \ue^\perp\, .
\end{equation}

In this way, we can rewrite system \eqref{Euler_eps} in the wave form
\begin{equation}\label{wave system}
\begin{cases}
\veps\d_t R_\veps+\div \vec V_\veps=0\\
\veps \d_t \vec V_\veps +\nabla \Pi_\veps +\ue^\perp= \veps f_\veps\, .
\end{cases}
\end{equation}

Applying again Proposition \ref{prop:app_fine_tame_est}, one can show that the terms $\vrho_\veps \ue \otimes \ue$ and $R_\veps \ue^\perp$ are uniformly bounded in $L^\infty_T(H^s)$. Thus, it follows that $(f_\veps)_\veps \subset L^\infty_T(H^{s-1})$.


However, the uniform bounds in Section \ref{s:well-posedness_original_problem} are not enough for proving convergence in the weak formulation of the momentum equation. Indeed, on the one hand, those controls allow to pass to the limit in the $\d_t$ term and in the initial datum; on the other hand, the non-linear term and the Coriolis force are out of control. We postpone the convergence analysis of the Coriolis force in the next Paragraph \ref{ss:limit_E} and now we focus on the 
the convective term $\div (\vrho_\veps \ue \otimes \ue)$ in \eqref{weak-mom_E}. We proceed as follows: first of all, we reduce our study to the constant density case (see Lemma \ref{l:approx_convective_term} below); next, we apply the \textit{compensated compactness} argument.

\begin{lemma}\label{l:approx_convective_term}
Let $T>0$. For any test function $\vec \psi \in C_c^\infty([0,T[\times \R^2;\R^2)$, we get 
\begin{equation}\label{eq_approx_convective_term}
\limsup_{\veps\rightarrow 0^+}\left|\int_0^T \int_{\R^2}\vrho_\veps \ue\otimes \ue :\nabla \vec \psi \, \dxdt- \int_0^T\int_{\R^2} \ue \otimes \ue :\nabla \vec \psi \, \dxdt\right|=0\, .
\end{equation}
\end{lemma}
\begin{proof}
Let $\psi \in C_c^\infty([0,T[\times \R^2;\R^2)$ with $\Supp \vec \psi \subset [0,T]\times K$ for some compact set $K\subset \R^2$. Therefore, we can write
\begin{equation*}
\int_0^T \int_{K}\vrho_\veps \ue\otimes \ue :\nabla \vec \psi \, \dxdt= \int_0^T\int_{K} \ue \otimes \ue :\nabla \vec \psi \, \dxdt+\veps \int_0^T\int_{K} R_\veps \ue \otimes \ue :\nabla \vec \psi \, \dxdt\, .
\end{equation*}
As a consequence of the uniform bounds e.g. $(\ue)_\veps \subset L^\infty_T(H^s)$ and $(R_\veps)_\veps\subset L_T^\infty(L^\infty)$, the second integral in the right-hand side is of order $\veps$.
\qed
\end{proof}

Thanks to Lemma \ref{l:approx_convective_term}, we are reduced to study the convergence (with respect to $\veps$) of the integral
\begin{equation*}
-\int_0^T\int_{\R^2} \ue \otimes \ue :\nabla \vec \psi \, \dxdt=\int_0^T\int_{\R^2} \div(\ue \otimes \ue) \cdot \vec \psi \, \dxdt\, .
\end{equation*}
Owing to the divergence-free condition we can write:
\begin{equation}\label{eq:conv_term_rel}
\div (\ue \otimes \ue )=\ue \cdot \nabla \ue =\frac{1}{2}\nabla |\ue|^2+\omega_\veps \, \ue^\perp\, ,
\end{equation}
where we have denoted $\omega_\veps:=\curl \ue=-\d_2u^1+\d_1u^2$.

Notice that the former term, since it is a perfect gradient, vanishes identically when tested against $\vec \psi$ such that $\div \vec \psi=0$. 
As for the latter term we take advantage of equation \eqref{eq:wave_mom}. Taking the $\curl$, we get
\begin{equation}\label{relation_gamma}
\d_t  \wtilde\gamma_\veps=\curl f_\veps \, ,
\end{equation}
where we have set $\wtilde\gamma_\veps:= \curl \vec V_\veps$ with $\vec V_\veps:= \vrho_\veps \ue$. We recall also that $f_\veps$ defined in \eqref{def_f} is uniformly bounded in the space $L^\infty_T(H^{s-1})$.
Then, relation \eqref{relation_gamma} implies that the family $(\d_t \wtilde\gamma_\veps)_\veps$ is uniformly bounded in $L^\infty_T(H^{s-2})$. As a result, we get $(\wtilde\gamma_\veps)_\veps \subset W^{1,\infty}_T(H^{s-2})$. On the other hand, the sequence $(\nabla \wtilde\gamma_\veps) _\veps$ is also uniformly bounded in $L^\infty_T(H^{s-2})$.
At this point, the Ascoli-Arzel\`a theorem (we refer to Theorem \ref{app:th_A-A} in this regard) gives compactness of $(\wtilde\gamma_\veps)_\veps$ in e.g. $C^0_T(H^{s-2}_{\rm loc})$. Then, it converges (up to extracting a subsequence) to a tempered distribution $\wtilde\gamma$ in the same space. Thus, it follows that 
\begin{equation*}
\wtilde\gamma_\veps \longrightarrow \wtilde\gamma \quad \text{in}\quad C^0_T(H_{\rm loc}^{s-2})\, .
\end{equation*}
But since we already know the convergence $\vec{V}_\veps:=\vrho_\veps \ue \weakstar \vec u$ in e.g. $L^\infty_T(L^2)$, it follows that $\wtilde\gamma_\veps:= \curl \vec V_\veps \longrightarrow \omega:=\curl \vec u$ in $\mc D^\prime$, hence $\wtilde\gamma =\curl \vec u=\omega$.

Finally, writing $\vrho_\veps=1+\veps R_\veps$, we obtain 
$$ \wtilde\gamma_\veps:=\curl (\vrho_\veps \ue)=\omega_\veps+\veps \curl (R_\veps \ue)\, , $$
where the family $(\curl (R_\veps \ue))_\veps$ is uniformly bounded in $L^\infty_T(H^{s-1})$. From this relation and the previous analysis, we deduce the strong convergence (up to an extraction) for $\veps\rightarrow 0^+$:
\begin{equation*}
\omega_\veps\longrightarrow \omega \qquad \text{ in }\qquad L^\infty_T(H^{s-2}_{\rm loc})\,. 
\end{equation*}


In the end, we have proved the following convergence result for the convective term $\div (\vec u_\veps \otimes \ue)$. 
\begin{lemma}\label{lemma:limit_convective_term}
Let $T>0$. Up to passing to a suitable subsequence, one has the following convergence for $\veps\rightarrow 0^+$:
\begin{equation*}\label{limit_convective_term}
\int_0^T\int_{\R^2} \ue \otimes \ue :\nabla \vec \psi \, \dxdt\longrightarrow \int_0^T\int_{\R^2} \omega\, \vu^\perp \cdot \vec \psi\, \dxdt\, ,
\end{equation*}
for any test function $\vec \psi \in C_c^\infty([0,T[\times \R^2;\R^2)$ such that $\div \vec \psi=0$.
\end{lemma}
As a consequence of the previous lemma, performing equalities \eqref{eq:conv_term_rel} backwards, for the convective term $\vrho_\veps \ue \otimes \ue$ we find that 
\begin{equation}\label{lim_conv-term}
\int_0^T \int_{\R^2}\vrho_\veps \ue\otimes \ue :\nabla \vec \psi \, \dxdt \longrightarrow \int_0^T\int_{\R^2} \vu \otimes \vu :\nabla \vec \psi \, \dxdt
\end{equation}
for $\veps \rightarrow 0^+$ and for all smooth divergence-free test functions $\vec \psi$.

\subsection{Description of the limit system} \label{ss:limit_E} 
With the convergence established in Paragraph \ref{ss:wave_system}, we can pass to the limit in the momentum equation.

To begin with, we take a test-function $\vec\psi$ such that 
\begin{equation}\label{def:test_function}
\vec \psi =\nabla^\perp \varphi\quad \quad  \text{with}\quad \quad
 \varphi \in C_c^\infty ([0,T[\times \R^2;\R)\, .
\end{equation}

For such a $\vec\psi$, all the gradient terms vanish identically. First of all, we recall the momentum equation in its weak formulation:
\begin{equation}\label{weak_mom_limit}
\int_0^T\!\!\!\int_{\R^2}  
	\left( - \vre \ue \cdot \partial_t \vec\psi - \vre [\ue\otimes\ue]  : \nabla_x \vec\psi 
	+ \frac{1}{\ep} \,  \vre \ue^\perp \cdot \vec\psi  \right) \dxdt 
	= \int_{\R^2}\vrez \uez  \cdot \vec\psi (0,\cdot) \dx\, .
\end{equation}
Making use of the uniform bounds, we can pass to the limit in the $\d_t$ term and thanks to our assumptions and embeddings we have $\vrho_{0,\veps}\vu_{0,\veps}\weak \vu_0$ in e.g. $L_{\rm loc}^{2}$.


Let us consider now the Coriolis term. We can write:
\begin{equation*}
\int_0^T\int_{\R^2}\frac{1}{\veps}\vrho_\veps \ue^\perp \cdot \vec \psi\, \dxdt=\int_0^T\int_{\R^2}R_\veps \ue^\perp \cdot \vec \psi\, \dxdt+\int_0^T\int_{\R^2}\frac{1}{\veps} \ue^\perp \cdot \vec \psi\, \dxdt\, .
\end{equation*}
Since $\ue$ is divergence-free, the latter term vanishes when tested against such $\vec \psi$ defined as in \eqref{def:test_function}. On the other hand, again thanks to Lemma \ref{l:reg_Ru_veps}, one can get 
 \begin{equation*}
 \int_0^T\int_{\R^2}R_\veps \ue^\perp \cdot \vec \psi\, \dxdt\longrightarrow \int_0^T\int_{\R^2}R \vu^\perp \cdot \vec \psi\, \dxdt\, .
 \end{equation*}
In the end, letting $\veps \rightarrow 0^+$ in \eqref{weak_mom_limit}, we gather (remembering also \eqref{lim_conv-term})   
\begin{equation*}
\int_0^T\!\!\!\int_{\R^2}  
	\left( -  \vu \cdot \partial_t \vec\psi -  \vu \otimes\vu  : \nabla_x \vec\psi 
	+  \,  R \vu^\perp \cdot \vec\psi  \right) \dxdt 
	= \int_{\R^2} \vu_0  \cdot \vec\psi (0,\cdot) \dx\, ,
\end{equation*}
for any test function $\vec \psi$ defined as in \eqref{def:test_function}. 

From this relation, we immediately obtain that 
\begin{equation*}
\d_t \vu +\div (\vu \otimes \vu) +R\vu^\perp +\nabla \Pi=0\, ,
\end{equation*}
for a suitable pressure term $\nabla \Pi$. This term appears as a result of the weak formulation of the problem. It can be viewed as a Lagrangian multiplier associated to the to the divergence-free constraint on $\vu$. Finally, the quantity $R$ satisfies the transport equation found in \eqref{eq:limit_transp_R}. 

We conclude this paragraph, devoting our attention to the analysis of the regularity of $\nabla \Pi$. We apply the $\div$ operator to the momentum equation in \eqref{QH-Euler_syst}, deducing that $\Pi$  satisfies 
\begin{equation}\label{time_reg_pressure}
-\Delta \Pi= \div G \qquad \text{where}\qquad G:= \vec u \cdot \nabla \vec u+R\vec u^\perp\, .
\end{equation}
On the one hand, 
Lemma \ref{lem:Lax-Milgram_type} gives 
\begin{equation*}
\|\nabla \Pi\|_{L^2}\leq C\|G\|_{L^2}\leq C\left(\|\vec u\|_{L^2}\|\nabla \vec u\|_{L^\infty}+\|R\|_{L^\infty}\|\vec u\|_{L^2}\right)\, .
\end{equation*}
This implies that $\nabla \Pi\in L^\infty_T(L^2)$.

On the other hand, owing to the divergence-free condition on $\vec u$, we have 
\begin{equation*}
\|\Delta \Pi\|_{H^{s-1}}\leq C\left(\|\vec u\|^2_{H^s}+\|R\|_{L^\infty}\|u\|_{H^s}+\|\nabla R\|_{H^{s-1}}\|\vec u\|_{L^\infty}\right)\, ,
\end{equation*}
where we have also used Proposition \ref{prop:app_fine_tame_est}.

In the end, we deduce that $\Delta \Pi \in L^\infty_T(H^{s-1})$. Thus, we conclude that $\nabla \Pi\in L^\infty_T(H^s)$.

At this point, employing classical results on solutions to transport equations in Sobolev spaces, we may infer the claimed $C^0$ time regularity of $\vec u$ and $R$. Moreover, thanks to the fact that $R$ and $\vec u$ are both continuous in time, from the elliptic equation \eqref{time_reg_pressure}, we get that also $\nabla \Pi \in C^0_T(H^s)$.

\section{Well-posedness for the quasi-homogeneous system}\label{s:well-posedness_Q-H}
In this section, for the reader's convenience, we review the well-posedness theory of the quasi-homogeneous Euler system \eqref{system_Q-H_thm}, in particular, the ``asymptotically global'' well-posedness result presented in \cite{C-F_sub}. In the first Subsection \ref{sub:local_well-pos_H^S}, we sketch the local well-posedness theorem for system \eqref{system_Q-H_thm} in the $H^s$ framework. Actually, equations \eqref{system_Q-H_thm} are locally well-posedness in all $B^s_{p,r}$ Besov spaces, under the condition \eqref{Lip_cond}. We refer to \cite{C-F_RWA} where the authors apply the standard Littlewood-Paley machinery to the quasi-homogeneous ideal MHD system to recover local in time well-posedness in spaces $B^s_{p,r}$ for any $1<p<+\infty$. The case $p=+\infty$ was reached in \cite{C-F_sub} with a different approach based on the vorticity formulation of the momentum equation (see also Subsection \ref{ss:W-P_Besov} for more details concerning the ``critical'' case $p=+\infty$).

In Subsection \ref{ss:improved_lifespan}, we explicitly derive the lower bound for the lifespan of solutions to system \eqref{system_Q-H_thm}. The reason in detailing the derivation of \eqref{T-ast:improved} for $T^\ast$ is due to the fact that it is much simpler than the one presented in \cite{C-F_sub}, where (due to the presence of the magnetic field) the lifespan behaves like the fifth iterated logarithm of the norms of the initial oscillation $R_0$ and the initial magnetic field. In addition, that lower bound (see \eqref{T-ast:improved} below) improves the standard lower bound coming from the hyperbolic theory, where the lifespan is bounded from below by the inverse of the norm of the initial data. 


\subsection{Local well-posedness in $H^s$ spaces}\label{sub:local_well-pos_H^S}
In this subsection, we state the well-posedness result for system \eqref{system_Q-H_thm} in the $H^s$ functional framework with $s>2$, in which we have analysed the well-posedness issue for system \eqref{Euler_eps}.  We limit ourselves to present the proof, by energy methods, of uniqueness of solutions (see Subsection \ref{s:uniq_QH-E}) and the implications of the continuation criterion (see Subsection \ref{ss:subsect_cont_cri}): in order to show that, we need some preparatory material, stated in Subsection \ref{s:existence_QH-E}.
\begin{theorem}\label{thm:well-posedness_Q-H-Euler_bis}
Take $s>2$. Let $\big(R_0,u_0 \big)$ be initial data such that $R_0\in L^{\infty}$, with $\nabla R_0\in H^{s-1}$, and
the divergence-free vector field $\vu_0 \,\in  H^s$.

Then, there exists a time $T^\ast > 0$ such that, on $[0,T^\ast]\times\R^2$, problem \eqref{system_Q-H_thm} has a unique solution $(R,\vu, \nabla \Pi)$ with:
\begin{itemize}
 \item $R\in C^0\big([0,T^\ast]\times \R^2\big)$ and $\nabla R\in C^0_{T^\ast}(H^{s-1}(\R^2))$;
 \item $\vu$ and $\nabla \Pi$ belong to $C^0_{T^\ast}( H^{s}(\R^2))$.
 \end{itemize}
Moreover, if $(R, \vec u, \nabla \Pi)$ is a solution to \eqref{system_Q-H_thm} on $[0,T^\ast[\, \times \R^2$ ($T^\ast<+\infty$) with the properties described above, and  
\begin{equation*}
\int_0^{T^\ast}  \big\| \nabla \vu(t) \big\|_{L^\infty}   \dt < +\infty\,, 
\end{equation*}
then the triplet $(R, \vu, \nabla \Pi)$ can be continued beyond $T^\ast$ into a solution of system \eqref{system_Q-H_thm} with the same regularity.
\end{theorem}

\subsubsection{Uniqueness by an energy argument}\label{s:uniq_QH-E}

Uniqueness in our functional framework is a consequence of the following stability result, whose proof is based on an energy method for the difference of two solutions to the quasi-homogeneous Euler system \eqref{system_Q-H_thm}.
We present here the classical proof with the $C^1_T$ regularity assumption on the time variable (see condition $(i)$ in the theorem below). In order to relax this requirement, one has to argue as done in Theorem \ref{th:uniq_bis}, with the additional $L^2$ integrability condition on densities. 

\begin{theorem}\label{thm:stability_criterion}
Let $(R_1, \vec u_1)$ and $(R_2, \vec u_2)$ be two solutions to the quasi-homogeneous Euler system \eqref{system_Q-H_thm}. Assume that, for some $T>0$, one has the following properties:
\begin{enumerate}[(i)]
\item the two quantities $\delta R\,:=\,R_1-R_2$ and $\delta \vec u\,:=\,\vec u_1-\vec u_2$ belong to the space $C^1\big([0,T];L^2(\R^2)\big)$;
\item $\vec u_1 \in L^1\big([0,T];W^{1, \infty}(\R^2)\big)$ and $\nabla R_1 \in L^1\big([0,T];L^\infty(\R^2)\big)$.
\end{enumerate}

Then, for all $t\in[0,T]$, we have the stability inequality: 
\begin{equation}\label{stab_est_QH-E}
\|\delta R(t)\|_{L^2}^2+\|\delta \vu(t)\|_{L^2}^2\leq C \left(\|\delta R_0\|_{L^2}^2+\|\delta \vu_0\|_{L^2}^2\right) \, e^{CB(t)}\, ,
\end{equation}
for a universal constant $C>0$, where we have defined
\begin{equation}\label{Def_A}
B(t):= \displaystyle \int_0^t\left(\|\nabla R_1(\tau )\|_{L^\infty}+\| \vu_1(\tau )\|_{W^{1,\infty}}\right)\, \detau\, .
\end{equation}
\end{theorem}
\begin{proof}
First of all, we take the difference of the two systems \eqref{system_Q-H_thm} solved by the triplets $(R_1,\vu_1, \nabla \Pi_1)$ and $(R_2,\vu_2, \nabla \Pi_2)$, obtaining
\begin{equation}\label{system_diff_QH-E}
\begin{cases}
\d_t \delta R+\vu_2 \cdot \nabla \delta R=-\delta \vu \cdot \nabla R_1\\
\d_t \delta \vu+\vu_2 \cdot \nabla \delta \vu + R_2\, \delta \vu^\perp+\nabla \delta \Pi =-\delta \vu \cdot \nabla \vu_1 -\delta R\, \vu_1^\perp\\
\div \delta \vec u=0\, ,
\end{cases}
\end{equation}
where $\delta \Pi := \Pi_1-\Pi_2$.

We start by testing the first equation of \eqref{system_diff_QH-E} against $\delta R$ and we get 
\begin{equation*}
\frac{1}{2}\frac{d}{dt}\|\delta R\|_{L^2}^2=-\int_{\R^2}(\delta \vu \cdot \nabla R_1)\, \delta R\leq \frac{1}{2}\|\nabla R_1\|_{L^\infty}\left(\|\delta R\|_{L^2}^2+\|\delta \vu\|_{L^2}^2\right)\, .
\end{equation*}
Next, testing the second equation on $\delta \vu$, due to the divergence-free conditions $\div \vu_1=\div \vu_2=0$, we gather 
\begin{equation*}
\frac{1}{2}\frac{d}{dt}\|\delta \vu\|_{L^2}^2=-\int_{\R^2}(\delta \vu \cdot \nabla \vu_1) \cdot\, \delta \vu-\int_{\R^2}(\delta R\; \vu_1^\perp)\cdot \delta \vu\leq \|\nabla \vu_1\|_{L^\infty}\|\delta \vu\|_{L^2}^2+\frac{1}{2}\| \vu_1\|_{L^\infty}\left(\|\delta R\|_{L^2}^2+\|\delta \vu\|_{L^2}^2\right) .
\end{equation*}
Putting the previous inequalities together, we finally infer 
\begin{equation*}
\frac{1}{2}\frac{d}{dt}\left(\|\delta R\|_{L^2}^2+\|\delta \vu\|_{L^2}^2\right)\leq C \left(\|\nabla R_1\|_{L^\infty}+\| \vu_1\|_{W^{1,\infty}}\right)\left(\|\delta R\|_{L^2}^2+\|\delta \vu\|_{L^2}^2\right)\, .
\end{equation*}
An application of Gr\"onwall's lemma gives us the stability estimate \eqref{stab_est_QH-E}, i.e. 
\begin{equation*}
\|\delta R(t)\|_{L^2}^2+\|\delta \vu(t)\|_{L^2}^2\leq C \left(\|\delta R_0\|_{L^2}^2+\|\delta \vu_0\|_{L^2}^2\right) \, e^{CB(t)}\, ,
\end{equation*}
for a universal constant $C>0$ and $B(t)$ defined as in \eqref{Def_A}.
\qed
\end{proof}

At this point, the uniqueness in the claimed framework (see Theorem \ref{thm:well-posedness_Q-H-Euler_bis}) follows from the previous statement. Let us sketch the proof.

We take an initial datum $(R_0, \vec u_0)$ satisfying the assumptions in Theorem \ref{thm:well-posedness_Q-H-Euler_bis}. We consider two solutions $(R_1, \vec u_1)$ and $(R_2, \vec u_2)$ of system \eqref{system_Q-H_thm}, related to the initial datum $(R_0, \vec u_0)$. Moreover, those solutions have to fulfill the regularity properties stated in the quoted theorem. 

Now, due to embeddings, we have only to detail how the previous solutions match the condition $(i)$ in Theorem \ref{thm:stability_criterion}. We focus on the regularity of $\delta R$, since similar arguments apply to $\delta \vec u$. 

We look at the first equation in \eqref{system_diff_QH-E}: $\delta R$ is transported by the divergence-free vector field $\vec u_2$, with in addition the presence of an ``external force'' $g:=-\delta \vec u \cdot \nabla R_1$. Thanks to the regularity properties presented in Theorem \ref{thm:well-posedness_Q-H-Euler_bis} and embeddings, we know that $\delta\vec u\in C^0_T(L^2)$ and $R_1\in C^0_T(W^{1,\infty})$. Thus, one can deduce that $g\in C^0_T(L^2)$. Therefore, from classical results for transport equations, we get that $\delta R\in C^1_T(L^2)$, as claimed. 

In the end, recalling that at the initial time $(\delta R, \delta \vec u)_{|t=0}=0$, we can apply Theorem \ref{thm:stability_criterion} to infer that $\|(\delta R, \delta \vec u)\|_{L^\infty_T(L^2)}=0$. This implies the desired uniqueness.

\subsubsection{A priori estimates}\label{s:existence_QH-E}

We start by bounding the $L^p$ norms of the solutions. First, since $R$ is transported by $\vu$ we have, for any $t\geq 0$,
\begin{equation*}
\|R(t)\|_{L^\infty}=\|R_0\|_{L^\infty}\, .
\end{equation*}
In addition, an energy estimate for the momentum equation in \eqref{system_Q-H_thm} yields 
\begin{equation}\label{eq:L^2_velocity}
\|\vu(t)\|_{L^2}\leq \|\vu_0\|_{L^2}\, .
\end{equation}
Making use of the dyadic blocks $\Delta_j$, for $i=1,2$ we find 
\begin{equation}\label{QH-Euler_vor_dyadic}
\begin{cases}
\d_t\Delta_j\,  \d_iR+\vu \cdot \nabla \Delta_j\,  \d_i R=[\vu\cdot \nabla,\Delta_j]\, \d_i R-\Delta_j(\d_i \vec u \cdot \nabla R)\\
\d_t \Delta_j \vec u +\vu \cdot \nabla \Delta_j \vu + \Delta_j \nabla \Pi=[\vu\cdot \nabla, \Delta_j]\vu-\Delta_j (R\vu^\perp)\, .
\end{cases}
\end{equation}

Following the same lines performed in Subsection \ref{ss:unif_est}, we can write 
\begin{equation*}
\begin{split}
2^{j(s-1)}\|\Delta_j\nabla R(t)\|_{L^2}+2^{js}\|\Delta_j \vu(t)\|_{L^2}&\leq C\left(2^{j(s-1)}\|\Delta_j \nabla R_0\|_{L^2}+2^{js}\|\Delta_j \vu_0\|_{L^2}\right)\\
&+C\int_0^t c_j\, (\tau)\|\vu(\tau) \|_{H^s}\|R(\tau)\|_{L^\infty}\, \detau \\
&+C\int_0^t c_j(\tau)\left(\|\vu(\tau) \|_{H^s}\|\nabla R(\tau)\|_{H^{s-1}}+\|\vu(\tau) \|_{H^s}^2\right) \, \detau \, ,
\end{split}
\end{equation*}
for suitable sequences $(c_j(t))_{j\geq -1}$ belonging to the unit sphere of $\ell^2$.

Now, we define for all $t\geq0$:
\begin{equation}\label{def_E(t)}
\wtilde E(t):=\|R(t)\|_{L^\infty}+\|\nabla R(t)\|_{H^{s-1}}+\|\vu(t)\|_{H^s}\, .
\end{equation}
Thanks to the previous bounds, employing Minkowski's inequality (see Section \ref{sec:assorted_ineq}), we gather 
\begin{equation*}
\wtilde E(t)\leq C\, \wtilde E(0)+C\int_0^t \wtilde E(\tau)^2 \, \detau\, .
\end{equation*}
At this point, the goal is to close the estimate, bounding the integral on the right-hand side in a small time. 

To this purpose, we define the time $T^\ast >0$ such that
\begin{equation}\label{def_T}
T^\ast :=\sup \left\{t>0:\int_0^t \wtilde E(\tau)^2 \, \detau \leq \wtilde E(0) \right\}\, .
\end{equation}
Then, we deduce $\wtilde E(t)\leq C\, \wtilde E(0)$ for all times $t\in [0,T^\ast ]$ and for some positive constant $C=C(s)$.

\subsubsection{The continuation criterion}\label{ss:subsect_cont_cri}
This subsection is devoted to the implications of the continuation result (Proposition \ref{th:cont-crit} below) for solutions to system \eqref{system_Q-H_thm}. The proof is omitted, since it is an easy adaptation of the more complex case we will present in Subsection \ref{subsec:cont_crit_besov}.
\begin{proposition} \label{th:cont-crit}
Let $T > 0$ and let $(R, \vu)$ be a solution to system \eqref{system_Q-H_thm} on $[0,T[\,\times\R^2$, enjoying the properties described in the previous Theorem \ref{thm:well-posedness_Q-H-Euler} for all $t<T$. Assume that 
\begin{equation}\label{cond_crit-cont}
\int_0^{T}  \big\| \nabla \vu(t) \big\|_{L^\infty}   \dt < +\infty\,.
\end{equation}
Then, 
$$ \sup_{0\leq t<T}\left(\|R\|_{L^\infty}+\|\nabla R\|_{H^{s-1}}+\|\vec u\|_{H^s}\right)<+\infty	\, . $$
\end{proposition}

As an immediate corollary we have that if $T<+\infty$, then the couple $(R, \vu)$ can be continued beyond $T$ into a solution of system \eqref{system_Q-H_thm} with the same regularity.

As a matter of fact, Proposition \ref{th:cont-crit} ensures that $\|R\|_{L^\infty_T(L^\infty)}$, $\|\nabla R\|_{L^\infty_T(H^{s-1})}$ and $\|\vec u\|_{L^\infty_T({H^s})}$ are finite. From the previous Subsection \ref{s:existence_QH-E}, we know that there exists a time $\oline \tau$ depending on $s$, $\|R\|_{L^\infty_T(L^\infty)}$, $\|\nabla R\|_{L^\infty_T(H^{s-1})}$, $\|\vec u\|_{L^\infty_T({H^s})}$ and on the norm of the data such that for all $\widetilde{T}<T$, the quasi-homogeneous system with data $\big(R(\widetilde{T}),\vec u(\widetilde{T})\big)$ has a unique solution until time $\oline \tau$. Now, taking $\widetilde{T}=T-\oline \tau/2$, we get a continuation of $(R,\vec u)$ up to time $T+\oline \tau/2$.

\subsection{Well-posedness in Besov spaces} \label{ss:W-P_Besov}
The main goal of this subsection is to review the lifespan estimate presented in \cite{C-F_sub} (for the MHD system) in order to get \eqref{improved_low_bound} (we refer also to Subsection \ref{ss:improved_lifespan} for the details of the proof). To show that, one has to work in critical Besov spaces where one can take advantage of the improved estimates for linear transport equations \textit{à la} Hmidi-Keraani-Vishik. In order to ensure that the condition \eqref{Lip_cond} is satisfied, the lowest regularity space we can reach is $B^1_{\infty,1}$. In addition, since $\vec u\in B^1_{\infty,1}$, we have that the $B^0_{\infty,1}$ norm of the $\curl \vu$ can be bounded \textit{linearly} with respect to $\|\nabla \vu\|_{L^1_t(L^\infty)}$, instead of exponentially as in classical $B^s_{p,r}$ estimates (see Theorem \ref{thm:improved_est_transport}).

Finally, we construct a ``bridge'' between $H^s$ and $B^1_{\infty,1}$ Besov spaces establishing a continuation criterion, in the spirit of the one by Beale-Kato-Majda in \cite{B-K-M} (see Subsection \ref{subsec:cont_crit_besov}).

We start by proving the local well-posedness result for system \eqref{system_Q-H_thm} in $B^s_{\infty,r}$ and, in particular, in the end-point space $B^1_{\infty,1}$. In this regard, Subsection \ref{subsec:aprioriBesov} is devoted to the \emph{a priori} estimates, presenting also the standard lower bound (coming from the hyperbolic theory) for the lifespan of solutions. Next, we construct the smooth approximate solutions (in Subsection \ref{ss:approx_sol_QH-E}) showing the uniform bounds for those regular solutions in Subsection \ref{ss:unif_bounds_QH-E}, and sketching the convergence (in the regularisation parameter $n$) argument in Subsection \ref{subsect:conv_argument}. 

\begin{theorem}\label{thm:W-P_besov_spaces_p-infty}
Let $(s,r)\in \R\times [1,+\infty]$ such that $s>1$ or $s=r=1$. Let $(R_0, \vec u_0)$ be an initial datum such that $R_0\in B^s_{\infty, r}(\R^2)$ and the divergence-free vector field $\vec u_0\in L^2(\R^2)\cap B^s_{\infty,r}(\R^2)$. Then, there exists a time $T^\ast>0$ such that system \eqref{system_Q-H_thm} has a unique solution $(R, \vec u)$ with the following regularity properties, if $r<+\infty$:
\begin{itemize}
\item $R\in C^0([0,T^\ast];B^s_{\infty,r}(\R^2))\cap C^1([0,T^\ast];B^{s-1}_{\infty,r}(\R^2)) $;
\item $\vec u$ and $\nabla \Pi$ belong to $C^0([0,T^\ast];B^s_{\infty,r}(\R^2))\cap C^1([0,T^\ast];L^2(\R^2)\cap B^{s-1}_{\infty,r}(\R^2))$. 
\end{itemize}
 In the case when $r=+\infty$, we need to replace $C^0([0,T^\ast];B^s_{\infty,r}(\R^2))$ by the space $C_w^0([0,T^\ast];B^s_{\infty,r}(\R^2))$.
\end{theorem}
We highlight that the physically relevant $L^2$ condition on $\vec u$, in the previous theorem, is necessary to control the low frequency part of the solution, so as to reconstruct the velocity from its $\curl$ (see Lemma \ref{l:rel_curl} below).
\subsubsection{A priori estimate in $B^s_{\infty,r}$}\label{subsec:aprioriBesov}
To begin with, we prove a general relation between a function and its $\curl$ that will be useful in the sequel. 
\begin{lemma}\label{l:rel_curl}
Assume $f\in (L^2 \cap B^s_{\infty ,r})(\R^2)$ to be divergence-free. Denote by $\curl f:=-\d_2f^1+\d_1f^2$ its $\curl$ in $\R^2$.
Then, we have
\begin{equation}\label{eq:rel_curl}
\|f\|_{L^2\cap B^s_{\infty ,r}}\sim\|f\|_{L^2}+\|\curl f\|_{B^{s-1}_{\infty, r}}\, .
\end{equation}
\end{lemma}
\begin{proof}
Using the divergence-free condition $\div f=0$, we can write the \textit{Biot-Savart law}:
\begin{equation*}
f^1=(-\Delta )^{-1}\d_2\, \curl f \quad \quad \text{and}\quad \quad f^2=-(-\Delta )^{-1}\d_1\, \curl f\, .
\end{equation*}
From that, we deduce 
\begin{equation*}
\|f\|_{B^s_{\infty ,r}}\sim \left\|\Delta_{-1}(-\Delta )^{-1}\sum_{i=1}^2(-1)^i\d_i\, \curl f\right\|_{L^\infty}+\left\|  \mathbbm{1}_{\{\nu \geq 0\}}\, 2^{\nu s}\|\Delta_{\nu}(-\Delta )^{-1}\sum_{i=1}^2(-1)^i\d_i\, \curl f\|_{L^\infty}\right\|_{\ell^r}\, .
\end{equation*}
On the one hand, if $\nu \geq 0$ we know that $\Delta_{\nu}\curl f$ is spectrally supported in an annulus, on which the symbol of $(-\Delta )^{-1}\d_i $ is smooth. Hence by employing Bernstein inequalities of Lemma \ref{l:bern}, we gather 
\begin{equation*}
2^{\nu s}\|\Delta_{\nu}(-\Delta )^{-1}\sum_{i=1}^2\d_i\, \curl f\|_{L^\infty}\sim \, 2^{(s-1)\nu}\|\Delta_{\nu}\curl f\|_{L^\infty}\, .
\end{equation*}
On the other hand, using the fact that the symbol of $(-\Delta )^{-1}\nabla \curl$ is homogeneous of degree zero and bounded on the unit sphere, Bernstein inequalities yield 
\begin{equation*}
\|\Delta_{-1}(-\Delta )^{-1}\sum_{i=1}^2\d_i\, \curl f\|_{L^\infty}\leq C \|\Delta_{-1}(-\Delta )^{-1}\nabla \curl f\|_{L^\infty}\leq C\|f\|_{L^2}\, .
\end{equation*}
Therefore,
\begin{equation*}
\|f\|_{B^s_{\infty ,r}}\leq C\left( \|f\|_{L^2}+\|\curl f\|_{B_{\infty, r}^{s-1}}\right)\, .
\end{equation*}
This completes the proof of the lemma.
\qed 
\end{proof}

\medskip

In the sequel of this subsection, we will show \textit{a priori} estimates for smooth solutions in the relevant norms. 

We start by recalling that the $L^2$ norm of the velocity field is preserved. In other words, we have:
\begin{equation}\label{L2_velocity}
 \|\vec u(t)\|_{L^2}=\|\vec u_0\|_{L^2}\, .
\end{equation}

Thanks to Lemma \ref{l:rel_curl}, in order to bound $\vu$ in $B^s_{\infty ,r}$, it will be enough to focus on estimates for $\curl \vu$ in $B^{s-1}_{\infty ,r}$. Hence, we apply the $\curl$ operator to the second equation in system \eqref{system_Q-H_thm} to get 
\begin{equation}\label{QH-Euler_vorticity_Bes}
\begin{cases}
\d_tR+\vu \cdot \nabla R=0\\
\d_t \omega +\vu \cdot \nabla \omega=-\div (R\vu)\, ,
\end{cases}
\end{equation}
where we recall $\omega :=\curl \vu=-\d_2u^1+\d_1u^2$.

Now, since $R$ is transported by $\vu$ we have, for any $t\geq 0$,
\begin{equation*}
\|R(t)\|_{L^\infty}=\|R_0\|_{L^\infty}\leq \|R_0\|_{B^s_{\infty ,r}}\, .
\end{equation*}
At this point we apply the dyadic blocks $\Delta_j$ to the system \eqref{QH-Euler_vorticity_Bes} and we find 
\begin{equation}\label{QH-Euler_vor_dyadic_Bes}
\begin{cases}
\d_t\Delta_j R+\vu \cdot \nabla \Delta_j R=[\vu\cdot \nabla,\Delta_j]R\\
\d_t \Delta_j \omega +\vu \cdot \nabla \Delta_j \omega=[\vu\cdot \nabla, \Delta_j]\omega-\Delta_j\div (R\vu)\, .
\end{cases}
\end{equation}
For the term $\div (R\vec u)$, we have
\begin{equation}\label{est_force term_Bes}
\|\div (R\vu)\|_{B^{s-1}_{\infty,r}}\leq C\,  \| R\vu\|_{B^{s}_{\infty,r}}\leq C \left(\|R\|_{L^\infty}\|\vu \|_{B^s_{\infty ,r}}+\|\vu\|_{L^\infty}\|R\|_{B^s_{\infty ,r}}\right)\, .
\end{equation}

Next, employing the commutator estimates (see Lemma \ref{l:commutator_est}), we get 
\begin{equation}\label{eq:commutator_R_Bes}
\begin{split}
2^{js}\|[\vu\cdot \nabla, \Delta_j]R\|_{L^\infty}&\leq C\, c_j(t)\,\left( \|\nabla \vu \|_{L^\infty}\|R\|_{B^s_{\infty,r}}+ \|\nabla \vu \|_{B^{s-1}_{\infty,r}}\|\nabla R\|_{L^\infty}\right)\\
&\leq C\, c_j(t)\, \|\vu \|_{B^s_{\infty,r}}\|R\|_{B^s_{\infty,r}}
\end{split}
\end{equation}
and
\begin{equation}\label{eq:commutator_omega_Bes}
\begin{split}
2^{j(s-1)}\|[\vu\cdot \nabla, \Delta_j]\omega\|_{L^\infty}&\leq C\, c_j(t)\,\left( \|\nabla \vu \|_{L^\infty}\|\omega\|_{B^{s-1}_{\infty,r}}+ \|\nabla \vu \|_{B^{s-1}_{\infty,r}}\|\omega\|_{L^\infty}\right)\\
&\leq C\, c_j(t)\, \|\vu \|_{B^s_{\infty,r}}^2\, ,
\end{split}
\end{equation}
for suitable sequences $(c_j(t))_{j\geq -1}$ belonging to the unit sphere of $\ell^r$.

\begin{remark}
We point out that we need the second estimate in Lemma \ref{l:commutator_est} to deal with \eqref{eq:commutator_omega_Bes} in the cases $s<2$, and $s=2$ and $r\neq 1$. In those cases, the Besov space $B^{s-1}_{\infty,r}$ is not contained in the Lipschitz space $W^{1,\infty}$.   
\end{remark}

Summing up estimates \eqref{est_force term_Bes}, \eqref{eq:commutator_R_Bes} and \eqref{eq:commutator_omega_Bes}, one may derive 
\begin{equation}\label{Bound_E}
\begin{split}
2^{js}\|\Delta_jR(t)\|_{L^\infty}+2^{j(s-1)}\|\Delta_j \omega(t)\|_{L^\infty}&\leq C\left(2^{js}\|\Delta_jR_0\|_{L^\infty}+2^{j(s-1)}\|\Delta_j \omega_0\|_{L^\infty}\right)\\
&+C\int_0^t c_j(\tau)\left(\|\vu(\tau) \|_{B^s_{\infty,r}}\|R(\tau)\|_{B^s_{\infty,r}}+\|\vu(\tau) \|_{B^s_{\infty,r}}^2\right) \, \detau.
\end{split}
\end{equation}
At this point, we define for all $t\geq0$:
\begin{equation*}
E(t):=\|R(t)\|_{B^s_{\infty,r}}+\|\vu(t)\|_{L^2}+\|\omega (t)\|_{B^{s-1}_{\infty,r}}\, .
\end{equation*}
Thanks to the $L^2$ estimate \eqref{L2_velocity} and the bound \eqref{Bound_E}, employing the Minkowski inequality, one may infer that 
\begin{equation*}
E(t)\leq C\, E(0)+\int_0^t E(\tau)^2 \, \detau\, .
\end{equation*}
We define now $T^\ast >0$ such that
\begin{equation*}
T^\ast =\sup \left\{t>0:\int_0^tE(\tau)^2 \, \detau \leq E(0) \right\}\, .
\end{equation*}
Then, we deduce $E(t)\leq C\, E(0)$ for all times $t\in [0,T^\ast ]$ and for some positive constant $C=C(s,r,d)$. Therefore, for all $t\in [0,T^\ast]$, we gather
\begin{equation*}
\int_0^tE(\tau)^2 \, \detau \leq CtE(0)^2\, .
\end{equation*}
By using the definition of $T^\ast$ and Lemma \ref{l:rel_curl}, we finally argue that 
\begin{equation}\label{est:T-star}
T^\ast \geq \frac{C}{\|R_0\|_{B^s_{\infty,r}}+\|\vu_0\|_{L^2\cap B^s_{\infty,r}}}\, . 
\end{equation}
In other words, we have shown that one can close the estimates for a small time $T^\ast$, which is bounded from below by \eqref{est:T-star}.

\subsubsection{Construction of approximate solutions} \label{ss:approx_sol_QH-E}
Since the material in this subsection is standard and already presented in Subsection \ref{sec:construction_smooth_sol} for system \eqref{Euler_eps}, we will only sketch it. 
 
For any $n\in \N$, let 
\begin{equation*}
(R^n_0, \vu_0^n):=(S_n R_0,\, S_n \vu_0)\, ,
\end{equation*}
where $S_n$ is the low frequency cut-off operator as in \eqref{eq:S_j}. By the assumption $\vu_0\in L^2$, we have $\vu_0^n \in H^\infty$ and similarly $R_0^n \in C^\infty_b$. Moreover, one has 
\begin{equation}\label{conv-properties}
\begin{split}
R_0^n\rightarrow R_0 \quad &\text{in}\quad B^s_{\infty,r}\\
\vu_0^n\rightarrow \vu_0 \quad &\text{in}\quad L^2 \cap B^s_{\infty,r}\, .
\end{split}
\end{equation}
Now, we will define the sequence of approximate solutions. First of all, we take $(R^0,\vu^0)=(R^0_0,\vu^0_0)$. Then, for all $\sigma \in \R$ we get $R^0\in C^0(\R_+;B^\sigma_{\infty,r})$ and $\vu^0\in C^0(\R_+;H^\sigma)$ with $\div \vu^0=0$. Next, we assume that $(R^n,\vu^n)$ is given such that, for all $\sigma \in \R$, 
\begin{equation*}
R^n\in C^0(\R_+;B^{\sigma}_{\infty,r}),\quad \vu^n\in C^0(\R_+;H^\sigma)\quad \text{and}\quad \div \vu^n=0\, .
\end{equation*}
We start by defining $R^{n+1}$ as the unique solution to the linear transport equation 
\begin{equation}\label{approx_R_QH-E}
\begin{cases}
\d_tR^{n+1}+\vu^n \cdot \nabla R^{n+1}=0\\
R^{n+1}_{|t=0}=R_0^{n+1}\, ,
\end{cases}
\end{equation}
and we deduce that $R^{n+1}\in C^0(\R_+;B^\sigma_{\infty, r})$ for all $\sigma \in \R$.

Next, we solve the linear transport equation with the divergence-free condition: 
\begin{equation}\label{approx_u_QH-E}
\begin{cases}
\d_t\vu^{n+1}+\vu^n \cdot \nabla \vu^{n+1}+\nabla \Pi^{n+1}=-R^{n+1}\vu^{\perp ,n}\\
\div \vu ^{n+1}=0\\
\vu^{n+1}_{|t=0}=\vu_0^{n+1}\, .
\end{cases}
\end{equation} 
We point out that the right-hand side term belongs to $L^1_{\rm loc}(\R_+;H^\sigma)$ for any $\sigma \in \R$. At this point, one can solve the previous linear problem by energy methods (see Propositions 3.2 and 3.4 in \cite{D_AT}) to find an unique solution $\vu^{n+1}\in C^0(\R_+;H^\sigma)$.

\subsubsection{Uniform bounds}\label{ss:unif_bounds_QH-E} 
We show now uniform bounds for the sequence $(R^n,\vu^n)_{n\in \N}$ constructed in the previous Paragraph \ref{ss:approx_sol_QH-E}.

We argue by induction and prove that there exists a time $T^\ast>0$ such that, for all $n\in \N$ and $t\in [0,T^\ast]$, one has
\begin{align}
&\|R^n(t)\|_{L^\infty}\leq C\|R_0\|_{L^\infty}\label{eq:induc_R}\\
&\|R^n(t)\|_{B^s_{\infty,r}}+\|\vu^n(t)\|_{L^2 \cap B^s_{\infty,r}}\leq CK_0\, e^{CK_0t}\label{eq:induc_u}\, ,
\end{align}
where the constant $C>0$ does not depend on the data neither on the solutions, and where we have defined 
\begin{equation*}
K_0:=\|R^n_0\|_{B^s_{\infty,r}}+\|\vu^n_0\|_{L^2 \cap B^s_{\infty,r}}\, .
\end{equation*}
It is clear that the couple $(R^0,\vu^0)$ satisfies the previous bounds. Assume now that $(R^n,\vu^n)$ verifies \eqref{eq:induc_R} and \eqref{eq:induc_u} on some interval $[0,T^\ast]$. Then, we have to prove the same properties for the step $n+1$.

We start by bounding $R^{n+1}$. We deduce that, for any $t\geq 0$,
\begin{equation*}
\|R^{n+1}(t)\|_{L^\infty}=\|R^{n+1}_0\|_{L^\infty}\leq C\|R_0\|_{L^\infty}\leq C\|R_0\|_{B^s_{\infty ,r}}\, .
\end{equation*} 
Next, employing an energy estimate for the velocity field, one can get 
\begin{equation*}
\begin{split}
\|\vu^{n+1}(t)\|_{L^2}&\leq \|\vu^{n+1}_0\|_{L^2}+C\int_0^t\|R^{n+1}(\tau)\vu^{\perp, n}(\tau)\|_{L^2}\, \detau\\
&\leq C  \|\vu_0\|_{L^2}+C\|R_0\|_{B^s_{\infty ,r}}\int_0^t\|\vu^{n}(\tau)\|_{L^2}\, \detau\, .
\end{split}
\end{equation*}
At this point, to get uniform bounds for the Besov norms, we resort the vorticity formulation:
\begin{equation*}
\d_t \omega^{n+1}+\vu^n\cdot \nabla \omega^{n+1}=\mathcal{L}(\nabla \vu^n,\nabla \vu^{n+1})+\div (R^{n+1}\vu^n)\, ,
\end{equation*}
where
\begin{equation}\label{eq:def_L}
\mathcal{L}(\nabla \vu^n,\nabla \vu^{n+1})=\sum_{k=1}^2\d_2 u_k^n\, \d_k u_1^{n+1}-\d_1 u_k^n\, \d_k u_2^{n+1}\, .
\end{equation}
Since the bound for $\div (R^{n+1}\vu^n)$ is analogous to the one performed in \eqref{est_force term_Bes}, it remains to bound $\mathcal{L}(\nabla \vu^n,\nabla \vu^{n+1})$ in $B^{s-1}_{\infty,r}$.
\begin{lemma}\label{lem:L}
Let $(\vec{v},\vec{w})$ be a couple of divergence-free vector fields in $B^s_{\infty,r}$. Then, one has 
\begin{equation*}
\|\mathcal{L}(\nabla \vec{v},\nabla \vec{w})\|_{B^{s-1}_{\infty,r}}\leq C\left(\|\nabla \vec{v}\|_{L^\infty}\|\vec{w}\|_{B^s_{\infty, r}}+\|\nabla \vec{w}\|_{L^\infty}\|\vec{v}\|_{B^s_{\infty, r}}\right)\, .
\end{equation*} 
\end{lemma}
\begin{proof}
The estimate easily follows from Corollary \ref{cor:tame_est} if $s>1$. Then, we have to show the bound when $\nabla \vec{v}$ and $\nabla \vec{w}$ are in $B^0_{\infty ,1}$ which is not an algebra.

Due to the fact that $\vec v$ and $\vec w$ are divergence-free, we can write 
\begin{equation}\label{rel_L}
\mathcal{L}(\nabla \vec{v},\nabla \vec{w})=\sum_{k=1}^2 \d_k(w^1\, \d_2v^k)-\d_k(w^2\, \d_1v^k)\, .
\end{equation}
Now, making use of Bony decomposition (we refer to Section \ref{app_paradiff} for more details), we have 
\begin{equation*}
\mathcal{L}(\nabla \vec{v},\nabla \vec{w})=\mathcal{L}_{\mathcal{T}}(\nabla \vec v,\nabla \vec w)+\mathcal{L}_{\mathcal{R}}(\nabla \vec{v},\nabla \vec{w})\, ,
\end{equation*}
where 
\begin{equation*}
\mathcal{L}_{\mathcal{T}}(\nabla \vec{v},\nabla \vec{w}):=\sum_{k=1}^2\mathcal{T}_{\d_kw^1}(\d_2v^k)+\mathcal{T}_{\d_2v^k}(\d_k w^1)-\mathcal{T}_{\d_kw^2}(\d_1v^k)-\mathcal{T}_{\d_1v^k}(\d_k w^2)
\end{equation*}
and
\begin{equation*}
\mathcal{L}_{\mathcal{R}}(\nabla \vec{v},\nabla \vec{w})=\sum_{k=1}^2 \mathcal{R}(\d_kw^1,\, \d_2v^k)-\mathcal{R}(\d_kw^2,\, \d_1v^k)\, .
\end{equation*}
On the one hand, thanks to Proposition \ref{T-R}, we can estimate the paraproducts in the following way:
\begin{equation*}
\|\mathcal{T}_{\nabla \vec{v}}(\nabla \vec{w})\|_{B^0_{\infty,1}}+\|\mathcal{T}_{\nabla \vec{w}}(\nabla \vec{v})\|_{B^0_{\infty,1}}\leq C\left( \|\nabla \vec{v}\|_{L^\infty}\|\nabla \vec{w}\|_{B^0_{\infty,1}}+\|\nabla \vec{w}\|_{L^\infty}\|\nabla \vec{v}\|_{B^0_{\infty,1}}\right)\, .
\end{equation*}
On the other hand, due to relation \eqref{rel_L} we may write
\begin{equation*}
\mathcal{L}_{\mathcal{R}}(\nabla \vec{v},\nabla \vec{w})=\sum_{k=1}^2 \d_k\mathcal{R}(w^1,\, \d_2v^k)-\d_k\mathcal{R}(w^2,\, \d_1v^k)\, .
\end{equation*}
Now, again thanks to Proposition \ref{T-R} we have
\begin{equation*}
\|\d_k\mathcal{R}(w^2,\, \d_1v^k)\|_{B^0_{\infty,1}}\leq C\, \|\mathcal{R}(w^2,\, \d_1v^k)\|_{B^1_{\infty,1}}\leq C\|\nabla \vec w\|_{B^0_{\infty,\infty}}\|\vec v\|_{B^1_{\infty,1}}\leq C\|\nabla \vec w\|_{L^\infty}\|\vec v\|_{B^1_{\infty,1}}
\end{equation*}
since $L^\infty\hookrightarrow B^0_{\infty,\infty}$. Similar argumentations apply to $\|\d_k\mathcal{R}(w^1,\, \d_2v^k)\|_{B^0_{\infty,1}}$.

Then, one has
\begin{equation*}
\|\mathcal{L}(\nabla \vec{v},\nabla \vec{w})\|_{B^{0}_{\infty,r}}\leq C\left(\|\nabla \vec{v}\|_{L^\infty}\|\vec{w}\|_{B^1_{\infty, r}}+\|\nabla \vec{w}\|_{L^\infty}\|\vec{v}\|_{B^1_{\infty, r}}\right)\, .
\end{equation*}  
This concludes the proof in the case $s=1$.
\qed
\end{proof}

\medskip

Therefore, applying Lemma \ref{lem:L} with $\vec v=\vu^n$ and $\vec w=\vu^{n+1}$, one can get 
\begin{equation*}
\|\mathcal{L}(\nabla \vu^n,\nabla \vu^{n+1})\|_{B^{s-1}_{\infty,r}}\leq C\left(\|\nabla \vu^n\|_{L^\infty}\|\vu^{n+1}\|_{B^s_{\infty, r}}+\|\nabla \vu^{n+1}\|_{L^\infty}\|\vu^n\|_{B^s_{\infty, r}}\right)\, .
\end{equation*} 

Reached this point, one can exactly proceed as in the proof for the \textit{a priori} estimates, finding that 
\begin{equation*}
\begin{split}
2^{js}\|\Delta_jR^{n+1}(t)\|_{L^\infty}+2^{j(s-1)}\|\Delta_j \omega^{n+1}(t)\|_{L^\infty}&\leq C\left(2^{js}\|\Delta_jR^{n+1}_0\|_{L^\infty}+2^{j(s-1)}\|\Delta_j \omega^{n+1}_0\|_{L^\infty}\right)\\
&+C\int_0^t c_j(\tau)\left(\|\vu^{n+1} \|_{B^s_{\infty,r}}+\|R^{n+1}\|_{B^s_{\infty,r}}\right)\|\vu^n \|_{B^s_{\infty,r}} \, \detau,
\end{split}
\end{equation*}
where the sequence $(c_j(t))_{j\geq -1}$ belongs to the unit sphere of $\ell^r$. 

Now, we define for all $t\geq 0$:
\begin{equation*}
\oline E^{n+1}(t):=\|R^{n+1}(t)\|_{B^s_{\infty, r}}+ \|\vu^{n+1}(t)\|_{L^2\cap B^s_{\infty,r}}\, .
\end{equation*}
Thus, recalling Lemma 
\ref{l:rel_curl}, from the previous inequalities we obtain 
\begin{equation*}
\oline E^{n+1}(t)\leq C\, \oline E^{n+1}(0)+C\int_0^t \oline E^{n+1}(\tau)\|\vu^n(\tau)\|_{L^2\cap B^s_{\infty,r}}\, \detau.
\end{equation*}
An application of Gr\"onwall arguments and the fact that $\oline E^{n+1}(0)\leq CK_0$, give 
\begin{equation*}
\oline E^{n+1}(t)\leq CK_0\exp \left(C\int_0^t\|\vu^n(\tau)\|_{L^2\cap B^s_{\infty,r}}\detau \right),
\end{equation*}
where $K_0:=\|R^n_0\|_{B^s_{\infty,r}}+\|\vu^n_0\|_{L^2 \cap B^s_{\infty,r}}$.

Next, from the inductive assumption \eqref{eq:induc_u}, we get 
\begin{equation*}
\int_0^t\|\vu^n(\tau)\|_{L^2\cap B^s_{\infty,r}}\detau \leq e^{CK_0t}-1
\end{equation*}
and we notice that for $0\leq x\leq1$ one has $e^x-1\leq x+x^2\leq 2x$. So, if we choose $T^\ast>0$ such that $CK_0T^\ast\leq 1$, we have 
\begin{equation*}
\oline E^{n+1}(t)\leq CK_0\exp (e^{CK_0t}-1)\leq CK_0\exp (CK_0t)\quad \quad \text{for}\; \; t\in [0,T^\ast]\, .
\end{equation*}
In this way we have completed the proof of the uniform bounds.

\subsubsection{Convergence}\label{subsect:conv_argument}
We show now convergence of the sequence $(R^n,\vu^n)_{n\in \N}$ towards a solution $(R,\vu)$ of the original problem. 
The proof follows the arguments already performed in Subsection \ref{ss:conv_H^s}: we limit ourselves to highlight only the main steps.

We define
\begin{equation*}
\widetilde{R}^n:=R^n-R^n_0
\end{equation*}
which satisfies 
\begin{equation*}
\begin{cases}
\d_t \widetilde{R}^{n+1}=-\vu^n\cdot \nabla R^{n+1}\\
\widetilde{R}^{n+1}_{|t=0}=0\, .
\end{cases}
\end{equation*}
Thus, one can check that $(\widetilde{R}^n)_{n\in \N}$ is uniformly bounded in $C^0([0,T];L^2)$.

Now, we will prove that $(\widetilde{R}^n,\vu^n)_{n\in \N}$ is a Cauchy sequence in $C^0([0,T];L^2)$. For any couple $(n,l)\in \N^2$, we introduce
\begin{equation*}
\begin{split}
&\delta \widetilde{R}^{n,l}:=\widetilde{R}^{n+l}-\widetilde{R}^{n}\\
&\delta R^{n,l}:=R^{n+l}-R^n\\
&\delta \vu^{n,l}:=\vu^{n+l}-\vu^n\\
&\delta \Pi^{n,l}:=\Pi^{n+l}-\Pi^n\, ,
\end{split}
\end{equation*}
and we have that $\div \delta \vec u^{n,l}=0$ for any $(n,l)\in \N^2$. 

Taking the difference between the $(n+l)$-iterate and the $n$-iterate, we may find 
\begin{equation}\label{syst_approx_conv}
\begin{cases}
\d_t \delta \widetilde{R}^{n,l}+\vu^{n+l-1}\cdot \nabla \delta \widetilde{R}^{n,l}=-\delta \vu^{n-1,l}\cdot \nabla R^n+\vu^{n+l-1} \cdot \nabla \delta R_0^{n,l}\\
\d_t \delta \vu^{n,l} +\vu^{n+l-1}\cdot \nabla \delta \vu^{n,l}+\nabla \delta \Pi^{n,l}=-\delta \vu^{n-1,l}\cdot \nabla \vu^n -R^{n+l}(\delta \vu^{\perp})^{n-1,l}-\delta R^{n,l}\vu^{\perp ,n-1}\, ,
\end{cases}
\end{equation}
supplemented with the initial data $(\delta \widetilde{R}^{n,l} ,\delta \vu^{n,l})_{|t=0}=(0,\delta \vu^{n,l}_0)$.

An energy estimate for the first equation of \eqref{syst_approx_conv} yields 
\begin{equation*}
\|\delta \widetilde{R}^{n,l}(t)\|_{L^2}\leq C\int_0^t \|\delta \vu ^{n-1,l}\|_{L^2}\|\nabla R^n\|_{L^\infty}+\|\vu^{n+l-1}\|_{L^2}\|\nabla \delta R_0^{n,l}\|_{L^\infty}\,  \detau\, ,
\end{equation*}
and similarly from the second equation we obtain 
\begin{equation*}
\begin{split}
\|\delta \vu^{n,l}(t)\|_{L^2}&\leq C\|\delta \vu_0^{n,l}\|_{L^2}+C\int_0^t\left(\|\delta \vu^{n-1,l}\|_{L^2}\|\nabla \vu^n\|_{L^\infty}+\|\delta \vu^{n-1,l}\|_{L^2}\|R^{n+l}\|_{L^\infty}\right)\, \detau\\
&+C\int_0^t\left(\|\delta \widetilde{R}^{n,l}\|_{L^2}+\|\delta R^{n,l}_0\|_{L^\infty}\right)\|\vu^{n-1}\|_{L^2\cap L^\infty}\, \detau \, .
\end{split}
\end{equation*}

Employing the uniform bounds established in Paragraph \ref{ss:unif_bounds_QH-E} and the embeddings, we note that 
\begin{equation*}
\sup_{t\in [0,T^\ast]}\left(\|\nabla R^n(t)\|_{L^\infty}+\|\nabla \vu^n(t)\|_{L^\infty}+\| R^{n+l}(t)\|_{L^\infty}\right)+\int_0^{T^\ast}\|\vu^{n+l-1}\|_{L^2}+\|\vu^{n-1}\|_{L^2\cap L^\infty}\dt \leq C_{T^\ast},
\end{equation*}  
for a constant $C_{T\ast}$ which depends only on $T^\ast$ and on the initial data. 

Therefore, thanks to the Gr\"onwall lemma, we get 
\begin{equation*}
\|\delta \widetilde{R}^{n,l}(t)\|_{L^2}+\|\delta \vu^{n,l}(t)\|_{L^2}\leq C_{T^\ast}\left(\|\delta R^{n,l}_0\|_{W^{1,\infty}}+\|\delta \vu_0^{n,l}\|_{L^2}+\int_0^t\|\delta \vu^{n-1,l}(\tau)\|_{L^2}\detau \right) ,
\end{equation*}
for all $t\in [0,T^\ast]$.

As already done in Subsection \ref{ss:conv_H^s}, after setting
\begin{equation*}
F_0^n:=\sup_{l\geq 0}\left(\|\delta R^{n,l}_0\|_{W^{1,\infty}}+\|\delta \vu_0^{n,l}\|_{L^2}\right)\quad \quad \text{ and }\quad \quad G^n(t):=\sup_{l\geq 0}\sup_{[0,t]}\left(\|\delta \widetilde{R}^{n,l}\|_{L^2}+\|\delta \vec u^{n,l}\|_{L^2}\right)\, ,
\end{equation*}
we may infer that, for all $t\in [0,T^\ast]$, 
\begin{equation*}
G^n(t)\leq C_{T^\ast}\sum_{k=0}^{n-1}\frac{(C_{T^\ast}T^\ast)^k}{k!}F_0^{n-k}+\frac{(C_{T^\ast} T^\ast)^n}{n!}G^0(t)\, ,
\end{equation*}
and, bearing in mind \eqref{conv-properties},
we have 
\begin{equation*}
\lim_{n\rightarrow +\infty}F^n_0=0\, .
\end{equation*}
Hence,
\begin{equation*}
\lim_{n\rightarrow +\infty}\sup_{l\geq 0}\sup_{t\in [0,T^\ast]}\left(\|\delta \widetilde{R}^{n,l}(t)\|_{L^2}+\|\delta \vu^{n,l}(t)\|_{L^2}\right)=0\, .
\end{equation*}

This property implies that $(\widetilde{R}^n)_{n\in \N}$ and $(\vu^n)_{n\in \N}$ are Cauchy sequences in $C^0_{T^\ast}(L^2)$. Hence, converge to some function $\widetilde{R}$ and $\vu$ in the same space. 

Define $R:=\widetilde{R}+R_0$. We notice that, owing to the embedding $L^2\hookrightarrow B^{-1}_{\infty,2}$, and thanks to the uniform bounds and to an interpolation argument, the sequence $(\vu^n)_{n\in \N}$ strongly converges in any intermediate space $L^\infty_{T^\ast} (B^\sigma_{\infty,r})$ with $\sigma <s$ and in particular in $L^\infty([0,T^\ast]\times \R^2)$. Moreover, we have that $R^n=\widetilde{R}^n+R^n_0$ strongly converges to $R$ in $L^\infty_{T^\ast}(L^2_{\rm loc})$. This is enough to pass to the limit in the weak formulation of \eqref{approx_R_QH-E} and \eqref{approx_u_QH-E} finding that $(R,\vu)$ is a weak solution to the original problem for a suitable pressure term $\nabla \Pi$. The regularity for $(R,\vu)$ in $B^s_{\infty,r}$ follows by the uniform bounds and Fatou's property in Besov spaces. 

Moreover, an argument similar to the one performed in Subsection \ref{ss:conv_H^s} apply here to show the desired regularity for the pressure term, after noticing that
\begin{equation*}
\|\nabla \Pi\|_{L^2\cap B^s_{\infty,r}}\sim \|\nabla \Pi\|_{L^2}+\|\Delta \Pi\|_{B^{s-1}_{\infty,r}}\, .
\end{equation*}



Finally, employing classical results for transport equations in Besov spaces (remember Theorem \ref{thm_transport}), we can get the claimed time continuity of $R$ with values in $B^s_{\infty,r}$, of $\vu$ with values in $L^2\cap B^s_{\infty,r}$ and of $\nabla \Pi$ with values in $L^2\cap B^s_{\infty,r}$. In addition, the sought regularity properties for the time derivatives $\d_t R$ and $\d_t \vu$ follow from an analysis of system \eqref{system_Q-H_thm}.

\subsubsection{Continuation criterion in Besov spaces}\label{subsec:cont_crit_besov}
We conclude this section showing the following continuation criterion for solutions of system \eqref{system_Q-H_thm} in $B^s_{\infty,r}$, where the couple $(s,r)$ satisfies the Lipschitz condition \eqref{Lip_cond}.
\begin{proposition}\label{prop:cont_criterion_Bes}
Let $(R_0,\vu_0)\in B^s_{\infty,r}\times (L^2\cap B^s_{\infty,r})$ with $\div \vu_0=0$. Given a time $T>0$, let $(R,\vu)$ be a solution of \eqref{system_Q-H_thm} on $[0,T[$ that belongs to $L^\infty_t(B^s_{\infty,r})\times L^\infty_t(L^2\cap B^s_{\infty,r})$ for any $t\in [0,T[$. If we assume that  
\begin{equation}\label{cont_cond_Bes}
\int_0^T\|\nabla \vu \|_{L^\infty}\, \dt<+\infty\, ,
\end{equation} 
then $(R, \vu)$ can be continued beyond $T$ into a solution of \eqref{system_Q-H_thm} with the same regularity.


Moreover, the lifespan of a solution $(R,\vu)$ to system \eqref{system_Q-H_thm} does not depend on $(s,r)$ and, in particular, the lifespan of solutions in Theorem \ref{thm:well-posedness_Q-H-Euler} is the same as the lifespan in
$B^1_{\infty,1}\times \left( L^2 \cap B^1_{\infty,1}\right)$.
\end{proposition}
\begin{proof}
It is enough to show that, under condition \eqref{cont_cond_Bes}, the solution $(R,\vu)$ remains bounded in the space $L^\infty_T(B^s_{\infty,r})\times L^\infty_T(L^2\cap B^s_{\infty,r})$. Recalling the \textit{a priori} estimates for the non-linear terms in system \eqref{QH-Euler_vor_dyadic_Bes}, we have 
\begin{equation*}
\begin{split}
2^{js}\|[\vu\cdot \nabla, \Delta_j]R\|_{L^\infty}\leq C\, c_j(t)\,\left( \|\nabla \vu \|_{L^\infty}+ \|\nabla R\|_{L^\infty}\right)\left( \|R\|_{B^s_{\infty,r}}+ \| \vu \|_{B^{s}_{\infty,r}}\right)
\end{split}
\end{equation*}
and
\begin{equation*}
\begin{split}
2^{j(s-1)}\|[\vu\cdot \nabla, \Delta_j]\omega\|_{L^\infty}\leq C c_j(t)\, \|\nabla \vu\|_{L^\infty}\|\vu\|_{B^s_{\infty,r}}\, ,
\end{split}
\end{equation*}
where we have used the fact that $\|\omega\|_{L^\infty}\leq C\|\nabla \vu\|_{L^\infty}$ and $\|\omega\|_{B^{s-1}_{\infty,r}}\leq C\| \vu\|_{B^s_{\infty,r}}$.
Moreover, from relation \eqref{est_force term_Bes}, we obtain 
\begin{equation*}
\begin{split}
\|\div (R\vu)\|_{B^{s-1}_{\infty,r}}\leq C \left(\|R\|_{L^\infty}+\|\vu\|_{L^\infty}\right) \left(\|\vu \|_{B^s_{\infty ,r}}+\|R\|_{B^s_{\infty ,r}}\right)\, .
\end{split}
\end{equation*}
Summing the previous bounds, for all $0\leq t\leq T$, we get
\begin{equation*}
\begin{split}
\|R(t)\|_{B^s_{\infty,r}}+\|\omega (t)\|_{B^{s-1}_{\infty,r}}&\leq C \left(\|R_0\|_{B^s_{\infty,r}}+\|\omega_0\|_{B^{s-1}_{\infty,r}}\right)\\
&+C\int_0^t\left(\|\nabla \vu\|_{L^\infty}+\|R\|_{W^{1,\infty}}+\|\vu\|_{L^\infty}\right)\left( \|R\|_{B^s_{\infty,r}}+ \| \vu \|_{B^{s}_{\infty,r}}\right) \, \detau \, .
\end{split}
\end{equation*}

At this point, we have to find estimates for $\|\vu\|_{L^\infty}$ and $\|R\|_{W^{1,\infty}}$. To deal with $\|\vu\|_{L^\infty}$, we separate low and hight frequencies deducing 
\begin{equation*}
\|\vu\|_{L^\infty}\leq \|\Delta_{-1}\vu\|_{L^\infty}+\sum_{j\geq 0}\|\Delta_j \vu\|_{L^\infty} \leq C\|\vu_0\|_{L^2} +\sum_{j\geq 0}\|\Delta_j \vu\|_{L^\infty} \, ,
\end{equation*}
where we have also employed the Bernstein inequalities (see Lemma \ref{l:bern}).

For the high frequency terms, we can write 
\begin{equation*}
\sum_{j\geq 0}\|\Delta_j \vu\|_{L^\infty}\leq C\sum_{j\geq 0}2^{-j}\|\Delta_j \nabla \vu\|_{L^\infty}\leq C\|\nabla \vu\|_{L^\infty}\, .
\end{equation*}
Therefore, 
\begin{equation}\label{u_low-high_freq}
\|\vu\|_{L^\infty}\leq C\left(\|\vu_0\|_{L^2}+\|\nabla \vu\|_{L^\infty}\right)\, .
\end{equation}
Now, we focus on the bound for $\|R\|_{W^{1,\infty}}$. On the one hand, $\|R(t)\|_{L^\infty}=\|R_0\|_{L^\infty}$, for all $t\geq 0$.
On the other hand, differentiating the continuity equation, we obtain 
\begin{equation}\label{R_Lipschitz}
\|\nabla R(t)\|_{L^\infty_T(L^\infty)}\leq \|\nabla R_0\|_{L^\infty}\exp \left(C\int_0^T\|\nabla \vu\|_{L^\infty}\, \dt \right)\, .
\end{equation}

Thus, using the previous relations and recalling equation \eqref{eq:rel_curl}, we finally have 
\begin{equation*}
\begin{split}
\|R(t)\|_{B^s_{\infty,r}}+\|\vu(t)\|_{L^2\cap B^s_{\infty,r}}&\leq C\left(\|R_0\|_{B^s_{\infty,r}}+\|\vu_0\|_{L^2\cap B^s_{\infty,r}}\right)\\
&+C\int_0^t\left(\|\nabla \vu \|_{L^\infty}+\|R_0\|_{W^{1,\infty}}+\|\vu_0\|_{L^2}\right)\left( \|R\|_{B^s_{\infty,r}}+ \| \vu \|_{L^2\cap B^{s}_{\infty,r}}\right) \, \detau \, .
\end{split}
\end{equation*}
In the end, employing Gr\"onwall's type inequalities, we may conclude that, under the assumption \eqref{cont_cond_Bes},
\begin{equation*}
\sup_{t\in [0,T]}\left(\|R(t)\|_{B^s_{\infty,r}}+\|\vu(t)\|_{L^2\cap B^s_{\infty,r}}\right)<+\infty\, .
\end{equation*}
\qed
\end{proof}

\subsection{The asymptotically global well-posedness result}\label{ss:improved_lifespan}
In this paragraph 
we focus on finding an asymptotic behaviour (in the regime of small oscillations for the densities) for the lifespan of solutions to system \eqref{system_Q-H_thm}. Namely, for small fluctuations $R_0$ of size $\delta>0$, the lifespan of solutions to this system tends to infinity when $\delta\rightarrow 0^+$. To show that, we have to take advantage of the \textit{linear} estimate in Theorem \ref{thm:improved_est_transport} for the transport equations in Besov spaces with zero regularity index. For that reason, it is important to work with the vorticity formulation of \eqref{system_Q-H_thm}, since $\omega \in B^0_{\infty,1}$. Thanks to the continuation criterion presented in Proposition \ref{prop:cont_criterion_Bes}, it is enough to find the bound of the lifespan in the lowest regularity space $B^1_{\infty,1}$. 


To begin with, we recall relation \eqref{eq:rel_curl}, i.e.
\begin{equation*}
\|f\|_{L^2\cap B^s_{\infty ,r}}\sim\|f\|_{L^2}+\|\curl f\|_{B^{s-1}_{\infty, r}}\, .
\end{equation*}

Therefore, due to the previous relation, we can define (for $t\geq 0$)
\begin{equation}\label{def_mathcal_E}
\mathcal{E}(t):= \|\vu(t)\|_{L^2}+\|\omega(t)\|_{B^0_{\infty,1}}\sim \|\vu(t)\|_{L^2\cap B^1_{\infty,1}}\, .
\end{equation}

Since the $L^2$ norm of the velocity field is preserved, to control $\vu$ in $B^1_{\infty ,1}$, it will be enough to find estimates for $\curl \vu$ in $B^{0}_{\infty ,1}$. Hence, we apply again the $\curl$ operator to the second equation in system \eqref{system_Q-H_thm} to get the system \eqref{QH-Euler_vorticity_Bes}, i.e. 
\begin{equation*}\label{QH-Euler_vorticity}
\begin{cases}
\d_tR+\vu \cdot \nabla R=0\\
\d_t \omega +\vu \cdot \nabla \omega=-\div (R\vu)\, .
\end{cases}
\end{equation*}

Making use of Theorem \ref{thm:improved_est_transport}, we obtain 
\begin{equation*}
\|\omega (t)\|_{B^0_{\infty,1}}\leq C \left( \|\omega_0\|_{B^0_{\infty,1}}+\int_0^t\|\div (R\vu)\|_{B^0_{\infty,1}}\detau\right)\left(1+\int_0^t\|\nabla \vu\|_{L^\infty}\detau\right)\, .
\end{equation*}
Now, we look at the bound for $\div (R\vu)$, finding that 
\begin{equation*}
\|\div (R\vu)\|_{B^0_{\infty,1}}\leq C \left(\|R\|_{L^\infty}\|\vu\|_{B^1_{\infty,1}}+\|\vu\|_{L^\infty}\|R\|_{B^1_{\infty,1}}\right)\leq C \|R\|_{B^1_{\infty,1}}\mathcal{E}(\tau)\, .
\end{equation*}
Then, we deduce 
\begin{equation}\label{eq:en_est_1}
\mathcal{E}(t)\leq C \left(\mc E(0)+\int_0^t\mc E(\tau)\|R(\tau)\|_{B^1_{\infty,1}}\, \detau \right)\left(1+\int_0^t\mc E(\tau)\, \detau\right)\, .
\end{equation}
At this point, Theorem \ref{thm_transport} implies that 
\begin{equation*}
\|R(t)\|_{B^1_{\infty,1}}\leq \|R_0\|_{B^1_{\infty,1}}\exp \left(C\int_0^t \mc E(\tau)\, \detau\right)\, .
\end{equation*}
Plugging this bound into \eqref{eq:en_est_1} gives
\begin{equation*}
\mc E(t)\leq C \left(1+\int_0^t\mc E(\tau)\, \detau\right)\left(\mc E(0)+\|R_0\|_{B^1_{\infty,1}}\int_0^t \mc E(\tau)\exp \left(\int_0^\tau \mc E(s)\, \ds \right)\, \detau \right)\, .
\end{equation*}
We now define
\begin{equation}\label{def_T^ast}
T^\ast:=\sup \left\{t>0:\|R_0\|_{B^1_{\infty,1}}\int_0^t \mc E(\tau)\exp \left(\int_0^\tau \mc E(s)\, \ds \right)\, \detau \leq \mc E(0)\right\}\, .
\end{equation}
Then, for all $t\in [0,T^\ast]$, we deduce 
\begin{equation*}
\mc E(t)\leq C\left(1+\int_0^t\mc E(\tau)\, \detau\right)\mc E(0)
\end{equation*}
and thanks to the Gr\"onwall's lemma we infer 
\begin{equation}\label{eq:en_est_2}
\mc E(t)\leq C\mc E(0)\,  e^{C\mc E(0)t}\, ,
\end{equation}
for a suitable constant $C>0$.

It remains to find a control on the integral of $\mc E(t)$. We have 
\begin{equation*}
\int_0^t\mc E(\tau) \, \detau \leq e^{C\mc E(0)t}-1
\end{equation*}
and, due to the previous bound \eqref{eq:en_est_2}, we get
\begin{equation}\label{estimate_integral_energy}
\begin{split}
\|R_0\|_{B^1_{\infty,1}}\int_0^t \mc E(\tau)\exp \left(\int_0^\tau \mc E(s)\, \ds \right)\, \detau &\leq C\|R_0\|_{B^1_{\infty,1}}\int_0^t \mc E(0)\, e^{C\mc E(0)\tau}\exp \left(e^{C\mc E(0)\tau}-1 \right)\, \detau \\
&\leq C\|R_0\|_{B^1_{\infty,1}}\left(\exp \left(e^{C\mc E(0)t}-1\right)-1\right)\, .
\end{split}
\end{equation}
Finally, by definition \eqref{def_T^ast} of $T^\ast$, we can argue that
\begin{equation*}
\mc E(0)\leq C\|R_0\|_{B^1_{\infty,1}}\left(\exp \left(e^{C\mc E(0)T^\ast}-1\right)-1\right)\, ,
\end{equation*}
which gives the following lower bound for the lifespan of solutions:
\begin{equation*}
T^\ast\geq \frac{C}{\mc E(0)}\log\left(\log \left(C\frac{\mc E(0)}{\|R_0\|_{B^1_{\infty,1}}}+1\right)+1\right)\, .
\end{equation*}
From there, recalling the definition \eqref{def_mathcal_E} for $\mc E(0)$, we have 
\begin{equation}\label{T-ast:improved}
T^\ast\geq \frac{C}{\|\vu_0\|_{L^2\cap B^1_{\infty,1}}}\log\left(\log \left(C\frac{\|\vu_0\|_{L^2\cap B^1_{\infty,1}}}{\|R_0\|_{B^1_{\infty,1}}}+1\right)+1\right)\, ,
\end{equation}
for a suitable constant $C>0$. This is the claimed lower bound stated in Theorem \ref{thm:well-posedness_Q-H-Euler}.

\section{The lifespan of solutions to the primitive problem}\label{s:lifespan_full}
The main goal of this section is to present an ``asymptotically global'' well-posedness result for system \eqref{Euler-a_eps_1}, when the size of fluctuations of the densities goes to zero, in the spirit of Subsection \ref{ss:improved_lifespan}. We start by showing a continuation type criterion for system \eqref{Euler-a_eps_1} and discussing the related consequences (see Subsection \ref{ss:cont_criterion+consequences} below for details).
We conclude this section presenting the asymptotic behaviour of the lifespan of solutions to system \eqref{Euler-a_eps_1}: the lifespan may be very large, if the size of non-homogeneities $a_{0,\veps}$ defined in \eqref{def_a_veps} is small (see relation \eqref{asymptotic_time} below). We point out that it is \textit{not} clear at all that the global existence holds in a fast rotation regime without any assumption of smallness on the size of the densities. 
 
\subsection{The continuation criterion and consequences}\label{ss:cont_criterion+consequences}
In this paragraph, we start by presenting a continuation type result in Sobolev spaces for system \eqref{Euler-a_eps_1}, in the spirit of the Beale-Kato-Majda continuation criterion \cite{B-K-M}. 
The proof is an adaptation of the arguments in \cite{B-L-S} by Bae, Lee and Shin.
\begin{proposition}\label{prop:cont_criterion_original_prob}
Take $\veps \in\;  ]0,1]$ fixed. Let $(a_{0,\veps},\vu_{0,\veps})\in L^\infty \times H^s$ with $\nabla a_{0,\veps}\in H^{s-1} $ and $\div \vu_{0,\veps}=0$. Given a time $T>0$, let $(a_\veps,\vu_\veps, \nabla \Pi_\veps)$ be a solution of \eqref{Euler-a_eps_1} on $[0,T[$ that belongs to $L^\infty_t(L^\infty)\times L^\infty_t(H^s)\times L^\infty_t(H^s)$ and $\nabla a_\veps \in L^\infty_t(H^{s-1}) $ for any $t\in [0,T[$. If we assume that  
\begin{equation}\label{cont_cond_orig_prob}
\int_0^T\|\nabla \vu_\veps \|_{L^\infty}\, \dt<+\infty\, ,
\end{equation} 
then $(a_\veps, \vu_\veps,\nabla \Pi_\veps)$ can be continued beyond $T$ into a solution of \eqref{Euler-a_eps_1} with the same regularity.
\end{proposition}
\begin{proof}
As already pointed out in the proof of Proposition \ref{prop:cont_criterion_Bes}, it would be enough to show that 
$$ \sup_{0\leq t<T}\left(\|\vec u_\veps\|_{H^s}+\|\nabla a_\veps\|_{H^{s-1}}\right)<+\infty \, . $$
Since $\veps \in \, ]0,1]$ is fixed and does not play any role in the following proof, for notational convenience, we set it equal to 1. 

We start by recalling that, from the continuity equation of \eqref{Euler-a_eps_1}, one gets
\begin{equation}\label{cont_eq_nabla_a}
 \d_t \d_i a+\vec u\cdot \nabla \d_i a =-\d_i \vec u \cdot \nabla a\quad \quad \text{for}\; i=1,2\, . 
\end{equation}
So, applying the operator $\Delta_j$ to the above relation and using the divergence-free condition $\div \vec u=0$, one has 
$$ \d_t \Delta_j \d_i a+\vec u \cdot \nabla \Delta_j \d_i a=-\Delta_j\left(\d_i \vec u \cdot \nabla a\right)+[\vec u\cdot \nabla,\Delta_j]\d_i a \, .$$
Therefore, thanks to the commutator estimates (see Lemma \ref{l:commutator_est}), one may argue that 
\begin{equation*}
2^{j(s-1)}\|[\vec u \cdot \nabla,\Delta_j]\d_i a\|_{L^2}\leq C\, c_j(t)\left( \|\nabla \vec u\|_{L^\infty}+\|\nabla a\|_{L^\infty}\right)\left(\|\nabla a\|_{H^{s-1}}+\|\vec u\|_{H^s}\right)\, ,
\end{equation*}
where $(c_j(t))_{j\geq -1}$ is a sequence in the unit ball of $\ell^2$, and due to Corollary \ref{cor:tame_est} one has
\begin{equation*}
\|\d_i \vec u\cdot \nabla a\|_{H^{s-1}}\leq C\left(\|\nabla \vec u\|_{L^\infty}\|\nabla a\|_{H^{s-1}}+\|\nabla \vec u\|_{H^{s-1}}\|\nabla a\|_{L^\infty}\right)\, .
\end{equation*}
At this point, we recall the bounds for the momentum equation in system \eqref{Euler-a_eps_1}. First of all, we apply the non-homogeneous dyadic blocks $\Delta_j$, getting
\begin{equation*}
\d_t \Delta_j \vec u+\vec u\cdot \nabla \Delta_j \vec u+\Delta_j \vec u^\perp +\nabla \Delta_j \Pi +\Delta_j\left(a\nabla \Pi\right)=[\vec u \cdot \nabla,\Delta_j]\vec u\,.
\end{equation*}
Then, we obtain 
\begin{equation*}
2^{js}\|[\vec u\cdot \nabla, \Delta_j]\vec u\|_{L^2}\leq C \, c_j(t)\|\nabla \vec u\|_{L^\infty}\|\vec u\|_{H^s}
\end{equation*}
with $(c_j(t))_{j\geq -1}$ a sequence in the unit ball of $\ell^2$, and in addition we have
\begin{equation*}
\|a\nabla \Pi\|_{H^s}\leq C \left(\|a\|_{L^\infty}\|\nabla \Pi\|_{H^s}+\|\nabla \Pi\|_{L^\infty}\|\nabla a\|_{H^{s-1}}\right)\, .
\end{equation*}
Summing up the previous inequalities, for all $t\in [0,T[$ we may infer that
\begin{equation*}
\begin{split}
\|\nabla a (t)\|_{H^{s-1}}+\|\vec u(t)\|_{H^s}&\leq \left(\|\nabla a_0\|_{H^{s-1}}+\|\vec u_0\|_{H^s}\right)+C\int_0^t\|a\|_{L^\infty}\|\nabla \Pi\|_{H^s}\detau\\
&+C\int_0^t \left(\|\nabla a\|_{L^\infty}+\|\nabla \vec u\|_{L^\infty}+\|\nabla \Pi\|_{L^\infty}\right)\left(\|\nabla a \|_{H^{s-1}}+\|\vec u\|_{H^s}\right)\detau\, .
\end{split}
\end{equation*}
To close the proof, under the hypothesis of the theorem, we have to find the bounds $\|\nabla a\|_{L^1_T(L^\infty)}$, $\|\nabla \Pi\|_{L^1_T(L^\infty)}$ and $\|\nabla \Pi\|_{L^1_T(H^s)}$. 

From the continuity equation \eqref{cont_eq_nabla_a}, we obtain 
\begin{equation*}
\|\nabla a\|_{L^\infty_T(L^\infty)}\leq C\|\nabla a_0\|_{L^\infty}\exp\left(\int_0^T\|\nabla \vec u\|_{L^\infty}\detau\right)\, .
\end{equation*}
Now, we focus on the estimate $\|\nabla \Pi\|_{H^s}$. 
We recall again the elliptic equation
\begin{equation}\label{ellip_eq_cont_thm}
-\div (A\nabla \Pi)=\div F\, \quad \text{where}\quad  F:= \vu \cdot \nabla \vu+ \vu^\perp \quad \text{and}\quad A:=1/\vrho\, . 
\end{equation}
From the previous relation, it is a standard matter to deduce that (see also Proposition \ref{app:prop7_danchin}):
\begin{equation*}
\|\nabla \Pi\|_{H^s}\leq C\left(\|\nabla \vec u\|_{L^\infty}\|\vec u\|_{H^s}+\|\vec u\|_{H^s}+\|\nabla a\|_{L^\infty}\|\nabla \Pi\|_{H^{s-1}}+\|\nabla \Pi \|_{L^\infty}\|\nabla a\|_{H^{s-1}}\right)\, .
\end{equation*}
Using an interpolation argument, one has 
$$ \|\nabla \Pi\|_{H^{s-1}}\leq C\|\nabla \Pi\|_{L^2}^{1/s}\|\nabla \Pi\|_{H^s}^{1-1/s} $$
and due to the Young's inequality we end up with 
\begin{equation*}
\|\nabla a\|_{L^\infty}\|\nabla \Pi\|_{H^{s-1}}\leq C\left(\|\nabla a\|_{L^\infty}^s \|\nabla \Pi\|_{L^2}+\left(1-\frac{1}{s}\right)\|\nabla \Pi\|_{H^s}\right)\, .
\end{equation*}
In addition, we already know that 
$$ \|\nabla \Pi\|_{L^2}\leq C \left(\|\vec u\|_{L^2}\|\nabla \vec u\|_{L^\infty}+\|\vec u\|_{L^2}\right)\leq C \left(\|\vec u_0\|_{L^2}\|\nabla \vec u\|_{L^\infty}+\|\vec u_0\|_{L^2}\right)\leq C \left(\|\nabla \vec u\|_{L^\infty}+1\right) \, .$$
As it is apparent the constant term on the right-hand side will be irrelevant in the next computations: hence, it will be omitted supposing e.g. that $\|\nabla \vec u\|_{L^\infty}>1$. 

At this point, we have only to take care of the $L^\infty$ bound for the pressure term. Thanks to an application of Gagliardo-Nirenberg (see Theorem \ref{app:thm_G-N}) and Young inequalities, we get 
\begin{equation*}
\|\nabla \Pi\|_{L^\infty}\leq C \|\Delta \Pi\|_{L^4}^{2/3}\|\nabla \Pi\|_{L^2}^{1/3}\leq C\left(\|\Delta \Pi\|_{L^4}+\|\nabla \Pi\|_{L^2}\right)\leq C\left(\|\Delta \Pi\|_{L^4}+\|\nabla \vec u\|_{L^\infty}\right)\, .
\end{equation*}
Again from the elliptic equation \eqref{ellip_eq_cont_thm}, one can find
\begin{equation*}
\Delta \Pi=-\vrho \nabla a\cdot \nabla \Pi-\vrho \, \div \left( \vec u\cdot \nabla \vec u\right)-\vrho\, \div \vec u^\perp
\end{equation*}
and then 
\begin{equation*}
\begin{split}
\|\Delta \Pi\|_{L^4}&\leq C\left(\|\vrho\|_{L^\infty}\|\nabla a\|_{L^\infty}\|\nabla \Pi\|_{L^4}+\|\vrho\|_{L^\infty}\|\nabla \vec u\|_{L^\infty}\|\nabla \vec u\|_{L^4}+\|\vrho\|_{L^\infty}\|\nabla \vec u\|_{L^4}\right)\\
&\leq C\left(\|\nabla a\|_{L^\infty}\|\nabla \Pi\|_{L^4}+\|\nabla \vec u\|_{L^\infty}\|\nabla \vec u\|_{L^4}+\|\nabla \vec u\|_{L^4}\right)\, ,
\end{split}
\end{equation*}
where we have employed the fact that the densities are bounded from above.

Once again, due to the Gagliardo-Nirenberg's inequality, we obtain 
$$ \|\nabla \Pi\|_{L^4}\leq C\|\Delta \Pi\|_{L^4}^{1/3}\|\nabla \Pi\|_{L^2}^{2/3}\, . $$
So, 
\begin{equation*}
\|\Delta \Pi\|_{L^4}\leq C\left(\|\nabla a\|_{L^\infty}\|\Delta \Pi\|_{L^4}^{1/3}\|\nabla \Pi\|_{L^2}^{2/3}+\|\nabla \vec u\|_{L^\infty}\|\nabla \vec u\|_{L^4}+\|\nabla \vec u\|_{L^4}\right)\, .
\end{equation*}
Therefore, Young inequality implies that 
\begin{equation*}
\begin{split}
\|\Delta \Pi\|_{L^4}&\leq C\left(\|\nabla a\|_{L^\infty}^{3/2}\|\nabla \Pi\|_{L^2}+\|\nabla \vec u\|_{L^\infty}\|\nabla \vec u\|_{L^4}+\|\nabla \vec u\|_{L^4}\right)\\
&\leq C\left(\|\nabla a\|_{L^\infty}^{3/2}\|\nabla \vec u\|_{L^\infty}+\|\nabla \vec u\|_{L^\infty}\|\nabla \vec u\|_{L^4}+\|\nabla \vec u\|_{L^4}\right)\, .
\end{split}
\end{equation*}
At the end, we ensure that 
\begin{equation*}
\|\nabla \Pi\|_{L^\infty}\leq C\left(\|\nabla a\|_{L^\infty}^{3/2}\|\nabla \vec u\|_{L^\infty}+\|\nabla \vec u\|_{L^\infty}\|\nabla \vec u\|_{L^4}+\|\nabla \vec u\|_{L^4}+\|\nabla \vec u\|_{L^\infty}\right)\, .
\end{equation*}
Now, we have to estimate $\|\nabla \vec u\|_{L^4}$: to do so, we take advantage of the vorticity formulation
$$ \d_t \omega+\vec u\cdot \nabla \omega=-\nabla a \wedge \nabla \Pi\, , $$
where $\nabla a \wedge \nabla \Pi:=\d_1 a\, \d_2 \Pi-\d_2 a \, \d_1 \Pi$.

From that formulation, we get the following bound for all $t\in [0,T[$ :
\begin{equation*}
\|\omega(t)\|_{L^4}\leq \|\omega_0\|_{L^4}+C\int_0^t\|\nabla a\|_{L^\infty}\|\nabla \Pi\|_{L^4}\detau\, .
\end{equation*}
As done in the previous computations, we can deduce that 
\begin{equation*}
\begin{split}
\|\omega(t)\|_{L^4}&\leq \|\omega_0\|_{L^4}+C\int_0^t \|\nabla a\|_{L^\infty}\|\Delta \Pi\|_{L^4}^{1/3}\|\nabla \vec u\|_{L^\infty}^{2/3}\\
&\leq \|\omega_0\|_{L^4}+C\int_0^t \|\nabla a\|_{L^\infty}\left(\|\nabla a\|_{L^\infty}^{3/2}\|\nabla \vec u\|_{L^\infty}+\|\nabla \vec u\|_{L^\infty}\|\nabla \vec u\|_{L^4}+\|\nabla \vec u\|_{L^4}\right)^{1/3}\|\nabla \vec u\|_{L^\infty}^{2/3}\\
&\leq \|\omega_0\|_{L^4}+C\int_0^t \left(\|\nabla a\|_{L^\infty}^{3/2}\|\nabla \vec u\|_{L^\infty}+\|\nabla a\|_{L^\infty}\|\nabla \vec u\|_{L^\infty}\|\omega\|_{L^4}^{1/3}+\|\nabla a\|_{L^\infty}\|\nabla \vec u\|_{L^\infty}^{2/3}\|\omega\|_{L^4}^{1/3}\right).
\end{split}
\end{equation*}
At this point, we apply the Young's inequality to infer:
\begin{itemize}
\item[(i)] $\|\nabla a\|_{L^\infty}^{3/2}\|\nabla \vec u\|_{L^\infty}\leq C\left(\|\nabla a\|_{L^\infty}\|\nabla \vec u\|_{L^\infty}+\|\nabla a\|_{L^\infty}^{2}\|\nabla \vec u\|_{L^\infty}\right)$;
\item[(ii)] $\|\nabla a\|_{L^\infty}\|\nabla \vec u\|_{L^\infty}\|\omega\|_{L^4}^{1/3}\leq C\left(\|\nabla a\|_{L^\infty}\|\nabla \vec u\|_{L^\infty}\|\omega\|_{L^4}+\|\nabla a\|_{L^\infty}\|\nabla \vec u\|_{L^\infty}\right)$;
\item[(iii)] $\|\nabla a\|_{L^\infty}\|\nabla \vec u\|_{L^\infty}^{2/3}\|\omega\|_{L^4}^{1/3}\leq C\left(\|\nabla a\|_{L^\infty}\|\nabla \vec u\|_{L^\infty}+\|\nabla a\|_{L^\infty}\|\omega\|_{L^4}\right)$.
\end{itemize}

\medskip

In the end, for all $t\in [0,T[\, $,
\begin{equation*}
\begin{split}
\|\omega(t)\|_{L^4}&\leq \|\omega_0\|_{L^4}+C\int_0^t\left(\|\nabla a\|_{L^\infty}\|\nabla \vec u\|_{L^\infty}\|\omega\|_{L^4}+\|\nabla a\|_{L^\infty}\|\omega\|_{L^4}\right) \detau\\
&+C\int_0^t\left(\|\nabla a\|_{L^\infty}^2\|\nabla \vec u\|_{L^\infty}+\|\nabla a\|_{L^\infty}\|\nabla \vec u\|_{L^\infty}\right)\detau \, . 
\end{split}
\end{equation*}
Hence, thanks to the Gr\"onwall's lemma, we get 
$$ \|\omega(t)\|_{L^4}\leq \exp\left[C\int_0^t \|\nabla a\|_{L^\infty}\left(\|\nabla \vec u\|_{L^\infty}+1\right) \right]\left[\|\omega_0\|_{L^4}+C\int_0^t\|\nabla a\|_{L^\infty}\|\nabla \vec u\|_{L^\infty}\left(\|\nabla a\|_{L^\infty}+1\right) \right] .$$
Recalling condition \eqref{cont_cond_orig_prob}, the theorem is thus proved.
\qed
\end{proof}

\medskip

At this point, we discuss some consequences of the previous result. 

In particular, it would be enough to control $\vec u_\veps$ in $L^\infty_T(L^2\cap B^1_{\infty,1})$ in order to have the existence of the solution until the time $T$. 
Indeed, if we are able to control the norm $\|\ue \|_{L^\infty_T(L^2\cap B^1_{\infty,1})}$, then we are able to bound $\|\nabla \ue \|_{L^\infty_T(L^\infty)}$. This will imply \eqref{cont_cond_orig_prob} and, therefore, the solution will exist until time $T$. 


Let us give some details. First of all, we have 
\begin{equation*}
\|\nabla \ue \|_{L^\infty_T(L^\infty)}\leq C\, \| \ue \|_{L^\infty_T(B^1_{\infty,1})}\, .
\end{equation*} 
As already pointed out in Lemma \ref{l:rel_curl}, to control the $B^1_{\infty,1}$ norm of $\ue$ it is enough to have a $L^2$ estimate for $\ue$ and a $B^0_{\infty, 1}$ estimate for its $\curl$. Those estimates are the topic of the next Subsection \ref{ss:asym_lifespan}, provided that the time $T>0$ is defined as in \eqref{def_T^ast_veps} below. 

Therefore, $\|\nabla \ue\|_{L^\infty_T(L^\infty)}<+\infty$ and so 
$$ \int_0^T \|\nabla \ue\|_{L^\infty}<+\infty\, ,$$
for such $T>0$. 

Finally, we note that we have already shown the existence and uniqueness of solutions to system \eqref{Euler-a_eps_1} in the Sobolev spaces $H^s$ with $s>2$ (see Section \ref{s:well-posedness_original_problem}) and, thanks to Proposition \ref{p:embed}, those spaces are continuously embedded in the space $B^1_{\infty,1}$. Thus, the solution will exist until $T$.

\subsection{The asymptotic lifespan}\label{ss:asym_lifespan}
In this paragraph we focus our attention on the lifespan $T^\ast_\veps$ of solutions $(\vrho_\veps, \vec u_\veps, \nabla \Pi_\veps)$ to the primitive system \eqref{Euler-a_eps_1}. 
We point out that if we consider the initial densities as in \eqref{in_vr_E}, i.e. $\vrho_{0,\veps}=1+\veps R_{0,\veps}$, it is not clear to us how to show that $T_\veps^\ast \longrightarrow +\infty$ when $\veps\rightarrow 0^+$. Nevertheless, on the one hand, as soon as the densities are of the form $\vrho_{0,\veps}=1+\veps^{1+\alpha}R_{0,\veps}$ (with $\alpha >0$), we obtain that $T^\ast_\veps \sim \log \log \left(1/\veps\right)$.
On the other hand, independently of the rotational effects, we can state an ``asymptotically global'' well-posedness result in the regime of small oscillations: namely, we get $T^\ast_\veps \geq T^\ast(\delta)$, with $T^\ast (\delta)\longrightarrow +\infty$, when the size $\delta >0$ of $R_{0,\veps}$ goes to $0^+$ (this is coherent with the result in \cite{D-F_JMPA} for a density-depend fluid in the absence of Coriolis force).

Therefore, the main goal of this subsection is to prove estimate \eqref{improv_life_fullE} of Theorem \ref{W-P_fullE}.

First of all, we have to take advantage of the vorticity formulation of system \eqref{Euler-a_eps_1}. To do so, we apply the $\curl$ operator to the momentum equation, obtaining 
\begin{equation}\label{curl_eq}
\d_t  \omega_\veps +\ue \cdot \nabla  \omega_\veps+ \nabla a_\veps\,  \wedge\,  \nabla \Pi_\veps=0\, ,
\end{equation}
where we recall $\omega_\veps:=\curl \ue$ and $\nabla a_\veps\,  \wedge\,  \nabla \Pi_\veps:=\d_1 a_\veps \,\d_2\Pi_\veps-\d_2 a_\veps \d_1\Pi_\veps$.

We notice that the vorticity formulation is the key point to bypass the issues coming from the Coriolis force, whose singular effects disappear in \eqref{curl_eq}.

Next, we make use of Theorem \ref{thm:improved_est_transport} and we deduce that 
\begin{equation}\label{transp_curl}
\| \omega_\veps \|_{B^0_{\infty,1}}\leq C \left(\| \omega_{0,\veps}\|_{B^0_{\infty,1}}+\int_0^t \|\nabla a_\veps\, \wedge\, \nabla \Pi_\veps \|_{B^0_{\infty,1}}\detau \right)\left(1+\int_0^t\|\nabla \ue\|_{L^\infty} \detau\right)\, .
\end{equation}
We start by bounding the $B^0_{\infty,1}$ norm of $\nabla a_\veps\, \wedge\, \nabla \Pi_\veps$. We observe that 
\begin{equation}\label{l-p_dec_wedge}
\begin{split}
\d_1a_\veps\, \d_2 \Pi_\veps-\d_2a_\veps\, \d_1\Pi_\veps&=\mc T_{\d_1a_\veps}\d_2 \Pi_\veps-\mc T_{\d_2a_\veps}\d_1 \Pi_\veps+\mc T_{\d_2\Pi_\veps}\d_1 a_\veps-\mc T_{\d_1\Pi_\veps}\d_2 a_\veps\\
&+\d_1\mc R(a_\veps -\Delta_{-1}a_\veps,\, \d_2 \Pi_\veps)-\d_2\mc R(a_\veps -\Delta_{-1}a_\veps,\, \d_1 \Pi_\veps)\\
&+\mc R(\d_1\Delta_{-1}a_\veps,\, \d_2 \Pi_\veps)+\mc R(\d_2\Delta_{-1}a_\veps,\, \d_1 \Pi_\veps)\, .
\end{split}
\end{equation}


Applying Proposition \ref{T-R} directly to the terms involving the paraproduct $\mc T$, we have 
\begin{equation*}
\|\mc T_{\nabla a_\veps}\nabla \Pi_\veps\|_{B^0_{\infty,1}}+\|\mc T_{\nabla \Pi_\veps}\nabla  a_\veps\|_{B^0_{\infty,1}}\leq C \left(\|\nabla a_\veps\|_{L^\infty}\|\nabla \Pi_\veps\|_{B^0_{\infty,1}}+\|\nabla a_\veps\|_{B^0_{\infty,1}}\|\nabla \Pi_\veps\|_{L^\infty}\right)\, .
\end{equation*}
Next, we have to deal with the remainders $\mc R$. We start by bounding the $B^0_{\infty,1}$ norm of $\d_1\mc R(a_\veps -\Delta_{-1}a_\veps,\, \d_2 \Pi_\veps)$. 
One has:
\begin{equation*}
\begin{split}
\|\d_1\mc R(a_\veps -\Delta_{-1}a_\veps,\, \d_2 \Pi_\veps)\|_{B^0_{\infty,1}}&\leq C\|\mc R(a_\veps -\Delta_{-1}a_\veps,\, \d_2 \Pi_\veps)\|_{B^1_{\infty,1}}\\
&\leq C\left(\|\nabla \Pi_\veps\|_{B^0_{\infty,\infty}}\|(\text{Id}-\Delta_{-1})\, a_\veps\|_{B^1_{\infty,1}}\right)\\
&\leq C \left(\|\nabla \Pi_\veps\|_{L^\infty}\|\nabla a_\veps\|_{B^0_{\infty,1}}\right)\, ,
\end{split}
\end{equation*}
where we have employed the localization properties of the Littlewood-Paley decomposition. In a similar way, one can argue for 
$\d_2\mc R(a_\veps -\Delta_{-1}a_\veps,\, \d_1 \Pi_\veps)$.

It remains to bound $\mc R(\d_1\Delta_{-1}a_\veps,\, \d_2 \Pi_\veps)$, and analogously one can treat the term $\mc R(\d_2\Delta_{-1}a_\veps,\, \d_1 \Pi_\veps)$ in \eqref{l-p_dec_wedge}. 
We obtain that 
\begin{equation*}
\|\mc R(\d_1\Delta_{-1}a_\veps,\, \d_2 \Pi_\veps)\|_{B^0_{\infty,1}}\leq C \|\mc R(\d_1\Delta_{-1}a_\veps,\, \d_2 \Pi_\veps)\|_{B^1_{\infty,1}}\leq C\left(\|\nabla \Pi_\veps\|_{L^\infty}\|\d_1\Delta_{-1}a_\veps\|_{B^1_{\infty,1}}\right)\, .
\end{equation*}
Employing the spectral properties of operator $\Delta_{-1}$, one has that 
\begin{equation*}
\|\d_1\Delta_{-1}a_\veps\|_{B^1_{\infty,1}}\leq C\|\Delta_{-1}\nabla a_\veps\|_{L^\infty}\, .
\end{equation*}
Then, 
$$ \|\mc R(\d_1\Delta_{-1}a_\veps,\, \d_2 \Pi_\veps)\|_{B^0_{\infty,1}}\leq C \left(\|\nabla \Pi_\veps\|_{L^\infty}\|\nabla a_\veps\|_{B^0_{\infty,1}}\right)\, .$$
Finally, we get
\begin{equation*}
\|\nabla a_\veps\, \wedge\, \nabla \Pi_\veps \|_{B^0_{\infty,1}}\leq C\left(\|\nabla a_\veps\|_{L^\infty}\|\nabla \Pi_\veps\|_{B^0_{\infty,1}}+\|\nabla a_\veps\|_{B^0_{\infty,1}}\|\nabla \Pi_\veps\|_{L^\infty}\right)\, .
\end{equation*}
So plugging the previous estimate in \eqref{transp_curl}, one gets
\begin{equation*}
\| \omega_\veps \|_{B^0_{\infty,1}}\leq C \left(\| \omega_{0,\veps}\|_{B^0_{\infty,1}}+\int_0^t \|\nabla a_\veps\|_{B^0_{\infty,1}} \| \nabla \Pi_\veps \|_{B^0_{\infty,1}}\detau \right)\left(1+\int_0^t\|\nabla \ue\|_{L^\infty} \detau\right)\, .
\end{equation*}
At this point, we define
\begin{equation}\label{def_energy_and_A}
E_\veps(t):=\|\ue (t)\|_{L^2\cap B^1_{\infty,1}}\qquad \text{and}\qquad \mc A_\veps(t):=\|\nabla a_\veps (t)\|_{B^0_{\infty,1}}\, .
\end{equation}
In this way, we have 
\begin{equation}\label{energy}
E_\veps(t)\leq C \left(E_\veps (0)+\int_0^t \mc A_\veps(\tau) \|\nabla \Pi_\veps(\tau)\|_{B^0_{\infty,1}}\detau \right)\left(1+\int_0^t E_\veps (\tau )\detau \right)\, .
\end{equation}
Next, we recall that, for $i=1,2$:
\begin{equation*}
\d_t\, \d_i a_\veps+\ue \cdot \nabla \, \d_i a_\veps=-\d_i \ue \cdot \nabla a_\veps
\end{equation*}
and, due to the divergence-free condition on $\vec u_\veps$, we can write
\begin{equation*}
\d_i \ue \cdot \nabla a_\veps=\sum_j\d_i \ue^j \, \d_ja_\veps=\sum_j\Big(\d_i(\ue^j\, \d_j a_\veps)-\d_j(\ue^j\, \d_ia_\veps)\Big)\, .
\end{equation*}
So, using Proposition \ref{T-R} and the fact that 
\begin{equation*}
\|\mc \d_i \mc R(\ue^j,\, \d_j a_\veps)\|_{B^0_{\infty,1}}\leq C\, \|\mc R(\ue^j,\, \d_j a_\veps)\|_{B^1_{\infty,1}}\leq C\, \|\nabla a_\veps\|_{B^0_{\infty,1}}\|\ue\|_{B^1_{\infty,1}}\, ,
\end{equation*}
we may finally get
\begin{equation*}
\|\d_i \ue \cdot \nabla a_\veps\|_{B^0_{\infty,1}}\leq C \|\nabla a_\veps\|_{B^0_{\infty,1}}\|\ue \|_{B^1_{\infty,1}}\, .
\end{equation*}
Thus, 
\begin{equation*}
\|\nabla a_\veps (t)\|_{B^0_{\infty,1}}\leq \|\nabla a_{0, \veps}\|_{B^0_{\infty,1}}\exp \left(C\int_0^t \| \ue\|_{B^1_{\infty,1}}\detau \right)\, .
\end{equation*}
Therefore, recalling \eqref{def_energy_and_A}, one has
\begin{equation}\label{est_A}
\mc A_\veps(t)\leq \mc A_\veps (0) \exp \left( C\int_0^t E(\tau) \detau\right)\, .
\end{equation}
The next goal is to bound the pressure term in $B^0_{\infty,1}$. Actually, we shall bound its $B^1_{\infty,1}$ norm. Similarly to the analysis performed in Subsection \ref{ss:unif_est} for the $H^s$ norm (see e.g. inequality \eqref{est_Pi_H^s_1} in this respect), there exists some exponent $\lambda \geq 1$ such that
\begin{equation*}\label{est_pres}
\|\nabla \Pi_\veps\|_{B^1_{\infty,1}}\leq C\left(\left(1+\veps \|\nabla a_\veps \|_{B^0_{\infty,1}}^\lambda \right)\|\nabla \Pi_\veps \|_{L^2}+\veps\, \|\vrho_\veps\, \div (\ue \cdot \nabla \ue)\|_{B^0_{\infty,1}}+\|\vrho_\veps \, \div \ue^\perp\|_{B^0_{\infty,1}}\right)\, .
\end{equation*}
The $L^2$ estimate for the pressure term follows in a similar way to one performed in \eqref{est:Pi_L^2}, i.e. 
\begin{equation}\label{L^2-est-pressure}
\|\nabla \Pi_\veps\|_{L^2}\leq C\, \veps \, \|\ue \|_{L^2}\|\nabla \ue \|_{L^\infty}+\|\ue \|_{L^2}\, .
\end{equation}

Next, as showed above in the bound for $\|\d_i \ue \cdot \nabla a_\veps\|_{B^0_{\infty,1}}$, combining Bony's decomposition with the fact that $\div (\ue \cdot \nabla \ue)=\nabla \ue :\nabla \ue$, we may infer:
\begin{equation*}
\|\div (\ue \cdot \nabla \ue)\|_{B^0_{\infty,1}}\leq C\, \|\ue \|^2_{B^1_{\infty,1}}\, .
\end{equation*} 
Now, our scope is to estimate the $B^1_{\infty,1}$ norm of the density. To do so, we make use of the following proposition, whose proof can be found in \cite{D_JDE}. 
\begin{proposition}
Let $I$ be an open interval of $\R$ and $\oline F:I\rightarrow \R$ a smooth function. Then, for any compact subset $J\subset I$, $s>0$ and $(p,r)\in [1,+\infty]^2$, there exists a constant $C$ such that for any function $g$ valued in $J$ and with gradient in $B^{s-1}_{p,r}$, we have $\nabla \big(\oline F(g)\big)\in B^{s-1}_{p,r}$ and 
$$\left\|\nabla \big(\oline F(g)\big)\right\|_{B^{s-1}_{p,r}}\leq C\left\|\nabla g\right\|_{B^{s-1}_{p,r}}\, .$$
\end{proposition}
Then, from the definition of $B^1_{\infty,1}$ and the previous proposition, the $B^1_{\infty,1}$ estimate for $\vrho_\veps$ reads:
\begin{equation*}
\|\vrho_\veps\|_{B^1_{\infty,1}}\leq C\, \left( \oline\vrho+\veps\, \|\nabla a_\veps\|_{B^0_{\infty,1}}\right)\, .
\end{equation*}
Finally, plugging the $L^2$ estimate \eqref{L^2-est-pressure} and all the above inequalities in \eqref{est_pres}, one may conclude that
\begin{equation}\label{est_pres_1}
\begin{split}
\|\nabla \Pi_\veps\|_{B^1_{\infty,1}}&\leq C \left(1+\veps\, \|\nabla a_\veps \|_{B^0_{\infty,1}}^\lambda \right)\left(\veps\, \|\ue\|_{L^2}\|\nabla \ue\|_{B^0_{\infty,1}}+\|\ue \|_{L^2}\right)\\
&+C \left(1+\veps\, \|\nabla a_\veps \|_{B^0_{\infty,1}} \right)\left(\veps\, \| \ue\|_{B^1_{\infty,1}}^2+\|\ue \|_{B^1_{\infty,1}}\right)\\
&\leq C (1+\veps\, \mc A_\veps^\lambda)(\veps\, E_\veps^2+E_\veps)+C(1+\veps \mc A_\veps)(\veps\, E_\veps^2+E_\veps)\\
&\leq C (\veps \, E_\veps^2+E_\veps)(1+\veps\,\mc A_\veps +\veps\, \mc A_\veps^\lambda)\\
&\leq C (\veps \, E_\veps^2+E_\veps)(1 +\veps\, \mc A_\veps^\lambda)\, .
\end{split}
\end{equation}
We insert now in \eqref{energy} the estimates found in \eqref{est_pres_1} and in \eqref{est_A}, deducing that 
\begin{equation}\label{est_E}
\begin{split}
&E_\veps(t)\leq C\left(E_\veps(0)+\mc B_\veps(0)\int_0^t \exp \left(C \int_0^\tau E_\veps (s)\, \ds\right) \left(\veps\, E_\veps^2(\tau)+ E_\veps(\tau)\right)\, \detau \right)\\
&\qquad \qquad \qquad \qquad \qquad \qquad \qquad \qquad \qquad \qquad \qquad \qquad \qquad \times \left(1+\int_0^t E_\veps (\tau) \detau \right) ,
\end{split}
\end{equation}
where we have set $\mc B_\veps (0):=\mc A_\veps(0)+\veps\, \mc A_\veps(0)^{\lambda +1} $.

At this point, we define $T_\veps^\ast>0$ such that
\begin{equation}\label{def_T^ast_veps}
T_\veps^\ast:= \sup \left\{t>0 : \, \mc B_\veps(0)\int_0^t \exp \left(C \int_0^\tau E_\veps (s)\, \ds\right) \left(\veps\, E_\veps^2(\tau)+ E_\veps(\tau)\right)\, \detau \leq E_\veps(0) \right\}\, .
\end{equation}
So, from \eqref{est_E} and using Gr\"onwall's inequality, we obtain that
\begin{equation*}
E_\veps(t)\leq C\,  E_\veps(0)e^{CtE_\veps(0)}\, ,
\end{equation*}
for all $t\in [0,T_\veps^\ast]$.

The previous estimate implies that, for all $t\in [0,T_\veps^\ast]$, one has 
\begin{equation*}
\int_0^t E_\veps(\tau) \detau \leq  e^{CtE_\veps(0)}-1\, .
\end{equation*}
Analogously to inequality \eqref{estimate_integral_energy} in Subsection \ref{ss:improved_lifespan}, we can argue that  
\begin{equation*}
\mc B_\veps(0)\int_0^t \exp \left(C \int_0^\tau E_\veps (s)\, \ds\right)  E_\veps(\tau)\, \detau\leq C \mc B_\veps(0) \left(\exp \left(e^{CtE_\veps(0)}-1\right)-1\right)\, .
\end{equation*}
Then, it remains to control 
$$ \veps \mc B_\veps(0) \int_0^t \exp \left(C \int_0^\tau E_\veps (s)\, \ds\right)  E_\veps^2(\tau)\, \detau\, . $$ 
For this term, we may infer that 
\begin{equation*}
\begin{split}
\veps \mc B_\veps(0)\int_0^t \exp \left(C \int_0^\tau E_\veps (s)\, \ds\right)  E_\veps^2(\tau)\, \detau &\leq C\, \veps \,\mc B_\veps(0) \int_0^t E_\veps^2(0)e^{C\tau E_\veps(0)}\, \exp\left(e^{C \tau E_\veps(0)}-1\right)\, \detau\\
&\leq C\, \veps\, \mc B_\veps(0) E_\veps (0) \left(\exp \left(e^{CtE_\veps(0)}-1\right)-1\right)\, .
\end{split}
\end{equation*}
In the end, by definition \eqref{def_T^ast_veps} of $T^\ast_\veps$, we deduce
\begin{equation}\label{asymptotic_time}
T^\ast_\veps\geq \frac{C}{E_\veps(0)}\log\left(\log\left(\frac{CE_\veps(0)}{\max \{\mc B_\veps(0),\, \veps \, \mc B_\veps(0)E_\veps(0)\}}+1\right)+1\right),
\end{equation} 
for a suitable constant $C>0$. This concludes the proof of Theorem \ref{W-P_fullE}.

\chapter*{Future perspectives}\addcontentsline{toc}{chapter}{Future perspectives}
We conclude this manuscript pointing out some possible upcoming goals. 

First of all, we would like to continue the studies started in Chapter \ref{chap:multi-scale_NSF} and Chapter \ref{chap:BNS_gravity} in two different directions:
\begin{itemize}
\item on the one hand, we will investigate the regimes not covered yet, i.e. either $m=1$ with the centrifugal effects or $(m+1)/2<n<m$;
\item on the other hand, we would focus on the well-posedness analysis for the Oberbeck-Boussinesq limiting system.
\end{itemize}
In a second instance, we would like to dedicate ourselves to the analysis of a different system that describes the evolution of temperature on the ocean surface: the so-called \emph{surface quasi-geostrophic system}. Specifically, we are interested in the asymptotic analysis in regimes of fast rotational effects and we would like to inspect the well-posedness of such system on manifolds, like the sphere. 

To conclude, we highlight that we would like to spend more time regarding the lifespan of solutions to Euler equations in order to study  more deeply the stabilization effects due to the rotation.

\chapter*{Perspectives d'avenir}\addcontentsline{toc}{chapter}{Perspectives d'avenir}
Nous concluons ce manuscrit en soulignant quelque but possible à venir.

Tout d'abord, nous continuerons les études commencées dans le chapitre \ref{chap:multi-scale_NSF} et dans le chapitre \ref{chap:BNS_gravity} vers deux directions différentes :
\begin{itemize}
\item d'une part, nous étudierons les régimes pas encore couverts dans nos études précédentes, soit $m=1$ avec les effets centrifuges soit $(m+1)/2<n<m$;
\item d'autre part, nous nous dédierons à l'analyse du caractère bien posé pour le système limite du type Oberbeck-Boussinesq.
\end{itemize}
Dans un second temps, nous nous consacrerons à l'analyse d'un autre système décrivant l'évolution de la température à la surface de l'océan, qu'on appelle le \emph{système quasi-géostrophique à la surface}. 
Plus précisément, nous nous intéressons à l'analyse asymptotique des régimes en rotation rapide et nous aimerions examiner le caractère bien posé d'un tel système sur des variétés, comme la sphère.

En conclusion, nous soulignons qu'on aimerait consacrer plus de temps à l’étude de la durée de vie des solutions des équations d'Euler afin de comprendre plus en profondeur les effets de stabilisation dûs à la rotation.


\appendix 
\counterwithout{definition}{section}
\counterwithin{definition}{chapter}

\chapter{Tools}\label{app:Tools}
The goal of this appendix is to present the tools, which have been useful in our analysis. Unless otherwise specified, we refer to Chapter 2 of \cite{B-C-D} for details.

\section{Littlewood-Paley theory: introduction} \label{app:LP}

We start by exhibiting some tools from Littlewood-Paley theory.

For simplicity of exposition, we deal with the $\R^d$ case, with $d\geq1$; however, the whole construction can be adapted also to the $d$-dimensional torus $\TT^d$, and to the ``hybrid'' case
$\R^{d_1}\times\TT^{d_2}$.

First of all, we introduce the \emph{Littlewood-Paley decomposition}. For this
we fix a smooth radial function $\chi$ such that it satisfies the following properties: 
\begin{itemize}
\item[(i)] $\Supp\chi\subset B(0,2)$;
\item[(ii)] $\chi\equiv 1$ in a neighborhood of the ball $B(0,1)$;
\item[(iii)] the map $r\mapsto\chi(r\,e)$ is non-increasing over $\R_+$ for all unitary vectors $e\in\R^d$.
\end{itemize}
Set now 
$$\varphi\left(\xi\right):=\chi\left(\xi\right)-\chi\left(2\xi\right)\quad  \text{and} \quad\vphi_j(\xi):=\vphi(2^{-j}\xi) \quad \text{for all}\quad j\geq0.$$
The dyadic blocks $(\Delta_j)_{j\in\Z}$ are defined by\footnote{We agree  that  $f(D)$ stands for 
the pseudo-differential operator $u\mapsto\mc{F}^{-1}[f(\xi)\,\what u(\xi)]$.} 
$$
\Delta_j\,:=\,0\quad\mbox{ if }\; j\leq-2,\qquad\Delta_{-1}\,:=\,\chi(D)\qquad\mbox{ and }\qquad
\Delta_j\,:=\,\varphi(2^{-j}D)\quad \mbox{ if }\;  j\geq0\,.
$$
For any $j\geq0$ fixed, we  also introduce the \emph{low frequency cut-off operator}
\begin{equation} \label{eq:S_j}
S_j\,:=\,\chi(2^{-j}D)\,=\,\sum_{k\leq j-1}\Delta_{k}\,.
\end{equation}
Note that $S_j$ is a convolution operator. More precisely, after defining
$$
K_0\,:=\,\mc F^{-1}\chi\qquad\qquad\mbox{ and }\qquad\qquad K_j(x)\,:=\,\mathcal{F}^{-1}[\chi (2^{-j}\cdot)] (x) = 2^{jd}K_0(2^j x)\,,
$$
for all $j\in\N$ and all tempered distributions $u\in\mc S'$ we have that $S_ju\,=\,K_j\,*\,u$.
Thus the $L^1$ norm of $K_j$ is independent of $j\geq0$, hence $S_j$ maps continuously $L^p$ into itself, for any $1 \leq p \leq +\infty$.

Moreover, the following property states the usefulness of such a decomposition.
\begin{lemma}
For any $u\in\mc{S}'$, then one has the equality $u=\sum_{j}\Delta_ju$ in the sense of $\mc{S}'$.
\end{lemma}
Let us also recall the so-called \emph{Bernstein inequalities}, which describe the way derivatives take effect on the spectrally localized functions.
  \begin{lemma} \label{l:bern}
Let  $0<r<R$.   A constant $C$ exists so that, for any non-negative integer $k$, any couple $(p,q)$ 
in $[1,+\infty]^2$, with  $p\leq q$,  and any function $u\in L^p$,  we  have, for all $\lambda>0$,
\begin{equation*}
\begin{split}
&(i)\quad \text{if}\quad \Supp\, \widehat u \subset   B(0,\lambda R),\quad
\text{then}\quad
\|\nabla^k u\|_{L^q}\, \leq\,
 C^{k+1}\,\lambda^{k+d\left(\frac{1}{p}-\frac{1}{q}\right)}\,\|u\|_{L^p}; \\
 &(ii)\quad \text{if}\quad{\Supp}\, \widehat u \subset \{\xi\in\R^d\,:\, \lambda r\leq|\xi|\leq \lambda R\},
\quad \text{then}\quad C^{-k-1}\,\lambda^k\|u\|_{L^p}\,
\leq\,
\|\nabla^k u\|_{L^p}\,
\leq\,
C^{k+1} \, \lambda^k\|u\|_{L^p}\,.
\end{split}
\end{equation*}
\end{lemma}   

By use of Littlewood-Paley decomposition, we can define the class of Besov spaces.
\begin{definition} \label{d:B}
  Let $s\in\R$ and $1\leq p,r\leq+\infty$. The \emph{non-homogeneous Besov space}
$B^{s}_{p,r}$ is defined as the subset of tempered distributions $u$ for which
$$
\|u\|_{B^{s}_{p,r}}\,:=\,
\left\|\left(2^{js}\,\|\Delta_ju\|_{L^p}\right)_{j\geq -1}\right\|_{\ell^r}\,<\,+\infty\,.
$$
\end{definition}

The spaces just defined have the following important topological properties (which have been broadly employed in Chapter \ref{chap:Euler}).
\begin{proposition}\label{proposition_Fatou}
Let the triplet $(s,p,r)$ be in $\R\times [1,+\infty]^2$. The set $B^s_{p,r}$ is a Banach space and satisfies the Fatou property, namely, if $(f_n)_{n\in \N}$ is a bounded sequence of $B^s_{p,r}$, then an element $f$ of $B^{s}_{p,r}$ and a subsequence $f_{\psi(n)}$ exist such that 
$$ \lim_{n\rightarrow +\infty}f_{\psi(n)}=f \, \text{ in }\mc S^\prime \quad \quad \text{and}\quad \quad \|f\|_{B^s_{p,r}}\leq C\liminf_{n\rightarrow +\infty}\|f_{\psi(n)}\|_{B^s_{p,r}}\, .$$
\end{proposition}

In addition, Besov spaces are interpolation spaces between Sobolev spaces. In fact, for any $k\in\N$ and~$p\in[1,+\infty]$
we have the chain of continuous embeddings $$ B^k_{p,1}\hookrightarrow W^{k,p}\hookrightarrow B^k_{p,\infty}.$$
In the case when $1<p<+\infty$, the previous chain of embeddings can be refined to
$$B^k_{p, \min (p, 2)}\hookrightarrow W^{k,p}\hookrightarrow B^k_{p, \max(p, 2)}\, .$$
In particular, for all $s\in\R$ we deduce that $B^s_{2,2}\equiv H^s$, with equivalence of norms:
\begin{equation} \label{eq:LP-Sob}
\|f\|_{H^s}\,\sim\,\left(\sum_{j\geq-1}2^{2 j s}\,\|\Delta_jf\|^2_{L^2}\right)^{\!\!1/2}\,.
\end{equation}
Observe that, from that equivalence, we easily get the following property: 
\begin{lemma}\label{lemma_sobolev_H^s}
For any $f\in H^s$ and any $j\in \N$, one has
\begin{equation} \label{est:sobolev}
\left\|\big(\Id-S_j\big)f\right\|_{H^\s}\,\leq\,C\,\|\nabla f\|_{H^{s-1}}\,2^{-j(s-\s)}\quad \quad \text{for all}\quad \quad \,\s\leq s\,,
\end{equation}
where $C>0$ is a ``universal'' constant, independent of $f$, $j$, $s$ and $\s$. 
\end{lemma}
\begin{proof}
We make use of the characterization \eqref{eq:LP-Sob} of $H^\s$ and we write 
\begin{equation*}
\begin{split}
\left\|(\Id-S_{j}) f\right\|_{H^{\s}}^{2}& \leq C\sum_{k\geq j}2^{2k\s}\left\|\Delta_{k}f \right\|_{L^{2}}^{2}\, 2^{2ks}\; 2^{-2ks}
\leq C \sum_{k\geq j}2^{-2k(s-\s)}\; 2^{2ks}\left\|\Delta_{k}f \right\|_{L^{2}}^{2}\\
&\leq C \sum_{k\geq j}2^{-2k(s-\s)}\; 2^{2k(s-1)}\left\|\Delta_{k}\nabla f \right\|_{L^{2}}^{2}\leq C\sum_{k\geq j}2^{-2k(s-\s)}\;\left\| \nabla f \right\|_{H^{s-1}}^{2}\\
&\leq C \; 2^{-2j(s-\s)}\|\nabla f\|^2_{H^{s-1}}\, ,
\end{split}
\end{equation*}
for all $\s\leq s$. 
Then, we obtain \eqref{est:sobolev}, concluding the proof of the lemma.
\qed
\end{proof}

As an immediate consequence of the first Bernstein inequality (see Lemma \ref{l:bern}), one gets the following embedding result, which generalises the Sobolev embeddings.
\begin{proposition}\label{p:embed}
The space $B^{s_1}_{p_1,r_1}$ is continuously embedded in the space $B^{s_2}_{p_2,r_2}$ for all indices satisfying $p_1\,\leq\,p_2$ and either
$s_2\,<\,s_1-d\big(1/p_1-1/p_2\big)$, or $s_2\,=\,s_1-d\big(1/p_1-1/p_2\big)$ and $r_1\leq r_2$.
\end{proposition}
In particular, we get the following chain of continuous embeddings:
$$ B^s_{p,r}\hookrightarrow B^{s-d/p}_{\infty,r}\hookrightarrow B^0_{\infty,1}\hookrightarrow L^\infty \, , $$
whenever the triplet $(s,p,r)\in \R\times [1,+\infty]^2$ satisfies 
\begin{equation*}
s>\frac{d}{p} \quad \quad \quad \text{or}\quad \quad \quad s=\frac{d}{p} \quad \text{and}\quad r=1\, .
\end{equation*}
This last chain of embeddings has been fundamental and of constant use in Chapter \ref{chap:Euler}. 



\medskip

In the sequel, we recall some definitions and properties of paradifferential calculus. Moreover, using those notions, we will focus on the study of transport equations in Besov spaces.

\section{Paradifferential calculus}\label{app_paradiff}
Let us introduce the Bony decomposition (see \cite{Bony}). 
Formally, the product of two tempered distributions $u$ and $v$ can be decomposed into
$$ u\,v=\mathcal{T}_uv+\mathcal{T}_vu+\mathcal{R}(u,v)\, , $$
where the \textit{paraproduct} $\mathcal{T}$ and the \textit{remainder} $\mathcal{R}$ are defined as follows: 
$$ \mathcal{T}_uv=\sum_jS_{j-1}u\, \Delta_jv \quad \text{and}\quad  \mathcal{R} (u,v):= \sum_j\sum_{|k-j|\leq 1}\Delta_ju\, \Delta_kv\, . $$
The paraproduct and remainder operators have nice continuity properties. The following ones have been of constant throughout the manuscript.
\begin{proposition}\label{T-R}
For any $(s,p,r)\in \R \times [1, +\infty ]^2$ and $t>0$, the paraproduct operator $\mathcal{T}$ maps continuously $L^\infty \times B^s_{p,r}$ into $B^s_{p,r}$ and $B^{-t}_{\infty,\infty}\times B^s_{p,r}$ into $B^{s-t}_{p,r}$. Moreover, we have the following estimates
\begin{equation*}
\|\mathcal{T}_u v\|_{B^s_{p,r}}\leq C\|u\|_{L^\infty} \|\nabla v\|_{B^{s-1}_{p,r}} \quad \text{and}\quad \|\mathcal{T}_u v\|_{B^{s-t}_{p,r}}\leq C\|u\|_{B^{-t}_{\infty,\infty}} \|\nabla v\|_{B^{s-1}_{p,r}}\, .
\end{equation*} 
For any $(s_1,p_1,r_1)$ and $(s_2,p_2,r_2)$ in $\R \times [1, +\infty ]^2$ such that $s_1+s_2>0$, $1/p:=1/p_1+1/p_2\leq 1$ and $1/r:=1/r_1+1/r_2\leq 1$, the remainder operator $\mathcal{R}$ maps continuously $B^{s_1}_{p_1,r_1}\times B^{s_2}_{p_2,r_2}$ into $B^{s_1+s_2}_{p,r}$. In the case when $s_1+s_2=0$ and $1/r_1+1/r_2=1$, the operator $\mathcal{R}$ is continuous from $B^{s_1}_{p_1,r_1}\times B^{s_2}_{p_2,r_2}$ to $B^{0}_{p,\infty}$.  
\end{proposition}

As a consequence of the Proposition \ref{T-R}, the spaces $B^s_{p ,r}$ are Banach algebras, provided that condition 
\begin{equation}\label{cond_algebra}
s>\frac{d}{p} \quad \quad \quad \text{or}\quad \quad \quad s=\frac{d}{p} \quad \text{and}\quad r=1
\end{equation}
holds for $s>0$ and $(p,q)\in [1,+\infty]^2$. Moreover, in that case, we have the so-called \textit{tame estimates}. 

\begin{corollary}\label{cor:tame_est}
Let $(s,p,r)\in\; ]0,+\infty[\; \times [1,+\infty ]^2$ satisfy \eqref{cond_algebra}. Then, for every $f,g\in L^\infty \cap B^s_{p ,r}$, one has 
\begin{equation*}
\|fg\|_{B^s_{p ,r}}\leq C \left(\|f\|_{L^\infty}\|g\|_{B^s_{p ,r}}+\|f\|_{B^s_{p ,r}}\|g\|_{L^\infty}\right)\, .
\end{equation*}
\end{corollary}
\begin{remark}
The space $B^0_{\infty ,1}$ is \textit{not} an algebra. If $f,g \in B^0_{\infty ,1}$, applying Proposition \ref{T-R}, one can bound the paraproducts $\mathcal{T}_fg$ and $\mathcal{T}_gf$ but \textit{not} the remainder $\mathcal{R} (f,g)$.
\end{remark}

To end this paragraph, we present a fine estimate for products in which one of the two functions is only bounded in $L^\infty$ but its gradient belongs to the Besov space $B^{s-1}_{p,r}$. 
\begin{proposition}\label{prop:app_fine_tame_est}
Let $(s,p,r)\in\; ]0,+\infty[\; \times \, [1,+\infty]^2$ satisfy condition \eqref{cond_algebra}. Assume that $g\in L^\infty \cap B^{s}_{p,r}$ and $f$ is a bounded function such that $\nabla f\in B^{s-1}_{p,r}$. Then, the product $fg$ belongs to $L^\infty \cap B^s_{p,r}$ and one has the following estimate:
\begin{equation*}
\|fg\|_{B^s_{p,r}}\leq C\left(\|f\|_{L^\infty}\|g\|_{B^s_{p,r}}+\|\nabla f\|_{B^{s-1}_{p,r}}\|g\|_{L^\infty}\right)\, .
\end{equation*}
\end{proposition}
\begin{proof}
Taking advantage of Bony decomposition, one can write 
\begin{equation*}
fg =\mathcal{T}_{f}g +\mathcal{T}_{g }f + \mathcal{R}(f, g )
\end{equation*}
and employing Proposition \ref{T-R}, we get
\begin{equation*}
\begin{split}
\|\mathcal{T}_{f}g\|_{B^s_{p,r}}&\leq C\|f\|_{L^\infty}\|g \|_{B^s_{p,r}}\\
\|\mathcal{T}_{g }\, f\|_{B^s_{p,r}}&\leq C\|g \|_{L^\infty}\|\nabla f\|_{B^{s-1}_{p,r}}\\
\|\mathcal{R}(f ,g)\|_{B^s_{p,r}}&\leq C\|f\|_{B^0_{\infty,\infty}}\|g \|_{B^s_{p,r}}\leq C\|f\|_{L^\infty}\|g  \|_{B^s_{p,r}}\, .
\end{split}
\end{equation*}
This completes the proof of the proposition.
\qed
\end{proof}

\section{Commutator estimates}\label{app:comm_est}
In this paragraph, we recall the main commutator estimates widely employed throughout the chapter \ref{chap:Euler}. 
\begin{definition}
We say that the triplet $(s,p,r)\in \R \times [1, +\infty ]^2$ satisfies the Lipschitz condition if 
\begin{equation} \label{Lip_cond}
s>1+d/p \quad \text{and}\quad r\in [1,+\infty ]\quad \quad \quad \text{or}\quad \quad \quad s=1+d/p \quad \text{and}\quad r=1\, .
\end{equation}
\end{definition}

\smallskip

The proof of the following Lemma \ref{l:commutator_pressure} can be found in \cite{D_JDE} by Danchin. 

\begin{lemma}\label{l:commutator_pressure}
Let $(s,p,r)\in \R \times [1, +\infty ]^2$ satisfy condition \eqref{Lip_cond} and $\sigma$ be in $]-1,\, s-1]$. Assume that $w\in B^{\sigma}_{p,r}$ and $A$ is a bounded function on $\R^d$ such that $\nabla A\in B^{s-1}_{p,r}$. Then, there exists a constant $C=C(s,p,r,\sigma ,d)$ such that for all $i\in \{1, \dots, d\}$, we have: 
\begin{equation*}
\|\d_i[A, \Delta_j]w\|_{L^p}\leq C\, c_j \, 2^{-j\sigma}\|\nabla A\|_{B^{s-1}_{p,r}}\|w\|_{B^{\sigma}_{p,r}}\quad \text{for all }j\geq -1\, ,
\end{equation*}
with $\|(c_j)_{j\geq -1}\|_{\ell^r}=1$.
\end{lemma}

The next statement concerns a standard commutator estimate between the transport operator and the frequency localisation operator.  
\begin{lemma}\label{l:commutator_est}
Assume that $v\in B^s_{p ,r}$ with $(s,p,r)$ satisfying the Lipschitz condition \eqref{Lip_cond}. Denote by $[v\cdot \nabla ,\Delta_j]f=(v\cdot \nabla)\Delta_j f-\Delta_j(v\cdot \nabla) f$ the commutator between the transport operator $v\cdot \nabla$ and the frequency localisation operator $\Delta_j$. Then, for every $f\in B^{s}_{p,r}$, 
\begin{equation*}
\left\|\left(2^{js}\|[v\cdot \nabla ,\Delta_j]f\|_{L^p}\right)_j\right\|_{\ell^r}\leq C\left(\|\nabla v\|_{L^\infty}\|f\|_{B^{s}_{p ,r}}+\|\nabla v\|_{B^{s-1}_{p ,r}}\|\nabla f\|_{L^\infty}\right)
\end{equation*}
and also, for every $f\in B^{s-1}_{p ,r}$, 
\begin{equation*}
\left\|\left(2^{j(s-1)}\|[v\cdot \nabla ,\Delta_j]f\|_{L^p}\right)_j\right\|_{\ell^r}\leq C\left(\|\nabla v\|_{L^\infty}\|f\|_{B^{s-1}_{p ,r}}+\|\nabla v\|_{B^{s-1}_{p ,r}}\| f\|_{L^\infty}\right)\, ,
\end{equation*}
for some constant $C=C(s,p,d)>0$.
\end{lemma}

Finally, the next result deals with commutators between paraproduct operators and Fourier multipliers.
\begin{lemma}
Let $\kappa$ be a smooth function on $\R^d$, which is homogeneous of degree $m$ away from a neighborhood of $0$. Take $(s,p,r)\in \R \times [1,+\infty]^2$ and $v$ a vector field such that $\nabla v\in L^\infty$. Then, for every $f\in B^s_{p ,r}$, one has
\begin{equation*}
\left\|[T_v,\kappa (D)]f\right\|_{B^{s-m+1}_{p ,r}}\leq C(s,d)\, \|\nabla v\|_{L^\infty}\|f\|_{B^s_{p ,r}}\, .
\end{equation*}
\end{lemma}

\section{Transport equations}
In this paragraph, we deal with the transport equations in non-homogeneous Besov spaces. We refer to Chapter 3 of \cite{B-C-D} for additional details. We study, in $\R_+\times \R^d$, the initial value problem 
\begin{equation}\label{general_transport}
\begin{cases}
\d_t f+v\cdot \nabla f =g\\
f_{|t=0}=f_0\, .
\end{cases}
\end{equation}
We always assume the velocity field $v=v(t,x)$ to be a Lipschitz divergence-free function. 
In the case when the Lipschitz condition \eqref{Lip_cond} is satisfied, we have the embedding $B^s_{p ,r}\hookrightarrow W^{1,\infty}$.

We state now the main well-posedness result concerning problem \eqref{general_transport} in Besov spaces. We point out also that the notation $C^0_w([0,T]; X)$, with $X$ a Banach space, refers to the space of functions which are continuous in time with values in $X$ endowed with its weak topology.

\begin{theorem}\label{thm_transport}
Let $(s,p,r)\in \R\times [1,+\infty ]^2$ satisfy the Lipschitz condition \eqref{Lip_cond}. Given $T>0$, take $g\in L^1_T(B^s_{p ,r})$. Assume that $v\in L^1_T(B^s_{p ,r})$ and that there exist two real numbers $q>1$ and $\sigma >0$ such that $v\in L^q_T(B^{-\sigma}_{\infty ,\infty})$. Finally, let $f_0\in B^s_{p ,r}$ be the initial datum. Then, the transport equation \eqref{general_transport} has a unique solution $f$ in:
\begin{itemize}
\item the space $C^0\left([0,T];B^s_{p ,r}\right)$, if $r<+\infty$;
\item the space $\left(\displaystyle \bigcap_{s^\prime <s} C^0\left([0,T];B^{s^\prime}_{p ,\infty}\right)\right) \cap C^0_w\left([0,T];B^s_{p ,\infty}\right)$, if $r=+\infty$.
\end{itemize}
Moreover, the unique solution satisfies the following estimate:
\begin{equation*}
\|f\|_{L^\infty_T(B^s_{p ,r})}\leq \exp \left(C\int_0^T\|\nabla v\|_{B^{s-1}_{p ,r}}\, \detau \right)\left(\|f_0 \|_{B^{s}_{p ,r}}+\int_0^T\exp \left(-C\int_0^t\|\nabla v\|_{B^{s-1}_{p ,r}}\, \detau \right)\|g\|_{B^{s}_{p ,r}}\, \dt\right),
\end{equation*}
for some constant $C=C(s,p,r,d)>0$.
\end{theorem}


To conclude this paragraph, we show a refinement of Theorem \ref{thm_transport}, proved by Vishik in \cite{Vis} and in a different way by Hmidi and Keraani (see \cite{H-K}). It states that, if $\div v=0$ and the Besov regularity index is $s=0$, the estimate in Theorem \ref{thm_transport} can be replaced by an inequality which is \textit{linear} with respect to $\|\nabla v\|_{L^1_T(L^\infty)}$.

\begin{theorem}\label{thm:improved_est_transport}
Given $T>0$, assume that $v$ is a  divergence-free vector field such that $\nabla v\in L^1_T(L^\infty )$ and let $g\in L^1_T(B^0_{\infty ,r})$. Take $r\in [1,+\infty ]$ and $f_0\in L^1_T(B^0_{\infty ,r})$. Then, there exists a constant $C=C(d)$ such that, for any solution $f$ to problem \eqref{general_transport} in
$C^0([0,T]; B^0_{\infty ,r})$ (or with the usual modification of $C^0$ into $C^0_w$ if $r=+\infty$), we have 
\begin{equation*}
\|f\|_{L^\infty_T(B^0_{\infty ,r})}\leq C\left(\|f_0\|_{B^0_{\infty ,r}}+\|g\|_{L^1_T(B^0_{\infty ,r})}\right)\left(1+\int_0^T\|\nabla v(\tau )\|_{L^\infty}\, \detau \right)\, .
\end{equation*}
\end{theorem}

\chapter{Some well-known results}\label{appendixA}

This appendix is thought for the reader's convenience in order to quickly check the statement of some famous theorems. For that reason, we limit ourself to state such results, used throughout the whole thesis. Unless indicated otherwise, we refer to the introductory part of book \cite{F-N} for details (see also Chapter 10) and to Chapter 1 of \cite{B-C-D}.

\section{Embedding theorems}
We start by recalling the \emph{Rellich-Kondrachov embedding theorem}. 
\begin{theorem}\label{app:sob_embedd_thm}
Let $Q\subset \R^d$ be a bounded Lipschitz domain.
\begin{itemize}
\item[(i)] If $kp<d$ and $p\geq 1$, then the space $W^{k,p}(Q)$ is continuously embedded in $L^q(Q)$ for any $1\leq q\leq p^\ast:=(dp)/(d-kp)$. \\
 Moreover, the embedding is compact if $k>0$ and $q<p^\ast$. 
\item[(ii)] If $kp=d$, the space $W^{k,p}(Q)$ is compactly embedded in $L^q(Q)$ for any $q\in [1,+\infty)$.
\item[(iii)] If $kp>d$, then $W^{k,p}(Q)$ is continuously embedded in\footnote{The symbol $\lfloor z \rfloor$ denotes the integer part of $z$.} $\mc C^{k-\lfloor d/p \rfloor-1,\, \nu}(\oline Q)$, where either $\nu=\lfloor d/p \rfloor +1-d/p$ if $(d/p)\not \in \Z$ or $\nu$ is an arbitrary positive number in $(0,1)$ if $(d/p)\in \Z$. \\
Moreover, the embedding is compact if $0<\nu < \lfloor d/p \rfloor +1 -d/p$. 
\end{itemize}
\end{theorem}
As a direct consequence of the previous theorem we have the \emph{embedding theorem for dual Sobolev spaces}.
\begin{theorem}\label{app:thm_dual_Sob}
Let $Q\subset \R^d$ be a bounded domain. Let $k>0$ and $q<+\infty$ satisfy 
\begin{itemize}
\item $q>p^\ast/(p^\ast -1)$ where $p^\ast:=(dp)/(d-kp)$, if $kp<d$;
\item $q>1$, if $kp=d$;
\item $q\geq 1$, if $kp>d$.
\end{itemize}
Then, the space $L^q(Q)$ is compactly embedded into the space $W^{-k,p^\prime}(Q)$, with $1/p+1/p^\prime=1$.
\end{theorem}
\section{Mollifiers}\label{section:molli}
Here we recall the main properties of the mollifiers. Such properties are widely employed in Section \ref{s:proof-1} of Chapter \ref{chap:multi-scale_NSF}.

We start by fixing a smooth, radial and radially decreasing function $\chi \in{C}^\infty_c(\mbb{R}^d)$, such that 
$$0\leq\chi\leq1, \quad \quad \chi(x)=0 \; \text{ for }|x|\geq1 \quad \quad \text{and}\quad \quad \int_{\R^d}\chi(x)\dx=1.$$
Next, we define the \emph{mollifying kernel} $\big(\chi_\veps\big)_{\veps>0}$ by the formula
$$
\chi_\veps(x)\,:=\,\frac{1}{\veps^d}\,\,\chi\!\left(\frac{x}{\veps}\right)\qquad\qquad\text{for any}\; \veps>0\; \text{ and }\;x\in\R^d\,.
$$
Then, for any tempered distribution $\mf S$ on $\R^d$ and any $\veps>0$, we introduce
$$
\mf S_{\veps}\,:=\,\chi_{\veps}\,*\,\mf S\,.
$$
\begin{theorem}\label{app:thm_mollifiers}
Let $X$ a Banach space. If $\mf S\in L^1_{\rm loc}(\R^d;X)$, then we have $\mf S_\veps \in C^\infty(\R^d; X)$. In addition, we have the following properties: 
\begin{itemize}
\item[(i)] if $\mf S\in L^p_{\rm loc}(\R^d;X)$ with $1\leq p<+\infty$, then $\mf S_\veps \in L^p_{\rm loc}(\R^d;X)$ and 
$$ \mf S_\veps \longrightarrow \mf S \quad \text{in}\quad L^p_{\rm loc}(\R^d;X) \quad \text{as}\quad \veps\rightarrow 0^+\, ; $$
\item[(ii)] if $\mf S\in L^p(\R^d;X)$ with $1\leq p<+\infty$, then $\mf S_\veps \in L^p(\R^d;X)$, 
$$ \| \mf S_\veps \|_{L^p(\R^d;X)}\leq \|\mf S\|_{L^p(\R^d;X)} $$
and 
$$ \mf S_\veps \longrightarrow \mf S \quad \text{in}\quad L^p(\R^d;X) \quad \text{as}\quad \veps\rightarrow 0^+\, ; $$
\item[(iii)] if $\mf S\in L^\infty(\R^d;X)$, then $\mf S_\veps \in L^\infty(\R^d;X)$ and 
$$ \| \mf S_\veps \|_{L^\infty(\R^d;X)}\leq \|\mf S\|_{L^\infty(\R^d;X)}\, ; $$
\item[(iv)] if $\mf S \in C^k(Q; X)$ where $k$ is a non-negative integer and $Q\subset \R^d$ is a ball, then $(\partial^\alpha \mf S)_\veps (x)=\partial^\alpha \mf S_\veps(x)$ for all $x\in Q$, $\veps \in (0, {\rm dist}[x, \partial Q])$ and for any multi-index $\alpha$ such that $|\alpha |\leq k$. Moreover, 
$$ \|\mf S_\veps\|_{C^k(\oline B;X)}\leq \|\mf S\|_{C^k(\oline V;X)} $$
for any $\veps \in (0, {\rm dist}[\partial B,\partial V])$, where $B$ and $V$ are balls in $\R^d$ such that $\oline B\subset V\subset \oline V\subset Q$. Finally, 
$$ \mf S_\veps \longrightarrow \mf S \quad \text{in}\quad C^k(\oline B;X) \quad \text{as}\quad \veps\rightarrow 0^+\, . $$
\end{itemize}
\end{theorem}
\section{Some assorted inequalities}\label{sec:assorted_ineq}
We start this section by mentioning some classical inequalities: \emph{Young's inequality} (see Chapter 4 of \cite{Brezis} in this respect), \emph{H\"older's inequality} and \emph{Minkowski's inequality}.
\begin{proposition}[Young's inequality]\label{app:young_ineq}
If $ a\geq 0$ and $ b\geq 0$ are non-negative real numbers and if $p>1$ and $q>1$ are real numbers such that 
$$\frac {1}{p}+\frac {1}{q}=1\, ,$$
then\footnote{It is sometimes convenient to use the formulation $ab \leq \veps a^p+ \veps^{-1/(p-1)}b^q$, with $\veps>0$.}
$$ ab\leq \frac {a^{p}}{p}+\frac {b^{q}}{q}\, .$$
\end{proposition} 
\begin{proposition}[H\"older's inequality]\label{app:hold_ineq}
Let $(X, \mu)$ be a measure space and $(p,q,r)$ in $[1,+\infty]^3$ be such that 
$$ \frac{1}{p}+\frac{1}{q}=\frac{1}{r}\, . $$
If the couple $(f,g)$ belongs to $L^p(X,\mu)\times L^q(X,\mu)$, then the product $fg$ belongs to $L^r(X,\mu)$ and 
$$ \|fg\|_{L^r}\leq \|f\|_{L^p}\|g\|_{L^q}\, . $$
\end{proposition}
\begin{proposition}[Minkowski's inequality]\label{app:mink} 
Let $(X_1, \mu_1)$ and $(X_2, \mu_2)$ be two measure spaces and $f$ a non-negative measurable function over $X_1 \times X_2$. For all $1\leq p\leq q\leq +\infty$, we have 
$$ \left\| \|f(\cdot ,x_2)\|_{L^p(X_1,\mu_1)}\right\|_{L^q(X_2,\mu_2)}\leq \left\| \|f(x_1 ,\cdot)\|_{L^q(X_2,\mu_2)}\right\|_{L^p(X_1,\mu_1)}\, .  $$
\end{proposition}
The next two results are fundamental to give uniform bounds in Sobolev norms for the velocity fields and the temperatures. Such propositions are the so-called \emph{generalized Poincar\'e inequality} and \emph{generalized Korn-Poincar\'e inequality} respectively.
\begin{proposition}\label{app:poincare_prop}
Let $1\leq p\leq +\infty$, $0<r<+\infty$, $V_0>0$ and let $Q\subset \R^d$ be a bounded Lipschitz domain. Then, there exists a positive constant $c=c(p,r,V_0)$ such that
$$ \| v\|_{W^{1,p}(Q)}\leq c\left[\|\nabla  v\|_{L^{p}(Q;\, \R^{d})}+\left(\int_V | v|^r\dx \right)^{1/r}\right] \, ,$$
for any measurable $V\subset Q$ with $|V|\geq V_0$ and any $ v\in W^{1,p}(Q)$.
\end{proposition}
\begin{proposition} \label{app:korn-poincare_prop}
Let $Q\subset \R^d$, with $d>2$, be a bounded Lipschitz domain, and let $1<p<+\infty$, $M_0>0$, $K>0$, $\alpha>1$. Then, there exists a positive constant $c=c(p,M_0,K,\alpha)$ such that the inequality 
$$ \| v\|_{W^{1,p}(Q;\R^d)}\leq c\left(\|\nabla  v+\, ^t\nabla  v-\frac{2}{d}\, \div  v \, \Id\|_{L^p(Q;\R^d)}+\int_Q f| v| \dx\right) $$
holds for any $ v\in W^{1,p}(Q;\R^d)$ and any non-negative function $f$ such that 
$$ 0<M_0\leq\int_Q f\dx \quad \text{and}\quad \int_Q f^\alpha \dx\leq K \; \; \text{for a certain }\alpha>1\, . $$
\end{proposition}
We conclude this paragraph, by stating a classical result in the theory of Sobolev spaces that relates the $L^{p}$ norms of the weak derivatives of a function: the so-called \textit{Gagliardo–Nirenberg inequality} (see \cite{Nir} for details of the proof).
\begin{theorem}\label{app:thm_G-N}
Let $f$ be a function belonging to $L^q(\R^d)$ and its derivatives of order $m$, denoted by $D^m f$, belong to $L^r(\R^d)$ with $1\leq q,r\leq +\infty$. For the derivatives $D^jf$, where $0\leq j<m$, the following inequality holds
\begin{equation}\label{ineq:G-N}
\|D^j f\|_{L^p}\leq C \|D^m f\|_{L^r}^{\alpha}\|f\|_{L^q}^{1-\alpha}\, ,
\end{equation}
where
$$ \frac{1}{p}=\frac{j}{d}+\alpha\left(\frac{1}{r}-\frac{m}{d}\right)+(1-\alpha )\frac{1}{q} \, ,$$
for all $\alpha$ in the interval 
$$ \frac{j}{m}\leq \alpha \leq 1 \, , $$
with the following exceptional cases:
\begin{itemize}
\item[(i)] if $j=0$, $rm< d$ and $q=+\infty$, then we make the additional assumption that either $f$ tends to zero at infinity or $f\in L^b(\R^d)$ for some $b>0$;
\item[(ii)] if $1<r<+\infty$ and $m-j-d/r$ is a non-negative integer, then \eqref{ineq:G-N} holds only for $\alpha$ satisfying $j/m\leq \alpha<1$.
\end{itemize}
\end{theorem}
\section{Elliptic estimates}
This part of the appendix is devoted to a classical result pertaining $L^2$ solutions of the following elliptic equation:
\begin{equation}\label{app:eq_elliptic-est}
-\div (A \nabla \Pi)=\div F \quad \text{in}\quad \R^d\, , 
\end{equation}
where $A=A(x)$ is a given suitably smooth bounded function satisfying 
\begin{equation}\label{bound-A_ellip}
 A_\ast :=\inf_{x\in \R^d} A(x)>0\, . 
\end{equation}
\begin{lemma}\label{lem:Lax-Milgram_type}
For all vector fields $F$ in $L^2(\R^d)$, there exists a tempered distribution $\Pi$ (unique up to constant functions) such that $\nabla \Pi \in L^2(\R^d)$, and equation \eqref{app:eq_elliptic-est} is satisfied. Moreover, we have 
\begin{equation}\label{app:est_Pi_L^2}
A_\ast \|\nabla \Pi\|_{L^2}\leq \|F\|_{L^2}\, . 
\end{equation} 
\end{lemma}
We refer to \cite{D_JDE} for details of the proof. 

Let us now state higher order estimates for the pressure term in Besov spaces (we refer again to \cite{D_JDE} for details). 
\begin{proposition}\label{app:prop7_danchin}
Let $1<p<+\infty$ and $1\leq r\leq +\infty$. Let $A$ be a bounded function satisfying \eqref{bound-A_ellip} such that $\nabla A\in B^{s-1}_{p,r}$ for some $s>1+d/p$ or $s\geq 1+d/p$ if $r=1$. 
\begin{itemize}
\item[(i)] If $\sigma \in \; ]1,s]$ and $\nabla \Pi \in B^{\sigma}_{p,r}$ satisfies \eqref{app:eq_elliptic-est} for some function $F$ such that $\div F\in B^{\sigma -1}_{p,r}$, then we have for some constant $C=C(s, \sigma, p, d)$ that 
$$ A_\ast\|\nabla \Pi\|_{B^{\sigma}_{p,r}}\leq C\left( \|\div F\|_{B^{\sigma-1}_{p,r}}+A_\ast (1+A_\ast^{-1}\|\nabla A\|_{B^{s-1}_{p,r}})^\sigma \|\nabla \Pi\|_{L^p}\right). $$
\item[(ii)] If $2\leq p<+\infty$ and $F$ is in $L^2$ and satisfies $\div F\in B^{\sigma-1}_{p,r}$ for some $\sigma \in \; ]1+d/p-d/2,\, s]$, then \eqref{app:eq_elliptic-est} has a unique solution $\Pi$ (up to constant functions) such that $\nabla \Pi \in L^2 \cap B^{\sigma}_{p,r}$. Furthermore, inequality \eqref{app:est_Pi_L^2} is satisfied and there exist a positive exponent $\gamma =\gamma (\sigma,p,d)$ and a positive constant $C=C(s,\sigma,p,d)$ such that 
$$ A_\ast\|\nabla \Pi\|_{B^{\sigma}_{p,r}}\leq C\left( \|\div F\|_{B^{\sigma-1}_{p,r}}+ (1+A_\ast^{-1}\|\nabla A\|_{B^{s-1}_{p,r}})^\gamma \|\nabla \Pi\|_{L^2}\right). $$
\item[(iii)]If $\sigma >1$ and $1<p<+\infty$, then the following inequality holds:
$$ A_\ast\|\nabla \Pi\|_{B^{\sigma}_{p,r}}\leq C\left( \|\div F\|_{B^{\sigma-1}_{p,r}}+ \|\nabla A\|_{L^\infty}\|\nabla \Pi\|_{B^{\sigma -1}_{p,r}}+ \|\nabla \Pi\|_{L^\infty}\|\nabla A\|_{B^{\sigma -1}_{p,r}}\right). $$
\end{itemize}
\end{proposition}

\section{Compactness theorems}
Next, a ``milestone'' in our analysis is the compactness of special quantities in order to perform the asymptotic process. In this sense, the \emph{Ascoli-Arzel\`a} and \emph{Aubin-Lions} theorems come in our ``rescue''.

\begin{theorem}[Ascoli-Arzel\`a Theorem] \label{app:th_A-A}
Let $Q\subset \R^d$ be compact and $X$ be a compact topological metric space endowed with a metric $d_X$. Let $(v_k)_k$ be a sequence of functions in $\mc C^0(Q;X)$ that is equi-continuous, i.e. for any $\veps >0$ there exists $\delta >0$ such that 
$$d_X\Big(v_k(y),\, v_k(z)\Big)\leq \veps \, , $$
provided $|y-z|<\delta$ independently of $k$. 

Then, $(v_k)_k$ is precompact in $\mc C^0(Q;X)$, meaning that there exists a subsequence (not relabelled) and a function $v\in \mc C^0(Q;X)$ such that 
$$ \sup_{y\in Q}d_X\Big(v_k(y),\, v(y)\Big)\longrightarrow 0 \quad \text{as}\quad k\rightarrow +\infty \, . $$ 
\end{theorem}
The proof of the next statement can be found in Section 5.3 of Chapter 2 of \cite{B-P} (see also Section 5 of Chapter 1 in \cite{Lions_1969}).
\begin{theorem}[Aubin-Lions Theorem]\label{app:Aubin_thm}
Let $X_0 \subset X_1 \subset X_2$ be three Banach spaces. We assume that the embedding of $X_1$ in $X_2$ is continuous and that the embedding of $X_0$ in $X_1$ is compact. Let $(p,r)$ such that $1\leq p,r\leq +\infty$. For $T>0$, we define
$$ E_{p,r}:=\left\{v\in L^p(0,T;X_0) \quad \text{and}\quad \frac{dv}{dt}\in L^r(0,T;X_2)\right\}\, .  $$
Then,
\begin{itemize}
\item[(i)] if $p<+\infty$, the embedding of $E_{p,r}$ in $L^p(0,T;X_1)$ is compact;
\item[(ii)] if $p=+\infty$ and $r>1$, the embedding of $E_{p,r}$ in $\mc C^0(0,T; X_1)$ is compact.
\end{itemize}

\end{theorem}
Now, we present the celebrated \emph{Div-Curl Lemma} which,  roughly speaking, ensures that the product of two functions weakly converge, if each function weakly converges and in addition one has information on the $\div$ of one and the $\curl$ of the other. 
\begin{theorem}\label{app:div-curl_lem}
Let $Q\subset \R^d$ be an open set. Assume that
\begin{equation*}
\begin{split}
 u_n\weak  u \quad &\text{weakly in }L^p(Q;\R^d)\\
 v_n\weak v \quad  &\text{weakly in }L^q(Q;\R^d)
\end{split}
\end{equation*}
where
$$ \frac{1}{p}+\frac{1}{q}=\frac{1}{r}<1\, . $$
In addition, let $\div  u_n$ be precompact in $W^{-1,s}(Q)$ and $\curl  v_n$ be precompact in $W^{-1,s}(Q;\R^{d\times d})$, for a certain $s>1$. Then, 
$$  u_n\cdot v_n\weak  u \cdot  v \quad \text{weakly in }L^r(Q)\, . $$
\end{theorem}
\section{Gr\"onwall's lemma}\label{app:sect_gron}
To conclude this appendix, we mention a \emph{Gr\"onwall estimate} which has been of constant use in Chapter \ref{chap:Euler}.
\begin{lemma}\label{app:lem_gron}
Let $a\in L^1(0,T)$, $a\geq 0$, $\beta \in L^1(0,T)$, $b_0\in \R$ and 
$$b(\tau )=b_0+\int_0^{\tau}\beta (t)\dt $$
be given. Let $f\in L^\infty(0,T)$ satisfy 
$$ f(\tau )\leq b(\tau)+\int_0^{\tau}a(t)\, f(t) \dt \quad \text{ for a.a. }\tau \in [0,T]. $$
Then, one has 
$$ f(\tau)\leq b_0 \exp\left(\int_0^{\tau}a(t)\dt \right)+\int_0^{\tau}\beta (t)\exp \left(\int_t^{\tau}a(s) \ds \right)\dt $$
for a.a. $\tau \in [0,T]$.
\end{lemma}



\newpage
\null
\thispagestyle{empty}
\newpage


\addcontentsline{toc}{chapter}{\bibname}
{\small
\begin{thebibliography}{xxx}

\bibitem{Ali} S. Alinhac:
{\it ``Hyperbolic partial differential equations''}.
Universitext, Springer, Dordrecht (2009).


\bibitem{B-M-N_EMF} A. Babin, A. Mahalov, B. Nicolaenko: {\it Global splitting, integrability and regularity of 3D Euler and
Navier-Stokes equations for uniformly rotating fluids}. European J. Mech. B/Fluids, {\bf 15} (1996), n. 3,
291-300.

\bibitem{B-M-N_AA} A. Babin, A. Mahalov, B. Nicolaenko: {\it Regularity and integrability of 3D Euler and Navier-Stokes
equations for rotating fluids}. Asymptot. Anal., {\bf 15} (1997), n. 2, 103-150.




\bibitem{B-L-S} H. Bae, W. Lee, J. Shin: {\it A blow-up criterion for the inhomogeneous incompressible Euler equations}. Nonlinear Anal., {\bf 196} (2020), 111774.


\bibitem{B-C-D} H. Bahouri, J.-Y. Chemin, R. Danchin: 
{\it ``Fourier Analysis and Nonlinear Partial Differential Equations''}.
Grundlehren der Mathematischen Wissenschaften (Fundamental Principles of Mathematical Sciences),
{\bf 343}, Springer, Heidelberg (2011).

\bibitem{B-K-M} J. Beale, T. Kato, A. Majda: 
{\it Remarks on the breakdown of smooth solutions for the 3-D Euler
equations}.
Comm. Math. Phys., 
{\bf 94} (1984), n. 1, 61-66.

\bibitem{B-F-P} E. Bocchi, F. Fanelli, C. Prange:
{\it Anisotropy and stratification effects in the dynamics of fast rotating compressible fluids}.
Ann. Inst. H. Poincar\'e Anal. Non Lin\'eaire. To appear (2021).


\bibitem{Bony} J.-M. Bony: 
{\it Calcul symbolique et propagation des singularit\'es pour les \'equations aux d\'eriv\'ees partielles non lin\'eaires}.
Ann. Sci. \'Ecole Norm. Sup., \textbf{14} (1981), n. 2, 209-246.

\bibitem{B-P} F. Boyer, P. Fabrie:
{\it ``Mathematical tools for the study of the incompressible Navier-Stokes equations and related models''}.
Applied Mathematical Sciences, {\bf 183}, Springer, New York, (2013). 


\bibitem{B-D_JMPA} D. Bresch, B. Desjardins:
{\it On the existence of global weak solutions to the Navier-Stokes equations for viscous compressible and heat conducting fluids}. 
J. Math. Pures Appl. (9), {\bf 87} (2007), n. 1, 57-90.

\bibitem{B-D-GV} D. Bresch, B. Desjardins, D. G\'erard-Varet:
{\it Rotating fluids in a cylinder}.
Discrete Contin. Dyn. Syst., {\bf 11} (2004), n. 1, 47-82.

\bibitem{BGL} D. Bresch, M. Gisclon, C.-K. Lin:
{\it An example of low Mach (Froude) number effects for compressible flows with nonconstant density (height) limit}.
Math. Model. and Num. Anal., {\bf 39} (2005), n. 3, 477-486.

\bibitem{Brezis} H. Brezis:
{\it ``Functional analysis, Sobolev spaces and partial differential equations''}.
Universitext, Springer, New York (2011).

\bibitem{Cha} F. Charve: 
{\it Global well-posedness and asymptotics for a geophysical fluid system}. Comm. Partial Differential Equations, {\bf 29} (2004), n. 11-12, 1919-1940.



\bibitem{C-D-G-G} J.-Y. Chemin, B. Desjardins, I. Gallagher, E. Grenier:
{\it ``Mathematical geophysics. An introduction to rotating fluids and the Navier-Stokes equations''}.
Oxford Lecture Series in Mathematics and its Applications, {\bf 32}, Oxford University Press, Oxford (2006).

\bibitem{C-F_sub} D. Cobb, F. Fanelli:
{\it Els\"asser formulation of the ideal MHD and improved lifespan in two space dimensions}.
Submitted (2020).

\bibitem{C-F_Nonlin} D. Cobb, F. Fanelli:
{\it On the fast rotation asymptotics of a non-homogeneous incompressible MHD system}.
Nonlinearity, {\bf 34} (2021), n. 4, 2483.

\bibitem{C-F_RWA} D. Cobb, F. Fanelli:
{\it Rigorous derivation and well-posedness of a quasi-homogeneous ideal MHD system}. Nonlinear Anal. Real World Appl., {\bf 60} (2021), 103284.


\bibitem{C-R} B. Cushman-Roisin, J.-M. Beckers:{\it ``Introduction to geophysical fluid dynamics.
Physical and numerical aspects. Second edition''}. Geophysics Series, {\bf 101}, Elsevier/Academic Press, Amsterdam (2011).



\bibitem{D_AT} R. Danchin: {\it The inviscid limit for density dependent incompressible fluids}. Ann. Fac. Sci. Toulouse Math., {\bf 15} (2006), n. 4, 637-688.


\bibitem{D_JDE} R. Danchin: {\it On the well-posedness of the incompressible density-dependent Euler equations in the $L^p$ framework}. J. Differential
Equations, {\bf 248} (2010), n. 8, 2130-2170.

\bibitem{D-F_JMPA} R. Danchin, F. Fanelli: {\it The well-posedness issue for the density-dependent Euler equations in endpoint Besov spaces}. J. Math. Pures Appl. (9), {\bf 96} (2011), n. 3, 253-278.

\bibitem{DeA-F} F. De Anna, F. Fanelli:
{\it Global well-posedness and long-time dynamics for a higher order quasi-geostrophic type equation}.
J. Funct. Anal., {\bf 274} (2018), n. 8, 2291-2355.

\bibitem{DS-F-S-WK} D. Del Santo, F. Fanelli, G. Sbaiz, A. Wróblewska-Kamińska: 
{\it A multiscale problem for viscous heat-conducting fluids in fast rotation}. 
J. Nonlinear Sci., {\bf 31} (2021), n. 1, 21.

\bibitem{DS-F-S-WK_sub} D. Del Santo, F. Fanelli, G. Sbaiz, A. Wróblewska-Kamińska: 
{\it On the influence of gravity in the dynamics of geophysical flows}. 
Submitted (2021).


\bibitem{Dut} A. Dutrifoy: {\it Examples of dispersive effects in non-viscous rotating fluids}. J. Math. Pures Appl. (9), {\bf 84} (2005), n. 3, 331-356.




\bibitem{Ebin} D. G. Ebin:
{\it Viscous fluids in a domain with frictionless boundary}.
In ``Global analysis-analysis on manifolds'', Teubner-Texte Math., {\bf 57}, Teubner, Leipzig (1983), 93-110.

\bibitem{F_JMFM} F. Fanelli:
{\it A singular limit problem for rotating capillary fluids with variable rotation axis}.
J. Math. Fluid Mech., {\bf 18} (2016), n. 4, 625-658.

\bibitem{F_MA} F. Fanelli:
{\it Highly rotating viscous compressible fluids in presence of capillarity effects}.
Math. Ann., {\bf 366} (2016), n. 3-4, 981-1033.


\bibitem{F_2019} F. Fanelli:
{\it Incompressible and fast rotation limit for barotropic Navier-Stokes equations at large Mach numbers}.
Physica D, {\bf 428} (2021), 133049.

\bibitem{Fan-G} F. Fanelli, I. Gallagher:
{\it Asymptotics of fast rotating density-dependent fluids in two space dimensions}.
Rev. Mat. Iberoam., {\bf 35} (2019), n. 6, 1763-1807.

\bibitem{F-Z} F. Fanelli, E. Zatorska:
{\it Low Mach number limit for degenerate Navier-Stokes equations in presence of strong stratification}.
Submitted (2021).

\bibitem{Feireisl} E. Feireisl:
{\it ``Dynamics of viscous compressible fluids''}.
Oxford Lecture Series in Mathematics and its Applications, Oxford University Press, Oxford (2004).

\bibitem{F-G-GV-N} E. Feireisl, I. Gallagher, D. G\'erard-Varet, A. Novotn\'y: 
{\it Multi-scale analysis of compressible viscous
and rotating fluids}. 
Comm. Math. Phys., {\bf 314} (2012), n. 3, 641-670.

\bibitem{F-G-N} E. Feireisl, I. Gallagher, A. Novotn\'y:
{\it A singular limit for compressible rotating fluids}.
SIAM J. Math. Anal., {\bf 44} (2012), n. 1, 192-205.

\bibitem{F-J-N} E. Feireisl, B. J. Jin, A. Novotn\'y:
{\it Relative entropies, suitable weak solutions, and weak-strong uniqueness for the compressible Navier-Stokes system}. 
J. Math. Fluid Mech., {\bf 14} (2012), n. 4, 717-730.

\bibitem{F-K-N-Z} E. Feireisl, R. Klein, A. Novotn\'y, E. Zatorska:
{\it On singular limits arising in the scale analysis of stratified fluid flows}.
Math. Models Methods Appl. Sci., {\bf 26} (2016), n. 3, 419-443.

\bibitem{F-L-N} E. Feireisl, Y. Lu, A. Novotn\'y: {\it Rotating compressible fluids under strong stratification}. Nonlinear
Anal. Real World Appl., {\bf 19} (2014), 11-18.


\bibitem{F-N} E. Feireisl, A. Novotn\'y:
{\it ``Singular limits in thermodynamics of viscous fluids''}.
Advances in Mathematical Fluid Mechanics, Birkh\"auser Verlag, Basel (2009).


\bibitem{F-N_CPDE} E. Feireisl, A. Novotn\'y:
{\it Multiple scales and singular limits for compressible rotating fluids with general initial data}.
Comm. Partial Differential Equations, {\bf 39} (2014), n. 6, 1104-1127.



\bibitem{F-N_AMPA} E. Feireisl, A. Novotn\'y:
{\it Scale interactions in compressible rotating fluids}.
Ann. Mat. Pura Appl., {\bf 193} (2014), n. 6, 1703-1725.


\bibitem{F-N-P} E. Feireisl, A. Novotný, H. Petzeltová:
{\it On the existence of globally defined weak solutions to the Navier-Stokes equations''}. J. Math. Fluid Mech., {\bf 3} (2001), n. 4, 358-392.

\bibitem{F-N-S} E. Feireisl, A. Novotn\'y, Y. Sun:
{\it Suitable weak solutions to the Navier-Stokes equations of compressible viscous fluids}. 
Indiana Univ. Math. J., {\bf 60} (2011), n. 2, 611-631.


\bibitem{F-Scho} E. Feireisl, M. E. Schonbek:
{\it On the Oberbeck-Boussinesq approximation on unbounded domains}.
Nonlinear partial differential equations, 131-168, Abel Symp., {\bf 7}, Springer, Heidelberg (2012).


\bibitem{Gall} I. Gallagher: 
{\it The tridimensional Navier-Stokes equations with almost bidimensional data:
stability; uniqueness; and life span}. Int. Math. Research Notices, (1997), n. 18, 919-935.

\bibitem{G-SR_Mem} I. Gallagher, L. Saint-Raymond:
{\it Mathematical study of the betaplane model: equatorial waves and convergence results}.
M\'em. Soc. Math. Fr., {\bf 107} (2006).



\bibitem{G-SR_2006} I. Gallagher, L. Saint-Raymond:
{\it Weak convergence results for inhomogeneous rotating fluid equations}.
J. Anal. Math., {\bf 99} (2006), 1-34.


\bibitem{Ger} P. Germain:
{\it Weak-strong uniqueness for the isentropic compressible Navier-Stokes system}. 
J. Math. Fluid Mech., {\bf 13} (2011), n. 1, 137-146.

\bibitem{H-K} T. Hmidi, S. Keraani:
{\it Incompressible viscous flows in borderline Besov spaces}.
Arch. Ration. Mech. Anal., {\bf 189} (2008), n. 2, 283-300.

\bibitem{J-J-N} D. Jessl\'e, B. J. Jin, A. Novotn\'y: {\it Navier-Stokes-Fourier system on unbounded domains:
weak solutions, relative entropies, weak-strong uniqueness}. SIAM J. Math. Anal., {\bf 45} (2013), n. 3,
1907-1951.

\bibitem{K-C-D} P. Kundu, I. Cohen, D. Dowling:  
{\it ``Fluid mechanics. Fifth edition''}. 
Academic Press, Elsevier, Oxford (2012).


\bibitem{K-M-N} Y.-S. Kwon, D. Maltese, A. Novotn\'y:
{\it Multiscale analysis in the compressible rotating and heat conducting fluids}.
J. Math. Fluid Mech., {\bf 20} (2018), n. 2, 421-444.

\bibitem{K-N} Y.-S. Kwon, A. Novotn\'y:
{\it Derivation of geostrophic equations as a rigorous limit of compressible rotating and heat conducting fluids with the general initial data}.
Discrete Contin. Dyn. Syst., {\bf 40} (2020), n. 1, 395-421.

\bibitem{Lions_1969} J.-L. Lions:
{\it ``Quelques méthodes de résolution des problèmes aux limites non linéaires''}. Dunod, Gauthier-Villars, Paris (1969). 

\bibitem{PL_Lions} P.-L. Lions:
{\it Compacité des solutions des équations de Navier-Stokes compressibles isentropiques}. C. R. Acad. Sci. Paris Sér. I Math., {\bf 317} (1993), n. 1, 115-120.

\bibitem{Lions_2} P.-L. Lions:
{\it ``Mathematical topics in fluid mechanics. Vol. 2. Compressible models''}.
Oxford Lecture Series in Mathematics, Oxford University Press, New York (1998).


\bibitem{L-M} P.-L. Lions, N. Masmoudi: {\it Incompressible limit for a viscous compressible fluid}. J. Math. Pures
Appl., {\bf 77} (1998), n. 6, 585-627.

\bibitem{Maj} A. Majda: {\it ``Introduction to PDEs and waves for the atmosphere and ocean''}. Courant Lecture Notes
in Mathematics, {\bf 9}, New York University, Courant Institute of Mathematical Sciences, New York;
American Mathematical Society, Providence, RI (2003).

\bibitem{Masm} N. Masmoudi:
{\it Rigorous derivation of the anelastic approximation}.
J. Math. Pures Appl. (9), {\bf 88} (2007), n. 3, 230-240.


\bibitem{M-2008} G. M\'etivier:
{\it ``Para-differential calculus and applications to the Cauchy problem for nonlinear systems''}.
Centro di Ricerca Matematica ``Ennio De Giorgi'' (CRM) Series, {\bf 5}, Edizioni della Normale, Pisa (2008).


\bibitem{Nir} L. Nirenberg:
{\it On elliptic partial differential equations}.
Ann. Scuola Norm. Sup. Pisa Cl. Sci. (3), {\bf 13} (1959), 115-162.

\bibitem{Ped} J. Pedlosky:
{\it ``Geophysical fluid dynamics''}.
Springer-Verlag, New-York (1987).

\bibitem{Pr} G. Prodi:
{\it Un teorema di unicit\`a per le equazioni di Navier-Stokes}. 
Ann. Mat. Pura Appl. (4), {\bf 48} (1959), 173-182.

\bibitem{Rud} W. Rudin:
{\it ``Principles of mathematical analysis. Third edition''}. International Series in Pure and Applied Mathematics, McGraw-Hill Book Co., New York-Auckland-Düsseldorf (1976).

\bibitem{Sbaiz} G. Sbaiz:
{\it Fast rotation limit for the $2$-D non-homogeneous incompressible Euler equations}.
Submitted (2021).

\bibitem{Scro} S. Scrobogna: 
{\it Highly rotating fluids with vertical stratification for periodic data and vanishing
vertical viscosity}. Rev. Mat. Iberoam., {\bf 34} (2018), n. 1, 1-58.

\bibitem{Ser} J. Serrin: 
{\it The initial value problem for the Navier-Stokes equations}. Nonlinear Problems (Proc. Sympos., Madison, Wis., 1962), Univ. of Wisconsin Press, Madison, Wis. (1963), 69-98. 

\bibitem{Sturm} T. W. Sturm: {\it ``Open channel hydraulics''}. McGraw-Hill, Singapore (2001).

\bibitem{Val} G. K. Vallis: {\it ``Atmospheric and oceanic fluid dynamics: fundamentals and large-scale circulation''}. Cambridge University Press, Cambridge (2006).

\bibitem{Vis} M. Vishik: {\it Hydrodynamics in Besov spaces}. Arch. Ration. Mech. Anal., {\bf 145} (1998), n. 3, 197-214.

\bibitem{WK} A. Wr\'oblewska-Kami\'nska:
{\it The asymptotic analysis of the complete fluid system on a varying domain: from the compressible to the incompressible flow}.
SIAM J. Math. Anal., {\bf 49} (2017), n. 5, 3299-3334.

\bibitem{Z} R. Kh. Zeytounian: {\it ``Theory and applications of viscous fluid flows''}. Springer-Verlag, Berlin (2004).

\end{thebibliography}
}
\end{document}